\documentclass[leqno,11pt]{amsart}
\usepackage{amssymb,amsmath,latexsym,amsfonts,amsbsy,amsthm,mathtools,graphicx,color,subeqnarray,cases,tikz} 

\usepackage{amsfonts}
\usepackage{color}
\usepackage{amssymb,amsthm,amsmath,hyperref}
\usepackage{epsfig}
\usepackage{graphicx,float}
\usepackage{CJK}
\usepackage{setspace}
\usepackage{cancel}
\usepackage{verbatim}
\usepackage{ulem}
 \newtheorem{thm}{Theorem}[section]
 \newtheorem{coro}[thm]{Corollary}
 \newtheorem{lem}[thm]{Lemma}
 \newtheorem{prop}[thm]{Proposition}
 \theoremstyle{definition}
 
 \newtheorem{rem}[thm]{Remark}
 \numberwithin{equation}{section}

\def\dl{\delta}
\def\tl{\tilde}

\def\Dl{\Delta}

\def\sig{\sigma}

\def\N{\mathbb{N}}
\def\T{\mathbb{T}}
\def\R{\mathbb{R}}
\def\Z{\mathbb{Z}}

\def\nn{\nonumber}

\def\sq{\sqrt}
\def\eps{\epsilon}

\def\fr{\frac}
\def\al{\alpha}

\def\la{\langle}
\def\ra{\rangle}

\def\pr{\partial}
\def\nb{\nabla}

\def\les{\lesssim}
\def\lm{\lambda}

\def\om{\omega}

\def\l|{\left\|}
\def\r|{\right\|}

\newcommand{\beq}{\begin{eqnarray}}
\newcommand{\eeq}{\end{eqnarray}}
\newcommand{\beqno}{\begin{eqnarray*}}
\newcommand{\eeqno}{\end{eqnarray*}}
\newcommand{\be}{\begin{equation}}
\newcommand{\ee}{\end{equation}}
\newcommand{\beno}{\begin{equation*}}
\newcommand{\eeno}{\end{equation*}}
\newtheorem{theorem}{Theorem}[section]

\newtheorem{Lemma A.1}{Lemma A.1}
\theoremstyle{definition}

\theoremstyle{remark}

\newcommand{\brak}[1]{\left\langle #1 \right\rangle} 
\newcommand{\norm}[1]{\left\| #1 \right\|}
\newcommand{\abs}[1]{\left| #1 \right|}

\allowdisplaybreaks[4]

\begin{document}
\title[Vlasov-Poisson-Landau equation]{Landau damping and the long-time collisionless limit of the Vlasov-Poisson-Landau Equation}

\author{Jacob Bedrossian}
\address[J. Bedrossian]{Department of Mathematics, University of California, Los Angeles, CA 90095, USA}
\email{jacob@math.ucla.edu}

\author{ Weiren Zhao}
\address[W. Zhao]{Department of Mathematics, New York University Abu Dhabi, Saadiyat Island, P.O. Box 129188, Abu Dhabi, United Arab Emirates.}
\email{zjzjzwr@126.com, wz19@nyu.edu}

\author{Ruizhao Zi}
\address[R. Zi]{School of Mathematics and Statistics, and Key Laboratory of Nonlinear Analysis \& Applications (Ministry of Education), Central China Normal University, Wuhan,  430079,  P. R. China.}
\email{rzz@ccnu.edu.cn}

\date{\today}

\begin{abstract}
In this paper, we study the Vlasov-Poisson-Landau Equations on $\mathbb{T}^3\times \mathbb{R}^3$ with small collision frequency $\nu\ll 1$.  
We prove that for $\nu$-independent perturbations of the global Maxwellians in Gevrey-$2_-$, solutions display uniform-in-$\nu$ Landau damping and enhanced dissipation. Moreover, the collisionless limit holds, that is, as $\nu\to 0_+$ for $0<t\ll \nu^{-\frac{1}{3}}$, the $\nu > 0$ solutions converge uniformly (and in much stronger norms) to the solution of the Vlasov-Poisson equation with the same initial data. To our knowledge, this work is hence the first justification that the collisionless prediction matches those of collisional plasmas in the nonlinear equations.    

The interaction between Landau damping and collisions requires several new ideas: 
\begin{enumerate}
    \item an infinite-regularity commuting vector field method, merged with Guo's weighted energy methods for the Landau operator and hypocoercivity to extract the enhanced dissipation; 
    \item A novel nearly-physical side treatment of the collisionless Vlasov echoes; 
    \item A new set of decomposition methods to treat the effects of the nonlinear collisions in the Volterra equation for the density (i.e., the ``collisional echoes'') 
    \item A new quasi-linearization method for treating the effect of the slowly evolving homogeneous modes over long times.
\end{enumerate}
As a side result, we also prove Landau damping and enhanced dissipation of $O(\epsilon\nu^{1/3})$ Sobolev-space perturbations of homogeneous distributions that are only $O(\epsilon)$ perturbations of global Maxwellians, generalizing the recent results of Chaturvedi, Luk, and Nguyen. 
As another side result, our methods also provide a nearly-completely physical-side proof of Mouhot and Villani's theorem in the full range of Gevrey-$3_-$.

\end{abstract}

\maketitle

\setcounter{tocdepth}{1}
{\small\tableofcontents}

\section{Introduction}
The Vlasov–Poisson–Landau (VPL) equation is one of the fundamental kinetic equations that describe the evolution of a collisional plasma, incorporating both long-range electrostatic mean-field effects (through the self-generated electric field) and Coulomb collision effects between charged particles (through the Landau collision operator); see the physics textbook \cite{goldston2020introduction} for a more complete discussion. 

In this paper, we consider the (single-species) VPL equation in the weakly collisional regime with Coulomb potential on $\mathbb{T}^3_x\times\mathbb{R}^3_v$:
\be\label{VPL}
\begin{cases}
\pr_tF+v\cdot\nb_xF+E(t,x)\cdot\nb_vF=\nu {\bf Q}(F,F)(t,x,v),\\[1mm]
\displaystyle E(t,x)=-\nb_x(-\Dl_x)^{-1}\left(\int_{\mathbb{R}^3}F(t,x,v)dv-\fr{1}{(2\pi)^3}\int_{\mathbb {T}^3}\int_{\mathbb{R}^3}F(t,x,v)dvdx\right),\\[1mm]
F(t=0, x, v)=F_{\mathrm{in}}(x,v),
\end{cases}
\ee
where $\mathbb{T}_x$ is normalized to length $2\pi$, and the bilinear operator ${\bf Q}$ is the Landau collision operator, given by 
\begin{align}
{\bf Q}(G,F)(t,x,v)\nn=&\pr_{v_i}\int_{\R^3}\Phi^{ij}(v-\tl{v})\left[G(t,x,\tl{v})(\pr_{v_j}F)(t,x,v)-F(t,x,v)(\pr_{v_j}G)(t,x,\tl{v})\right]d\tl{v}\\
\nn=&\pr_{v_i}\left[(\Phi^{ij}*G)(t,x,v)(\pr_{v_j}F)(t,x,v)\right]-\pr_{v_i}\left[ (\Phi^{ij}*\pr_{v_j}G)(t,x,v)F(t,x,v)\right],
\end{align}
with
\beno
\Phi^{ij}(z)=\fr{1}{|z|}\left(\dl_{ij}-\fr{z_iz_j}{|z|^2}\right).
\eeno
Here, we will normalize the collision invariants of the initial data as: for $i=1,2,3$, 
\begin{align}\label{eq:normalization}
\begin{aligned}
    &\int_{\mathbb{T}^3\times \R^3}F_{\rm in}(x,v)dvdx=(2\pi)^3\pi^{\fr32}
,\quad  \int_{\mathbb{T}^3\times \R^3}v_iF_{\rm in}(x,v)dvdx=0, \\  &\int_{\mathbb{T}^3\times \R^3}|v|^2F_{\rm in}(x,v)dvdx=\fr32\pi^{\fr32}(2\pi)^3. 
\end{aligned}
\end{align}
The coefficient $0<\nu\ll 1$ is the inverse Knudsen number $K_n$, which is the ratio of the mean free path to a characteristic length scale, i.e., essentially the collision frequency. In many applications of interest in plasma physics, this parameter is very small \cite{boyd2003physics,goldston2020introduction}.
Global existence of smooth solutions near Maxwellians was proved by Guo in \cite{guo2012vlasov}, although these results do not hold uniformly $\nu$. 

The $\nu = 0$ equation is called \emph{Vlasov-Poisson} and is frequently studied as a simplification of VPL when $\nu$ is small. 
In 1946, Landau predicted the phenomenon known as \emph{Landau damping} in the linearized VP equations, wherein analytic initial perturbations to a global Maxwellian in $\mathbb T^d$ results in exponentially decaying electric fields; in fact 
\begin{align*}
 \abs{\hat{E}(t,k)} \lesssim e^{-\lambda_0 \abs{kt}},    
\end{align*}
with finite regularity initial perturbations resulting in $\brak{kt}^{-\sigma}$ decay rates. 
Essentially, one expects decay at basically the same rate as the kinetic free transport equation $\partial_t f + v \cdot \nabla_x f = 0$. 
This was later proved to hold in the nonlinear Vlasov-Poisson equations for all sufficiently smooth and sufficiently small perturbations of stable homogeneous equilibria in a periodic box by Mouhot-Villani \cite{MouhotVillani2011}; see also \cite{BMM2016, grenier2021landau, ionescu2024nonlinear} for perturbations in sharper Gevrey regularity classes (see also \cite{caglioti1998time,hwang2009existence} for earlier results starting from $t \pm \infty$). 
Note that the nonlinear oscillations known as plasma echoes rule out such kinds of weakly-nonlinear results in finite regularity \cite{bedrossian2020nonlinear}; see also the experiments and discussions in the physics literature \cite{gould1967plasma}. 

Collisions in kinetic theory are not simple dissipative regularizations. Indeed, collisional effects lead to hydrodynamic behavior (i.e., local Maxwellians with hydrodynamic fields solving compressible Euler), not kinetic, and a priori, there is no reason to believe that Landau damping will necessarily hold for $\nu$ moderate/large or, even for all time with $0 < \nu \ll 1$, as one could worry that the tendency towards hydrodynamic behavior may inhibit the Landau damping on long enough time-scales. 
Indeed, numerical evidence strongly suggests that even weak collisions can slow down the rate of Landau damping (See Figure 4.5 (a), (b) in \cite{ye2024energy}). 
Determining to what degree this is possible is necessary to be certain that the collisionless predictions really correspond to the experimental observations. 
Previous works on Vlasov-Fokker-Planck \cite{bedrossian2017suppression} and VPL \cite{chaturvedi2023vlasov} showed nonlinear dynamics agree with linear predictions for small, Sobolev regularity perturbations $\epsilon \lesssim \nu^{1/3}$ and moreover that the interaction of phase mixing and collisions leads to an enhanced dissipation effect, as predicted in plasmas by physicists in \cite{lenard1958plasma, o1968effect, su1968collisional} and as observed for passive scalars \cite{bedrossian2017enhanced, zelati2020relation, wei2021diffusion} and hydrodynamics \cite{bedrossian2016enhanced, coti2020enhanced, li2023metastability, lizhao2024asymptotic, Kelvin1887, WeiZhangZhao2020} already. 
However, this does not imply anything about the much larger initial conditions studied by Mouhot and Villani. 
For this reason, we are interested in proving a \emph{long-time collisionless limit} and global-in-time stable, quantitative decay estimates. 

In this paper, we focus on the mathematical justification of the collisionless limit from the VPL equation to the VP equation over long time-scales in order to justify Landau damping uniformly in small collision parameters and to verify the predicted long-time behavior as $\nu\to 0_+$. We prove that for times $t \ll \nu^{-1/3}$, solutions converge to the solution of the Vlasov--Poisson equation in suitably strong norms (and so exhibit the kind of Landau damping predicted by the linearized VP equation for long times). Over even longer times, the enhanced dissipation kicks in and overpowers the solution. Landau damping also continues, but possibly at a much slower rate (though the enhanced dissipation is quite powerful by this time anyway).   
We studied this problem for the case of the Vlasov-Fokker-Planck equation in \cite{BZZ2024VPFP}; however, the case of Landau collisions is much harder and requires a completely different technique.

\subsection{Main results}
In this paper, we consider the initial data $F_{\rm in}$ close to the global Maxwellian 
$
\mu(v)=e^{-|v|^2},
$
which is a steady state solution to \eqref{VPL} satisfying \eqref{eq:normalization}. Let $f$ be such that
\[
F=\mu+\sqrt{\mu}f,
\]
with initial condition
$
f|_{t=0}=f_{\rm in},
$
which by \eqref{eq:normalization} satisfies
\[
\int_{\mathbb{T}^3}\int_{\R^3}f_{\rm in}\sq{\mu}dvdx=\int_{\mathbb{T}^3}\int_{\R^3}f_{\rm in}v_j\sqrt{\mu}dvdx=\int_{\mathbb{T}^3}\int_{\R^3}f_{\rm in}|v|^2\sqrt{\mu}dvdx+\int_{\mathbb{T}^3}|E_{\rm in}|^2dx=0.
\]
Rewriting the nonlinearity as follows 
\begin{align*}
{\bf Q}(\mu+\mu^\fr12f,\mu+\mu^\fr12f)=&\mu^\fr12\big(\mathcal{K}[f]+\mathcal{A}[f]+\Gamma(f,f)\big)=\mu^\fr12\big(-Lf+\Gamma(f,f)\big),\\
\nn E(t,x)\cdot\nb_vF
=&\mu^\fr12\big( E(t,x)\cdot\nb_vf-E(t,x)\cdot v f-2(E(t,x)\cdot v)\mu^\fr12\big), 
\end{align*} 
the VPL equation can be rewritten as 
\begin{align}\label{pVPL}
\begin{cases}
\pr_tf+v\cdot\nb_xf+ E(t,x)\cdot\nb_vf-E(t,x)\cdot v f-2E(t,x)\cdot v\mu^\fr12+\nu Lf=\nu\Gamma(f,f),\\[3mm]
\displaystyle E(t,x)=-\nb_x\phi(t,x), \quad -\Dl_x\phi=\int_{\R^3}f(t,x,v)\mu^\fr12(v)dv=\rho(t,x),
\end{cases}
\end{align}
where
\be\label{def-Gamma}
\begin{aligned}
\Gamma(g_1, g_2)=&\pr_{v_i}\left[[\Phi^{ij}*(\mu^\fr12g_1)]\pr_{v_j}g_2\right]-[\Phi^{ij}*(v_i\mu^\fr12g_1)]\pr_{v_j}g_2\\
&-\pr_{v_i}\left[\Phi^{ij}*(\mu^\fr12\pr_{v_j}g_1)g_2 \right] +\left[\Phi^{ij}*(v_i\mu^\fr12\pr_{v_j}g_1) \right]g_2
\end{aligned}
\ee
is the nonlinear (quadratic) Landau collision operator and
$
L=-\mathcal{K}-\mathcal{A}
$ is the linearized Landau collision operator, decomposed into the nonlocal part $\mathcal{K}$ and the local part $\mathcal{A}$ with
\begin{align}\label{def-K}
&\mathcal{K}[g_1]=-\mu^{-\fr12}(v)\pr_{v_i}\left[\mu(v)\int_{\R^3}\Phi^{ij}(v-\tl{v})\mu^{\fr12}(\tl{v})\left(\pr_{v_j}g_1(t,x,\tl{v})+\tl{v}_jg_1(t,x,\tl{v}) \right)d\tl{v}\right],\\
\label{def-A}
&\mathcal{A}[g_2]=\pr_{v_i}\Big[\sig^{ij}\pr_{v_j}g_2\Big]-v_iv_j\sig^{ij}g_2+\pr_{v_i}\sig^ig_2.
\end{align}
Here, the coefficients are given by
\be\label{def-sig}
\sig^{ij}=\Phi^{ij}*\mu,\quad \sig^i=\Phi^{ij}*(v_j\mu)\overset{\rm by\ \eqref{concel1}}{=}v_j\Phi^{ij}*\mu=v_j\sig^{ij}.
\ee
Provided solutions remain smooth and sufficiently well-localized in velocity, it is easy to verify that the following conservation laws hold for all $t\geq 0$,
\begin{gather}
\label{conser-mass}\int_{\mathbb{T}^3}\int_{\R^3}f(t,x,v)\sq{\mu}dvdx=\int_{\mathbb{T}^3}\int_{\R^3}f_{\rm in}\sq{\mu}dvdx,\\
\label{conser-momen}\int_{\mathbb{T}^3}\int_{\R^3}f(t,x,v)v_j\sq{\mu}dvdx=\int_{\mathbb{T}^3}\int_{\R^3}f_{\rm in}v_j\sqrt{\mu}dvdx,\\
\label{conser-en}\int_{\mathbb{T}^3}\int_{\R^3}f(t,x,v)|v|^2\sqrt{\mu}dvdx+\int_{\T^3}|E|^2dx=\int_{\mathbb{T}^3}\int_{\R^3}f_{\rm in}|v|^2\sqrt{\mu}dvdx+\int_{\T^3}|E_{\rm in}|^2dx.
\end{gather}
Our first theorem describes the long-time dynamics in $\nu \to 0$ and $t \to \infty$.
Since we are interested in proving that the linearized predictions of Landau damping hold (essentially) uniformly in $\nu$, we must work in Gevrey regularity. 
Specifically, we work in Gevrey regularity until $t \sim \nu^{-1/2}$, at which point the problem is collision dominated and we describe the remainder of the relaxation process in Sobolev regularity. 
The initial $t\lesssim \nu^{-1/2}$ regime makes up the majority of the work in this paper. 
As the theorem is rather technical, we will summarize the main results; See section \ref{sec:Outline} for the detailed statements. 
\begin{theorem}\label{Thm: main}
    Let $\frac12<s<\frac{2}{3}$. Let $\lambda>0$. There exist $\eps_0,\nu_0>0$ and $\ell_0$ such that if $\eps \in(0,\eps_0)$, $\nu \in (0,\nu_0)$ and $\ell \geq \ell_0$, if the initial data $f_{\mathrm{in}}$ satisfies
\begin{align}\label{conserv-initial}
    \int_{\T^3\times\R^3}f_{\rm in}\sqrt{\mu}dvdx=\int_{\T^3\times\R^3}f_{\rm in}v_j\sqrt{\mu}dvdx =\int_{\T^3\times\R^3}f_{\rm in}|v|^2\sqrt{\mu}dvdx+\|E_{\rm in}\|_{L^2_x}^2 = 0, 
\end{align}
    and 
    \begin{align*}
    \|\brak{v}^\ell e^{\lambda \abs{\nabla}^s}f_{\mathrm{in}}\|_{L^2} = \eps \leq \eps_0,
    \end{align*}
    then the VPL equation \eqref{pVPL} with all $0 \leq \nu<\nu_0$ has a global smooth solution $(f, E)$ which satisfies the following (with all implicit constants independent of $\eps,\nu,$ and $t$),
    \begin{itemize} \item For $t\in [0,\nu^{-\fr12}]$: 
    \begin{itemize}
    \item Uniform Landau damping:
$$\|e^{\lambda'\brak{\nabla_x t}^s} E(t)\|_{L^2_x}\lesssim \epsilon,$$
     for some $\lambda > \lambda'>0$ independent of $\nu$.
     \item Enhanced dissipation: denoting $f_{\neq} := f - \frac{1}{(2\pi)^3}\int_{\mathbb T^d} f(x,v) dx$, 
     \begin{align*}
     \norm{f_{\neq}}_{L^2} \lesssim \epsilon\brak{\nu^{1/3} t}^{-\ell/6}.    
     \end{align*}
     \end{itemize}
    \item For $t\geq \nu^{-\fr12}$, for any $N_1$, there holds the Sobolev rate Landau damping and enhanced dissipation
    \begin{align*}
    & \norm{\brak{t \nabla_x}^{N_1} E(t)}_{L^2_x}  \lesssim \epsilon, \\ 
    & \norm{f_{\neq}}_{L^2}  \lesssim \epsilon \brak{\nu^{1/3} t}^{-\ell/6}. 
    \end{align*}
    \end{itemize}
Here, $\langle v\rangle=\Big(1+\sum\limits_{i=1}^3v_i^2\Big)^{\frac{1}{2}}$. 
\end{theorem}

Our second result studies the collisionless limit on the sharp $\nu^{-1/3}$ time-scale. 
\begin{theorem}\label{thm:limit}
    Let $(f^{(\nu)}, E^{(\nu)})$ be the solution obtained in Theorem \ref{Thm: main} with $\nu\in (0,\nu_0)$. Let $(f^{(0)}, E^{(0)})$ be the solution to the Vlasov-Poisson equation with the same initial data $f_{\mathrm{in}}$, namely, $\nu=0$ in \eqref{pVPL}. Then there exists $c_0$ independent of $\nu$, such that for all $t\in [0,c_0\nu^{-\frac{1}{3}}]$, it holds that
    \begin{align*}
        \|f^{(\nu)}-f^{(0)}\|_{L^2}\lesssim \nu t^3.
    \end{align*}
\end{theorem}

The following remarks are in order:
\begin{rem}
    By improving the regularity estimates of coefficients appearing in the linear Landau operator (see Lemma \ref{lem: sig_ij}), one can extend our results for Gevrey index $s\in (\fr12,1)$. In this paper, we mainly focus on the case that $s$ is close to the endpoint $1/2$.
\end{rem}
\begin{rem}
    The results hold for the collisionless case, namely, $\nu=0$. Indeed, contained in our work is an almost-entirely physical-side proof of Mouhot and Villani's theorem all the way down to $s > 1/3$, also without the standard $x \mapsto x+tv$ change of variables (i.e., Gevrey regularity is not quantified with Fourier multipliers on the shifted unknowns but instead with infinite series of commuting vector fields). The restriction $s > 1/2$ only arises from the collision terms.
\end{rem}
\begin{rem}
    In plasma physics, Landau damping is analogous to \emph{inviscid damping} in fluid dynamics. For recent results on inviscid damping in shear flows---both for the Euler equation and the inhomogeneous Euler equation---we refer to \cite{BM2015, IonescuJia2020cmp, IJ2020, MasmoudiZhao2020} and \cite{chen2025nonlinear, zhao2025inviscid}. A similar corresponding long-time, inviscid limit has been proved in a periodic channel with Navier-type boundary conditions in \cite{bedrossian2025pseudo,bedrossian2024uniform,bedrossian2025stability}, however this work contends with largely different difficulties and with largely different methods. 
\end{rem}
\begin{rem}
Like the more recent proofs of Landau damping in Vlasov-Poisson \cite{BMM2016,grenier2021landau} and the work of Yan Guo on Vlasov-Poisson-Landau \cite{guo2012vlasov}, we believe our proof should extend in a straightforward manner to a two-species model (ion and electrons) provided that both have the same temperature (in energy units) and the same collision frequency (or at least roughly the same). The case of significantly different collision frequencies and initial temperatures may require some further new ideas. 
\end{rem}

\begin{rem}
The rigorous theorems stated in Section \ref{sec:Outline} contain much more precise estimates on $(f,E)$. 
\end{rem}
\begin{rem}
One naturally conjectures that the result may hold down to the sharp Gevrey index $s>\frac{1}{3}$. However, some new ideas seem to be required. We discuss this problem in the Appendix. Similarly, it would be a challenging problem to obtain the same results with the sharp Gevrey index $s=\frac{1}{3}$. We refer to a recent work \cite{ionescu2024LandauDampingsharp} for the collisionless case.  
Indeed, the requirement $s>1/2$ seems technical, as without the Vlasov nonlinear term $E\cdot \nabla_v$ (i.e., just the Landau equations), one can prove an analogous result in Sobolev regularity \cite{bedrossian2024taylor}.
\end{rem}
\begin{rem}
We refer to \cite{han2021asymptotic,huang-NHNguyen2022sharp,bedrossian2018landau} for Vlasov with screened Coulomb interactions on $\mathbb R^d \times \mathbb R^d$ (this model is suitable for ions) and to \cite{han2021linearized,bedrossian2022linearized,ionescu2024nonlinear,nguyen2024new,nguyen2024landau} for progress on Vlasov-Poisson on $\mathbb R^d \times \mathbb R^d$, which remains open in the nonlinear case for Maxwellian backgrounds. 
It is interesting to study the VPL with small collision coefficient $\nu$ on $\mathbb R^d \times \mathbb R^d$. We refer to \cite{strain2013vlasov} for the global stability result when $\nu=1$ and $d=3$. 
\end{rem}
\begin{rem}
    For $\nu$-dependent perturbations, the asymptotic stability threshold problem was studied in \cite{chaturvedi2023vlasov} for Sobolev perturbations. Specifically, they proved similar results but assuming $O(\nu^{1/3})$ sized initial data. This is expected to be optimal in Sobolev regularity as otherwise collisions are not strong enough to dominate the plasma echoes (known to exist in the collisionless equations \cite{bedrossian2020nonlinear}).  
    Contained in the proof of Theorem \ref{Thm: main} is an extension of their result to the case when the homogeneous background can be a full $O(\eps)$ away from Maxwellian (as opposed to $\nu^{1/3}$). This is essentially the content of our proof for $t \gtrsim \nu^{-1/2}$. It is not a straightforward generalization of their argument; see Section \ref{sec:Outline} for more discussions.
\end{rem}
\begin{rem}
    We refer to \cite{bedrossian2020nonlinear,tristani2017landau,luo2024weak, BZZ2024VPFP} for linear and nonlinear Fokker-Planck collisions. 
\end{rem}
\begin{rem}
    It is interesting to consider the stability threshold problem and study the relationship between regularity and perturbation size. It is well-studied for the two-dimensional Couette flow in a periodic cylinder $\mathbb T \times \mathbb R$. We refer to \cite{BVW2018, BMV2016, liMasmoudiZhao2023dynamical, liMasmoudiZhao2022asymptotic, MasmoudiZhao2020cpde, MasmoudiZhao2019, wei2023nonlinear} and references therein. 
\end{rem}

\subsection{Notations}\label{sec: notations}
In this paper, we denote by $\Z$ the set of integers, $\Z^3=\Z\times\Z\times\Z$, and $\Z^3_*=\Z^3\setminus\{\bf 0\}$, where ${\bf 0}=(0,0,0)\in\Z^3$. Moreover, we denote by $\N$ the set of nonnegative integers, namely, $\N=\{0,1,2,3,\cdots\}$. 
For $n=(n_1,n_2,n_3, n_4, n_5, n_6)\in \mathbb{N}^6$, we define $\tl{n}=(n_1, n_2, n_3, 0, 0, 0)$,  $\bar{n}=(0,0,0, n_4, n_5, n_6)$. To lighen  the notation, we also use $\bar{n}$ and $\tl{n}$ to denote $(n_4,n_5,n_6)$ and $(n_1,n_2,n_3)$, respectively, when no confusion arises.

For $m,n\in \mathbb{N}^6$, we define $n<m$ (or $n\leq m$) if $n_i<m_i$ (or $n_i\leq m_i$) for all $i=1,2,...,6$. If $n\le m$, we denote the multivariate combination numbers $C_{m}^{n}=\prod_{i=1}^6\fr{m_i!}{(m_i-n_i)!n_i!}$, $C_{m}^{\bar{n}}=\prod_{i=4}^6\fr{m_i!}{(m_i-n_i)!n_i!}=C^{\bar{n}}_{\bar{m}}$, and $C_{m}^{\tl{n}}=\prod_{i=1}^3\fr{m_i!}{(m_i-n_i)!n_i!}=C^{\tl{n}}_{\tl{m}}$.   
For $n\in\N^6$, we also denote
\[
C_{|n|}^{n}:=C_{|n|}^{n_1}C_{|n|-n_1}^{n_2}\cdots C_{|n|-(n_1+n_2+\cdots n_4)}^{n_5}=\fr{|n|!}{n_1!n_2!\cdots n_6!}.
\]
Then there holds
$1\le C_{|n|}^n\le 6^{|n|}$.

For $n\in\N$, we use the notation
\[
\Gamma_{s}(n)=2^{-25}(n!)^{\fr{1}{s}}(1+n)^{-12}.
\]
For $n\in\N^6$, we define $\Gamma_s(n)=\prod_{i=1}^6\Gamma_{s}(n_i)$.

Throughout this paper, we denote 
\begin{align}
a_{m,\lm,s}(t):=\fr{\lm^{|m|}(t)}{\Gamma_s(|m|)}C_{|m|}^m,\quad b_{m,n,s}:=\fr{C_m^nC_{|m|}^m}{C_{|n|}^nC_{|m-n|}^{m-n}}\fr{\Gamma_s(|n|)\Gamma_s(|m-n|)}{\Gamma_s(|m|)}, \label{def:a}
\end{align}

\[
 b_{m,n,n',s}:=\fr{C_m^n C_{m-n}^{n'}C_{|m|}^m}{C_{|n|}^nC_{|n'|}^{n'}C_{|m-n-n'|}^{m-n-n'}}\fr{\Gamma_s(|n|)\Gamma_s(|n'|)\Gamma_s(|m-n-n'|)}{\Gamma_s(|m|)},
\]
and
\begin{align}
 b_{m,n,n', n'',s}:=\nn&\fr{C_m^n C_{m-n}^{n'}C_{n'}^{n''}C_{|m|}^m}{C_{|n|}^nC_{|n''|}^{n''}C_{|n'-n''|}^{n'-n''}C_{|m-n-n'|}^{m-n-n'}}\fr{\Gamma_s(|n|)\Gamma_s(|n''|)\Gamma_s(|n'-n''|)\Gamma_s(|m-n-n'|)}{\Gamma_s(|m|)}\\
 =&\nn\Big(\fr{|n|!|n''|!|n'-n''|!|m-n-n'|!}{|m|!}\Big)^{\fr{1}{s}-1}\fr{2^{-25}(|m|+1)^{12}}{(|m-n|+1)^12(|n|+1)^12}\\
 &\times \fr{2^{-25}(|m-n|+1)^{12}}{(|m-n-n'|+1)^12(|n'|+1)^12}\fr{2^{-25}(|n'|+1)^{12}}{(|n''|+1)^12(|n'-n''|+1)^12}.
\end{align}
We use $a_{m,s}(t)$ to stand for $a_{m,\lm,s}(t)$ whenever without causing confusion. 
\begin{rem}\label{Rmk: coefficients a-b}
It holds that,
    \begin{gather*}
    a_{m,\lm,s}(t)C_m^n=b_{m,n,s}a_{n,\lm,s}(t)a_{m-n,\lm,s}(t),\\
    a_{m,\lm,s}(t)C_m^nC_{m-n}^{n'}=b_{m,n,n',s}a_{n,\lm,s}(t)a_{n',\lm,s}(t)a_{m-n-n',\lm,s}(t),\\
    a_{m,\lm,s}(t)C_m^nC_{m-n}^{n'}C_{n'}^{n''}=b_{m,n,n',n''}a_{n,\lm,s}(t)a_{n'',\lm,s}(t)a_{n'-n'',\lm,s}(t)a_{m-n-n',\lm,s}(t).
\end{gather*}
\end{rem}

For $f, g\in L^2(\mathbb{T}^3\times\mathbb{R}^3)$, let us denote by $\la f, g\ra_{x,v}$ the inner product of $f$ and $g$ in $L^2(\mathbb{T}^3\times\mathbb{R}^3)$. Moreover, for $p\in[1,\infty]$, we use $\|\cdot\|_{L^p_{x,v}}$ to denote $\|\cdot\|_{L^p(\T^3\times\R^3)}$. 
\begin{itemize}
    \item If the domain $\mathbb{T}^3\times\mathbb{R}^3$ is replaced by $\mathcal{\R}^3$($\T^3$), the notations $\la\cdot,\cdot\ra_{x,v}$ and $\|\cdot\|_{L^p(\T^3\times\R^3)}$ will be changed to be $\la\cdot,\cdot\ra_{v}$ and $\|\cdot\|_{L^p_v}(\|\cdot\|_{L^p_x})$, respectively; 
    \item Provided that no ambiguity arises, for $\ell\in\R$, we shall abbreviate   the weighted  inner product $\big\la \la v\ra^{\ell}f,\la v\ra^{\ell}g \big\ra_{x,v}$ and the weighted $L^2$ norm $\|\la v\ra^{\ell}f\|_{L^2_{x,v}}$ as   $\la f,g\ra_{\ell}$ and $\|f\|_{L^2_{x,v}(\ell)}$, respectively. 
    \item Sometimes, we use the notation $\|f\|^2_{L^2_{x,v}(|v|\le \zeta)}$ to denote $\int_{|v|\le\zeta}|f(x,v)|^2dxdv$ for given $\zeta>0$.
\end{itemize}

We also define the Fourier transform of $f$ by 
\[
\hat{f}_k(t,\eta)=\mathcal{F}_{x,v}[f]_k(\eta)=\fr{1}{(2\pi)^3}\int_{\mathbb{T}^3\times\R^3}f(t, x, v) e^{-ik\cdot x}e^{-i\eta\cdot v}dxdv,
\]
and
\[
\mathcal{F}_{x}[f]_k(t,v)=\fr{1}{(2\pi)^3}\int_{\mathbb{T}^3}f(t, x, v) e^{-ik\cdot x}dx,\quad \mathcal{F}_{v}[f](t, x,\eta)=\int_{\mathbb{R}^3}f(t, x, v) e^{-i\eta\cdot v}dv.
\]
To lighten the notation, sometimes we also use $\hat{f}_k(t,v)$ to denote $\mathcal{F}_{x}[f]_k(t,v)$ without causing confusion.

Throughout this paper, let us denote by ${\rm P}_0$ the projection operator on the null space of the linear Landau operator $L$, namely,
\begin{align}\label{projection}
    {\rm P}_0f(t,x,v)=a(t,x)\sqrt{\mu}+\sum_{j=1}^3b_j(t,x)v_j\sqrt{\mu}+c(t,x)|v|^2\sqrt{\mu},
\end{align}
where 
\begin{align*}
    a(t,x)&=\frac{\int_{\mathbb{R}^3}f(t,x,v)\sqrt{\mu}(v)dv}{\int_{\mathbb{R}^3}\mu(v)dv}-\fr32c(t,x),\\
    b_j(t,x)&=\frac{\int_{\mathbb{R}^3}f(t,x,v)v_j\sqrt{\mu}(v)dv}{\int_{\mathbb{R}^3}v_j^2\mu(v)dv},\\
    c(t,x)&=\frac{\int_{\mathbb{R}^3}f(t,x,v)(|v|^2-\fr32)\sqrt{\mu}(v)dv}{\int_{\mathbb{R}^3}(|v|^2-\fr32)^2\mu(v)dv}. 
\end{align*}

\section{Outline of the proof and main ideas} \label{sec:Outline}
In this section, we discuss the main difficulties and the new ideas required for the proof. 

\subsection{Norms and vector fields}
As observed in \cite{MouhotVillani2011}, one cannot quantify higher regularity using usual derivatives. 
Here we employ the method of commuting vector fields, specifically vector fields which commute with $\partial_t + v \cdot \nabla_x$. 
Hence, for $\al\in\N^3$, we define the vector fields
$Y^\al=Y_1^{\al_1}Y_2^{\al_2}Y_3^{\al_3}$ with $Y_i=\partial_{v_i}+t\partial_{x_i}, i=1, 2, 3$. We also denote $Y=\nb_v+t\nb_x$, $Y^*=\nb_v-t\nb_x$ and $Z=(\nb_x,Y)$.  Sometimes, we use the notation that $\tilde{Y}=\nabla_v+ikt$ and $\tl{Y}^{\al}=(\pr_{v_1}+ik_1t)^{\al_1}(\pr_{v_2}+ik_2t)^{\al_2}(\pr_{v_3}+ik_3t)^{\al_3}$.
For $m\in\N^6$, we define the family of vector fields 
$Z^m:=(\nabla_x, Y)^m=\pr_{x}^{\tl{m}}Y^{\bar{m}}=\partial_{x_1}^{m_1}\partial_{x_2}^{m_2}\partial_{x_3}^{m_3}Y_1^{m_4}Y_2^{m_5}Y_3^{m_6}$. For a smooth function $f$ defined on $[0,T]\times \mathbb{T}^3\times \mathbb{R}^3$, we denote $f^{(m)}(t, x, v)=Z^mf(t, x, v)$. Then it is easy to check that
\begin{align}\label{Z-commute}
    Z^m(\partial_t+v\cdot\nabla_x)f=(\partial_t+v\cdot\nabla_x)f^{(m)}. 
\end{align}
Note that if $f$ is independent of $v$, then 
\begin{align*}
   f^{(m)}(t, x)= (\nabla_x, t\nabla_x)^mf(t, x)=t^{n_4+n_5+n_6}\partial_{x_1}^{n_1+n_4}\partial_{x_2}^{n_2+n_5}\partial_{x_3}^{n_3+n_6} f(t, x).
\end{align*}

In this paper, for a given small constant $a\in(0,\fr{s}{10})$, we choose 
\begin{align}\label{def-lm}
\lm(t):=\lm_\infty+\fr{\tl\dl}{(1+t)^{a}},
\end{align}
with $\lm_\infty$ and $\tl{\dl}$ two positive constants to be determined later. Then it is easy to see that $\lm_\infty\le\lm(t)\le\lm(0)=\lm_\infty+\tl\dl$. By the mean value theorem, it is not difficult to verify that 
\begin{align}\label{lm-gap}
\lm(\tau)-\lm(t)\ge a\tl{\dl}\fr{t-\tau}{(1+t)(1+\tau)^{a}}, \quad{\rm for\ \ all\ \ } \tau<t.
\end{align}
This gap plays an important role in dealing with the \underline{plasma echoes}, as was used in \cite{BMM2016,grenier2021landau} (although here we will use the trick slightly differently).

Next, we recall the norms used in \cite{Guo2002} for measuring the dissipative terms coming from the linearized Landau equations. For $\ell\in\R$, we denote 
\begin{align}\label{L-norm}
        |g|_{\sigma,\ell}^2=\int_{\mathbb{R}^3}\langle v\rangle^{2\ell}
        \left[\sigma^{ij}\partial_{v_i}g\partial_{v_j}g+\sigma^{ij}v_iv_j |g|^2\right]dv. 
\end{align}
Similarly,
\begin{align}\label{L-norm'}
        \|g\|_{\sigma,\ell}^2=\int_{\mathbb{T}^3}\int_{\mathbb{R}^3}\langle v\rangle^{2\ell}
        \left[\sigma^{ij}\partial_{v_i}g\partial_{v_j}g+\sigma^{ij}v_iv_j |g|^2\right]dvdx. 
\end{align}
In particular, if $\ell=0$, we simply write $|\cdot|_\sig=|\cdot|_{\sig,0}$, and $\|\cdot\|_\sig=\|\cdot\|_{\sig,0}$.
We refer to the Appendix for more discussion of the basic properties of this norm.

Now we define the Gevrey norms associated with the vector fields $Z$. For any smooth function $f$ defined on $[0,T]\times \mathbb{T}^3\times \mathbb{R}^3$ and weight function $w=w(t,x,v)$, we define
\begin{align}\label{def-Gnorm}
    \|f\|_{\mathcal{G}_{s, w}^{\lambda,N}}^2
    &=\sum_{|\beta|\leq N} \kappa^{2|\beta|}\sum_{0\leq m\in \mathbb{N}^6}a_{m,\lm,s}^2(t)\big\|w(t, x,v)f^{(m+\beta)}\big\|_{L^2_{x,v}}^2,
\end{align}
where $\lambda$ can be a constant or a decreasing function of $t$ chosen as in \eqref{def-lm},  $\kappa\in(0,1)$ is a constant, $N\in\N$, and the combinatorial coefficients $a_{m,\lambda,s}(t)$ are defined in \eqref{def:a}. 
We also use the following conventions:
\begin{itemize}
    \item if $N=0$, we use $\|\cdot\|_{\mathcal{G}^{\lm}_{s,\ell}}$ to denote $\|\cdot\|_{\mathcal{G}^{\lm,0}_{s,\la v\ra^\ell}}$;

    \item if the $L^2_{x,v}$ norm $\|\cdot\|_{L^2_{x,v}}$ on the right hand side of \eqref{def-Gnorm} is replaced by $\|\cdot\|_{L^2_v}$, we use the notation $\|\cdot\|_{\tl{\mathcal{G}}^{\lm,N}_{s,w}}$ to denote the corresponding norm;
    
    \item if $w(t,x,v)=\la v\ra^\ell$, we use the notation $\|\cdot\|_{\mathcal{G}^{\lm,N}_{s,\ell}}$ to denote $\|\cdot\|_{\mathcal{G}^{\lm,N}_{s,\la v\ra^\ell}}$, and similar convention applies to $\|\cdot\|_{\tl{\mathcal{G}}^{\lm,N}_{s,\la v\ra^\ell}}$;

    \item if the $L^2_{x,v}$ norm $\|\cdot\|_{L^2_{x,v}}$ with weight $\la v\ra^\ell$ on the right hand side of \eqref{def-Gnorm} is replaced by $\|\cdot\|_{\sig,\ell}$ defined in \eqref{L-norm'}, we use the notation $\|\cdot\|_{\mathcal{G}^{\lm,N}_{s,\sig,\ell}}$ to denote the corresponding norm;
    \item if $f$ is independent of $v$, then the weight function $w(t,x,v)$ is taken to be 1, and we use the notation $\|\cdot\|_{\mathcal{G}^{\lm,N}_s}$ to denote the corresponding norm of $f$, namely,  
    \begin{align}\label{def-Gnorm'}
    \|f\|_{\mathcal{G}_{s}^{\lambda,N}}^2
    =\sum_{|\beta|\le N}\kappa^{2|\beta|}\sum_{0\leq m\in \mathbb{N}^6}a_{m,\lm,s}^2(t)\big\|(\nabla_x,t\nabla_x)^{\beta+m}f\big\|_{L^2_{x}}^2.
    \end{align}

    \item We define a variant norm in terms of the Fourier transform in $x$, that is if we write  
    $f(x,v)=\sum_{k\in\Z^3_*}e^{ik\cdot x}f_k(v)$ then we define the Gevrey norm $\bar{\mathcal{G}}_{s, \ell}^{\lambda,N}$ with time-independent radius $\lambda$ by
\begin{align}\label{def-varGnorm}
    \|f\|_{\bar{\mathcal{G}}_{s, \ell}^{\lambda,N}}^2
    &=\sup_{k\in\Z^3_*}\sum_{\beta\in\N^3,|\beta|\leq N} \kappa^{2|\beta|}\sum_{0\leq m\in \mathbb{N}^3}a_{m,\lm,s}^2(0)\big\|\la v\ra^{\ell}(Y^*)^{m+\beta}f_k\big\|_{L^2_{v}}^2.
\end{align}
See Section \ref{sec: notations} for the Fourier analysis conventions. 
\end{itemize}

\subsection{Time-splitting}
We split into two time regions $[0,\nu^{-\fr{1}{2}}]$ and $[\nu^{-\fr12},\infty)$.

The reason for the splitting is that we lack a single argument that can treat the nonlinear collisions in all regimes at once. 
In the first time region, we propagate Gevrey regularity control, and the main aim is to prove uniform Landau damping and enhanced dissipation. However, after a certain point, the nonlinear collision terms need a more precise treatment. 
Unfortunately, the precise treatment we employ would be very hard to reconcile with the Gevrey regularity required to control the Vlasov nonlinearity in the early time regimes.

Hence, at time $t=\nu^{-\fr12}$ (well after the enhanced dissipation has eliminated most of the $x$-dependence), we move to the global time regime, which we approach in Sobolev regularity. At this time, we have the following Sobolev estimates due to the enhanced dissipation
\begin{align*}
    \|f_0(\nu^{-\fr12},v)\|_{H^{2\sigma}}\lesssim \epsilon,\quad \|f_{\neq}(\nu^{-\fr12},x,v)\|_{H^{\sigma+1}}\leq \epsilon \nu^{\ell/36-\sigma/2}.
\end{align*}
Here 
\begin{align*}
    f_0(t,v)=\frac{1}{(2\pi)^3}\int_{\mathbb{T}^3}f(t,x,v)dx,\quad \text{and}\quad f_{\neq}(t,x,v)=f-f_0.
\end{align*}  
For the global time regime, $t\geq \nu^{-\fr12}$, we can consider it as a secondary initial value problem. Let $t'=t-\nu^{-\fr12}$ and $f'(t',x,v)=f(t,x,v)$, then $f'$ solves \eqref{pVPL} with initial data
\begin{align*}
    f'(0,x,v)=f(\nu^{-\fr12},x,v).
\end{align*}
We drop the notation `prime' and keep using the notation $f(t,x,v)$. The global stability problem is now reduced to studying the stability of homogeneous states that are close to the global Maxwellian in a $\nu$-independent manner with inhomogeneous perturbations in Sobolev regularity of size $O(\nu^{\ell/36-\sigma/2})$. 
More precisely, we specifically consider the initial value problem with initial condition $F=\mu+\sqrt{\mu}({\rm I}-\mathrm{P}_0)f_0(\nu^{-\fr12},v)+\sqrt{\mu}\tilde{f}$, where the perturbation $\tilde{f} = f_{\neq}(\nu^{-\fr12},x,v)+\mathrm{P}_0 f_0(\nu^{-1/2},v)$ is of size $O(\nu^{1/3})$ in a weighted Sobolev space (here $\mathrm{P}_0$ is the projection onto the kernel of $L$). 
Hence, this long-time regime is a generalization of the work of Chaturvedi-Luk-Nguyen \cite{chaturvedi2023vlasov} to cover the case of initial data which has a homogeneous component that is close to Maxwellian only in a $\nu$-independent manner (rather than $\nu^{1/3}$ as in \cite{chaturvedi2023vlasov}).  
This result is analogous to some corresponding results in fluid mechanics \cite{BMV2016, bedrossian2024uniform, liMasmoudiZhao2022asymptotic}, although we will need to employ very different methods to treat this case. 
A similar time-splitting was also used in the long-time inviscid limit result \cite{bedrossian2025pseudo,bedrossian2024uniform} (using \cite{bedrossian2025stability} to treat the long times) although again, for somewhat different reasons and with entirely different methods.

\subsection{Landau damping and enhanced dissipation regime: $t \leq \nu^{-1/2}$}
Let us focus on the first time region $t\in [0,\nu^{-\fr12}]$. Due to the absence of $\nu$-dependent smallness, roughly speaking, the main difficulty is coming up with a method that can treat both the plasma echoes and the Landau collision operator at the same time. 
What we prove precisely is the following (See section \ref{sec: notations} for the notation regarding some of the technical quantities): 
\begin{theorem}\label{Thm:shorttime}
    Let $\frac12<s<\fr23$. Let $\lambda_{\mathrm{in}}>0$. There exist $\lambda_{\infty}>0$, $\nu_0>0,\, N_0>0,\, \ell_0>0,\, \kappa_0>0,\, \sigma_0>0$, and $ \epsilon_0=\epsilon_0(s, \lambda_{\mathrm{in}},\lambda_{\infty}, N_0,\ell_0,\kappa_0)>0$ such that for all $N\geq N_0$, $\ell\geq \ell_0$, $0<\kappa<\kappa_0$ and $0<\epsilon<\epsilon_0$, if the initial data $f_{\mathrm{in}}$ satisfies \eqref{conserv-initial} and 
    \begin{align*}
    \|f_{\mathrm{in}}\|_{\mathcal{G}_{s, \ell}^{\lambda_{\mathrm{in}},N}}\leq \epsilon,
    \end{align*}
    then the VPL equation \eqref{pVPL} with all $0<\nu<\nu_0$ has a smooth solution $(f, E)$ defined on $t \in [0,\nu^{-1/2}]$ which satisfies the following uniformly on that time interval, 
    \begin{itemize}
    \item Boundedness of the homogeneous part:
    $$\|f_0(t,v)\|_{\mathcal{G}_{s,\ell}^{\lambda_{\infty},N}}\lesssim \epsilon$$
    \item Enhanced dissipation: 
$$     \|f_{\ne}(t, x+tv, v)\|_{\mathcal{G}_{s,\ell/4}^{\lambda_{\infty},N}}\lesssim \frac{\epsilon}{\langle \nu^{\frac{1}{3}}t\rangle^{\ell/6}},$$
    \item Uniform Landau damping:
$$
\|e^{c_0 \lambda_\infty \abs{t \nabla_x}^s}E(t)\|_{L^2_x}\lesssim \epsilon$$
     for some $c_0>0$ independent of $\nu$. 
    \end{itemize}
\end{theorem}

We now outline the proof. 

\subsubsection{Fourier multipliers vs vector fields}
In order to treat $\nu$-independent perturbations, we need to study the nonlinear interactions carefully. The by-now classical way to deal with the collisionless nonlinear interaction is to study the plasma echoes through a nonlinear Volterra equation written on the frequency side, as in \cite{BMM2016} and \cite{grenier2021landau}, since the critical times are defined in terms of Fourier modes. However, the linear Landau collision operator does not have a good representation on the frequency side due to the inhomogeneity of the collisions in $v$. Thus, to make our energy well-adapted to the Landau collision operator, we quantify Gevrey regularity using a commuting vector field method instead of time-dependent Fourier multipliers as in previous works \cite{BMM2016, BZZ2024VPFP}. 
Specifically by observing that $\nabla_x$ and $Y = \nabla_v + t \nabla_x$ commute with the free transport operator $\partial_t + v \cdot \nabla_x$. 
Note that these derivatives are equivalent to the $\nabla_z,\nabla_v$ derivatives used in e.g. \cite{BMM2016} after the change of variables $x \mapsto z = x + tv$. 

\subsubsection{Loss of velocity localization}
The difficulty of controlling high velocities in the Vlasov-Poisson-Landau equations is a classical issue. We recall two ideas from Yan Guo \cite{Guo2002,guo2012vlasov} to deal with this issue. The first problem is that, due to the degeneracy of the collisions at high velocities, we employ the $\|\cdot\|_{\sigma,\ell}$ norm from \cite{Guo2002}, which we recalled in \eqref{L-norm}  and \eqref{L-norm'}.  

The second problem is the $-E\cdot vf$ term that appears when conjugating the equation with the Maxwellian, which corresponds to the difficulty of controlling the acceleration of particles in the tail by the electric field. For this, we use the idea from \cite{guo2012vlasov} and introduce the new variable
\begin{align}\label{def-g}
g(t,x,v):=e^{\phi}f(t,x,v),
\end{align}
which satisfies the somewhat better equation
\begin{align*}
\pr_tg+v\cdot\nb_xg+E\cdot\nb_vg+\nu Lg= \nu e^{-\phi}\Gamma(g,g)+2e^{\phi}E\cdot v\sqrt{\mu}+\pr_t\phi g.    
\end{align*}
where 
\begin{align}\label{pt-phi}
 \pr_t\phi=-(-\Dl_x)^{-1}\nb_x\cdot M,
\end{align}
with the first moment defined by
$M(t,x):=\int_{\R^3}vf(t,x,v)\mu^{\fr12}(v)dv$. 
This is due to the observations that for $E=-\nb_x\phi$ and $-\Dl_x\phi=\rho$, there holds
\begin{align*}
    e^{\phi}(v\cdot\nabla_xf-E\cdot vf)=v\cdot\nabla_x (e^{\phi}f),
\end{align*}
and 
\begin{align*}
\pr_t\rho=-\nb_x\cdot\int_{\R^3}v f(t,x,v)\sqrt{\mu}dv=-\nb_x\cdot M(t,x).
\end{align*}

\subsubsection{Enhanced dissipation and Gevrey regularity hypocoercivity}
The enhanced collisional dissipation is one of the key stability mechanisms for times $t \gtrsim \nu^{-1/3}$. To obtain this enhanced dissipation, we use a hypocoercivity energy method using vector fields (see e.g. \cite{Coti2020,bedrossian2024taylor} and \cite{chaturvedi2023vlasov} for finite regularity examples). 
Here we adapt this method to Gevrey regularity. 
Let us denote the $(n,\ell)$ level energy by (for universal parameters ${\rm A}_0 > 1$, $\kappa < 1$ to be chosen by the proof)
\begin{align}\label{eq:E_l^n}
\begin{aligned}
    \mathcal{E}_{\ell}^{n}(g(t))=&{\rm A}_0\sum_{|\al|\le1}\left\|\la v\ra^{\ell-2|\al|}\pr_x^{\al}g^{(n)}\right\|_{L^2_{x,v}}^2
    +\sum_{1\leq |\alpha|\leq 2}\left\|\la v\ra^{\ell-2|\alpha|}\kappa^{|\alpha|}\nu^{\frac{|\alpha|}{3}}\pr_v^{\alpha}g^{(n)}\right\|_{L^2_{x,v}}^2\\
    &+2\kappa\nu^\fr13\left\la\nb_{x}g^{(n)},\la v\ra^{2\ell-4}\nb_{v}g^{(n)} \right\ra_{x,v}+\kappa^2 (1+ t)^{-2}\left\|\la v\ra^{\ell-2}Yg^{(n)}\right\|_{L^2_{x,v}}^2.
\end{aligned}
\end{align}
At each fixed level of regularity, this is similar to the standard hypocoercivity schemes, found in e.g. \cite{villani2009hypocoercivity,beck2013metastability,chaturvedi2023vlasov}, with three adjustments. The first is the slight adjustment to the velocity localization based on $x$-derivatives as  in \cite{Guo2002,guo2012vlasov}. The second is the inclusion of $(\nu^{1/3}\partial_v)^2$ (also used in \cite{chaturvedi2023vlasov}) and the third is the inclusion of $\kappa^2 (1+t)^{-2}\left\|\la v\ra^{\ell-2}Yg^{(n)}\right\|_{L^2_{x,v}}^2$, which is to balance the $\nb_x$ derivatives with the vector field $Y$ and to take advantage of the nonlinear transport structure. More precisely, without this term, in the energy estimate, the term $\langle \nabla_x E\cdot \nabla_v g^{(n)}, \nabla_x g^{(n)}\rangle$ loses one derivative in $v$. By adding this new norm to the energy, this term is now under control
\begin{align*}
\langle \nabla_x E\cdot \nabla_v g^{(n)}, \nabla_x g^{(n)}\rangle
=&\langle (1+t)\nabla_x E\cdot Y/(1+t)g^{(n)}, \nabla_x g^{(n)}\rangle\\
&+\left\langle \fr{t}{1+t}\nabla_x E\cdot \nabla_x g^{(n)}, \nabla_x g^{(n)}\right\rangle.
\end{align*}
We refer to \eqref{est-TLH-E32} for detailed estimates. 
Such a norm is not required in \cite{chaturvedi2023vlasov} since the additional $\nu^{1/3}$ smallness implies that it suffices to take  $\nu^{1/3}\nb_v$ instead of $\nb_v$. Since our perturbation size is $\nu$-independent, it is crucial that the norm respects the transport structure in order to control the Vlasov nonlinearity independently of $\nu$. At last, we remark that we use $Y/(1+t)$, rather than $Y$, due to issues with the Landau collision operator. We refer to Remark \ref{rem-Y/t} and \eqref{low-order-Y} for more detailed explanations. Correspondingly, we have the dissipation terms
\begin{align}\label{def-dissipation-nl}
    \nn\mathcal{D}^{n}_\ell(g(t))=&{\rm A}_0\nu^{\fr23}\sum_{|\al|\le1}\left\|\pr_x^\al g^{(n)}\right\|_{\sig,\ell-2|\al|}^2+\kappa\left\|\la v\ra^{\ell-2}\nb_x g^{(n)}\right\|_{L^2_{x,v}}^2\\
    &
    +\sum_{1\leq |\alpha|\leq 2}\nu^{\fr23}\left\|\kappa^{|\alpha|}\nu^{\fr{|\alpha|}{3}}\pr_v^{\alpha}g^{(n)}\right\|_{\sig,\ell-2|\alpha|}^2 
    +\kappa^2\nu^\fr23(1+t)^{-2}\left\|Yg^{(n)}\right\|_{\sig,\ell-2}^2.
\end{align}
We now define the $\ell$-level weighted Gevrey energy with $\ell\geq \ell_0$
\begin{align}\label{energy-l}
    \mathcal{E}_{\ell}(g(t)):=\sum_{m\in\N^6}\sum_{\substack{\beta\in\N^6, |\beta|\le N}}\kappa^{2|\beta|}a_{m,\lm,s}^2(t)\mathcal{E}_{\ell}^{m+\beta}(g(t)),
\end{align}
where the coefficients $a_{m,\lambda,s}$ are defined in \eqref{def:a}. 

We also define the associated dissipation terms,
\begin{align}\label{dissipation-l}
\mathcal{D}_{\ell}(g(t)):=\sum_{m\in\N^6}\sum_{\substack{\beta\in\N^6, |\beta|\le N}}\kappa^{2|\beta|}a_{m,\lm,s}^2(t)\mathcal{D}_{\ell}^{m+\beta}(g(t)),
\end{align}
and the `Cauchy-Kowalewskaya' type terms (CK terms)
\begin{align}\label{CK-l}
    \mathcal{CK}_{\ell}(g(t)):=-2\fr{\dot{\lm}(t)}{\lm(t)}\sum_{m\in\N^6}\sum_{\substack{\beta\in\N^6, |\beta|\le N}}\kappa^{2|\beta|}|m|a_{m,\lm,s}^2(t)\mathcal{E}_{\ell}^{m+\beta}(g(t)).
\end{align}
We also use the notation
\begin{align}\label{var-CK-l}
    \mathfrak{CK}_{\ell}[g(t)]:=-2\fr{\dot{\lm}(t)}{\lm(t)}\sum_{m\in\N^6}\sum_{\substack{\beta\in\N^6, |\beta|\le N}}\kappa^{2|\beta|}|m|a_{m,\lm,s}^2(t)\big\|g^{(m+\beta)}(t)\big\|^2_{L^2_{x,v}(\ell)},
\end{align}
for the sake of presentation. Recalling \eqref{eq:E_l^n}, we have
\begin{align*}
    \mathcal{CK}_{\ell}(g(t))\approx &\sum_{|\al|\le1}\mathfrak{CK}_{\ell-2|\al|}\big[\sqrt{\rm A}_0\pr_x^\al g(t)\big]
    +\sum_{1\le|\al|\le 2}\mathfrak{CK}_{\ell-2|\al|}\big[\kappa^{|\al|}\nu^\fr{|\al|}{3}\pr_v^\al g(t)\big]\\
    &+\mathfrak{CK}_{\ell-2}\big[\kappa(1+t)^{-1}Y g(t)\big].
\end{align*}

We will be trading velocity localization for enhanced dissipation, for which the following lemma is useful. 
\begin{lem}\label{lem: E_l<D_l-2}
    It holds for $\iota= 1, 2, ..., [\ell/3]+1$ and $\alpha\in \mathbb{N}^3$ with $|\alpha|=1,2$ that
    \begin{align}\label{E_l<D_l-2-1}
        \sum_{|\al|\le1}\big\|\pr_x^{\al}g_{\neq}^{(n)}\big\|_{L^2_{x,v}(\ell-2\iota-2|\al|)}^2+\nu^{\frac{2|\alpha|}{3}}\big\|\pr_v^{\alpha}g_{\neq}^{(n)}\big\|_{L^2_{x,v}(\ell-2\iota-2|\alpha|)}^2
        \lesssim \mathcal{D}^{n}_{\ell-2\iota+2}(g_{\neq}(t)).
    \end{align}
    It holds for $\iota=3,4,...,[\ell/3]+1$ that
    \begin{align}\label{E_l<D_l-2-2}
        \nn&\nu^{\fr{\iota}{3}}(1+t)^{\iota-3}\big\|Yg_{\neq}^{(n)}\big\|_{L^2_{x,v}(\ell-2\iota-2)}^2\\
        \lesssim& \nu^{\fr13}\big(\nu^{\fr13}(1+t)\big)^{\iota-3}\mathcal{D}^{n}_{\ell-2\iota+6}(g_{\neq}(t))+\nu^{\fr13}\big(\nu^{\fr13}(1+t)\big)^{\iota-1}\mathcal{D}^{n}_{\ell-2\iota+2}(g_{\neq}(t)). 
    \end{align}
\end{lem}
\begin{proof}
    For $\iota=1, 2, \cdots, [\ell/3]+1$, by Poincar\'e's inequality in $x$ and \eqref{coercive}, recalling \eqref{def-dissipation-nl}, we have
    \begin{align*}
        &\sum_{|\al|\le1}\big\|\pr_x^{\al}g_{\neq}^{(n)}\big\|_{L^2_{x,v}(\ell-2\iota-2|\al|)}^2+\nu^{\frac{2|\alpha|}{3}}\big\|\pr_v^{\alpha}g_{\neq}^{(n)}\big\|_{L^2_{x,v}(\ell-2\iota-2|\alpha|)}^2\\
        \les&\big\|\nb_xg_{\neq}^{(n)}\big\|_{L^2_{\ell-2\iota}}^2+\nu^{\fr{2}{3}}\big\|g_{\neq}^{(n)}\big\|_{{\sig,\ell-2\iota}}^2+\nu^{\fr43}\big\|\nb_vg_{\neq}^{(n)}\big\|_{{\sig,\ell-2\iota-2}}^2\les \mathcal{D}^{n}_{\ell-2\iota+2}(g_{\ne}(t)),
    \end{align*}
    which gives \eqref{E_l<D_l-2-1}.
For $\iota=3, 4, \cdots, [\ell/3]+1$, using  \eqref{coercive} again,
\begin{align*}
    &\nu^{\fr{\iota}{3}}(1+t)^{\iota-3}\big\|Yg_{\neq}^{(n)}\big\|_{L^2_{x,v}(\ell-2\iota-2)}^2\\
    \les& \nu^{\fr13}\big(\nu^{\fr13}(1+t)\big)^{\iota-3}\Big(\nu^{\fr23}\big\|\nb_vg_{\neq}^{(n)}\big\|_{L^2_{x,v}(\ell-2\iota-2)}^2\Big)+\nu^{\fr13}\big(\nu^{\fr13}(1+t)\big)^{\iota-1}\big\|\nb_xg_{\neq}^{(n)}\big\|_{L^2_{x,v}(\ell-2\iota-2)}^2\\
    \les& \nu^{\fr13}\big(\nu^{\fr13}(1+t)\big)^{\iota-3}\mathcal{D}^{n}_{\ell-2\iota}(g_{\ne}(t))+\nu^{\fr13}\big(\nu^{\fr13}(1+t)\big)^{\iota-1}\mathcal{D}^{n}_{\ell-2\iota}(g_{\ne}(t)).
\end{align*}
    Then \eqref{E_l<D_l-2-2} follows immediately.
\end{proof}

We now define the total energy,  incorporating both the unweighted  lowest-order energy component and the enhanced dissipation effects:
\begin{align}\label{total-en}
    \mathcal{E}(g(t))={\rm B}_0\|g(t)\|^2_{L^2_{x,v}}
    +\mathcal{E}_{\ell}(g(t))
    +\sum_{\iota=1}^{[\ell/3]+1} w_{\iota}^2(t)\mathcal{E}_{\ell-2\iota}(g_{\neq}(t)),
\end{align}
where $w_{\iota}(t):=\big(\kappa_0\nu^{\fr13}(1+t)\big)^{\fr{\iota}{2}}$.
We also define the associated total dissipation term and CK terms
\begin{align}
    \label{total-dis}&\mathcal{D}(g(t))={ \nu^{\fr23}\dl{\rm B_0}\|({\rm I}-{\rm P}_0)g(t)\|^2_{\sig}}+\mathcal{D}_{\ell}(g(t))+\sum_{\iota=1}^{[\ell/3]+1}w_{\iota}^2(t)\mathcal{D}_{\ell-2\iota}(g_{\neq}(t)),\\
    \label{total-CK}&\mathcal{CK}(g(t))=\mathcal{CK}_{\ell}(g(t))+\sum_{\iota=1}^{[\ell/3]+1}w_{\iota}^2(t)\mathcal{CK}_{\ell-2\iota}(g_{\neq}(t)).
\end{align}
Due to the time-weight $\big(\kappa_0\nu^{\fr13}(1+t)\big)^{\iota}$, there are `bad terms' when the time-derivatives hit the time-weight, which by Lemma \ref{lem: E_l<D_l-2} can be controlled by the dissipation terms, namely, 
\begin{align*}
&\quad{|\text{bad terms}|}\\
&\leq
    C\kappa_0\nu^{\frac13}\sum_{\iota=1}^{[\ell/3]+1}\iota\big(\kappa_0\nu^{\fr13}(1+t)\big)^{\iota-1}\Big(\sum_{|\al|\le1}\big\|\pr_x^{\al}g_{\neq}^{(n)}\big\|_{L^2_{\ell-2\iota-2|\al|}}^2+\nu^{\frac{2}{3}}\big\|\nb_vg_{\neq}^{(n)}\big\|_{L^2_{\ell-2\iota-2}}^2\Big)\\
    &\quad +C\sum_{\iota=3}^{[\ell/3]+1}(\iota-2)\kappa_0^{\iota}\nu^{\fr{\iota}{3}}(1+t)^{\iota-3}\big\|Yg_{\neq}^{(n)}\big\|_{L^2_{\ell-2\iota-2}}^2\\
    &\leq C\kappa_0 \nu^{\fr13}\sum_{\iota=0}^{[\ell/3]} \big(\kappa_0\nu^{\fr13}(1+t)\big)^{\iota}\mathcal{D}_{\ell-2\iota}^n(g_{\neq}(t)),
\end{align*}
See more details in \eqref{dt-(1+t)}.
Note that by taking the projection to the non-zero modes, for $\iota\geq 1$, we cannot keep the transport structure, namely, the interaction  $E\cdot\nabla_vf_{0}$ will lead to derivative losses. 
To overcome this, we need to take advantage of the fast decay of $E$ in lower regularity and the extra $\nu^{\frac{\iota}{3}}$ smallness, namely, we use the following estimates for $W\in \{\nabla_x, \nu^{\fr{|\alpha|}{3}}\nabla_v^{\alpha}, Y/(1+t)\}$ with $\alpha\in \mathbb{N}^3$ and $1\leq |\alpha|\leq 2$,
\begin{align*}
    &\left|\nu^{\fr{\iota}{3}}t^{\iota}\langle E\cdot \nabla_vWZ^{m} g_0, \langle v\rangle^{2\ell-4\iota-4}WZ^{m}g_{\neq}\rangle\right|\\
    \lesssim & \langle t\rangle^{\iota+4}\|E\|_{H^5}\left\|\nu^{\fr12}\nabla_vWY^{m}g_0\right\|_{L^2}^{\fr{1-s}{1-s/2}}\mathcal{CK}_{\ell}(g_0(t))^{\fr{s/2}{2-s}}\mathcal{CK}_{\ell}(g_{\neq}(t))^{\fr12}+\text{lower order terms}.
\end{align*}
Here we use the fact that $\nu^{\frac{1}{2}\frac{1-s}{1-s/2}}\ge \nu^{\fr13}\ge \nu^{\iota/3}$ for $s>\fr12$ and $\iota\geq 1$, see more details in \eqref{est-Eg0-LH-iota}. We do not view this as a fundamental place where $s>1/2$ is used as there is likely a more elegant solution. 

\subsubsection{Propagation of Gevrey regularity for the linear Landau equation}
The classical approach to obtaining the density estimate relies on a Volterra-type equation. A crucial step in this analysis is the study of the following linearized Landau equation:
\begin{align}\label{eq:linearized 1}
    \partial_tf+v\cdot \nabla_x f+\nu Lf=0,
\end{align}
and establishing two key properties of the semigroup $S(t)$ (or $S_k(t)$ --- its Fourier transform in $x$): the enhanced dissipation estimate and the propagation of vector field Gevrey regularity. In \cite{chaturvedi2023vlasov}, the enhanced dissipation estimate was derived along with the propagation of Sobolev regularity. The main technique employed was hypocoercivity, leveraging the dissipative properties of the linear Landau operator, as studied in \cite{Guo2002, guo2012vlasov}, within Sobolev spaces. 
In this work, we extend this analysis to the Gevrey regularity setting. We refer to section \ref{sec: Linear Landau operator} for more details of the linear Landau operator in Gevrey class and section \ref{sec: Linear estimates} for the estimates of the semigroup $S(t)$. 
The density can now be written in terms of a Volterra equation (see also \eqref{rho_k}-\eqref{N_k})
\begin{align}
\hat{\rho}(t,k) + \int_0^t K_k(t-\tau) \hat{\rho}(\tau,k) d\tau =N_a+ N_b +N_c,
\end{align}
with 
\begin{align*}
K_k(t)=\fr{2ki}{|k|^2}\int_{\R^3}S_k(t)[v\mu^{\fr12}(v)]\mu^{\fr12}(v) dv.
\end{align*}
The non-homogeneous terms are given by
\begin{itemize}
    \item the initial contribution
    \begin{align*}
        N_a=\int_{\mathbb{R}^3}S_k(t)[\hat{f}_{\rm in}]\sqrt{\mu}(v)dv,
    \end{align*}
    \item the collisionless effects
    \begin{align*}
        N_b = \int_0^t\int_{\mathbb{R}^3}S_k(t-\tau)\mathcal{F}_{x}[E\cdot\nabla_vf]_k(\tau, v)\sqrt{\mu}(v)dvd\tau,
    \end{align*}
    \item and the collisional effects
    \begin{align}
        N_c= \nu \int_0^t\int_{\mathbb{R}^3}S_k(t-\tau)\mathcal{F}_{x}[\Gamma(f,f)]_k(\tau, v)\sqrt{\mu}(v)dvd\tau. \label{def:N2}
    \end{align}
\end{itemize}
After obtaining good properties on the semigroup and using the Penrose condition as in the collisionless case, the Volterra equation can be solved, and the density estimate (which is the primary difficulty) is reduced to getting good space-time estimates on these two key nonlinear interactions; see Proposition \ref{prop:kernel G} for details. 

\subsubsection{Plasma echoes and their control in Gevrey class on the physical side}
The main difficulty in estimating $N_b$ is caused by the plasma echoes, which are nonlinear effects localized in time-frequency space; see \cite{BMM2016,bedrossian2020nonlinear} for more in-depth discussions of the effect. 
As such, the previous studies \cite{BMM2016,grenier2021landau,ionescu2024LandauDampingsharp} rely on the techniques developed for expanding the term in the Fourier transform and studying the time-dependent Fourier multiplier $e^{\lambda(t)|k,kt|^s}$. To control the plasma echoes, a slowly decaying $\lambda(t)$ is chosen (see \eqref{def-lm}), for some small $a$, and then the localization of the plasma frequencies to very specific times as a function of frequency means that 
\begin{align}\label{eq:e/e}
    e^{(\lambda(t)-\lambda(\tau))|k,k\tau|^s}\lesssim e^{-c_0\frac{t-\tau}{\tau^{1+a}}|k,k\tau|^s},\quad\text{for}\quad t/2\leq \tau\leq t
\end{align}
can gain a significant amount of regularity in the nonlinear estimate. This was introduced in \cite{BMM2016} and is also used in \cite{grenier2021landau, ionescu2024LandauDampingsharp} (the method in \cite{MouhotVillani2011} instead sheds this regularity per iteration in the Newton scheme). 
However, due to the Landau collision operator, such a Fourier multiplier is not well-adapted to the VPL equations. We instead use the Gevrey class defined on the physical side. Thus, for the $m$-th derivative $(k,kt)^m$, we have the ratio $\Big(\frac{\lambda(t)}{\lambda(\tau)}\Big)^{|m|}$. The new idea is that instead of achieving time decay directly through the Fourier multiplier \eqref{eq:e/e}, we gain $|m|$, which corresponds to regularity. Since the regularity $(k,kt)^m$ is linked to time $t$, obtaining additional regularity leads to extra decay. We refer to \eqref{lm-gap1}, \eqref{lm-gap2}, \eqref{eq:gain m}, \eqref{lm-gap3}, and \eqref{lm-gap5} for more details and the estimates that follow illustrate how they are used. 

Regardless of the norm we choose however, we do need to estimate the plasma echo contributions via the Fourier transform. 
The Landau collision mixes the frequencies of $v$-variables, and as a consequence, how the semigroup $S_k(t)$ moves the frequencies is unclear. 
Hence, we introduce the dual operator $S_k^*(t)$. 
This allows us to write the nonlinear contribution on the Fourier side as  
\begin{align*}
    \sum_{l\in \mathbb{Z}^3}\int_{\mathbb{R}^3}i\widehat{E}_l(\tau)\cdot\xi\hat{f}_{k-l}(t,\xi)\bar{\hat{\mathfrak{g}}}_k(t-\tau, \xi)d\xi
\end{align*}
where $\hat{\mathfrak{g}}_k(t-\tau, \xi)$ is the Fourier transform of $S_k^*(t-\tau)[\sqrt{\mu}]$ which behaves well enough, namely we prove, 
\begin{align*}
    |\hat{\mathfrak{g}}_k(t-\tau, \xi)|\lesssim e^{-c_0|\xi-k(t-\tau)|^{s_L}}e^{-c_0(\nu^{\fr13}(t-\tau))^{\fr{1}{3}}}. 
\end{align*} 

\subsubsection{Collision effects and the density decomposition}\label{sec:density decomposition}
The control of the collisional nonlinearity is much harder, despite the extra $\nu$-dependent smallness. To determine what kind of control is required, we need to trace back to the energy estimate of the distribution and study the reaction term, more precisely, for $W\in \{\partial_x, (1+t)^{-1}Y\}$ (ignoring the velocity weights for now)
\begin{align*}
    \left\langle \partial_x(\partial_x,t\partial_x)^{m}E\cdot \nabla_v f_{\neq}, WY^mf\right\rangle_{x,v}\quad \text{and}\quad
    \left\langle \partial_x(\partial_x,t\partial_x)^{m}E\cdot \nabla_v f_{0}, WY^mf\right\rangle_{x,v}.
\end{align*}
One can formally rewrite them on the Fourier side,  
\begin{align*}
    I_{\neq}=\sum_{k\in\Z^3}\sum_{l\in\Z^3_*,l\neq k}\int_{\mathbb{R}^3} l(l,lt)^{m}\hat{E}_l(t)\cdot \eta \hat{f}_{k-l}(\eta)\overline{\widehat{Wf^{(m)}}_k(t,\eta)}d\eta,
\end{align*}
and
\begin{align*}
    I_0=\sum_{k\neq 0}\int_{\mathbb{R}^3} k(k,kt)^{m}\hat{E}_k\cdot \widehat{\nabla f}_{0}(\eta)\overline{\widehat{Wf^{(m)}}_k(t,\eta)}d\eta. 
\end{align*}
For the first term, we focus on the worst-case scenario: $|k-l|t\approx |\eta|\leq \frac{1}{100}|lt|$. 
Notice that for $t\leq \nu^{-\frac12}$, the `CK' term is stronger than the enhanced dissipation as long as $s>\fr13$, namely, 
\begin{align*}
    \frac{\langle lt\rangle^{s/2}}{\langle t\rangle^{\fr{1+a}{2}}}\geq |l|^{s/2}\nu^{\frac{1}{4}(1+a-s)}\geq\nu^{\fr16}.
\end{align*}
In order to integrate the contribution of $I_{\neq}$, it heuristically makes sense to attempt 
\begin{align*}
\norm{I_{\neq}}_{L^1_t} \lesssim \epsilon\norm{\frac{\abs{\nabla}^{s/2}}{\langle t \rangle^{\frac{1+a}{2}}} \langle \nabla \rangle^m f}_{L^2_t L^2_{x,v}} \left\|\frac{1}{|l|}\frac{\langle t\rangle^{\fr{1+a}{2}}}{\langle lt\rangle^{s/2}}\langle t\rangle l(l,lt)^{m}\hat{\rho}_l\right\|_{L^2_tL^2_l},   
\end{align*}
(of course the real estimates are  harder). 
The factor involving $\rho$ is hence the kind of quantity we really need to control, as the first factor will be estimated using the CK terms in the energy estimate on $f$. 

The main problem in obtaining the required estimate on $\rho$ is the collision term $N_c$ (recall \eqref{def:N2}) in which the low-high interaction with the zero mode is the worst, namely, we would need to estimate the following quantity
\begin{align*}
    \left\|\frac{\langle t\rangle^{\frac{3+a-s}{2}}}{|k|^{1+\fr{s}{2}}} k(k,kt)^{m}\nu \int_0^t\int_{\mathbb{R}^3}S_k(t-\tau)\partial_{v_i}\mathcal{F}_{x}[\Phi_{ij}\ast(\sqrt{\mu}f_{0,\text{low}})\partial_{v_j}f_{\neq,\text{high}})]_k(\tau)\sqrt{\mu}dvd\tau \right\|_{L_t^2l_k^2}, 
\end{align*}
where `High' and `Low' refer to a relative frequency decomposition in the manner of a para-product. 
By summing up all $m$ level estimates, we can obtain Gevrey regularity. Let us formally write everything in terms of Fourier multipliers (although this will not be possible rigorously), and so if we use the duality argument mentioned above and suitable estimates on the convolution with $\Phi_{ij}$ (see \eqref{eq:Fourier-Phi} below), then we would need to estimate 
\begin{align*}
    \left\|\frac{\langle t\rangle^{\frac{3+a-s}{2}}}{|k|^{\fr{s}{2}}} e^{\lambda(t)|k,kt|^{s}}\nu \int_0^t\int_{\mathbb{R}^6}\mathbf{1}_{|\eta-\xi|\leq\fr{|\xi|}{10}}\fr{\xi\eta}{|\xi-\eta|^2}\widehat{(\sqrt{\mu}f_{0})}(\xi-\eta)\hat{f}_{k}(\tau,\eta)\hat{\mathfrak{g}}_k(\xi) d\eta d\xi d\tau \right\|_{L_t^2l_k^2}. 
\end{align*} 
The worst contribution is $|\xi-\eta|\leq \fr{|\xi|}{10}\approx |k(t-\tau)|\leq \frac{1}{1000}|kt|$. By the heuristic \eqref{eq:e/e}, we should mainly only need to focus on the case $|k||t-\tau|\approx |\xi|\lesssim |k|\frac{\langle t\rangle^{1+a}}{|k t|^s}$. This contribution then becomes the following
\begin{align*}
    &\bigg\|\frac{\langle t\rangle^{\frac{3+a-s}{2}}}{|k|^{\fr{s}{2}}} \nu \int_{\fr{9t}{10}}^t\int_{\mathbb{R}^6}e^{-c_0\frac{|\xi|}{|k|\tau^{1+a}}|k,k\tau|^s}\mathbf{1}_{|\eta-\xi|\leq\fr{|\xi|}{10}}\fr{\xi\eta}{|\xi-\eta|^2}e^{c\lambda(\tau)|\xi-\eta|^s}\widehat{(\sqrt{\mu}f_{0})}(\xi-\eta)\\
    &\quad\quad\quad \quad \quad 
    \times \frac{\langle \tau\rangle^{\frac{1+a}{2}}}{\langle k\tau\rangle^{\frac{s}{2}}}\frac{\langle \eta+k\tau\rangle^{\frac{s}{2}}}{\langle \tau\rangle^{\frac{1+a}{2}}}e^{\lambda(\tau)|k,\eta+k\tau|^s}|\hat{f}_{k}(\tau,\eta)|\frac{(|\xi|/|k|)^{1+\varepsilon}}{\langle t-\tau\rangle^{1+\varepsilon}}\\
    &\quad\quad\quad\quad \quad \times e^{\lambda(\tau)|\xi-k(t-\tau)|^s}\hat{\mathfrak{g}}_k(\xi) d\eta d\xi d\tau\bigg\|_{L^2_tl_k^2}\\
    &\lesssim \epsilon \left\|\frac{|\nb_x,Y|^{\fr{s}{2}}}{\langle t\rangle^{\fr{1+a}{2}}} \nu t^{2+a-s+2(1+a-s)+(1+\epsilon)(1+a-s)} f_{\neq}\right\|_{L_t^2L^2_{x,v}}. 
\end{align*}
In order to hope to benefit from the enhanced dissipation estimates, we would need $(4+\varepsilon)(1+a-s)+1\leq 3$ for $a, \varepsilon>0$ small enough, which produces the requirement $s>\fr{1}{2}$. This is where we fundamentally use the Gevrey-$2_-$ regularity requirement. The above estimates are quite different from the density estimates used to estimate $N_b$, i.e., the Vlasov nonlinearity. 
Hence, its natural to decompose the density into two contributions, 
\begin{align}
    \rho=\rho^{(1)}+\rho^{(2)},
\end{align}
where $\rho^{(1)}$ captures mainly collisionless contributions,  while $\rho^{(2)}$ captures mainly collisional contributions. We may expect that $\rho^{(1)}$ behaves the same as VP which decays faster than $\rho^{(2)}$. More precisely, by applying the semi-group $S_k$, we have 
\begin{align*}
    (\mathrm{Id+\mathfrak{L}_t^{app}})\rho=\mathcal{N}(t)
\end{align*}
where $\mathcal{N}(t)$ represents non-homogeneous terms and 
\begin{equation}\label{eq:Fourier of Lrho}
    \widehat{\mathfrak{L}_t^{\mathrm{app}}\rho}_k=\int_0^tK_k(t-\tau)\hat{\rho}_k(\tau)d\tau,
\end{equation}
with 
\begin{align*}
K_k(t)=\fr{2ki}{|k|^2}\int_{\R^3}S_k(t)[v\mu^{\fr12}(v)]\mu^{\fr12}(v) dv.
\end{align*}
We refer to \eqref{rho_k} for more details. By the Penrose stability condition, we obtain that $(\mathrm{Id+\mathfrak{L}_t^{app}})$ is invertible, we refer to Proposition \ref{prop:kernel G} for more details. The above-mentioned decomposition is given by
\begin{align}
\hat{\rho}^{(1)}_k(t)=\mathcal{N}^{(1)}_k(t)+\int_0^tG_k(t-\tau)\mathcal{N}^{(1)}_k(\tau)d\tau,
\end{align}
and 
\begin{align}
\hat{\rho}^{(2)}_k(t)=\mathcal{N}^{(2)}_k(t)+\int_0^tG_k(t-\tau)\mathcal{N}^{(2)}_k(\tau)d\tau,
\end{align}
respectively, where $G$ is the solution kernel to the Volterra equation (given in Proposition \ref{prop:kernel G}), and (using the notation defined in Section \ref{sec-Bony} for the paraproduct decompositions)
\begin{align*}
    \mathfrak{N}^{(1)}(t,x,v)=&-E_1(t,x)\cdot\nb_vf-\mathcal{T}'_{E_2^j}\pr_{v_j}f+E_1(t,x)\cdot v f+\mathcal{T}'_{E_2^j}v^jf ,\\
    \mathfrak{N}^{(2)}(t,x,v)=&-\mathcal{T}_{\pr_{v_j}f}E^j_2+\mathcal{T}_{v_jf}E^j_2+\nu\Gamma(f,f),\\
    \mathcal{N}_k^{(1)}(t)=&\int_{\R^3}S_k(t)[\hat{f}_{\rm in}(v)]\mu^{\fr12}(v)dv+\int_0^t\int_{\R^3}S_k(t-\tau)[\hat{\mathfrak{N}}^{(1)}_k(\tau,v)]\mu^{\fr12}(v)dvd\tau,\\
    \mathcal{N}_k^{(2)}(t)=&\int_0^t\int_{\R^3}S_k(t-\tau)[\hat{\mathfrak{N}}^{(2)}_k(\tau,v)]\mu^{\fr12}(v)dvd\tau,
\end{align*}
and $E_i=-\nb_x(-\Dl_x)^{-1}\rho^{(i)}$, $i=1,2$. 
We also note that to close the estimate, we need the enhanced dissipation weight $\nu t^3$, which will introduce a $\langle v\rangle^{12}$ loss of velocity localization. 
The enhanced dissipation, of course, cannot help in the treatment of $I_0$; however, this term does not have the additional power of $t$ loss. That is, we only need to estimate the following 
\begin{align*}
    &\left\|\frac{1}{|k|}\frac{\langle t\rangle^{\fr{1+a}{2}}}{\langle kt\rangle^{s/2}} (\nu^{\fr13}t)^{\fr{\ell'}{2}} k(k,kt)^{m}\widehat{\rho}_k\right\|_{L_t^2l_k^2}\\
    &\lesssim  \bigg\|\frac{\langle t\rangle^{\frac{1+a-s}{2}}}{|k|^{1+\fr{s}{2}}} \nu \int_{\fr{9t}{10}}^t\int_{\mathbb{R}^6}
    e^{-c_0\frac{|\xi|}{|k|\tau^{1+a}}|k,k\tau|^s}\mathbf{1}_{|\eta-\xi|\leq\fr{|\xi|}{10}}\fr{\xi\eta}{|\xi-\eta|^2}e^{c\lambda(\tau)|\xi-\eta|^s}\widehat{(\sqrt{\mu}f_{0})}(\xi-\eta)\\
    &\quad\quad\quad \quad \quad 
    \times \frac{\langle \tau\rangle^{\frac{1+a}{2}}}{\langle k\tau\rangle^{\frac{s}{2}}}\frac{\langle \eta+k\tau\rangle^{\frac{s}{2}}}{\langle \tau\rangle^{\frac{1+a}{2}}}e^{\lambda(\tau)|k,\eta+k\tau|^s} (\nu^{\fr13}t)^{\fr{\ell'}{2}} |\widehat{f}_{k}(\tau,\eta)|\frac{(|\xi|/|k|)^{1+\epsilon}}{\langle t-\tau\rangle^{1+\epsilon}}\\
    &\quad\quad\quad\quad \quad \times e^{\lambda(\tau)|\xi-k(t-\tau)|^s}\mathfrak{g}_k(\xi) d\eta d\xi d\tau\bigg\|_{L^2_tl_k^2} + \text{easy terms} \\
    &\lesssim \epsilon \left\|\frac{|\nb_x,Y|^{\fr{s}{2}}}{\langle t\rangle^{\fr{1+a}{2}}} \nu t^{(4+\epsilon)(1+a-s)}(\nu^{\fr13}t)^{\fr{\ell'}{2}} f_{\neq}\right\|_{L_t^2L^2_{x,v}}+ \text{easy terms}.
\end{align*}
Thus we need $\nu t^{(4+\epsilon)(1+a-s)}\leq \nu t^2\leq 1$, namely, $t\leq \nu^{-\fr12}$. This is why we obtain the Gevrey control only for $t\in [0,\nu^{-\fr12}]$. We refer to Proposition \ref{prop-rho2-iota} for more details.

\subsubsection{Bootstrap argument for the proof of Theorem \ref{Thm:shorttime}}
Let us now precisely outline the bootstrap argument that completes the proof of Theorem \ref{Thm:shorttime}.   We define the following controls referred to in the sequel as \emph{the bootstrap hypotheses} for $0\leq t\leq T\le\nu^{-\fr12}$:
\begin{align}
\label{bd-en}&\sup_{t\in [0,T]}\mathcal{E}(g(t))
+\fr14\nu^{\fr13}\int_0^T\mathcal{D}(g(t))dt
+\int_0^T\mathcal{CK}(g(t))dt\leq 8{\rm C}_g\epsilon^2,\\
        \label{H-rho1}&\big\|\langle t\rangle^b|\nabla_x|^{\fr32}(\rho^{(1)}, M^{(1)})\big\|^2_{L^2_t\mathcal{G}_s^{\lambda,N-1}}\le 8{\rm C}_{0} \epsilon^2,\\
        \label{bd-rho2-1}&\big\|\langle t\rangle^{\frac{3+a-s}{2}}(\rho^{(2)},M^{(2)})\big\|^2_{L^2_t\mathcal{G}_s^{\lambda,N}}\le 8\epsilon^2,\\
        \label{bd-rho2-2}&\big\|\langle t\rangle^{\frac{1+a-s}{2}}w_{\iota}(t)(\rho^{(2)},M^{(2)})\big\|^2_{L^2_t\mathcal{G}_s^{\lambda,N}}\le 8\epsilon^2, \quad{\rm for}\quad\iota=1, 2,\cdots, [\ell/3]+1,\\
        \label{H-phi}& \sup_{t\in[0,T]}\|\phi(t)\|_{\mathcal{G}^{\lm,N}_s}\le\fr12,
\end{align}
for universal constants $b$, ${\rm C}_{0}$, and ${\rm C}_g$ chosen by the proof. 
\begin{prop}\label{prop:shorttime}
Let  $M_\mu, M_\sig$ and $\dl$ be the constants appearing in \eqref{e-prmu}, \eqref{eq: est sigma_ij} and \eqref{coercive-0}, respectively, and $\ell, b\in \N$ satisfy
\begin{align}\label{res-final-1}
    \ell>18, \quad b\ge \fr{[\ell/3]}{2}+2, \quad N\ge \max\{2b+12, 2[\ell/3]+8\},
\end{align}
and $s, a\in\R$ satisfy
\begin{align}\label{res-final-2}
    \fr12<s<\fr23,\quad 0<a<s-\fr12,\quad{\rm and}\quad a\le \fr{s}{b+4}.
\end{align}
Assume that the estimates \eqref{bd-en}--\eqref{H-phi} hold for all $t\in [0,T]$ with $T\leq \nu^{-\fr{1}{2}}$. 
 Then there exist positive constants $\underline{\lm}, \underline{\kappa}, \nu_0$ and $\eps_0$ small enough, and ${\rm A}_0, {\rm B}_0$ large enough, depending on $M_\mu, M_\sig, s, N$, $\delta$, and $\ell$, such that for all $\lm(t), \kappa, \kappa_0$ and $\nu$  satisfying
 \begin{equation}\label{restriction: lm-kappa-nu}
     \lm(0)\le \underline{\lm},\quad \kappa\le\underline{\kappa}, \quad \nu\le\nu_0,  \quad \quad \kappa_0\ll\fr{\kappa}{\ell {\rm A}_0},
 \end{equation}
and $0<\epsilon<\epsilon_0$,  these same estimates hold with all the occurrences of 8 on
the right-hand side of \eqref{bd-en}--\eqref{bd-rho2-2} replaced by 4, and the occurrence of $\fr12$ on
the right-hand side of \eqref{H-phi} replaced by $\fr14$. As a consequence, the above estimates hold until  $T=\nu^{-\frac{1}{2}}$. 
\end{prop}

Before proceeding any further, let us outline the following two estimates that follow immediately from the bootstrap hypotheses \eqref{bd-en}--\eqref{bd-rho2-2}.

\begin{prop}
Assume that \eqref{res-final-1} and \eqref{res-final-2} hold.    Under the bootstrap hypotheses \eqref{bd-en}--\eqref{bd-rho2-2}, for $t\in[0,\nu^{-\fr12}]$, there hold 
    \begin{align}
    &\label{bd-rho-1}
        \big\|\la t\ra^{\fr{3+a-s}{2}}(\rho, M)\big\|_{L^2_t\mathcal{G}^{\lm,N}_s}\les \eps,\\
    &\label{bd-rho-2}
    \big\|\langle t\rangle^{\frac{1+a-s}{2}}w_{\iota}(t)(\rho,M)\big\|_{L^2_t\mathcal{G}_s^{\lambda,N}}\les \eps, \quad{\rm for}\quad\iota=1, 2,\cdots, [\ell/3]+1,.
    \end{align}
\end{prop}

The first part of this paper is to prove the improvement of \eqref{bd-en}---\eqref{H-phi}. We refer to section \ref{sec: est f} for the improvement of \eqref{bd-en}, refer to section \ref{sec: density est} for the improvement of \eqref{H-rho1}, \eqref{bd-rho-1}, and \eqref{bd-rho-2}, and refer to section \ref{sec: relationship} for the improvement of \eqref{H-phi}.

\subsection{The collisionless limit}
By taking $\nu=0$, our proof works directly for the VP equation. Let $(f^{(\nu)}, E^{(\nu)})$ and $(f^{(0)}, E^{(0)})$ be the solutions to the VPL equation and VP equation with the same initial data respectively. We then have $f^{(\delta)}=f^{(\nu)}-f^{(0)}$ and $E^{(\delta)}$ solve 
\begin{align}\label{VPL-VP}
\begin{cases}
\pr_tf^{(\delta)}+v\cdot\nb_xf^{(\delta)}+ E^{(\nu)}(t,x)\cdot\nb_vf^{(\delta)}+E^{(\delta)}\cdot \nabla_vf^{(0)}-E^{(\nu)}\cdot v f^{(\delta)}\\
-E^{(\delta)}(t,x)\cdot v f^{0}
-2E^{(\delta)}(t,x)\cdot v\mu^\fr12+\nu Lf^{\delta}=-\nu Lf^{(0)}+\nu\Gamma(f^{\nu},f^{\nu}),\\[3mm]
\displaystyle E^{(\delta)}(t,x)=-\nb_x\phi^{(\delta)}(t,x), \quad -\Dl_x\phi^{(\delta)}=\int_{\R^3}f^{(\delta)}(t,x,v)\mu^\fr12(v)dv=\rho^{(\delta)}(t,x).
\end{cases}
\end{align}
The collisionless limit can be proved by following the same argument as in \cite{BZZ2024VPFP} with the fact that $\nb_v=Y-t\nb_x$ which implies that
\begin{align*}
    &\|\nu L f^{0}\|_{L^2}\lesssim \epsilon \nu \langle t\rangle^2, \quad
    \|\nu \Gamma(f^{\nu},f^{\nu})\|_{L^2}\lesssim \epsilon^2\nu  \langle t\rangle^2,\quad
    \|\langle \nb_x,Y\rangle^{N}\nb_vf^{(0)}\|_{L^2}\lesssim \epsilon \langle t\rangle.
\end{align*}
A weaker Gevrey regularity estimate involving the vector field $Y$ follows similarly. 
The last inequality is to treat $E^{(\delta)}\cdot \nabla_vf^{(0)}$ and the linear growth is absorbed by the Landau damping of $E^{(\delta)}$, which is similar to the treatment of high-low interactions in Lemma \ref{lem-transport}. 
We omit the details for brevity, leaving them to interested readers.

\subsection{Long-time relaxation regime: $t \geq \nu^{-1/2}$}

\subsubsection{Quasi-linearization}
By the time $t=\nu^{-\fr12}$, 
the enhanced dissipation ensures $\nu$-dependent smallness of the non-zero modes in lower regularities, while we only have $\nu$-independent smallness of the zero mode (i.e., it is still $\approx \eps$). 
To handle this regime, we will need to keep track of this regularity discrepancy.  
The main idea is to use a quasi-linearization around $f_0$, treating the non-zero modes essentially as a small perturbation of $\mu + f_0$.  
Here, for $t\geq \nu^{-\frac{1}{2}}$ we let $f(t, x, v)=f_0^L(t-\nu^{-\fr12},v)+\tilde{f}(t-\nu^{-\fr12},x,v)$ where $(f_0^L(t,v), \tilde{f}(t,x,v))$ solves
\begin{align}\label{eq:quasi-linearized}
    \left\{
    \begin{aligned}
        &\partial_{t}f_0^L+\nu Lf_0^L-\nu \Gamma (f_0^L,f_0^L)=0,\\
        &\partial_{t}\tilde{f}+v\cdot \nabla_x\tilde{f}+\tilde{E}\cdot\nb_v \tilde{f}-\tilde{E}\cdot v\tilde{f}-2\tilde{E}\cdot v\mu^{\fr12}+\tilde{E}\cdot \nabla_v f_0^{L}-\tilde{E}\cdot v f_0^L
        +\nu L \tilde{f}\\
        &\quad\quad\quad\quad\quad\quad\quad 
        =\nu \Gamma(\tilde{f},f_0^L)+\nu \Gamma(f_0^L,\tilde{f})+\nu \Gamma(\tilde{f},\tilde{f}),\\
        &\tilde{E}=-\nabla_x\tilde{\phi}=-\nb_x(-\Delta_x)^{-1}\tilde{\rho}=-\nb_x(-\Delta_x)^{-1}\int_{\mathbb{R}^3}\tilde{f}(t,x,v)\sqrt{\mu}(v)dv,\\
        &f_0^{L}(0,v)=({\rm Id}-\mathrm{P}_0)f_{0}(\nu^{-\fr12},v),\quad \tilde{f}(0,x,v)=f_{\ne}(\nu^{-\fr12},x,v)+\mathrm{P}_0f_0(\nu^{-1/2},v).
    \end{aligned}
    \right.
\end{align}
Recall that here $\mathrm{P}_0$ is the projection to the subspace ${\rm span} \{\mu^{\fr12}, v_i\mu^{\fr12}, (|v|^2-3/2)\mu^{\fr12}\}$ with $i=1,2,3$. Moreover, by the conservation laws \eqref{conser-mass}--\eqref{conser-en} and \eqref{conserv-initial}, one deduces that
\begin{align}\label{eq: P_0f_0}
    \mathrm{P}_0f_0(\nu^{-\fr12},v)=-\fr{\|E(\nu^{-\fr{1}{2}})\|_{L^2_x}^2}{\int_{\R^3}\big[(|v|^2-\fr32)\mu^{\fr12}\big]^2dv}(|v|^2-\fr{3}{2})\mu^{\frac{1}{2}}(v),
\end{align}
which note is extremely small due to the Landau damping. 
\begin{rem}\label{Rmk: P_0f_0}
     It holds that for all $N'\in\N$ and $\ell'>0$, with $\fr{N}{4}-\fr{N'}{2}\ge\fr13$ and  $\fr{\ell}{36}-\fr{N'}{2}\geq \fr13$, $\ell'\le \fr{\ell}{3}-2$ that 
     \begin{align*}
         \|\tilde{\rho}(0, \cdot)\|_{H^{N'+1}}
         +\|(e^{\tilde{\phi}}-1)(0,\cdot)\|_{H^{N'+1}}+\|\tilde{f}(0,\cdot,\cdot)\|_{H^{N'+1}_{\ell'}}\leq \epsilon \nu^{\fr13}. 
     \end{align*}
     Indeed, recalling \eqref{def-g}, we have 
     $f_{\ne}=\big((e^{-\phi}-1)g\big)_{\ne}+g_{\ne}$. Then using the continuity estimate \eqref{est-compose} of the composite function $e^{-\phi}-1$ in Gevrey norm, the decay estimate of density \eqref{decay-rhoM},   and the bootstrap hypothesis \eqref{bd-en}, we find that
     \begin{align*}
     &\sum_{|\al|+|\beta|\le N'}\big\|\la v\ra^{\ell'}\pr_x^\al \pr_v^\beta f_{\ne}(\nu^{-\fr12},\cdot,\cdot)\big\|_{L^2_{x,v}}\\
     \les&\sum_{|\al|+|\beta|\le N'}\sum_{\beta'\le\beta}\big\|\la v\ra^{\ell'}\pr_x^\al (\pr_v+\nu^{-\fr12}\pr_x)^{\beta-\beta'}(\nu^{-\fr12}\pr_x)^{\beta'} f_{\ne}(\nu^{-\fr12},\cdot,\cdot)\big\|_{L^2_{x,v}}\\
     \les&\nu^{-\fr{N'}{2}}\sum_{|\al|+|\beta|\le N'}\sum_{\beta'\le\beta}\big\|\la v\ra^{\ell'}\pr_x^{\al+\beta'} (\pr_v+\nu^{-\fr12}\pr_x)^{\beta-\beta'}f_{\ne}(\nu^{-\fr12},\cdot,\cdot)\big\|_{L^2_{x,v}}\\
     \les& \nu^{-\fr{N'}{2}}\|f_{\ne}(\nu^{-\fr12})\|_{\mathcal{G}^{\lm,N/4}_{s,\ell-2[\ell/3]-2}}\\
     \les&\nu^{-\fr{N'}{2}}\|\rho(\nu^{-\fr12})\|_{\mathcal{G}^{\lm,N/4}_s}\|g(\nu^{-\fr12})\|_{\mathcal{G}^{\lm,N/4}_{s,\ell-2[\ell/3]-2}}+\nu^{-\fr{N'}{2}}\|g_{\ne}(\nu^{-\fr12})\|_{\mathcal{G}^{\lm,N/4}_{s,\ell-2[\ell/3]-2}}\\
     \les&\nu^{-\fr{N'}{2}}(\nu^{\fr12})^{\fr{N}{2}}\eps^2+\nu^{-\fr{N'}{2}}(\nu^{\fr16})^{\fr{[\ell/3]+1}{2}}\eps\les \nu^{\fr{\ell}{36}-\fr{N'}{2}}\eps\les \nu^{\fr13}\eps.
     \end{align*}
\end{rem}
Hence, $\mu + \sqrt{\mu}f^L_0$ solves the homogeneous (nonlinear) Landau equation while $\tilde{f}$ is being treated as the small perturbation. The kernel contribution ${\rm P}_0 f_0$ must be included in $\tilde{f}$, as otherwise $f_0^L$ would not have the correct asymptotic dynamics and the background $\mu + \sqrt{\mu} f^L_0$ would converge to a Maxwellian with the wrong kinetic energy. Note that ${\rm P}_0f_0$ is extremely small.

The $f_0^L$ retains its high regularity and velocity localization uniformly in  (a consequence of e.g., \cite{Guo2002}), while $\tilde{f}$ has $\nu$-dependent smallness but we will propagate much lower regularity and weaker localization. 

For clarity, we introduce the following notation: First, for $f_0^L$, we use the following norms:
 for $\ell\in \R$, $N\in\N$ and $\beta\in\N^3$, let us denote
\begin{align}
   \big\|f_0^L\big\|_{H^N_\ell}^2=\sum_{|\beta|\le N} \big\|\la v\ra^\ell\pr_v^\beta f_0^L\|_{L^2_v}^2+{\rm B}_1\big\|f_0^L\|_{L^2_v}^2,
\end{align}
and
\begin{align}
    \|f_0^L\|^2_{H^N_{\sig,\ell}}=\sum_{|\beta|\le N} \big|\pr_v^\beta f_0^L|_{\sig, \ell}^2+{\rm B}_1\dl\big|({\rm I}-{\rm P}_0)f_0^L|_{\sig}^2,
\end{align}
where ${\rm B}_1$ is a large constant determined in the proof, depending on $N$ and $\dl$ appearing in  \eqref{coercive-0}. By the construction of $f_0^L(0,v)$ (see \eqref{eq:quasi-linearized}) and the conservation laws, we have $\mathrm{P}_0f_0^L(t,v)\equiv 0$ for all $t\geq 0$. Thus, it holds that 
\begin{align*}
    \big|({\rm I}-{\rm P}_0)f_0^L|_{\sig}
=\big|f_0^L|_{\sig}.
\end{align*}

The energy functional for $\tl{f}$ is defined as follows:
\begin{align}\label{eq:tlE_l^n}
    \tl{\mathcal{E}}_{\ell}^{n}(\tl{f}(t))=\nn&{\rm A}_0\sum_{|\al|\le1}\left\|e^{\phi}\la v\ra^{\ell-2|\al|-2|n|}\pr_x^{\al}\tl{f}^{(n)}\right\|_{L^2_{x,v}}^2
    +\sum_{1\leq |\alpha|\leq 2}\left\|e^{\phi}\la v\ra^{\ell-2|\alpha|-2|n|}\kappa^{|\alpha|}\nu^{\frac{|\alpha|}{3}}\pr_v^{\alpha}\tl{f}^{(n)}\right\|_{L^2_{x,v}}^2\\
    &+2\kappa\nu^\fr13\left\la\nb_{x}\tl{f}^{(n)},e^{2\phi}\la v\ra^{2\ell-4-4|n|}\nb_{v}\tl{f}^{(n)} \right\ra_{x,v}.
\end{align}
Compared with \eqref{eq:E_l^n}, the $Y/(1+t)$ involved term is absent in \eqref{eq:tlE_l^n}. Moreover, we allow  the power of the $v$-weight to decrease as the order of the vector field $Z$ increases in \eqref{eq:tlE_l^n}. Correspondingly, the dissipation at $(n,\ell)$ level is given by
\begin{align}\label{def-tlD-nl}
    \nn\tl{\mathcal{D}}^{n}_\ell(\tl{f}(t))=\nn&{\rm A}_0\nu^{\fr23}\sum_{|\al|\le1}\left\|e^{\phi}\pr_x^\al \tl{f}^{(n)}\right\|_{\sig,\ell-2|\al|-2|n|}^2
    +\kappa\left\|e^{\phi}\la v\ra^{\ell-2-2|n|}\nb_x\tl{f}^{(n)}\right\|_{L^2_{x,v}}^2
   \\
    &+\sum_{1\leq |\alpha|\leq 2}\nu^{\fr23}\left\|e^{\phi}\kappa^{|\alpha|}\nu^{\fr{|\alpha|}{3}}\pr_v^{\alpha}\tl{f}^{(n)}\right\|_{\sig,\ell-2|\alpha|-2|n|}^2.
\end{align}
The total energy and dissipation for $\tl{f}$ are defined as follows.
\begin{align*}
    {\mathbb{E}}_{N,\ell}(\tl{f}(t))={\rm B}_0\|e^{\phi}\tl{f}(t)\|^2_{L^2_{x,v}}+\sum_{|n|\le N}\kappa^{2|n|}\tl{\mathcal{E}}_{\ell}^{n}(\tl{f}(t)),
\end{align*}
and
\begin{align*}
   {\mathbb{D}}_{N,\ell}(\tl{f}(t))=\nu^{\fr23}{\rm B}_0\big\|({\rm I}-{\rm P}_0)e^{\phi}\tl{f}(t)\big\|_{\sig}^2&+\sum_{|n|\le N}\kappa^{2|n|}\tl{\mathcal{D}}_{\ell}^{n}(\tl{f}(t)).
\end{align*}

Finally, for the density $\rho$, we introduce the following space: for $N\in\N$,  let us define 
\begin{align}\label{def-tlHN}
\tl{H}^N(\T^3)=\Big\{\rho\in L^2(\T^3): \sum_{\substack{\beta\in\N^6, |\beta|\le N}} \kappa^{2|\beta|}\|Z^\beta \rho\|^2_{L^2(\T^3)}<\infty\Big\}.
\end{align}
In the following, we will abbreviate $\tl{H}^N(\T^3)$ as $\tl{H}^N_x$.

The precise theorem we prove is the following, which is of independent interest, as it generalizes the results of \cite{chaturvedi2023vlasov} to backgrounds which are $O(\eps)$ (as opposed to $O(\eps \nu^{1/3})$) away from Maxwellian. 
\begin{theorem} \label{thm:longtime}
Let $\mathfrak{b}\geq \fr{1}{3}$. 
    There exist $\nu_0>0,\, \tilde{N}_0>0,\, \tilde{\ell}_0>0,\,\tilde{\kappa}_0>0,$ and $\epsilon_0>0$ such that for all $\tilde{N}\geq \tilde{N}_0,\, \tilde{\ell}\geq \tilde{\ell}_0$ and $0\leq \epsilon <\epsilon_0$, if the initial data $(f_0^{L}(0,v), \tilde{f}(0,x,v))$ satisfies
    \begin{align*}
        &\int_{\mathbb{T}^3\times \mathbb{R}^3}(f_0^{L}(0,v), \tilde{f}(0,x,v))\sqrt{\mu}(v)dvdx=\int_{\mathbb{R}^3}f_0^{L}(0,v) |v|^2\sqrt{\mu}(v)dv=0,\\
        &\int_{\mathbb{T}^3\times \mathbb{R}^3}(f_0^{L}(0,v), \tilde{f}(0,x,v))v_j\sqrt{\mu}(v)dvdx=0,\quad \text{for}\quad j=1,2,3,\\
        &\|\langle v\rangle^{\tilde{\ell}+6}\langle \nb_v\rangle^{\tilde{N}+6}f_0^{L}(0,v)\|_{L^2}\leq \epsilon,\quad 
        \|\langle v\rangle^{\tilde{\ell}}\langle \nb_{x,v}\rangle^{\tilde{N}+1}\tilde{f}(0,x,v)\|_{L^2(\mathbb{T}^3\times \mathbb{R}^3)}\leq \epsilon \nu^{\mathfrak{b}},
    \end{align*}
    then the solution $(f_0^{L},\tilde{f})$ to \eqref{eq:quasi-linearized} satisfies for all $t\geq 0$,
    \begin{align*}
        &\|\langle v\rangle^{\tilde{\ell}+6}\langle \nb_v\rangle^{\tilde{N}+6}f_0^{L}(t,v)\|_{L^2}\leq \epsilon,\\
        &{\mathbb{E}}_{\tl{N},\tl{\ell}}(\tl{f}(t))\leq \epsilon^2\nu^{2\mathfrak{b}}(\min\{1+t,\nu^{-\fr13}\})^2, \quad
        {\mathbb{E}}_{\tl{N}-2,\tl{\ell}}(\tl{f}(t))\leq \epsilon^2\nu^{2\mathfrak{b}}.
    \end{align*}
    Moreover, the Landau damping holds
    \begin{align*}
        \|\langle t\nb_x\rangle^{\tilde{N}/2}E(t)\|_{L^2}\lesssim \epsilon \nu^{\mathfrak{b}}. 
    \end{align*}
\end{theorem}
\begin{rem} 
    To prove Theorem \ref{Thm: main}, by choosing $\tilde{N}_0$ and $\tilde{\ell}_0$ in a proper way, we can make $\mathfrak{b}\geq 1$, which can lead to an easier proof of Theorem \ref{thm:longtime}. However, we chose to fully generalize the results in \cite{chaturvedi2023vlasov} and reach the optimal $\nu^{\fr13}$ for perturbations in Sobolev spaces. 
\end{rem}
\begin{rem}\label{Rmk: enhanced dissipation}
Once we obtain the global control of the Sobolev norm with respect to the vector field, it is easy to prove the following enhanced dissipation estimate
\begin{align*}
    \|f_{\neq}\|_{L^2}\lesssim \epsilon\left\langle \nu^{\fr13}t\right\rangle^{-\ell/6},
\end{align*}
by using the enhanced dissipation type energy estimate, i.e., 
\begin{align*}
    \mathbb{E}^{ed}(g(t))=\sum_{\iota=1}^{[\ell/3]+1}w_{\iota}^2(t)\mathcal{E}_{\ell-2\iota}^0(g_{\neq}),
\end{align*}
where the energy function $\mathcal{E}_{\ell}^{0}$ is given in \eqref{eq:E_l^n} with $n=0$. 
\end{rem}
We conclude this section by proving Theorem \ref{Thm: main} assuming Theorems \ref{Thm:shorttime} and \ref{thm:longtime}. 
\begin{proof}[Proof of Theorem \ref{Thm: main}]
    By combining Theorem \ref{Thm:shorttime} (for $t\leq \nu^{-\fr12}$), Remark \ref{Rmk: P_0f_0} (at $t=\nu^{-\fr12}$), Theorem \ref{thm:longtime} with $\tilde{N}_0>N_1$ (for $t\geq \nu^{-\fr12}$), and Remark \ref{Rmk: enhanced dissipation} (enhanced dissipation), we proved Theorem \ref{Thm: main}. 
\end{proof}
\subsubsection{The new solution operator}
Although the $f_0^L$ is small, as discussed in section \ref{sec:density decomposition}, the term $\Gamma(f_0^L,\tilde{f})$ could cause the main growth. To deal with it precisely, we must incorporate it into the linearized equation. This is the main new difficulty that we have to confront relative to \cite{chaturvedi2023vlasov}, i.e., everywhere $S_k(t)$ appeared, we need to replace this with the two-parameter linear propagator of the following linearized equation:
\begin{align}\label{eq:linearized 2}
    \partial_t {\rm f}+v\cdot \nabla_x {\rm f}+\nu L{\rm f}-\nu \Gamma(f_0^L,{\rm f})-\nu \Gamma({\rm f},f_0^L)=0.
\end{align}
Compared with \eqref{eq:linearized 1},  the nonlinear collisions associated with the zero mode interactions, namely, $\Gamma(f_0^L,{\rm f})$ and $ \Gamma({\rm f},f_0^L)$ are involved. 
We would like to emphasize that the solution operator $\mathbb{S}(t,\tau)$ is a \emph{two-parameter} semigroup, because here $f_0^L$ is time-dependent. 

One can  recover the density $\rho$ with the aid of the solution operator $\mathbb{S}(t,\tau)$. To this end, note first that the conservation of mass and our choice of initial data ensure that $f_0^L$ contributes nothing to the density, namely, 
\begin{align}\label{exp-tl-rho}
\rho(t,x)=\tl{\rho}(t,x)=\int_{\R^3}\tl{f}(t,x,v)\mu^{\fr12}(v)dv.
\end{align}
Then
we apply the solution operator $\mathbb{S}(t,\tau)$ (or $\mathbb{S}_k(t,\tau)$ -- its Fourier transform in $x$) to the second equation of \eqref{eq:quasi-linearized} and obtain the expression for $\tl{f}$:
\begin{align*}
    \tilde{f}(t)=\mathbb{S}(t,0)[\tilde{f}(0)]+2\int_0^t \tilde{E}(\tau)\mathbb{S}(t,\tau)[v\mu^{\fr12}]d\tau+\int_0^t\mathbb{S}(t,\tau)[{\bf N}(\tau)]d\tau,
\end{align*}
where 
\begin{align}\label{eq:N}
    {\bf N}(t,x,v)=-E\cdot\nb_v \tilde{f}+E\cdot v\tilde{f}-E\cdot \nabla_v f_0^{L}+E\cdot v f_0^{L}+\nu\Gamma(\tilde{f},\tilde{f}),
\end{align}
we have used $E$ to replace $\tl{E}$ in \eqref{eq:N} due to the fact $\rho=\tl{\rho}$. Accordingly, we infer from \eqref{exp-tl-rho} that 
\begin{align}\label{ref-rho}
    \rho(t)+\mathfrak{L}_t[\rho(\tau)]=\mathfrak{I}+
    \int_{\mathbb{R}^3}\int_0^t\mathbb{S}(t,\tau)[{\bf N}(\tau)] \mu^{\fr12}(v)d\tau dv, 
\end{align}
where 
\[
\mathfrak{L}_t[\rho(\tau)]=-2\int_{\mathbb{R}^3}\int_0^t \nabla_x(-\Delta_x)^{-1}\rho(\tau)\mathbb{S}(t,\tau)[v\mu^{\fr12}]\mu^{\fr12}(v)d\tau dv
\]
 and $\mathfrak{I}=\int_{\mathbb{R}^3}\mathbb{S}(t,0)[\tilde{f}(0)]\mu^{\fr12}dv$. Note that this is no longer quite a standard Volterra equation. 

\subsubsection{Invertibility and stability}
To solve for $\rho$ from \eqref{ref-rho}, we have to show that the operator ${\rm Id}+\mathfrak{L}_t$ is invertible for all $t\geq 0$. As mentioned above, without the two quasi-linear terms $\Gamma(f_0^L,{\rm f})$ and $ \Gamma({\rm f},f_0^L)$, the solution operator $S(t)$ associated with \eqref{eq:linearized 1} is a semi-group and the operator 
\[
\mathfrak{L}_t^{\mathrm{app}}[\rho]=-2\int_{\mathbb{R}^3}\int_0^t \nabla_x(-\Delta_x)^{-1}\rho(\tau)S(t-\tau)[v\mu^{\fr12}]\mu^{\fr12}(v)d\tau dv
\]
is a convolution with respect to the time variable, whose Fourier transform is given in \eqref{eq:Fourier of Lrho}. Thanks to the Penrose condition, the invertibility of ${\rm Id}+\mathfrak{L}^{\mathrm{app}}_t$ can be proved  using the Laplace transform; we refer to section \ref{sec:penrose} for more details. Here we use the smallness of $f_0^L$ and show that the error of $\mathbb{S}(t,\tau)[v\mu^{\fr12}]-S(t-\tau)[v\mu^{\fr12}]$ is under control, namely, we have
\begin{align*}
    \rho(t)
    =&({\rm Id}+\mathfrak{L}^{\mathrm{app}}_t)^{-1}\left[\mathfrak{I}+
    \int_{\mathbb{R}^3}\int_0^t\mathbb{S}(t,\tau)[{\bf N}(\tau)] \mu^{\fr12}(v)d\tau dv\right]\\
    &+({\rm Id}+\mathfrak{L}^{\mathrm{app}}_t)^{-1}\left[2\int_{\mathbb{R}^3}\int_0^t \nabla_x(-\Delta_x)^{-1}\rho(\tau)\big(\mathbb{S}(t,\tau)[v\mu^{\fr12}]-S(t-\tau)[v\mu^{\fr12}]\big)\mu^{\fr12}(v)d\tau dv\right]. 
\end{align*}
A key observation is that the difference $\mathbb{S}(t,\tau)[v\mu^{\fr12}]-S(t-\tau)[v\mu^{\fr12}]$ is not only sufficiently small but also possesses adequate vector field regularity (depending on $f_0^L$), enabling us to obtain sufficiently fast decay through Landau damping. Consequently, the second term in the representation of the  density $\rho$ above can be  absorbed by the left-hand side as a perturbation. See more details in Proposition \ref{prop:perturb-rho}.

\subsubsection{Bootstrap argument for the proof of Theorem \ref{thm:longtime}}
In this section, we introduce the bootstrap argument for the proof of Theorem \ref{thm:longtime}. Let us now introduce the bootstrap hypotheses. We assume that for $t\leq T^*$, there holds
    \begin{align}
    \label{H-f_0^L}&\sup_{t\le T^*}\|f_0^L(t)\|_{H^{\tl N+4}_{\tl{\ell}+3}}^2+\fr12\nu\int_0^{T^*}\|f_0^L(t')\|_{H^{\tl N+4}_{\sig,\tl{\ell}+3}}^2dt'\le 8{\rm C}_L\eps^2,\\
        \label{H-tlf-H}&\sup_{t\le T^*}{\mathbb{E}}_{\tl{N},\tl{\ell}}(\tl{f}(t))+\nu^{\fr13}\int_0^{T^*}\mathbb{D}_{\tl{N},\tl{\ell}}(\tl{f}(t))dt\leq 8{\rm C}_{\tl f,1}\epsilon^2\nu^{2\mathfrak{b}}(\min\{1+t,\nu^{-\fr13}\})^2,\\
        \label{H-tlf-L}&
        \sup_{t\le T^*}{\mathbb{E}}_{\tl{N}-2,\tl{\ell}}(\tl{f}(t))+\nu^{\fr13}\int_0^{T^*}\mathbb{D}_{\tl{N}-2,\tl{\ell}}(\tl{f}(t))dt\leq 8{\rm C}_{\tl{f},2}\epsilon^2\nu^{2\mathfrak{b}},\\
        \label{H-tlrho}&\sup_{t\le T^*}\|\rho(t)\|^2_{\tl{H}^{\tl{N}}_x}+\nu^{\fr13}\int_0^{T^*}\|\rho(t)\|^2_{H^{\tl{N}}_x}dt\le 8{\rm C}_\rho \eps^2\nu^{2\mathfrak{b}}.
    \end{align}
\begin{rem}
    The estimate of $f_0^L$ is independent of the bootstrap hypotheses \eqref{H-tlf-H}-\eqref{H-tlrho}. 
\end{rem}
We have the following bootstrap proposition. 
\begin{prop}
    Assume that the estimates \ref{H-f_0^L}---\ref{H-tlrho} hold for $t\in [0,T]$. Then there exist $\tilde{\kappa}_0$, $\epsilon_0$ and $\nu_0$ small enough so that these estimates hold with all occurrences of $8$ on the right-hand side of \ref{H-f_0^L}---\ref{H-tlrho} replaced by $4$. As a consequence, the above estimates hold for $t\in [0,\infty)$. 
\end{prop}
The second part of this paper is to prove the improvement of \eqref{H-f_0^L}---\eqref{H-tlrho}. We refer to section \ref{sec:f_0^L} for the improvement of \eqref{H-f_0^L}, refer to section \ref{sec: tlf} for the improvement of \eqref{H-tlf-H} and \eqref{H-tlf-L}, and refer to section \ref{sec: density} for the improvement of \eqref{H-tlrho}.

{\Large \part{Landau damping and enhanced dissipation regime $t \leq \nu^{-1/2} $}}

\section{The relation between the distribution function and the good unknown}\label{sec: relationship}
As mentioned in section \ref{sec:Outline}, we introduce Guo's good unknown $g=e^{\phi}f$. In this section, we study the Gevrey regularity of $f$ and $g$ under the bootstrap hypotheses. 
\begin{lem}\label{lem-compose}
Let $q(\phi)=e^{-\phi}-1$. If $\|\phi\|_{\mathcal{G}^{\lm,N}_s}\le\fr12$, we have
\begin{align}\label{est-compose}
    \|q(\phi)\|_{\mathcal{G}^{\lm,N}_{s}}\le C\|\phi\|_{\mathcal{G}^{\lm,N}_{s}}.
\end{align}
\end{lem}
\begin{proof}
Recall the definition of $\tl{H}^N_x$ in \eqref{def-tlHN}.
   Clearly, if $N>\fr32$, $\tl{H}^N_x$ is an algebra, and we denote by $C^*_{N,\kappa}\ge1$    the  constant appearing in the product estimate $\|fg\|_{\tl{H}^N_x}\le C^*_{N,\kappa}\|f\|_{\tl{H}^N_x}\|g\|_{\tl{H}^N_x}$. Recalling \eqref{def-Gnorm'}, one can write
\begin{align*}
    \|q(\phi)\|_{\mathcal{G}_{s}^{\lambda,N}}^2
    =\sum_{ m\in \mathbb{N}^6}a_{m,\lm,s}^2(t)\big\|Z^{m}q(\phi)\big\|_{\tl{H}^N_{x}}^2.
\end{align*}
For any $m\in \N^6\setminus\{\bf{0}\}$, let us denote $D_{m}=\{\gamma\in\N^6\setminus\{{\bf0}\}: \gamma\le m\}$. We define a partition of $m$ by a map $\pi: D_{m}\rightarrow \N$  satisfying $\sum_{\gamma\in D_m}\gamma\pi(\gamma)=m$.
Thanks to Fa\`{a} di Bruno's formula (see \cite{TBN20} for more details), we have
\begin{align}\label{Bruno}
    Z^{m}(q\circ\phi)=\nn&\sum_{\sum_{\gamma\in D_{m}}\gamma\pi(\gamma)=m}\fr{m!}{\prod_{\gamma\in D_{m}}\pi(\gamma)!}q^{(\sum_{\gamma\in D_m}\pi(\gamma))}(\phi)\prod_{\gamma\in D_m}\Big(\fr{Z^{\gamma}\phi}{\gamma!}\Big)^{\pi(\gamma)}\\
    =&\sum_{k=1}^{|m|}q^{(k)}(\phi)\sum_{\substack{\sum_{\gamma\in D_m}\gamma\pi(\gamma)=m\\ \sum_{\gamma\in D_m}\pi(\gamma)=k}}\fr{m!}{\prod_{\gamma\in D_m}\pi(\gamma)!}\prod_{
    \substack{\gamma\in D_m}}\Big(\fr{Z^{\gamma}\phi}{\gamma!}\Big)^{\pi(\gamma)}.
\end{align}
Consequently, recalling that $\Gamma_s(|m|)=2^{-25}(|m|!)^{\fr1s}(|m|+1)^{-12}$, we have
\begin{align*}
    &\fr{\lm^{|m|}}{\Gamma_s(|m|)}C_{|m|}^{m}Z^m(q\circ\phi)\\
    =&2^{25}\fr{(|m|+1)^{12}}{(|m|!)^{\fr1s-1}}\sum_{k=1}^{|m|}q^{(k)}(\phi)\sum_{\substack{\sum_{\gamma\in D_m}\gamma\pi(\gamma)=m\\ \sum_{\gamma\in D_m}\pi(\gamma)=k}}\fr{1}{\prod_{\gamma\in D_m}\pi(\gamma)!}\prod_{
    \substack{\gamma\in D_m}}\left(\fr{\lm^{|\gamma|} Z^{\gamma}\phi}{\gamma!}\right)^{\pi(\gamma)}\\
    =&\sum_{k=1}^{|m|}\fr{1}{2^{25(k-1)}}q^{(k)}(\phi)\sum_{\substack{\sum_{\gamma\in D_m}\gamma\pi(\gamma)=m\\ \sum_{\gamma\in D_m}\pi(\gamma)=k}}\fr{1}{\prod_{\gamma\in D_m}\pi(\gamma)!}\prod_{
    \substack{\gamma\in D_m}}\left(\fr{\lm^{|\gamma|} }{\Gamma_s(|\gamma|)} C_{|\gamma|}^{\gamma}Z^{\gamma}\phi\right)^{\pi(\gamma)}\\&\times \Big(\fr{\prod_{\gamma\in D_m}(|\gamma|!)^{\pi(\gamma)}}{|m|!}\Big)^{\fr{1}{s}-1}\fr{(|m|+1)^{12}}{\prod_{\gamma\in D_{m}}(|\gamma|+1)^{12\pi(\gamma)}}.
\end{align*}
Then for any $N'\in\N\setminus\{0\}$, using the algebra property of $\tl{H}^N_x$, we have
\begin{align*}
    &\sum_{\substack{m\in\N^6\\1\le|m|\le N'}}\left\|\fr{\lm^{|m|}}{\Gamma_s(|m|)}C_{|m|}^{m}Z^{m}(q\circ\phi)\right\|_{\tl{H}^N_x}^2\\
    \le&\sum_{\substack{m\in\N^6\\1\le|m|\le N'}}\Bigg(\sum_{k=1}^{|m|}\fr{(C^*_{N,\kappa})^k}{2^{25(k-1)}}\| q^{(k)}(\phi)\|_{\tl{H}^N_x}\\
    &\times \sum_{\substack{\sum_{\gamma\in D_m}\gamma\pi(\gamma)=m\\ \sum_{\gamma\in D_m}\pi(\gamma)=k}}\fr{1}{\prod_{\gamma\in D_m}\pi(\gamma)!}\prod_{\gamma\in D_m}\left\|\fr{\lm^{|\gamma|} }{\Gamma_s(|\gamma|)} C_{|\gamma|}^{\gamma}Z^{\gamma}\phi\right\|_{\tl{H}^N_x}^{\pi(\gamma)}\Bigg)^2\\
    \le&\|e^{-\phi}\|_{\tl{H}^N_x}^2\sum_{\substack{m\in\N^6\\1\le|m|\le N'}}\Bigg(\sum_{k=1}^{|m|}\Big(\fr{(C^*_{N,\kappa})^k}{2^{25(k-1)}}\Big)^2\fr{1}{k!}\sum_{\substack{\sum_{\gamma\in D_m}\gamma\pi(\gamma)=m\\ \sum_{\gamma\in D_m}\pi(\gamma)=k}}\fr{1}{\prod_{\gamma\in D_m}\pi(\gamma)!}\Bigg)\\
    &\times \sum_{k=1}^{|m|}\sum_{\substack{\sum_{\gamma\in D_m}\gamma\pi(\gamma)=m\\ \sum_{\gamma\in D_m}\pi(\gamma)=k}}\fr{k!}{\prod_{\gamma\in D_m}\pi(\gamma)!}\prod_{\gamma\in D_m}\left\|\fr{\lm^{|\gamma|} }{\Gamma_s(|\gamma|)} C_{|\gamma|}^{\gamma}Z^{\gamma}\phi\right\|_{\tl{H}^N_x}^{2\pi(\gamma)}\\ 
    \le&C\|e^{-\phi}\|_{\tl{H}^N_x}^2\sum_{k=1}^{N'}\sum_{\substack{m\in\N^6\\k\le|m|\le N'}}\sum_{\substack{\sum_{\gamma\in D_m}\gamma\pi(\gamma)=m\\ \sum_{\gamma\in D_m}\pi(\gamma)=k}}\fr{k!}{\prod_{\gamma\in D_m}\pi(\gamma)!}\prod_{\gamma\in D_m}\left\|\fr{\lm^{|\gamma|} }{\Gamma_s(|\gamma|)} C_{|\gamma|}^{\gamma}Z^{\gamma}\phi\right\|_{\tl{H}^N_x}^{2\pi(\gamma)}\\
    \le&C\|e^{-\phi}\|_{\tl{H}^N_x}^2\sum_{k=1}^{N'}\left(\sum_{1\le|\gamma|\le N'}\left\|\fr{\lm^{|\gamma|} }{\Gamma_s(|\gamma|)} C_{|\gamma|}^{\gamma}Z^{\gamma}\phi\right\|_{\tl{H}^N_x}^{2}\right)^k. 
\end{align*}
Using the elementary inequality $|e^x-1|\le |x|e^{|x|}$  and \eqref{Bruno} again, noting that $\|\phi\|_{\mathcal{G}^{\lm,N}_s}<1$, we find that  
\begin{align*}
    \|q\circ\phi\|_{\tl{H}^N_x}^2=&\|q\circ\phi\|_{L^2_x}^2+\sum_{1\le|\beta|\le N}\|(\kappa Z)^\beta (q\circ\phi)\|_{L^2_x}^2\\
    \le&e^{2\|\phi\|_{L^\infty_x}}\|\phi\|_{L^2_x}^2+C_N\|e^{-\phi}\|_{L^\infty_x}^2\|\phi\|^2_{\tl{H}^N_x}\\
    \le&C_N\|\phi\|^2_{\tl{H}^N_x},
\end{align*}
and thus
\begin{align*}
    \|e^{-\phi}\|_{\tl{H}^N_x}^2\le C_N.
\end{align*}
It follows that
\begin{align*}
    \|q(\phi)\|_{\mathcal{G}^{\lm,N}_s}=&\|q\circ\phi\|_{\tl{H}^N_x}^2+\sum_{\substack{0<m\in\N^6}}\left\|\fr{\lm^{|m|}}{\Gamma_s(|m|)}C_{|m|}^{m}Z^{m}(q\circ\phi)\right\|_{\tl{H}^N_x}^2\\
    \le&C_N\|\phi\|^2_{\tl{H}^N_x}+C\|e^{-\phi}\|_{\tl{H}^N_x}^2\sum_{k=1}^{\infty}\left(\sum_{0<\gamma\in \N^6}\left\|\fr{\lm^{|\gamma|} }{\Gamma_s(|\gamma|)} C_{|\gamma|}^{\gamma}Z^{\gamma}\phi\right\|_{\tl{H}^N_x}^{2}\right)^k\\
    \le&C_N\|\phi\|_{\mathcal{G}^{\lm,N}_s}, 
\end{align*}
due to $\|\phi\|_{\mathcal{G}^{\lm,N}_s}\le\fr12$.
\end{proof}

\begin{coro}\label{coro-compose}
   Let $\ell\in\R$. Under the conditions of Lemma \ref{lem-compose},  there hold for $\alpha\in \mathbb{N}^3$
    \begin{align}\label{fg1}
    \|\nb_v^{\alpha}f\|_{\mathcal{G}^{\lm,N}_{s,\ell}}\les \Big(1+\| \phi\|_{\mathcal{G}^{\lm,N}_{s}}\Big)\|\nb_v^{\alpha}g\|_{\mathcal{G}^{\lm,N}_{s,\ell}},
    \end{align}
    
    \begin{align}\label{fg2}
    \|\nb_xf\|_{\mathcal{G}^{\lm,N}_{s,\ell-2}}\les \Big(1+\| \phi\|_{\mathcal{G}^{\lm,N}_{s}}\Big)\Big(\|\nb_x\phi\|_{\mathcal{G}^{\lm,N}_{s}}\|g\|_{\mathcal{G}^{\lm,N}_{s,\ell}}+\|\nb_xg\|_{\mathcal{G}^{\lm,N}_{s,\ell-2}}\Big),
    \end{align}
    
    \begin{align}\label{fg3}
    \|Yf\|_{\mathcal{G}^{\lm,N}_{s,\ell-2}}\les \Big(1+\| \phi\|_{\mathcal{G}^{\lm,N}_{s}}\Big)\Big(t\|\nb_x\phi\|_{\mathcal{G}^{\lm,N}_{s}}\|g\|_{\mathcal{G}^{\lm,N}_{s,\ell}}+\|Yg\|_{\mathcal{G}^{\lm,N}_{s,\ell-2}}\Big),
    \end{align}

    \begin{align}\label{fg4}
    \|\nb_v^{\alpha}f\|_{\mathcal{G}^{\lm,N}_{s,\sig,\ell}}\les \Big(1+\| \phi\|_{\mathcal{G}^{\lm,N}_{s}}\Big)\|\nb_v^{\alpha}g\|_{\mathcal{G}^{\lm,N}_{s,\sig,\ell}},
    \end{align}
    
    \begin{align}\label{fg5}
    \|\nb_xf\|_{\mathcal{G}^{\lm,N}_{s,\sig,\ell-2}}\les \Big(1+\| \phi\|_{\mathcal{G}^{\lm,N}_{s}}\Big)\Big(\|\nb_x\phi\|_{\mathcal{G}^{\lm,N}_{s}}\|g\|_{\mathcal{G}^{\lm,N}_{s,\sig,\ell}}+\|\nb_xg\|_{\mathcal{G}^{\lm,N}_{s,\sig,\ell-2}}\Big),
    \end{align}
    
    \begin{align}\label{fg6}
    \|Yf\|_{\mathcal{G}^{\lm,N}_{s,\sig,\ell-2}}\les \Big(1+\| \phi\|_{\mathcal{G}^{\lm,N}_{s}}\Big)\Big(t\|\nb_x\phi\|_{\mathcal{G}^{\lm,N}_{s}}\|g\|_{\mathcal{G}^{\lm,N}_{s,\sig,\ell}}+\|Yg\|_{\mathcal{G}^{\lm,N}_{s,\sig,\ell-2}}\Big).
    \end{align}
\end{coro}
\begin{proof}
    Since $g=e^{\phi}f$, we have $[\nb_v^{\alpha},e^{\phi}]=0$, 
\begin{align}\label{decom-f}
    f=q(\phi)g+g,\quad \nb_xf=-q(\phi)\nb_x\phi g-\nb_x\phi g+q(\phi)\nb_xg+\nb_xg,
\end{align}
and
\begin{align}
     Yf=-tq(\phi)\nb_x\phi g-t\nb_x\phi g+q(\phi)Yg+Yg.
\end{align}
Then \eqref{fg1}--\eqref{fg3} follows from the  the product estimate  \eqref{product1} and Lemma \ref{lem-compose}, and \eqref{fg4}--\eqref{fg6} follow from the  the product estimate  \eqref{product2} and Lemma \ref{lem-compose}, respectively.
\end{proof}

\begin{prop}\label{lem-Linfty-rho}
Assume that \eqref{res-final-1} and \eqref{res-final-2} hold. Then  under the bootstrap hypotheses \eqref{bd-en}--\eqref{H-phi},  for $t\in[0,\nu^{-\fr12}]$, there holds 
    \begin{align}\label{Linfty-rho}
        \|(\nb_x\rho, \nb_xM)\|_{L^\infty_t\mathcal{G}^{\lm,N}_s}\les \epsilon.
    \end{align}

    \begin{align}\label{decay-rhoM}
    &\|(\rho(t),M(t))\|_{\mathcal{G}_s^{\lambda,N/4}}\lesssim \fr{\epsilon}{\langle t\rangle^{N/2}},
    \end{align}
    and
    \begin{align}
        \|E(t)\|_{L^2}\lesssim \epsilon e^{-c_0\lambda_{\infty}\langle t\rangle^s}. 
    \end{align}
    As a consequence, the constant $\fr12$ on the right hand side of the bootstrap hypothesis \eqref{H-phi} can be improved to be  $\fr14$. 
\end{prop}
\begin{proof}
Let us first  bound $\|\rho(t)\|_{\mathcal{G}^{\lm,N}_s}$.
    Recalling that
    \[
    \rho(t,x)=\int_{\R^3}f(t,x,v)\sqrt{\mu}dv,
    \]
    for any $m,\beta\in\N^6$ with $|\beta|\le N$, by Plancherel's theorem,  we have
\begin{align*}
    &\big|(k,kt)^{m+\beta}\hat{\rho}_k(t)\big|\\
    =&\Big|\sum_{n\le m,\beta'\le\beta}C_{m}^nC_{\beta}^{\beta'}\int_{\R^3}(k,\eta+kt)^{n+\beta}\mathcal{F}_{x,v}[f]_k(t,\eta)(0,-\eta)^{m-n+\beta-\beta'}\overline{\mathcal{F}_v[\sqrt{\mu}]}(\eta)d\eta\Big|\\
    \le&\sum_{n\le m,\beta'\le\beta}C_{m}^n C_{\beta}^{\beta'}\big\|\mathcal{F}_x[f^{(n+\beta)}]_k\big\|_{L^2_{v}}\big\|\pr_v^{\bar{m}-\bar{n}+\bar{\beta}-\bar{\beta'}}\sqrt{\mu}\big\|_{L^2_v}.
\end{align*}
Consequently, recalling \eqref{def-Gnorm'}, using Lemma \ref{lem-summable}, and noting that 
\begin{align*}
    a_{m,\lm,s}(t)\les a_{\bar{m},\lm,s}(t),\quad {\rm for\ \ all}\quad 0<s\le1,
\end{align*}
one deduces that
\begin{align}\label{up:rho-f}
    \|\rho(t)\|^2_{\mathcal{G}^{\lm,N}_s}
\le\nn&C\sum_{k\in\N^3_*}\sum_{m\in\N^6,|\beta|\le N}\bigg(\sum_{n\le m,\beta'\le\beta}b_{m,n,s}\Big(\kappa^{|\beta'|}a_{n,\lm,s}(t)\big\|\mathcal{F}_x[f^{(n+\beta')}]_k(t)\big\|_{L^2_{v}}\Big)\\
\nn&\times  \Big(\kappa^{|\beta-\beta'|}a_{m-n,\lm,s}(t)\big\|\pr_v^{\bar{m}-\bar{n}+\bar{\beta}-\bar{\beta'}}\sqrt{\mu}\big\|_{L^2_v}\Big)\bigg)^2\\
\le\nn&C\sum_{k\in\N^3_*}\sum_{m\in\N^6,|\beta|\le N}\kappa^{2|\beta|}a^2_{m,\lm,s}(t)\big\|\mathcal{F}_x[f^{(m+\beta)}]_k(t)\big\|^2_{L^2_{v}} \|\sqrt{\mu}\|^2_{\mathcal{G}^{\lm, N}_{s,0}}\\
\le&C\|f_{\ne}(t)\|^2_{\mathcal{G}^{\lm,N}_{s,0}}.
\end{align}
The bootstrap hypothesis \eqref{H-phi} enables us to use \eqref{fg1} to bound $\|f_{\ne}(t)\|_{\mathcal{G}^{\lm,N}_{s,0}}$. Then thanks to the bootstrap hypothesis \eqref{bd-en},  we are led to
\begin{align*}
    \|\rho(t)\|_{\mathcal{G}^{\lm,N}_s}\le C\big(1+\| \phi\|_{\mathcal{G}^{\lm,N}_{s}}\big)\|g\|_{\mathcal{G}^{\lm,N}_{s,0}}\le C\sqrt{{\rm C}_g}\eps \big(1+\|\rho(t)\|_{\mathcal{G}^{\lm,N}_s}\big).
\end{align*}
Then we have
\begin{align}\label{bd-Linfty-rho}
    \|\rho\|_{L^\infty_t\mathcal{G}^{\lm,N}_s}\le 2C\sqrt{{\rm C_g}}\eps,
\end{align}
provided $\eps$ is so small that $C\sqrt{{\rm C}_g}\eps\le \fr12$. Clearly, \eqref{bd-Linfty-rho} implies that 
\begin{align}\label{bd-Linfty-phi}
\|\nb_x\phi\|_{L^\infty_t\mathcal{G}^{\lm,N}_s}\les\eps.
\end{align}
On the other hand, similar to \eqref{up:rho-f}, we arrive at
\begin{align*}
    \|\nb_x\rho(t)\|_{\mathcal{G}^{\lm,N}_{s}}\les \|\nb_xf(t)\|_{\mathcal{G}^{\lm,N}_{s,0}}.
\end{align*}
Combining this with \eqref{fg2}, and using \eqref{bd-Linfty-phi}, we obtain the bound for $\nb_x\rho$ in \eqref{Linfty-rho} immediately. For $\nb_xM$, recalling that $M(t,x)=\int_{\R^3}f(t,x,v)v\sqrt{\mu}dv$, it is clear to see that the above estimates for $\nb_x\rho$ applies to $\nb_xM$, and thus \eqref{Linfty-rho} holds.

To prove \eqref{decay-rhoM}, recalling \eqref{def-Gnorm'}, and noting that 
\[
\la t\ra^2\le|k,kt|^2=\sum_{j=1}^6 (k,kt)^{2e_j}, \quad{\rm for}\quad k\in\Z^3_*,
\]
one easily deduces that
\begin{align}\label{decay-rho-recursion}
    \|(\rho(t),M(t))\|_{\mathcal{G}^{\lm,N'}_s}\le\fr{\kappa^{-1}}{\la t\ra}\|(\rho(t),M(t))\|_{\mathcal{G}^{\lm,N'+1}_s},\quad {\rm for\ \ any}\quad N'\in\N.
\end{align}
Combining this with \eqref{Linfty-rho}, \eqref{decay-rhoM} follows immediately.
\end{proof}

\section{The Landau operator in vector-field Gevrey regularity}

In this section, we introduce some basic estimates of the Landau operator in the Gevrey norms built from the vector fields. The aim of the following lemma is to obtain estimates on the linear collision operator that arises when studying the time-derivative of energies. 

Let us first recall some basic properties of $\Phi^{ij}$ in the Landau operator. 
\begin{rem}
    For any fixed $i$ or $j$, there holds
\begin{align}\label{concel1}
\sum_{j}\Phi^{ij}(v-\tl{v})(v_j-\tl{v}_j)=\sum_{i}\Phi^{ij}(v-\tl{v})(v_i-\tl{v}_i)=0.
\end{align}
Moreover, a direct calculation gives us that the Fourier transform of $\Phi^{ij}$ is 
\begin{align}\label{eq:Fourier-Phi}
    \widehat{\Phi^{ij}}(\xi)=8\pi \frac{\xi_i\xi_j}{|\xi|^4}. 
\end{align}
\end{rem}

The main coercive estimate for the linear Landau operator $L$ in Gevrey space is stated as follows.
\begin{lem}\label{lem-coercive-L}
    Let $\zeta>0, M_1,M_2\in\N$,  $\ell>2N_1$,  and $\al\in\N^3, \beta\in\N^6, m\in\N^6$.  Then for all $0<\lm\le\fr{1}{2e}\max\{M_{\sig},M_{\mu}\}$ and $0<\kappa\le 1$, there exist positive constants $C_{s,N_1,N_2, M_\sig, M_\mu}$, $C_{N_1,N_2}$ and $C_{\zeta}$, such that
\begin{align}\label{coer-Landau}
    \nn&\sum_{m\in \N^6}\sum_{|\al|\le N_{1},|\beta|\le N_{2}}\kappa^{2|\al|+2|\beta|}a_{m,\lm,s}^2(t)\left\la\pr_v^\al Z^{\beta+m}Lg, \pr_v^\al g^{(m+\beta)} \right\ra_{\ell-2|\al|} \\
    \nn\ge&\big(1-\lm C_{s,N_1,N_2, M_\sig, M_\mu}-\kappa{C_{N_1,N_2}}-\zeta\big)\\
    \nn&\times\sum_{m\in \N^6}\sum_{|\al|\le N_{1},|\beta|\le N_{2}}\kappa^{2|\al|+2|\beta|}a_{m,\lm,s}^2(t)\big\|\pr_v^\al  g^{(m+\beta)}\big\|_{\sig,\ell-2|\al|}^2\\
    &-C_\zeta \sum_{m\in \N^6}\sum_{|\beta|\le N_{2}}\kappa^{2|\beta|}a_{m,\lm,s}^2(t)\big\|  g^{(m+\beta)}\big\|_{L^2_{x,v}(|v|\le2\zeta)}^2.
\end{align}

\end{lem}

\begin{proof} 
To begin with, we write
\begin{align*}
    &\left\la \pr_v^\al Z^{m+\beta} Lg, \pr_v^\al g^{(m+\beta)} \right\ra_{\ell-2|\al|}\\
    =&\left\la L\pr_v^\al g^{(m+\beta)}, \pr_v^\al g^{(m+\beta)} \right\ra_{\ell-2|\al|}+\left\la [\pr_v^\al Z^{m+\beta}, L] g, \pr_v^\al g^{(m+\beta)} \right\ra_{\ell-2|\al|}.
\end{align*}
If $|\al|+|\beta|+|m|=0$, the second term on the right-hand side above vanishes. Otherwise, we further write
\begin{align}\label{commutator-L}
    [\pr_v^\al Z^{m+\beta}, L] g\nn=&\pr_v^\al  Z^\beta Z^m Lg-\pr_v^\al Z^\beta  L  Z^m g+\pr_v^\al Z^\beta  L  Z^m g-L\pr_v^\al Z^\beta Z^m  g\\
    =&\pr_v^\al Z^\beta[ Z^m, L]g+[\pr_v^\al Z^\beta, L]g^{(m)}.
\end{align}
Recalling the definitions of $\mathcal{A}$ and $\mathcal{K}$, and noting that $L=-\mathcal{A}-\mathcal{K}$, by Lemma 5 of \cite{Guo2002}, we have
\begin{align}\label{L2-Landau}
    \nn&\left\la L\pr_v^\al g^{(m+\beta)}, \pr_v^\al g^{(m+\beta)} \right\ra_{\ell-2|\al|}\\
  \nn=&\left\|\pr_v^\al  g^{(m+\beta)}\right\|_{\sig,\ell-2|\al|}^2+\int_{\T^3\times\R^3} \sig^{ij}\pr_{v_j}\pr_v^\al  g^{(m+\beta)} \pr_v^\al  g^{(m+\beta)} \pr_{v_i}\la v\ra^{2\ell-4|\al|}dvdx\\
  \nn&-\int_{\T^3\times\R^3} \pr_{v_i}\sig^{i}\pr_v^\al  g^{(m+\beta)} \pr_v^\al  g^{(m+\beta)} \la v\ra^{2\ell-4|\al|}dvdx\\
  \nn&-\int_{\T^3\times\R^3}\mathcal{K}\big[\pr_v^\al g^{(m+\beta)}\big]\pr_v^\al g^{(m+\beta)}\la v\ra^{2\ell-4|\al|}dvdx\\
  \ge&(1-\fr{\zeta}{2})\big\|\pr_v^\al  g^{(m+\beta)}\big\|_{\sig,\ell-2|\al|}^2-C_{\zeta}\big\|\bar{\chi}_{\zeta}\pr_v^\al  g^{(m+\beta)}\big\|_{L^2_{x,v}}^2,
\end{align}
where $\bar{\chi}$ is chosen as \eqref{def-chi}.

If $|\al|>0$,  by Sobolev space interpolation and \eqref{coercive}, we are led to
\begin{align}\label{low-order}
    \big\|\bar{\chi}_{\zeta}\pr_v^\al  g^{(m+\beta)}\big\|_{L^2_{x,v}}^2
    \le\zeta'\big\|\pr_v^{\al}g^{(m+\beta)}\big\|_{\sig,\ell-2|\al|}^2+C_{\zeta'}\big\| g^{(m+\beta)}\big\|_{L^2_{v}(|v|\le2\zeta)}^2.
\end{align}

\noindent{\bf Step I: treatments of $\pr_v^\al Z^\beta[ Z^m, L]g$.}
More precisely, we write
\begin{align*}
    [Z^m,L]g=\sum_{1\le i\le2}{\bf c}_{\mathcal{A},i}^{m}+\sum_{1\le i\le4}{\bf c}_{\mathcal{K},i}^{m},
\end{align*}
where
\begin{align*}
    {\bf c}_{\mathcal{A},1}^{m}[g]&=-\sum_{\substack{n\le m\\
    |n|\ge1}}C_{m}^n\pr_{v_i}\big(Z^n\sig^{ij}\pr_{v_j}g^{(m-n)}\big),\\
    {\bf c}_{\mathcal{A},2}^{m}[g]&=\sum_{\substack{n\le m\\
    |n|\ge1}}C_{m}^n Z^n\big( \sig^{ij}v_iv_j - \pr_{v_i}\sig^{i} \big) g^{(m-n)},
\end{align*}
and
\begin{align*}
    {\bf c}_{\mathcal{K},1}^{m}[g]&=\sum_{\substack{0<n\in m\\ n'\le n}}C_m^nC_n^{n'}\pr_{v_i}\int_{\R^3}\Phi^{ij}(v-\tl{v})Z^{n'}\mu^{\fr12}(v)Z^{n-n'}\mu^{\fr12}(\tl{v})\pr_{v_j}g^{(m-n)}(t,x,\tl{v})d\tl{v},\\
    {\bf c}_{\mathcal{K},2}^{m}[g]&=\sum_{\substack{0<n\in m\\ n'\le n}}C_m^nC_n^{n'}\pr_{v_i}\int_{\R^3}\Phi^{ij}(v-\tl{v})Z^{n'}\mu^{\fr12}(v)Z^{n-n'}\big(\tl{v}_j\mu^{\fr12}(\tl{v})\big)g^{(m-n)}(t,x,\tl{v})d\tl{v},\\
    {\bf c}_{\mathcal{K},3}^{m}[g]&=-\sum_{\substack{0<n\in m\\ n'\le n}}C_m^nC_n^{n'}\int_{\R^3}\Phi^{ij}(v-\tl{v})Z^{n'}\big(v_i\mu^{\fr12}(v)\big)Z^{n-n'}\mu^{\fr12}(\tl{v})\pr_{v_j}g^{(m-n)}(t,x,\tl{v})d\tl{v},\\
    {\bf c}_{\mathcal{K},4}^{m}[g]&=-\sum_{\substack{0<n\in m\\ n'\le n}}C_m^nC_n^{n'}\int_{\R^3}\Phi^{ij}(v-\tl{v})Z^{n'}\big(v_i\mu^{\fr12}(v)\big)Z^{n-n'}\big(\tl{v}_j\mu^{\fr12}(\tl{v})\big) g^{(m-n)}(t,x,\tl{v})d\tl{v}.
\end{align*}
{$\diamond$ \it Estimates for $\left\la \pr_v^\al Z^\beta {\bf c}_{\mathcal{A},1}^{m}[g], \pr_v^\al g^{(m+\beta)} \right\ra_{\ell-2|\al|}$.} Integrating by parts, we have
\begin{align*}
    &\left\la \pr_v^\al Z^\beta {\bf c}_{\mathcal{A},1}^{m}[g], \pr_v^\al g^{(m+\beta)} \right\ra_{\ell-2|\al|}\\
    =&\sum_{\substack{n\le m\\
    |n|\ge1}}C_{m}^n\sum_{\substack{\al'\le\al\\
    \beta'\le\beta}}C_\al^{\al'}C_\beta^{\beta'}\int_{\mathbb{T}^3\times\R^3}\pr_v^{\al'}Z^{n+\beta'}\sig^{ij}\pr_{v_j}\pr_v^{\al-\al'} g^{(m-n+\beta-\beta')}\\ &\qquad\qquad\qquad\qquad\qquad\qquad\times\pr_{v_i}\pr_v^\al g^{(m+\beta)}\la v\ra^{2\ell-4|\al|}dvdx\\
    &+\sum_{\substack{n\le m\\
    |n|\ge1}}C_{m}^n\sum_{\substack{\al'\le\al\\
    \beta'\le\beta}}C_\al^{\al'}C_\beta^{\beta'}\int_{\mathbb{T}^3\times\R^3}\pr_v^{\al'}Z^{n+\beta'}\sig^{ij}\pr_{v_j}\pr_v^{\al-\al'} g^{(m-n+\beta-\beta')}\\ &\qquad\qquad\qquad\qquad\qquad\qquad\times\pr_v^\al g^{(m+\beta)} \pr_{v_i}\la v\ra^{2\ell-4|\al|}dvdx\\
    =&{\rm I}_{\mathcal{A},1;(1)}^m+{\rm I}_{\mathcal{A},1;(2)}^m.
\end{align*}
Note that for all $\al\in\N^3, \beta\in\N^6$ with $|\al|+|\beta|\le N_*:=N_1+N_2$, and $\al'\le\al, \beta'\le\beta$, there exists $C_{s, N_*}\ge (\fr{3N_*}{s})^{\fr{3N_*}{s}}$, such that
\begin{align}\label{losing ridus}
\fr{\Gamma_s(\al'+\bar{\beta'}+\bar{n})}{\Gamma_s(n)}\le C\prod_{i=1}^3\left(\fr{(n_{i+3}+\al'_i+\beta'_{i+3})!}{n_{i+3}!}\right)^{\fr{1}{s}}\fr{(n_{i+3}+1)^{12}}{(n_{i+3}+\al'_i+\beta'_{i+3}+1)^{12}}\le C_{s, N_*} e^{|n|}.
\end{align}
In the following, the constant $C_{s, N_*}$ may change from line to line, but always depends only on $N_*$ and $s$.

If $|\al'|+|\beta'|+|n|\ge2$, combining \eqref{eq: est sigma_ij} with \eqref{losing ridus} and \eqref{equiv-gamma}, we find that
\begin{align*}
    \left|{\rm I}_{\mathcal{A},1;(1)}^m\right|\le& C_{s,N_*}\sum_{\substack{\al'\le\al\\
    \beta'\le\beta}}C_\al^{\al'}C_\beta^{\beta'} M^{|\al'|+|\beta'|}_{\sig}\sum_{\substack{n\le m\\
    |n|\ge1}}C_{m}^n (eM_{\sig})^{|n|}\fr{\Gamma_s(|n|)}{C_{|n|}^n}\\
    &\times \left\|\la v\ra^{-\fr32}\nb_v\pr_v^{\al-\al'} g^{(m-n+\beta-\beta')} \right\|_{L^2_{x,v}(\ell-2|\al|)} \left\|\la v\ra^{-\fr32}\nb_v\pr_v^{\al} g^{(m+\beta)} \right\|_{L^2_{x,v}(\ell-2|\al|)}.
\end{align*}
It follows from this, \eqref{coercive}, and Lemmas \ref{lem-summable}, \ref{lem-convolution} that
\begin{align}\label{est-IA11}
    \nn&\sum_{\substack{m\in\N^6}}a_{m,\lm,s}^2(t)\big|{\rm I}_{\mathcal{A},1;(1)}^m\big|\\
    \nn\le&C_{s,N_*}\sum_{\substack{\al'\le\al\\
    \beta'\le\beta}}C_\al^{\al'}C_\beta^{\beta'} M^{|\al'|+|\beta'|}_{\sig}\sum_{\substack{m\in\N^6}}\sum_{\substack{n\le m\\
    |n|\ge1}}b_{m,n,s} (\lm eM_{\sig})^{|n|}\\
    \nn&\times \Big(a_{m-n,\lm,s}(t)\left\|\pr_v^{\al-\al'} g^{(m-n+\beta-\beta')} \right\|_{\sig,\ell-2|\al|+2|\al'|}\Big)
    \Big(a_{m,\lm,s}(t)\left\|\pr_v^{\al} g^{(m+\beta)} \right\|_{\sig,\ell-2|\al|}\Big)\\
    \nn\le&C_{s,N_*}\sum_{\substack{\al'\le\al\\
    \beta'\le\beta}}C_\al^{\al'}C_\beta^{\beta'} M^{|\al'|+|\beta'|}_{\sig}\Bigg(\sum_{\substack{n\in\N^6, |n|\ge1}}(\lm eM_{\sig})^{2|n|}\Bigg)^{\fr12}\\
    &\times \left\|a_{m,\lm,s}(t)\big\|\pr_v^{\al-\al'} g^{(m+\beta-\beta')} \big\|_{\sig,\ell-2|\al|+2|\al'|}\right\|_{L^2_m} \left\|a_{m,\lm,s}(t)\big\|\pr_v^{\al} g^{(m+\beta)} \big\|_{\sig,\ell-2|\al|}\right\|_{L^2_m}.
\end{align}

If $|\al'|+|\beta'|+|n|=1$, keeping in mind that $|n|\ge1$, we must have $|\al'|=|\beta'|=0$, $|n|=1$ and $Z^n=\pr_{v_l}$ for some $l\in \{1, 2, 3\}$. Thus, thanks to the fact that  $\sig^{ij}=\sig^{ji}$, integrating by parts, and using \eqref{eq: est sigma_ij} and \eqref{coercive}, we are led to
\begin{align*}
{\rm I}^m_{\mathcal{A},1;(1)}=&\sum_{|n|=1}C_m^n\int_{\mathbb{T}^3\times\mathbb{R}^3}\pr_{v_l}\sig^{ij}\pr_{v_j}\pr_v^\al g^{(m-n+\beta)}\pr_{v_l}\pr_{v_i}\pr_v^\al  g^{(m-n+\beta)}\la v\ra^{2\ell-4|\al|}dvdx\\
=&-\fr12\sum_{|n|=1}C_m^n\int_{\mathbb{T}^3\times\mathbb{R}^3}\pr_{v_l}\left(\pr_{v_l}\sig^{ij} \la v\ra^{2\ell-4|\al|}\right) \left(\pr_{v_j}\pr_v^\al g^{(m-n+\beta)}\pr_{v_i}\pr_v^\al  g^{(m-n+\beta)}\right)dvdx\\
\le&M_\sig^2\sum_{|n|=1}C_m^n\Gamma_s(2n)\big\|\la v\ra^{-\fr32}\nb_v \pr_v^\al  g^{(m-n+\beta)}\big\|_{L^2_{x,v}(\ell-2|\al|)}^2\\
\le&C M_\sig^2\sum_{|n|=1}C_m^n\Gamma_s(2n)\big\| \pr_v^\al  g^{(m-n+\beta)}\big\|_{\sig, \ell-2|\al|}^2.
\end{align*}
Then
\begin{align}\label{est-IA11'}
\nn&\sum_{\substack{m\in\N^6}}a_{m,\lm,s}^2(t)\big|{\rm I}_{\mathcal{A},1;(1)}^m\big|\\
\nn\le&C (\lm M_\sig)^2\sum_{\substack{m\in\N^6}}\sum_{\substack{n\le m, |n|=1}}\fr{C_m^n (C^m_{|m|})^2}{(C_{|m-n|}^{m-n})^2}\fr{\Gamma_s(2n)\Gamma_s(|m-n|)^2}{\Gamma_s(|m|)^2} \\
\nn&\times a_{m-n,\lm,s}^2(t)\big\| \pr_v^\al  g^{(m-n+\beta)}\big\|_{\sig, \ell-2|\al|}^2\\
\nn\le&C (\lm M_\sig)^2\sum_{\substack{|n|=1\\ n\in\in\N^6}}\sum_{\substack{m\ge n}} a_{m-n,\lm,s}^2(t)\big\| \pr_v^\al  g^{(m-n+\beta)}\big\|_{\sig, \ell-2|\al|}^2\\
\le&C (\lm M_\sig)^2\sum_{\substack{m\in\N^6 }} a_{m,\lm,s}^2(t)\big\| \pr_v^\al  g^{(m+\beta)}\big\|_{\sig, \ell-2|\al|}^2,
\end{align}
where we have used
\begin{align*}
   \sup_{m\ge n, |n|=1} \fr{C_m^n (C^m_{|m|})^2}{(C_{|m-n|}^{m-n})^2}\fr{\Gamma_s(2n)\Gamma_s(|m-n|)^2}{\Gamma_s(|m|)^2}\le C.
\end{align*}

For the estimates of ${\rm I}_{\mathcal{A},1;(2)}^m$, similar to \eqref{est-IA11}, we have
\begin{align}\label{est-IA12}
    \nn&\sum_{\substack{m\in\N^6}}a_{m,\lm,s}^2(t)\big|{\rm I}_{\mathcal{A},1;(2)}^m\big|\\
    \nn\le&C_{s,N_*}\sum_{\substack{\al'\le\al\\
    \beta'\le\beta}}C_\al^{\al'}C_\beta^{\beta'} M^{|\al'|+|\beta'|}_{\sig}\Bigg(\sum_{\substack{n\in\N^6,|n|\ge1}}(\lm eM_{\sig})^{2|n|}\Bigg)^{\fr12}\\
    \nn&\times\left\|a_{m,\lm,s}(t)\big\|\la v\ra^{-\fr32}\nb_v\pr_v^{\al-\al'} g^{(m+\beta-\beta')} \big\|_{L^2_{x,v}(\ell-2|\al|)}\right\|_{L^2_m}\\
    \nn&\times\left\|a_{m,\lm,s}(t)\big\|\la v\ra^{-\fr32}\pr_v^{\al} g^{(m+\beta)} \big\|_{L^2_{x,v}(\ell-2|\al|)}\right\|_{L^2_m}\\
    \nn\le&\lm eM_{\sig}C_{s,N_*}\left\|a_{m,\lm,s}(t)\big\|\pr_v^{\al} g^{(m+\beta)} \big\|_{\sig,\ell-2|\al|}\right\|_{L^2_m}\\
    &\times\sum_{\substack{\al'\le\al\\
    \beta'\le\beta}}C_\al^{\al'}C_\beta^{\beta'} M^{|\al'|+|\beta'|}_{\sig}\left\|a_{m,\lm,s}(t)\big\|\pr_v^{\al-\al'} g^{(m+\beta-\beta')} \big\|_{\sig,\ell-2|\al|+2|\al'|}\right\|_{L^2_m},
\end{align}
provided $\lm eM_\sig\le\fr12$.

\noindent{$\diamond$ \it Estimates for $\left\la \pr_v^\al Z^\beta {\bf c}_{\mathcal{A},2}^{m}[g], \la v\ra^{2\ell-4|\al|}\pr_v^\al Z^{\beta}g^{(m)} \right\ra_{x,v}$.} We write
\begin{align*}
    {\rm I}^m_{\mathcal{A},2}=&\sum_{\substack{n\le m\\
    |n|\ge1}}C_{m}^n\sum_{\substack{\al'\le\al\\
    \beta'\le\beta}}C_\al^{\al'}C_\beta^{\beta'}\int_{\mathbb{T}^3\times\R^3}\pr_v^{\al'}Z^{n+\beta'}\big(\sig^{ij}v_iv_j-\pr_{v_i}\sig^i\big)\pr_v^{\al-\al'} g^{(m-n+\beta-\beta')}\\ &\qquad\qquad\qquad\qquad\qquad\qquad\times\pr_v^\al g^{(m+\beta)}\la v\ra^{2\ell-4|\al|}dvdx.
\end{align*}
Thanks to \eqref{eq: est sigma_ij} and \eqref{coercive}, similar to \eqref{est-IA11} and \eqref{est-IA12}, we arrive at
\begin{align}\label{est-IA2}
    \nn&\sum_{\substack{m\in\N^6}}a_{m,\lm,s}^2(t)\big|{\rm I}_{\mathcal{A},2}^m\big|\\
    \nn\le&C_{s,N_*}\sum_{\substack{\al'\le\al\\
    \beta'\le\beta}}C_\al^{\al'}C_\beta^{\beta'} M^{|\al'|+|\beta'|}_{\sig}\Bigg(\sum_{\substack{n\in\N^6,|n|\ge1}}(\lm eM_{\sig})^{2|n|}\Bigg)^{\fr12}\\
    \nn&\times\left\|a_{m,\lm,s}(t)\big\|\la v\ra^{-1}\pr_v^{\al-\al'} g^{(m+\beta-\beta')} \big\|_{L^2_{x,v}(\ell-2|\al|)}\right\|_{L^2_m}\\
    \nn&\times\left\|a_{m,\lm,s}(t)\big\|\la v\ra^{-1}\pr_v^{\al} g^{(m+\beta)} \big\|_{L^2_{x,v}(\ell-2|\al|)}\right\|_{L^2_m}\\
    \nn\le&\lm eM_{\sig}C_{s,N_*}\left\|a_{m,\lm,s}(t)\big\|\pr_v^{\al} g^{(m+\beta)} \big\|_{\sig,\ell-2|\al|}\right\|_{L^2_m}\\
    &\times\sum_{\substack{\al'\le\al\\
    \beta'\le\beta}}C_\al^{\al'}C_\beta^{\beta'} M^{|\al'|+|\beta'|}_{\sig}\left\|a_{m,\lm,s}(t)\big\|\pr_v^{\al-\al'} g^{(m+\beta-\beta')} \big\|_{\sig,\ell-2|\al|+2|\al'|}\right\|_{L^2_m},
\end{align}
provided $\lm eM_\sig\le\fr12$.

\noindent{$\diamond$ \it Estimates for $\left\la \pr_v^\al Z^\beta{\bf c}_{\mathcal{K},1}^{m}[g],\pr_v^\al g^{(m+\beta)}\right\ra_{\ell-2|\al|}$.} We further split it into two parts by integrating by parts:
\begin{align*}
    &\left\la \pr_v^\al Z^\beta{\bf c}_{\mathcal{K},1}^{m}[g],\pr_v^\al g^{(m+\beta)}\right\ra_{\ell-2|\al|}\\
    =&-\sum_{\substack{\al'\le\al,\al''\le\al'\\
    \beta'\le\beta,\beta''\le\beta'}}C_\al^{\al'}C_{\al'}^{\al''}C_\beta^{\beta'}C_{\beta'}^{\beta''}\sum_{\substack{0<n\le m\\ n'\le n}}C_m^nC_n^{n'}\int_{\mathbb{T}^3\times\mathbb{R}^3}\int_{\R^3}\Phi^{ij}(v-\tl{v})\\
    &\times\pr_v^{\al''}Z^{n'+\beta''}\mu^{\fr12}(v)\pr_v^{\al'-\al''}Z^{n-n'+\beta'-\beta''}\mu^{\fr12}(\tl{v})\\
    &\times\pr_{v_j}\pr_v^{\al-\al'} g^{(m-n+\beta-\beta')}(t,x,\tl{v})d\tl{v} \pr_{v_i}\pr_v^\al g^{(m+\beta)}(t,x,v) \la v\ra^{2\ell-4|\al|}dvdx\\
    &-\sum_{\substack{\al'\le\al,\al''\le\al'\\
    \beta'\le\beta,\beta''\le\beta'}}C_\al^{\al'}C_{\al'}^{\al''}C_\beta^{\beta'}C_{\beta'}^{\beta''}\sum_{\substack{0<n\le m\\ n'\le n}}C_m^nC_n^{n'}\int_{\mathbb{T}^3\times\mathbb{R}^3}\int_{\R^3}\Phi^{ij}(v-\tl{v})\\
    &\times\pr_v^{\al''}Z^{n'+\beta''}\mu^{\fr12}(v)\pr_v^{\al'-\al''}Z^{n-n'+\beta'-\beta''}\mu^{\fr12}(\tl{v})\\
    &\times\pr_{v_j}\pr_v^{\al-\al'} g^{(m-n+\beta-\beta')}(t,x,\tl{v})d\tl{v} \pr_v^\al g^{(m+\beta)}(t,x,v) \pr_{v_i}\la v\ra^{2\ell-4|\al|}dvdx\\
    =&{\rm I}_{\mathcal{K},1;(1)}^m+{\rm I}_{\mathcal{K},1;(2)}^m.
\end{align*}
In view of \eqref{losing ridus}, \eqref{e-prmu}, and \eqref{bound-product-Gamma}, for all $|\al|+|\beta|\le N_*, n\in\N^6$, and $\al'\le\al, \beta'\le\beta$
\begin{align}
\nn&\sum_{\substack{\al''\le \al'\\
\beta''\le\beta'\\
n'\le n}}C_{\al'}^{\al''}C_{\beta'}^{\beta''}C_n^{n'}\left|\pr_v^{\al''}Z^{n'+\beta''}\mu^{\fr12}(v)\pr_v^{\al'-\al''}Z^{n-n'+\beta'-\beta''}\mu^{\fr12}(\tl{v})\right|\\
    \nn\le& \sum_{\substack{\al''\le \al'\\
\beta''\le\beta'\\
n'\le n}}C_{\al'}^{\al''}C_{\beta'}^{\beta''}C_n^{n'}M_{\mu}^{|\al'|+|\beta'|+|n|}\Gamma_1(\al''+\tl{\beta}''+\tl{n}')\Gamma_1(\al'-\al''+\tl{\beta}'-\tl{\beta}''+\tl{n}-\tl{n}')\mu^{\fr14}(v)\mu^{\fr14}(\tl{v})\\
    \nn\le&C_{s,N_*}M_{\mu}^{|\al'|+|\beta'|+|n|} \sum_{n'\le n}C_n^{n'}\Gamma_1(n')\Gamma_1(n-n')e^{|n'|+|n-n'|}\mu^{\fr14}(v)\mu^{\fr14}(\tl{v})\\
    \nn\le&C_{s,N_*}M_{\mu}^{|\al'|+|\beta'|+|n|}\Gamma_1(n)e^{|n|}\mu^{\fr14}(v)\mu^{\fr14}(\tl{v}).
\end{align}
Thus,  by Cauchy-Schwarz inequality first w.r.t. the $(v,\tl{v})$ variable and then to the $x$ variable, we get
\begin{align}\label{cs-vtlv}
    \big|{\rm I}_{\mathcal{K},1;(1)}^m \big|
    \nn\le&C_{s,N_*}\sum_{\substack{\al'\le\al\\
    \beta'\le\beta}}C_\al^{\al'}C_\beta^{\beta'}M_{\mu}^{|\al'|+|\beta'|}\sum_{\substack{0<n\le m}}C_m^n \left(eM_{\mu}\right)^{|n|}\Gamma_1(n)\Bigg\|\fr{\mu^{\fr18}(v)\mu^{\fr18}(\tl{v})}{|v-\tl{v}|}
    \Bigg\|_{L^2(\R^3\times\R^3)}\\
    \nn&\times \left\|\big|\nb_{v}\pr_v^{\al-\al'} g^{(m-n+\beta-\beta')}(t,x,\tl{v})\big| \right.\\
    \nn&\qquad\qquad\qquad\cdot\left.\big|\nb_v\pr_v^\al g^{(m+\beta)}(t,x,v)\big|  \la v\ra^{2\ell-4|\al|} \mu^{\fr18}(v)\mu^{\fr18}(\tl{v})\right\|_{L^1(\T^3;L^2(\R^3\times\R^3))}\\
    \nn\le&C_{s,N_*}\sum_{\substack{\al'\le\al\\
    \beta'\le\beta}}C_\al^{\al'}C_\beta^{\beta'}M_{\mu}^{|\al'|+|\beta'|}\sum_{\substack{0<n\le m}}C_m^n \left(eM_{\mu}\right)^{|n|}\Gamma_1(n)\\
    &\times \big\|\pr_v^{\al-\al'} g^{(m-n+\beta-\beta')} \big\|_{\sig,\ell-2|\al|+2|\al'|}\big\|\pr_v^{\al} g^{(m+\beta)} \big\|_{\sig,\ell-2|\al|}.
\end{align}
Then similar to \eqref{est-IA11}, for $\lm eM_{\mu}\fr12$, we have
\begin{align}\label{es-IK}
    \nn&\sum_{m\in\N^6}a_{m,\lm,s}^2(t)\big|{\rm I}_{\mathcal{K},1;(1)}^m \big|\\
    \nn\le&(\lm eM_{\mu})C_{s,N_*} \left\|a_{m,\lm,s}(t)\big\|\pr_v^{\al} g^{(m+\beta)} \big\|_{\sig,\ell-2|\al|}\right\|_{L^2_m}\sum_{\substack{\al'\le\al\\
    \beta'\le\beta}}C_\al^{\al'}C_\beta^{\beta'} M_{\mu}^{|\al'|+|\beta'|}\\
    &\times\left\|a_{m,\lm,s}(t)\big\|\pr_v^{\al-\al'} g^{(m+\beta-\beta')} \big\|_{\sig,\ell-2|\al|+2|\al'|}\right\|_{L^2_m}.
\end{align}
The rest commutator terms stemming from $\mathcal{K}$ can be treated in a similar manner, and  all of them can be bounded by the right hand side of \eqref{es-IK}.  We omit the details to avoid unnecessary repetition. 

Combining the estimates \eqref{est-IA11}--\eqref{est-IA2} and \eqref{es-IK}, for $\lm e\max\{M_\sig,M_\mu\}\le\fr12$, there exists a constant $C_{s,N_1,N_2,M_\sig, M_\mu}$, depending on $s, N_1, N_2, M_\sig$ and $M_\mu$, such that 
\begin{align}\label{es-comZm}
    \nn&\sum_{m\in \N^6}\sum_{|\al|\le N_{1},|\beta|\le N_{2}}\kappa^{2|\al|+2|\beta|}a_{m,\lm,s}^2(t)\left\la\pr_v^\al Z^{\beta}[ Z^m, L]g,  \pr_v^\al g^{(m+\beta)} \right\ra_{\ell-2|\al|}\\
    \nn\le&\lm C_{s,N_*}\max\{M_\sig,M_\mu\}\sum_{\substack{|\al|\le N_1,
    |\beta|\le N_2}}\Big(\kappa \max\{M_{\sig},M_{\mu}\}\Big)^{|\al|+|\beta|} \\
    \nn &\times\sum_{m\in \N^6}\sum_{|\al|\le N_{1},|\beta|\le N_{2}}\kappa^{2|\al|+2|\beta|}a_{m,\lm,s}^2(t)\big\|\pr_v^{\al} g^{(m+\beta)} \big\|_{\sig,\ell-2|\al|}^2\\
    \le&\lm C_{s,N_1,N_2, M_\sig,M_\mu} \sum_{m\in \N^6}\sum_{|\al|\le N_{1},|\beta|\le N_{2}}\kappa^{2|\al|+2|\beta|}a_{m,\lm,s}^2(t)\big\|\pr_v^{\al} g^{(m+\beta)} \big\|_{\sig,\ell-2|\al|}^2.
\end{align}

\noindent{\bf Step II: treatments of $[\pr_v^\al Z^\beta, L]g^{(m)}$.} Now
\begin{align*}
[\pr_v^\al Z^\beta, L]g^{(m)}=\sum_{1\le i\le 2}{\bf c}_{\mathcal{A},i}^{\al,\beta}+\sum_{1\le i\le4}{\bf c}_{\mathcal{K},i}^{\al,\beta},
\end{align*}
where
\begin{align*}
    {\bf c}_{\mathcal{A},1}^{\al,\beta}=&-\sum_{\substack{\al'\le\al,\beta'\le\beta\\
    |\al'|+|\beta'|\ge1}}C_{\al}^{\al'}C_{\beta}^{\beta'}\pr_{v_i}\Big(\pr_v^{\al'}Z^{\beta'}\sig^{ij}\pr_{v_j}\pr_v^{\al-\al'}g^{(m+\beta-\beta')}\Big),\\
    {\bf c}_{\mathcal{A},2}^{\al,\beta}=&\sum_{\substack{\al'\le\al,\beta'\le\beta\\
    |\al'|+|\beta'|\ge1}}C_{\al}^{\al'}C_{\beta}^{\beta'}\pr_v^{\al'}Z^{\beta'}\Big(\sig^{ij}v_iv_j-\pr_{v_i}\sig^{i}\Big)\pr_v^{\al-\al'}g^{(m+\beta-\beta')},
\end{align*}
and
\begin{align*}
    {\bf c}_{\mathcal{K},1}^{\al,\beta}&=\sum_{\substack{\al'\le\al,\beta'\le\beta\\
    |\al'|+|\beta'|\ge1}}\sum_{\substack{\al''\le\al'\\
    \beta''\le\beta'}}C_{\al}^{\al'}C_{\al'}^{\al''}C_{\beta}^{\beta'}C_{\beta'}^{\beta''}\pr_{v_i}\int_{\R^3}\Phi^{ij}(v-\tl{v})\pr_v^{\al''}Z^{\beta''}\mu^{\fr12}(v)\pr_v^{\al'-\al''}Z^{\beta'-\beta''}\mu^{\fr12}(\tl{v})\\
    &\qquad\qquad\qquad\qquad\qquad\qquad\qquad\qquad\times\pr_{v_j}\pr_v^{\al-\al'}g^{(m+\beta-\beta')}(t,x,\tl{v})d\tl{v},\\
    {\bf c}_{\mathcal{K},2}^{\al,\beta}&=\sum_{\substack{\al'\le\al,\beta'\le\beta\\
    |\al'|+|\beta'|\ge1}}\sum_{\substack{\al''\le\al'\\
    \beta''\le\beta'}}C_{\al}^{\al'}C_{\al'}^{\al''}C_{\beta}^{\beta'}C_{\beta'}^{\beta''}\pr_{v_i}\int_{\R^3}\Phi^{ij}(v-\tl{v})\pr_v^{\al''}Z^{\beta''}\mu^{\fr12}(v)\pr_v^{\al'-\al''}Z^{\beta'-\beta''}\big(\tl{v}_j\mu^{\fr12}(\tl{v})\big)\\
    &\qquad\qquad\qquad\qquad\qquad\qquad\qquad\qquad\times \pr_v^{\al-\al'}g^{(m+\beta-\beta')}(t,x,\tl{v})d\tl{v},\\
    {\bf c}_{\mathcal{K},3}^{\al,\beta}&=-\sum_{\substack{\al'\le\al,\beta'\le\beta\\
    |\al'|+|\beta'|\ge1}}\sum_{\substack{\al''\le\al'\\
    \beta''\le\beta'}}C_{\al}^{\al'}C_{\al'}^{\al''}C_{\beta}^{\beta'}C_{\beta'}^{\beta''}\int_{\R^3}\Phi^{ij}(v-\tl{v})\pr_v^{\al''}Z^{\beta''}\big(v_i\mu^{\fr12}(v)\big)\pr_v^{\al'-\al''}Z^{\beta'-\beta''}\mu^{\fr12}(\tl{v})\\
    &\qquad\qquad\qquad\qquad\qquad\qquad\qquad\qquad\times\pr_{v_j}\pr_v^{\al-\al'}g^{(m+\beta-\beta')}(t,x,\tl{v})d\tl{v},\\
    {\bf c}_{\mathcal{K},4}^{\al,\beta}&=-\sum_{\substack{\al'\le\al,\beta'\le\beta\\
    |\al'|+|\beta'|\ge1}}\sum_{\substack{\al''\le\al'\\
    \beta''\le\beta'}}C_{\al}^{\al'}C_{\al'}^{\al''}C_{\beta}^{\beta'}C_{\beta'}^{\beta''}\int_{\R^3}\Phi^{ij}(v-\tl{v})\pr_v^{\al''}Z^{\beta''}\big(v_i\mu^{\fr12}(v)\big)\pr_v^{\al'-\al''}Z^{\beta'-\beta''}\big(\tl{v}_j\mu^{\fr12}(\tl{v})\big)\\
    &\qquad\qquad\qquad\qquad\qquad\qquad\qquad\qquad\times\pr_v^{\al-\al'}g^{(m+\beta-\beta')}(t,x,\tl{v})d\tl{v}.
\end{align*}

Similar to the proof of Lemma 4.6 in \cite{chaturvedi2023vlasov}, we have
\begin{align}\label{es-cA1-beta}
    \nn&\left\la {\bf c}_{\mathcal{A},1}^{\al,\beta}, \pr_v^{\al}g^{(m+\beta)}\right\ra_{\ell-2|\al|}\\
   \nn=& \sum_{\substack{\al'\le\al,\beta'\le\beta\\
    |\al'|+|\beta'|\ge1}}C_{\al}^{\al'}C_{\beta}^{\beta'}\int_{\T^3\times\R^3}\pr_v^{\al'}Z^{\beta'}\sig^{ij}\pr_{v_j}\pr_v^{\al-\al'}g^{(m+\beta-\beta')}\pr_{v_i}\pr_v^\al g^{(m+\beta)} \la v\ra^{2\ell-4|\al|}dvdx\\
    \nn&+\sum_{\substack{\al'\le\al,\beta'\le\beta\\
    |\al'|+|\beta'|\ge1}}C_{\al}^{\al'}C_{\beta}^{\beta'}\int_{\T^3\times\R^3}\pr_v^{\al'}Z^{\beta'}\sig^{ij}\pr_{v_j}\pr_v^{\al-\al'}g^{(m+\beta-\beta')}\pr_v^\al g^{(m+\beta)} \pr_{v_i}\la v\ra^{2\ell-4|\al|}dvdx\\
   \nn \le& C\sum_{\substack{\al'\le\al,\beta'\le\beta\\
    |\al'|+|\beta'|\ge1}}\big\|\pr_v^{\al-\al'}g^{(m+\beta-\beta')}\big\|_{\sig,\ell-2|\al|{ +2|\al'|}}\big\|\pr_v^{\al}g^{(m+\beta)}\big\|_{\sig,\ell-2|\al|}\\
    &+C\sum_{\al'\le\al,|\al'|=1}\big\|\pr_v^{\al-\al'} g^{(m+\beta)}\big\|_{\sig,\ell-2|\al|{ +2|\al'|}}^2+C\sum_{\beta'\le\beta,|\beta'|=1}\big\|\pr_v^{\al} g^{(m+\beta-\beta')}\big\|_{\sig,\ell-2|\al|}^2,
\end{align}
and
\begin{align}\label{es-cA2-beta}
    \nn&\left\la {\bf c}_{\mathcal{A},2}^{\al,\beta}, \pr_v^{\al}g^{(m+\beta)}\right\ra_{\ell-2|\al|}\\
    \nn\le&C\sum_{\substack{\al'\le\al,\beta'\le\beta\\
    |\al'|+|\beta'|\ge1}}\int_{\mathbb{T}^3\times\R^3}\la v\ra^{2\ell-4|\al|-2}\big|\pr_v^{\al-\al'}g^{(m+\beta-\beta')}\big|\big|\pr_v^{\al}g^{(m+\beta)}\big| dvdx\\
    \le&C\sum_{\substack{\al'\le\al,\beta'\le\beta\\
    |\al'|+|\beta'|\ge1}}\big\|\pr_v^{\al-\al'}  g^{(m+\beta-\beta')}\big\|_{\sig, \ell-2|\al|{ +2|\al'|}}\big\|\pr_v^\al  g^{(m+\beta)}\big\|_{\sig, \ell-2|\al|}.
\end{align}
Thus, for $i=1,2$, there exists a constant $C_{N_1,N_2}$ depending only on $N_1$ and $N_2$, such that
\begin{align}\label{es-cAi}
    \nn&\sum_{m\in\N^6}\sum_{|\al|\le N_{1},|\beta|\le N_{2}}\kappa^{2|\al|+2|\beta|}a_{m,\lm,s}^2(t)\left\la {\bf c}_{\mathcal{A},i}^{\al,\beta}, \pr_v^\al g^{(m+\beta)}\right\ra_{\ell-2|\al|}\\
    \le&\kappa C_{N_1,N_2}\sum_{m\in \N^6}\sum_{|\al|\le N_{1},|\beta|\le N_{2}}\kappa^{2|\al|+2|\beta|}a_{m,\lm,s}^2(t)\big\|\pr_v^\al  g^{(m+\beta)}\big\|_{\sig, \ell-2|\al|}^2.
\end{align}

Similar to \eqref{cs-vtlv}, for $i=1,2,3,4$, we find that
\begin{align}\label{es-cK-beta}
    \nn&\left\la {\bf c}_{\mathcal{K},i}^{\al,\beta}, \pr_v^{\al}g^{(m+\beta)}\right\ra_{\ell-2|\al|}\\
    \le& C\sum_{\substack{\al'\le\al,\beta'\le\beta\\
    |\al'|+|\beta'|\ge1}}\big\|\pr_v^{\al-\al'}  g^{(m+\beta-\beta')}\big\|_{\sig,\ell-2|\al|+2|\al'|}\big\|\pr_v^\al  g^{(m+\beta)}\big\|_{\sig,\ell-2|\al|}.
\end{align}
Thus, \eqref{es-cAi} still holds with ${\bf c}_{\mathcal{A},i}^{\al,\beta}$ replaced by ${\bf c}_{\mathcal{K},i}^{\al,\beta}$, $i=1,2,3,4$.

Taking $\zeta'=\fr{\zeta}{2C_{\zeta}}$ in \eqref{low-order} with $C_{\zeta}$ the constant appearing in \eqref{L2-Landau}, then we conclude  from \eqref{L2-Landau}, \eqref{low-order}, \eqref{es-comZm} and \eqref{es-cAi}  that \eqref{coer-Landau} holds.
\end{proof}

\begin{rem}\label{rem-Y/t}
If $\pr_v^\al$ is replaced by $(1+t)^{-1}Y^\al$ with $|\al|=1$, then \eqref{coer-Landau} still holds with minor changes. Indeed, one can follow {\bf Step I} and {\bf Step II}  line by line to get the corresponding estimates for the commutators $(1+t)^{-1}Y^\al Z^\beta[Z^m,L]g$ and $(1+t)^{-1}[Y^\al Z^\beta,L]Z^m g$. The only difference lies in the treatment of the lower-order term in \eqref{L2-Landau}. More precisely, similar to \eqref{low-order}, using the fact $Y^\al=\pr_v^\al+t\pr_x^\al$ due to $|\al|=1$, similar to \eqref{low-order},   we have
\begin{align}\label{low-order-Y}
   \nn &(1+t)^{-2}\big\|\bar{\chi}_{\zeta}Y^\al  g^{(m+\beta)}\big\|_{L^2_{v}}^2\le 2(1+t)^{-2}\big\|\bar{\chi}_{\zeta}\pr_v^\al  g^{(m+\beta)}\big\|_{L^2_{v}}^2+2\big\|\bar{\chi}_{\zeta}\pr_x^\al  g^{(m+\beta)}\big\|_{L^2_{v}}^2 \\
    \nn\le&\fr{\zeta'}{2}(1+t)^{-2}\big|\pr_v^{\al} g^{(m+\beta)}\big|_{\sig,\ell-2|\al|}^2+C_{\zeta'}\big\| g^{(m+\beta)}\big\|_{L^2_{v}(|v|\le2\zeta)}^2+2\big\|\nb_x  g^{(m+\beta)}\big\|_{L^2_{v}(|v|\le2\zeta)}^2\\
    \nn\le& \zeta'(1+t)^{-2}\big|Y^{\al} g^{(m+\beta)}\big|_{\sig,\ell-2|\al|}^2+\zeta'\big|\pr_x^{\al} g^{(m+\beta)}\big|_{\sig,\ell-2|\al|}^2\\
    &+C_{\zeta'}\big\| g^{(m+\beta)}\big\|_{L^2_{v}(|v|\le2\zeta)}^2+2\big\|\nb_x  g^{(m+\beta)}\big\|_{L^2_{v}(|v|\le2\zeta)}^2.
\end{align}
\end{rem}

\begin{coro}\label{coro-cross}
Let $\ell>2$, and $\lm$ and $\kappa$ be chosen as in Lemma \ref{lem-coercive-L}. Then  there exist  positive constants $C_{s,N,M_\sig,M_\mu}$, and $C_N$, and a universal constant $C$, such that
    \begin{align}\label{es-inner}
\nn&\sum_{m\in\N^6}\sum_{\substack{\beta\in\N^6\\|\beta|\le N}}\kappa^{2|\beta|+1}a_{m,\lm,s}^2(t)\bigg(\left\la Z^{m+\beta}L\nb_{x}g,  \nb_{v}g^{(m+\beta)}\right\ra_{\ell-2}\\
\nn&\quad\quad\quad\quad\quad\quad\quad\quad\quad+ \left\la \nb_{v}Z^{m+\beta} Lg, \nb_{x}g^{(m+\beta)}\right\ra_{\ell-2}\bigg)\\
\le&\big(C+\lm C_{s,N,M_\sig, M_\mu}+\kappa C_N\big) \|\nb_xg\|_{\mathcal{G}^{\lm,N}_{s,\sig,\ell-2}}\sum_{|\al|\le1}\|\kappa\pr_v^\al g\|_{\mathcal{G}^{\lm,N}_{s,\sig,\ell-2|\al|}}.
    \end{align}
\end{coro}
\begin{proof}
We first write the cross terms as follows:
    \begin{align}\label{cross-terms}
    \nn&\left\la Z^{m+\beta}L\pr_{x_j}g,  \la v\ra^{2\ell-4}\pr_{v_j}g^{(m+\beta)}\right\ra_{x,v}+ \left\la \pr_{v_j}Z^{m+\beta} Lg, \la v\ra^{2\ell-4}\pr_{x_j}g^{(m+\beta)}\right\ra_{x,v}\\
    \nn=&\left\la L\pr_{x_j}g^{(m+\beta)},  \la v\ra^{2\ell-4}\pr_{v_j}g^{(m+\beta)}\right\ra_{x,v}+ \left\la \pr_{v_j} Lg^{(m+\beta)}, \la v\ra^{2\ell-4}\pr_{x_j}g^{(m+\beta)}\right\ra_{x,v}\\
    \nn&+\left\la Z^\beta[Z^{m},L]\pr_{x_j}g,  \la v\ra^{2\ell-4}\pr_{v_j}g^{(m+\beta)}\right\ra_{x,v}+\left\la [Z^{\beta},L]\pr_{x_j}g^{(m)},  \la v\ra^{2\ell-4}\pr_{v_j}g^{(m+\beta)}\right\ra_{x,v}\\
    &+ \left\la \pr_{v_j}Z^\beta[Z^{m}, L]g, \la v\ra^{2\ell-4}\pr_{x_j}g^{(m+\beta)}\right\ra_{x,v}+ \left\la \pr_{v_j}[Z^{\beta}, L]g^{(m)}, \la v\ra^{2\ell-4}\pr_{x_j}g^{(m+\beta)}\right\ra_{x,v}.
\end{align}
By Lemma \ref{coro-upL-1}, we have
\begin{align*}
    &\left\la L\pr_{x_j}g^{(m+\beta)},  \pr_{v_j}g^{(m+\beta)}\right\ra_{\ell-2}+ \left\la \pr_{v_j} Lg^{(m+\beta)}, \pr_{x_j}g^{(m+\beta)}\right\ra_{\ell-2}\\
    \le&C \big\| \pr_{x_j}g^{(m+\beta)}\big\|_{\sig,\ell-2}\left(\big\|g^{(m+\beta)}\big\|_{\sig,\ell-2}+\big\| \nb_{v}g^{(m+\beta)}\big\|_{\sig,\ell-2}\right).
\end{align*}
Then 
\begin{align}\label{cross-main}
    \nn& \sum_{m\in\N^6}\sum_{\substack{\beta\in\N^6\\ |\beta|\le N}}\kappa^{2|\beta|+1}a_{m,\lm,s}^2(t)\Big(\big\la L\nb_{x}g^{(m+\beta)},  \nb_{v}g^{(m+\beta)}\big\ra_{\ell-2}\\
    \nn&\quad\quad\quad\quad\quad\quad\quad\quad\quad\quad\quad +\big\la \nb_{v} Lg^{(m+\beta)}, \nb_{x}g^{(m+\beta)}\big\ra_{\ell-2}\Big)\\
    \le&C \|\nb_xg\|_{\mathcal{G}^{\lm,N}_{s,\sig,\ell-2}}\sum_{|\al|\le1}\|\kappa\pr_v^\al g\|_{\mathcal{G}^{\lm,N}_{s,\sig,\ell-2|\al|}}.
\end{align}

For the commutators in the cross terms \eqref{cross-terms}, we use the notations in Lemma \ref{lem-coercive-L} to write
\begin{align*}
    &\left\la Z^\beta[Z^{m},L]\pr_{x_l}g,  \la v\ra^{2\ell-4}\pr_{v_l}g^{(m+\beta)}\right\ra_{x,v}+ \left\la \pr_{v_l}Z^\beta[Z^{m}, L]g, \la v\ra^{2\ell-4}\pr_{x_l}g^{(m+\beta)}\right\ra_{x,v}\\
    =&\sum_{1\le \iota\le2}\int_{\T^3\times\R^3}\left(Z^\beta{\bf c}_{\mathcal{A},\iota}^m[\pr_{x_l}g]\pr_{v_l}g^{(m+\beta)}+\pr_{v_l}Z^\beta{\bf c}_{\mathcal{A},\iota}^m[g]\pr_{x_l}g^{(m+\beta)}\right)\la v\ra^{2\ell-4}dvdx\\
    &+\sum_{1\le \iota\le4}\int_{\T^3\times\R^3}\left(Z^\beta{\bf c}_{\mathcal{K},\iota}^m[\pr_{x_l}g]\pr_{v_l}g^{(m+\beta)}+\pr_{v_l}Z^\beta{\bf c}_{\mathcal{K},\iota}^m[g]\pr_{x_l}g^{(m+\beta)}\right)\la v\ra^{2\ell-4}dvdx\\
    =&\sum_{1\le\iota \le2}{\bf C}^m_{\mathcal{A},\iota}+\sum_{1\le\iota\le 4}{\bf C}^m_{\mathcal{K},\iota},
\end{align*}
and similarly,
\begin{align*}
    &\left\la [Z^{\beta},L]\pr_{x_l}g^{(m)},  \la v\ra^{2\ell-4}\pr_{v_l}g^{(m+\beta)}\right\ra_{x,v}+ \left\la \pr_{v_l}[Z^{\beta}, L]g^{(m)}, \la v\ra^{2\ell-4}\pr_{x_l}g^{(m+\beta)}\right\ra_{x,v}\\
    =&\sum_{1\le \iota\le2}\int_{\T^3\times\R^3}\left({\bf c}_{\mathcal{A},\iota}^{\beta}\big[\pr_{x_l}g^{(m)}\big]\pr_{v_l}g^{(m+\beta)}+\pr_{v_l}{\bf c}_{\mathcal{A},\iota}^{\beta}\big[g^{(m)}\big]\pr_{x_l}g^{(m+\beta)}\right)\la v\ra^{2\ell-4}dvdx\\
    &+\sum_{1\le \iota\le4}\int_{\T^3\times\R^3}\left({\bf c}_{\mathcal{K},\iota}^{\beta}\big[\pr_{x_l}g^{(m)}\big]\pr_{v_l}g^{(m+\beta)}+\pr_{v_l}{\bf c}_{\mathcal{K},\iota}^{\beta}\big[g^{(m)}\big]\pr_{x_l}g^{(m+\beta)}\right)\la v\ra^{2\ell-4}dvdx\\
    =&\sum_{1\le\iota \le2}{\bf C}^{\beta}_{\mathcal{A},\iota}+\sum_{1\le\iota \le4}{\bf C}^{\beta}_{\mathcal{K},\iota}.
\end{align*}
{\bf Treatments of ${\bf C}_{\mathcal{A},\iota}^m$ and ${\bf C}_{\mathcal{K},\iota}^m$.} Similar to the treatments of $\left\la \pr_v^\al Z^\beta {\bf c}_{\mathcal{A},i}^{m}, \pr_v^\al g^{(m+\beta)} \right\ra_{\ell-2|\al|}, i=1,2$ in Lemma \ref{lem-coercive-L}, we have
\begin{align*}
{\bf C}^m_{\mathcal{A},1}
=&\sum_{\substack{n\le m,|n|\ge1\\\beta'\le\beta}}C_m^nC_\beta^{\beta'}\int_{\T^3\times\R^3}Z^{n+\beta'}\sig^{ij}\pr_{v_j}\pr_{x_l}g^{(m-n+\beta-\beta')}\pr_{v_i}\pr_{v_l}g^{(m+\beta)}\la v\ra^{2\ell-4}dvdx\\
&+\sum_{\substack{n\le m,|n|\ge1\\\beta'\le\beta}}C_m^nC_\beta^{\beta'}\int_{\T^3\times\R^3}Z^{n+\beta'}\sig^{ij}\pr_{v_j}\pr_{v_l}g^{(m-n+\beta-\beta')}\pr_{v_i}\pr_{x_l}g^{(m+\beta)}\la v\ra^{2\ell-4}dvdx\\
&+\sum_{\substack{n\le m,|n|\ge1\\\beta'\le\beta}}C_m^nC_\beta^{\beta'}\int_{\T^3\times\R^3}\pr_{v_l}Z^{n+\beta'}\sig^{ij}\pr_{v_j}g^{(m-n+\beta-\beta')}\pr_{v_i}\pr_{x_l}g^{(m+\beta)}\la v\ra^{2\ell-4}dvdx\\
&+\sum_{\substack{n\le m,|n|\ge1\\\beta'\le\beta}}C_m^nC_\beta^{\beta'}\int_{\T^3\times\R^3}Z^{n+\beta'}\sig^{ij}\pr_{v_j}\pr_{x_l}g^{(m-n+\beta-\beta')}\pr_{v_l}g^{(m+\beta)}\pr_{v_i}\la v\ra^{2\ell-4}dvdx\\
&-\sum_{\substack{n\le m,|n|\ge1\\\beta'\le\beta}}C_m^nC_\beta^{\beta'}\int_{\T^3\times\R^3}Z^{n+\beta'}\sig^{ij}\pr_{v_j}g^{(m-n+\beta-\beta')}\pr_{v_l}\pr_{x_l}g^{(m+\beta)}\pr_{v_i}\la v\ra^{2\ell-4}dvdx\\
&-\sum_{\substack{n\le m,|n|\ge1\\\beta'\le\beta}}C_m^nC_\beta^{\beta'}\int_{\T^3\times\R^3}Z^{n+\beta'}\sig^{ij}\pr_{v_j}g^{(m-n+\beta-\beta')}\pr_{x_l}g^{(m+\beta)}\pr_{v_l}\pr_{v_i}\la v\ra^{2\ell-4}dvdx\\
=&\sum_{1\le\iota\le6}{\bf C}_{\mathcal{A},1;(\iota)}^m,
\end{align*}
and
\begin{align*}
    {\bf C}_{\mathcal{A},2}^m=&\sum_{\substack{n\le m,|n|\ge1\\\beta'\le\beta}}C_m^nC_\beta^{\beta'}\int_{\T^3\times\R^3}Z^{n+\beta'}\big(\sig^{ij}v_iv_j-\pr_{v_i}\sig^i\big)\pr_{x_l}g^{(m-n+\beta-\beta')}\pr_{v_l}g^{(m+\beta)}\la v\ra^{2\ell-4}dvdx\\
    &-\sum_{\substack{n\le m,|n|\ge1\\\beta'\le\beta}}C_m^nC_\beta^{\beta'}\int_{\T^3\times\R^3}Z^{n+\beta'}\big(\sig^{ij}v_iv_j-\pr_{v_i}\sig^i\big)g^{(m-n+\beta-\beta')}\pr_{v_l}\pr_{x_l}g^{(m+\beta)}\la v\ra^{2\ell-4}dvdx\\
    &-\sum_{\substack{n\le m,|n|\ge1\\\beta'\le\beta}}C_m^nC_\beta^{\beta'}\int_{\T^3\times\R^3}Z^{n+\beta'}\big(\sig^{ij}v_iv_j-\pr_{v_i}\sig^i\big)g^{(m-n+\beta-\beta')}\pr_{x_l}g^{(m+\beta)}\pr_{v_l}\la v\ra^{2\ell-4}dvdx.
\end{align*}

If $|\beta'|+|n|=1$, then $|\beta'|=0$ and $|n|=1$ since $|n|\ge1$. We infer from this, together with  the fact  $\sig^{ij}=\sig^{ij}$ and \eqref{eq: est sigma_ij}, that
\begin{align}\label{i-b-p}
    \nn&{\bf C}^m_{\mathcal{A},1;(1)}+{\bf C}^m_{\mathcal{A},1;(2)}\\
    \nn=&\sum_{n\le m,|n|=1}C_m^n\int_{\T^3\times\R^3}Z^n\sig^{ij}\Big(\pr_{v_j}\pr_{x_l} g^{(m-n+\beta)}Z^n\pr_{v_i}\pr_{v_l} g^{(m-n+\beta)}\\
    \nn&\qquad\qquad\qquad\qquad\qquad\qquad+\pr_{v_i}\pr_{v_l}Z^\beta g^{(m-n)}Z^n\pr_{v_j}\pr_{x_l} g^{(m-n+\beta)}\Big)\la v\ra^{2\ell-4}dvdx\\
    \nn=&-\sum_{n\le m,|n|=1}C_m^n\int_{\T^3\times\R^3}Z^n\big(\la v\ra^{2\ell-4}Z^n\sig^{ij}\big)\pr_{v_j}\pr_{x_l}Z^\beta g^{(m-n)}\pr_{v_i}\pr_{v_l}Z^\beta g^{(m-n)}dvdx\\
    \le&CM_{\sig}^2\sum_{n\le m,|n|=1}C_m^n\Gamma_s(2n)\big\|\pr_{x_l}g^{(m-n+\beta)}\big\|_{\sig,\ell-2}\big\|\pr_{v_l}g^{(m-n+\beta)}\big\|_{\sig,\ell-2}.
\end{align}
Then similar to \eqref{est-IA11'}, we find that
\begin{align}\label{es-CmA1}
\nn&\sum_{|\beta|\le N}\kappa^{2|\beta|+1}\sum_{\substack{m\in\N^6}}a_{m,\lm,s}^2(t)\big|{\bf C}^m_{\mathcal{A},1;(1)}+{\bf C}^m_{\mathcal{A},1;(2)}\big|\\
\le&C(\lm M_{\sig})^2\|\nb_xg\|_{\mathcal{G}^{\lm,N}_{s,\sig,\ell-2}} \|\kappa\nb_vg\|_{\mathcal{G}^{\lm,N}_{s,\sig,\ell-2}}.
\end{align}

If $|\beta'|+|n|\ge2$, all the integrands in ${\bf C}_{\mathcal{A},1}^m$ possess enough decay in $v$ without resorting to integrating by parts as did for ${\bf C}^m_{\mathcal{A},1;(1)}+{\bf C}^m_{\mathcal{A},1;(2)}$, and so does ${\bf C}_{\mathcal{A},2}^m$. Thus, similar to \eqref{es-comZm}, now we have for $\iota=1,2$, there exist a constant $C_{s,N,M_{\sig}, M_{\mu}}$ depending on $s, N, M_{\sig}$ and $M_{\mu}$, such that
\begin{align}\label{es-CmAK}
\sum_{|\beta|\le N}\kappa^{2|\beta|+1}\sum_{\substack{m\in\N^6}}a_{m,\lm,s}^2(t)\left|{\bf C}_{\mathcal{A},\iota}^m\right|
\le C_{s,N,M_{\sig}, M_{\mu}}\lm \|\nb_xg\|_{\mathcal{G}^{\lm,N}_{s,\sig,\ell-2}}\sum_{|\al|\le1}\|\kappa\pr_v^\al g\|_{\mathcal{G}^{\lm,N}_{s,\sig,\ell-2|\al|}}.
\end{align}
For ${\bf C}_{\mathcal{K},\iota}^m, 1\le\iota\le4$, integrating by parts, we have
\begin{align*}
    &\int_{\T^3\times\R^3}\left(Z^\beta{\bf c}_{\mathcal{K},\iota}^m[\pr_{x_l}g]\pr_{v_l}g^{(m+\beta)}+\pr_{v_l}Z^\beta{\bf c}_{\mathcal{K},\iota}^m[g]\pr_{x_l}g^{(m+\beta)}\right)\la v\ra^{2\ell-4}dvdx\\
    =&2\int_{\T^3\times\R^3}Z^\beta{\bf c}_{\mathcal{K},\iota}^m[\pr_{x_l}g]\pr_{v_l}g^{(m+\beta)}\la v\ra^{2\ell-4}dvdx\\
    &+\int_{\T^3\times\R^3}Z^\beta{\bf c}_{\mathcal{K},\iota}^m[\pr_{x_l}g]g^{(m+\beta)}\pr_{v_l}\la v\ra^{2\ell-4}dvdx.
\end{align*}
Recalling the definitions of ${\bf c}_{\mathcal{K},\iota}^m[\cdot]$ in Lemma \ref{lem-coercive-L}, in view of \eqref{cs-vtlv} and \eqref{es-IK}, it is not difficulty to verify that \eqref{es-CmAK} still holds with ${\bf C}_{\mathcal{A},\iota}^m$ replaced by ${\bf C}_{\mathcal{K},\iota}^m$.

\noindent{\bf Treatments of ${\bf C}_{\mathcal{A},\iota}^\beta$ and ${\bf C}_{\mathcal{K},\iota}^\beta$.} Now we turn to estimate ${\bf C}_{\mathcal{A},\iota}^\beta$ and ${\bf C}_{\mathcal{K},\iota}^\beta$. Like the treatments of ${\bf C}_{\mathcal{A},\iota}^m$ and ${\bf C}_{\mathcal{K},\iota}^m$, the most tricky one is ${\bf C}^\beta_{\mathcal{A},1}$, which is written explicitly as follows
\begin{align*}
{\bf C}^\beta_{\mathcal{A},1}
=&\sum_{\substack{\beta'\le\beta,|\beta'|\ge1}}C_\beta^{\beta'}\int_{\T^3\times\R^3}Z^{\beta'}\sig^{ij}\pr_{v_j}\pr_{x_l}g^{(m+\beta-\beta')}\pr_{v_i}\pr_{v_l}g^{(m+\beta)}\la v\ra^{2\ell-4}dvdx\\
&+\sum_{\substack{\beta'\le\beta,|\beta'|\ge1}}C_\beta^{\beta'}\int_{\T^3\times\R^3}Z^{\beta'}\sig^{ij}\pr_{v_j}\pr_{v_l}g^{(m+\beta-\beta')}\pr_{v_i}\pr_{x_l}g^{(m+\beta)}\la v\ra^{2\ell-4}dvdx\\
&+\sum_{\substack{\beta'\le\beta,|\beta'|\ge1}}C_\beta^{\beta'}\int_{\T^3\times\R^3}\pr_{v_l}Z^{\beta'}\sig^{ij}\pr_{v_j}g^{(m+\beta-\beta')}\pr_{v_i}\pr_{x_l}g^{(m+\beta)}\la v\ra^{2\ell-4}dvdx\\
&+\sum_{\substack{\beta'\le\beta,|\beta'|\ge1}}C_\beta^{\beta'}\int_{\T^3\times\R^3}Z^{\beta'}\sig^{ij}\pr_{v_j}\pr_{x_l}g^{(m+\beta-\beta')}\pr_{v_l}g^{(m+\beta)}\pr_{v_i}\la v\ra^{2\ell-4}dvdx\\
&-\sum_{\substack{\beta'\le\beta,|\beta'|\ge1}}C_\beta^{\beta'}\int_{\T^3\times\R^3}Z^{\beta'}\sig^{ij}\pr_{v_j}g^{(m+\beta-\beta')}\pr_{v_l}\pr_{x_l}g^{(m+\beta)}\pr_{v_i}\la v\ra^{2\ell-4}dvdx\\
&-\sum_{\substack{\beta'\le\beta,|\beta'|\ge1}}C_\beta^{\beta'}\int_{\T^3\times\R^3}Z^{\beta'}\sig^{ij}\pr_{v_j}g^{(m+\beta-\beta')}\pr_{x_l}g^{(m+\beta)}\pr_{v_l}\pr_{v_i}\la v\ra^{2\ell-4}dvdx\\
=&\sum_{1\le \tilde{j}\le6}{\bf C}_{\mathcal{A},1;(\tilde{j})}^\beta.
\end{align*}
If $|\beta'|=1$, we compute ${\bf C}_{\mathcal{A},1;(1)}^\beta$ and ${\bf C}_{\mathcal{A},1;(2)}^\beta$ together, and use integrating by parts  as in \eqref{i-b-p}. Otherwise, ${\bf C}^{\beta}_{\mathcal{A},i}, i=1,2$ and ${\bf C}^{\beta}_{\mathcal{K},\iota}, \iota=1, 2, 3, 4$ can be treated in a direct manner as \eqref{es-cA1-beta}, \eqref{es-cA2-beta} and \eqref{es-cK-beta}. Therefore, similar to \eqref{es-cAi}, for $(\#,*)=(\mathcal{A},i),i=1,2$ or $(\#,*)=(\mathcal{K},\iota),\iota=1,2,\cdots,6$,  we have
\begin{align}\label{es-Cbeta}
\sum_{|\beta|\le N}\kappa^{2|\beta|+1}\sum_{\substack{m\in\N^6}}a_{m,\lm,s}^2(t)\big|{\bf C}_{\#,*}^\beta\big|
\le \kappa C_N \|\nb_xg\|_{\mathcal{G}^{\lm,N}_{s,\sig,\ell-2}}\sum_{|\al|\le1}\|\kappa\pr_v^\al g\|_{\mathcal{G}^{\lm,N}_{s,\sig,\ell-2|\al|}},
\end{align}
for some constant $C_N$ depending only on $N$.
It follows from  \eqref{cross-main}, \eqref{es-CmA1}, \eqref{es-CmAK}, \eqref{es-Cbeta} that \eqref{es-inner} holds. 
\end{proof}

 The following lemma is used to handle the truncated lowest-order term that appeared on the right-hand side of \eqref{coer-Landau}.
\begin{lem}\label{lem-lower order}
    Under the bootstrap hypotheses, for any $\zeta>0$, it holds that
    \begin{align}\label{bd-g-low}
    \| g\|_{L^2_{x,v}(|v|\le 2\zeta)}
    &\les \|E\|^2_{L^2_x}
    +\|({\rm I}-{\rm P}_0)g\|_{\sig}
    +\|\nb_xg\|_{L^2_{x,v}(|v|\le2\zeta)}.
\end{align}
\end{lem}

\begin{proof}
Noting that $\nb_xf=(e^{-\phi}-1)(e^{\phi}\nb_xf)+e^{\phi}\nb_xf$, and $|e^{-\phi}-1|\le |\phi| e^{|\phi|}$,
by using the decomposition $f=f_0+f_{\ne}$ and Poincar\'e inequality, recalling that $g=e^{\phi}f$, we have
\begin{align}
    \|\nb_xf\|_{L^2_{x,v}(|v|\le2\zeta)}\nn\les&\big(1+\|\phi\|_{L^\infty_x}e^{\|\phi\|_{L^\infty_x}}\big)\|e^{\phi}\nb_xf\|_{L^2_{x,v}(|v|\le2\zeta)}\\
    \les\nn& \|\nb_x(e^{\phi}f)\|_{L^2_{x,v}(|v|\le2\zeta)}+\|e^{\phi}\nb_x\phi f\|_{L^2_{x,v}(|v|\le2\zeta)}\\
    \les\nn&\|\nb_xg\|_{L^2_{x,v}(|v|\le2\zeta)}+\|\nb_x\phi\|_{L^\infty_x}\|f_0\|_{L^2_{x,v}(|v|\le2\zeta)}\\
    \nn&+\|\nb_x\phi\|_{L^\infty_x}\|\nb_xf_{\ne}\|_{L^2_{x,v}(|v|\le2\zeta)}.
\end{align}
Then, \eqref{Linfty-rho} together with the fact that $\|\nb_x\phi\|_{L^\infty_x}\lesssim \|\rho\|_{H^3}\lesssim \epsilon$ ensures that
\begin{align}\label{bd-nb_xf}
    \|\nb_xf\|_{L^2_{x,v}(|v|\le2\zeta)}
    \les\|\nb_xg\|_{L^2_{x,v}(|v|\le2\zeta)}+\|f_0\|_{L^2_{x,v}(|v|\le2\zeta)}.
\end{align}

On the other hand, thanks to the decomposition 
\begin{align*}
f={\rm P_0}f_0+({\rm I}-{\rm P}_0)f_0+f_{\ne},
\end{align*}
where  ${\rm P}_0$ is defined in \eqref{projection}, we can use Poincar\'e inequality, \eqref{bd-nb_xf} and
 \eqref{coercive}    to obtain
\begin{align*}
    \| e^{\phi}f\|_{L^2_{x,v}(|v|\le 2\zeta)}\les&\big(1+\|\phi\|_{L^\infty_x}e^{\|\phi\|_{L^\infty_x}}\big) \Big(\|  f_0\|_{L^2_{x,v}(|v|\le 2\zeta)}+\| \nb_xf\|_{L^2_{x,v}(|v|\le 2\zeta)}\Big)\\
    \les&\| {\rm P}_0f_0\|_{L^2_{x,v}}+\|({\rm I}-{\rm P}_0)f_0\|_{L^2_{x,v}(|v|\le 2\zeta)}+\|\nb_xg\|_{L^2_{x,v}(|v|\le 2\zeta)}\\
    \les&|(\bar{a},\bar{b},\bar{c})|^2+\|e^{\phi}({\rm I}-{\rm P}_0)f\|_{\sig}+\|\nb_xg\|_{L^2_{x,v}(|v|\le 2\zeta)},
\end{align*}
where we have used the fact 
\begin{align}\label{bd-P0f0}
    \|{\rm P}_0f_0\|_{L^2_{x,v}}^2\les |(\bar{a},\bar{b},\bar{c})|^2,
\end{align}
with
\begin{align*}
    |(\bar{a},\bar{b},\bar{c})|^2=|\bar{a}|^2+\sum_{j=1}^3|\bar{b}_j|^2+|\bar{c}|^2,
\end{align*}
and
\begin{align*}
    \bar{a}=\int_{\T^3\times\R^3}f\sqrt{\mu}dvdx,\quad 
    \bar{b}_j=\int_{\T^3\times\R^3}f v_j\sqrt{\mu}dvdx,\quad \bar{c}=\int_{\T^3\times\R^3}f|v|^2\sqrt{\mu}dvdx.
\end{align*}

To prove \eqref{bd-P0f0}, direct calculations show that
\begin{align*}
    \|{\rm P_0}f\|_{L^2_{x,v}}^2
    \les&\int_{\mathbb{T}^3}|a(t,x)|^2dx+\sum_{j=1}^3\int_{\mathbb{T}^3}|b_j(t,x)|^2dx+\int_{\mathbb{T}^3}|c(t,x)|^2dx.
\end{align*}
On the other hand, $a(t,x), b_j(t,x), 1\le j\le3$ and $c(t,x)$ can be solved as linear combinations of $\int_{\R^3}f\sqrt{\mu}dv, \int_{\R^3}fv_j\sqrt{\mu}dv, 1\le j\le3$ and $\int_{\R^3}f|v|^2\sqrt{\mu}dv$, we thus further have
\begin{align*}
    \|{\rm P_0}f\|_{L^2_{x,v}}^2\les \int_{\mathbb{T}^3}\left|\int_{\R^3} f\sqrt{\mu}dv\right|^2dx+\sum_{j=1}^3\int_{\mathbb{T}^3}\left|\int_{\R^3} fv_j\sqrt{\mu}dv\right|^2dx+\int_{\mathbb{T}^3}\left|\int_{\R^3} f|v|^2\sqrt{\mu}dv\right|^2dx.
\end{align*}
In particular,
\begin{align*}
    \|{\rm P_0}f_0\|_{L^2_{x,v}}^2
    \les&\left|\int_{\R^3} f_0\sqrt{\mu}dv\right|^2+\sum_{j=1}^3\left|\int_{\R^3} f_0v_j\sqrt{\mu}dv\right|^2+\left|\int_{\R^3} f_0|v|^2\sqrt{\mu}dv\right|^2\\
    \les&|\bar{a}|^2+\sum_{j=1}^3|\bar{b}_j|^2+|\bar{c}|^2=|(\bar{a},\bar{b},\bar{c})|^2.
\end{align*}

By conservation laws,
\[
\int_{\mathbb{T}^3}\int_{\R^3}f\sq{\mu}dvdx=\int_{\mathbb{T}^3}\int_{\R^3}fv_j\sqrt{\mu}dvdx=\int_{\mathbb{T}^3}\int_{\R^3}f|v|^2\sqrt{\mu}dvdx+\int_{\T^3}|E|^2dx=0.
\]
Thus, $\bar{a}=\bar{b}_j=0, j=1,2,3$ and $|\bar{c}|\le \int_{\mathbb{T}^3}|E|^2dx$. Then \eqref{bd-g-low} follows immediately.
\end{proof}

\section{Estimates on $f$ for $t\leq \nu^{-\fr12}$}\label{sec: est f}
To begin with, let us recall the first equation of \eqref{pVPL} 
\begin{align}
    \pr_tf+v\cdot\nb_xf+E\cdot\nb_vf-E\cdot vf+\nu Lf=2E\cdot v\sqrt{\mu}+\nu\Gamma(f,f).
\end{align}
Recalling that $g(t,x,v)=e^{\phi}f(t,x,v)$, then $g(t,x,v)$ solves
\begin{align}\label{eq-g-0}
\pr_tg+v\cdot\nb_xg+E\cdot\nb_vg+\nu Lg= \nu e^{-\phi}\Gamma(g,g)+2e^{\phi}E\cdot v\sqrt{\mu}+\pr_t\phi g.    
 \end{align}
For $n\in\N^6$, noting that $[Z,\pr_t+v\cdot\nb_x]=0$, it is easy to see that $g^{(n)}:=Z^n g$ solves
\begin{align}\label{eq-g}
\nn&\pr_tg^{(n)}+v\cdot\nb_xg^{(n)}+Z^{n}\big(E\cdot\nb_vg\big)+\nu Z^n Lg\\
=& \nu Z^n \big(e^{-\phi}\Gamma(g,g)\big)+2Z^n\big(e^{\phi}E\cdot v\sqrt{\mu}\big)+Z^n\left(\pr_t\phi g\right).
 \end{align}
Then we have
\begin{align}\label{eq-gxj}
\nn&\pr_t\pr_{x_j}g^{(n)}+v\cdot\nb_x\pr_{x_j}g^{(n)}+\pr_{x_j}Z^{n}\left(E\cdot\nb_vg\right)+\nu Z^n L\pr_{x_j}g\\
=&\nu\pr_{x_j}Z^n\big(e^{-\phi}\Gamma(g,g)\big)+2\pr_{x_j}Z^n\big(e^{\phi}E\cdot v\sqrt{\mu}\big)+\pr_{x_j}Z^n\left(\pr_t\phi g\right),
\end{align}
and
 \begin{align}\label{eq-gvj}
\nn&\pr_t\pr_{v_j}g^{(n)}+v\cdot\nb_x \pr_{v_j}g^{(n)}+\pr_{x_j}g^{(n)}+\pr_{v_j}Z^{n}\left(E\cdot\nb_vg\right)+\nu \pr_{v_j}Z^n Lg\\
=&\nu\pr_{v_j}Z^n\big(e^{-\phi}\Gamma(g,g)\big)+2\pr_{v_j}Z^n\big(e^{\phi}E\cdot v\sqrt{\mu}\big)+\pr_{v_j}Z^n\left(\pr_t\phi g\right).
 \end{align}

For $\iota=0, 1, 2, \cdots,[\ell/3]+1$, we make the following notational conventions: for any function ${\rm f}={\rm f}(t,x,v)$, 
\begin{align}\label{f-iota*}
{\rm f}_{\iota*}=
\begin{cases}
{\rm f},\quad\quad\quad\quad\qquad\ {\rm if}\quad\iota=0;\\
{\rm f}_{\iota\ne}=w_{\iota}(t){\rm f}_{\ne},\quad {\rm if}\quad \iota=1, 2, \cdots,[\ell/3]+1.
\end{cases}
\end{align}
To simplify  the presentation, we also introduce the following notation:
\begin{align*}
    L[\mu]_{\iota;1)}:=&\sum_{m\in\N^6,|\beta|\le N}\kappa^{2|\beta|}a_{m,\lm,s}^2(t)\bigg(\sum_{|\al|\le1}2{\rm A}_0\left\la \pr_{x}^\al Z^{m+\beta}\big(e^{\phi}E\cdot v\sqrt{\mu}\big)_{\iota*},\pr_x^{\al}g^{(m+\beta)}_{\iota*}\right\ra_{\ell-2|\al|-2\iota}\\
    &+2\kappa^2(1+t)^{-2}\sum_{|\al|=1}\left\la Y^{\al} Z^{m+\beta}\big(e^{\phi}E\cdot v\sqrt{\mu}\big)_{\iota*},Y^{\al}g^{(m+\beta)}_{\iota*}\right\ra_{\ell-2-2\iota}\\
&+2\sum_{1\le|\al|\le2}\kappa^{2|\al|}\nu^{\fr{2|\al|}{3}}\left\la \pr_v^\al Z^{m+\beta}\big(e^{\phi}E\cdot v\sqrt{\mu} \big)_{\iota*},\pr_v^\al g^{(m+\beta)}_{\iota*}\right\ra_{\ell-2|\al|-2\iota} \bigg),   
\end{align*}

\begin{align*}
    L[\mu]_{\iota;2)}:=&\kappa\nu^{\fr13}\sum_{m\in\N^6,|\beta|\le N}\kappa^{2|\beta|}a_{m,\lm,s}^2(t)\bigg(2\left\la \nb_{x} Z^{m+\beta}\big(e^{\phi}E\cdot v\sqrt{\mu}\big)_{\iota*},\nb_{v}g^{(m+\beta)}_{\iota*}\right\ra_{\ell-2-2\iota}\\
    &+2\left\la \nb_{v} Z^{m+\beta}\big(e^{\phi}E\cdot v\sqrt{\mu}\big)_{\iota*},\nb_{x}g^{(m+\beta)}_{\iota*}\right\ra_{\ell-2-2\iota} \bigg),  
\end{align*}

\begin{align*}
    N[\pr_t\phi]_{\iota;1)}:=&\sum_{m\in\N^6,|\beta|\le N}\kappa^{2|\beta|}a_{m,\lm,s}^2(t)\Bigg[{\rm A}_0\sum_{|\al|\le1}\left\la \pr_{x}^\al Z^{m+\beta}\left(\pr_t\phi g\right)_{\iota*},\pr_x^{\al}g^{(m+\beta)}_{\iota*}\right\ra_{\ell-2|\al|-2\iota}\\
    &+\kappa^2(1+t)^{-2}\sum_{|\al|=1}\left\la Y^{\al} Z^{m+\beta}\left(\pr_t\phi g\right)_{\iota*},Y^{\al}g^{(m+\beta)}_{\iota*}\right\ra_{\ell-2-2\iota}\\
    &+\sum_{1\le|\al|\le 2}\kappa^{2|\al|}\nu^{\fr{2|\al|}{3}}\left\la \pr_v^\al Z^{m+\beta}\left(\pr_t\phi g\right)_{\iota*},\pr_v^\al g^{(m+\beta)}_{\iota*}\right\ra_{\ell-2|\al|-2\iota}\Bigg],
\end{align*}
 
\begin{align*}
    N[\pr_t\phi]_{\iota;2)}:=&\kappa\nu^{\fr13}\sum_{m\in\N^6,|\beta|\le N}\kappa^{2|\beta|}a_{m,\lm,s}^2(t)\Bigg[\left\la \nb_{x} Z^{m+\beta}\left(\pr_t\phi g\right)_{\iota*},\nb_{v}g^{(m+\beta)}_{\iota*}\right\ra_{\ell-2-2\iota}\\
    &+\left\la \nb_{v} Z^{m+\beta}\left(\pr_t\phi g\right)_{\iota*},\nb_{x}g^{(m+\beta)}_{\iota*}\right\ra_{\ell-2-2\iota}\Bigg],
\end{align*}

\begin{align*}
    N[E]_{\iota;1)}:=&-\sum_{m\in\N^6,|\beta|\le N}\kappa^{2|\beta|}a_{m,\lm,s}^2(t)\Bigg[{\rm A}_0\sum_{|\al|\le1}\left\la \pr_{x}^\al Z^{m+\beta}\left(E\cdot\nb_v g\right)_{\iota*},\pr_x^{\al}g^{(m+\beta)}_{\iota*}\right\ra_{\ell-2|\al|-2\iota}\\
    &+\kappa^2(1+t)^{-2}\sum_{|\al|=1}\left\la Y^{\al} Z^{m+\beta}\left(E\cdot\nb_v g\right)_{\iota*},Y^{\al}g^{(m+\beta)}_{\iota*}\right\ra_{\ell-2-2\iota}\\
    &+\sum_{1\le|\al|\le 2}\kappa^{2|\al|}\nu^{\fr{2|\al|}{3}}\left\la \pr_v^\al Z^{m+\beta}\left(E\cdot\nb_v g\right)_{\iota*},\pr_v^\al g^{(m+\beta)}_{\iota*}\right\ra_{\ell-2|\al|-2\iota}\Bigg],
\end{align*}

\begin{align*}
    N[E]_{\iota;2)}:=&-\kappa\nu^{\fr13}\sum_{m\in\N^6,|\beta|\le N}\kappa^{2|\beta|}a_{m,\lm,s}^2(t)\Bigg[\left\la \nb_{x} Z^{m+\beta}\left(E\cdot\nb_v g\right)_{\iota*},\nb_{v}g^{(m+\beta)}_{\iota*}\right\ra_{\ell-2-2\iota}\\
    &+\left\la \nb_{v} Z^{m+\beta}\left(E\cdot\nb_v g\right)_{\iota*},\nb_{x}g^{(m+\beta)}_{\iota*}\right\ra_{\ell-2-2\iota}\Bigg],
\end{align*}

    \begin{align*}
    N[\Gamma]_{\iota;1)}:=&\nu\sum_{m\in\N^6,|\beta|\le N}\kappa^{2|\beta|}a_{m,\lm,s}^2(t)\Bigg[{\rm A}_0\sum_{|\al|\le1}\left\la \pr_{x}^\al Z^{m+\beta}\left(e^{-\phi}\Gamma(g,g)\right)_{\iota*},\pr_x^{\al}g^{(m+\beta)}_{\iota*}\right\ra_{\ell-2|\al|-2\iota}\\
    &+\kappa^2(1+t)^{-2}\sum_{|\al|=1}\left\la Y^{\al} Z^{m+\beta}\left(e^{-\phi}\Gamma(g,g)\right)_{\iota*},Y^{\al}g^{(m+\beta)}_{\iota*}\right\ra_{\ell-2-2\iota}\\
    &+\sum_{1\le|\al|\le 2}\kappa^{2|\al|}\nu^{\fr{2|\al|}{3}}\left\la \pr_v^\al Z^{m+\beta}\left(e^{-\phi}\Gamma(g,g)\right)_{\iota*},\pr_v^\al g^{(m+\beta)}_{\iota*}\right\ra_{\ell-2|\al|-2\iota}\Bigg],
\end{align*}
and
\begin{align*}
    N[\Gamma]_{\iota;2)}:=&\kappa\nu^{\fr43}\sum_{m\in\N^6,|\beta|\le N}\kappa^{2|\beta|}a_{m,\lm,s}^2(t)\Bigg[\left\la \nb_{x} Z^{m+\beta}\left(e^{-\phi}\Gamma(g,g)\right)_{\iota*},\nb_{v}g^{(m+\beta)}_{\iota*}\right\ra_{\ell-2-2\iota}\\
    &+\left\la \nb_{v} Z^{m+\beta}\left(e^{-\phi}\Gamma(g,g)\right)_{\iota*},\nb_{x}g^{(m+\beta)}_{\iota*}\right\ra_{\ell-2-2\iota}\Bigg].
\end{align*}

We first establish the following  energy estimates for $g$.

\begin{prop}\label{prop-g}
 Let \eqref{res-final-1}--\eqref{restriction: lm-kappa-nu} hold. Assume that $g$ is the solution to \eqref{eq-g}, and the bootstrap hypotheses \eqref{bd-en}--\eqref{H-phi} hold. Then 
 there holds
\begin{align}\label{en-ineq-Linear-g}
    \nn&\fr{d}{dt}\mathcal{E}(g(t))+\fr{1}{4}\nu^{\fr13}\mathcal{D}(g(t))+\mathcal{CK}(g(t))\\
    \nn\le& C\nu {\rm A}_0\|E\|_{L^2_x}^4+C \fr{\sqrt{{\rm B}_0}}{{\rm A}_0}\big(\sqrt{{\rm B}_0}\|g\|_{L^2_{x,v}}\big)\big(\nu {\rm A}_0\|g\|_{\sig}^2\big)\\
    \nn&+{\rm B}_0\|\pr_t\phi\|_{L^\infty_x}\|g\|_{L^2_{x,v}}^2+C{\rm B}_0\| E\|_{L^2_x}\|g\|_{L^2_{x,v}}\\
    \nn&+2\sum_{\iota=0}^{[\ell/3]+1}\Big[L[\mu]_{\iota;1)}+L[\mu]_{\iota;2)}+N[\pr_t\phi]_{\iota;1)}+N[\pr_t\phi]_{\iota;2)}\\
    &+N[E]_{\iota;1)}+N[E]_{\iota;2)}+N[\Gamma]_{\iota;1)}+N[\Gamma]_{\iota;2)}\Big].
\end{align}

\end{prop}

\begin{proof}
Taking the $L^2$ inner product of \eqref{eq-g} with $g$ yields
\begin{align*}
    \fr{1}{2}\fr{d}{dt}\| g \|_{L^2_{x,v}}^2+\nu\left\la Lg, g\right\ra_{x,v}
    =&\int_{\mathbb{T}^3\times \R^3}\pr_t\phi g^2   dvdx\\
    &+\int_{\mathbb{T}^3\times\R^3 }\left(\nu e^\phi\Gamma(f,f) +2e^{\phi}E\cdot v\mu^{\fr12}\right) g dvdx.
\end{align*}
By \eqref{coercive-0}, we arrive at
\begin{align}\label{en-0}
\nn&\fr12\fr{d}{dt}\| g \|_{L^2_{x,v}}^2+\dl\nu\|({\rm I}-{\rm P}_0)g\|_{\sig}^ 2\\
\le&\|\pr_t\phi\|_{L^\infty_x}\|g\|_{L^2_{x,v}}^2+\left|\int_{\mathbb{T}^3\times\R^3}\left(\nu e^\phi\Gamma(f,f) +2e^{\phi}E\cdot v\mu^{\fr12}\right) g dvdx\right|.
\end{align}
Recalling that $f=q(\phi)g+g$, by using Theorem 3  of \cite{Guo2002} and the bootstrap hypotheses, one deduces that
\begin{align}\label{est-Gamma0}
\nn&\nu{\rm B}_0\left|\int_{\mathbb{T}^3\times\R^3} e^\phi\Gamma(f,f)  g dvdx\right|\\
\le&\nu{\rm B}_0\Big(\|f\|_{L^2_{x,v}}\|g\|_{\sig}^2+\|g\|_{L^2_{x,v}}\|f\|_{\sig}\|g\|_{\sig}\Big)\le  C{\rm B}_0 \|g\|_{L^2_{x,v}}\big(\nu \|g\|_{\sig}^2\big),
\end{align}
and
\begin{align}\label{est-linear0}
    &{\rm B}_0\left|\int_{\mathbb{T}^3\times\R^3}2e^{\phi}E\cdot v\mu^{\fr12} g dvdx\right|\le C {\rm B}_0\|e^\phi E\|_{L^2_x}\|g\|_{L^2_{x,v}}\le C\| E\|_{L^2_x}\big({\rm B}_0\|g\|_{L^2_{x,v}}\big).
\end{align}

For $\al=(\al_1,\al_2,\al_3)$ with $|\al|\le1$, and $\ell>4, \iota=0, 1, 2, \cdots,[\ell/3]+1$, taking the inner product of the equation for $\pr_x^{\al}g^{(m+\beta)}_{\iota*}$ with $\la v\ra^{2\ell-4|\al|-4\iota}\pr_x^{\al}g^{(m+\beta)}_{\iota*}$, we are led to
\begin{align}\label{en-pr_xg}
    \nn&\fr12\fr{d}{dt}\big\|\pr_x^{\al}g^{(m+\beta)}_{\iota*}\big\|_{L^2_{x,v}(\ell-2|\al|-2\iota)}^2+\nu\Big\la Z^{m+\beta} L\pr^{\al}_{x}g_{\iota*},\pr_x^{\al}g^{(m+\beta)}_{\iota*}\Big\ra_{\ell-2|\al|-2\iota}\\
    =\nn&\fr{\iota \kappa_0\nu^{\fr13}}{2}\big(\kappa_0\nu^{\fr13}(1+t)\big)^{-1} \big\|\pr_x^{\al}g^{(m+\beta)}_{\iota*}\big\|_{L^2_{x,v}(\ell-2|\al|-2\iota)}^2\\
    \nn&-\left\la \pr_{x}^\al Z^{m+\beta}\left(E\cdot\nb_vg\right)_{\iota*},\pr_x^{\al}g^{(m+\beta)}_{\iota*}\right\ra_{\ell-2|\al|-2\iota}\\
    \nn&+\nu\left\la \pr_{x}^\al Z^{m+\beta}\big(e^{-\phi}\Gamma(g,g)\big)_{\iota*},\pr_x^{\al}g^{(m+\beta)}_{\iota*}\right\ra_{\ell-2|\al|-2\iota}\\
    \nn&+2\left\la \pr_{x}^\al Z^{m+\beta}\big(e^{\phi}E\cdot v\sqrt{\mu}\big)_{\iota*},\pr_x^{\al}g^{(m+\beta)}_{\iota*}\right\ra_{\ell-2|\al|-2\iota}\\
    &+\left\la \pr_{x}^\al Z^{m+\beta}\left(\pr_t\phi g\right)_{\iota*},\pr_x^{\al}g^{(m+\beta)}_{\iota*}\right\ra_{\ell-2|\al|-2\iota}.
\end{align}
Similarly, for $|\al|=1$,
\begin{align}\label{en-Yg}
    \nn&\fr12\fr{d}{dt}\big\|(1+t)^{-1}Y^{\al}g^{(m+\beta)}_{\iota*}\big\|_{L^2_{x,v}(\ell-2|\al|-2\iota)}^2\\
    \nn&+\nu (1+t)^{-2}\left\la Y^{\al} Z^{m+\beta} Lg_{\iota*},Y^{\al}g^{(m+\beta)}_{\iota*}\right\ra_{\ell-2|\al|-2\iota}\\
    =\nn&\fr{(\iota-2) (1+t)^{-3}}{2} \big\|Y^{\al}g^{(m+\beta)}_{\iota*}\big\|_{L^2_{x,v}(\ell-2|\al|-2\iota)}^2\\
    \nn&-(1+t)^{-2}\left\la Y^{\al} Z^{m+\beta}\left(E\cdot\nb_vg\right)_{\iota*},Y^{\al}g^{(m+\beta)}_{\iota*}\right\ra_{\ell-2|\al|-2\iota}\\
    \nn&+\nu(1+t)^{-2}\left\la Y^{\al} Z^{m+\beta}\big(e^{-\phi}\Gamma(g,g)\big)_{\iota*},Y^{\al}g^{(m+\beta)}_{\iota*}\right\ra_{\ell-2|\al|-2\iota}\\
    \nn&+2(1+t)^{-2}\left\la Y^{\al} Z^{m+\beta}\big(e^{\phi}E\cdot v\sqrt{\mu}\big)_{\iota*},Y^{\al}g^{(m+\beta)}_{\iota*}\right\ra_{\ell-2|\al|-2\iota}\\
    &+(1+t)^{-2}\left\la Y^{\al} Z^{m+\beta}\left(\pr_t\phi g\right)_{\iota*},Y^{\al}g^{(m+\beta)}_{\iota*}\right\ra_{\ell-2|\al|-2\iota}.
\end{align}

For $1\le|\al|\le 2$, taking the inner product of the equation for $\pr_{v}^{\al}g^{(m+\beta)}_{\iota*}$ with $\la v\ra^{2\ell-4|\al|-4\iota}\pr_v^\al g^{(m+\beta)}_{\iota*}$ yields
\begin{align}\label{en-pr_vg}
    \nn&\fr12\fr{d}{dt}\big\|\pr_v^\al g^{(m+\beta)}_{\iota*}\big\|_{L^2_{x,v}(\ell-2|\al|-2\iota)}^2+{  \left\la\big[\pr_v^\al, v\cdot\nb_x g^{(m+\beta)}_{\iota*}\big],  \pr_v^\al g^{(m+\beta)}_{\iota*}\right\ra_{\ell-2|\al|-2\iota}}\\
    \nn&+\nu \left\la \pr_v^\al Z^{m+\beta} Lg_{\iota*},\pr_v^\al g^{(m+\beta)}_{\iota*}\right\ra_{\ell-2|\al|-2\iota}\\
    =\nn&\fr{\iota \kappa_0\nu^{\fr13}}{2} \big(\kappa_0\nu^{\fr13}(1+t)\big)^{-1} \big\|\pr_v^\al g^{(m+\beta)}_{\iota*}\big\|_{L^2_{x,v}(\ell-2|\al|-2\iota)}^2\\
    \nn&-\left\la \pr_v^\al Z^{m+\beta}\left(E\cdot\nb_vg\right)_{\iota*},\pr_v^\al g^{(m+\beta)}_{\iota*}\right\ra_{\ell-2|\al|-2\iota}\\
    \nn&+\nu\left\la \pr_v^\al Z^{m+\beta}\big(e^{-\phi}\Gamma(g,g)\big)_{\iota*},\pr_v^\al g^{(m+\beta)}_{\iota*}\right\ra_{\ell-2|\al|-2\iota}\\
    \nn&+2\left\la \pr_v^\al Z^{m+\beta}\big(e^{\phi}E\cdot v\sqrt{\mu} \big)_{\iota*},\pr_v^\al g^{(m+\beta)}_{\iota*}\right\ra_{\ell-2|\al|-2\iota}\\
    &+\left\la \pr_v^\al Z^{m+\beta}\left(\pr_t\phi g\right)_{\iota*},\pr^\al_vg^{(m+\beta)}_{\iota*}\right\ra_{\ell-2|\al|-2\iota}.
\end{align}
Moreover, the evolution of the inner product of $\pr_{x_j}g^{(m+\beta)}_{\iota*}$ and $\la v\ra^{2\ell-4-4\iota}\pr_{v_j}g^{(m+\beta)}_{\iota*}$ takes the form of
\begin{align}\label{en-cross}
    \nn&\fr{d}{dt}\left\la\pr_{x_j}g^{(m+\beta)}_{\iota*},\pr_{v_j}g^{(m+\beta)}_{\iota*} \right\ra_{\ell-2-2\iota}+\left\| \pr_{x_j}g^{(m+\beta)}_{\iota*}\right\|_{L^2_{x,v}(\ell-2-2\iota)}^2\\
    \nn&+\nu\left\la Z^{m+\beta}L\pr_{x_j}g_{\iota*},  \pr_{v_j}g^{(m+\beta)}_{\iota*}\right\ra_{\ell-2-2\iota}+\nu \left\la \pr_{v_j}Z^{m+\beta} Lg_{\iota*}, \pr_{x_j}g^{(m+\beta)}_{\iota*}\right\ra_{\ell-2-2\iota}\\
    =\nn&\iota\kappa_0\nu^{\fr13}\big(\kappa_0\nu^{\fr13}(1+t)\big)^{-1}\left\la\pr_{x_j}g^{(m+\beta)}_{\iota*},\pr_{v_j}g^{(m+\beta)}_{\iota*} \right\ra_{\ell-2-2\iota}\\
    \nn&-\left\la \pr_{x_j} Z^{m+\beta}\left(E\cdot\nb_vg\right)_{\iota*},\pr_{v_j}g^{(m+\beta)}_{\iota*}\right\ra_{\ell-2-2\iota}\\
    \nn&-\left\la \pr_{v_j} Z^{m+\beta}\left(E\cdot\nb_vg\right)_{\iota*},\pr_{x_j}g^{(m+\beta)}_{\iota*}\right\ra_{\ell-2-2\iota}\\
    \nn&+\nu\left\la \pr_{x_j} Z^{m+\beta}
    \big(e^{-\phi}\Gamma(g,g)\big)_{\iota*},\pr_{v_j}g_{\iota*}^{(m+\beta)}\right\ra_{\ell-2-2\iota}\\
    \nn&+\nu\left\la \pr_{v_j} Z^{m+\beta}\big(e^{-\phi}\Gamma(g,g)\big)_{\iota*},\pr_{x_j}g^{(m+\beta)}_{\iota*}\right\ra_{\ell-2-2\iota}\\
    \nn&+2\left\la \pr_{x_j} Z^{m+\beta}\big(e^{\phi}E\cdot v\sqrt{\mu}\big)_{\iota*},\pr_{v_j}g^{(m+\beta)}_{\iota*}\right\ra_{\ell-2-2\iota}\\
    \nn&+2\left\la \pr_{v_j} Z^{m+\beta}\big(e^{\phi}E\cdot v\sqrt{\mu}\big)_{\iota*},\pr_{x_j}g^{(m+\beta)}_{\iota*}\right\ra_{\ell-2-2\iota}\\
    \nn&+\left\la \pr_{x_j} Z^{m+\beta}\left(\pr_t\phi g\right)_{\iota*},\pr_{v_j}g^{(m+\beta)}_{\iota*}\right\ra_{\ell-2-2\iota}\\
    &+\left\la \pr_{v_j} Z^{m+\beta}\left(\pr_t\phi g\right)_{\iota*},\pr_{x_j}g^{(m+\beta)}_{\iota*}\right\ra_{\ell-2-2\iota}.
\end{align}
By Poincar\'e's inequality and \eqref{coercive}, in view of  Lemma \ref{lem: E_l<D_l-2}, we have
\begin{align}\label{dt-(1+t)}
    \nn&\sum_{\iota=1}^{[\ell/3]+1}\fr{\iota \kappa_0\nu^{\fr13}}{2}\big(\kappa_0\nu^{\fr13}(1+t)\big)^{-1}\bigg({\rm A}_0\sum_{|\al|\le1}\big\|\pr_x^{\al}g^{(m+\beta)}_{\iota\ne}\big\|_{L^2_{x,v}(\ell-2|\al|-2\iota)}^2\\
    \nn&+\sum_{1\le|\al|\le2}\big\|\kappa^{|\al|}\nu^{\fr{|\al|}{3}}\pr_v^{\al}g^{(m+\beta)}_{\iota\ne}\big\|_{L^2_{x,v}(\ell-2|\al|-2\iota)}^2+2\kappa\nu^{\fr13}\left\la\nb_{x}g^{(m+\beta)}_{\iota\ne},\nb_{v}g^{(m+\beta)}_{\iota\ne} \right\ra_{\ell-2-2\iota}\bigg)\\
    \nn&+\sum_{\iota=3}^{[\ell/3]+1} \fr{(\iota-2) (1+t)^{-3}}{2}\kappa^2\big\|Yg^{(m+\beta)}_{\iota\ne}\big\|_{L^2_{x,v}(\ell-2-2\iota)}^2\\
    \nn\le&\bar{C}\fr{{\rm A}_0}{\kappa}\kappa_0\nu^\fr{1}{3}\sum_{\iota=1}^{[\ell/3]+1}\iota \big(\kappa_0\nu^{\fr13}(1+t)\big)^{\iota-1}\Big(\kappa\big\|\nb_xg^{(m+\beta)}_{\ne}\big\|_{L^2_{x,v}(\ell-2\iota)}^2\\
    \nn&+\sum_{1\le|\al|\le2}\big\|\kappa^{|\al|}\nu^{\fr{|\al|}{3}}\pr_v^\al g^{(m+\beta)}_{\ne}\big\|_{L^2_{x,v}(\ell-2|\al|-2\iota)}^2\Big)\\
    \nn&+\kappa_0\nu^{\fr13}\sum_{\iota=3}^{[\ell/3]+1} (\iota-2) \Big(\big(\kappa_0\nu^{\fr13}(1+t)\big)^{\iota-3}\kappa^2\nu^{\fr23}\big\|\nb_vg^{(m+\beta)}_{\ne}\big\|_{L^2_{x,v}(\ell-2-2\iota)}^2\\
    \nn&+\big(\kappa_0\nu^{\fr13}(1+t)\big)^{\iota-1}\kappa\big\|\nb_xg^{(m+\beta)}_{\ne}\big\|_{L^2_{x,v}(\ell-2-2\iota)}^2\Big)\\
    \nn\le&\bar{C}\ell \fr{{\rm A}_0}{\kappa} \kappa_0\nu^{\fr13}\sum_{\iota=1}^{[\ell/3]+1}\big(\kappa_0\nu^{\fr13}(1+t)\big)^{\iota-1}\Big(\kappa\big\|\nb_xg^{(m+\beta)}_{\ne}\big\|_{L^2_{x,v}(\ell-2(\iota-1)-2)}^2\\
    \nn&+\kappa^2\nu^{\fr23}\sum_{0\le|\al|\le1}\big\|\kappa^{|\al|}\nu^{\fr{|\al|}{3}}\pr_v^\al g^{(m+\beta)}_{\ne}\big\|_{\sig,\ell-2|\al|-2\iota}^2\Big)\\
    \le&\bar{C}\ell \fr{{\rm A}_0}{\kappa} \kappa_0\nu^\fr{1}{3}\sum_{\iota=0}^{[\ell/3]} w^2_{\iota}(t)\mathcal{D}^{m+\beta}_{\ell-2\iota}(g_{\ne}(t)),
\end{align}
where $\bar{C}$ is a generic constant depending only on $\fr{1}{c}$ appearing in \eqref{coercive}.

Noting that
\begin{align*}
    \big[\pr_v^\al, v\cdot\nb_x g^{(m+\beta)}_{\iota*}\big]=\sum_{\substack{\al'\le\al, |\al'|=1}}C_\al^{\al'}\pr_x^{\al'}\pr_v^{\al-\al'}g_{\iota*}^{(m+\beta)},
\end{align*}
integrating by parts and using \eqref{coercive},  we then bound the second term on the left-hand side of \eqref{en-pr_vg} as follows:
\begin{align}\label{bad-cross}
    \nn&\sum_{1\le|\al|\le2}\kappa^{2|\al|}\nu^{\fr{2|\al|}{3}}\left\la\big[\pr_v^\al, v\cdot\nb_x g^{(m+\beta)}_{\iota*}\big],  \pr_v^\al g^{(m+\beta)}_{\iota*}\right\ra_{\ell-2|\al|-2\iota}\\
    =\nn&\kappa^{2}\nu^{\fr{2}{3}}\left\la\nb_x g^{(m+\beta)}_{\iota*},  \nb_vg^{(m+\beta)}_{\iota*}\right\ra_{\ell-2-2\iota}\\
    \nn&-\kappa^{4}\nu^{\fr{4}{3}}\sum_{\substack{\al'\le\al\\|\al'|=1, |\al|=2}}C_\al^{\al'}\left\la\pr_x^{\al'}g_{\iota*}^{(m+\beta)}, \pr_v^{\al-\al'}\pr_v^\al g^{(m+\beta)}_{\iota*}\right\ra_{\ell-4-2\iota}\\
    \nn&+\kappa^{4}\nu^{\fr{4}{3}}\sum_{\substack{\al'\le\al\\|\al'|=1, |\al|=2}}C_\al^{\al'}\left\la\pr_x^{\al'}g_{\iota*}^{(m+\beta)}, \pr_v^{\al-\al'}\la v\ra^{2\ell-8-4\iota}\pr_v^\al g^{(m+\beta)}_{\iota*}\right\ra_{x,v}\\
    \le\nn&\kappa \left(\kappa\nu^{\fr13}\big\|\nb_xg^{(m+\beta)}_{\iota*}\big\|_{L^2_{x,v}(\ell-2-2\iota)}^2\right)+C\kappa^2 \left(\nu\sum_{0\le|\al|\le 2}\big\|\kappa^{|\al|}\nu^{\fr{|\al|}{3}}\pr_v^\al g_{\iota*}^{(m+\beta)}\big\|_{\sig,\ell-2|\al|-2\iota}^2\right)\\
    \le& C\kappa\left(\nu^{\fr13}\mathcal{D}^{m+\beta}_{\ell-2\iota}(g_{\iota*}(t))\right),
\end{align}
where $C$ is a positive constant depending on $\ell$ and $\fr{1}{c}$ appearing in \eqref{coercive}.

By Lemma \ref{lem-coercive-L} and Remark \ref{rem-Y/t}, we find that
there exist small positive constants $\underline{\lm},\underline{\kappa}\in(0,1)$, depending on $M_\sig, M_{\mu}, s$ and $N$, such that if
\begin{align}\label{underline:lm-kappa}
    \lm(0)=\lm_{\infty}+\tl{\delta}\le\underline{\lm}, \quad \kappa\le\underline{\kappa},
\end{align}
there holds
\begin{align}\label{est-dissip1}
    {\rm Dis}_1(g_{\iota*})=\nn&\nu {\rm A}_0\sum_{m\in\N^6}\sum_{\substack{\beta\in\N^6\\ |\beta|\le N}}\kappa^{2|\beta|}a_{m,\lm,s}^2(t)\sum_{|\al|\le1}\left\la Z^{m+\beta} L\pr^{\al}_{x}g_{\iota*},\pr_x^{\al}g_{\iota*}^{(m+\beta)}\right\ra_{\ell-2|\al|-2\iota}\\
    \ge\nn&(\fr{15}{16}-\zeta)\nu^{\fr13} \sum_{|\al|\le1}\big\|\sqrt{{\rm A}_0}\nu^{\fr13} \pr_x^\al g_{\iota*}\big\|_{\mathcal{G}^{\lm,N}_{s,\sig, \ell-2|\al|-2\iota}}^2\\
    &-\nu{\rm A}_0 C_\zeta \sum_{m\in \N^6}\sum_{|\beta|\le N}\kappa^{2|\beta|}a_{m,\lm,s}^2(t)\sum_{|\al|\le1}\left\| \pr_x^\al  g^{(m+\beta)}_{\iota*}\right\|_{L^2_{x,v}(|v|\le2\zeta)}^2,
\end{align}

\begin{align}\label{est-dissip2}
    {\rm Dis}_2(g_{\iota*})=\nn&\nu\sum_{1\le|\al|\le2} \sum_{m\in\N^6}\sum_{\substack{\beta\in\N^6\\ |\beta|\le N}}\kappa^{2|\beta|}a_{m,\lm,s}^2(t)\left\la \kappa^{|\al|}\nu^{\fr{|\al|}{3}}\pr_v^\al Z^{m+\beta} Lg_{\iota*},\kappa^{|\al|}\nu^{\fr{|\al|}{3}}\pr_v^\al g_{\iota*}^{(m+\beta)}\right\ra_{\ell-2|\al|-2\iota}\\
    \ge\nn&(1-\zeta)\nu \sum_{1\le|\al|\le2}\big\|\kappa^{|\al|}\nu^{\fr{|\al|}{3}} \pr_v^\al g_{\iota*}\big\|_{\mathcal{G}^{\lm,N}_{s, \sig, \ell-2|\al|-2\iota}}^2\\
    \nn&-\fr{1}{16}\nu\sum_{0\le|\al|\le2}\big\|\kappa^{|\al|}\nu^{\fr{|\al|}{3}} \pr_v^\al g_{\iota*}\big\|_{\mathcal{G}^{\lm,N}_{s, \sig, \ell-2|\al|-2\iota}}^2\\
    &-\sum_{1\le|\al|\le2}\nu^{1+\fr{2|\al|}{3}}\kappa^{2|\al|} C_\zeta \sum_{m\in \N^6}\sum_{|\beta|\le N}\kappa^{2|\beta|}a_{m,\lm,s}^2(t)\big\|   g^{(m+\beta)}_{\iota*}\big\|_{L^2_{x,v}(|v|\le2\zeta)}^2,
\end{align}
and 
\begin{align}\label{est-dissip3}
{\rm Dis}_3(g_{\iota*})=\nn&\nu \sum_{m\in\N^6}\sum_{\substack{\beta\in\N^6\\ |\beta|\le N}}\kappa^{2|\beta|+2}a_{m,\lm,s}^2(t)(1+ t)^{-2}\left\la Y Z^{m+\beta} Lg_{\iota*},Yg_{\iota*}^{(m+\beta)}\right\ra_{\ell-2-2\iota}\\
\ge\nn&(\fr{15}{16}-\zeta)\nu^{\fr13}( 1+t)^{-2}\big\|\nu^{\fr13}\kappa Y  g_{\iota*}\big\|_{\mathcal{G}^{\lm,N}_{s,\sig,\ell-2-2\iota}}^2-\nu \zeta'C_{\zeta} \kappa^{2}\|\nb_xg_{\iota*}\|_{\mathcal{G}^{\lm,N}_{s,\sig,\ell-2-2\iota}}^2\\
    \nn&-\nu C_{\zeta'}C_{\zeta} \sum_{m\in \N^6}\sum_{|\beta|\le N}\kappa^{2|\beta|+2}a_{m,\lm,s}^2(t)\big\|   g_{\iota*}^{(m+\beta)}\big\|_{L^2_{x,v}(|v|\le2\zeta)}^2\\
    &-\nu C_{\zeta}\sum_{m\in \N^6}\sum_{|\beta|\le N}\kappa^{2|\beta|+2}a_{m,\lm,s}^2(t)\big\| \nb_x  g_{\iota*}^{(m+\beta)}\big\|_{L^2_{x,v}(|v|\le2\zeta)}^2.
\end{align}
For the cross terms, let us denote
\begin{align}\label{def-dissip4}
    {\rm Dis}_4(g_{\iota*})=&\nu^{\fr43}\sum_{m\in\N^6}\sum_{\substack{\beta\in\N^6\\|\beta|\le N}}\kappa^{2|\beta|+1}a_{m,\lm,s}^2(t)\bigg(\left\la Z^{m+\beta}L\nb_{x}g_{\iota*},  \nb_{v}g_{\iota*}^{(m+\beta)}\right\ra_{\ell-2-2\iota}\\
    &+ \left\la \nb_{v}Z^{m+\beta} Lg_{\iota*}, \nb_{x}g_{\iota*}^{(m+\beta)}\right\ra_{\ell-2-2\iota}\bigg).
\end{align}
We infer from Corollary \ref{coro-cross} that
\begin{align}\label{est-cro}
|{\rm Dis}_4(g_{\iota*})|\le\nn&C{\rm A}_0^{-\fr12} \nu^{\fr13}\big\|\sqrt{{\rm A}_0}\nu^{\fr13}\nb_xg_{\iota*}\big\|_{\mathcal{G}^{\lm,N}_{s,\sig,\ell-2-2\iota}}\\
&\times\Big(\big\|\kappa \nu^{\fr23}\nb_v g_{\iota*}\big\|_{\mathcal{G}^{\lm,N}_{s,\sig,\ell-2-2\iota}}+\nu^{\fr13}\kappa {\rm A}_0^{-\fr12}\big\|\sqrt{{\rm A}_0}\nu^{\fr13}g_{\iota*}\big\|_{\mathcal{G}^{\lm,N}_{s,\sig,\ell-2\iota}}\Big).
    \end{align}
Let us take $\zeta=\fr{1}{8}$,  $\zeta'$ so small that
\[
\zeta'C_\zeta \le \fr{1}{32},
\]
and ${\rm A}_0$ so large that
\begin{equation}\label{large-A_0}
    C{\rm A}_0^{-\fr12}\le \fr{1}{32},
\end{equation}
then it follows from \eqref{bad-cross}, \eqref{est-dissip1}--\eqref{est-cro} that
\begin{align}\label{lower-Dis}
\nn&\sum_{j=1}^4{\rm Dis}_j(g_{\iota*})+\nu^{\fr13}\kappa\|\nb_xg_{\iota*}\|^2_{\mathcal{G}^{\lm,N}_{s,\ell-2-2\iota}}\\
\nn&+\sum_{m\in\N^6}\sum_{\substack{\beta\in\N^6\\ |\beta|\le N}}\kappa^{2|\beta|}a_{m,\lm,s}^2(t)\sum_{1\le|\al|\le2}\kappa^{2|\al|}\nu^{\fr{2|\al|}{3}}\left\la\big[\pr_v^\al, v\cdot\nb_x g^{(m+\beta)}_{\iota*}\big],  \pr_v^\al g^{(m+\beta)}_{\iota*}\right\ra_{\ell-2|\al|-2\iota}\\
\ge& \fr{1}{2}\nu^{\fr13}\mathcal{D}_{\ell-2\iota}(g_{\iota*}(t))-\nu {\rm A}_0C \sum_{m\in \N^6}\sum_{|\beta|\le N}\kappa^{2|\beta|}a_{m,\lm,s}^2(t)\sum_{|\al|\le1}\left\| \pr_x^\al  g^{(m+\beta)}_{\iota*}\right\|_{L^2_{x,v}(|v|\le2\zeta)}^2.
\end{align}

For the last term on the right-hand side of \eqref{lower-Dis}, if $\iota=0$, one can use the decomposition $g=g_{\ne}+g_0$ and Poincar\'{e} inequality to obtain
\begin{align}\label{est-low1}
\nn&\nu {\rm A}_0C \sum_{m\in \N^6}\sum_{|\beta|\le N}\kappa^{2|\beta|}a_{m,\lm,s}^2(t)\sum_{|\al|\le1}\left\| \pr_x^\al  g^{(m+\beta)}\right\|_{L^2_{x,v}(|v|\le2\zeta)}^2\\
\le\nn&\big(\nu^{\fr23} {\rm A}_0\kappa^{-1}C\big) \Big(\nu^{\fr13}\kappa\| \nb_x  g\|_{\mathcal{G}^{\lm,N}_{s,0}}^2\Big)\\
&+\nu {\rm A}_0C \sum_{m\in \N^6}\sum_{|\beta|\le N}\kappa^{2|\beta|}a_{m,\lm,s}^2(t)\big\|   g^{(m+\beta)}_{0}\big\|_{L^2_{v}(|v|\le2\zeta)}^2.
\end{align}
If $|m+\beta|>0$, there exists $\bar{e}_l\in\N^6$ for some $l\in\{1, 2,3\}$, such that $g_0^{(m+\beta)}=\pr_{v_l}g_0^{(m+\beta-\bar{e}_l)}$. Then by \eqref{coercive}, we have
\begin{align*}
    \big\|g_0^{(m+\beta)}\big\|_{L^2_v(|v|\le2\zeta)}\les \big\|g^{(m+\beta-\bar{e}_l)}\big\|_{\sig,\ell}^2.
\end{align*}
Two sub-cases will be involved. If $|\beta|>0$, we rewrite $\kappa^{2|\beta|}=\kappa^2\kappa^{2|\beta-\bar{e}_l|}$. If $|m|>0$, assume without loss of generality that $m_l\ge1$, then the following fact holds:
\begin{align*}
    a_{m,\lm,s}^2(t)=&\lm^2a_{m-\bar{e}_l,\lm,s}^2(t)\Big[\big(\fr{1+|m|}{|m|}\big)^{12}\fr{1}{|m|^{\fr1s-1}}\fr{1}{m_l}\Big]^2
    \le2^{24}\lm^2a_{m-\bar{e}_l,\lm,s}^2(t).
\end{align*}
Therefore,
\begin{align}\label{est-low2}
    \nn&\nu {\rm A}_0C \sum_{\substack{m,\beta\in \N^6,|\beta|\le N\\|m+\beta|>0}}\kappa^{2|\beta|}a_{m,\lm,s}^2(t)\big\|   g^{(m+\beta)}_{0}\big\|_{L^2_{v}(|v|\le2\zeta)}^2\\
    \le&\max\{\kappa^2,\lm^2\}C\Big(\nu^{\fr13} \big\|\sqrt{{\rm A}_0}\nu^{\fr13}g\big\|^2_{\mathcal{G}^{\lm,N}_{s,\sig,\ell}}\Big).
\end{align}
If $|m+\beta|=0$, we need to bound $\nu{\rm A}_0C\|g_0\|_{L^2_v(|v|\le2\zeta)}^2$. By Lemma \ref{lem-lower order}, we find that
\begin{align}\label{bd-g0-low}
\nn&\nu{\rm A}_0C\|g_0\|_{L^2_v(|v|\le2\zeta)}^2\\
\le&\nu{\rm A}_0C\left\|\la v\ra^{\ell-2}\nb_xg\right\|_{L^2_{x,v}}^2+\nu{\rm A}_0C\|E\|^4_{L^2_{x}}+\nu {\rm A}_0C  \left\|({\rm I}-{\rm P}_0)g\right\|_{\sig}^2.
\end{align}

For the case $\iota=1, 2,\cdots,[\ell/3]+1$, $g_0$ is absent, now \eqref{est-low1} reduces to
\begin{align}\label{est-low3}
\nn&\nu {\rm A}_0C \sum_{m\in \N^6}\sum_{|\beta|\le N}\kappa^{2|\beta|}a_{m,\lm,s}^2(t)\sum_{|\al|\le1}\left\| \pr_x^\al  g^{(m+\beta)}_{\iota\ne}\right\|_{L^2_{x,v}(|v|\le2\zeta)}^2\\
\le&\big(\nu^{\fr23} {\rm A}_0\kappa^{-1}C\big) \Big(\nu^{\fr13}\kappa\| \nb_x  g_{\iota\ne}\|_{\mathcal{G}^{\lm,N}_{s,0}}^2\Big).
\end{align}

If $\iota=0$, we further take the constants $\underline{\lm}$ and $\underline{\kappa}$ in \eqref{underline:lm-kappa} smaller, if necessary, so that the right hand side of \eqref{est-low2} can be absorbed  by $\fr18\mathcal{D}_\ell(g(t))$. Moreover, let $\nu_0 $ be such that
\begin{align}\label{small-nu_0}
\nu_0^{\fr23}{\rm A}_0\kappa^{-1}C\le\fr{1}{32},
\end{align}
then for all $\nu\le\nu_0$,
\[
\nu{\rm A}_0C\left\|\la v\ra^{\ell-2}\nb_xg\right\|_{L^2_{x,v}}^2\le \nu{\rm A}_0C\left\|\nb_xg\right\|_{\mathcal{G}^{\lm,N}_{s,\ell-2}}^2\le\fr{1}{32}\nu^{\fr13}\kappa\|\nb_xg\|_{\mathcal{G}^{\lm,N}_{s,\ell-2}}^2.
\]
Meanwhile, let ${\rm B}_0$ be such that
\begin{align}\label{large-B_0}
{\rm B}_0\dl\ge2 {\rm A}_0C,
\end{align}
here $C$ is the constant appearing in \eqref{bd-g0-low}.
Then it follows from \eqref{lower-Dis}--\eqref{bd-g0-low} that
\begin{align}\label{bd-dis}
    \nn&{\rm B}_0\dl \nu \|({\rm I}-{\rm P}_0)g\|^2_{\sig}+\sum_{j=1}^4{\rm Dis}_j(g)+\nu^{\fr13}\kappa\|\nb_xg\|^2_{\mathcal{G}^{\lm,N}_{s,\ell-2}}\\
    \nn&+\sum_{m\in\N^6}\sum_{\substack{\beta\in\N^6\\ |\beta|\le N}}\kappa^{2|\beta|}a_{m,\lm,s}^2(t)\sum_{1\le|\al|\le2}\kappa^{2|\al|}\nu^{\fr{2|\al|}{3}}\left\la\big[\pr_v^\al, v\cdot\nb_x g^{(m+\beta)}\big],  \pr_v^\al g^{(m+\beta)}\right\ra_{\ell-2|\al|}\\
\ge&\fr14\nu^{\fr13}\Big({\rm B}_0\dl \nu^{\fr23} \|({\rm I}-{\rm P}_0)g\|^2_{\sig}+\mathcal{D}_{\ell}(g(t))\Big)-C\nu {\rm A}_0\|E\|_{L^2_x}^4.
\end{align}

If $\iota=1,\cdots, [\ell/3]+1$, it follows from \eqref{lower-Dis} and \eqref{est-low3} that
\begin{align}\label{bd-dis-iota}
    \nn&\sum_{j=1}^4{\rm Dis}_j(g_{\iota\ne})+\nu^{\fr13}\kappa\|\nb_xg_{\iota\ne}\|^2_{\mathcal{G}^{\lm,N}_{s,\ell-2-2\iota}}\\
    \nn&+\sum_{m\in\N^6}\sum_{\substack{\beta\in\N^6\\ |\beta|\le N}}\kappa^{2|\beta|}a_{m,\lm,s}^2(t)\sum_{1\le|\al|\le2}\kappa^{2|\al|}\nu^{\fr{2|\al|}{3}}\left\la\big[\pr_v^\al, v\cdot\nb_x g^{(m+\beta)}_{\iota\ne}\big],  \pr_v^\al g^{(m+\beta)}_{\iota\ne}\right\ra_{\ell-2|\al|-2\iota}\\
\ge&\fr14\nu^{\fr13}\mathcal{D}_{\ell-2\ell}(g_{\iota\ne}(t))=\fr14\nu^{\fr13}w^2_{\iota}(t)\mathcal{D}_{\ell-2\ell}(g_{\ne}(t)).
\end{align}
Now we infer from \eqref{bd-dis}, \eqref{bd-dis-iota} that, to absorb the right hand side of \eqref{dt-(1+t)}, one can choose $\kappa_0$ so small that
\begin{align}\label{small-kappa_0}
    \ell\fr{\rm A_0}{\kappa}\kappa_0\le\fr{1}{8\bar{C}},
\end{align}
where $\bar{C}$ is the constant appearing in \eqref{dt-(1+t)}.

Collecting the estimates in \eqref{en-0}--\eqref{est-linear0}, \eqref{dt-(1+t)},
\eqref{bd-dis}, \eqref{bd-dis-iota}, choosing the constants $\underline{\lm}, \underline{\kappa},{\rm A}_0, \nu_0, {\rm B}_0$ and $\kappa_0$ successively satisfying \eqref{underline:lm-kappa}, \eqref{large-A_0}, \eqref{small-nu_0},\eqref{large-B_0}, and \eqref{small-kappa_0}, recalling the definitions of $\mathcal{E}(g(t))$, $\mathcal{D}(g(t))$ and $\mathcal{CK}(g(t))$ in \eqref{total-en}--\eqref{total-CK}, we find that for all $\lm,\kappa$ and $\nu$ satisfying $\lm\le\underline{\lm}, \kappa\le\underline{\kappa}$ and $\nu\le\nu_0$, the estimate \eqref{en-ineq-Linear-g} holds.
\end{proof}

The following four lemmas are dedicated to dealing with the contributions from the right hand side of the energy inequality \eqref{en-ineq-Linear-g}.

We first go to bound the linear contributions $L[\mu]_{\iota;1)}$ and $L[\mu]_{\iota;2)}$.
\begin{lem}\label{lem-Linear contributions}
Under the bootstrap hypotheses \eqref{bd-en}--\eqref{H-phi}, there holds  
\begin{align}\label{est-Lmu}
    \nn&2\sum_{\iota=0}^{[\ell/3]+1}\Big[L[\mu]_{\iota;1)}+L[\mu]_{\iota;2)}\Big]\\
    \nn\le&C(\sqrt{\rm A_0}+1)\sum_{\iota=1}^{[\ell/3]+1}\Bigg[\Big({  w_{\iota}(t)\la t\ra^{\fr{1+a-s}{2}}\|\rho\|_{\mathcal{G}^{\lm,N}_s}}\Big)\bigg(w_{\iota}(t)\mathcal{CK}_{\ell-2\iota}( g_{\ne}(t))^{\fr12}\\
    \nn&+\fr{1}{\la t\ra^{\fr{1+a}{2}}}w_{\iota}(t)\mathcal{E}_{\ell-2\iota}(g_{\ne}(t))^{\fr12}\bigg)+\Big({  w_{\iota}(t)\|\rho\|_{\mathcal{G}^{\lm,\fr{N}{2}}_s}}\Big)\Big(w_{\iota}(t)\mathcal{E}_{\ell-2\iota}(g_{\ne}(t))^{\fr12}\Big)\Bigg]\\
    &+C(\sqrt{{\rm A}_0}+1)\|\rho\|_{\mathcal{G}^{\lm,N}_{s}}\mathcal{E}_{\ell}(g(t))^{\fr12}.
\end{align}
\end{lem}
\begin{proof}
The proof will be divided  into two cases according to whether $\iota=0$ or not.

{\bf Case 1: $\iota=0$.}
Note that for $|\al|=1$, 
\begin{gather*}
    \pr_x^\al (e^\phi E\cdot v\sqrt{\mu})=\pr_x^\al (e^\phi E)\cdot v\sqrt{\mu},\\
     Y^\al (e^\phi E\cdot v\sqrt{\mu})=t\pr_x^\al (e^\phi E)\cdot v\sqrt{\mu}+ (e^\phi E)\cdot \pr_v^\al(v\sqrt{\mu}),
\end{gather*}
and for $1\le|\al|\le2$, there holds
\begin{align*}
    \pr_v^\al (e^\phi E\cdot v\sqrt{\mu})=(e^\phi E)\cdot \pr_v^\al(v\sqrt{\mu}).
\end{align*}
Then by the product estimate in $\mathcal{G}^{\lm,N}_{s,\ell-2|\al|}$ and Lemma \ref{lem-compose} with $-\phi$ replaced by $\phi$, we find that
\begin{align}\label{est-Lmu01}
    L[\mu]_{0;1)}
    \le\nn& C\Big(\sqrt{\rm{A}_0}+\kappa+\kappa\nu^{\fr13}\Big)\big(\|e^\phi-1\|_{\mathcal{G}^{\lm,N}_{s}}+1\big)\\
    \nn&\times\Big(\|\nb_x\phi\|_{\mathcal{G}^{\lm,N}_{s}}\|E\|_{\mathcal{G}^{\lm,N}_{s}}+\|\nb_xE\|_{\mathcal{G}^{\lm,N}_{s}}+\|E\|_{\mathcal{G}^{\lm,N}_{s}}\Big)\\
    \nn&\times\bigg[\Big(\sqrt{{\rm A}_0}\sum_{|\al|\le1}\|\pr_x^\al g\|_{\mathcal{G}^{\lm,N}_{s,\ell-2|\al|}}\Big)+\kappa(1+t)^{-1}\|Yg\|_{\mathcal{G}^{\lm,N}_{s,\ell-2}}\\
    \nn&+\sum_{1\le|\al|\le2}\big\|\kappa^{|\al|}\nu^{\fr{|\al|}{3}}\pr_v^\al g\big\|_{\mathcal{G}^{\lm,N}_{s,\ell-2|\al|}}\bigg]\\
    \le&C(\sqrt{{\rm A}_0}+1)\|\rho\|_{\mathcal{G}^{\lm,N}_{s}}\mathcal{E}_{\ell}(g(t))^{\fr12}.
\end{align}

{\bf Case 2: $\iota=1, 2, \cdots, [\ell/3]+1$.} Now
\begin{align*}
    L[\mu]_{\iota;1)}=&\sum_{m\in\N^6,|\beta|\le N}\kappa^{2|\beta|}a_{m,\lm,s}^2(t)\bigg(\sum_{|\al|\le1}2{\rm A}_0\left\la \pr_{x}^\al Z^{m+\beta}\big(e^{\phi}E\cdot v\sqrt{\mu}\big)_{\iota\ne},\pr_x^{\al}g^{(m+\beta)}_{\iota\ne}\right\ra_{\ell-2|\al|-2\iota}\\
    &+2(1+t)^{-2}\kappa^2\sum_{|\al|=1}\left\la Y^{\al} Z^{m+\beta}\big(e^{\phi}E\cdot v\sqrt{\mu}\big)_{\iota\ne},Y^{\al}g^{(m+\beta)}_{\iota\ne}\right\ra_{\ell-2|\al|-2\iota}\\
&+2\sum_{1\le|\al|\le2}\kappa^{2|\al|}\nu^{\fr{2|\al|}{3}}\left\la \pr_v^\al Z^{m+\beta}\big(e^{\phi}E\cdot v\sqrt{\mu} \big)_{\iota\ne},\pr_v^\al g^{(m+\beta)}_{\iota\ne}\right\ra_{\ell-2|\al|-2\iota} \bigg)=\sum_{1\le i\le3} L[\mu]_{\iota;1)}^{(i)},  
\end{align*}
with
\begin{align*}
    \big(e^{\phi}E\cdot v\sqrt{\mu}\big)_{\iota\ne}=w_{\iota}(t)\big(e^{\phi}E\cdot v\sqrt{\mu}\big)_{\ne}.
\end{align*}
We can further split $L[\mu]_{\iota;1)}^{(1)}$ into two terms:
\begin{align*}
    L[\mu]_{\iota;1)}^{(1)}
    =&2{\rm A}_0\sum_{|\al|\le1}\sum_{m\in\N^6,|\beta|\le N}\kappa^{2|\beta|}a_{m,\lm,s}^2(t)\Big(\sum_{\substack{n\le m,\beta'\le\beta\\|\beta'|\le |\beta|/2}}+\sum_{\substack{n\le m,\beta'\le\beta\\|\beta'|>|\beta|/2}}\Big)C_{\beta}^{\beta'}C_m^n\\
    &\times w_{\iota}(t)\int_{\T^3\times\R^3} Z^{n+\beta'}\big(\pr_{x}^\al(e^{\phi}E)\big)_{\ne}\cdot Z^{m-n+\beta-\beta'}(v\sqrt{\mu})\pr_x^{\al}g^{(m+\beta)}_{\iota\ne}\la v\ra^{2\ell-4|\al|-4\iota }dxdv\\
    =&L[\mu]_{\iota;1)}^{(1),{\rm LH}}+L[\mu]_{\iota;1)}^{(1),{\rm HL}}.
\end{align*}

To bound $L[\mu]_{\iota;1)}^{(1),{\rm HL}}$, note first that for $k\ne0$,
\begin{align}\label{up-kkt}
    1=\big(|k,kt|^{-\fr{s}{2}}\la t\ra^{\fr{1+a}{2}}\big)\big(|k,kt|^{\fr{s}{2}}\la t\ra^{-\fr{1+a}{2}}\big)\les\big(|k|^{-\fr{s}{2}}\la t\ra^{\fr{1+a-s}{2}}\big)\big(|k,\eta+kt|^{\fr{s}{2}}\la t\ra^{-\fr{1+a}{2}}\big)\la\eta\ra^{\fr{s}{2}} .
\end{align}
Then by Lemma \ref{lem-compose} and Corollary \ref{coro-CK}, one deduces that
\begin{align}\label{est-Lmuiota11HL}
    \nn&L[\mu]_{\iota;1)}^{(1),{\rm HL}}\\
    \le\nn&C {\rm A}_0w_{\iota}(t)\sum_{|\al|\le1}\sum_{m\in\N^6,|\beta|\le N}\kappa^{2|\beta|}\sum_{\substack{n\le m,\beta'\le\beta\\|\beta'|>|\beta|/2}}b_{m,n,s}\\
    \nn&\times\sum _{k\in\Z^3_*}\int_{\R^3} a_{n,\lm,s}(t) \fr{\la t\ra^{\fr{1+a-s}{2}}}{|k|^\fr{s}{2}}\Big|\mathcal{F}_{x}\big[Z^{n+\beta'}\pr_{x}^\al(e^{\phi}E)\big]_k\Big| \\
    \nn&\times a_{m-n,\lm,s}(t)\Big|\la\eta\ra \mathcal{F}_v\big[Z^{m-n+\beta-\beta'}(v\sqrt{\mu})\la v\ra^{\ell-2|\al|-2\iota }\big](\eta)\Big|\\
    \nn&\times a_{m,\lm,s}(t)\Big|\fr{1}{\la t\ra^{\fr{1+a}{2}}}\mathcal{F}_{x,v}\big[| Z_x|^\fr{s}{2}\big(\pr_x^{\al}g^{(m+\beta)}_{\iota\ne}\la v\ra^{\ell-2|\al|-2\iota }\big)\big]_k(\eta)\Big|d\eta\\
    \le\nn&C{\rm A}_0w_{\iota}(t)\sum_{|\al|\le1}\big\|\la t\ra^{\fr{1+a-s}{2}}|\nb_x|^{-\fr{s}{2}}\pr_x^\al(e^{\phi}E)\big\|_{\mathcal{G}^{\lm,N}_{s}}\big\|\la\nb_v\ra(v\sqrt{\mu})\big\|_{\mathcal{G}^{\lm,\fr{N}{2}}_{s,\ell-2|\al|-2\iota}}\\
    \nn&\times \fr{1}{\la t\ra^{\fr{1+a}{2}}}\bigg(\sum_{m\in\N^6,|\beta|\le N}\kappa^{2|\beta|}a_{m,\lm,s}^2(t)\left\|| Z|^{\fr{s}{2}}\big(\pr_x^\al g_{\iota\ne}^{(m+\beta)}\la v\ra^{\ell-2|\al|-2\iota}\big)\right\|^2_{L^2_{x,v}}\bigg)^{\fr12}\\
    \le\nn&C{\rm A}_0w_{\iota}(t)\la t\ra^{\fr{1+a-s}{2}}\big(1+\|\phi\|_{\mathcal{G}^{\lm,N}_{s}}\big)\Big(\|\nb_x\phi\|_{\mathcal{G}^{\lm,N}_{s}}\|E\|_{\mathcal{G}^{\lm,N}_{s}}+\|\nb_xE\|_{\mathcal{G}^{\lm,N}_s}+\|E\|_{\mathcal{G}^{\lm,N}_s}\Big)\\
    \nn&\times w_{\iota}(t)\sum_{|\al|\le1}\Big(\mathfrak{CK}_{\ell-2|\al|-2\iota}[\pr_x^\al g_{\ne}(t)]^{\fr12}+\fr{1}{\la t\ra^{\fr{1+a}{2}}}\|\pr_x^\al g_{\ne}\|_{\mathcal{G}^{\lm,N}_{s,\ell-2|\al|-2\iota}}\Big)\\
    \le& C\sqrt{{\rm A}_0}\Big[w_{\iota}(t)\la t\ra^{\fr{1+a-s}{2}}\|\rho\|_{\mathcal{G}^{\lm,N}_s}\Big]\Big(w_{\iota}(t)\mathcal{CK}_{\ell-2\iota}( g_{\ne}(t))^{\fr12}+\fr{1}{\la t\ra^{\fr{1+a}{2}}}w_{\iota}(t)\mathcal{E}_{\ell-2\iota}(g_{\ne}(t))^{\fr12}\Big).
    \end{align}
For $L[\mu]_{\iota;1)}^{(1),{\rm LH}}$, there is no need to use the CK term. Indeed,
\begin{align}\label{est-Lmuiota11LH}
    \nn&L[\mu]_{\iota;1)}^{(1),{\rm LH}}\\
    \le\nn&C {\rm A}_0 w_{\iota}(t)\sum_{|\al|\le1}\big\|\pr_x^\al(e^{\phi}E)\big\|_{\mathcal{G}^{\lm,\fr{N}{2}}_{s}}\big\|v\sqrt{\mu}\|_{\mathcal{G}^{\lm,N}_{s,\ell-2|\al|-2\iota}} \left\|\pr_x^\al g_{\iota\ne}\right\|_{\mathcal{G}^{\lm,N}_{s,\ell-2|\al|-2\iota}}\\
    \le\nn&C {\rm A}_0 w_{\iota}(t)\big(1+\|\phi\|_{\mathcal{G}^{\lm,\fr{N}{2}}_{s}}\big)\Big(\|\nb_x\phi\|_{\mathcal{G}^{\lm,\fr{N}{2}}_{s}}\|E\|_{\mathcal{G}^{\lm,\fr{N}{2}}_{s}}+\|\nb_xE\|_{\mathcal{G}^{\lm,\fr{N}{2}}_s}+\|E\|_{\mathcal{G}^{\lm,\fr{N}{2}}_s}\Big)\\
    \nn&\times \left\|\pr_x^\al g_{\iota\ne}\right\|_{\mathcal{G}^{\lm,N}_{s,\ell-2|\al|-2\iota}}\\
    \le& C\sqrt{{\rm A}_0}\Big({  w_{\iota}(t)\|\rho\|_{\mathcal{G}^{\lm,\fr{N}{2}}_s}}\Big)\Big(w_{\iota}(t)\mathcal{E}_{\ell-2\iota}(g_{\ne}(t))^{\fr12}\Big).
    \end{align}
Clearly, $L[\mu]_{\iota;1)}^{(2)}$ and $L[\mu]_{\iota;1)}^{(3)}$ can be treated in the same way. We thus have
\begin{align}\label{est-Lmuiota1}
    L[\mu]_{\iota;1)}\le\nn& C(\sqrt{\rm A_0}+1)\Bigg[\Big({  w_{\iota}(t)\la t\ra^{\fr{1+a-s}{2}}\|\rho\|_{\mathcal{G}^{\lm,N}_s}}\Big)\bigg(w_{\iota}(t)\mathcal{CK}_{\ell-2\iota}( g_{\ne}(t))^{\fr12}\\
    &+\fr{1}{\la t\ra^{\fr{1+a}{2}}}w_{\iota}(t)\mathcal{E}_{\ell-2\iota}(g_{\ne}(t))^{\fr12}\bigg)+\Big({  w_{\iota}(t)\|\rho\|_{\mathcal{G}^{\lm,\fr{N}{2}}_s}}\Big)\Big(w_{\iota}(t)\mathcal{E}_{\ell-2\iota}(g_{\ne}(t))^{\fr12}\Big)\Bigg].
\end{align}
Moreover, it is easy to see that the upper bounds in \eqref{est-Lmu01} and \eqref{est-Lmuiota1} are still valid for $L[\mu]_{\iota;2)}$ for $\iota=0$ and $\iota=1,2,\cdots [\ell/3]+1$, respectively, we omit the details here. Thus, \eqref{est-Lmu} holds and the proof of Lemma \ref{lem-Linear contributions} is completed.
\end{proof}

 Next we turn to bound $N[\pr_t\phi]_{\iota;1)}$ and $N[\pr_t\phi]_{\iota;2)}$.
\begin{lem}\label{lem-pr_tphi}
Under the bootstrap hypotheses \eqref{bd-en}--\eqref{H-phi}, there holds
\begin{align}\label{est-ptphi}
    2\sum_{\iota=0}^{[\ell/3]+1}\Big[N[\pr_t\phi]_{\iota;1)}+N[\pr_t\phi]_{\iota;2)}\Big]
    \le \nn& C\Big(\mathcal{E}_{\ell}(g(t))+\sum_{\iota=1}^{[\ell/3]+1}w_{\iota}^2(t)\mathcal{E}_{\ell-2\iota}(g_{\ne}(t))\Big)\|M(t)\|_{\mathcal{G}^{\lm,N}_s}\\
    \nn&+C\mathcal{E}_\ell(g(t))^{\fr12}\sum_{\iota=1}^{[\ell/3]+1}\Bigg[\Big({  w_{\iota}(t)\la t\ra^{\fr{1+a-s}{2}}\|M\|_{\mathcal{G}^{\lm,N}_s}}\Big)\\
    \nn&\times\bigg(w_{\iota}(t)\mathcal{CK}_{\ell-2\iota}( g_{\ne}(t))^{\fr12}+\fr{1}{\la t\ra^{\fr{1+a}{2}}}w_{\iota}(t)\mathcal{E}_{\ell-2\iota}(g_{\ne}(t))^{\fr12}\bigg)\\
    &+\Big({  w_{\iota}(t)\|M\|_{\mathcal{G}^{\lm,\fr{N}{2}}_s}}\Big)\Big(w_{\iota}(t)\mathcal{E}_{\ell-2\iota}(g_{\ne}(t))^{\fr12}\Big)\Bigg].
\end{align}
\end{lem}
\begin{proof}
Clearly, $N[\pr_t\phi]_{\iota;2)}$ and $N[\pr_t\phi]_{\iota;1)}$ possess the same structure, and obey the same upper bounds. We will estimate  them together.

{\bf Case 1: $\iota=0$.}
Using  the product estimate \eqref{product1} and \eqref{pt-phi}, similar to \eqref{est-Lmu01}, we obtain
 \begin{align*}
    N[\pr_t\phi]_{0;1)}+N[\pr_t\phi]_{0;2)}\le C&\mathcal{E}_{\ell}(g(t))\|M(t)\|_{\mathcal{G}^{\lm,N}_s}.
 \end{align*}

{\bf Case 2: $\iota=1, 2, \cdots, [\ell/3]+1$.} Note that
\begin{align*}
    (\pr_t\phi g)_{\ne}=\pr_t\phi g_0+(\pr_t\phi g_{\ne})_{\ne}.
\end{align*}
According to this decomposition of $\pr_t\phi g$, it is natural to split $N[\pr_t\phi]_{\iota;i)}$ into two parts, namely, $N[\pr_t\phi]_{\iota;i)}=N[\pr_t\phi]_{\iota;i)}^{\ne}+N[\pr_t\phi]_{\iota;i)}^{0}$, where $N[\pr_t\phi]_{\iota;i)}^{\ne}$ is the analogue of $N[\pr_t\phi]_{\iota;i)}$ with $(\pr_t\phi g)_{\ne}$ replaced by $(\pr_t\phi g_{\ne})_{\ne}$, and $N[\pr_t\phi]_{\iota;i)}^{0}$ is also the analogue of $N[\pr_t\phi]_{\iota;i)}$  with $(\pr_t\phi g)_{\ne}$ replaced by $\pr_t\phi g_{0}, i=1,2$.

$N[\pr_t\phi]_{\iota;i)}^{\ne}, i=1,2$ can be treated in the same way as the case $\iota=0$ because now the time weight $w_{\iota}(t)$ can hit on $g_{\ne}$. We state the result here:
\begin{align*}
     N[\pr_t\phi]_{\iota;1)}^{\ne}+N[\pr_t\phi]_{\iota;2)}^{\ne}\le C&\big(w_{\iota}^2(t)\mathcal{E}_{\ell-2\iota}(g_{\ne}(t))\big)\|M(t)\|_{\mathcal{G}^{\lm,N}_s}.
 \end{align*}

As for $N[\pr_t\phi]_{\iota;i)}^0, i=1, 2$, thanks to \eqref{up-kkt},  one can follow the estimate of  $L[\mu]_{\iota,1)}^{(1)}$ (see \eqref{est-Lmuiota11HL} and \eqref{est-Lmuiota11LH})  line by line to obtain
\begin{align*}
    &N[\pr_t\phi]_{\iota;1)}^{0}+N[\pr_t\phi]_{\iota;2)}^{0}\\
    \le& C\Big(\sqrt{\rm A_0}\|\la \nb_v\ra g_0\|_{\mathcal{G}^{\lm,\fr{N}{2}}_{s,\ell}}+\kappa(1+t)^{-1}\|Y g_0\la\nb_v\ra \|_{\mathcal{G}^{\lm,\fr{N}{2}}_{s,\ell-2}}+\sum_{1\le|\al|\le2}\big\|\kappa^{|\al|}\nu^{\fr{|\al|}{3}}\pr_v^\al\la\nb_v\ra g_0\big\|_{\mathcal{G}^{\lm,\fr{N}{2}}_{s,\ell-2|\al|}}\Big)\\
    &\times \Big({  w_{\iota}(t)\la t\ra^{\fr{1+a-s}{2}}\|M\|_{\mathcal{G}^{\lm,N}_s}}\Big)\bigg(w_{\iota}(t)\mathcal{CK}_{\ell-2\iota}( g_{\ne}(t))^{\fr12}+\fr{1}{\la t\ra^{\fr{1+a}{2}}}w_{\iota}(t)\mathcal{E}_{\ell-2\iota}(g_{\ne}(t))^{\fr12}\bigg)\\
    &+C\Big(\sqrt{\rm A_0}\|g_0\|_{\mathcal{G}^{\lm,N}_{s,\ell-2\iota}}+\kappa(1+t)^{-1}\| Y g_0\|_{\mathcal{G}^{\lm,N}_{s,\ell-2-2\iota}}+\sum_{1\le|\al|\le2}\big\|\kappa^{|\al|}\nu^{\fr{|\al|}{3}}\pr_v^\al g_0\big\|_{\mathcal{G}^{\lm,N}_{s,\ell-2|\al|-2\iota}}\Big)\\
    &\times\Big({  w_{\iota}(t)\|M\|_{\mathcal{G}^{\lm,\fr{N}{2}}_s}}\Big)\Big(w_{\iota}(t)\mathcal{E}_{\ell-2\iota}(g_{\ne}(t))^{\fr12}\Big).
\end{align*}
Collecting the above three estimates yields \eqref{est-ptphi}.
\end{proof}

The estimates for the transport terms $N[E]_{\iota;1)}$ and $N[E]_{\iota;2)}$ are established in the following lemma.
\begin{lem}\label{lem-transport}
Under the bootstrap hypotheses \eqref{bd-en}--\eqref{H-phi}, there holds
\begin{align}\label{est-transport}
    \nn&2\sum_{\iota=0}^{[\ell/3]+1}\Big[N[E]_{\iota;1)}+N[E]_{\iota;2)}\Big]\\
    \le \nn&C\la t\ra^{a+2}\|\la \nb_x\ra^2\rho\|_{\mathcal{G}^{\lm,\fr{N}{2}}_{s}}\mathcal{CK}_{\ell}(g(t))+\la t\ra \|\la \nb_x\ra^2\rho\|_{\mathcal{G}^{\lm,\fr{N}{2}}_{s}}\mathcal{E}_{\ell}(g(t))\\
    \nn&+C\left\|\la t\ra^{\fr{3+a-s}{2}}   \rho\right\|_{\mathcal{G}^{\lm,N}_{s}}\mathcal{E}_{\ell}(g(t))^{\fr12} \Big( \mathcal{CK}_{\ell}(g(t))^{\fr12}+\fr{1}{\la t\ra^{\fr{1+a}{2}}}\mathcal{E}_{\ell}(g(t))^{\fr12}\Big)\\
    \nn&+C\sum_{\iota=1}^{[\ell/3]+1}w_{\iota}^2(t)\Bigg[\la t\ra^{a+2}\|\la \nb_x\ra^2\rho\|_{\mathcal{G}^{\lm,\fr{N}{2}}_{s}}\mathcal{CK}_{\ell-2\iota}(g_{\ne}(t))+\la t\ra \|\la \nb_x\ra^2\rho\|_{\mathcal{G}^{\lm,\fr{N}{2}}_{s}}\mathcal{E}_{\ell-2\iota}(g_{\ne}(t))\\
    \nn&+\left\|\la t\ra^{\fr{3+a-s}{2}}   \rho\right\|_{\mathcal{G}^{\lm,N}_{s}}\mathcal{E}_{\ell-2\iota}(g_{\ne}(t))^{\fr12} \Big( \mathcal{CK}_{\ell-2\iota}(g_{\ne}(t))^{\fr12}+\fr{1}{\la t\ra^{\fr{1+a}{2}}}\mathcal{E}_{\ell-2\iota}(g_{\ne}(t))^{\fr12}\Big)\Bigg]\\
    \nn&+C\sum_{\iota=1}^{[\ell/3]+1}\Bigg[w_{\iota}(t)\|\rho\|_{\mathcal{G}^{\lm,\fr{N}{2}}_s}\mathcal{E}_\ell(g_0(t))^{\fr12}\mathcal{E}_{\ell-2\iota}(g_{\iota\ne}(t))^{\fr12}\\
    \nn&+\la t\ra^{\iota+\fr{1+a}{2-s}+\fr{s}{2}}\big\||\nb_x|^{\fr{s}{2}}\rho\big\|_{\mathcal{G}^{\lm,\fr{N}{2}}_s}\Big(\mathcal{CK}_{\ell}(g_{\ne}(t))^{\fr12}+\fr{\mathcal{E}_\ell(g_{\ne}(t))^{\fr12}}{\la t\ra^{\fr{1+a}{2}}}\Big)\mathcal{E}_\ell(g_0(t))^{\fr12}\\
    \nn&+\la t\ra^{\iota+\fr{1+a}{2-s}+\fr{s}{2}}\big\||\nb_x|^{\fr{s}{2}}\rho\big\|_{\mathcal{G}^{\lm,\fr{N}{2}}_s}\Big(\mathcal{CK}_{\ell}(g_{\ne}(t))^{\fr12}+\fr{\mathcal{E}_\ell(g_{\ne}(t))^{\fr12}}{\la t\ra^{\fr{1+a}{2}}}\Big)\\
    \nn&\times\Big(\nu^{\fr16}\mathcal{D}_\ell(g_0(t))^{\fr12}+\mathcal{CK}_{\ell}(g_{0}(t))^{\fr12}+\fr{1}{\la t\ra^{\fr{1+a}{2}}}\mathcal{E}_\ell(g_0(t))^{\fr12}\Big)\\
    &+\Big( w_{\iota}(t)\la t\ra^{\fr{1+a-s}{2}}\|\rho\|_{\mathcal{G}^{\lm,N}_s}\Big)\Big(\mathcal{CK}_{\ell-2\iota}(g_{\iota\ne}(t))^{\fr12}+\fr{1}{\la t\ra^{\fr{1+a}{2}}}\mathcal{E}_{\ell-2\iota}(g_{\iota\ne}(t))^{\fr12}\Big)\mathcal{E}_{\ell}(g_0(t))^{\fr12}\Bigg].
\end{align}
\end{lem}
\begin{proof}
{\bf Case 1: $\iota=0$.}
Let $W\in\big\{\pr_x^\al, (1+ t)^{-|\al|}Y^\al:  0\le|\al|\le1\big\}\cup \big\{ \pr_v^\al: 1\le|\al|\le2\big\}$ with $\al\in\N^3$. The transport nonlinearities are of the following form: 
\begin{align*}
    &-\left\la W Z^{m+\beta}\left(E\cdot\nb_vg\right), Wg^{(m+\beta)}\right\ra_{\ell-2|\al|}\\
    =&(\ell-2|\al|)\int_{\T^3\times\R^3}   E\cdot v \la v\ra^{2\ell-4|\al|-2} \big|Wg^{(m+\beta)}\big|^2dxdv\\
    &-\left\la   [Z^{m+\beta},E\cdot\nb_v] W g,Wg^{(m+\beta)}\right\ra_{\ell-2|\al|}\\
    &-\left\la   Z^{m+\beta}(W E\cdot\nb_vg), Wg^{(m+\beta)}\right\ra_{\ell-2|\al|}\\
    =&T[E]_1^{m,\beta}+T[E]_2^{m,\beta}+T[E]_3^{m,\beta},
\end{align*}
and
\begin{align*}
    &-\left\la \pr_{x_j}Z^{m+\beta}(E\cdot\nb_vg), \pr_{v_j}g^{(m+\beta)} \right\ra_{\ell-2}-\left\la \pr_{v_j}Z^{m+\beta}(E\cdot\nb_vg), \pr_{x_j}g^{(m+\beta)} \right\ra_{\ell-2}\\
    =&(\ell-2)\int_{\T^3\times\R^3}E\cdot v\la v\ra^{2\ell-6}\pr_{x_j}g^{(m+\beta)}\pr_{v_j}g^{(m+\beta)}dvdx\\
    &-\Big(\left\la [Z^{m+\beta},E\cdot\nb_v]\pr_{x_j}g, \pr_{v_j}g^{(m+\beta)}\right\ra_{\ell-2}+\left\la [Z^{m+\beta},E\cdot\nb_v]\pr_{v_j}g, \pr_{x_j}g^{(m+\beta)}\right\ra_{\ell-2}\Big)\\
    &-\left\la Z^{m+\beta}\big( \pr_{x_j}E\cdot\nb_v g\big), \pr_{v_j}g^{(m+\beta)}\right\ra_{\ell-2}=T_c[E]_1^{m,\beta}+T_c[E]_2^{m,\beta}+T_c[E]_3^{m,\beta},
\end{align*}
here $c$ stands for `cross terms'. In particular, if $|\al|=0$, $T[E]_3^{m,\beta}$ vanishes. It is easy to see that $T_c[E]_j^{m,\beta}$ can be treated  in the same way as $T[E]_j^{m,\beta}, j=1,2,3$, respectively. In the following, we only give the details of the treatments of $T[E]_j^{m,\beta}, j=1,2,3$ to avoid unnecessary repetition.

Clearly, 
\begin{align*}
    \big|T[E]_1^{m,\beta}\big|\le \ell \|E(t)\|_{L^\infty_x}\big\|Wg^{(m+\beta)}\big\|^2_{L^2_{x,v}(\ell-2|\al|)}.
\end{align*}
Then
\begin{align}\label{est-TE1}
    \sum_{m\in\N^6}\sum_{|\beta|\le N}\kappa^{2|\beta|}a_{m,s}^2(t)\big|T[E]_1^{m,\beta}\big|\le \ell \|E(t)\|_{L^\infty_x}\|Wg\|^2_{\mathcal{G}^{\lm,N}_{s,\ell-2|\al|}}.
\end{align}

\noindent{$\diamond$ \underline{Treatments of $T[E]_3^{m,\beta}$}.} By the low-high and high-low decomposition, we write 
\begin{align*}
    T[E]_3^{m,\beta}=&-\sum_{\substack{n\le m,\beta'\le\beta\\|\beta'|\le|\beta|/2}}C_{\beta}^{\beta'}C_{m}^n\left\la   Z^{n+\beta'}W E\cdot\nb_vg^{(m-n+\beta-\beta')},\la v\ra^{2\ell-4|\al|}Wg^{(m+\beta)}\right\ra_{x,v}\\
    &-\sum_{\substack{n\le m,\beta'\le\beta\\|\beta'|>|\beta|/2}}C_{\beta}^{\beta'}C_{m}^n\left\la   Z^{n+\beta'}W E\cdot\nb_vg^{(m-n+\beta-\beta')},\la v\ra^{2\ell-4|\al|}Wg^{(m+\beta)}\right\ra_{x,v}\\
    =&T^{\rm LH}[E]^{m,\beta}_3+T^{\rm HL}[E]^{m,\beta}_3.
\end{align*}
We first consider $T^{\rm LH}[E]^{m,\beta}_3$, which can be further split into two parts:
\begin{align*}
    T^{\rm LH}[E]^{m,\beta}_3=&-\sum_{\substack{0<n\le m,\beta'\le\beta\\ |\beta'|\le |\beta|/2}}C_\beta^{\beta'} C_{m}^n\left\la   Z^{n+\beta'}W E\cdot\nb_vg^{(m-n+\beta-\beta')},\la v\ra^{2\ell-4|\al|}Wg^{(m+\beta)}\right\ra_{x,v}\\
    &-\sum_{\beta'\le\beta,|\beta'|\le |\beta|/2}C_{\beta}^{\beta'}\left\la   Z^{\beta'}W E\cdot\nb_vg^{(m+\beta-\beta')},\la v\ra^{2\ell-4|\al|}Wg^{(m+\beta)}\right\ra_{x,v}\\
    =&T^{\rm LH}[E]^{m,\beta}_{3;(1)}+T^{\rm LH}[E]^{m,\beta}_{3;(2)}.
\end{align*}
To deal with $T^{\rm LH}[E]^{m,\beta}_{3;1)}$,
we need to use the {\rm CK} term. Before proceeding any further, noting that $\pr_{v_j}=Y_j-t\pr_{x_j}$, we can further split $T^{\rm LH}[E]^{m,\beta}_{3;(1)}$ into two parts:
\begin{align*}
    T^{\rm LH}[E]^{m,\beta}_{3;(1)}\les& \sum_{\substack{0<n\le m,\beta'\le\beta\\|\beta'|\le |\beta|/2}}C_{m}^n\big\| \la\nb_x\ra^2Z^{n+\beta'} WE^j\big\|_{L^2_x}\Big(t\big\|\pr_{x_j} g^{(m-n+\beta-\beta')}\big\|_{L^2_{x,v}(\ell-2|\al|)}
    \\ & +\big\| Y_jg^{(m-n+\beta-\beta')}\big\|_{L^2_{x,v}(\ell-2|\al|)}\Big)\big\|W g^{(m+\beta)} \big\|_{L^2_{x,v}(\ell-2|\al|)}\\
    =:& T^{\rm LH}[E]^{m,\beta}_{3;(1),t\nb_x}+T^{\rm LH}[E]^{m,\beta}_{3;(1),Y}.
\end{align*}
$T^{\rm LH}[E]^{m,\beta}_{3;(1),t\nb_x}$ and $T^{\rm LH}[E]^{m,\beta}_{3;(1),Y}$ can be treated exactly in the same way, we only give the details for the former one which has extra $t$ growth. Indeed, by Remark  \ref{rem-summable} and Lemma \ref{lem-convolution}, we have
\begin{align}\label{est-TLH-E31}
    \nn&\sum_{m\in\N^6}\sum_{|\beta|\le N}\kappa^{2|\beta|} a_{m,s}^2(t)\big|T^{\rm LH}[E]^{m,\beta}_{3;(1),t\nb_x}\big|\\
    =\nn&-\fr{t}{2\dot{\lm}(t)}\sum_{\substack{|\beta|\le N}} \sum_{\substack{ \beta'\le\beta\\ |\beta'|\le|\beta|/2}}\kappa^{2|\beta|}\sum_{\substack{m\in\N^6}}\sum_{\substack{ 0<n\le m}}\fr{C_m^nC_{|m|}^m }{C_{|n|}^nC_{|m-n+\tl{e}_j|}^{m-n+\tl{e}_j}}\fr{\Gamma_s(|n|)\Gamma_s(|m-n+\tl{e}_j|)}{\Gamma_s(|m|)\sqrt{|m-n+\tl{e}_j||m|}}\\
    \nn&\times \Big(\kappa^{|\beta'|}a_{n,s}(t)\left\|\la\nb_x\ra^2Z^{n+\beta'}WE^j \right\|_{L^2_x}\Big)\\
    \nn&\times \Big(\sqrt{-2(|m-n+\tl{e}_j|)\dot{\lm}(t)/\lm(t)}\kappa^{|\beta-\beta'|}a_{m-n+\tl{e}_j}(t)\big\| g^{(m-n+\tl{e}_j+\beta-\beta')} \big\|_{L^2_{x,v}(\ell-2|\al|)}\Big)\\
    \nn&\times \Big(\sqrt{-2|m|\dot{\lm}(t)/\lm(t)}\kappa^{|\beta|}a_{m,s}(t)\big\|W g^{(m+\beta)}\big\|_{L^2_{x,v}(\ell-2|\al|)}\Big)\\
    \les&\la t\ra^{a+2}\big\|\la\nb_x\ra^2WE\big\|_{\mathcal{G}^{\lm,\fr{N}{2}}_{s}}\mathfrak{CK}_{\ell}[g]^{\fr12}\mathfrak{CK}_{\ell-2|\al|}[Wg]^{\fr12}.
\end{align}

For the treatment of $T^{\rm LH}[E]^{m,\beta}_{3;(2)}$, note first that $T^{\rm LH}[E]^{m,\beta}_{3;(2)}$ does not vanish only when $W=\pr_x^\al$ or $(1+ t)^{-|\al|}Y^\al$ with $|\al|=1$. The extra $v$ derivative on $g^{(m+\beta-\beta')}$ leads us to introduce $(1+ t)^{-2}\big\| Yg\|^2_{\mathcal{G}^{\lm,N}_{s,\ell-2}}$ in the energy functional $\mathcal{E}_{\ell}(g)$, see \eqref{eq:E_l^n} and \eqref{energy-l}. Indeed, using again the fact $\pr_{v_j}=Y_j-t\pr_{x_j}$, one deduces that
\begin{align}\label{est-TLH-E32}
    \nn&\sum_{m\in\N^6}\sum_{|\beta|\le N}\kappa^{2|\beta|} a_{m,s}^2(t)\big|T^{\rm LH}[E]^{m,\beta}_{3;(2)}\big|\\
    \les\nn&\sum_{\substack{|\beta|\le N}}\sum_{\substack{ \beta'\le\beta\\ |\beta'|\le|\beta|/2}}\kappa^{|\beta'|}\big\|\la\nb_x\ra^2Z^{\beta'}W E \big\|_{L^2_x}\sum_{\substack{m\in\N^6}}\left(\kappa^{|\beta|}a_{m,s}(t)\big\| W g^{(m+\beta)}\big\|_{L^2_{x,v}(\ell-2)}\right)\\
    \nn&\times \kappa^{|\beta-\beta'|}a_{m,s}(t)\Big(\big\| Yg^{(m+\beta-\beta')} \big\|_{L^2_{x,v}(\ell-2)}+t\big\| \nb_xg^{(m+\beta-\beta')} \big\|_{L^2_{x,v}(\ell-2)}\Big)\\
    \les& \la t\ra \sum_{|\beta|\le \fr{N}{2}}\kappa^{|\beta|}\big\|\la\nb_x\ra^2Z^{\beta}\nb_x E \big\|_{L^2_x}\Big(\big\|\la t\ra^{-1}Yg\big\|_{\mathcal{G}^{\lm,N}_{s,\ell-2}}+\big\|\nb_xg\big\|_{\mathcal{G}^{\lm,N}_{s,\ell-2}}\Big)\|Wg\|_{\mathcal{G}^{\lm,N}_{s,\ell-2}}.
\end{align}

Now we turn to estimate $T^{\rm HL}[E]_3^{m,\beta}$, which will be divided into three parts on the Fourier side:
\begin{align*}
    T^{\rm HL}[E]_3^{m,\beta}
    =&-\sum_{k\in\Z^3_*}\sum_{\substack{\beta'\le\beta\\|\beta'|>|\beta|/2}}C_{\beta}^{\beta'}\sum_{n\le m}C_{m}^n \int_{\R^3}   \mathcal{F}_x[Z^{n+\beta'}W E]_k(t)\\
    &\cdot  \mathcal{F}_{x,v}\big[\la v\ra^{\ell-2|\al|}\nb_vg^{(m-n+\beta-\beta')}\big]_{0}(t,\eta)\overline{\mathcal{F}_{x,v}\big[\la v\ra^{\ell-2|\al|}Wg^{(m+\beta)}\big]_k(t,\eta)} d\eta\\
    &-\sum_{\substack{\beta'\le\beta\\|\beta'|>|\beta|/2}}C_{\beta}^{\beta'}\sum_{n\le m}C_{m}^n \sum_{l\in\Z^3_*}\int_{\R^3}   \mathcal{F}_x[Z^{n+\beta'}W E]_l(t)\\
    &\cdot  \mathcal{F}_{x,v}\big[\la v\ra^{\ell-2|\al|}\nb_vg^{(m-n+\beta-\beta')}\big]_{-l}(t,\eta)\overline{\mathcal{F}_{x,v}\big[\la v\ra^{\ell-2|\al|}Wg^{(m+\beta)}\big]_0(t,\eta)} d\eta\\
    &-\sum_{k\in\Z^3_*}\sum_{\substack{\beta'\le\beta\\|\beta'|>|\beta|/2}}C_{\beta}^{\beta'}\sum_{n\le m}C_{m}^n \sum_{l\in\Z^3_*,l\ne k}\int_{\R^3}   \mathcal{F}_x[Z^{n+\beta'}W E]_l(t)\cdot \\
    & \mathcal{F}_{x,v}\big[\la v\ra^{\ell-2|\al|}\nb_vg^{(m-n+\beta-\beta')}\big]_{k-l}(t,\eta)\overline{\mathcal{F}_{x,v}\big[\la v\ra^{\ell-2|\al|}Wg^{(m+\beta)}\big]_k(t,\eta)} d\eta\\
    =&T^{\rm HL}[E]_{3,(1)}^{m,\beta}+T^{\rm HL}[E]_{3,(2)}^{m,\beta}+T^{\rm HL}[E]_{3,(3)}^{m,\beta}.
\end{align*}
The treatment of $T^{\rm HL}[E]_{3,(1)}^{m,\beta}$ is straightforward, since now $\nb_v$ is a good derivative on $g_0$ which is at low frequency. More precisely, combining the inequality \eqref{up-kkt} with Lemmas \ref{lem-summable} and \ref{lem-convolution}  and Corollary  \ref{coro-CK}, similar to \eqref{est-Lmuiota11HL}, we have 
\begin{align}\label{est-THL-E31}
    \nn&\sum_{m\in\N^6,|\beta|\le N}\kappa^{2|\beta|}a^2_{m,s}(t)\big|T^{\rm HL}[E]_{3,(1)}^{m,\beta}\big|\\
    \nn\les&\big\|\la t\ra^{\fr{1+a-s}{2}}|\nb_x|^{-\fr{s}{2}}W E(t)\big\|_{\mathcal{G}^{\lm,N}_s}\big\|\la\nb_v\ra\nabla_v g_0(t)\big\|_{\mathcal{G}^{\lm,\fr{N}{2}}_{s,\ell-2|\al|}}\big\|\la t\ra^{-\fr{1+a}{2}}|Z|^{\fr{s}{2}}Wg(t)\big\|_{\mathcal{G}^{\lm,N}_{s,\ell-2|\al|}}\\
    \les&\big\|\la t\ra^{\fr{1+a-s}{2}}\rho(t)\big\|_{\mathcal{G}^{\lm,N}_s}\|g_0(t)\big\|_{\mathcal{G}^{\lm,N}_{s,\ell}}\Big(\mathfrak{CK}_{\ell-2|\al|}[Wg(t)]^{\fr12}+\la t\ra^{-\fr{1+a}{2}}\big\|Wg(t)\big\|_{\mathcal{G}^{\lm,N}_{s,\ell-2|\al|}}\Big).
\end{align}
For $T^{\rm HL}[E]_{3,(2)}^{m,\beta}$, we first need to swap the positions of $\la v\ra^{\ell-2|\al|}$ and $\nb_v$ in front of  $g^{(m-n+\beta-\beta')}$, and thus split it into two parts:
\begin{align*}
    T^{\rm HL}[E]_{3,(2)}^{m,\beta}
    =&-\sum_{\substack{\beta'\le\beta\\|\beta'|>|\beta|/2}}C_{\beta}^{\beta'}\sum_{n\le m}C_{m}^n \sum_{k\in\Z^3_*}\int_{\R^3}   \mathcal{F}_x[Z^{n+\beta'}W E]_k(t)\\
    &\times\cdot i\eta \mathcal{F}_{x,v}\big[\la v\ra^{\ell-2|\al|}g^{(m-n+\beta-\beta')}\big]_{-k}(t,\eta)\overline{\mathcal{F}_{x,v}\big[\la v\ra^{\ell-2|\al|}Wg^{(m+\beta)}\big]_0(t,\eta)} d\eta\\
    &+\sum_{\substack{\beta'\le\beta\\|\beta'|>|\beta|/2}}C_{\beta}^{\beta'}\sum_{n\le m}C_{m}^n \int_{\T\times\R^3}   Z^{n+\beta'}W E(t)\cdot \nb_v\la v\ra^{\ell-2|\al|} g^{(m-n+\beta-\beta')}(t,x,v)\\
    &\times\la v\ra^{\ell-2|\al|}Wg^{(m+\beta)}_0(t,v)dx dv=T^{\rm HL}[E]_{3,(2);d}^{m,\beta}+T^{\rm HL}[E]_{3,(2);e}^{m,\beta},
\end{align*}
here `d' stands for `difficult' and `e' stands for `easy'.
Noting that for $|k|\ne 0$, there hold
\begin{align*}
    |\eta|\les&\la t\ra |-k,\eta-kt|\les \fr{1}{\la t\ra^{\fr{1+a}{2}}}\big(|k|^{-\fr{s}{2}}\la t\ra^{\fr{3+a-s}{2}}\big)|k|^{\fr{s}{2}}|-k,\eta-kt|,
\end{align*}
and
\begin{align*}
    &\big\||\nb_x|^{\fr{s}{2}}|Z|\big(\la v\ra^{\ell-2|\al|} g^{(m-n+\beta-\beta')}\big) \big\|_{L^2_{x,v}}\\
    \les& \big\||\nb_x|^{\fr{s}{2}}|Z| g^{(m-n+\beta-\beta')} \big\|_{L^2_{x,v}(\ell-2|\al|)}+\big\|\nb_v\la v\ra^{\ell-2|\al|} |\nb_x|^{\fr{s}{2}}g^{(m-n+\beta-\beta')} \big\|_{L^2_{x,v}}\\
    \les& \big\||\nb_x|^{\fr{s}{2}}|Z| g^{(m-n+\beta-\beta')} \big\|_{L^2_{x,v}(\ell)}.
\end{align*}
 Then by Lemmas \ref{lem-summable} and \ref{lem-convolution}, we have
\begin{align}\label{est-THL-E32d}
    \nn&\sum_{m\in\N^6,|\beta|\le N}\kappa^{2|\beta|}a^2_{m,s}(t)T^{\rm HL}[E]_{3,(2);d}^{m,\beta}\\
    \les\nn& \sum_{|\beta|\le N}\sum_{\substack{\beta'\le\beta\\|\beta'|>|\beta|/2}}\sum_{m\in\N^6}\sum_{n\le m}b_{m,n,s} \\
    \nn&\times \kappa^{|\beta'|}a_{n,s}(t)    \big\|\la t\ra^{\fr{3+a-s}{2}}|\nb_x|^{-\fr{s}{2}}Z^{n+\beta'}W E\big\|_{L^2_x}\\
    \nn&\times \kappa^{|\beta-\beta'|}a_{m-n,s}(t)\big\| |\nb_x|^{\fr{s}{2}}|Z|\big(\la v\ra^{\ell-2|\al|} g^{(m-n+\beta-\beta')}\big)\big\|_{L^2_{x,v}}\\
    \nn&\times \kappa^{|\beta|}a_{m,s}(t)  \left\|\fr{1}{\la t\ra^{\fr{1+a}{2}}}  \big(\la v\ra^{\ell-2|\al|}Wg^{(m+\beta)}_0\big)\right\|_{L^2_v}\\
    \les\nn&\big\|\la t\ra^{\fr{3+a-s}{2}}|\nb_x|^{-\fr{s}{2}}WE\big\|_{\mathcal{G}^{\lm,N}_{s}}\||\nb_x|^{\fr{s}{2}}|Z|g\|_{\mathcal{G}^{\lm,\fr{N}{2}}_{s,\ell}} \left\|\fr{1}{\la t\ra^{\fr{1+a}{2}}}  Wg_0\right\|_{\mathcal{G}^{\lm,N}_{s,\ell-2|\al|}}\\
    \les&\big\|\la t\ra^{\fr{3+a-s}{2}}\rho\big\|_{\mathcal{G}^{\lm,N}_{s}}\|g\|_{\mathcal{G}^{\lm,N}_{s,\ell}} \left\|\fr{1}{\la t\ra^{\fr{1+a}{2}}}  Wg_0\right\|_{\mathcal{G}^{\lm,N}_{s,\ell-2|\al|}}.
\end{align}
The treatment of $T^{\rm HL}[E]_{3,(2);e}^{m,\beta}$ is similar and easier due to the absence of factor $\eta$:
\begin{align}\label{est-THL-E32e}
    \sum_{m\in\N^6,|\beta|\le N}\kappa^{2|\beta|}a^2_{m,s}(t)T^{\rm HL}[E]_{3,(2);e}^{m,\beta}
    \les\big\|\la t\ra^{\fr{1+a}{2}}\rho\big\|_{\mathcal{G}^{\lm,N}_{s}}\|g\|_{\mathcal{G}^{\lm,N}_{s,\ell}} \left\|\fr{1}{\la t\ra^{\fr{1+a}{2}}}  Wg_0\right\|_{\mathcal{G}^{\lm,N}_{s,\ell-2|\al|}}.
\end{align}

As for $T^{\rm HL}[E]_{3,(3)}^{m,\beta}$, like $T^{\rm HL}[E]_{3,(2)}^{m,\beta}$, we  split it into two parts as well:
\begin{align*}
    T^{\rm HL}[E]_{3,(3)}^{m,\beta}=&-\sum_{\substack{\beta'\le\beta\\|\beta'|>|\beta|/2}}C_{\beta}^{\beta'}\sum_{n\le m}C_{m}^n \sum_{k\in\Z^3_*}\sum_{l\in\Z^3_*,l\ne k}\int_{\R^3}   \mathcal{F}_x[Z^{n+\beta'}W E]_l(t)\cdot i\eta\\
    &\times \mathcal{F}_{x,v}\big[\la v\ra^{\ell-2|\al|}g^{(m-n+\beta-\beta')}\big]_{k-l}(t,\eta)\\
    &\times\overline{\mathcal{F}_{x,v}\big[\la v\ra^{\ell-2|\al|}Wg^{(m+\beta)}\big]_k(t,\eta)} d\eta\\
    &+\sum_{\substack{\beta'\le\beta\\|\beta'|>|\beta|/2}}C_{\beta}^{\beta'}\sum_{n\le m}C_{m}^n \int_{\T^3\times\R^3}   Z^{n+\beta'}W E(t)\cdot \nb_v\la v\ra^{\ell-2|\al|} g^{(m-n+\beta-\beta')}_{\ne}(t,x,v)\\
    &\times\la v\ra^{\ell-2|\al|}Wg^{(m+\beta)}(t,x,v)dxdv=T^{\rm HL}[E]_{3,(3);d}^{m,\beta}+T^{\rm HL}[E]_{3,(3);e}^{m,\beta}.
\end{align*}
If $t\ge10|k-l,\eta+(k-l)t|$, nothing that $l\ne k$ and $|l|>0$,  we now have $t\ge10$ and  $\fr{10}{11}|\eta+kt|\le|lt|\le \fr{10}{9}|\eta+kt|$, and hence 
\begin{align*}
    |\eta|\les& \la t\ra |k-l,\eta+(k-l)t|=\fr{\la t\ra^{\fr{1+a}{2}}}{|lt|^{\fr{s}{2}}}\la t\ra |k-l,\eta+(k-l)t|\fr{|lt|^{\fr{s}{2}}}{\la t\ra^{\fr{1+a}{2}}}\\
    \les&\big(|l|^{-\fr{s}{2}}\la t\ra^{\fr{3+a-s}{2}}\big)|k-l,\eta+(k-l)t|\fr{|\eta+kt|^{\fr{s}{2}}}{\la t\ra^{\fr{1+a}{2}}}.
\end{align*}
If $t\le10|k-l,\eta+(k-l)t|$, for $k\ne l$ there holds,
\begin{align*}
    |\eta|\les& \la t\ra |k-l,\eta+(k-l)t|\les |k-l,\eta+(k-l)t|^2\\
    \les&\big(|l|^{-\fr{s}{2}}\la t\ra^{\fr{1+a}{2}}\big) \big(|k-l|^{\fr{s}{2}}|k-l,\eta+(k-l)t|^2\big)\fr{|k|^{\fr{s}{2}}}{\la t\ra^{\fr{1+a}{2}}}.
\end{align*}
It follows these two inequalities and Corollary \ref{coro-CK} that
\begin{align}\label{est-THL-E33d}
    \nn&\sum_{m\in\N^6,|\beta|\le N}\kappa^{2|\beta|}a^2_{m,s}(t)T^{\rm HL}[E]_{3,(3);d}^{m,\beta}\\
    \les\nn&\sum_{|\beta|\le N}\sum_{\substack{\beta'\le\beta\\|\beta'|>|\beta|/2}}\sum_{m\in\N^6}\sum_{n\le m}b_{m,n,s} \Big(\kappa^{|\beta'|}a_{n,s}(t) \big\| \la t\ra^{\fr{3+a-s}{2}}   |\nb_x|^{-\fr{s}{2}}Z^{n+\beta'}W E(t)\big\|_{L^2_x}\Big)\\
    \nn&\times \kappa^{|\beta-\beta'|}a_{m-n,s}(t)\big\||Z|^2\big(\la v\ra^{\ell-2|\al|} |\nb_x|^{2+\fr{s}{2}} g^{(m-n+\beta-\beta')}\big)\big\|_{L^2_{x,v}}\\
    \nn&\times\kappa^{|\beta|}a_{m,s}(t)\left\|\fr{|Z|^{\fr{s}{2}}}{\la t\ra^{\fr{1+a}{2}}}  \big(\la v\ra^{\ell-2|\al|}Wg^{(m+\beta)}\big)\right\|_{L^2_{x,v}}\\
    \les\nn&\left\|\la t\ra^{\fr{3+a-s}{2}}   |\nb_x|^{-\fr{s}{2}}\rho\right\|_{\mathcal{G}^{\lm,N}_{s}}\left\||Z|^2 |\nb_x|^{2+\fr{s}{2}}g \right\|_{\mathcal{G}^{\lm,\fr{N}{2}}_{s,\ell}}\\
    \nn&\times \left\|\kappa^{|\beta|}a_{m,s}(t)\left\|\fr{|Z|^{\fr{s}{2}}}{\la t\ra^{\fr{1+a}{2}}}  \left(\la v\ra^{\ell-2|\al|}Wg^{(m+\beta)}(t)\right)\right\|_{L^2_{x,v}}\right\|_{L^2_{m,\beta}}\\
    \les&\left\|\la t\ra^{\fr{3+a-s}{2}}   \rho\right\|_{\mathcal{G}^{\lm,N}_{s}}\left\| g(t)\right\|_{\mathcal{G}^{\lm, N}_{s,\ell}} \Big( \mathfrak{CK}_{\ell-2|\al|}[Wg(t)]^{\fr12}+\fr{1}{\la t\ra^{\fr{1+a}{2}}}\left\| W g(t)\right\|_{\mathcal{G}^{\lm, N}_{s,\ell-2|\al|}}\Big).
\end{align}
To bound $T^{\rm HL}[E]_{3,(3);e}^{m,\beta}$, note that
\begin{align*}
1\les \big(|l|^{-\fr{s}{2}}\la t\ra^{\fr{1+a}{2}}\big)|k-l|^{\fr{s}{2}}\fr{|k|^{\fr{s}{2}}}{\la t\ra^{\fr{1+a}{2}}}.
\end{align*}
Then similar to the treatment of  $T^{\rm HL}[E]_{3,(3);d}^{m,\beta}$ above, we have
\begin{align}\label{est-THL-E33e}
    \nn&\sum_{m\in\N^6,|\beta|\le N}\kappa^{2|\beta|}a^2_{m,s}(t)T^{\rm HL}[E]_{3,(3);e}^{m,\beta}\\
    \les\nn&\left\|\la t\ra^{\fr{1+a}{2}}   |\nb_x|^{-\fr{s}{2}}\rho\right\|_{\mathcal{G}^{\lm,N}_{s}}\left\| |\nb_x|^{2+\fr{s}{2}}g \right\|_{\mathcal{G}^{\lm,\fr{N}{2}}_{s,\ell}} \\
    \nn&\times\left\|\kappa^{|\beta|}a_{m,s}(t)\left\|\fr{|\nb_x|^{\fr{s}{2}}}{\la t\ra^{\fr{1+a}{2}}}  \big(\la v\ra^{\ell-2|\al|}Wg^{(m+\beta)}\big)\right\|_{L^2_{x,v}}\right\|_{L^2_{m,\beta}}\\
    \les&\left\|\la t\ra^{\fr{3+a-s}{2}}   \rho\right\|_{\mathcal{G}^{\lm,N}_{s}}\left\| g(t)\right\|_{\mathcal{G}^{\lm, N}_{s,\ell}} \Big( \mathfrak{CK}_{\ell-2|\al|}[Wg(t)]^{\fr12}+\fr{1}{\la t\ra^{\fr{1+a}{2}}}\left\| W g(t)\right\|_{\mathcal{G}^{\lm, N}_{s,\ell-2|\al|}}\Big).
\end{align}

\noindent{$\diamond$ \underline{Treatments of $T[E]_2^{m,\beta}$}.} To begin with, like \eqref{commutator-L},  we split the commutator $[Z^{m+\beta}, E\cdot\nb_v]$   into two parts:
\begin{align*}
    T[E]_2^{m,\beta}=\nn&-\sum_{\beta'\le\beta}C_{\beta}^{\beta'}\sum_{0<n\le m}C_m^n\left\la Z^{n+\beta'}E\cdot \nb_vW g^{(m-n+\beta-\beta')}, \la v\ra^{2\ell-4|\al|}W g^{(m+\beta)} \right\ra_{x,v}\\
    \nn&-\sum_{0<\beta'\le\beta}C_{\beta}^{\beta'}\left\la Z^{\beta'}E\cdot \nb_vW g^{(m+\beta-\beta')}, \la v\ra^{2\ell-4|\al|}W g^{(m+\beta)} \right\ra_{x,v}\\
    =&T[E]_{2;(1)}^{m,\beta}+T[E]_{2;(2)}^{m,\beta}.
\end{align*}
As the treatment of $T[E]^{m,\beta}_3$, we need to use the low-high and high-low decomposition:
\begin{align*}
    T[E]_{2;(1)}^{m,\beta}=&-\Big(\sum_{\substack{0<n\le m, \beta'\le\beta\\|\beta'|\le|\beta|/2}}+\sum_{\substack{0<n\le m, \beta'\le\beta\\|\beta'|>|\beta|/2}}\Big)C_{\beta}^{\beta'}C_m^n\\
    &\times\left\la Z^{n+\beta'}E\cdot \nb_vW g^{(m-n+\beta-\beta')}, \la v\ra^{2\ell-4|\al|}W g^{(m+\beta)} \right\ra_{x,v}\\
    =&T^{\rm LH}[E]_{2;(1)}^{m,\beta}+T^{\rm HL}[E]_{2;(1)}^{m,\beta},
\end{align*}
and
\begin{align*}
    T[E]_{2;(2)}^{m,\beta}=&-\Big(\sum_{\substack{0<\beta'\le\beta\\|\beta'|\le|\beta|/2}}+\sum_{\substack{0<\beta'\le\beta\\|\beta'|>|\beta|/2}}\Big)C_{\beta}^{\beta'}\left\la Z^{\beta'}E\cdot \nb_vW g^{(m+\beta-\beta')}, \la v\ra^{2\ell-4|\al|}W g^{(m+\beta)} \right\ra_{x,v}\\
    =&T^{\rm LH}[E]_{2;(2)}^{m,\beta}+T^{\rm HL}[E]_{2;(2)}^{m,\beta}.
\end{align*}

We first treat $T^{\rm LH}[E]^{m,\beta}_{2;(1)}$ and $T^{\rm LH}[E]^{m,\beta}_{2;(2)}$.
For $T^{\rm LH}[E]^{m,\beta}_{2;(1)}$, one can treat it  in the same way as that of $T^{\rm LH}[E]^{m,\beta}_{3;(1)}$. Indeed, similar to \eqref{est-TLH-E31}, we have
\begin{align}\label{est-TE121LH}
    \sum_{m\in\N^6}\sum_{|\beta|\le N}\kappa^{2|\beta|} a_{m,s}^2(t)\big|T^{\rm LH}[E]^{m,\beta}_{2;(1)}\big|
    \les\la t\ra^{a+2}\big\|\la\nb_x\ra^2E\big\|_{\mathcal{G}^{\lm,\fr{N}{2}}_{s}}\mathfrak{CK}_{\ell-2|\al|}[Wg].
\end{align}
As for  $T^{\rm LH}[E]_{2;(2)}^{m,\beta}$, since $|\beta'|>0$, there is no need to involve the vector field $(1+t)^{-1}Y$ as the treatment of $T^{\rm LH}[E]_{3;(2)}^{m,\beta}$ in \eqref{est-TLH-E32}. In fact, noting that $\nb_v=Y-t\nb_x$, then there holds
\begin{align*}
    \sum_{\substack{|\beta-\beta'|\le N-1}}\big\|\nb_vW g^{(m+\beta-\beta')}\big\|_{L^2_{x,v}(\ell-2|\al|)}\les\la t\ra \sum_{\substack{|\beta|\le N}}\big\|W g^{(m+\beta)}\big\|_{L^2_{x,v}(\ell-2|\al|)}.
\end{align*}
Thus,
\begin{align}\label{est-TLH-E22}
\nn&\sum_{m\in\N^6}\sum_{|\beta|\le N}\kappa^{2|\beta|} a_{m,s}^2(t)\big|T^{\rm LH}[E]_{2;(2)}^{m,\beta}\big|\\
\nn\les&\sum_{|\beta|\le N}\sum_{m\in\N^6}\sum_{0<|\beta'|\le N/2}\kappa^{|\beta'|} \big\|\la\nb_x\ra^2Z^{\beta'}E\big\|_{L^2_x}\\
\nn&\times\sum_{\substack{|\beta-\beta'|\le N-1}}\kappa^{|\beta-\beta'|}a_{m,s}(t)\big\|\nb_vW g^{(m+\beta-\beta')}\big\|_{L^2_{x,v}(\ell-2|\al|)} \\
\nn&\times \Big(\kappa^{|\beta|} a_{m,s}(t)\big\|W g^{(m+\beta)}\big\|_{L^2_{x,v}(\ell-2|\al|)}\Big)\\
\les&\la t\ra \sum_{|\beta|\le\fr{N} {2}}\kappa^{|\beta|}\big\|Z^{\beta}\la\nb_x\ra^2E\big\|_{L^2_x} \|Wg\|^2_{\mathcal{G}^{\lm,N}_{s,\ell-2|\al|}}.
\end{align}

Next we turn to treat  $T^{\rm HL}[E]_{2;(1)}^{m,\beta}$ and $T^{\rm HL}[E]_{2;(2)}^{m,\beta}$. Given that both  $T^{\rm HL}[E]_{2;(1)}^{m,\beta}$ and $T^{\rm HL}[E]_{2;(2)}^{m,\beta}$ share the same structure as $T^{\rm HL}[E]_{3}^{m,\beta}$, one can follow the estimates of $T^{\rm HL}[E]_{3}^{m,\beta}$ \eqref{est-THL-E31}--\eqref{est-THL-E33e} line by line to get
\begin{align}\label{est-THL-E2}
\nn&\sum_{m\in\N^6}\sum_{|\beta|\le N}\kappa^{2|\beta|} a_{m,s}^2(t)\Big(\big|T^{\rm HL}[E]_{2;(1)}^{m,\beta}\big|+\big|T^{\rm HL}[E]_{2;(2)}^{m,\beta}\big|\Big)\\
\les&\left\|\la t\ra^{\fr{3+a-s}{2}}   \rho\right\|_{\mathcal{G}^{\lm,N}_{s}}\left\|W g(t)\right\|_{\mathcal{G}^{\lm, N}_{s,\ell-2|\al|}} \Big( \mathfrak{CK}_{\ell-2|\al|}[Wg(t)]^{\fr12}+\fr{1}{\la t\ra^{\fr{1+a}{2}}}\left\| W g(t)\right\|_{\mathcal{G}^{\lm, N}_{s,\ell-2|\al|}}\Big).
\end{align}

Now collecting the estimates \eqref{est-TE1}--\eqref{est-THL-E2}, and noting that the estimates for $T[E]_j^{m,\beta}$ apply to $T_c[E]_j^{m,\beta}, j=1,2,3$, respectively,  we conclude that
\begin{align}\label{est-NE01}
    N[E]_{0;1)}+N[E]_{0;2)}\nn\les& \la t\ra^{a+2}\|\la \nb_x\ra^2\rho\|_{\mathcal{G}^{\lm,\fr{N}{2}}_{s}}\mathcal{CK}_{\ell}(g(t))+\la t\ra \|\la \nb_x\ra^2\rho\|_{\mathcal{G}^{\lm,\fr{N}{2}}_{s}}\mathcal{E}_{\ell}(g(t))\\
    &+\left\|\la t\ra^{\fr{3+a-s}{2}}   \rho\right\|_{\mathcal{G}^{\lm,N}_{s}}\mathcal{E}_{\ell}(g(t))^{\fr12} \Big( \mathcal{CK}_{\ell}(g(t))^{\fr12}+\fr{1}{\la t\ra^{\fr{1+a}{2}}}\mathcal{E}_{\ell}(g(t))^{\fr12}\Big).
\end{align}

{\bf Case 2: $\iota=1,2,\cdots,[\ell/3]+1$.} For $W\in\big\{\pr_x^\al, (1+ t)^{-|\al|}Y^\al:  0\le|\al|\le1\big\}\cup \big\{ \pr_v^\al: 1\le|\al|\le2\big\}$ with $\al\in\N^3$, we write 
\begin{align*}
    &-\left\la W Z^{m+\beta}\left(E\cdot\nb_vg\right)_{\iota\ne},\la v\ra^{2\ell-4\iota-4|\al|}Wg_{\iota\ne}^{(m+\beta)}\right\ra_{x,v}\\
    =&-w_{\iota}(t)\left\la W Z^{m+\beta}\left(E\cdot\nb_vg_{0}\right),Wg_{\iota\ne}^{(m+\beta)}\right\ra_{\ell-2\iota-2|\al|}\\
    &-\left\la W Z^{m+\beta}\left(E\cdot\nb_vg_{\iota\ne}\right),Wg_{\iota\ne}^{(m+\beta)}\right\ra_{\ell-2\iota-2|\al|},
\end{align*}
and
\begin{align*}
    &-\left\la \pr_{x_j}Z^{m+\beta}(E\cdot\nb_vg)_{\iota\ne}, \pr_{v_j}g_{\iota\ne}^{(m+\beta)} \right\ra_{\ell-2-2\iota}-\left\la \pr_{v_j}Z^{m+\beta}(E\cdot\nb_vg)_{\iota\ne}, \pr_{x_j}g_{\iota\ne}^{(m+\beta)} \right\ra_{\ell-2-2\iota}\\
    =&-w_{\iota}(t)\left\la \pr_{x_j}Z^{m+\beta}(E\cdot\nb_vg_0), \pr_{v_j}g_{\iota\ne}^{(m+\beta)} \right\ra_{\ell-2-2\iota}-w_{\iota}(t)\left\la \pr_{v_j}Z^{m+\beta}(E\cdot\nb_vg_0), \pr_{x_j}g_{\iota\ne}^{(m+\beta)} \right\ra_{\ell-2-2\iota}\\
    &-\left\la \pr_{x_j}Z^{m+\beta}(E\cdot\nb_vg_{\iota\ne}), \pr_{v_j}g_{\iota\ne}^{(m+\beta)} \right\ra_{\ell-2-2\iota}-\left\la \pr_{v_j}Z^{m+\beta}(E\cdot\nb_vg_{\iota\ne}), \pr_{x_j}g_{\iota\ne}^{(m+\beta)} \right\ra_{\ell-2-2\iota}.
\end{align*}
Accordingly, $N[E]_{\iota;j)}$ can be split into two parts: $N[E]_{\iota;j)}=N[E]^0_{\iota;j)}+N[E]^{\ne}_{\iota;j)}, j=1,2$. Clearly, $N[E]^{\ne}_{\iota;j)}$ preserves the same  structure as that of $N[E]_{0;j)}, j=1,2$, all the estimates in {\bf Case 1} hold with $g$ replaced by $g_{\iota\ne}$. More precisely, similar to \eqref{est-NE01}, now we have
\begin{align}\label{est-NEiota1-ne}
    \nn&N[E]_{\iota;1)}^{\ne}+N[E]_{\iota;2)}^{\ne}\\
    \nn\les& \la t\ra^{a+2}\|\la \nb_x\ra^2\rho\|_{\mathcal{G}^{\lm,\fr{N}{2}}_{s}}\mathcal{CK}_{\ell-2\iota}(g_{\iota\ne}(t))+\la t\ra \|\la \nb_x\ra^2\rho\|_{\mathcal{G}^{\lm,\fr{N}{2}}_{s}}\mathcal{E}_{\ell-2\iota}(g_{\iota\ne}(t))\\
    &+\left\|\la t\ra^{\fr{3+a-s}{2}}   \rho\right\|_{\mathcal{G}^{\lm,N}_{s}}\mathcal{E}_{\ell-2\iota}(g_{\iota\ne}(t))^{\fr12} \Big( \mathcal{CK}_{\ell-2\iota}(g_{\iota\ne}(t))^{\fr12}+\fr{1}{\la t\ra^{\fr{1+a}{2}}}\mathcal{E}_{\ell-2\iota}(g_{\iota\ne}(t))^{\fr12}\Big).
\end{align}

$N[E]_{\iota;1)}^0$ and $N[E]_{\iota;2)}^0$ can be bounded in  the same way, we only give the details of the treatments of   $N[E]^0_{\iota;1)}$. To this end, we further split $-w_{\iota}(t)\left\la W Z^{m+\beta}\left(E\cdot\nb_vg_{0}\right),Wg_{\iota\ne}^{(m+\beta)}\right\ra_{\ell-2\iota-2|\al|}$ as follows:
\begin{align*}
    &-w_{\iota}(t)\left\la W Z^{m+\beta}\left(E\cdot\nb_vg_{0}\right),Wg_{\iota\ne}^{(m+\beta)}\right\ra_{\ell-2\iota-2|\al|}\\
    =&-\sum_{\beta'\le\beta,|\beta'|\le|\beta|/2}\sum_{n\le m}C_{\beta}^{\beta'}C_{m}^nw_{\iota}(t)\left\la  Z^{n+\beta'}E\cdot\nb_v Wg_{0}^{(m-n+\beta-\beta')},Wg_{\iota\ne}^{(m+\beta)}\right\ra_{\ell-2\iota-2|\al|}\\
    &-\sum_{\beta'\le\beta,|\beta'|\le|\beta|/2}\sum_{n\le m}C_{\beta}^{\beta'}C_{m}^nw_{\iota}(t)\left\la  Z^{n+\beta'}W E\cdot\nb_vg_{0}^{(m-n+\beta-\beta')},Wg_{\iota\ne}^{(m+\beta)}\right\ra_{\ell-2\iota-2|\al|}\\
    &-\sum_{\beta'\le\beta,|\beta'|>|\beta|/2}\sum_{n\le m}C_{\beta}^{\beta'}C_{m}^nw_{\iota}(t)\left\la Z^{n+\beta'}E\cdot\nb_v Wg_{0}^{(m-n+\beta-\beta')},Wg_{\iota\ne}^{(m+\beta)}\right\ra_{\ell-2\iota-2|\al|}\\
    &-\sum_{\beta'\le\beta,|\beta'|>|\beta|/2}\sum_{n\le m}C_{\beta}^{\beta'}C_{m}^nw_{\iota}(t)\left\la Z^{n+\beta'}W E\cdot\nb_vg_{0}^{(m-n+\beta-\beta')},Wg_{\iota\ne}^{(m+\beta)}\right\ra_{\ell-2\iota-2|\al|}\\
    =&T^{\rm LH}_{\iota;0}[E]^{m,\beta}_{1)}+T^{\rm LH}_{\iota;0}[E]^{m,\beta}_{2)}+T^{\rm HL}_{\iota;0}[E]^{m,\beta}_{1)}+T^{\rm HL}_{\iota;0}[E]^{m,\beta}_{2)}.
\end{align*}
Note that if $|\al|=0$, $W$ vanishes, $T^{\rm LH}_{\iota;0}[E]^{m,\beta}_{1)}$ and $T^{\rm LH}_{\iota;0}[E]^{m,\beta}_{2)}$  are identical and thus reduce to one single term. Meanwhile, $T^{\rm HL}_{\iota;0}[E]^{m,\beta}_{1)}$ and $T^{\rm HL}_{\iota;0}[E]^{m,\beta}_{2)}$ reduce to one single term as well.

For $T^{\rm LH}_{\iota;0}[E]^{m,\beta}_{1)}$, let us divide $T^{\rm LH}_{\iota;0}[E]^{m,\beta}_{1)}$ into two parts as before:
\begin{align*}
    &T^{\rm LH}_{\iota;0}[E]^{m,\beta}_{1)}\\
    =&\sum_{\substack{n\le m,\beta'\le\beta\\
    |\beta'|\le|\beta|/2}}C_{\beta}^{\beta'}C_{m}^nw_{\iota}(t)\left\la  Z^{n+\beta'}E\cdot\nb_v\la v\ra^{\ell-2\iota-2|\al|} Wg_{0}^{(m-n+\beta-\beta')},Wg_{\iota\ne}^{(m+\beta)}\la v\ra^{\ell-2\iota-2|\al|}\right\ra_{x,v}\\
    &-\sum_{\substack{n\le m,\beta'\le\beta\\
    |\beta'|\le|\beta|/2}}C_{\beta}^{\beta'}C_{m}^nw_{\iota}(t)\left\la  Z^{n+\beta'}E\cdot\nb_v\big( Wg_{0}^{(m-n+\beta-\beta')}\la v\ra^{\ell-2\iota-2|\al|}\big),Wg_{\iota\ne}^{(m+\beta)}\la v\ra^{\ell-2\iota-2|\al|}\right\ra_{x,v}\\
    =&T^{\rm LH}_{\iota;0}[E]^{m,\beta}_{1);e}+T^{\rm LH}_{\iota;0}[E]^{m,\beta}_{1);d}.
\end{align*}
The first term $T^{\rm LH}_{\iota;0}[E]^{m,\beta}_{1);e}$ can be treated  straightforwardly: 
\begin{align}\label{est-Eg0-LH-iota-e}
    \nn&\sum_{m\in\N^6,|\beta|\le N}\kappa^{2|\beta|}a_{m,\lm,s}^2(t)\big|T^{\rm LH}_{\iota;0}[E]^{m,\beta}_{1);e}\big|\\
    \les& w_{\iota}(t)\|E\|_{\mathcal{G}^{\lm,\fr{N}{2}}_s}\|Wg_0\|_{\mathcal{G}^{\lm,N}_{s,\ell-2|\al|}}\|Wg_{\iota\ne}\|_{\mathcal{G}^{\lm,N}_{s,\ell-2\iota-2|\al|}}.
\end{align}
For the second term $T^{\rm LH}_{\iota;0}[E]^{m,\beta}_{1);d}$, since now $g_0$ is at high frequency, we have to deal with the extra $v$ derivative carefully. 
Our strategy is to use the CK term of $Wg_{\iota\ne}$, and simultaneously interpolate between  the dissipation and CK for $g_0$. More precisely, noting that for $\iota\ge1$ and $s\ge\fr12$, there holds
\begin{align*}
    \nu^{\fr{\iota}{3}}|\eta|\le\nu^{\fr13} |\eta|^{1-\fr{s}{2}}\la\eta+kt\ra^{\fr{s}{2}}\la kt\ra^{\fr{s}{2}}=\big(\nu^{\fr12}|\eta|\big)^{\fr{1-s}{1-\fr{s}{2}}} \big(|\eta|^{\fr{s}{2}}\big)^{\fr{\fr{s}{2}}{1-\fr{s}{2}}}\la\eta+kt\ra^{\fr{s}{2}}\la kt\ra^{\fr{s}{2}}.
\end{align*}
Moreover, it is easy to see that
\begin{align*}
    &\Big\|\nb_v\big(Wg_0^{(m-n+\beta-\beta')}\la v\ra^{\ell-2\iota-2|\al|}\big)\Big\|_{L^2_v}\\
    \les&\Big\|\nb_vWg_0^{(m-n+\beta-\beta')}\la v\ra^{\ell-2\iota-2|\al|}\Big\|_{L^2_v}+\Big\|Wg_0^{(m-n+\beta-\beta')}\la v\ra^{\ell-2\iota-2|\al|}\Big\|_{L^2_v},
\end{align*}
Then using Corollary \ref{coro-CK}, the fact $w_{\iota}^2(t)\les \nu^{\fr{\iota}{3}}\la t\ra^{\iota}$, and \eqref{coercive}, we have
\begin{align}\label{est-Eg0-LH-iota}
    \nn&\sum_{m\in\N^6,|\beta|\le N}\kappa^{2|\beta|}a_{m,\lm,s}^2(t)\big|T^{\rm LH}_{\iota;0}[E]^{m,\beta}_{1);d}\big|\\
    \nn\les&\la t\ra^{\iota}\sum_{m\in\N^6,|\beta|\le N}\kappa^{2|\beta|}\sum_{\substack{n\le m,\beta'\le\beta\\|\beta'\le|\beta|/2}}b_{m,n,s}\sum_{k\in\Z^3_*}a_{n,\lm,s}(t)\big|\la kt\ra^{\fr{s}{2}}\mathcal{F}_x\big[Z^{n+\beta'}E\big]_k\big|\\
    \nn&\times  a_{m-n,\lm,s}(t)\Big(\nu^{\fr12}\Big\|\nb_v\big(Wg_0^{(m-n+\beta-\beta')}\la v\ra^{\ell-2\iota-2|\al|}\big)\Big\|_{L^2_v}\Big)^{\fr{1-s}{1-\fr{s}{2}}}\\
    \nn&\quad\quad\quad\quad \ \  \times\Big(\Big\||\nb_v|^{\fr{s}{2}}\big(Wg_0^{(m-n+\beta-\beta')}\la v\ra^{\ell-2\iota-2|\al|}\big)\Big\|_{L^2_v}\Big)^{\fr{\fr{s}{2}}{1-\fr{s}{2}}}\\
    \nn&\times a_{m,\lm,s}(t)\Big\|\mathcal{F}_x\big[\la Y\ra^{\fr{s}{2}}\big(Wg_{\ne}^{(m+\beta)}\la v\ra^{\ell-2\iota-2|\al|}\big)\big]_k\Big\|_{L^2_v}\\
    \les\nn&\la t\ra^{\iota+\fr{1+a}{2}(1+\fr{\fr{s}{2}}{1-\fr{s}{2}})}\big\|\la t\nb _x\ra^{\fr{s}{2}} E\big\|_{\mathcal{G}^{\lm,\fr{N}{2}}_s}\Big(\nu^{\fr12}\big\|\nb_vWg_0\|_{\mathcal{G}^{\lm,N}_{s,\ell-2\iota-2|\al|}}+\nu^{\fr12}\big\|Wg_0\|_{\mathcal{G}^{\lm,N}_{s,\ell-2\iota-2|\al|}}\Big)^{\fr{1-s}{1-\fr{s}{2}}}\\
    \nn&\times\Big(\mathfrak{CK}_{\ell-2\iota-2|\al|}[Wg_{0}(t)]^{\fr12}+\fr{1}{\la t\ra^{\fr{1+a}{2}}}\|Wg_0\|_{\mathcal{G}^{\lm,N}_{s,\ell-2\iota-2|\al|}}\Big)^{\fr{\fr{s}{2}}{1-\fr{s}{2}}}\\
    \nn&\times\Big(\mathfrak{CK}_{\ell-2\iota-2|\al|}[Wg_{\ne}(t)]^{\fr12}+\fr{1}{\la t\ra^{\fr{1+a}{2}}}\|Wg_{\ne}\|_{\mathcal{G}^{\lm,N}_{s,\ell-2\iota-2|\al|}}\Big)\\
    \les\nn&\la t\ra^{\iota+\fr{1+a}{2-s}+\fr{s}{2}}\big\||\nb_x|^\fr{s}{2}E\big\|_{\mathcal{G}^{\lm,\fr{N}{2}}_{s}}\Big(\mathfrak{CK}_{\ell-2|\al|}[Wg_{\ne}(t)]^{\fr12}+\fr{\|Wg_{\ne}\|_{\mathcal{G}^{\lm,N}_{s,\ell-2|\al|}}}{\la t\ra^{\fr{1+a}{2}}}\Big)\big\|Wg_0\|_{\mathcal{G}^{\lm,N}_{s,\ell-2|\al|}}\\
    \nn&+\la t\ra^{\iota+\fr{1+a}{2-s}+\fr{s}{2}}\big\||\nb_x|^\fr{s}{2}E\big\|_{\mathcal{G}^{\lm,\fr{N}{2}}_{s}}\Big(\mathfrak{CK}_{\ell-2|\al|}[Wg_{\ne}(t)]^{\fr12}+\fr{\|Wg_{\ne}\|_{\mathcal{G}^{\lm,N}_{s,\ell-2|\al|}}}{\la t\ra^{\fr{1+a}{2}}}\Big)\\
    &\times\Big(\nu^{\fr16}\big\|\nu^{\fr13}Wg_0\|_{\mathcal{G}^{\lm,N}_{s,\sig,\ell-2|\al|}}+\mathfrak{CK}_{\ell-2|\al|}[Wg_{0}(t)]^{\fr12}+\fr{1}{\la t\ra^{\fr{1+a}{2}}}\|Wg_0\|_{\mathcal{G}^{\lm,N}_{s,\ell-2|\al|}}\Big).
\end{align}
Clearly, $T^{\rm LH}_{\iota;0}[E]^{m,\beta}_{2)}$ can be treated in the same way. Similar to \eqref{est-Eg0-LH-iota-e} and \eqref{est-Eg0-LH-iota}, we arrive at
\begin{align}\label{est-Eg0-LH-iota-2}
    \nn&\sum_{m\in\N^6,|\beta|\le N}\kappa^{2|\beta|}a_{m,\lm,s}^2(t)\big|T^{\rm LH}_{\iota;0}[E]^{m,\beta}_{2)}\big|\\
    \les\nn&w_{\iota}(t)\|WE\|_{\mathcal{G}^{\lm,\fr{N}{2}}_s}\|g_0\|_{\mathcal{G}^{\lm,N}_{s,\ell}}\|Wg_{\iota\ne}\|_{\mathcal{G}^{\lm,N}_{s,\ell-2\iota-2|\al}}\\
    \nn&+\la t\ra^{\iota+\fr{1+a}{2-s}+\fr{s}{2}}\big\||\nb_x|^\fr{s}{2}WE\big\|_{\mathcal{G}^{\lm,\fr{N}{2}}_{s}}\Big(\mathfrak{CK}_{\ell-2|\al|}[Wg_{\ne}(t)]^{\fr12}+\fr{\|Wg_{\ne}\|_{\mathcal{G}^{\lm,N}_{s,\ell-2|\al|}}}{\la t\ra^{\fr{1+a}{2}}}\Big)\big\|g_0\|_{\mathcal{G}^{\lm,N}_{s,\ell}}\\
    \nn&+\la t\ra^{\iota+\fr{1+a}{2-s}+\fr{s}{2}}\big\||\nb_x|^\fr{s}{2}WE\big\|_{\mathcal{G}^{\lm,\fr{N}{2}}_{s}}\Big(\mathfrak{CK}_{\ell-2|\al|}[Wg_{\ne}(t)]^{\fr12}+\fr{\|Wg_{\ne}\|_{\mathcal{G}^{\lm,N}_{s,\ell-2|\al|}}}{\la t\ra^{\fr{1+a}{2}}}\Big)\\
    &\times\Big(\nu^{\fr16}\big\|\nu^{\fr13}g_0\|_{\mathcal{G}^{\lm,N}_{s,\sig,\ell}}+\mathfrak{CK}_{\ell}[g_{0}(t)]^{\fr12}+\fr{1}{\la t\ra^{\fr{1+a}{2}}}\|g_0\|_{\mathcal{G}^{\lm,N}_{s,\ell}}\Big).
\end{align}

The high-low frequency part $T^{\rm HL}_{\iota;0}[E]^{m,\beta}_{1)}$ and $T^{\rm HL}_{\iota;0}[E]^{m,\beta}_{2)}$ can be treated in a similar manner as that of $L[\mu]_{\iota;1)}^{(1),{\rm HL}}$. In fact, thanks to \eqref{up-kkt}, similar to \eqref{est-Lmuiota11HL},  we have
\begin{align}\label{est-HL-iota}
    \nn&\sum_{m\in\N^6,|\beta|\le N}\kappa^{2|\beta|}a_{m,\lm,s}^2(t)\Big(\big|T^{\rm HL}_{\iota;0}[E]^{m,\beta}_{1)}\big|+\big|T^{\rm HL}_{\iota;0}[E]^{m,\beta}_{2)}\big|\Big)\\
    \nn\les&\Big[ w_{\iota}(t)\la t\ra^{\fr{1+a-s}{2}}\|\rho\|_{\mathcal{G}^{\lm,N}_s}\Big]\Big[\mathfrak{CK}_{\ell-2\iota-2|\al|}[Wg_{\iota\ne}(t)]^{\fr12}+\fr{1}{\la t\ra^{\fr{1+a}{2}}}\|Wg_{\iota\ne}\|_{\mathcal{G}^{\lm,N}_{s,\ell-2\iota-2|\al|}}\Big]\\
    &\times\Big(\big\|\la \nb_v\ra\nb_vWg_0\big\|_{\mathcal{G}^{\lm,\fr{N}{2}}_{s,\ell}}+\big\|\la \nb_v\ra\nb_vg_0\big\|_{\mathcal{G}^{\lm,\fr{N}{2}}_{s,\ell}}\Big).
\end{align}

Now collecting the estimates \eqref{est-Eg0-LH-iota-e}--\eqref{est-HL-iota}, and noting that $N[E]_{\iota;1)}^0$ and  $N[E]_{\iota;2)}^0$ eventually have the same upper bound, we conclude that
\begin{align}\label{est-NEiota-0}
    \nn&N[E]_{\iota;1)}^0+N[E]_{\iota;2)}^0\\
    \les\nn&w_{\iota}(t)\|\rho\|_{\mathcal{G}^{\lm,\fr{N}{2}}_s}\mathcal{E}_\ell(g_0(t))^{\fr12}\mathcal{E}_{\ell-2\iota}(g_{\iota\ne}(t))^{\fr12}\\
    \nn&+\la t\ra^{\iota+\fr{1+a}{2-s}+\fr{s}{2}}\big\||\nb_x|^{\fr{s}{2}}\rho\big\|_{\mathcal{G}^{\lm,\fr{N}{2}}_s}\Big(\mathcal{CK}_{\ell}(g_{\ne}(t))^{\fr12}+\fr{\mathcal{E}_\ell(g_{\ne}(t))^{\fr12}}{\la t\ra^{\fr{1+a}{2}}}\Big)\mathcal{E}_\ell(g_0(t))^{\fr12}\\
    \nn&+\la t\ra^{\iota+\fr{1+a}{2-s}+\fr{s}{2}}\big\||\nb_x|^{\fr{s}{2}}\rho\big\|_{\mathcal{G}^{\lm,\fr{N}{2}}_s}\Big(\mathcal{CK}_{\ell}(g_{\ne}(t))^{\fr12}+\fr{\mathcal{E}_\ell(g_{\ne}(t))^{\fr12}}{\la t\ra^{\fr{1+a}{2}}}\Big)\\
    \nn&\times\Big(\nu^{\fr16}\mathcal{D}_\ell(g_0(t))^{\fr12}+\mathcal{CK}_{\ell}(g_{0}(t))^{\fr12}+\fr{1}{\la t\ra^{\fr{1+a}{2}}}\mathcal{E}_\ell(g_0(t))^{\fr12}\Big)\\
    &+\Big( w_{\iota}(t)\la t\ra^{\fr{1+a-s}{2}}\|\rho\|_{\mathcal{G}^{\lm,N}_s}\Big)\Big(\mathcal{CK}_{\ell-2\iota}(g_{\iota\ne}(t))^{\fr12}+\fr{1}{\la t\ra^{\fr{1+a}{2}}}\mathcal{E}_{\ell-2\iota}(g_{\iota\ne}(t))^{\fr12}\Big)\mathcal{E}_{\ell}(g_0(t))^{\fr12}.
\end{align}
It follows from \eqref{est-NE01}, \eqref{est-NEiota1-ne} and \eqref{est-NEiota-0} that \eqref{est-transport} holds.
\end{proof}

Finally, we bound the  the collision nonlinearities $N[\Gamma]_{\iota;1)}$ and $N[\Gamma]_{\iota;2)}$.
\begin{lem}\label{lem-nonlinear-collision}
Under the bootstrap hypotheses \eqref{bd-en}--\eqref{H-phi}, there holds
\begin{align}\label{est-nonlinear-collision}
    \nn&2\sum_{\iota=0}^{[\ell/3]+1}\Big[N[\Gamma]_{\iota;1)}+N[\Gamma]_{\iota;2)}\Big]\\
    \nn\le& C\mathcal{E}_\ell(g(t))^{\fr12}\Big(\nu^{\fr13}\mathcal{D}_\ell(g(t))+\sum_{\iota=1}^{[\ell/3]+1} \nu^{\fr13}w_{\iota}^2(t)\mathcal{D}_{\ell-2\iota}(g_{\ne}\Big)\\
    \nn&+C\sum_{\iota=1}^{[\ell/3]+1} \mathcal{E}_{\ell-2\iota}(g_{\iota\ne}(t))^{\fr12}\Big(\nu^{\fr13}\mathcal{D}_{\ell}(g(t))^{\fr12}\mathcal{D}_{\ell-2\iota}(g_{\iota\ne})^{\fr12}\Big)\\
    \nn&+C\sum_{\iota=1}^{[\ell/3]+1} \Big(w_{\iota}(t)\|\rho\|_{\mathcal{G}^{\lm,\fr{N}{2}}_s}\Big)\mathcal{E}_{\ell}(g(t))^{\fr12}\Big(\nu^{\fr13}\mathcal{D}_{\ell}(g(t))^{\fr12}\mathcal{D}_{\ell-2\iota}(g_{\iota\ne}(t))^{\fr12}\Big)\\
    &+C\sum_{\iota=1}^{[\ell/3]+1}\Big(w_{\iota}(t)\|\rho\|_{\mathcal{G}^{\lm,N}_s}\Big)\mathcal{E}_{\ell}(g(t))^{\fr12}\Big(\nu^{\fr23} \mathcal{D}_\ell(g(t))^{\fr12}\Big)\Big( w_{\iota}(t)\mathcal{E}_{\ell-2\iota}(g_{\ne}(t))^{\fr12}\Big). 
\end{align}
\end{lem}
\begin{proof}
{\bf Case 1: $\iota=0$.}
Note that for $|\al|=1$,
\begin{align*}
    &\nu\left\la \pr_{x}^\al Z^{m+\beta}\big(e^{-\phi}\Gamma(g,g)\big), \pr_x^{\al}g^{(m+\beta)}\right\ra_{\ell-2|\al|}\\
    =&\nu \left\la   Z^{m+\beta} \Gamma(\pr_{x}^\al f,g),\pr_x^{\al}g^{(m+\beta)}\right\ra_{\ell-2|\al|}+\nu \left\la   Z^{m+\beta} \Gamma(f,\pr_{x}^\al g),\pr_x^{\al}g^{(m+\beta)}\right\ra_{\ell-2|\al|}.
\end{align*}
By Lemma \ref{lem-NL-collision} and Remark \ref{rem-NL-collision}, we find that
\begin{align*}
    &N[\Gamma]_{0;1)}\\
\le\nn&C\nu \Big(\sqrt{{\rm A}_0}+\kappa(1+t)^{-1}+\kappa\nu^{\fr13}\Big)\Big(\|f\|_{\mathcal{G}^{\lm,N}_{s,\ell}}\|g\|_{\mathcal{G}^{\lm,N}_{s,\sig,\ell}}+\|f\|_{\mathcal{G}^{\lm,N}_{s,\sig,\ell}}\|g\|_{\mathcal{G}^{\lm,N}_{s,\ell}}\Big)\\
\nn&\times\Big(\sqrt{{\rm A}_0}\|g\|_{\mathcal{G}^{\lm,N}_{s,\sig,\ell}}+\kappa(1+t)^{-1}\|Yg\|_{\mathcal{G}^{\lm,N}_{s,\sig,\ell-2}}+\sum_{1\le|\al|\le2}\big\|\kappa^{|\al|}\nu^{\fr{|\al|}{3}}\pr_v^\al g\big\|_{\mathcal{G}^{\lm,N}_{s,\sig,\ell-2|\al|}}\Big)\\
\nn&+C\nu {\rm A}_0\Big(\|\nb_xf\|_{\mathcal{G}^{\lm,N}_{s,\ell-2}}\|g\|_{\mathcal{G}^{\lm,N}_{s,\sig,\ell}}+\|\nb_xf\|_{\mathcal{G}^{\lm,N}_{s,\sig,\ell-2}}\|g\|_{\mathcal{G}^{\lm,N}_{s,\ell}}\Big)\|\nb_xg\|_{\mathcal{G}^{\lm,N}_{s,\sig,\ell-2}}\\
\nn&+C\nu {\rm A}_0\Big(\|\nb_xg\|_{\mathcal{G}^{\lm,N}_{s,\ell-2}}\|f\|_{\mathcal{G}^{\lm,N}_{s,\sig,\ell}}+\|\nb_xg\|_{\mathcal{G}^{\lm,N}_{s,\sig,\ell-2}}\|f\|_{\mathcal{G}^{\lm,N}_{s,\ell}}\Big)\|\nb_xg\|_{\mathcal{G}^{\lm,N}_{s,\sig,\ell-2}}\\
\nn&+C\nu \kappa^2(1+t)^{-2}\Big(\|Yf\|_{\mathcal{G}^{\lm,N}_{s,\ell-2}}\|g\|_{\mathcal{G}^{\lm,N}_{s,\sig,\ell}}+\|Yf\|_{\mathcal{G}^{\lm,N}_{s,\sig,\ell-2}}\|g\|_{\mathcal{G}^{\lm,N}_{s,\ell}}\Big)\|Yg\|_{\mathcal{G}^{\lm,N}_{s,\sig,\ell-2}}\\
\nn&+C\nu \kappa^2(1+t)^{-2}\Big(\|Yg\|_{\mathcal{G}^{\lm,N}_{s,\ell-2}}\|f\|_{\mathcal{G}^{\lm,N}_{s,\sig,\ell}}+\|Yg\|_{\mathcal{G}^{\lm,N}_{s,\sig,\ell-2}}\|f\|_{\mathcal{G}^{\lm,N}_{s,\ell}}\Big)\|Yg\|_{\mathcal{G}^{\lm,N}_{s,\sig,\ell-2}}\\
\nn&+C\nu\kappa^2\nu^{\fr23} \Big(\|\nb_vf\|_{\mathcal{G}^{\lm,N}_{s,\ell-2}}\|g\|_{\mathcal{G}^{\lm,N}_{s,\sig,\ell}}+\|\nb_vf\|_{\mathcal{G}^{\lm,N}_{s,\sig,\ell-2}}\|g\|_{\mathcal{G}^{\lm,N}_{s,\ell}}\Big)\|\nb_vg\|_{\mathcal{G}^{\lm,N}_{s,\sig,\ell-2}}\\
\nn&+C\nu\kappa^2\nu^{\fr23} \Big(\|f\|_{\mathcal{G}^{\lm,N}_{s,\ell}}\|\nb_vg\|_{\mathcal{G}^{\lm,N}_{s,\sig,\ell-2}}+\|f\|_{\mathcal{G}^{\lm,N}_{s,\sig,\ell}}\|\nb_vg\|_{\mathcal{G}^{\lm,N}_{s,\ell-2}}\Big)\|\nb_vg\|_{\mathcal{G}^{\lm,N}_{s,\sig,\ell-2}}\\
\nn&+C\nu\kappa^4\nu^{\fr43} \Big(\|f\|_{\mathcal{G}^{\lm,N}_{s,\ell}}\|\Dl_vg\|_{\mathcal{G}^{\lm,N}_{s,\sig,\ell-4}}+\|f\|_{\mathcal{G}^{\lm,N}_{s,\sig,\ell}}\|\Dl_vg\|_{\mathcal{G}^{\lm,N}_{s,\ell-4}}\Big)\|\Dl_vg\|_{\mathcal{G}^{\lm,N}_{s,\sig,\ell-4}}\\
\nn&+C\nu\kappa^4\nu^{\fr43} \sum_{0\le|\al|\le1}\Big(\|\pr_v^\al f\|_{\mathcal{G}^{\lm,N}_{s,\ell-2|\al|}}\|\nb_vg\|_{\mathcal{G}^{\lm,N}_{s,\sig,\ell-2}}+\|\pr_v^\al f\|_{\mathcal{G}^{\lm,N}_{s,\sig,\ell-2|\al|}}\|\nb_vg\|_{\mathcal{G}^{\lm,N}_{s,\ell-2}}\Big)\|\Dl_v g\|_{\mathcal{G}^{\lm,N}_{s,\sig,\ell-4}}\\
\nn&+C\nu\kappa^4\nu^{\fr43}\sum_{1\le|\al|\le2} \Big(\|\pr_v^\al f\|_{\mathcal{G}^{\lm,N}_{s,\ell-2|\al|}}\|g\|_{\mathcal{G}^{\lm,N}_{s,\sig,\ell}}+\|\pr_v^\al f\|_{\mathcal{G}^{\lm,N}_{s,\sig,\ell-2|\al|}}\|g\|_{\mathcal{G}^{\lm,N}_{s,\ell}}\Big)\|\Dl_v g\|_{\mathcal{G}^{\lm,N}_{s,\sig,\ell-4}}.
\end{align*}
It follows from this,  Corollary \ref{coro-compose} and Lemma \ref{lem-Linfty-rho} that
\begin{align}\label{est-NGamma01}
N[\Gamma]_{0;1)}\nn\le&C\Big(\sqrt{{\rm A}_0}\|g\|_{\mathcal{G}^{\lm,N}_{s,\sig,\ell}}+\kappa(1+t)^{-1}\|Yg\|_{\mathcal{G}^{\lm,N}_{s,\sig,\ell-2}}+\sum_{1\le|\al|\le2}\big\|\kappa^{|\al|}\nu^{\fr{|\al|}{3}}\nb_vg\big\|_{\mathcal{G}^{\lm,N}_{s,\sig,\ell-2|\al|}}\Big)\\
\nn&\times\nu \|g\|_{\mathcal{G}^{\lm,N}_{s,\ell}} \Big(\sqrt{{\rm A}_0}\|g\|_{\mathcal{G}^{\lm,N}_{s,\sig,\ell}}\Big)\\
\nn&+C\nu{\rm A}_0\sum_{|\al|\le1}\|\pr_x^\al g\|_{\mathcal{G}^{\lm,N}_{s,\ell-2|\al|}}\sum_{|\al|\le1}\|\pr_x^\al g\|_{\mathcal{G}^{\lm,N}_{s,\sig,\ell-2|\al|}}\|\nb_xg\|_{\mathcal{G}^{\lm,N}_{s,\ell-2}}\\
\nn&+C\nu\kappa^2(1+t)^{-2}\Big[\Big(t\|g\|_{\mathcal{G}^{\lm,N}_{s,\ell}}+\|Yg\|_{\mathcal{G}^{\lm,N}_{s,\ell}}\Big)\|g\|_{\mathcal{G}^{\lm,N}_{s,\sig,\ell}}\\
\nn&+\Big(t\|g\|_{\mathcal{G}^{\lm,N}_{s,\sig,\ell}}+\|Yg\|_{\mathcal{G}^{\lm,N}_{s,\sig,\ell}}\Big)\|g\|_{\mathcal{G}^{\lm,N}_{s,\ell}}\Big]\|Yg\|_{\mathcal{G}^{\lm,N}_{s,\sig,\ell-2}}\\
\nn&+C\nu \kappa^2\nu^{\fr23}\Big(\|\nb_vg\|_{\mathcal{G}^{\lm,N}_{s,\ell-2}}\|g\|_{\mathcal{G}^{\lm,N}_{s,\sig,\ell}}+\|\nb_vg\|_{\mathcal{G}^{\lm,N}_{s,\sig,\ell-2}}\|g\|_{\mathcal{G}^{\lm,N}_{s,\ell-2}}\Big)\|\nb_vg\|_{\mathcal{G}^{\lm,N}_{s,\sig,\ell}}\\
\nn&+C\nu\kappa^4\nu^{\fr43} \Big(\|g\|_{\mathcal{G}^{\lm,N}_{s,\ell}}\|\Dl_vg\|_{\mathcal{G}^{\lm,N}_{s,\sig,\ell-4}}+\|g\|_{\mathcal{G}^{\lm,N}_{s,\sig,\ell}}\|\Dl_vg\|_{\mathcal{G}^{\lm,N}_{s,\ell-4}}\Big)\|\Dl_vg\|_{\mathcal{G}^{\lm,N}_{s,\sig,\ell-4}}\\
\nn&+C\nu\kappa^4\nu^{\fr43} \sum_{0\le|\al|\le1}\Big(\|\pr_v^\al g\|_{\mathcal{G}^{\lm,N}_{s,\ell-2|\al|}}\|\nb_vg\|_{\mathcal{G}^{\lm,N}_{s,\sig,\ell-2}}\\
\nn&\quad\quad\quad\quad\quad\quad\quad\quad\quad+\|\pr_v^\al g\|_{\mathcal{G}^{\lm,N}_{s,\sig,\ell-2|\al|}}\|\nb_vg\|_{\mathcal{G}^{\lm,N}_{s,\ell-2}}\Big)\|\Dl_v g\|_{\mathcal{G}^{\lm,N}_{s,\sig,\ell-4}}\\
&\le  C\mathcal{E}_\ell(g(t))^{\fr12} \Big(\nu^{\fr13}\mathcal{D}_\ell(g(t))\Big).
\end{align}
Similarly, using Lemma \ref{lem-NL-collision}, Corollary \ref{coro-compose} and Lemma \ref{lem-Linfty-rho} again, we are led to
\begin{align}\label{est-NGamma02}
    \nn&N[\Gamma]_{0;2)}\\
    \le\nn&C\nu\kappa\nu^\fr13\Big(\sum_{|\al|\le1}\|\pr_x^\al g\|_{\mathcal{G}^{\lm,N}_{s,\ell-2|\al|}}\sum_{|\al|\le1}\|\pr_x^\al f\|_{\mathcal{G}^{\lm,N}_{s,\sig,\ell-2|\al|}}\\
    \nn&+\sum_{|\al|\le1}\|\pr_x^\al g\|_{\mathcal{G}^{\lm,N}_{s,\sig,\ell-2|\al|}}\sum_{|\al|\le1}\|\pr_x^\al f\|_{\mathcal{G}^{\lm,N}_{s,\ell-2|\al|}}\Big)\|\nb_vg\|_{\mathcal{G}^{\lm,N}_{s,\sig,\ell-2}}\\
    \nn&+C\nu\kappa\nu^\fr13\Big(\sum_{|\al|\le1}\|\pr_v^\al g\|_{\mathcal{G}^{\lm,N}_{s,\ell-2|\al|}}\|f\|_{\mathcal{G}^{\lm,N}_{s,\sig,\ell}}+\sum_{|\al|\le1}\|\pr_v^\al g\|_{\mathcal{G}^{\lm,N}_{s,\sig,\ell-2|\al|}}\| f\|_{\mathcal{G}^{\lm,N}_{s,\ell}}\Big)\|\nb_xg\|_{\mathcal{G}^{\lm,N}_{s,\sig,\ell-2}}\\
    \nn&+C\nu\kappa\nu^\fr13\Big(\| g\|_{\mathcal{G}^{\lm,N}_{s,\ell}}\|\nb_v f\|_{\mathcal{G}^{\lm,N}_{s,\sig,\ell-2}}
    +\|g\|_{\mathcal{G}^{\lm,N}_{s,\sig,\ell}}\|\nb_v f\|_{\mathcal{G}^{\lm,N}_{s,\ell-2}}\Big)\|\nb_xg\|_{\mathcal{G}^{\lm,N}_{s,\sig,\ell-2}}\\
    \le\nn&\nu^{\fr13}C\sum_{|\al|\le1}\|\pr_x^\al g\|_{\mathcal{G}^{\lm,N}_{s,\ell-2|\al|}}\sum_{|\al|\le1}\Big(\nu^\fr13\|\pr_x^\al g\|_{\mathcal{G}^{\lm,N}_{s,\sig,\ell-2|\al|}}\Big)\Big(\kappa\nu^{\fr23}\|\nb_vg\|_{\mathcal{G}^{\lm,N}_{s,\sig,\ell-2}}\Big)\\
    \nn&+C\nu^\fr13\Big(\sum_{|\al|\le1}\kappa\nu^{\fr13}\|\pr_v^\al  g\|_{\mathcal{G}^{\lm,N}_{s,\ell-2|\al|}}\nu^{\fr13}\|g\|_{\mathcal{G}^{\lm,N}_{s,\sig,\ell}}\\
    \nn&+\sum_{|\al|\le1}\kappa\nu^{\fr23}\|\pr_v^\al  g\|_{\mathcal{G}^{\lm,N}_{s,\sig,\ell-2|\al|}}\|g\|_{\mathcal{G}^{\lm,N}_{s,\ell}}\Big)\Big(\nu^{\fr13}\|\nb_xg\|_{\mathcal{G}^{\lm,N}_{s,\sig,\ell-2}}\Big)\\
    &\le  C\mathcal{E}_\ell(g(t))^{\fr12} \Big(\nu^{\fr13}\mathcal{D}_\ell(g(t))\Big).
\end{align}

{\bf Case 2: $\iota=1, 2, \cdots, [\ell/3]+1$.} We write
\begin{align*}
    &\Big(e^{-\phi}\Gamma(g,g)\Big)_{\iota\ne}\\
    =&\Big(\big(1+q(\phi)\big)\Gamma(g_{\iota\ne},g)\Big)_{\ne}+\Big(\big(1+q(\phi)\big)\Gamma(g_{0},g_{\iota\ne})\Big)_{\ne}+w_{\iota}(t)q(\phi)_{\ne}\Gamma(g_0,g_0)\\
    =&\Gamma_{\iota;\ne(1)}+\Gamma_{\iota;\ne(2)}+\Gamma_{\iota;0},
\end{align*}
and accordingly we split $N[\Gamma]_{\iota;j)}$ into three parts and denote
\[
N[\Gamma]_{\iota;j)}=N[\Gamma]_{\iota;j)}^{\ne(1)}+N[\Gamma]_{\iota;j)}^{\ne(2)}+N[\Gamma]_{\iota;j)}^{0},\quad j=1,2.
\]
Note that $\Gamma_{\iota;\ne(1)}=\Gamma(g_{\iota\ne},f)_{\ne}$, and $\Gamma_{\iota;\ne(2)}=\Gamma(\tl{g},g_{\iota\ne})_{\ne}$, where $\tl{g}=\big(1+q(\phi)\big)g_0$. By the product estimates \eqref{product1}, \eqref{product2}, Lemma \ref{lem-compose}, Lemma \ref{lem-Linfty-rho}, \eqref{positive-sig1} and \eqref{positive-sig2} in Corollary \ref{coro-sig}, we have 
\begin{align*}
    \|\nb_x\tl{g}\|_{\mathcal{G}^{\lm,N}_{s,\ell-2\iota-2}}\les \|\nb_x\phi\|_{\mathcal{G}^{\lm,N}_{s}} \big(1+\|\phi\|_{\mathcal{G}^{\lm,N}_s}\big)\|g_0\|_{\mathcal{G}^{\lm,N}_{s,\ell-2\iota-2}}\les \|g\|_{\mathcal{G}^{\lm,N}_{s,\ell}},
\end{align*} 

\begin{align*}
    \|Y\tl{g}\|_{\mathcal{G}^{\lm,N}_{s,\ell-2\iota-2}}\les& \big(1+\|\phi\|_{\mathcal{G}^{\lm,N}_s}\big)\|Yg_0\|_{\mathcal{G}^{\lm,N}_{s,\ell-2\iota-2}}+t\|\nb_x\phi\|_{\mathcal{G}^{\lm,N}_{s}}\big(1+\|\phi\|_{\mathcal{G}^{\lm,N}_s}\big)\|g_0\|_{\mathcal{G}^{\lm,N}_{s,\ell-2\iota-2}}\\
    \les& \|Y g\|_{\mathcal{G}^{\lm,N}_{s,\ell-2}}+t\| g\|_{\mathcal{G}^{\lm,N}_{s,\ell}},
\end{align*}

\begin{align*}
    \|\pr_v^\al\tl{g}\|_{\mathcal{G}^{\lm,N}_{s,\ell-2\iota-2|\al|}}\les& \big(1+\|\phi\|_{\mathcal{G}^{\lm,N}_s}\big)\|\pr_v^\al g_0\|_{\mathcal{G}^{\lm,N}_{s,\ell-2\iota-2|\al|}}
    \les \|\pr_v^\al g\|_{\mathcal{G}^{\lm,N}_{s,\ell-2|\al|}}, \quad 0\le|\al|\le2,
\end{align*}
and
\begin{align*}
    \|\nb_x\tl{g}\|_{\mathcal{G}^{\lm,N}_{s,\sig,\ell-2\iota-2}}\les \|\nb_x\phi\|_{\mathcal{G}^{\lm,N}_{s}} \big(1+\|\phi\|_{\mathcal{G}^{\lm,N}_s}\big)\|g_0\|_{\mathcal{G}^{\lm,N}_{s,\sig,\ell-2\iota-2}}\les \|g\|_{\mathcal{G}^{\lm,N}_{s,\sig,\ell}},
\end{align*} 

\begin{align*}
    \|Y\tl{g}\|_{\mathcal{G}^{\lm,N}_{s,\sig,\ell-2\iota-2}}\les& \big(1+\|\phi\|_{\mathcal{G}^{\lm,N}_s}\big)\|Yg_0\|_{\mathcal{G}^{\lm,N}_{s,\sig,\ell-2\iota-2}}+t\|\nb_x\phi\|_{\mathcal{G}^{\lm,N}_{s}}\big(1+\|\phi\|_{\mathcal{G}^{\lm,N}_s}\big)\|g_0\|_{\mathcal{G}^{\lm,N}_{s,\sig,\ell-2\iota-2}}\\
    \les& \|Y g\|_{\mathcal{G}^{\lm,N}_{s,\sig,\ell-2}}+t\| g\|_{\mathcal{G}^{\lm,N}_{s,\ell}},
\end{align*}

\begin{align*}
    \|\pr_v^\al\tl{g}\|_{\mathcal{G}^{\lm,N}_{s,\sig, \ell-2\iota-2|\al|}}\les& \big(1+\|\phi\|_{\mathcal{G}^{\lm,N}_s}\big)\|\pr_v^\al g_0\|_{\mathcal{G}^{\lm,N}_{s,\sig,\ell-2\iota-2|\al|}}
    \les \|\pr_v^\al g\|_{\mathcal{G}^{\lm,N}_{s,\sig,\ell-2|\al|}}, \quad 0\le|\al|\le2,
\end{align*}
Consequently, following the estimates of $N[\Gamma]_{0;1)}$ and $N[\Gamma]_{0;2)}$ line by line, similar to \eqref{est-NGamma01} and \eqref{est-NGamma02}, we are led to
\begin{align}\label{est-NGamma-iota-ne}
    \nn&\sum_{j=1}^2\Big(N[\Gamma]_{\iota;1)}^{\ne(j)}+N[\Gamma]_{\iota;2)}^{\ne(j)}\Big)\\
    \le&C\mathcal{E}_{\ell-2\iota}(g_{\iota\ne}(t))^{\fr12}\Big(\nu^{\fr13}\mathcal{D}_{\ell}(g(t))^{\fr12}\mathcal{D}_{\ell-2\iota}(g_{\iota\ne})^{\fr12}\Big)+C\mathcal{E}_{\ell}(g(t))^{\fr12}\Big(\nu^{\fr13}\mathcal{D}_{\ell-2\iota}(g_{\iota\ne})\Big).
\end{align}

We are left to bound $N[\Gamma]_{\iota;j)}^0, j=1,2$. To this end,  note first that
\begin{gather*}
    \pr_x^\al \big(q(\phi)_{\ne}\Gamma(g_0,g_0)\big)=-\pr_x^\al\phi\big(1+q(\phi)\big)\Gamma(g_0,g_0),\\
    Y^\al \big(q(\phi)_{\ne}\Gamma(g_0,g_0)\big)=-t\pr_x^\al\phi\big(1+q(\phi)\big)\Gamma(g_0,g_0)+q(\phi)_{\ne}\pr_v^{\al}\Gamma(g_0,g_0),
\end{gather*}
for $|\al|=1$, and
\begin{align*}
    \pr_v^\al \big(q(\phi)_{\ne}\Gamma(g_0,g_0)\big)=q(\phi)_{\ne}\pr_v^\al\Gamma(g_0,g_0),\quad \quad{\rm for}\ \  1\le|\al|\le2.
\end{align*}
For $W\in \big\{\pr_x^\al, (1+ t)^{-|\al|}Y^\al: 0\le|\al|\le1\big\}\cup \big\{\pr_v^\al: 1\le|\al|\le2 \big\}$ with $\al\in\N^3$, we write
\begin{align*}
    &WZ^{m+\beta}\big(q(\phi)_{\ne}\Gamma(g_0,g_0)\big)\\
    =&Z^{m+\beta}\big(Wq(\phi)_{\ne}\Gamma(g_0,g_0)\big)+Z^{m+\beta}\big(q(\phi)_{\ne}W\Gamma(g_0,g_0)\big)\\
    =&\Big(\sum_{\substack{n\le m,\beta'\le\beta\\|\beta'|\le|\beta|/2}}+\sum_{\substack{n\le m,\beta'\le\beta\\|\beta'|>|\beta|/2}}\Big)C_{\beta}^{\beta'}C_m^n Z^{n+\beta'}Wq(\phi)_{\ne} Z^{m-n+\beta-\beta'}\Gamma(g_0,g_0)\\
    &+\Big(\sum_{\substack{n\le m,\beta'\le\beta\\|\beta'|\le|\beta|/2}}+\sum_{\substack{n\le m,\beta'\le\beta\\|\beta'|>|\beta|/2}}\Big)C_{\beta}^{\beta'}C_m^n Z^{n+\beta'}q(\phi)_{\ne}W Z^{m-n+\beta-\beta'}\Gamma(g_0,g_0).
\end{align*}
Based on these decompositions, we split $N[\Gamma]_{\iota;j)}^0$ into four parts:
\begin{align*}
    N[\Gamma]_{\iota;j)}^0=N[\Gamma]_{\iota;j),1}^{0;{\rm LH}}+N[\Gamma]_{\iota;j),1}^{0;{\rm HL}}+N[\Gamma]_{\iota;j),2}^{0;{\rm LH}}+N[\Gamma]_{\iota;j),2}^{0;{\rm HL}}, \quad j=1,2,
\end{align*}
here the last subscript 1 indicates $W$ hits on $q(\phi)$, and the last subscript 2 indicates $W$ hits on $\Gamma(g_0,g_0)$. 

For $N[\Gamma]_{\iota;j),1}^{0,{\rm LH}}$ and $N[\Gamma]_{\iota;j),2}^{0,{\rm LH}}$, $q(\phi)$ is at low frequency, hence can be used to absorb the extra time grow factor $w_{\iota}(t)$. Moreover, noting that $q(\phi)$ is independent of $v$ variable, Lemma \ref{lem-NL-collision} still  applies to the estimates of $N[\Gamma]_{\iota;j),1}^{0,{\rm LH}}$ and $N[\Gamma]_{\iota;j),2}^{0,{\rm LH}}$. Indeed, similar to \eqref{est-NGamma-iota-ne}, we have
\begin{align}\label{est-NGamma-iota-0-LH}
    \nn&\sum_{j=1}^2\Big(N[\Gamma]_{\iota;j),1}^{0;{\rm LH}}+N[\Gamma]_{\iota;j),2}^{0;{\rm LH}}\Big)\\
    \le\nn& C \Big(w_{\iota}(t)\|\nb_x\phi\|_{\mathcal{G}^{\lm,\fr{N}{2}}_s}\Big)\mathcal{E}_{\ell}(g_0(t))^{\fr12}\Big(\nu^{\fr13}\mathcal{D}_{\ell}(g_0(t))^{\fr12}\mathcal{D}_{\ell-2\iota}(g_{\iota\ne}(t))^{\fr12}\Big)\\
    \le &C \Big(w_{\iota}(t)\|\rho\|_{\mathcal{G}^{\lm,\fr{N}{2}}_s}\Big)\mathcal{E}_{\ell}(g(t))^{\fr12}\Big(\nu^{\fr13}\mathcal{D}_{\ell}(g(t))^{\fr12}\mathcal{D}_{\ell-2\iota}(g_{\iota\ne}(t))^{\fr12}\Big).
\end{align}

Finally, we  turn to treat $N[\Gamma]_{\iota;j),1}^{0,{\rm HL}}$ and $N[\Gamma]_{\iota;j),2}^{0,{\rm HL}}$. When $W$ hits on $q(\phi)$, the estimates are easier, so we only give the details when $W$ hits on $\Gamma(g_0,g_0)$. Instead of  integrating by part over the $v$ variable, recalling \eqref{def-Gamma}, and noting that $\Phi^{ij}=\Phi^{ji}$, we directly write
\begin{align}\label{eq-Gammag0g0}
    \Gamma(g_0,g_0)=\nn&\big(\Phi^{ij}*(\mu^{\fr12}g_0)\big)\pr_{v_iv_j}g_0-2\big(\Phi^{ij}*(v_i\mu^{\fr12}g_0)\big)\pr_{v_j}g_0\\
    &-\big(\Phi^{ij}*(\mu^{\fr12}\pr_{v_iv_j}g_0)\big)g_0+2\big(\Phi^{ij}*(v_i\mu^{\fr12}\pr_{v_j}g_0)\big)g_0,
\end{align}
and hence for $0\le|\al|\le2$,
\begin{align}\label{eq-WZGammag0g0}
    \nn&\pr_v^{\al}Z^{m-n+\beta-\beta'}\Gamma(g_0,g_0)\\
    =\nn&\sum_{\substack{\al'\le\al\\ \al''\le\al'}}C_{\al}^{\al'}C_{\al'}^{\al''}\sum_{\substack{n'\le m-n\\ \beta''\le\beta-\beta'}}C_{\beta-\beta'}^{\beta''}C_{m-n}^{n'}\sum_{\substack{n''\le n'\\\beta'''\le\beta''}}C^{\beta'''}_{\beta''}C^{n''}_{n'}\\
    \nn&\times\Phi^{ij}*\big(\pr_v^{\al''}Z^{n''+\beta'''}\mu^{\fr12}\pr_v^{\al'-\al''}Z^{n'-n''+\beta''-\beta'''}g_0\big)\pr_{v}^{\al-\al'}Z^{m-n-n'+\beta-\beta'-\beta''}\pr_{v_iv_j}g_0\\
    &+{\rm similar\ terms}.
\end{align}
It is worth pointing out that at least one $v_l$ derivative hits on one of $g_0$ for each term on the right-hand side of \eqref{eq-Gammag0g0} with $l\in\{1,2,3\}$, so does for each term on the right-hand side of \eqref{eq-WZGammag0g0}. Furthermore, the fact $\iota\ge1$ under consideration enables us to use \eqref{coercive} to bound one of the $g_0$ terms, which has at least one $v_l$ derivative. Moreover precisely,
similar to \eqref{est-Gamma1s}--\eqref{est-convolution}, we get
\begin{align*}
    &\big|\Phi^{ij}*\big(\pr_v^{\al''}Z^{n''+\beta'''}\mu^{\fr12}\pr_v^{\al'-\al''}Z^{n'-n''+\beta''-\beta'''}g_0\big)\big|\\
    \le&C \fr{\Gamma_s(|n''|)}{C_{|n''|}^{n''}}\la v\ra^{-1}\big\|\mu^{\fr18}\pr_v^{\al'-\al''}Z^{n'-n''+\beta''-\beta'''}g_0\big\|_{L^2_v},
\end{align*}
and analogous inequalities hold for all the rest convolutions on the right-hand side of \eqref{eq-WZGammag0g0}. Thus,
\begin{align*}
    &\left\la  Z^{n+\beta'}q(\phi)_{\ne}\pr_{v}^\al Z^{m-n+\beta-\beta'}\Gamma(g_0,g_0),Wg^{(m+\beta)}_{\iota\ne}\right\ra_{\ell-2|\al|-2\iota}\\
    \le&C\sum_{\substack{n'\le m-n\\ \beta''\le\beta-\beta'}}C_{m-n}^{n'}\sum_{\substack{n''\le n'\\\beta'''\le\beta''}}C^{n''}_{n'}\fr{\Gamma_s(|n''|)}{C_{|n''|}^{n''}} \|Z^{n+\beta'}q(\phi)\|_{L^2_x}\big\|W g_{\iota\ne}^{(m+\beta)}\big\|_{L^2_{x,v}(\ell-2|\al|-2\iota)}\\
    &\times\sum_{\substack{\al'\le\al\\ \al''\le\al'}}\Big[\big\|\mu^{\fr18}\pr_{v}^{\al'-\al''}Z^{n'-n''+\beta''-\beta'''}g_0\big\|_{L^2_v}\big\|Z^{m-n-n'+\beta-\beta'-\beta''}\nb_v^2\pr_v^{\al-\al' }g_0\big\|_{L^2_v(\ell-2|\al|-2\iota-1)}\\
    &+\big\|\mu^{\fr18}\pr_{v}^{\al'-\al''}Z^{n'-n''+\beta''-\beta'''}g_0\big\|_{L^2_v}\big\|Z^{m-n-n'+\beta-\beta'-\beta''}\nb_v\pr_v^{\al-\al' }g_0\big\|_{L^2_v(\ell-2|\al|-2\iota-1)}\\
    &+\big\|\mu^{\fr18}\pr_v^{\al'-\al''}Z^{n'-n''+\beta''-\beta'''}\nb_v^2g_0\big\|_{L^2_v}\big\|Z^{m-n-n'+\beta-\beta'-\beta''}\pr_v^{\al-\al'} g_0\big\|_{L^2_v(\ell-2|\al|-2\iota-1)}\\
    &+\big\|\mu^{\fr18}\pr_v^{\al'-\al''}Z^{n'-n''+\beta''-\beta'''} \nb_vg_0\big\|_{L^2_v}\big\|Z^{m-n-n'+\beta-\beta'-\beta''}\pr_v^{\al-\al' }g_0\big\|_{L^2_v(\ell-2|\al|-2\iota-1)}\Big]\\
    \le&C\sum_{\substack{n'\le m-n\\ \beta''\le\beta-\beta'}}C_{m-n}^{n'}\sum_{\substack{n''\le n'\\\beta'''\le\beta''}}C^{n''}_{n'}\fr{\Gamma_s(|n''|)}{C_{|n''|}^{n''}} \|Z^{n+\beta'}q(\phi)\|_{L^2_x}\big\|W g_{\iota\ne}^{(m+\beta)}\big\|_{L^2_{x,v}(\ell-2|\al|-2\iota)}\\
    &\times \sum_{\substack{\al'\le\al\\ \al''\le \al'}}\Big[\big\|\pr_v^{\al'-\al''}Z^{n'-n''+\beta''-\beta'''}g_0\big\|_{L^2_v(\ell-2|\al'-\al''|)}\big|Z^{m-n-n'+\beta-\beta'-\beta''}\nb_v\pr_v^{\al-\al'} g_0\big|_{\sig,\ell-2|\al-\al'|}\\
    &+\big\|\pr_v^{\al'-\al''}Z^{n'-n''+\beta''-\beta'''} g_0\big\|_{L^2_v(\ell-2|\al'-\al''|)}\big|Z^{m-n-n'+\beta-\beta'-\beta''}\pr_v^{\al-\al'} g_0\big|_{\sig,\ell-2|\al-\al'|}\\
    &+\big|\pr_v^{\al'-\al''}Z^{n'-n''+\beta''-\beta'''}\nb_v g_0\big|_{\sig,\ell-2|\al'-\al''|}\big\|Z^{m-n-n'+\beta-\beta'-\beta''}\pr_v^{\al-\al'} g_0\big\|_{L^2_v(\ell-2|\al-\al'|)}\\
    &+\big|\pr_v^{\al'-\al''}Z^{n'-n''+\beta''-\beta'''} g_0\big|_{\sig,\ell-2|\al'-\al''|}\big\|Z^{m-n-n'+\beta-\beta'-\beta''}\pr_v^{\al-\al'}g_0\big\|_{L^2_v(\ell-2|\al-\al'|)}\Big].
\end{align*}
Consequently,
\begin{align*}
&\nu w_{\iota}(t)\sum_{m\in\N^6,|\beta|\le N}\kappa^{2|\beta|}a_{m,\lm,s}^2(t)\sum_{\substack{n\le m,\beta'\le\beta\\|\beta'|>|\beta|/2}}C_{\beta}^{\beta'}C_m^n\\
&\times\left\la  Z^{n+\beta'}q(\phi)_{\ne}\pr_{v}^\al Z^{m-n+\beta-\beta'}\Gamma(g_0,g_0),Wg^{(m+\beta)}_{\iota\ne}\right\ra_{\ell-2|\al|-2\iota}\\
\le&C\nu w_{\iota}(t)\sum_{m\in\N^6,|\beta|\le N}\kappa^{2|\beta|}\sum_{\substack{n\le m,\beta'\le\beta\\|\beta'|>|\beta|/2}}\sum_{\substack{n'\le m-n\\ \beta''\le\beta-\beta'}}\sum_{\substack{n''\le n'\\\beta'''\le\beta''}}b_{m,n,n',n'',s}\sum_{\substack{\al'\le\al\\ \al''\le\al'}}\\
&\times \Big(a_{n,\lm,s}(t)\|Z^{n+\beta'}q(\phi)\|_{L^2_x}\Big)\lm^{|n''|}(t)\\
&\times\bigg[\Big(a_{n'-n'',\lm,s}(t)\big\|Z^{n'-n''+\beta''-\beta'''}\pr_v^{\al'-\al''}g_0\big\|_{L^2_v(\ell-2|\al'-\al''|)}\Big)\\
    &\times a_{m-n-n',\lm,s}(t)\big|Z^{m-n-n'+\beta-\beta'-\beta''}\la \nb_v\ra\pr_v^{\al-\al'} g_0\big|_{\sig,\ell-2|\al-\al'|}\\
    &+\Big(a_{n'-n'',\lm,s}(t)\big|Z^{n'-n''+\beta''-\beta'''}\la\nb_v\ra \pr_v^{\al'-\al''}g_0\big|_{\sig,\ell}\Big)\\
    &\times a_{m-n-n',\lm,s}(t)\big\|Z^{m-n-n'+\beta-\beta'-\beta''}\pr_v^{\al-\al'} g_0\big\|_{L^2_v(\ell-2|\al-\al'|)}\bigg]\\
    &\times a_{m,\lm,s}(t)\big\|W g_{\iota\ne}^{(m+\beta)}\big\|_{L^2_{x,v}(\ell-2|\al|-2\iota)}\\
    \le&C\nu\Big(w_{\iota}(t)\|q(\phi)\|_{\mathcal{G}^{\lm,N}_s}\Big)\|W g_{\iota\ne}\|_{\mathcal{G}^{\lm,N}_{s,\ell-2|\al|-2\iota}}\sum_{\substack{\al'\le\al\\ \al''\le\al'}}\Big(\|\pr_v^{\al'-\al''}g_0\|_{\mathcal{G}^{\lm,\fr{N}{2}}_{s,\ell-2|\al'-\al''|}}\|\la\nb_v\ra\pr_v^{\al-\al'} g_0\|_{\mathcal{G}^{\lm,\fr{N}{2}}_{s,\sig,\ell-2|\al-\al'|}}\\
    &+\|\la\nb_v\ra\pr_v^{\al'-\al''} g_0\|_{\mathcal{G}^{\lm,\fr{N}{2}}_{s,\sig,\ell-2|\al'-\al''|}}\|\pr_v^{\al-\al'} g_0\|_{\mathcal{G}^{\lm,\fr{N}{2}}_{s,\ell-2|\al-\al'|}}\Big).
\end{align*}
Similarly, 
\begin{align*}
&\nu w_{\iota}(t)\sum_{m\in\N^6,|\beta|\le N}\kappa^{2|\beta|}a_{m,\lm,s}^2(t)\sum_{\substack{n\le m,\beta'\le\beta\\|\beta'|>|\beta|/2}}C_{\beta}^{\beta'}C_m^n\\
&\times\left\la  Z^{n+\beta'}W q(\phi)_{\ne} Z^{m-n+\beta-\beta'}\Gamma(g_0,g_0),Wg^{(m+\beta)}_{\iota\ne}\right\ra_{\ell-2|\al|-2\iota}\\
    \le&C\nu\Big(w_{\iota}(t)\|\nb_x\phi\|_{\mathcal{G}^{\lm,N}_s}\Big)\|W g_{\iota\ne}\|_{\mathcal{G}^{\lm,N}_{s,\ell-2|\al|-2\iota}}\Big(\|g_0\|_{\mathcal{G}^{\lm,\fr{N}{2}}_{s,\ell}}\|\la\nb_v\ra g_0\|_{\mathcal{G}^{\lm,\fr{N}{2}}_{s,\sig,\ell}}\Big).
\end{align*}
It follows from the above two inequalities that
\begin{align}\label{est-NGamma-iota-0-HL}
    \nn&\sum_{j=1}^2\Big(N[\Gamma]^{0;{\rm HL}}_{\iota;j),1}+N[\Gamma]^{0;{\rm HL}}_{\iota;j),2}\Big)\\
    \le& C\Big(w_{\iota}(t)\|\rho\|_{\mathcal{G}^{\lm,N}_s}\Big)\mathcal{E}_{\ell}(g(t))^{\fr12}\Big(\nu^{\fr23} \mathcal{D}_\ell(g(t))^{\fr12}\Big)\Big( w_{\iota}(t)\mathcal{E}_{\ell-2\iota}(g_{\ne}(t))^{\fr12}\Big).
\end{align}
Collecting the estimates \eqref{est-NGamma01}--\eqref{est-NGamma-iota-0-LH}, and \eqref{est-NGamma-iota-0-HL}, we obtain \eqref{est-nonlinear-collision}.
\end{proof}

Now  we are in a position to improve the bootstrap hypothesis \eqref{bd-en}.
\begin{prop}\label{prop-g-nonlinear}
    Assume that \eqref{res-final-1}--\eqref{restriction: lm-kappa-nu}  hold. Under the bootstrap hypotheses \eqref{bd-en}--\eqref{H-phi}, for $t\in[0,T]$ with $T\le\nu^{-\fr12}$, we have
    \begin{align}\label{improve-H-en}
        \mathcal{E}(g(t))+\int_0^T\fr14\nu^{\fr13}\mathcal{D}(g(t))+\mathcal{CK}(g(t))dt\le \mathcal{E}(g(0))+\fr{1}{50}{C_g}\eps^2+C\eps^3.
    \end{align}
\end{prop}

\begin{proof}
We first deal with the linear terms. Clearly,
\begin{align}\label{L1-rho}
\int_0^T\|(\rho, M)\|_{\mathcal{G}^{\lm,N}_s}dt\les \big\|\la t\ra^{\fr{3+a-s}{2}}(\rho,M)\big\|_{L^2_t\mathcal{G}^{\lm,N}_s}.
\end{align}
Then by \eqref{bd-rho-1} and bootstrap hypothesis \ref{bd-en}, we have
\begin{align}\label{est-g-linear1}
    C{\rm B}_0\int_0^T \|E\|_{L^2_x} \|g\|_{L^2_{x,v}}dt\le\nn&C\sqrt{{\rm B}_0} \sup_t\mathcal{E}(g(t))^{\fr12}\int_0^T\big\|\rho\big\|_{L^2_x}dt\\
    \le& C\sqrt{{\rm B}_0}\sqrt{{\rm C}_g}\sqrt{{\rm C}_0}\eps^2\le \fr{1}{100}C_g\eps^2,
\end{align}
provided ${\rm C}_0\ll \fr{1}{{\rm B}_0}{\rm C}_g$. Moreover, by  \eqref{bd-rho-1}, \eqref{bd-rho-2}, and \eqref{decay-rho-recursion}, we infer from \eqref{est-Lmu} in Lemma \ref{lem-Linear contributions} that
\begin{align}\label{est-g-linear2}
    \nn&2\sum_{\iota=0}^{[\ell/3]+1}\int_0^T\Big[L[\mu]_{\iota;1)}+L[\mu]_{\iota;2)}\Big]dt\\
    \le\nn& C\sqrt{{\rm A}_0}\sup_{t\le T}\mathcal{E}_{\ell}(g(t))^{\fr12}\int_0^T\|\rho\|_{\mathcal{G}^{\lm,N}_{s}}dt\\
    \nn&+C\sqrt{{\rm A}_0}\sum_{\iota=1}^{[\ell/3]+1}\Bigg[\Big(\int_0^Tw_{\iota}^2(t)\mathcal{CK}_{\ell-2\iota}(g_{\ne}(t))dt\Big)^{\fr12}+\sup_t\Big(w^2_{\iota}(t)\mathcal{E}_{\ell-2\iota}(g_{\ne}(t))\Big)^{\fr12}\Bigg]\\
    \nn&\times \kappa^{-1}\big\|\langle t\rangle^{\frac{1+a-s}{2}}w_{\iota}(t)\rho\big\|_{L^2_t\mathcal{G}_s^{\lambda,N}}\\
    \le&C\kappa^{-1}\sqrt{{\rm A}_0}\sqrt{{\rm C}_{g}{\rm C}_{0}} \eps^2\le \fr{1}{100}{\rm C}_g\eps^2,
\end{align}
as long as ${\rm C}_0\ll \fr{\kappa^{2}}{{\rm A}_0}{\rm C}_g$.

Next we turn to estimate the nonlinear terms. Using \eqref{pt-phi}, we find that
\begin{align*}
    \|\pr_t\phi\|_{L^\infty_x}\les\big\||\nb_x|^2\pr_t\phi\big\|_{L^2_x}\les\|\nb_x\cdot M\|_{L^2_x}\les \kappa^{-1}\|M\|_{\mathcal{G}^{\lm,N}_s}.
\end{align*}
Combining this with \eqref{L1-rho} and \eqref{Linfty-rho}, we are led to
\begin{align}\label{est-g-nonlinear1}
    \nn& \int_0^T C\nu {\rm A}_0\|E\|_{L^2_x}^4+C \fr{\sqrt{{\rm B}_0}}{{\rm A}_0}\big(\sqrt{{\rm B}_0}\|g\|_{L^2_{x,v}}\big)\big(\nu {\rm A}_0\|g\|_{\sig}^2\big)+{\rm B}_0\|\pr_t\phi\|_{L^\infty_x}\|g\|_{L^2_{x,v}}^2dt\\
    \les\nn&\|\rho\|_{L^\infty_tL^2_x}^3\int_0^T\|\rho\|_{L^2_x}dt+\sup_t\mathcal{E}(g(t))^{\fr12}\Big(\nu^{\fr13}\int_0^T\mathcal{D}_\ell(g(t))dt\Big)\\
    &+\sup_{t}\mathcal{E}(g(t))\int_0^T\|M\|_{\mathcal{G}^{\lm,N}_s}dt\les \eps^3.
\end{align}

For the rest terms on the right hand side of  \eqref{en-ineq-Linear-g}, using \eqref{bd-rho-1}, \eqref{bd-rho-2}, and \eqref{Linfty-rho} again, we infer from  Lemmas \ref{lem-pr_tphi}--\ref{lem-nonlinear-collision} that
\begin{align}\label{est-g-nonlinear2}
     \nn&2\sum_{\iota=0}^{[\ell/3]+1}\int_0^T N[\pr_t\phi]_{\iota;1)}+N[\pr_t\phi]_{\iota;2)}dt\\
     \le \nn&C\sup_{t\le T}\mathcal{E}_{\ell}(g(t))\int_0^T\|M(t)\|_{\mathcal{G}^{\lm,N}_s}dt+C\sup_t\mathcal{E}_\ell(g(t))^{\fr12}\sum_{\iota=1}^{[\ell/3]+1}\big\|\la t\ra^{\fr{1+a-s}{2}}w_{\iota}(t)M\big\|_{L^2_t\mathcal{G}^{\lm,N}_s}\\
     &\times\Big[\Big(\int_0^T\mathcal{CK}_{\ell-2\iota}(g_{\iota\ne}(t))dt\Big)^{\fr12}+\sup_t\mathcal{E}_{\ell-2\iota}(g_{\iota\ne}(t))^{\fr12}\Big]\les\eps^3,
 \end{align}

 \begin{align}\label{est-g-nonlinear3}
     \nn&2\sum_{\iota=0}^{[\ell/3]+1}\int_0^T N[E]_{\iota;1)}+N[E]_{\iota;2)}dt\\
     \les 
     \nn&\Big(\big\|\la t\ra^{\fr{3+a-s}{2}}\rho\big\|_{L^2_t\mathcal{G}^{\lm,N}_s}+\max_{1\iota\le[\ell/3]+1}\big\|\la t\ra^{\fr{1+a-s}{2}}w_{\iota}(t)\rho\big\|_{L^2_t\mathcal{G}^{\lm,N}_s}\Big)\\
     \nn&\times\Big[\Big(\int_0^T\mathcal{CK}(g(t))dt\Big)^{\fr12}+\sup_{t\le T}\mathcal{E}(g(t))^{\fr12}\Big]\sup_{t\le T}\mathcal{E}(g(t))^{\fr12} \\
     &+\|\nb_x\rho\|_{L^\infty_t\mathcal{G}^{\lm,N}_s}\Big[\int_0^T\nu^{\fr13}\mathcal{D}(g(t))+\mathcal{CK}(g(t))dt+\sup_{t\le T}\mathcal{E}(g(t))\Big]\les \eps^3,
 \end{align}
 and
 \begin{align}\label{est-g-nonlinear4}
     \nn&2\sum_{\iota=0}^{[\ell/3]+1}\int_0^T N[\Gamma]_{\iota;1)}+N[\Gamma]_{\iota;2)}dt\\
     \les \nn&\nu^{\fr13}\int_0^T\mathcal{D}(g(t))dt\Big(1+\|\nb_x\rho\|_{L^\infty_t\mathcal{G}^{\lm,N}_s}\Big)\sup_{t\le T}\mathcal{E}(g(t))^{\fr12}\\
     &+\max_{1\le\iota\le[\ell/3]+1}\big\|\la t\ra^{\fr{1+a-s}{2}}w_{\iota}(t)\rho\big\|_{L^2_t\mathcal{G}^{\lm,N}_s}\Big(\nu^{\fr13}\int_0^T\mathcal{D}(g(t))dt\Big)^{\fr12}\sup_{t\le T}\mathcal{E}(g(t))\les \eps^3.
 \end{align}

 Now recalling \eqref{en-ineq-Linear-g}, collecting the estimates in \eqref{est-g-linear1}--\eqref{est-g-nonlinear4}, we obtain \eqref{improve-H-en}.
\end{proof}

\section{Linear estimates}\label{sec: Linear estimates}
In this section, we consider the linearized inhomogeneous Landau equation for $|k|\neq 0$:
\begin{align}\label{eq:linear-Landau}
    \begin{cases}
        \partial_t f+ik \cdot v f+\nu L f=0,\\
        f|_{t=0}=f_{\mathrm{in}}(v).
    \end{cases}
\end{align}
\subsection{Semigroup estimates and enhanced dissipation}\label{sec-semigroup est}
For $k\in\Z^3_*$, let us denote $f^{(n)}=\tl{Y}^nf$ with  $n\in \mathbb{N}^3$. Then similar to \eqref{eq:E_l^n}--\eqref{dissipation-l}, we introduce
\begin{align}\label{eq:varE_l^n}
\begin{aligned}
    \mathfrak{E}_{k,\ell}^{n}(f(t))=&{\rm A}_0\sum_{|\al|\le1}\left\|\la v\ra^{\ell-2|\al|}k^{\al}f^{(n)}\right\|_{L^2_{v}}^2+\kappa^2\nu^{\fr23}\left\|\la v\ra^{\ell-2}\nb_vf^{(n)}\right\|_{L^2_{v}}^2\\
    &+2\kappa\nu^\fr13\mathrm{Re}\,\left\la ik f^{(n)},\la v\ra^{2\ell-4}\nb_{v}f^{(n)} \right\ra_{v},
\end{aligned}
\end{align}

\begin{align}\label{def-varD-nl}
    \mathfrak{D}^{n}_{k,\ell}(f(t))={\rm A}_0\nu^{\fr23}\sum_{|\al|\le1}\left|k^\al f^{(n)}\right|_{\sig,\ell-2|\al|}^2+{  \kappa}|k|^2\left\|\la v\ra^{\ell-2}f^{(n)}\right\|_{L^2_{v}}^2+\kappa^2\nu^{\fr43}\left|\nb_vf^{(n)}\right|_{\sig,\ell-2}^2,
\end{align}

\begin{align}\label{def-varE}
    \mathsf{E}_{k,\ell}(f(t)):=\sum_{m\in\N^3}\sum_{\substack{\beta\in\N^3, |\beta|\le N}}\kappa^{2|\beta|}a_{m,\lm_L,s_L}^2\mathfrak{E}_{k,\ell}^{m+\beta}(f(t)).
\end{align}

\begin{align}\label{def-varD}
\mathsf{D}_{k,\ell}(f(t)):=\sum_{m\in\N^3}\sum_{\substack{\beta\in\N^3, |\beta|\le N}}\kappa^{2|\beta|}a_{m,\lm_L,s_L}^2\mathfrak{D}_{k,\ell}^{m+\beta}(f(t)),
\end{align}
with $a_{m,\lm_L,s_{L}}=\frac{\lm_L^{|m|}}{\Gamma_{s_L}(|m|)}C_{|m|}^m$, and $0< s_L<\fr23$.  

Moreover, similar to \eqref{total-en}--\eqref{total-CK}, we define 
\begin{align}\label{def-total-en}
    \mathsf{E}_{k,\ell}^{w}(f(t))=\sum_{\iota=0}^{[\ell/3]+1}w_{\iota}^2(t)\mathsf{E}_{k,\ell-2\iota}(f(t)),
\end{align}

\begin{align}
    \mathsf{D}^{w}_{k,\ell}(f(t))=\sum_{\iota=0}^{[\ell/3]+1}w_{\iota}^2(t)\mathsf{D}_{k,\ell-2\iota}(f(t)).
\end{align}

\begin{prop}\label{prop-semigroup}
For any $s_L\in (\fr12,\fr23)$, there is $\lambda_L$ such that the solution $f$ to \eqref{eq:linear-Landau} satisfies
    \begin{align}\label{semigroup-est}
        \mathsf{E}^{w}_{k,\ell}(f(t))+\fr14\nu^{\fr13}\int_0^t\mathsf{D}^{w}_{k,\ell}(f(t'))dt
        \les \mathsf{E}^{w}_{k,\ell}(f(0)).
    \end{align}
\end{prop}  
\begin{proof}
The energy inequality \eqref{semigroup-est} can be achieved by following the treatments of the linear part of Proposition \ref{prop-g} line by line. Here, we are considering a simpler case: the mode $k\in\Z^3_*$ is fixed, and hence the zero mode of $f$ will never be  involved when dealing with the lower order errors stemming from the coercive estimates for the Landau operator, see the last term of \eqref{est-low1}. We sketch the proof below.

Note first that for $m, \beta\in\N^3$, $f^{(m+\beta)}=\tl{Y}^{m+\beta}$ solves 
\begin{align*}
    \partial_t f^{(m+\beta)}+ik \cdot v f^{(m+\beta)}+\nu \tl{Y}^{m+\beta}L f=0.
\end{align*}
We denote $f_{[\iota]}=w_{\iota}(t)f$ for $\iota=0, 1, 2,\cdots, [\ell/3]+1$. In particular, $f_{[0]}=f$. Then similar to \eqref{en-pr_xg}, \eqref{en-pr_vg} and \eqref{en-cross}, we are led to
\begin{align}
  \nn&\fr{1}{2}\fr{d}{dt}\big\|\la v\ra^{\ell-2|\al|-2\iota}k^{\al}f_{[\iota]}^{(m+\beta)}\big\|_{L^2_{v}}^2+\nu\mathrm{Re}\,\big\la k^\al\tl{Y}^{m+\beta}L f_{[\iota]},\la v\ra^{2\ell-4|\al|-4\iota}k^\al f_{[\iota]}^{(m+\beta)} \big\ra_v\\
  =&\fr{\iota\kappa_0\nu^{\fr13}}{2}(\kappa_0\nu^{\fr13}(1+t))^{-1}\big\|\la v\ra^{\ell-2|\al|-2\iota}k^{\al}f_{[\iota]}^{(m+\beta)}\big\|_{L^2_{v}}^2,
    \end{align}

    \begin{align}
    \nn&\fr{1}{2}\fr{d}{dt}\big\|\la v\ra^{\ell-2-2\iota}\nb_vf_{[\iota]}^{(m+\beta)}\big\|_{L^2_{v}}^2+\mathrm{Re}\,\big\la ikf^{(m+\beta)}_{[\iota]},\la v\ra^{2\ell-4-4\iota} \nb_vf^{(m+\beta)}_{[\iota]}\big\ra_{v}\\
    \nn&+\nu\mathrm{Re}\,\big\la\tl{Y}^{m+\beta}\nb_vLf_{[\iota]},\la v\ra^{2\ell-4-4\iota}\nb_vf_{[\iota]}^{(m+\beta)} \big\ra_v\\
    =&\fr{\iota\kappa_0\nu^{\fr13}}{2}(\kappa_0\nu^{\fr13}(1+t))^{-1}\big\|\la v\ra^{\ell-2-2\iota}\nb_vf_{[\iota]}^{(m+\beta)}\big\|_{L^2_{v}}^2,
    \end{align}
and
    \begin{align}
    \nn&\fr{d}{dt}\mathrm{Re}\,\big\la ik f_{[\iota]}^{(m+\beta)},\la v\ra^{2\ell-4-4\iota}\nb_{v}f_{[\iota]}^{(m+\beta)} \big\ra_{v}+|k|^2\big\|\la v\ra^{\ell-2-2\iota}f_{[\iota]}^{(m+\beta)}\big\|_{L^2_v}^2
    \\
    =\nn&\iota\kappa_0\nu^{\fr13}(\kappa_0\nu^{\fr13}(1+t))^{-1}\mathrm{Re}\,\big\la ik f_{[\iota]}^{(m+\beta)},\la v\ra^{2\ell-4-4\iota}\nb_{v}f_{[\iota]}^{(m+\beta)} \big\ra_{v}\\
    &-\nu\mathrm{Re}\,\left(\big\la ik\tl{Y}^{m+\beta}Lf_{[\iota]},\la v\ra^{2\ell-4-4\iota}\nb_vf_{[\iota]}^{(m+\beta)}\big\ra_v+\big\la ikf_{[\iota]},\la v\ra^{2\ell-4-4\iota}\nb_v\tl{Y}^{m+\beta}Lf_{[\iota]}\big\ra_v\right).
    \end{align}
For the sake of presentation, we adapt the notations introduced in \eqref{est-dissip1}, \eqref{est-dissip2} and \eqref{def-dissip4}. More precisely,  we denote
\begin{align}
    {\rm Dis}_{k;1}(f_{[\iota]})=\nn&\nu {\rm A}_0\sum_{m\in\N^3}\sum_{\substack{\beta\in\N^3\\ |\beta|\le N}}\kappa^{2|\beta|}a_{m,\lm_L,s_L}^2\sum_{|\al|\le1}\mathrm{Re}\,\left\la \tl{Y}^{m+\beta} Lk^\al f_{[\iota]},k^\al f_{[\iota]}^{(m+\beta)}\la v\ra^{2\ell-4|\al|-4\iota}\right\ra_{v},
\end{align}
and the remaining terms ${\rm Dis}_{k;2}(f_{[\iota]})$ and ${\rm Dis}_{k;4}(f_{[\iota]})$ can be obtained by analogous modifications of ${\rm Dis}_{2}(g_{\iota*})$ and ${\rm Dis}_{4}(g_{\iota*})$. As mentioned above, the zero mode is excluded from the lower-order contributions because $k\ne0$. Thus, similar to \eqref{lower-Dis} and \eqref{est-low3}, we arrive at
\begin{align}\label{lb-dis}
    \nn&{\rm Dis}_{k;1}(f_{[\iota]})+{\rm Dis}_{k;2}(f_{[\iota]})+{\rm Dis}_{k;4}(f_{[\iota]})+\kappa\nu^{\fr13}|k|^2\|f_{[\iota]}\|_{\tl{\mathcal{G}}^{\lm_L,N}_{s_L,\ell}}^2\\
    \nn&+\kappa^2\nu^{\fr23}\sum_{m\in\N^3}\sum_{\substack{\beta\in\N^3\\ |\beta|\le N}}\kappa^{2|\beta|}a_{m,\lm_L,s_L}^2\mathrm{Re}\,\big\la ikf^{(m+\beta)}_{[\iota]},\la v\ra^{2\ell-4-4\iota} \nb_vf^{(m+\beta)}_{[\iota]}\big\ra_{v}\\
    \ge& \fr{\nu^{\fr13}}{4}\mathsf{D}_{k,\ell-2\iota}(f_{[\iota]}).
\end{align}
On the other hand, similar to \eqref{dt-(1+t)}, we have
\begin{align}\label{dt-(1+t)-var}
    \nn&\sum_{\iota=0}^{[\ell/3]+1}\fr{\iota \kappa_0\nu^{\fr13}}{2}\big(\kappa_0\nu^{\fr13}(1+t)\big)^{-1}\sum_{m\in\N^3}\sum_{\substack{\beta\in\N^3\\ |\beta|\le N}}\kappa^{2|\beta|}a_{m,\lm_L,s_L}^2\\
    \nn&\times\bigg({\rm A}_0\sum_{|\al|\le1}\big\|k^{\al}f_{[\iota]}^{(m+\beta)}\la v\ra^{\ell-2|\al|-2\iota}\big\|_{L^2_{v}}^2\\
    \nn&+\kappa^2\nu^{\fr23}\big\|\nb_vf_{[\iota]}^{(m+\beta)}\la v\ra^{\ell-2-2\iota}\big\|_{L^2_{v}}^2+2\kappa\nu^{\fr13}\mathrm{Re}\,\left\la ikf_{[\iota]}^{(m+\beta)},\la v\ra^{2\ell-4-4\iota}\nb_{v}f_{[\iota]}^{(m+\beta)} \right\ra_{v}\bigg)\\
    \le&\fr{\nu^{\fr13}}{8}\sum_{\iota=0}^{[\ell/3]} \mathsf{D}_{k,\ell-2\iota}(f_{[\iota]}(t)).
\end{align}
Collecting the above estimates, we conclude that \eqref{semigroup-est} holds.
\end{proof}

\subsection{Linear density estimates}
First we rewrite \eqref{pVPL} as follows
\be\label{l-VPL}
\begin{cases}
\pr_tf+v\cdot\nb_xf-2(E(t,x)\cdot v)\mu^\fr12+\nu Lf=\mathfrak{N}(t,x,v),\\[2mm]
E(t,x)=-\nb_x(-\Dl_x)^{-1}\rho,\\[2mm]
\rho=\displaystyle\int_{\R^3}f(t,x,v)\mu^\fr12(v)dv.
\end{cases}
\ee
where 
\begin{align*}
\mathfrak{N}(t,x,v)=-E(t,x)\cdot\nb_vf+E(t,x)\cdot v f+\nu\Gamma(f,f).
\end{align*}
To get the equation for the density, taking the Fourier transform in $x$ of \eqref{l-VPL}, we get 
\be\label{l-VPL-k}
\begin{cases}
\pr_t\hat{f}_k+ik\cdot v\hat{f}_k-2\hat{E}_k(t)\cdot v\mu^\fr12+\nu L\hat{f}_k=\hat{\mathfrak{N}}_k(t,v),\\[2mm]
\hat{E}_k(t)=-ik|k|^{-2}\hat{\rho}_k(t),\\[2mm]
\hat{\rho}_k(t)=\displaystyle\int_{\R^3}\hat{f}_k(t,v)\mu^\fr12(v)dv.
\end{cases}
\ee
Now we treat $-2\hat{E}_k(t)\cdot v\mu^\fr12$ as a source term as well, since it is not written in terms of $\hat{f}_k(t,v)$ but in terms of $\hat{\rho}_k$ directly. Thus, the first equation of \eqref{l-VPL-k} can be rewritten as
\be\label{eq-f_k}
\pr_t\hat{f}_k+ik\cdot v\hat{f}_k+\nu L\hat{f}_k=-2ik|k|^{-2}\cdot v\mu^\fr12(v)\hat{\rho}_k(t)+\hat{\mathfrak{N}}_k(t,v).
\ee
A key step to derive the equation for the density $\hat{\rho}_k(t)$ is to solve the above equation, i.e., to get the representation of $\hat{f}_k(t,v)$ in terms of the initial data  and the source terms. This leads to the inhomogeneous linear Landau equation written mode-by-mode in $x$
\be\label{Landau}
\pr_th+ik\cdot vh+\nu Lh=0,\quad h|_{t=0}=h_{\rm in}.
\ee
Let $S_k(t)$ be the semigroup associated to \eqref{Landau}, that is, the unique solution $h(t,v)$ to \eqref{Landau} is given by
\be
h(t,v)=S_k(t)[h_{\rm in}(v)].
\ee

It is a standard property of semigroups that $[S_k(t),ik\cdot v \mathrm{Id}+\nu L]=0$. 
Now we define the associated dual semigroup $S^*_k(t)$, namely, 
\begin{align*}
    \langle S_k(t)[\varphi_1], \varphi_2\rangle=\langle \varphi_1(v), S^*_k(t)[\varphi_2]\rangle.
\end{align*}
Then, we have that $\Phi_k(t,v)=S_k^*(t)[\varphi_2]$ solves 
\begin{align*}
    \pr_t\Phi_k-ik\cdot v\Phi_k+\nu L\Phi_k=0,\quad \Phi_k|_{t=0}=\varphi_2. 
\end{align*}

By Duhamel's principle, the solution $\hat{f}_k(t)$ to \eqref{eq-f_k} is then of the form
\be\label{f_k}
\begin{aligned}
\hat{f}_k(t)=&S_k(t)[\hat{f}_{\rm in}(v)]-\int_0^t2ik|k|^{-2}\hat{\rho}_k(\tau) \cdot S_k(t-\tau)[v\mu^{\fr12}(v)] d\tau\\
&+\int_0^tS_k(t-\tau)[\hat{\mathfrak{N}}_k(\tau,v)]d\tau.
\end{aligned}
\ee
Now we  are in a position to obtain the equation for $\hat{\rho}_k(t)$.  Multiplying \eqref{f_k} by $\mu(v)^{\fr12}$, and integrating the resulting equation with respect to $v$ over $\R^3$, we are led to
\be\label{rho_k}
\hat{\rho}_k(t)+\int_0^t K_k(t-\tau) \hat{\rho}_k(\tau)d\tau=\mathcal{N}_k(t),
\ee
where
\be\label{K_k}
K_k(t)=\fr{2ki}{|k|^2}\int_{\R^3}S_k(t)[v\mu^{\fr12}(v)]\mu^{\fr12}(v) dv.
\ee
and
\be\label{N_k}
\mathcal{N}_k(t)=\int_{\R^3}S_k(t)[\hat{f}_{\rm in}(v)]\mu^{\fr12}(v)dv+\int_0^t\int_{\R^3}S_k(t-\tau)[\hat{\mathfrak{N}}_k(\tau,v)]\mu^{\fr12}(v)dvd\tau.
\ee
To solve the Volterra equation \eqref{rho_k}, we shall  take the Fourier-Laplace transform of \eqref{rho_k} with respect to the time variable $t$, and then check the Penrose stability condition to ensure the solvability. In particular, some  pointwise estimates on $K_k(t)$ are needed. 

\subsection{Estimates of the kernel $K_k(t)$}
The basic idea is to get the  pointwise estimates of $K_k(t)$ by means of  energy estimates of $\fr{2ki}{|k|^2}S_k(t)[v\mu^{\fr12}(v)]$. We have the following lemma.

\begin{lem}\label{lem-point-K_k}
There exist   constants $0<\tl{c}<\fr12$ and $C>0$ depending only on $s_L$,  $\lambda_L$ and $n\in\N$, such that  
\begin{align}\label{point-K_k}
     |\pr_t^nK_k(t)|\le C |k|^{n-1}\min\left\{e^{-\fr{\tl c}{2}(\nu^{\fr13}t)^{\fr13}}, e^{-\fr{\tl c}{2}(\nu t)^{\fr23}}\right\}e^{-\fr{\tl{c}}{2}|kt|^{s_L}}. 
\end{align}
\end{lem}
\begin{proof}
For any fixed $n\in\N$, by using the multinomial theorem, we have
\begin{align}\label{multinomial}
    |\xi+kt|^{2n}=\Big(\sum_{1\le j\le 3}(\xi_j+k_jt)^2\Big)^n=\sum_{|\al|=n}C_n^{\al}(\xi+kt)^{2\al},
\end{align}
where $\al=(\al_1, \al_2,\al_3)$.
Consequently, thanks to the inequality $(n!)^22^n\le (2n)!$, we find that
\begin{align*}
e^{2c|\xi+kt|^{s_L}}
    =&\sum_{n\in\N}\fr{(2\tl c|\xi+kt|^{s_L})^{2n}}{(2n)!}+\sum_{n\in\N}\fr{(2\tl c|\xi+kt|^{s_L})^{2n+1}}{(2n+1)!}\\
    \le&(1+2\tl c|\xi+kt|^{s_L})\sum_{n\in\N}\left(\fr{\Big((2\tl c)^{\fr{1}{s_L}}|\xi+kt|\Big)^{2n}}{(n!)^{\fr{2}{s_L}}}\right)^{s_L}\fr{1}{2^{n}}\\
    \le&(1+2\tl c|\xi+kt|^{s_L})\left(\sum_{n\in\N}\fr{\Big((2\tl c)^{\fr{1}{s_L}}|\xi+kt|\Big)^{2n}}{(n!)^{\fr{2}{s}}}\right)^{s_L}\left(\sum_{n\in\N}\fr{1}{2^{\fr{n}{1-s_L}}}\right)^{1-s_L}\\
    \les&\la \xi+kt\ra^{s_L}\sum_{n\in\N}\fr{\Big((2\tl c)^{\fr{1}{s_L}}|\xi+kt|\Big)^{2n}}{(n!)^{\fr{2}{s_L}}}=\la \xi+kt\ra^{s_L}\sum_{n\in\N}\fr{(2\tl c)^{\fr{2n}{s_L}}\sum_{|\al|=n}C_n^{\al}(\xi+kt)^{2\al}}{(n!)^{\fr{2}{s_L}}}\\
    \le&\la \xi+kt\ra^{s_L}\sum_{n\in\N}\sum_{|\al|=n}\left(\fr{(2\tl c)^{\fr{n}{s_L}}}{(n!)^{\fr{1}{s_L}}(n+1)^{-12}}C_n^{\al}\right)^2(i\xi+ikt)^{\al}(-i\xi-ikt)^{\al}\\
    \les&\la \xi+kt\ra^{s_L}\sum_{m\in\N^3}\left(\fr{\lm_L^{|m|}}{\Gamma_{s_L}(|m|)}C_{|m|}^{m}\right)^2(i\xi+ikt)^{m}(-i\xi-ikt)^{m},
\end{align*}
provided $\lm_L\ge (2\tl c)^{\fr{1}{s_L}}$. Then by the linear estimate \eqref{semigroup-est}, we get that 
\begin{align*}
    \left\|e^{\tl c|\xi+kt|^{s_L}}\mathcal{F}_v\big[S_k(t)[v\sqrt{\mu}]\big]\right\|_{L^{2}_{\xi}}^2
    \les \sum_{m\in\N^3}\sum_{0\le|\beta|\le1}a^2_{m,\lm_L,s_L}\left\|{Y}^{m+\beta}S_k(t)[v\sqrt{\mu}]\right\|_{L^2_v}^2\les C.
\end{align*}
Recalling \eqref{K_k} and \eqref{Landau}, we get the representation of $\pr_tK_k(t)$
\begin{align*}
    \pr_tK_k(t)=&\fr{2ki}{|k|^2}\int_{\R^3} \pr_t\left(S_k(t)[v\mu^{\fr12}(v)]\right)\mu^{\fr12}(v) dv\\
    =&-\fr{2ki}{|k|^2}\int_{\R^3} (i k\cdot v+\nu L)\left(S_k(t)[v\mu^{\fr12}(v)]\right)\mu^{\fr12}(v) dv\\
    =&-\fr{2ki}{|k|^2}\int_{\R^3} \left(S_k(t)[v\mu^{\fr12}(v)]\right)(i k\cdot v+\nu L)\mu^{\fr12}(v) dv\\
    =&-\fr{2ki}{|k|^2}\int_{\R^3} \left(S_k(t)[v\mu^{\fr12}(v)]\right)i k\cdot v\mu^{\fr12}(v) dv,
\end{align*}
due to the fact $L \mu^{\fr12}(v)=0$. Consequently, for $n\in\N$, there holds
\begin{align}\label{ptn_K}
  \pr_t^{n}K_k(t)=(-1)^n\fr{2ki}{|k|^2}\int_{\R^3}  \left(S_k(t)[v\mu^{\fr12}(v)]\right)(i k\cdot v)^n\mu^{\fr12}(v) dv.
\end{align} 
Using Plancherel's theorem, and noting that $-|\xi+kt|^{s_L}\le-|kt|^{s_L}+|\xi|^{s_L}$ for any $0<s_L\le1$, there exists a polynomial $p_n(\xi)$ of degree $n$, such that 
\begin{align}\label{point-K}
    |\pr_t^nK_k(t)|
    \nn&\lesssim \fr{1}{|k|}\left|\int_{\mathbb{R}^3}\mathcal{F}_v\big[S_k(t)[v\sqrt{\mu}]\big](\xi)\overline{\mathcal{F}_v\big[(i k\cdot v)^n\mu^{\fr12}(v)\big]}(\xi)d\xi \right|\\
    \nn&\lesssim |k|^{n-1}\int_{\mathbb{R}^3}e^{-\tl c|\xi+kt|^{s_L}}\Big|e^{\tl c|\xi+kt|^{s_L}}\mathcal{F}_v\big[S_k(t)[v\sqrt{\mu}]\big](\xi)\Big|\big| p_n(\xi)e^{-\fr12|\xi|^2}\big|d\xi \\
    \nn&\lesssim |k|^{n-1}e^{-\tl c|kt|^{s_L}}\left\|p_n(\xi)e^{-\fr12|\xi|^2}e^{\tl c|\xi|^{s_L}}\right\|_{L^2_{\xi}}\left\|e^{\tl c|\xi+kt|^{s_L}}\mathcal{F}_v\big[S_k(t)[v\sqrt{\mu}]\big]\right\|_{L^{2}_{\xi}}\\
    &\les |k|^{n-1}e^{-\tl c|kt|^{s_L}}. 
\end{align}

On the other hand, we also have by Lemma 7.5 of \cite{chaturvedi2023vlasov} that
\begin{align*}
    |\pr_t^n K_k(t)|\les  |k|^{n-1}\min\Big\{e^{-c(\nu^{\fr13}t)^{\fr13}}, e^{-c(\nu t)^{\fr23}}\Big\}. 
\end{align*}
Thus \eqref{point-K_k} follows immediately. 
\end{proof}
\begin{coro}
There exist $\dl_1\in(0,1)$ and $C_{s_L, \tl c}\ge1$ depending only on $s_L$ and $\tl{c}$, such that
\begin{align}\label{est-K_k}
  \sum_{j=0}^\infty \left\|\fr{(\dl_1 |k|t)^j\big(|k|^{\fr32}K_k(t)+|k|^{\fr12}\pr_tK_k(t)\big)}{(j!)^{\fr{1}{s_L}}}\right\|_{L^2(t\ge0)}^2+|k|^5\sum_{j=0}^\infty \left\|\fr{(\dl_1 |k|t)^jtK_k(t)}{(j!)^{\fr{1}{s_L}}}\right\|_{L^2(t\ge0)}^2\le C_{s_L,\tl c}.
  \end{align}
\end{coro}
\begin{proof}
We infer from \eqref{point-K} that
\begin{align*}
&|k|^3\left\|e^{\fr{\tl c}{2}|kt|^{s_L}}K_k(t)\right\|_{L^2(t\ge0)}^2+|k|\left\|e^{\fr{\tl c}{2}|kt|^{s_L}}\pr_tK_k(t)\right\|_{L^2(t\ge0)}^2\\
\le&C|k|\int_0^\infty e^{-\tl{c}|kt|^{s_L}}dt=C\int_0^\infty e^{-\tl{c}t^{s_L}}dt
=\fr{C}{{s_L}}\int_0^\infty\tau^{\fr{1}{{s_L}}-1} e^{-\tl{c}\tau}d\tau=\fr{C}{s_L}\tl{c}^{-\fr{1}{s_L}}\Gamma(\fr{1}{{s_L}}),
\end{align*}
and similarly,
\begin{align*}
&\left\|e^{\fr{\tl c}{2}|kt|^{s_L}}t K_k(t)\right\|_{L^2(t\ge0)}^2\le\fr{C}{|k|^2}\int_0^\infty t^2e^{-\tl{c}|kt|^{s_L}}dt=\fr{C}{|k|^5}\int_0^\infty t^2 e^{-\tl{c}t^{s_L}}dt\\
=&\fr{1}{s_L}\fr{C}{|k|^5}\int_0^\infty\tau^{\fr{3}{s_L}-1} e^{-\tl{c}\tau}d\tau=\fr{C}{|k|^5}\fr{1}{s_L}\tl{c}^{-\fr{3}{s_L}}\Gamma(\fr{3}{s_L}).
\end{align*}

For any $p\in\N$, by Stirling's formula, there exist constants $\theta_1, \theta_2\in (0,1)$, such that
\[
\fr{(pj)!}{(j!)^p}=\fr{\sqrt{2\pi p j}\left(\fr{pj}{e}\right)^{pj}e^{\fr{\theta_1}{12pj}}}{\left(\sqrt{2\pi j}\left(\fr{j}{e}\right)^{j}e^{\fr{\theta_2}{12j}}\right)^{p}}\le2 p^{pj}.
\]
Now let us take $p=\left[\fr{2}{s_L}\right]\in \N$ and $\theta'\in (0,1]$, such that $\fr{2}{s_L}=p+1-\theta'=p\theta'+(1+p)(1-\theta')$. Then for all $0<\dl'\le\dl_1$ such that $(p+1)\dl_1^{s_L}\le\fr{\tl c}{2}$, by H\"{o}lder's inequality, we have 
\begin{align}\label{holder}
    \nn&\sum_{j\ge0}\fr{(\dl'|k| t)^{2j}}{(j!)^\fr{2}{s_L}}
    \nn=\sum_{j\ge0}
    \fr{((\dl'|k| t)^{s_L})^{jp\theta'+j(1+p)(1-\theta')}}{((j!)^p)^{\theta'}((j!)^{p+1})^{1-\theta'}}\\
    \nn=& \sum_{j\ge0}\left(\fr{((\dl'|k| t)^{s_L})^{pj}}{(pj)!}\fr{(pj)!}{(j!)^p}\right)^{\theta'}\left(\fr{((\dl' |k|t)^{s_L})^{(p+1)j}}{((p+1)j)!}\fr{((p+1)j)!}{(j!)^{p+1}}\right)^{1-\theta'}\\
    \le& 2\sum_{n\ge0}\fr{\left[\big((p+1)^{\fr{1}{s_L}}\dl'|k| t\big)^{s_L}\right]^{n}}{n!}=2 e^{(p+1)\dl'^{s_L} (|k|t)^{s_L}}\le2 e^{\fr{\tl{c}}{2}|kt|^{s_L}}.
\end{align}
Consequently,
\begin{align*}
&\sum_{j=0}^\infty \left\|\fr{(\dl_1 |k|t)^j\big(|k|^{\fr32}K_k(t)+|k|^{\fr12}\pr_tK_k(t)\big)}{(j!)^{\fr{1}{s_L}}}\right\|_{L^2(t\ge0)}^2\\
\le&2\left\|e^{\fr{c}{2}|kt|^{s_L}}\big(|k|^{\fr32}K_k(t)+|k|^{\fr12}\pr_tK_k(t)\big)\right\|_{L^2(t\ge0)}^2\le C,
\end{align*}
and
\begin{align*}
&|k|^5\sum_{j\ge0}\left\|\fr{(\dl_1|k|t)^jt}{(j!)^{\fr{1}{s_L}}}K_k(t)\right\|_{L^2(t\ge0)}^2\le2|k|^5\left\|e^{\fr{c}{2}|kt|^{s_L}}tK_k(t)\right\|_{L^2(t\ge0)}^2\le C.
\end{align*}
Thus we get \eqref{est-K_k}.
\end{proof}

\subsection{The Penrose condition}\label{sec:penrose}
The Fourier-Laplace transform of $K_k(t)$ is given by
\begin{align}\label{LK_k}
    \mathcal{L}[K_k](\lm+i\om)=\int_0^\infty e^{-i(\lm+i\om) t}K_k(t)dt, \quad \lm\in\mathbb{R},\, \om\le0.
\end{align}

The point-wise estimate \eqref{point-K_k} for $K_k(t)$ ensures that the \underline{Penrose condition} holds, namely, there are $\nu_1, \kappa_1\in (0,1)$ independent of $k\in\Z^3$ such that for all $\nu\in[0,\nu_1]$, there holds
\begin{align}\label{Penrose}
\inf_{ \mathrm{Im}\, (\lm+i\om)\le0}\inf_{k\in\Z^3_*}|1+\mathcal{L}[K_k](\lm+i\om)|
\ge\kappa_1>0.
\end{align}
See more details in Section 7 of \cite{chaturvedi2023vlasov}. Note that in this section, we use the variable $\lambda$ to denote the Fourier-Laplace transform in $t$, rather than the Gevrey radius as employed elsewhere in the paper.

Clearly, \eqref{Penrose} holds for $\om=0$. 
Taking $\om=0$ in \eqref{LK_k}, let us define
\begin{align*}
  \tilde{G}_k(\lambda)=-\frac{\mathcal{L}[K_k](\lambda)}{1+\mathcal{L}[K_k](\lambda)},
\end{align*}
and $G_k(t)$ be the Fourier inverse of $\widetilde{G}_k(\lambda)$, namely,
\[
G_k(t)=-\fr{1}{2\pi}\int_{\R}e^{i\lm t}\fr{\mathcal{L}[K_k](\lm)}{1+\mathcal{L}[K_k](\lm)}d\lm.
\]
Then by $K_k(0)=0$, we have
\[
\pr_tG_k(t)=-\fr{1}{2\pi}\int_{\R}e^{i\lm t}\fr{\mathcal{L}[\pr_tK_k](\lm)}{1+\mathcal{L}[K_k](\lm)}d\lm.
\]

\subsection{Estimates of $G_k(t)$}
\begin{prop}\label{prop:kernel G}
    Let $\rho(t,x)$ be the density that solves  \eqref{l-VPL} with initial data $f_{\rm in}(x,v)$, then the following representation holds:
\begin{align}\label{rep-rho}
\hat{\rho}_k(t)=\mathcal{N}_k(t)+\int_0^tG_k(t-\tau)\mathcal{N}_k(\tau)d\tau.
\end{align}
Moreover, there exist $\dl_2\in(0,1)$ and $C_{s_L,\tl{c},\kappa_1}\ge1$ depending on $s_L, c$ and $\kappa_1$ such that
\begin{align}\label{est-G_k}
    |G_k(t)|\le C_{s_L,\tl{c},\kappa_1} |k|^{-1}e^{-\dl_2|kt|^{s_L}}.
\end{align}
\end{prop}
\begin{proof}
By the Sobolev embedding, one deduces that
\begin{align*}
    &|k|e^{\dl_2(|kt|+1)^{s_L}}|G_k(t)|\\
    \les& |k|^{\fr12}\left\||k|e^{\dl_2(|kt|+1)^{s_L}}G_k(t)\right\|_{L^2_t(\R)}+|k|^{-\fr12}\left\|\pr_t\left(|k|e^{\dl_2(|kt|+1)^{s_L}}G_k(t)\right)\right\|_{L^2_t(\R)}\\
    \les&|k|^{\fr32}\left\|e^{\dl_2(|kt|+1)^{s_L}}G_k(t)\right\|_{L^2_t(\R)}+s_L\dl_2 |k|^{\fr32}\left\|(|kt|+1)^{s_L-1}e^{\dl_2(|kt|+1)^{s_L}}G_k(t)\right\|_{L^2_t(\R)}\\
    &+|k|^{\fr12}\left\|e^{\dl_2|kt|^{s_L}}\pr_tG_k(t)\right\|_{L^2_t(\R)}\\
    \les&|k|^{\fr32}\left\|e^{\dl_2|kt|^{s_L}}G_k(t)\right\|_{L^2_t(\R)}+|k|^{\fr12}\left\|e^{\dl_2|kt|^{s_L}}\pr_tG_k(t)\right\|_{L^2_t(\R)}.
\end{align*}
By the Taylor series of $e^{\dl_2|kt|^{s_L}}$, 
\begin{align*}
    e^{\dl_2|kt|^{s_L}}=\sum_{n=0}^\infty\fr{(
    \dl_2|kt|^{s_L})^n}{n!}=\sum_{n=0}^\infty\left(\fr{\left[(2\dl_2)^{\fr{1}{s_L}}|kt|\right]^n}{(n!)^{\fr{1}{s_L}}}\right)^{s_L}\fr{1}{2^n}\les \sum_{n=0}^\infty\fr{\left((2\dl_2)^{\fr{1}{s_L}}|kt|\right)^n}{(n!)^{\fr{1}{s_L}}}.
\end{align*}
Let us denote $\dl_3=(2\dl_2)^{\fr{1}{s_L}}$. Then by Minkowski inequality and Plancherel's theorem, we are led to
\begin{align*}
|k|^{\fr32}\left\|e^{\dl_2|kt|^{s_L}}G_k(t)\right\|_{L^2_t(\R)}\les |k|^{\fr32}\sum_{n=0}^\infty \left\|\fr{\left(\dl_3|kt|\right)^n}{(n!)^{\fr{1}{s_L}}}G_k(t) \right\|_{L^2_t(\R)}&=|k|^{\fr32}\sum_{n=0}^\infty\left\|\fr{\left(\dl_3|k|\pr_\lm\right)^n}{(n!)^{\fr1s}}\tilde{G}_k(\lm) \right\|_{L^2_\lm(\R)}\\
&=:{\rm S_1},
\end{align*}
and similarly,
\begin{align*}
|k|^{\fr12}\left\|e^{\dl_2|kt|^{s_L}}\pr_tG_k(t)\right\|_{L^2_t(\R)}\les |k|^{\fr12}\sum_{n=0}^\infty\left\|\fr{\left(\dl_3|k|\pr_\lm\right)^n}{(n!)^{\fr{1}{s_L}}}\dot{\tilde{G}}_k(\lm) \right\|_{L^2_\lm(\R)}=:{\rm S_2},
\end{align*}
where
\[
\dot{\tilde{G}}_k(\lm)=-\fr{\mathcal{L}[\pr_tK_k](\lm)}{1+\mathcal{L}[K_k](\lm)}.
\]
{\bf Estimate of $\rm S_1$.}
Let $\Upsilon(u)=-\fr{u}{1+u}=\fr{1}{u+1}-1$. Then $\Upsilon^{(n)}(u)=(-1)^nn!(u+1)^{-(n+1)}$. By using the Fa\`{a} di Bruno's formula, we have
\begin{align*}
\pr^n_{\lm}\big[\tilde{G}_k(\lambda)\big]=\sum\fr{n!}{m_1!1!^{m_1}m_2!2!^{m_2}\cdots m_n!n!^{m_n}}\cdot \Upsilon^{(m_1+\cdots+m_n)}(\mathcal{L}[K_k](\lambda))\cdot\prod_{j=1}^n\left(\pr^j_\lm \mathcal{L}[K_k](\lambda)\right)^{m_j},
\end{align*}
where the sum is over all $n$-tuples of nonnegative integers $(m_1,\cdots, m_n)$ satisfying the constraint 
\[
1\cdot m_1+2\cdot m_2+3\cdot m_3+\cdots n\cdot m_n=n.
\]
Let $C_*\ge1$ be the embedding constant in $H^1(\R)\hookrightarrow L^\infty(\R)$.
Then for any fix $N\in\N\setminus\{0\}$, we have
\begin{align*}
    &\sum_{n=1}^N\left\|\fr{(\dl_3 |k|\pr_\lm)^n}{(n!)^{\fr{1}{s_L}}} \tilde{G}_k(\lm)\right\|_{L^2_\lm}\\
    \le&\sum_{n=1}^N\fr{(\dl_3|k|)^n}{(n!)^{\fr{1}{s_L}}}\sum_{\sum_{j=1}^n jm_j=n} \fr{n!}{m_1!1!^{m_1}m_2!2!^{m_2}\cdots m_n!n!^{m_n}}\\
    &\times \left\|\Upsilon^{(m_1+\cdots m_n)}(\mathcal{L}[K_k](\lm))\right\|_{L^\infty_\lm}C_*^{n-1}\prod_{j=1}^n\left\|\pr_\lm^j\mathcal{L}[K_k](\lambda)\right\|_{H^1_{\lm}}^{m_j}\\
    =&\sum_{n=1}^N\sum_{l=1}^n\fr{\left\|\Upsilon^{(l)}(\mathcal{L}[K_k](\lm))\right\|_{L^\infty_\lm}}{l!}\sum_{\substack{\sum_{j=1}^n jm_j=n\\ \sum_{j=1}^n m_j=l}} \left(\fr{1!^{m_1}2!^{m_2}\cdots n!^{m_n}}{n!}\right)^{\fr{1}{s_L}-1}\\
    &\times \fr{l!}{m_1!m_2!\cdots m_n!}C_*^{n-1}\prod_{j=1}^n\left\|\fr{(\dl_3 |k|\pr_\lm)^j\mathcal{L}[K_k](\lambda)}{(j!)^{\fr{1}{s_L}}}\right\|_{H^1_{\lm}}^{m_j}\\
    \le&\sum_{l=1}^N\fr{\left\|\Upsilon^{(l)}(\mathcal{L}[K_k](\lm))\right\|_{L^\infty_\lm}}{l!}\left(\sum_{n=l}^N\sum_{\substack{\sum_{j=1}^n jm_j=n\\ \sum_{j=1}^n m_j=l}}  \fr{l!}{m_1!m_2!\cdots m_n!}\prod_{j=1}^n\left\|\fr{(C_*\dl_3 |k|\pr_\lm)^j\mathcal{L}[K_k](\lambda)}{(j!)^{\fr{1}{s_L}}}\right\|_{H^1_{\lm}}^{m_j}\right)\\
    \le& \sum_{l=1}^N\fr{\left\|\Upsilon^{(l)}(\mathcal{L}[K_k](\lm))\right\|_{L^\infty_\lm}}{l!}\left(\sum_{j=1}^N \left\|\fr{(C_*\dl_3 |k|\pr_\lm)^j\mathcal{L}[K_k](\lambda)}{(j!)^{\fr{1}{s_L}}}\right\|_{H^1_\lm}\right)^l,
\end{align*}
Thus,
\begin{align}\label{fdb1}
\nn&\sum_{n=1}^\infty\left\|\fr{(\dl_3 |k|\pr_\lm)^n}{(n!)^{\fr{1}{s_L}}} \tilde{G}_k(\lm)\right\|_{L^2_\lm}\\
\le& \sum_{l=1}^\infty\fr{\left\|\Upsilon^{(l)}(\mathcal{L}[K_k](\lm))\right\|_{L^\infty_\lm}}{l!}\left(\sum_{j=1}^\infty \left\|\fr{(C_*\dl_3 |k|\pr_\lm)^j\mathcal{L}[K_k](\lambda)}{(j!)^{\fr{1}{s_L}}}\right\|_{H^1_{\lm}}\right)^l.
\end{align}
Next we infer from \eqref{est-K_k} that
\begin{align*}
    &\sum_{j=1}^\infty \left\|\fr{(C_*\dl_3 |k|\pr_\lm)^j\mathcal{L}[K_k](\lambda)}{(j!)^{\fr{1}{s_L}}}\right\|_{H^1_{\lm}}\\
    \le&\left[\left(\sum_{j=1}^\infty \left\|\fr{(\dl_1 |k|\pr_\lm)^j\mathcal{L}[K_k](\lambda)}{(j!)^{\fr{1}{s_L}}}\right\|_{L^2_{\lm}}^2\right)^{\fr12}+\left(\sum_{j=1}^\infty \left\|\fr{(\dl_1 |k|\pr_\lm)^j\pr_\lm\mathcal{L}[K_k](\lambda)}{(j!)^{\fr{1}{s_L}}}\right\|_{L^2_{\lm}}^2\right)^{\fr12}\right]\left(\sum_{j=1}^\infty\left(\fr{C_*\dl_3}{\dl_1}\right)^{2j}\right)^\fr12\\
    =&\fr{\left(\fr{C_*\dl_3}{\dl_1}\right)^{2}}{1-\left(\fr{C_*\dl_3}{\dl_1}\right)^{2}}\left[\left(\sum_{j=1}^\infty \left\|\fr{(\dl_1 |k|t)^jK_k(t)}{(j!)^{\fr{1}{s_L}}}\right\|_{L^2(t\ge0)}^2\right)^{\fr12}+\left(\sum_{j=1}^\infty \left\|\fr{(\dl_1 |k|t)^jtK_k(t)}{(j!)^{\fr{1}{s_L}}}\right\|_{L^2(t\ge0)}^2\right)^{\fr12}\right]\\
    \le&\fr{\left(\fr{C_*\dl_3}{\dl_1}\right)^{2}}{1-\left(\fr{C_*\dl_3}{\dl_1}\right)^{2}}\sqrt{2C_{s_L,\tl{c}}}|k|^{-\fr32}
    \le\fr{\kappa_1}{2}|k|^{-\fr32},
\end{align*}
provided
\begin{align*}
    \dl_3\le\min\left\{\fr{1}{\sqrt{2}},\fr12\fr{\sqrt{\kappa_1}}{(2C_{s_L,\tl{c}})^{\fr14}}\right\}\fr{\dl_1}{C_*}.
\end{align*}
Substituting the above bound into \eqref{fdb1} and using \eqref{Penrose}, we have 
\begin{align*}
    {\rm S}_1\le|k|^\fr32\|\tl{G}_k(\lm)\|_{L^2_\lm}+\fr{1}{\kappa_1}\sum_{n=1}^\infty\fr{1}{2^n}\le\fr{1}{\kappa_1}\left(1+|k|^{\fr32}\|K_k(t)\|_{L^2(t\ge0)}\right)\le C_{s_L,\tl{c},\kappa_1}.
\end{align*}

\noindent{\bf Estimate of $\rm S_2$.}
Let us denote ${\bf f}(u)=\fr{1}{u+1}, \  {\bf g}(\lm)=-\mathcal{L}[\pr_tK_k](\lm)$. By Leibniz formula,
\begin{align*}
&\sum_{n=0}^\infty\left\|\fr{\left(\dl_3|k|\pr_\lm\right)^n}{(n!)^{\fr{1}{s_L}}}\dot{\tilde{G}}_k(\lm) \right\|_{L^2_\lm(\R)}=\sum_{n=0}^\infty\left\|\fr{\left(\dl_3|k|\pr_\lm\right)^n}{(n!)^{\fr{1}{s_L}}}\Big({\bf f}(\mathcal{L}[K_k](\lm)){\bf g}(\lm)\Big) \right\|_{L^2_\lm(\R)}\\
\le&\sum_{n=0}^\infty   \sum_{l\le n}\left\| \fr{(\dl_3|k|\pr_\lm)^l}{(l!)^{\fr{1}{s_L}}}{\bf f}(\mathcal{L}[K_k](\lm))\right\|_{L^\infty_\lm}\left\|\fr{(\dl_3|k|\pr_\lm)^{n-l}}{\big((n-l)!\big)^{\fr1s}}{\bf g}(\lm) \right\|_{L^2_\lm}\left(\fr{l!(n-l)!}{n!}\right)^{\fr{1}{s_L}-1}\\
\le&\sum_{l=0}^\infty\left\| \fr{(\dl_3|k|\pr_\lm)^l}{(l!)^{\fr{1}{s_L}}}{\bf f}(\mathcal{L}[K_k](\lm))\right\|_{L^\infty_\lm}\sum_{n=l}^\infty \left\|\fr{(\dl_3|k|\pr_\lm)^{n-l}}{\big((n-l)!\big)^{\fr{1}{s_L}}}{\bf g}(\lm) \right\|_{L^2_\lm}\\
=&\sum_{n=0}^\infty\left\| \fr{(\dl_3|k|\pr_\lm)^n}{(n!)^{\fr{1}{s_L}}}{\bf f}(\mathcal{L}[K_k](\lm))\right\|_{L^\infty_\lm}\sum_{n=0}^\infty \left\|\fr{(\dl_3|k|\pr_\lm)^{n}}{(n!)^{\fr{1}{s_L}}}\mathcal{L}[\pr_tK_k](\lm) \right\|_{L^2_\lm}.
\end{align*}
Similar to \eqref{fdb1},  by the Fa\`{a} di Bruno's formula and Sobolev embedding $H^1(\R)\hookrightarrow L^\infty(\R)$, and noting that ${\bf f}^{(l)}(u)=\Upsilon^{(l)}(u)$ for all $l=1, 2, 3, \cdots$,  we find that
\begin{align*}
&\sum_{n=1}^\infty\left\|\fr{(\dl_3 |k|\pr_\lm)^n}{(n!)^{\fr{1}{s_L}}} {\bf f}(\mathcal{L}[K_k](\lm))\right\|_{L^\infty_\lm}\\
\le& \sum_{l=1}^\infty\fr{\left\|\Upsilon^{(l)}(\mathcal{L}[K_k](\lm))\right\|_{L^\infty_\lm}}{l!}\left(\sum_{j=1}^\infty \left\|\fr{(C_*\dl_3 |k|\pr_\lm)^j\mathcal{L}[K_k](\lambda)}{(j!)^{\fr{1}{s_L}}}\right\|_{H^1_{\lm}}\right)^l\le\fr{1}{\kappa_1}|k|^{-\fr32},
\end{align*}
and hence
\begin{align*}
    \sum_{n=0}^\infty\left\| \fr{(\dl_3|k|\pr_\lm)^n}{(n!)^{\fr{1}{s_L}}}{\bf f}(\mathcal{L}[K_k](\lm))\right\|_{L^\infty_\lm}\le \fr{2}{\kappa_1}.
\end{align*}

  On the other hand, by Plancherel's theorem and \eqref{est-K_k}, we arrive at
\begin{align*}
    &|k|^\fr12\sum_{n=0}^\infty \left\|\fr{(\dl_3|k|\pr_\lm)^{n}}{(n!)^{\fr{1}{s_L}}}\mathcal{L}[\pr_tK_k](\lm) \right\|_{L^2_\lm}\\
    \le&\left(\sum_{n=0}^\infty \left\|\fr{(\sqrt{2}\dl_3|k|t)^{n}}{(n!)^{\fr{1}{s_L}}}|k|^{\fr12}\pr_tK_k(t) \right\|_{L^2(t\ge0)}^2\right)^{\fr12}\left(\sum_{n=0}^\infty \fr{1}{2^n}\right)^{\fr12}\le C_{s_L,\tl{c}},
\end{align*}
provided $\dl_3\le\fr{\dl_1}{\sqrt{2}}$.
It follows from the above two estimates that
${\rm S}_2\le\fr{2}{\kappa_1}C_{s_L,\tl{c}}$. Then \eqref{est-G_k} holds as long as we choose
\begin{align*}
\dl_2\le \fr12\left(\min\left\{\fr{1}{\sqrt{2}},\fr12\fr{\sqrt{\kappa_1}}{(2C_{s_L,\tl{c}})^{\fr14}}\right\}\fr{\dl_1}{C_*}\right)^{s_L}.
\end{align*}
The  proof of this lemma is complete. 
\end{proof}

\section{Density estimates in Gevrey class}\label{sec: density est}

In this section, we will bound the density $\rho$. To use   the linear estimates obtained in Section \ref{sec: Linear estimates}, we  choose $\lm_L$ and $s_L$ such that $\lm(0)\le\lm_L$ and $s< s_L$ in what follows. 
\subsection{Decomposition of the density $\rho$}\label{sec: decom-rho}

Recall that we decompose the density $\rho$ into two parts, namely, $\rho=\rho^{(1)}+\rho^{(2)}$, where $\rho^{(1)}$ and $\rho^{(2)}$ solve
\begin{align}\label{e-rho1}
    \hat{\rho}^{(1)}_k(t)=\mathcal{N}^{(1)}_k(t)+\int_0^tG_k(t-\tau)\mathcal{N}^{(1)}_k(\tau)d\tau,
\end{align}
and
\begin{align}\label{e-rho2}
    \hat{\rho}^{(2)}_k(t)=\mathcal{N}^{(2)}_k(t)+\int_0^tG_k(t-\tau)\mathcal{N}^{(2)}_k(\tau)d\tau,
\end{align}
respectively, where
\begin{align*}
    \mathcal{N}_k^{(1)}(t)=&\int_{\R^3}S_k(t)[\hat{f}_{\rm in}(v)]\mu^{\fr12}(v)dv+\int_0^t\int_{\R^3}S_k(t-\tau)[\hat{\mathfrak{N}}^{(1)}_k(\tau,v)]\mu^{\fr12}(v)dvd\tau,\\
    \mathcal{N}_k^{(2)}(t)=&\int_0^t\int_{\R^3}S_k(t-\tau)[\hat{\mathfrak{N}}^{(2)}_k(\tau,v)]\mu^{\fr12}(v)dvd\tau,
\end{align*}
with
\begin{align}\label{eq: N_1, N_2}
\begin{aligned}
    \mathfrak{N}^{(1)}(t,x,v):=&-E_1(t,x)\cdot\nb_vf-\mathcal{T}'_{E_2^j}\pr_{v_j}f+E_1(t,x)\cdot v f+\mathcal{T}'_{E_2^j}v^jf ,\\
    \mathfrak{N}^{(2)}(t,x,v):=&-\mathcal{T}_{\pr_{v_j}f}E^j_2+\mathcal{T}_{v_jf}E^j_2+\nu\Gamma(f,f),
\end{aligned}
\end{align}
and $E_i=-\nb_x(-\Dl_x)^{-1}\rho^{(i)}$, $i=1,2$.

\begin{lem}\label{lem-rho:N}
The following estimates for $\rho^{(1)}$ and $\rho^{(2)}$ hold:
\begin{align}\label{es-rho:N1}
    \left\|\la t\ra^b|\nb_x|^\fr32\rho^{(1)}\right\|_{L^2_t\mathcal{G}^{\lm,N-1}_s}\les \left\|\la t\ra^b|\nb_x|^\fr32\mathcal{N}^{(1)}\right\|_{L^2_t\mathcal{G}^{\lm,N-1}_s}.
\end{align}

\begin{align}\label{es-rho:N2}
    \left\|\langle t\rangle^{\fr{3+a-s}{2}}\rho^{(2)}\right\|_{L^2_t\mathcal{G}^{\lm,N}_s}
    \les \left\|
    \langle t\rangle^{\fr{3+a-s}{2}}\mathcal{N}^{(2)}\right\|_{L^2_t\mathcal{G}^{\lm,N}_s}.
\end{align}
and for $\iota=1,2,\cdots,  [\ell/3]+1$,
\begin{align}\label{es-rho:N2-iota}
    \left\|\langle t\rangle^{\fr{1+a-s}{2}}w_{\iota}(t)\rho^{(2)}\right\|_{L^2_t\mathcal{G}^{\lm,N}_s}
    \les \left\|
    \langle t\rangle^{\fr{1+a-s}{2}} w_{\iota}(t)\mathcal{N}^{(2)}\right\|_{L^2_t\mathcal{G}^{\lm,N}_s}.
\end{align}
\end{lem}

\begin{proof}
Here we only prove \eqref{es-rho:N2-iota}, since \eqref{es-rho:N1} and \eqref{es-rho:N2} can be proved in a similar manner. To simplify the presentation, for any $f_k\in \mathbb{C}$, $k\in \Z^3_*$, let us denote
\[
\|f_k\|^2_{\mathsf{G}^{\lm(t),N}_s}:=\sum_{m\in\N^6,|\beta|\le N}\kappa^{2|\beta|}a_{m,\lm,s}^2(t)|(k,kt)^{m+\beta}f_k|^2,
\]
and
\[
\|f_k\|^2_{\mathsf{\tl G}^{\lm(t),N}_s}:=\sum_{m\in\N^3,|\beta|\le N}\kappa^{2|\beta|}a_{m,\lm,s}^2(t)|(kt)^{m+\beta}f_k|^2.
\]
Then we can write 
\[
\|\rho\|^2_{\mathcal{G}^{\lm,N}_s}=\sum_{k\in\Z^3_*}\|\hat{\rho}_k\|^2_{\mathsf{G}^{\lm,N}_s}.
\]
Noting  that 
\begin{align*}
    \fr{\la t\ra^{\fr{1+a-s}{2}}w_{\iota}(t)}{\la \tau\ra^{\fr{1+a-s}{2}}w_{\iota}(\tau)}\les \la t-\tau\ra^{\fr{1+a-s+\iota}{2}},
\end{align*}
then by Lemma \ref{lem-summable} and the Cauchy-Schwarz inequality, we find that
\begin{align*}
&\la t\ra^{1+a-s}w_{\iota}^2(t)\sum_{k\in\Z^3_*}\Big\|\int_0^t G_k(t-\tau)\mathcal{N}^{(2)}_k(\tau)d\tau\Big\|_{\mathsf{G}^{\lm,N}_s}^2\\
\les&\sum_{k\in\Z^3_*}\sum_{|\beta|\le N}\kappa^{2|\beta|}\sum_{m\in\N^6}\Bigg[\sum_{\substack{\bar{n}\le m\\ \bar{\beta'}\le\beta} }b_{m,\bar{n},s}\int_0^t \la t-\tau\ra^{\fr{1+a-s+\iota}{2}} a_{\bar{n},\lm,s}(0)\Big| (k(t-\tau))^{\bar{n}+\bar{\beta'}}G_k(t-\tau)\Big|\\
&\times \la \tau\ra^{\fr{1+a-s}{2}}w_{\iota}(\tau) a_{m-\bar{n},\lm,s}(\tau)\left| (k,k\tau)^{m-\bar{n}+\beta-\bar{\beta'}}\mathcal{N}_k^{(2)}(\tau)\right|d\tau\Bigg]^2\\
\les&\sum_{k\in\Z^3_*}\sum_{\substack{m\in\N^6\\ |\beta|\le N}}\sum_{\substack{\bar{n}\le m\\ \bar{\beta'}\le\beta} }\int_0^t \la t-\tau\ra^{3+a-s+\iota} \kappa^{2|\bar{\beta'}|} a^2_{\bar{n},\lm,s}(0)\Big| (k(t-\tau))^{\bar{n}+\bar{\beta'}}G_k(t-\tau)\Big|^2\\
&\times \la \tau\ra^{1+a-s}w_{\iota}^2(\tau)\kappa^{2|\beta-\bar{\beta'}|} a_{m-\bar{n},\lm,s}^2(\tau)\left| (k,k\tau)^{m-\bar{n}+\beta-\bar{\beta'}}\mathcal{N}_k^{(2)}(\tau)\right|^2d\tau\\
\les&\sum_{k\in\Z^3_*}\int_0^t \la t-\tau\ra^{3+a-s+\iota} \| G_k(t-\tau)\|_{\mathsf{\tl G}^{\lm(0),N}_s}^2 \Big(\la \tau\ra^{1+a-s}w_{\iota}^2(\tau)\|\mathcal{N}_k^{(2)}(\tau)\|_{\mathsf{G}^{\lm(\tau),N}_s}^2\Big)d\tau.
\end{align*}
Consequently,
\begin{align*}
&\int_{0}^{T}\la t\ra^{1+a-s}w_{\iota}^2(t)\sum_{k\in\Z^3_*}\Big\|\int_0^t G_k(t-\tau)\mathcal{N}^{(2)}_k(\tau)d\tau\Big\|_{\mathsf{G}^{\lm(t),N}_s}^2dt\\
\les&\sum_{k\in\Z^3_*}\int_{0}^{T}\int_\tau ^{T} \la t-\tau\ra^{3+a-s+\iota} \| G_k(t-\tau)\|_{\mathsf{\tl  G}^{\lm(0),N}_s}^2 \Big(\la \tau\ra^{1+a-s}w_{\iota}^2(\tau)\|\mathcal{N}_k^{(2)}(\tau)\|_{\mathsf{G}^{\lm(\tau),N}_s}^2\Big)dt d\tau\\
\les&\sup_{k\in\Z^3_*}\int_0 ^{T} \la t\ra^{3+a-s+\iota} \| G_k(t)\|_{\mathsf{\tl G}^{\lm(0),N}_s}^2dt\int_{0}^{T}\la \tau\ra^{1+a-s}w_{\iota}^2(\tau)\sum_{k\in\Z^3_*}\|\mathcal{N}_k^{(2)}(\tau)\|_{\mathsf{G}^{\lm(\tau),N}_s}^2d\tau.
\end{align*}

We are left to bound $\sup_{k\in\Z^3_*}\int_0^{T}\la t\ra^{3+a-s+\iota}\|G_k(t)\|^2_{\mathsf{\tl G}^{\lm(0),N}_s}dt$.
To this end, noting that $C_{|m|}^m\le 3^{|m|}$ for $m\in\N^3$, and using the multinomial theorem again, see \eqref{multinomial} for instance, we infer from \eqref{holder} and \eqref{est-G_k} that
\begin{align*}
\|G_k(t)\|^2_{\mathsf{\tl G}^{\lm(0),N}_s}\le &\sum_{|\beta|\le N}\kappa^{2|\beta|}(kt)^{2\beta}\sum_{j=0}^\infty\fr{\Big(\sqrt{3}\lm(0)\Big)^{2j}}{\Gamma_s(j)^2}\sum_{|m|=j}C_{|m|}^m (kt)^{2m}\left|G_k(t)\right|^2\\
=&\sum_{|\beta|\le N}\kappa^{2|\beta|}(kt)^{2\beta}\sum_{j=0}^\infty\fr{\Big(\sqrt{3}\lm(0)|k|t\Big)^{2j}}{\Gamma_s(j)^2}\left|G_k(t)\right|^2\\
\le&C\la kt\ra^{2N}\sum_{j=0}^\infty\fr{\Big(\sqrt{3}\lm(0)|k|t\Big)^{2j}(j+1)^{24}}{(j!)^{\fr2s}}\left|G_k(t)\right|^2\\
\le&C|k|^{-2}\la kt\ra^{2N}\sum_{j=0}^\infty\fr{\big(6\lm(0)|k|t\big)^{2j}}{(j!)^{\fr2s}}e^{-2\dl_2|kt|^{s_L}}\\
\le&C|k|^{-2}\la kt\ra^{2N}e^{([\fr{2}{s}]+1)(6\lm(0))^s|kt|^s}e^{-2\dl_2|kt|^{s_L}}\\
\le&C|k|^{-2}\la kt\ra^{2N} e^{-\dl_2|kt|^{s_L}},
\end{align*}
due to $s_L>s$. Consequently,
\begin{align*}
    \int_0^{T}\la t\ra^{3+a-s+\iota}\|G_k(t)\|^2_{\mathsf{G}^{\lm(0),N}_s}dt
    \le C|k|^{-2}\int_0^{T}\la t\ra^{3+a-s+\iota} \la kt\ra^{2N}e^{-\dl_2|kt|^{s_L}}dt\le C|k|^{-3}.
\end{align*}
Then \eqref{es-rho:N2-iota} follows immediately.
\end{proof}

\subsection{Estimates for $\rho^{(1)}$}
To begin with,  let us define
\begin{align}\label{def-frak-g}
\mathfrak{g}_k(t,v)=\begin{cases}S_k^{*}(t)[\sqrt{\mu}(v)],\ \ {\rm if}\ \ k\in\Z^3_*;\\
\qquad\qquad\quad0, \ \ {\rm if}\ \ k=0,
\end{cases}
\quad{\rm and}\quad \mathfrak{g}(t,x,v)=\sum_{k\in\Z^3}\mathfrak{g}_k(t,v)e^{ik\cdot x}.
\end{align}

Recalling the definition of $\|\cdot\|_{\bar{\mathcal{G}}^{\lm,N}_{s,\ell}}$ in \eqref{def-varGnorm}, and in view of Proposition \ref{prop-semigroup}, we have the following bound for $\mathfrak{g}$.
\begin{lem}\label{lem:est-S^*}
For any $\ell_0\in\N, N_0\in\N$, there holds
\begin{align}\label{bd-frak-g1}
   \la \nu^{\fr13}t\ra^{3} \|\mathfrak{g}(t)\|_{\bar{\mathcal{G}}^{\lm_L,N_0}_{s_L,\ell_0}}^2\le C.
\end{align}
Furthermore, 
\begin{align}\label{bd-frak-g2}
   \la \nu^{\fr13}t\ra^{3} \|v\mathfrak{g}(t)\|_{\bar{\mathcal{G}}^{\lm_L,N_0}_{s_L,\ell_0}}^2\le C.
\end{align}
\end{lem}
\begin{proof}
Recalling \eqref{eq:varE_l^n}, \eqref{def-varE} and \eqref{def-total-en}, for any fixed $k\in\Z^3_*$, we find that
    \begin{align*}
        &\la\nu^{\fr13}t\ra^3\|\mathfrak{g}(t)\|_{\bar{\mathcal{G}}^{\lm_L,N_0}_{s_L,\ell_0}}^2\le\la\nu^{\fr13}t\ra^3\|\mathfrak{g}(t)\|_{\bar{\mathcal{G}}^{\lm_L,N_0}_{s_L,\ell_0+4}}^2 \\
        \les&\fr{1}{|k|^2}\mathsf{E}_{k,\ell_0+6}(\mathfrak{g}_k(t))+\fr{1}{|k|^2}w_{3}^2(t)\mathsf{E}_{k,\ell_0+6-6}(\mathfrak{g}_k(t))\\
        \les& \fr{1}{|k|^2}\mathsf{E}^w_{k,\ell_0+6}(\mathfrak{g}_k(t))\les \fr{1}{|k|^2}\mathsf{E}^w_{k,\ell_0+6}(\mathfrak{g}_k(0))\les \|\la \nb_v\ra\sqrt{\mu}\|^2_{\mathcal{G}^{\lm_L,N_0}_{s_L, \ell_0+6}}.
    \end{align*}
    Here the $\bar{\mathcal{G}}^{\lm_L,N_0}_{s_L,\ell_0}$ norm is defined in \eqref{def-varGnorm}
To prove \eqref{bd-frak-g2}, note that for $|m|+|\beta|>0$,
\begin{align*}
    (Y^*)^{m+\beta}(v_i\mathfrak{g})=v_i(Y^*)^{m+\beta}\mathfrak{g}+(m_i+\beta_i)(Y^*)^{m+\beta-e_i}\mathfrak{g}.
\end{align*}
Assume, w.l.o.g., that $m_i>0$. Then it is not difficult to verify that
\begin{align*}
    a_{m,\lm_L,s_L} m_i\les a_{m-e_i,\lm_L,s_L},
\end{align*}
and hence
\begin{align}\label{est-vg}
    \|v\mathfrak{g}(t)\|_{\bar{\mathcal{G}}^{\lm_L,N_0}_{s_L,\ell_0}}^2\les \|\mathfrak{g}(t)\|_{\bar{\mathcal{G}}^{\lm_L,N_0}_{s_L,\ell_0+1}}^2+\|\mathfrak{g}(t)\|_{\bar{\mathcal{G}}^{\lm_L,N_0}_{s_L,\ell_0}}^2\les \|\la \nb_v\ra\sqrt{\mu}\|^2_{\mathcal{G}^{\lm_L,N_0}_{s_L, \ell_0+7}}.
\end{align}
Then \eqref{bd-frak-g2} follows immediately.
\end{proof}

The aim of this section is to establish the following Proposition:

\begin{prop}\label{prop-rho1}
    Under the bootstrap hypotheses \eqref{bd-en}--\eqref{H-phi}, if
    \begin{align}\label{res-N-b}
        \max\{8,b+6\}\le\fr{N}{2},
    \end{align}

\begin{align}\label{restrction-s-rho1}
    s>\fr13,\quad{\rm with}\quad 0<a\le\min\left\{\fr{3s-1}{2},  \fr{s}{b+4}\right\},
\end{align}

    there holds
\begin{align}\label{bd-rho1}
    \big\|\la t\ra^b|\nb_x|^{\fr32}\rho^{(1)}\big\|^2_{L^2_t\mathcal{G}^{\lm,N-1}_{s} }\les\nn& \|\la\nb_x,\nb_v \ra^{2b+2}f_{\rm in}\|^2_{\mathcal{G}^{\lm(0),N-1}_{s,0}}\\
    \nn&+\left\| \la t\ra^b |\nb_x|^{\fr32}\rho^{(1)}\right\|^2_{L^2_t\mathcal{G}^{\lm,N-1}_s}{ \Big(\left\|g\right\|^2_{L^\infty_t\mathcal{G}^{\lm,N}_{s,1}}+\left\|\nb_xg\right\|^2_{L^\infty_t\mathcal{G}^{\lm,N}_{s,1}}\Big)}\\
    &+\left\|\la t\ra^{\fr{3+a-s}{2}} \rho^{(2)} \right\|^2_{L^2_t\mathcal{G}^{\lm,N}_{s}}{ \Big(\left\|g\right\|^2_{L^\infty_t\mathcal{G}^{\lm,N}_{s,1}}+\left\|\nb_xg\right\|^2_{L^\infty_t\mathcal{G}^{\lm,N}_{s,1}}\Big)}.
\end{align}

\end{prop}
\begin{proof}
To bound $\big\|\la t\ra^b|\nb_x|^{\fr32}\rho^{(1)}\big\|_{L^2_t\mathcal{G}^{\lm,N-1}_{s} }$, we estimate the terms in $\mathcal{N}^{(1)}_k(t)$ one by one.

\noindent{$\diamond$ \underline{\it Initial contribution}.}
Noting that $(k,kt)=(k,\eta)+(0,-\eta+kt)$ and that for $b\ge\fr12$, there holds
\begin{align*}
    \la t\ra^{2b}|k|^3\les\la k, kt\ra^{2b+2}\la t\ra^{-2}\les\la k, \eta\ra^{2b+2}\la -\eta+kt\ra^{2b+2}\la t\ra^{-2},
\end{align*}
using the dual semi-group $S^*_k(t)$ and the Plancherel's theorem, we have
\begin{align}\label{est-rho1-initial}
    \nn&\int_0^{T}\la t\ra^{2b}\sum_{k\in\Z^3_*}|k|^3\sum_{|\beta|\le N-1}\kappa^{2|\beta|}\sum_{m\in \N^6}\Big|a_{m,s}(t)(k,kt)^{m+\beta}\int_{\R^3}S_k(t)\big[(\widehat{f_{\rm in}})_k(v)\big]\mu^{\fr12}(v)dv\Big|^2dt\\
    =\nn&\int_0^{T}\la t\ra^{2b}\sum_{k\in\Z^3_*}|k|^3\sum_{|\beta|\le N-1}\kappa^{2|\beta|}\sum_{m\in \N^6}\bigg|\sum_{\substack{n\le m, \beta'\le\beta}}C_{\beta}^{\beta'} b_{m,n,s}\int_{\R^3}a_{n,s}(t)(k,\eta)^{n+\beta'}(\widehat{f_{\rm in}})_k(\eta)\\
    \nn&\times a_{m-n,s}(t)(0,-\eta+kt)^{m-n+\beta-\beta'}\overline{\mathcal{F}_v\big[S_k^*(t)[\mu^{\fr12}(v)]\big]}(\eta)d\eta\bigg|^2dt\\
    \les\nn&\int_0^{T}\la t\ra^{-2}\|\la\nb_x,\nb_v \ra^{2b+2}f_{\rm in}\|^2_{\mathcal{G}^{\lm(0),N-1}_{s,0}}\|\la Y^*\ra^{2b+2}\mathfrak{g}(t)\|^2_{\bar{\mathcal{G}}^{\lm(0),N-1}_{s,0}}dt\\
    \les&\|\la\nb_x,\nb_v \ra^{2b+2}f_{\rm in}\|^2_{\mathcal{G}^{\lm(0),N-1}_{s,0}}\sup_t\|\la Y^*\ra^{2b+2}\mathfrak{g}(t)\|^2_{\bar{\mathcal{G}}^{\lm(0),N-1}_{s,0}}.
\end{align}

\noindent{$\diamond$ \underline{\it Treatments of $E_1\cdot\nb_vf$}.} We now consider the collisionless nonlinear interaction $-E_1\cdot\nb_vf$  in $\mathfrak{N}^{(1)}$ in \eqref{eq: N_1, N_2}. To this end, let us denote
\begin{align}\label{T1E1}
    {\bf T}^{(1)}[E_1]
:=\nn&\int_0^{T}\la t\ra^{2b}\sum_{k\in\Z^3_*}|k|^{3}\sum_{|\beta|\le N-1}\kappa^{2|\beta|}\sum_{m\in\N^6}\Bigg|a_{m,s}(t)(k,kt)^{m+\beta}\\
\nn& \times\int_0^t\int_{\R^3}\sum_{l\in\Z^3_*}\hat{\rho}^{(1)}_l(\tau)\fr{l}{|l|^2}\cdot \eta \hat{f}_{k-l}(\tau,\eta)\bar{\hat{\mathfrak{g}}}_k(t-\tau,\eta) d\eta d\tau\Bigg|^2dt\\
=\nn&\int_0^{T}\la t\ra^{2b}\sum_{k\in\Z^3_*}|k|^{3}\sum_{|\beta|\le N-1}\kappa^{2|\beta|}\sum_{m\in\N^6}\Bigg|a_{m,s}(t)(k,kt)^{m}\\
\nn&\times\int_0^t\int_{\R^3}\Big(\sum_{\substack{\beta'\le\beta,|\beta'|\le|\beta|/2\\\beta''\le\beta-\beta'}}+\sum_{\substack{\beta'\le\beta,|\beta'|>|\beta|/2\\\beta''\le\beta-\beta'}}\Big)C_\beta^{\beta'}C_{\beta-\beta'}^{\beta''}\sum_{l\in\Z^3_*}(l,l\tau)^{\beta'}\hat{\rho}^{(1)}_l(\tau)\fr{l}{|l|^2}\\
\nn&\cdot \eta(k-l,\eta+(k-l)\tau)^{\beta''}\hat{f}_{k-l}(\tau,\eta)(0,-\eta+k(t-\tau))^{\beta-\beta'-\beta''}\bar{\hat{\mathfrak{g}}}_k(t-\tau,\eta) d\eta d\tau\Bigg|^2dt\\
:=&{\bf T}^{(1);\rm LH}[E_1]+{\bf T}^{(1);\rm HL}[E_1].
\end{align}

{\bf Case 1: $\tau\le \fr{2t}{3}$.} In this case, the gap \eqref{lm-gap} implies that (here and in what follows we focus on the case $t\ge1$, noting that the case $t\leq 1$ follows more or less immediately)
\begin{align}\label{lm-gap1}
    \fr{\lm(\tau)-\lm(t)}{\lambda(\tau)}\ge \fr{a\tl{\dl}}{6(\lm_\infty+\tl{\dl})}\fr{1}{(1+t)^{a}}
    \ge\fr{a\tl{\dl}}{6(\lm_\infty+\tl{\dl})2^{a}}t^{-a}=:\underline{c}t^{-a}.
\end{align}
Combining this with the elementary inequality $1-x\le e^{-x}$ for all $x\ge0$ yields
\begin{align}\label{lm-gap2}
\fr{\lm(t)}{\lm(\tau)}=1-\fr{\lm(\tau)-\lm(t)}{\lm(\tau)}\le1-\underline{c}t^{-a}\le e^{-\underline{c}t^{-a}}.
\end{align}
Consequently, for any  $b\in \N$, $|m|>0$, and $\Theta>0$, there holds
\begin{align}\label{eq:gain m}
|k|^{\fr32}\la t\ra^{b+3}\left(\fr{\lm(t)}{\lm(\tau)}\right)^{|m|}\le |k|^{\fr32}\la t\ra^{b+3} e^{-\underline{c}|m|t^{-a}}\le C\fr{|kt|^{b+3+a\Theta}}{|k|^{\fr32+b+a\Theta}|m|^{\Theta}}.
\end{align}
 Now choosing $\Theta=\fr{b+3}{s-a}$, for $ s\ge a(b+4)$, arguing as \eqref{est-interp1} and\eqref{est-interp2} in Lemma \ref{lem-interp1}, for $|m|>0$, we have
 \begin{align*}
    \nn&\left||k|^{\fr32}\la t\ra^b a_{m,s}(t) (k,kt)^{m+\beta}\right|\\
    \nn\le&C\fr{1}{|k|^{\fr32+b+a\Theta}\la t\ra^3}\left[\sum_{|\bar{\al}|=b+3}C_{|\bar{\al}|}^{\bar{\al}}a_{m+\bar{\al},s}^2(\tau)(k,kt)^{2(m+\bar{\al})+2\beta}\right]^{\fr{1-a\Theta}{2}}\\
&\times \left[\sum_{i=4}^6\sum_{|\bar{\al}|=b+3}C_{|\bar{\al}|}^{\bar{\al}}a_{m+\bar{\al}+\bar{e}_i,s}^2(\tau)(k,kt)^{2(m+\bar{\al}+\bar{e}_i)+2\beta}\right]^{\fr{a\Theta}{2}}.
\end{align*}
Denote $F_k(t,\tau,\eta)=\sum_{l\in\Z^3_*}\hat{\rho}^{(1)}_l(\tau)\frac{l}{|l|^2}\cdot  \eta\hat{f}_{k-l}(\tau,\eta)\bar{\hat{\mathfrak{g}}}_k(t-\tau,\eta)$ for the moment for the sake of presentation, we then deduce that
\begin{align}
    \nn&\sum_{0<m\in\N^6}\Bigg||k|^{\fr32}\la t\ra^b a_{m,s}(t) (k,kt)^{m+\beta}\int_0^{\fr{2t}{3}}\int_{\R^3}F_k(t,\tau,\eta)d\eta d\tau\Bigg|^2\\
    \nn\les&\fr{1}{|k|^{3+2b+2a\Theta}\la t\ra^6}\sum_{0<m\in\N^6}\Bigg[\int_0^{\fr{2t}{3}}\int_{\R^3}\Big(\sum_{|\bar{\al}|=b+3}a_{m+\bar{\al},s}^2(\tau)\big|(k,kt)^{(m+\bar{\al})+\beta} F_k(t,\tau,\eta)\big|^2\Big)^{\fr{1-a\Theta}{2}}\\
\nn&\times \Big(\sum_{i=4}^6\sum_{|\bar{\al}|=b+3}a_{m+\bar{\al}+\bar{e}_i,s}^2(\tau)\big|(k,kt)^{(m+\bar{\al}+\bar{e}_i)+\beta}F_k(t,\tau,\eta)\big|^2\Big)^{\fr{a\Theta}{2}}d\eta d\tau\Bigg]^2\\
\nn\les&\fr{1}{|k|^{3+2b+2a\Theta}\la t\ra^6}\sum_{0<m\in\N^6}\Bigg[\sum_{|\bar{\al}|=b+3}\int_0^{\fr{2t}{3}}\int_{\R^3}a_{m+\bar{\al},s}(\tau)\big|(k,kt)^{(m+\bar{\al})+\beta} F_k(t,\tau,\eta)\big|d\eta d\tau\Bigg]^{2(1-a\Theta)}\\
\nn&\times \Bigg[\sum_{i=4}^6\sum_{|\bar{\al}|=b+3}\int_0^{\fr{2t}{3}}\int_{\R^3}a_{m+\bar{\al}+\bar{e}_i,s}(\tau)\big|(k,kt)^{(m+\bar{\al}+\bar{e}_i)+\beta}F_k(t,\tau,\eta)\big|d\eta d\tau\Bigg]^{2a\Theta}\\
\nn\les&\fr{1}{|k|^{3+2b+2a\Theta}\la t\ra^6}\sum_{0<m\in\N^6}\Bigg[\int_0^{\fr{2t}{3}}\int_{\R^3}a_{m,s}(\tau)\big|(k,kt)^{m+\beta}F_k(t,\tau,\eta)\big|d\eta d\tau\Bigg]^{2}.
\end{align}
In particular, if $|m|=0$, 
\begin{align*}
    &\Bigg||k|^{\fr32}\la t\ra^b a_{0,s}(t) (k,kt)^{\beta}\int_0^{\fr{2t}{3}}\int_{\R^3}F_k(t,\tau,\eta)d\eta d\tau\Bigg|^2\\
    \les&\fr{1}{|k|^{3+2b+2a\Theta}\la t\ra^6}\Bigg||k t|^{b+4}  (k,kt)^{\beta}\int_0^{\fr{2t}{3}}\int_{\R^3}F_k(t,\tau,\eta)d\eta d\tau\Bigg|^2\\
    \les&\fr{1}{|k|^{3+2b+2a\Theta}\la t\ra^6} \sum_{|\bar{\al}|=b+4}C_{|\bar{\al}|}^{\bar{\al}}(k,kt)^{2(\bar{\al}+\beta)}\Bigg|\int_0^{\fr{2t}{3}}\int_{\R^3}a_{\bar{\al},s}(\tau)F_k(t,\tau,\eta)d\eta d\tau\Bigg|^2,
\end{align*}
where we have used the fact that $a_{\bar{\al},s}(\tau)$ is bounded from below for fixed $|\bar{\al}|=b+4$ in the last inequality above.
Combining the  above two inequalities for $|m|>0$ and $|m|=0$, respectively, we are led to
\begin{align}
    \nn&\sum_{m\in\N^6}\Bigg||k|^{\fr32}\la t\ra^b a_{m,s}(t) (k,kt)^{m+\beta}\int_0^{\fr{2t}{3}}\int_{\R^3}F_k(t,\tau,\eta)d\eta d\tau\Bigg|^2\\
\nn\les&\fr{1}{|k|^{3+2b+2a\Theta}\la t\ra^6}\sum_{m\in\N^6}\Bigg[\int_0^{\fr{2t}{3}}\int_{\R^3}a_{m,s}(\tau)\big|(k,kt)^{m+\beta}F_k(t,\tau,\eta)\big|d\eta d\tau\Bigg]^{2}.
\end{align}
Combining this with the fact 
\begin{align}\label{est-eta}
    |\eta|\les\la \tau\ra|k-l,\eta+(k-l)\tau|,
\end{align}
one can bound ${\bf T}^{(1)}[E_1]_{\tau\le\fr{2t}{3}}$ without resorting to the low-high and high-low   decomposition:
\begin{align}\label{est-T1E1-short}
    {\bf T}^{(1)}[E_1]_{\tau\le\fr{2t}{3}}\les\nn&\int_0^{T}\sum_{k\in\Z^3_*}\fr{1}{|k|^{3+2b+2a\Theta}} \sum_{|\beta|\le N-1}\kappa^{2|\beta|}\sum_{m\in\N^6}\\
    \nn&\times\left[\int_0^{\fr{2t}{3}}\int_{\R^3}a_{m,s}(\tau)\left|(k,kt)^{m+\beta}\sum_{l\neq 0}\hat{\rho}^{(1)}_l(\tau)\frac{l}{|l|^2}\fr{1}{\la t\ra^{3}}\cdot \eta\hat{ f}_{k-l}(\tau,\eta)\bar{\hat{\mathfrak{g}}}_k(t-\tau,\eta)\right| d\eta d\tau\right]^2dt\\
    \nn\les&\int_0^{T}\sum_{k\in\Z^3_*}\fr{1} {|k|^{3+2b+2a\Theta}}\sum_{|\beta|\le N-1}\kappa^{2|\beta|}\sum_{m\in\N^6}\Bigg[\sum_{\substack{\beta'\le\beta\\ \beta''\le\beta-\beta'}}\sum_{\substack{n\le m\\n'\le m-n}}C_\beta^{\beta'}C_{\beta-\beta'}^{\beta''}\\
    \nn&\times b_{m,n,n',s}\int^{\fr{2t}{3}}_0\int_{\R^3}  \sum_{l\in\Z^n_*}a_{n,s}(\tau) \left|(l,l\tau)^{n+\beta'}\hat{\rho}^{(1)}_l(\tau)\frac{1}{|l|}\right|\fr{1}{\la t\ra^2}\\
    \nn&\times a_{n',s}(\tau)\left|(k-l,\eta+(k-l)\tau)^{n'+\beta''} \widehat{ Zf}_{k-l}(\tau,\eta)\right| \\
    \nn&\times a_{m-n-n',s}(\tau)\left|(0,-\eta+k(t-\tau))^{m-n-n'+\beta-\beta'-\beta''}\bar{\hat{\mathfrak{g}}}_k(t-\tau,\eta)\right|d\eta d\tau\Bigg]^2dt\\
\les\nn&\int_0^{T}\sup_{k\in\Z^3_*}\sum_{|\beta|\le N-1}\kappa^{2|\beta|}\sum_{m\in\N^6}\Big(\sum_{\substack{\beta'\le\beta\\ \beta''\le\beta-\beta'}}\sum_{\substack{n\le m\\ n'\le m-n}}\int^{\fr{2t}{3}}_0\fr{b_{m,n, n', s}}{\la \tau\ra^2}d\tau\Big)\\
    \nn&\times \int^{\fr{2t}{3}}_0  \sum_{\substack{\beta'\le\beta\\ \beta''\le\beta-\beta'}}\sum_{\substack{n\le m\\ n'\le m-n}}a_{n,s}^2(\tau) \left\|(l,l\tau)^{n+\beta'}|l|^{-1}\hat{\rho}^{(1)}_l(\tau)\right\|_{L^2_l}^2\fr{1}{\la t\ra^2}\\
    \nn&\times a_{m-n,s}^2(\tau)\left\|(k-l,\eta+(k-l)\tau)^{m-n+\beta-\beta'}\widehat{ Zf}_{k-l}(\tau)\right\|_{L^2_{l,\eta}}^2\\
    \nn&\times a^2_{m-n-n',s}(0)\left\|(0,-\eta+k(t-\tau))^{m-n-n'+\beta-\beta'-\beta''}\bar{\hat{\mathfrak{g}}}_k(t-\tau)\right\|_{L^2_\eta}^2d\tau dt\\
    \nn\les&\int_0^{T}\int^{\fr{2t}{3}}_0 \fr{1}{\la t\ra^2} \left\||\nb_x|^{-1}\rho^{(1)}(\tau)\right\|^2_{\mathcal{G}^{\lm,N-1}_s}\left\|Z f(\tau) \right\|^2_{\mathcal{G}^{\lm,N-1}_{s,0}}\left\| \mathfrak{g}(t-\tau)\right\|^2_{\bar{\mathcal{G}}^{\lm(0), N-1}_{s,0}}d\tau dt\\
    \les&\left\||\nb_x|^{-1}\rho^{(1)}\right\|^2_{L^2_t\mathcal{G}^{\lm,N-1}_s}{ \left\| f \right\|^2_{L^\infty_t  \mathcal{G}^{\lm, N}_{s,0}}}
    \left\| \mathfrak{g}\right\|^2_{L^\infty_t  \bar{\mathcal{G}}^{\lm(0), N-1}_{s,0}}.
\end{align}

{\bf Case 2: $\fr{2t}{3}\le\tau\le t$.} We first consider ${\bf T}^{(1);\rm HL}[E_1] $ where $\rho^{(1)}$ is at high frequency. If $l=k$, the extra $v$ derivative on $f_0$ is harmless, while if $l\ne k$, one can see from \eqref{est-eta} that the extra $v$ derivative will lead to $\tau$ growth. Thus, two subcases will be involved.

{\it Case 2.1: $l\ne k$}. Noting that now $f$ is at low frequency, one can pay sufficient  derivatives on $f$ as well as on $\mathfrak{g}$ to gain a time integrable factor $\fr{1}{\la kt-l\tau\ra^3}$, which may be used to counteract the $\tau$ growth. Accordingly, it is natural to consider the following  two subcases $|kt-l\tau|\ge\fr{t}{2}$ and $|kt-l\tau|\le\fr{t}{2}$. To proceed, let us denote
\begin{align}\label{def-D}
    D=\{|kt-l\tau|\le{t}/{2}, k\ne l\},\quad D^*=\{|kt-l\tau|>{t}/{2}, k\ne l\},
\end{align}
and write
\begin{align*}
    {\bf T}^{(1);{\rm HL}}[E_1]_{\tau\approx t}
    :=&\int_0^{T}\la t\ra^{2b}\sum_{k\in\Z^3_*}|k|^{3}\sum_{|\beta|\le N-1}\kappa^{2|\beta|}\sum_{m\in\N^6}\Bigg|\int_{\fr{2t}{3}}^t\int_{\R^3}\left(\fr{\lm(t)}{\lm(\tau)}\right)^{|m|}\Big({\bf1}_{D}+{\bf1}_{D^*}\Big)\\
\nn&\times\sum_{\substack{\beta'\le\beta,|\beta'|>|\beta|/2\\\beta''\le\beta-\beta'}}C_{\beta}^{\beta'}C_{\beta-\beta'}^{\beta''}\sum_{\substack{n\le m\\n'\le m-n}}\sum_{l\in\Z^3_*} b_{m,n,n',s}(l,l\tau)^{n+\beta'}a_{n,s}(\tau)\hat{\rho}^{(1)}_l(\tau)\fr{l}{|l|^2}\\
\nn&\cdot a_{n',s}(\tau)\eta(k-l,\eta+(k-l)\tau)^{n'+\beta''}\hat{f}_{k-l}(\tau,\eta)\\
&\times a_{m-n-n',s}(\tau)(0,-\eta+k(t-\tau))^{m-n-n'+\beta-\beta'-\beta''}\bar{\hat{\mathfrak{g}}}_k(t-\tau,\eta) d\eta d\tau\Bigg|^2dt\\
=&{\bf T}^{(1);{\rm HL}}_{D}[E_1]_{\tau\approx t}+{\bf T}^{(1);{\rm HL}}_{D^*}[E_1]_{\tau\approx t}.
\end{align*}

For the former case when $|kt-l\tau|\le\fr{t}{2}$, and $k\ne l$, we have the following lower bound of $t-\tau$:
\[
|l(t-\tau)|=|(l-k)t-(kt-l\tau)|\ge |l-k|t-|kt-l\tau|\ge \fr{t}{2}.
\]
This enables us to use the gap between $\lm(t)$ and $\lm(\tau)$.   In fact, using the above lower bound of $t-\tau$, similar to \eqref{lm-gap1} and \eqref{lm-gap2}, we arrive at 
\begin{align}\label{lm-gap3}
  \fr{\lm(\tau)-\lm(t)}{\lambda(\tau)}\ge\underline{c}\tau^{-a}|l|^{-1},\quad{\rm and}\quad  \fr{\lambda(t)}{\lambda(\tau)}
    \le e^{-\underline{c}\tau^{-a}|l|^{-1}}.
\end{align}
Then for 
\begin{align}\label{restrction-s1}
    s>\fr13,\quad{\rm with}\quad 0<a\le\min\left\{\fr{3s-1}{2}, \fr{s}{2}\right\},
\end{align}
there exists a positive constant $\Theta$ satisfying 
\begin{align}\label{Theta1}
    \fr{1}{s-a}\le\Theta\le \min\left\{\fr{1}{a}, \fr{3}{1-a} \right\}.
    \end{align} 
It follows from \eqref{lm-gap3}--\eqref{Theta1} that,  for $|n|>0$,
\begin{align}\label{lm-gap5}
    \left(\fr{\lm(t)}{\lm(\tau)}\right)^{|n|}\fr{|l\tau|}{|l|^3}{\bf 1}_D\les \fr{\tau^{a\Theta}|l|^{\Theta}}{|n|^\Theta}\fr{|l\tau|}{|l|^3}=\fr{| l\tau|^{1+a\Theta}}{|n|^\Theta}\fr{|l|^{\Theta}}{|l|^{3+a\Theta}}\les \fr{| l\tau|^{1+a\Theta}}{|n|^\Theta}.
\end{align}
Let us denote 
\begin{align*}
    {\rm w}_1(\tau,l,n)=\begin{cases}
\fr{| l\tau|^{1+a\Theta}}{|n|^\Theta},\quad {\rm if}\quad |n|>0;\\[2mm]
\fr{|l\tau|}{|l|^3},\quad\quad \ \ {\rm if}\quad |n|=0.
    \end{cases}
\end{align*}
Then we infer from \eqref{lm-gap5} that
\begin{align}\label{lm-gap5'}
\left(\fr{\lm(t)}{\lm(\tau)}\right)^{|n|}\fr{|l\tau|}{|l|^3}{\bf 1}_D\les{\rm w}_1(\tau,l,n),\quad{\rm for\ \ all}\ \ n\in\N^6.
\end{align}
On the other hand, note that
\begin{align}\label{up-kt-ltau}
\la kt-l\tau \ra^2\les\la\eta+(k-l)\tau\ra^2\la k(t-\tau)-\eta\ra^2,
\end{align}
and 
\begin{align}\label{time-Jacobian}
    \int_0^\infty\fr{|l|}{\la kt-l\tau \ra^2}d\tau<\infty, \quad\int_0^\infty\fr{|k|}{\la kt-l\tau \ra^2}dt<\infty.
\end{align}
Then by \eqref{est-eta},  Corollaries \ref{coro-kernel}, \ref{coro-convolution}, and Lemma \ref{lem-interp1}, recalling the definition of $\|\cdot\|_{\bar{\mathcal{G}}^{\lm,N}_{s,\ell}}$ in \eqref{def-varGnorm}, we have
\begin{align}\label{est-resoance}
   \nn&{\bf T}^{(1);{\rm HL}}_{D}[E_1]_{\tau\approx t}\\
   \nn\les& \int_0^{T}\sum_{k\in\Z^3_*}\sum_{|\beta|\le N-1}\kappa^{2|\beta|}\sum_{m\in\N^6}\bigg[\int_{\fr{2t}{3}}^t\int_{\R^3}\sum_{\substack{\beta'\le\beta,|\beta'|>|\beta|/2\\\beta''\le\beta-\beta'}}\sum_{\substack{n\le m\\n'\le m-n}}\sum_{l\in\Z^3_*,l\ne k} \\
\nn&\fr{|l|b_{m,n,n',s}}{\la k-l\ra^2\la kt-l\tau \ra^2}\Big|\la \tau\ra^{b}(l,l\tau)^{n+\beta'}a_{n,s}(\tau)|l|^{\fr32}\hat{\rho}^{(1)}_l(\tau)\fr{l\tau}{|l|^3}\Big(\fr{\lm(t)}{\lm(\tau)}\Big)^{|n|}{\bf1}_{D}\Big|\\
\nn&\times\Big| a_{n',s}(\tau)(k-l,\eta+(k-l)\tau)^{n'+\beta''}\la\eta+(k-l)\tau\ra^2\mathcal{F}_{x,v}\big[|\nb_x|^{\fr{7}{2}}|Z|f\big]_{k-l}(\tau,\eta)\Big|\\
\nn&\times \Big|a_{m-n-n',s}(\tau)(0,-\eta+k(t-\tau))^{m-n-n'+\beta-\beta'-\beta''}\la-\eta+k(t-\tau)\ra^2\bar{\hat{\mathfrak{g}}}_k(t-\tau,\eta)\Big| d\eta d\tau\bigg]^2dt\\
\nn\les& \int_0^{T}\sum_{k\in\Z^3_*}\sum_{|\beta|\le N-1}\kappa^{2|\beta|}\sum_{m\in\N^6}\bigg(\int_{\fr{2t}{3}}^t\sum_{\substack{\beta'\le\beta,|\beta'|>|\beta|/2\\\beta''\le\beta-\beta'}}\sum_{\substack{n\le m\\n'\le m-n}}\sum_{l\in\Z^3_*,l\ne k}\fr{|l|b_{m,n,n',s}}{\la k-l\ra^4\la kt-l\tau \ra^2}d\tau\bigg) \\
\nn&\times\sum_{\substack{\beta'\le\beta,|\beta'|>|\beta|/2\\\beta''\le\beta-\beta'}}\sum_{\substack{n\le m\\n'\le m-n}}\sum_{l\in\Z^3_*,l\ne k}\int^{t}_{\fr{2t}{3}} \fr{|k|}{\la kt-l\tau \ra^2}\\
\nn&\times\Big|\la \tau\ra^{b}(l,l\tau)^{n+\beta'}a_{n,s}(\tau)|l|^{\fr32}\hat{\rho}^{(1)}_l(\tau){\rm w}_1(\tau,l,n)\Big|^2\\
\nn&\times\Big\| a_{n',s}(\tau)\mathcal{F}_{x,v}\big[\la Y\ra^2|\nb_x|^{4}|Z|f^{(n'+\beta'')}\big]_{k-l}(\tau,\eta)\Big\|_{L^2_\eta}^2\\
\nn&\times \Big\|a_{m-n-n',s}(0)\mathcal{F}_{x,v}\big[\la Y^*\ra^2\mathfrak{g}^{(m-n-n'+\beta-\beta'-\beta'')}\big]_k(t-\tau)\Big\|_{L^2_\eta}^2  d\tau dt \\
\les\nn&\int_0^{T}\sum_{k\in Z^3_*}\sum_{l\in\Z^3_*,l\ne k}\int_{\tau}^{T}\fr{|k|}{\la kt-l\tau\ra^2}dt\sum_{\substack{|\beta|\le N-1\\ m\in\N^6}}\kappa^{2|\beta|}\Big|\la \tau\ra^{b}(l,l\tau)^{m+\beta}a_{m,s}(\tau)|l|^{\fr32}\hat{\rho}^{(1)}_l(\tau){\rm w}_1(\tau,l,n)\Big|^2\\
\nn&\times \sum_{\substack{|\beta|\le \fr{N-1}{2}\\ m\in\N^6}}\kappa^{2|\beta|}\Big\| a_{m,s}(\tau)\mathcal{F}_{x}\big[\la Y\ra^2|\nb_x|^{4}|Z|f^{(m+\beta)}\big]_{k-l}(\tau,v)\Big\|_{L^2_v}^2d\tau \\
\nn&\times \sup_{t-\tau}\sum_{\substack{|\beta|\le \fr{N-1}{2}\\ m\in\N^6}}\kappa^{2|\beta|}\Big\| a_{m,s}(0)\mathcal{F}_{x}\big[\la Y^*\ra^2\mathfrak{g}^{(m+\beta)}\big]_k(t-\tau)\Big\|_{L^2_v}^2 \\
\les&\left\| \la t\ra^b |\nb_x|^{\fr32}\rho^{(1)}\right\|^2_{L^2_t\mathcal{G}^{\lm,N-1}_s}{ \left\|\la Y\ra^2|\nb_x|^4|Z|f(\tau)\right\|^2_{L^\infty_t\mathcal{G}^{\lm,\fr{N-1}{2}}_{s,0}}}\left\|\la Y^*\ra^2\mathfrak{g}\right\|^2_{  L^\infty_t\bar{\mathcal{G}}^{\lm(0),\fr{N-1}{2}}_{s,0}}.
\end{align}
The later case when $|kt-l\tau|>\fr{t}{2}$ is much easier to treat, because the $\tau$ loss  can directly be absorbed by $\fr{1}{\la kt-l\tau\ra}$. Thus, the above upper bound for ${\bf T}^{(1);{\rm HL}}_{D}[E_1]_{\tau\approx t}$ is still valid for ${\bf T}^{(1);{\rm HL}}_{D^*}[E_1]_{\tau\approx t}$ as long as one pays one more derivative on $f$ and $\mathfrak{g}$:
\begin{align}\label{est-nonresonace}
    {\bf T}^{(1);{\rm HL}}_{D^*}[E_1]_{\tau\approx t}\les \left\| \la t\ra^b |\nb_x|^{\fr32}\rho^{(1)}\right\|^2_{L^2_t\mathcal{G}^{\lm,N-1}_s}{ \left\|\la Y\ra^3|\nb_x|^4|Z|f\right\|^2_{L^\infty_t\mathcal{G}^{\lm,\fr{N-1}{2}}_{s,0}}}\left\|\la Y^*\ra^3\mathfrak{g}\right\|^2_{L^\infty_t\bar{\mathcal{G}}^{\lm(0),\fr{N-1}{2}}_{s,0}}.
\end{align}

{\it  Case 2.2: $l=k$.} Here there is no need to use the gap between $\lm(t)$ and $\lm(\tau)$. Similar to \eqref{est-resoance}, we have
\begin{align}\label{est-T1E1-l=k}
 \nn&{\bf T}^{(1);{\rm HL}}[E_1]_{\tau\approx t}\\
   \nn\les& \int_0^{T}\sum_{k\in\Z^3_*}\sum_{|\beta|\le N-1}\kappa^{2|\beta|}\sum_{m\in\N^6}\bigg[\int_{\fr{2t}{3}}^t\int_{\R^3}\sum_{\substack{\beta'\le\beta,|\beta'|>|\beta|/2\\\beta''\le\beta-\beta'}}\sum_{\substack{n\le m\\n'\le m-n}} \fr{b_{m,n,n',s}}{\la k(t-\tau) \ra^2}\\
\nn&\times\Big|\la \tau\ra^{b}|k|^{\fr32}a_{n,s}(\tau)(k,k\tau)^{n+\beta'}\hat{\rho}^{(1)}_k(\tau)\fr{k}{|k|^2}\Big|\\
\nn&\times\Big| a_{n',s}(\tau)(0,\eta)^{n'+\beta''}\la\eta\ra^2\mathcal{F}_{x,v}[\nb_vf]_{0}(\tau,\eta)\Big|\\
\nn&\times \Big|a_{m-n-n',s}(\tau)(0,-\eta+k(t-\tau))^{m-n-n'+\beta-\beta'-\beta''}\la-\eta+k(t-\tau)\ra^2\bar{\hat{\mathfrak{g}}}_k(t-\tau,\eta)\Big| d\eta d\tau\bigg]^2dt\\
\les&\left\| \la t\ra^b |\nb_x|^{\fr32}\rho^{(1)}\right\|^2_{L^2_t\mathcal{G}^{\lm,N-1}_s}{ \left\|\la \nb_v\ra^2\nb_vf_0\right\|^2_{L^\infty_t\mathcal{G}^{\lm,\fr{N-1}{2}}_{s,0}}}\left\|\la Y^*\ra^2\mathfrak{g}\right\|^2_{  L^\infty_t\bar{\mathcal{G}}^{\lm(0),\fr{N-1}{2}}_{s,0}}.
\end{align}

Now we turn to ${\bf T}^{(1);\rm HL}[E_1] $ where $\rho^{(1)}$ is at low frequency. Now the key issue is to deal with the summability in $k$. On the one hand,  noting that $|k|\le |l|\la k-l\ra$, one derivative in $x$ can hit on $f$  because $\|\nb_xf\|_{\mathcal{G}^{\lm,N}_{s,\ell-2}}$ is a constituent part of the energy functional of the distribution function $f$. More precisely, in view of \eqref{est-eta} again, one deduces that
\begin{align*}
    {\bf T}^{(1);{\rm LH}}[E_1]_{\tau\approx t}
&\les\int_0^{T}\sum_{k\in\Z^3_*}|k|\sum_{|\beta|\le N-1}\kappa^{2|\beta|}\sum_{m\in\N^6}\Bigg[\int_{\fr{2t}{3}}^t\int_{\R^3}\\
\nn&\sum_{\substack{\beta'\le\beta,|\beta'|\le|\beta|/2\\\beta''\le\beta-\beta'}}\sum_{\substack{n\le m\\n'\le m-n}}\sum_{l\in\Z^3_*} b_{m,n,n',s}\Big|\la \tau\ra^{b}(l,l\tau)^{n+\beta'}a_{n,s}(\tau)\hat{\rho}^{(1)}_l(\tau)\fr{l\tau}{|l|}\Big|\\
\nn&\times\Big| a_{n',s}(\tau)(k-l,\eta+(k-l)\tau)^{n'+\beta''}\mathcal{F}_{x,v}[\la\nb_x\ra |Z|f]_{k-l}(\tau,\eta)\Big|\\
&\times \Big|a_{m-n-n',s}(\tau)(0,-\eta+k(t-\tau))^{m-n-n'+\beta-\beta'-\beta''}\hat{\mathfrak{g}}_k(t-\tau,\eta)\Big| d\eta d\tau\Bigg]^2dt.
\end{align*}
 On the other hand, the $|k|^{\fr12}$ can be absorbed by the Jacobian when integrating with respect to the $t$ variable. Indeed, note that
\begin{align*}
    &\bigg[\int_{\R^3}\Big|a_{n',s}(\tau)(k-l,\eta+(k-l)\tau)^{n'+\beta''}\mathcal{F}_{x,v}[\la\nb_x\ra |Z|f]_{k-l}(\tau,\eta)\Big|\\
&\times \Big|a_{m-n-n',s}(\tau)(0,-\eta+k(t-\tau))^{m-n-n'+\beta-\beta'-\beta''}\hat{\mathfrak{g}}_k(t-\tau,\eta)\Big| d\eta\bigg]^2\\
\les&\int_{\R^3}\fr{1}{\la k(t-\tau)-\eta\ra^2}a_{n',s}^2(\tau)\Big|\mathcal{F}_{x,v}\big[\la\nb_x\ra |Z| f^{(n'+\beta'')}\big]_{k-l}(\tau,\eta)\Big|^2d\eta\\
&\times \Big\|a_{m-n-n',s}(0)  \mathcal{F}_{x,v}\big[\mathcal\la Y^*\ra \mathfrak{g}^{(m-n-n'+\beta-\beta'-\beta'')}\big]_k(t-\tau)\Big\|^2_{L^2_\eta}.
\end{align*}
Combining this  with Corollaries \ref{coro-kernel} and \ref{coro-convolution}, we are led to
\begin{align}\label{est-T1LHE1}
    {\bf T}^{(1);{\rm LH}}[E_1]_{\tau\approx t}
 \les\nn&\int_0^{T}\sum_{k\in\Z^3_*}\sum_{|\beta|\le N-1}\kappa^{2|\beta|}\sum_{m\in\N^6}\Bigg(\int_{\fr{2t}{3}}^t\sum_{\substack{\beta'\le\beta,|\beta'|\le|\beta|/2\\\beta''\le\beta-\beta'}}\sum_{\substack{n\le m\\n'\le m-n}}\sum_{l\in\Z^3_*} \fr{b_{m,n,n',s}}{\la l,l\tau\ra^{6}}d\tau\Bigg)\\
\nn&\times \sum_{l\in\Z^3_*}\int_{\fr{2t}{3}}^t\sum_{\substack{\beta'\le\beta,|\beta'|\le|\beta|/2\\\beta''\le\beta-\beta'}}\sum_{\substack{n\le m\\n'\le m-n}}  \Big|\la\tau\ra^{b} a_{n,s}(\tau)\la l,l\tau\ra^4(l,l\tau)^{n+\beta'}|l|^{-1}\hat{\rho}^{(1)}_l(\tau)\Big|^2\\
\nn&\times\int_{\R^3}\fr{|k|}{\la k(t-\tau)-\eta\ra^2}a_{n',s}^2(\tau)\Big|\mathcal{F}_{x,v}\big[\la\nb_x\ra |Z| f^{(n'+\beta'')}\big]_{k-l}(\tau,\eta)\Big|^2d\eta\\
\nn&\times \Big\|a_{m-n-n',s}(0)  \mathcal{F}_{x,v}\big[\mathcal\la Y^*\ra \mathfrak{g}^{(m-n-n'+\beta-\beta'-\beta'')}\big]_k(t-\tau)\Big\|^2_{L^2_\eta} d\tau dt\\
\les\nn&\int_0^{T}
\left\|\la \tau\ra^b |Z|^4|\nb_x|^{-1}\rho^{(1)}(\tau) \right\|^2_{\mathcal{G}^{\lm,\fr{N-1}{2}}_{s}} \left\| \la Y^*\ra\mathfrak{g}\right\|^2_{L^\infty_t\bar{\mathcal{G}}^{\lm(0),N-1}_{s,0}}\\
\nn&\times\sup_{l\in\Z^3_*}\sum_{k\in\Z^3_*}\sum_{|\beta|\le N-1}\kappa^{2|\beta|}\sum_{m\in\N^6}\int_{\R^3}\int_{\tau}^{T^*}\fr{|k|}{\la k(t-\tau)-\eta\ra^2}dt\\
\nn&\times a_{m,s}^2(\tau)\Big|\mathcal{F}_{x,v}\big[\la\nb_x\ra |Z| f^{(m+\beta)}\big]_{k-l}(\tau,\eta)\Big|^2d\eta d\tau \\
\les&\left\|\la t\ra^b |Z|^4|\nb_x|^{-1}\rho^{(1)} \right\|^2_{L^2_t\mathcal{G}^{\lm,\fr{N-1}{2}}_{s}}{ \left\| \la\nb_x\ra |Z|f\right\|^2_{L^\infty_t\mathcal{G}^{\lm,N-1}_{s,0}}}\left\| \la Y^*\ra\mathfrak{g}\right\|^2_{L^\infty_t\bar{\mathcal{G}}^{\lm(0),N-1}_{s,0}},
\end{align}
where we have used
\begin{align*}
    \int_{0}^\infty\fr{|k|}{\la k(t-\tau)-\eta\ra^2}dt \lesssim 1.
\end{align*}

\noindent{$\diamond$ \underline{\it Treatments of  $E_1\cdot vf$}.} Compared with $E_1\cdot \nb_vf$, the treatments of $E_1\cdot vf$ is easier, since there is no extra $v$ derivative on $f$ and thus the plasma echo appearing in ${\bf T}^{(1);{\rm HL}}_{D}[E_1]_{\tau\approx t}$ is absent now. More precisely, we write
\begin{align*}
    \tl{\bf T}^{(1)}[E_1]=&\int_0^{T}\la t\ra^{2b}\sum_{k\in\Z^3_*}|k|^3\sum_{|\beta|\le N-1}\kappa^{2|\beta|}\sum_{m\in\N^6}\Bigg|a_{m,s}(t)(k,kt)^{m+\beta}\\
    &\times \int_0^t\int_{\R^3}S_k(t-\tau)\big[\mathcal{F}_x[E_1\cdot vf]_k(\tau,v)\big]\mu^{\fr12}(v)dvd\tau\Bigg|^2dt\\
    =&\int_0^{T}\la t\ra^{2b}\sum_{k\in\Z^3_*}|k|^3\sum_{|\beta|\le N-1}\kappa^{2|\beta|}\sum_{m\in\N^6}\Bigg|a_{m,s}(t)(k,kt)^{m+\beta}\\
    &\times \int_0^t\int_{\R^3}\sum_{l\in\Z^3_*}\hat{\rho}^{(1)}_l(\tau)\fr{l}{|l|^2}\hat{f}_{k-l}(\tau,\eta)\cdot\overline{\mathcal{F}_v[v\mathfrak{g}_k]}(t-\tau,\eta)d\eta d\tau\Bigg|^2dt.
\end{align*}
Then similar to \eqref{est-T1E1-short}, \eqref{est-resoance} (without resorting to the gap between $\lm(t)$ and $\lm(\tau)$ and \eqref{est-eta}), \eqref{est-T1E1-l=k} and \eqref{est-T1LHE1}, we arrive at
\begin{align}\label{est-tlT1E1}
    \tl{\bf T}^{(1)}[E_1]\les\nn&\left\||\nb_x|^{-1}\rho^{(1)}\right\|^2_{L^2_t\mathcal{G}^{\lm,N-1}_s}{ \left\| f \right\|^2_{L^\infty_t  \mathcal{G}^{\lm, N-1}_{s,0}}}
    \left\| v\mathfrak{g}\right\|^2_{L^\infty_t  \bar{\mathcal{G}}^{\lm(0), N-1}_{s,0}}\\
    \nn&+\left\| \la \tau\ra^b |\nb_x|^{\fr32}\rho^{(1)}(\tau)\right\|^2_{L^2_t\mathcal{G}^{\lm,N-1}_s}\Big({ \left\|\la Y\ra^2|\nb_x|^4f(\tau)\right\|^2_{L^\infty_t\mathcal{G}^{\lm,\fr{N-1}{2}}_{s,0}}}+{ \left\|\la \nb_v\ra^2f_0\right\|^2_{L^\infty_t\mathcal{G}^{\lm,\fr{N-1}{2}}_{s,0}}}\Big)\\
    \nn&\times\left\|\la Y^*\ra^2(v\mathfrak{g})\right\|^2_{  L^\infty_t\bar{\mathcal{G}}^{\lm(0),\fr{N-1}{2}}_{s,0}}\\
    &+\left\|\la t\ra^b |Z|^4|\nb_x|^{-1}\rho^{(1)} \right\|^2_{L^2_t\mathcal{G}^{\lm,\fr{N-1}{2}}_{s}}{ \left\| \la\nb_x\ra f\right\|^2_{L^\infty_t\mathcal{G}^{\lm,N-1}_{s,0}}}\left\| \la Y^*\ra(v\mathfrak{g})\right\|^2_{L^\infty_t\bar{\mathcal{G}}^{\lm(0),N-1}_{s,0}}.
\end{align}

 \noindent{$\diamond$ \underline{\it Treatments of $\mathcal{T}'_{E_2^j}\pr_{v_j}f$}.} Like \eqref{T1E1}, recalling the definition of the Bony paraproduct decomposition in Section \ref{sec-Bony}, we write
\begin{align*}
    {\bf T}^{(1)}[E_2]=&\left\| \la t\ra^{b}|\nb_x|^{\fr{3}{2}}\int_0^t\int_{\R^3}S_k(t-\tau)\Big[\mathcal{F}_x\big[\mathcal{T}'_{E^j_2}\pr_{v_j}f\big]_k(\tau,v)\Big]\mu^{\fr12}(v)dvd\tau\right\|^2_{L^2_t\mathcal{G}^{\lm,N-1}_{s}}\\
=&\int_0^{T}\la t\ra^{2b}\sum_{k\in\Z^3_*}|k|^{3}\sum_{|\beta|\le N-1}\kappa^{2|\beta|}\sum_{m\in\N^6}\Bigg|a_{m,s}(t)(k,kt)^{m}\\
&\times\int_0^t\int_{\R^3}\Big(\sum_{\substack{\beta'\le\beta,|\beta'|\le|\beta|/2\\\beta''\le\beta-\beta'}}+\sum_{\substack{\beta'\le\beta,|\beta'|>|\beta|/2\\\beta''\le\beta-\beta'}}\Big)C_\beta^{\beta'}C_{\beta-\beta'}^{\beta''}\sum_{l\in\Z^3_*}\sum_{J\in\mathbb{B}}(l,l\tau)^{\beta'}\hat{\rho}^{(2)}_l(\tau)_{<16J}\fr{l_j}{|l|^2}\\
&\times \eta_j(k-l,\eta+(k-l)\tau)^{\beta''}\hat{f}_{k-l}(\tau,\eta)_{J}(0,-\eta+k(t-\tau))^{\beta-\beta'-\beta''}\bar{\hat{\mathfrak{g}}}_k(t-\tau,\eta) d\eta d\tau\Bigg|^2dt\\
=&{\bf T}^{(1);\rm LH}[E_2]+{\bf T}^{(1);\rm HL}[E_2].
\end{align*}
As the treatment of ${\bf T}^{(1)}[E_1]$, if $\tau\le \fr{2t}{3}$, one can treat ${\bf T}^{(1)}[E_2]$ without resorting to the low-high and high-low decomposition. In fact,
noting that 
\begin{align*}
    \nn&\sum_{J\in \mathbb{B}}\big|\hat{\rho}^{(2)}_l(\tau)_{<16J}\big| |\hat{f}_{k-l}(\tau,\eta)_{J}|\\
    =&\sum_{J\in \mathbb{B}}\varphi\Big(\fr{|l,l\tau|}{16J}\Big)\theta_J(|k-l,\eta+(k-l)\tau|)|\hat{\rho}^{(2)}_l(\tau)\hat{f}_{k-l}(\tau,\eta)|\le\big|\hat{\rho}^{(2)}_l(\tau)\hat{f}_{k-l}(\tau,\eta)\big|,
\end{align*}
then thanks to \eqref{est-eta},  similar to \eqref{est-T1E1-short}, one can ignore the frequency cutoff to obtain 
\begin{align}\label{est-T1E2-1}
    {\bf T}^{(1)}[E_2]_{\tau\le\fr{2t}{3}}\les\nn&\int_0^{T}\int^{\fr{2t}{3}}_0 \fr{1}{\la t\ra^2} \left\||\nb_x|^{-1}\rho^{(2)}(\tau)\right\|^2_{\mathcal{G}^{\lm,N-1}_s}\big\||Z| f(\tau) \big\|^2_{\mathcal{G}^{\lm,N-1}_{s,0}}\left\| \mathfrak{g}(t-\tau)\right\|^2_{\mathcal{G}^{\lm(0), N-1}_{s,0}}d\tau dt\\
    \les&\left\||\nb_x|^{-1}\rho^{(2)}\right\|^2_{L^2_t\mathcal{G}^{\lm,N-1}_s}{ \left\| f \right\|^2_{L^\infty_t  \mathcal{G}^{\lm, N}_{s,0}}}
    \left\| \mathfrak{g}\right\|^2_{L^\infty_t  \bar{\mathcal{G}}^{\lm(0), N-1}_{s,0}}.
\end{align}

If $\fr{2t}{3}\le\tau\le t$, ${\bf T}^{(1);\rm LH}[E_2]_{\tau\approx t}$ can be handled following the approach of treating ${\bf T}^{(1);\rm LH}[E_1]_{\tau\approx t}$:
\begin{align}\label{est-T1E2-2}
    {\bf T}^{(1);\rm LH}[E_2]_{\tau\approx t}\les\nn&\left\|\la t\ra^b |Z|^4|\nb_x|^{-1}\rho^{(2)} \right\|^2_{L^2_t\mathcal{G}^{\lm,\fr{N-1}{2}}_{s}}\big\| |Z|\la\nb_x\ra f\big\|^2_{L^\infty_t\mathcal{G}^{\lm,N-1}_{s,0}}\left\| \la Y^*\ra\mathfrak{g}\right\|^2_{L^\infty_t\bar{\mathcal{G}}^{\lm(0),N-1}_{s,0}}\\
    \les&\left\|\la t\ra^{\fr{3+a-s}{2}} \rho^{(2)} \right\|^2_{L^2_t\mathcal{G}^{\lm,N}_{s}}{ \big\| \la\nb_x\ra f\big\|^2_{L^\infty_t\mathcal{G}^{\lm,N}_{s,0}}}\left\| \mathfrak{g}\right\|^2_{L^\infty_t\bar{\mathcal{G}}^{\lm(0),N}_{s,0}},
\end{align}
provided
\begin{align*}
b+4-\fr{3+a-s}{2}-\fr12\le\fr{N}{2}.
\end{align*}
For ${\bf T}^{(1);\rm HL}[E_2]_{\tau\approx t}$,  note that on the support of the integrand above, there holds
\begin{align*}
     |l,l\tau|\le\fr34\cdot16J,\quad \fr{J}{2}\le|k-l,\eta+(k-l)\tau|\le\fr{3J}{2}, 
\end{align*}
and hence
\begin{align}\label{exchange}
    |l,l\tau|\le 24|k-l,\eta+(k-l)\tau|.
\end{align}
Then by \eqref{exchange},  one can move the extra growth in $\tau$  on $\rho^{(2)}(\tau)$ to  $f$. More precisely, 
combining \eqref{exchange} with \eqref{est-eta}, using  \eqref{up-kt-ltau} and \eqref{time-Jacobian}, similar to \eqref{est-resoance}, we find that
\begin{align}\label{est-T1E2-3}
   {\bf T}^{(1);\rm HL}[E_2]_{\tau\approx t}\les\nn&\int_0^{T}\sum_{k\in\Z^3_*}\sum_{|\beta|\le N-1}\kappa^{2|\beta|}\sum_{m\in\N^6}\Bigg[\int_{\fr{2t}{3}}^t\int_{\R^3}\sum_{\substack{\beta'\le\beta,|\beta'|>|\beta|/2\\\beta''\le\beta-\beta'}}\sum_{\substack{n\le m\\ n'\le m-n}}\sum_{l\in\Z^3_*}\fr{|l|b_{m,n,n',s}}{\la k-l\ra^2\la kt-l\tau \ra^2}\\
\nn&\times\Big|a_{n,s}(\tau)\la \tau\ra^{b+1}(l,l\tau)^{n+\beta'}|l|^{-\fr{1}{2}} \la\tau\ra^{\fr{3+a-s}{2}-b-1}\hat{\rho}^{(2)}_l(\tau)\Big|\\
\nn&\times \Big|a_{n',s}(\tau)\mathcal{F}_{x,v}\big[\la\nb_x\ra^{\fr72}\la Y\ra^2|Z|^{2+b-\fr{3+a-s}{2}}f^{(n'+\beta'')}\big]_{k-l}(\tau,\eta)\Big|\\
\nn&\times\Big|a_{m-n-n',s}(0)\mathcal{F}_{x,v}\big[\la Y^*\ra^2\mathfrak{g}^{(m-n-n'+\beta-\beta'-\beta'')}\big]_k(t-\tau,\eta)\Big| d\eta d\tau\Bigg]^2dt\\
\les\nn&\int_0^{T}\sum_{k\in\Z^3_*}\sum_{l\in\Z^3_*}\Big(\int_0^t\sum_{\substack{\beta'\le\beta,|\beta'|>|\beta|/2\\\beta''\le\beta-\beta'}}\sum_{\substack{n\le m\\ n'\le m-n}}\sum_{l\in\Z^3_*}\fr{|l|b_{m,n,n',s}}{\la k-l\ra^4\la kt-l\tau \ra^2}d\tau\Big)\\
\nn&\times \int_{\tau}^{T}\fr{|k|}{\la kt-l\tau \ra^2}dt\sum_{\substack{|\beta|\le N-1\\ m\in\N^6}}\kappa^{2|\beta|}\Big|a_{m,s}(\tau)(l,l\tau)^{m+\beta}\la \tau\ra^{\fr{3+a-s}{2}}|l|^{-\fr{s}{2}}\hat{\rho}^{(2)}_l(\tau)\Big|^2\\
\nn&\times \sum_{\substack{|\beta|\le\fr{N-1}{2}\\ m\in\N^6}}\kappa^{2|\beta|}\Big\|a_{m,s}(\tau)\mathcal{F}_{x}\big[\la\nb_x\ra^{4}\la Y\ra^2|Z|^{2+b-\fr{3+a-s}{2}}f^{(m+\beta)}\big]_{k-l}(\tau)\Big\|_{L^2_v}^2d\tau\\
\nn&\times \left\| \la Y^*\ra^2\mathfrak{g}\right\|^2_{L^\infty_t\bar{\mathcal{G}}^{\lm(0),\fr{N-1}{2}}_{s,0}} \\
\les\nn&\left\|\la t\ra^{\fr{3+a-s}{2}}\rho^{(2)}\right\|_{L^2_t\mathcal{G}^{\lm,N-1}_{s}}^2{ \big\| \la\nb_x\ra^4\la Y\ra^2 |Z|^{2+b-\fr{3+a-s}{2}}f\big\|^2_{L^\infty_t\mathcal{G}^{\lm,{\fr{N-1}{2}}}_{s,0}}}\\
&\times\left\| \la Y^*\ra^2\mathfrak{g}\right\|^2_{L^\infty_t\bar{\mathcal{G}}^{\lm(0),\fr{N-1}{2}}_{s,0}}.
\end{align}

 \noindent{$\diamond$ \underline{\it Treatments of $\mathcal{T}'_{E_2^j}v_jf$}.} Like the treatments of $\mathcal{T}'_{E_2^j}\pr_{v_j}f$, we write
\begin{align*}
    \tl{\bf T}^{(1)}[E_2]=&\int_0^{T}\la t\ra^{2b}\sum_{k\in\Z^3_*}|k|^3\sum_{|\beta|\le N-1}\sum_{m\in\N^6}\Bigg|a_{m,s}(t)(k,kt)^{m+\beta}\\
    &\times\int_0^t\int_{\R^3}\sum_{l\in\Z^3_*}\sum_{J\in\mathbb{B}}\hat{\rho}^{(2)}_l(\tau)_{<16J}\fr{l_j}{|l|^2} \mathcal{F}_{x,v}[v_jf]_{k-l}(\tau,\eta)_{J}\bar{\hat{\mathfrak{g}}}_k(t-\tau,\eta) d\eta d\tau\Bigg|^2dt.
\end{align*}
Then similar to \eqref{est-T1E2-1}, \eqref{est-T1E2-2} and \eqref{est-T1E2-3} without resorting to \eqref{est-eta}, we  obtain
\begin{align}\label{est-tlT1E2}
\tl{\bf T}^{(1)}[E_2]\les\nn&\left\||\nb_x|^{-1}\rho^{(2)}\right\|^2_{L^2_t\mathcal{G}^{\lm,N-1}_s}{ \left\|v f \right\|^2_{L^\infty_t  \mathcal{G}^{\lm, N-1}_{s,0}}}
    \left\| \mathfrak{g}\right\|^2_{L^\infty_t  \bar{\mathcal{G}}^{\lm(0), N-1}_{s,0}}\\
    \nn&+\left\|\la t\ra^{\fr{3+a-s}{2}} \rho^{(2)} \right\|^2_{L^2_t\mathcal{G}^{\lm,N}_{s}}{ \big\| \la\nb_x\ra (vf)\big\|^2_{L^\infty_t\mathcal{G}^{\lm,N-1}_{s,0}}}\left\| \mathfrak{g}\right\|^2_{L^\infty_t\bar{\mathcal{G}}^{\lm(0),N}_{s,0}}\\
    \nn&+\left\|\la t\ra^{\fr{3+a-s}{2}}\rho^{(2)}\right\|_{L^2_t\mathcal{G}^{\lm,N-1}_{s}}^2{ \big\| \la\nb_x\ra^4\la Y\ra^2 |Z|^{1+b-\fr{3+a-s}{2}}(vf)\big\|^2_{L^\infty_t\mathcal{G}^{\lm,{\fr{N-1}{2}}}_{s,0}}}\\
&\times\left\| \la Y^*\ra^2\mathfrak{g}\right\|^2_{L^\infty_t\bar{\mathcal{G}}^{\lm(0),\fr{N-1}{2}}_{s,0}}.
\end{align}

Collecting the estimates \eqref{est-rho1-initial}, \eqref{est-T1E1-short}, \eqref{est-resoance}--\eqref{est-T1E2-2}, \eqref{est-T1E2-3} and \eqref{est-tlT1E2}, and noting that the estimate for $\|v\mathfrak{g}\|_{\bar{\mathcal{G}}^{\lm(0),N_0}_{s,\ell_0}}$ in \eqref{est-vg} applies to $\|vf\|_{\mathcal{G}^{\lm,N}_{s,0}}$, then using Lemma \ref{lem:est-S^*}, \eqref{fg1}, \eqref{fg2}, and \eqref{Linfty-rho}, under the restrictions of \eqref{res-N-b} and \eqref{restrction-s-rho1}, we get \eqref{bd-rho1}.
\end{proof}

\subsection{Estimates for $\rho^{(2)}$} The aim of this subsection is to bound $\rho^{(2)}$. We first establish the following proposition for the $L^2_t\mathcal{G}^{\lm,N}_s$ estimate of $\la t\ra^{\fr{3+a-s}{2}}\rho^{(2)}$.
\begin{prop}\label{prop-rho2-0}
    Under the bootstrap hypotheses \eqref{bd-en}--\eqref{H-phi}, if
\begin{align}\label{restr-N}
    N\ge16,
\end{align}
and 
\begin{align}\label{res-s-a-rho2}
    s>\fr12,\quad 0<a<s-\fr12, \quad{\rm and}\quad a\le\fr{s}{6},
\end{align}
then  there holds
    \begin{align}\label{bd-rho2-iota=0}
        \nn&\big\|\la t\ra^{\fr{3+a-s}{2}}\rho^{(2)}\big\|^2_{L^2_t\mathcal{G}^{\lm,N}_s}\\
        \les\nn&\| g\|^2_{L^\infty_t\mathcal{G}^{\lm,N}_{s,2}}\big\|\la t\ra^{\fr{3+a-s}{2}}\rho^{(2)}\big\|^2_{L^2_t\mathcal{G}^{\lm,N}_s}+\|g\|^2_{L^\infty_t\mathcal{G}^{\lm,N}_{s,2}}\sum_{|\al|\le2}\big\|\nu^{\fr{|\al|}{3}}\pr_v^\al g\big\|^2_{L^\infty_t\mathcal{G}^{\lm,N}_{s,2}}\\
        \nn&+\|\la\nb_x\ra g\|^2_{L^\infty_t\mathcal{G}^{\lm,N}_{s,2}}\big\|\la t\ra^{\fr{3+a-s}{2}}\rho\big\|^2_{L^2_t\mathcal{G}^{\lm,N}_s}\sum_{|\al|\le2}\big\|\nu^{\fr{|\al|}{3}}\pr_v^{\al}g\big\|^2_{L^\infty_t\mathcal{G}^{\lm,N}_{s,2}}\\
        \nn&+\|\la\nb_x\ra g\|^2_{L^\infty_t\mathcal{G}^{\lm,N}_{s,2}}\Big(\nu^{\fr13}\int_0^{T}\nu^{\fr13}\la\tau\ra\sum_{|\al|\le2}\big\|\nu^{\fr{|\al|}{3}}\pr_v^\al g_{\ne}(\tau)\big\|^2_{\mathcal{G}^{\lm,N}_{s,2}}d\tau\Big)\\
        \nn&+\|g\|^2_{L^\infty_t\mathcal{G}^{\lm,N}_{s,2}}\Big(\nu^{\fr13}\int_0^{T}\nu\la\tau\ra^3\sum_{|\al|\le1}\big\|\nu^{\fr{|\al|}{3}}\pr_v^\al g_{\ne}(\tau)\big\|^2_{\mathcal{G}^{\lm,N}_{s,2}}d\tau\Big)\\
        &+\|g\|^2_{L^\infty_t\mathcal{G}^{\lm,N}_{s,2}}\int_0^{T}(\nu\la\tau\ra^3)^2\Big(\mathfrak{CK}_0[\nb_xg(\tau)]+\fr{1}{\la\tau\ra^{1+a}}\|\nb_xg(\tau)\|^2_{\mathcal{G}^{\lm,N}_{s,0}}\Big)d\tau.
    \end{align}
\end{prop}
\begin{proof}
By \eqref{es-rho:N2} and recalling the definition of $\mathcal{N}^{(2)}_k(t)$ in Section \ref{sec: decom-rho}, to complete the proof, we will bound $\big\|
    \langle t\rangle^{\fr{3+a-s}{2}}\mathcal{N}^{(2)}\big\|_{L^2_t\mathcal{G}^{\lm,N}_s}$ term by term.

\noindent{$\diamond$ \underline{\it Estimates of $\mathcal{T}_{\pr_{v_j}f}E^j_2$}.}
Now let us investigate 
\begin{align*}
    {\bf T}^{(2)}[E_2]=&\left\| \la t\ra^{\fr{3+a-s}{2}}\int_0^t\int_{\R^3}S_k(t-\tau)\big[\mathcal{F}_x[\mathcal{T}_{\pr_{v_j}f}E^j_2]_k(\tau,v)\big]\mu^{\fr12}(v)dvd\tau\right\|^2_{L^2_t\mathcal{G}^{\lm,N}_{s}}\\
=&\int_0^{T}\la t\ra^{3+a-s}\sum_{k\in\Z^3_*}\sum_{|\beta|\le N}\kappa^{2|\beta|}\sum_{m\in\N^6}\Bigg|a_{m,s}(t)(k,kt)^{m}\\
&\times\int_0^t\int_{\R^3}\Big(\sum_{\substack{\beta'\le\beta,|\beta'|\le|\beta|/2\\\beta''\le\beta-\beta'}}+\sum_{\substack{\beta'\le\beta,|\beta'|>|\beta|/2\\\beta''\le\beta-\beta'}}\Big)C_\beta^{\beta'}C_{\beta-\beta'}^{\beta''}\sum_{l\in\Z^3_*}\sum_{J\ge8}(l,l\tau)^{\beta'}\hat{\rho}^{(2)}_l(\tau)_J\fr{l_j}{|l|^2}\\
&\times \eta_j(k-l,\eta+(k-l)\tau)^{\beta''}\hat{f}_{k-l}(\tau,\eta)_{<J/8}(0,-\eta+k(t-\tau))^{\beta-\beta'-\beta''}\bar{\hat{\mathfrak{g}}}_k(t-\tau,\eta) d\eta d\tau\Bigg|^2dt\\
=&{\bf T}^{(2);\rm LH}[E_2]+{\bf T}^{(2);\rm HL}[E_2].
\end{align*}
 ${\bf T}^{(2);\rm HL}[E_2]$ can be bounded  in a similar manner as  ${\bf T}^{(1);\rm HL}[E_1]$. We need to divide the interval $[0,t]$ into $[0,\fr{2t}{3}]$ and $[\fr{2t}{3}, t]$ as well. If $0\le \tau\le \fr{2t}{3}$, similar to \eqref{eq:gain m}, for any $\Theta>0$, $|m|>0$, we have
 \begin{align}\label{eq:gain m2}
     \la t\ra^{\fr{3+a-s}{2}+3} \left(\fr{\lm(t)}{\lm(\tau)}\right)^{|m|}\le C\fr{|kt|^{\fr{3+a-s}{2}+3+a\Theta}}{|k|^{\fr{3+a-s}{2}+3+a\Theta}|m|^\Theta}\le C\fr{|kt|^{5+a\Theta}}{|k|^{\fr{3+a-s}{2}+3+a\Theta}|m|^\Theta}.
 \end{align}
Taking $\Theta=\fr{5}{s-a}$, {  for $s\ge 6a$}, using Lemma \ref{lem-interp1} and \eqref{est-eta}, similar to \eqref{est-T1E1-short}, we are led to
 \begin{align}\label{est-T2E2HL-short}
     {\bf T}^{(2); {\rm HL}}[E_2]_{\tau\le\fr{2t}{3}}\les \left\||\nb_x|^{-1}\rho^{(2)}\right\|^2_{L^2_t\mathcal{G}^{\lm,N}_s}{ \big\||Z| f \big\|^2_{L^\infty_t  \mathcal{G}^{\lm, \fr{N}{2}}_{s,0}}}
    \left\| \mathfrak{g}\right\|^2_{L^\infty_t  \bar{\mathcal{G}}^{\lm(0), \fr{N}{2}}_{s,0}}.
 \end{align}
 If $\fr{2t}{3}\le \tau\le t$, using again \eqref{est-eta},  following the estimates of ${\bf T}^{(1);{\rm HL}}[E_1]_{\tau\approx t}$ line by line (see \eqref{est-resoance}--\eqref{est-T1E1-l=k}), we arrive at
 \begin{align}\label{est-T2E2HL-t=tau}
     {\bf T}^{(2);{\rm HL}}[E_2]_{\tau\approx t}\les \left\| \la t\ra^{\fr{3+a-s}{2}} \rho^{(2)}\right\|^2_{L^2_t\mathcal{G}^{\lm,N}_s}{ \left\|\la Y\ra^3\la\nb_x\ra^3|Z|f(\tau)\right\|^2_{L^\infty_t\mathcal{G}^{\lm,\fr{N}{2}}_{s,0}}}\left\|\la Y^*\ra^3\mathfrak{g}\right\|^2_{L^\infty_t\bar{\mathcal{G}}^{\lm(0),\fr{N}{2}}_{s,0}}.
 \end{align}
 
 The Bony decomposition will be used when addressing  ${\bf T}^{(2);\rm LH}[E_2]$. Indeed, on the support of the integrand above, there holds
\begin{align*}
    |k-l,\eta+(k-l)\tau|\le\fr34\cdot\fr{J}{8},\quad \fr{J}{2}\le|l,l\tau|\le\fr{3J}{2}, 
\end{align*}
and hence
\begin{align*}
    |k-l,\eta+(k-l)\tau|\le \fr{3}{16}|l,l\tau|.
\end{align*}
Then
\begin{align}\label{upper-eta}
    |\eta|\les\la\tau\ra|k-l,\eta+(k-l)\tau|\les\la\tau\ra|l,l\tau|.
\end{align}
 Using this, one can move the $v$ derivative on $f$ to $\rho^{(2)}$. More precisely, if $0\le\tau\le\fr{2t}{3}$, \eqref{eq:gain m2} is still valid, similar to \eqref{est-T2E2HL-short}, we have
\begin{align}\label{est-T2E2LH-short}
     {\bf T}^{(2); {\rm LH}}[E_2]_{\tau\le\fr{2t}{3}}\les \left\||\nb_x|^{-1}|Z|\rho^{(2)}\right\|^2_{L^2_t\mathcal{G}^{\lm,\fr{N}{2}}_s}{ \big\| f \big\|^2_{L^\infty_t  \mathcal{G}^{\lm, N}_{s,0}}}
    \left\| \mathfrak{g}\right\|^2_{L^\infty_t  \bar{\mathcal{G}}^{\lm(0), N}_{s,0}}.
 \end{align}
 If $\fr{2t}{3}\le\tau\le t$, using \eqref{upper-eta} again, we modify the estimates of ${\bf T}^{(1);{\rm LH}}[E_1]_{\tau\approx t}$ slightly to obtain
\begin{align}\label{est-T2E2LH-t=tau}
    {\bf T}^{(2);{\rm LH}}[E_2]_{\tau\approx t}\les\nn&\int_0^{T}\sum_{|\beta|\le N}\kappa^{2|\beta|}\sum_{k\in\Z^3_*}\Bigg[\sum_{\substack{\beta'\le\beta\\ \beta''\le\beta-\beta'}}\sum_{\substack{n\le m\\ n'\le m-n}}\sum_{l\in\Z^3_*}\int_{\fr{2t}{3}}^t\fr{b_{m,n,n',s}}{\la t\ra \la l,l\tau\ra^3}\\
    \nn&\times \Big|a_{n,s}(\tau) \la \tau\ra^{\fr{3+a-s}{2}}
    \fr{\la\tau\ra^2}{|l|}|l,l\tau|^4(l,l\tau)^{n+\beta'}\hat{\rho}^{(2)}_l(\tau)\Big|\\
    \nn&\times \Big\|a_{n',s}(\tau)\mathcal{F}_{x,v}\big[f^{(n'+\beta'')}\big]_{k-l}(\tau)\Big\|_{L^2_\eta}\\
\nn&\times \Big\|a_{m-n-n',s}(0)\mathcal{F}_{x,v}\big[\mathfrak{g}^{(m-n-n'+\beta-\beta'-\beta'')}\big]_k(t-\tau)\Big\|_{L^2_\eta} d\tau\Bigg]^2dt\\
\les&\left\|\la t\ra^{\fr{3+a-s}{2}}|\nb_x|^{-1}|Z|^6\rho^{(2)}\right\|_{L^2_t\mathcal{G}^{\lm,\fr{N}{2}}_s}^2 { \|f\|^2_{L^\infty_t\mathcal{G}^{\lm,N}_{s,0}}}\|\mathfrak{g}\|^2_{L^\infty_t\bar{\mathcal{G}}^{\lm(0),N}_{s,0}}.
\end{align}

\noindent{$\diamond$ \underline{\it Estimates of $\mathcal{T}_{{v_j}f}E^j_2$}.} To this end, we write

\begin{align*}
    \tl{\bf T}^{(2)}[E_2]
:=&\int_0^{T}\la t\ra^{3+a-s}\sum_{k\in\Z^3_*}\sum_{|\beta|\le N}\kappa^{2|\beta|}\sum_{m\in\N^6}\Bigg|a_{m,s}(t)(k,kt)^{m+\beta}\\
&\times\int_0^t\int_{\R^3}\sum_{l\in\Z^3_*}\sum_{J\ge8}\hat{\rho}^{(2)}_l(\tau)_J\fr{l_j}{|l|^2} \mathcal{F}_{x,v}[v_jf]_{k-l}(\tau,\eta)_{<J/8}\bar{\hat{\mathfrak{g}}}_k(t-\tau,\eta) d\eta d\tau\Bigg|^2dt.
\end{align*}
Similar to \eqref{est-T2E2HL-short}, \eqref{est-T2E2HL-t=tau}, \eqref{est-T2E2LH-short} and \eqref{est-T2E2LH-t=tau}, without resorting to \eqref{upper-eta}, we have
\begin{align}\label{est-tlT2E2}
\tl{\bf T}^{(2)}[E_2]\les\nn&\left\||\nb_x|^{-1}\rho^{(2)}\right\|^2_{L^2_t\mathcal{G}^{\lm,N}_s}{ \big\|v f \big\|^2_{L^\infty_t  \mathcal{G}^{\lm, \fr{N}{2}}_{s,0}}}
    \left\| \mathfrak{g}\right\|^2_{L^\infty_t  \bar{\mathcal{G}}^{\lm(0), \fr{N}{2}}_{s,0}}\\
    \nn&+\left\| \la \tau\ra^{\fr{3+a-s}{2}} \rho^{(2)}\right\|^2_{L^2_t\mathcal{G}^{\lm,N}_s}{ \left\|\la Y\ra^2\la\nb_x\ra^3(vf)\right\|^2_{L^\infty_t\mathcal{G}^{\lm,\fr{N}{2}}_{s,0}}}\left\|\la Y^*\ra^2\mathfrak{g}\right\|^2_{L^\infty_t\bar{\mathcal{G}}^{\lm(0),\fr{N}{2}}_{s,0}}\\
    \nn&+\left\||\nb_x|^{-1}\rho^{(2)}\right\|^2_{L^2_t\mathcal{G}^{\lm,\fr{N}{2}}_s}{ \big\| vf \big\|^2_{L^\infty_t  \mathcal{G}^{\lm, N}_{s,0}}}
    \left\| \mathfrak{g}\right\|^2_{L^\infty_t  \bar{\mathcal{G}}^{\lm(0), N}_{s,0}}\\
    +&\left\|\la t\ra^{\fr{3+a-s}{2}}|\nb_x|^{-1}|Z|^4\rho^{(2)}\right\|_{L^2_t\mathcal{G}^{\lm,\fr{N}{2}}_s}^2 { \|vf\|^2_{L^\infty_t\mathcal{G}^{\lm,N}_{s,0}}}\|\mathfrak{g}\|^2_{L^\infty_t\bar{\mathcal{G}}^{\lm(0),N}_{s,0}}.
\end{align}

\noindent{$\diamond$ \underline{\it The collision contributions}.}
Now we turn to bound
\begin{align*}
&\nu\left\|\langle t\rangle^{\fr{3+a-s}{2}} \int_0^t\int_{\R^3}S_k(t-\tau)[(\widehat{\Gamma(f,f)})_k(\tau)]\mu^{\fr12}(v)dvd\tau \right\|_{L^2_t\mathcal{G}^{\lm,N}_{s}}\\
=&\nu\Bigg[\int_0^{T}\langle t\rangle^{3+a-s}\sum_{k\in\Z^3_*} \sum_{|\beta|\le N}\kappa^{2|\beta|}\sum_{m\in\N^6}a_{m,s}^2(t)\\
    &\times\left|(k,kt)^{m+\beta}\int_0^t\int_{\R^3} S_k(t-\tau)\big[ \widehat{\Gamma(f,f)}_{k}(\tau,v)\big]\sqrt{\mu}dvd\tau\right|^2dt\Bigg]^{\fr12},
\end{align*}
which will correspond to a plasma echo-like contribution from the nonlinear Landau collisions. 
Recalling \eqref{def-Gamma} and noting that $\Phi^{ij}=\Phi^{ji}$, applying again the duality argument (recalling \eqref{def-frak-g} for definition of $\mathfrak{g}$),  we write
\begin{align*}
    &\nu(k,kt)^{m+\beta}\int_0^t\int_{\R^3}S_k(t-\tau)\big[ \widehat{\Gamma(f,f)}_{k}(\tau,v)\big]\sqrt{\mu}dvd\tau\\
    =&\nu(k,kt)^{m+\beta}\int_0^t\int_{\R^3} \mathcal{F}_x\Big[\big(\Phi^{ij}*(\mu^\fr12f)\big)\pr_{v_iv_j}f\Big]_{k}(\tau,v)\mathfrak {g}_k(t-\tau,v)dvd\tau\\
    &+2\nu(k,kt)^{m+\beta}\int_0^t\int_{\R^3} \mathcal{F}_x\Big[\big(\Phi^{ij}*(\pr_{v_i}\mu^\fr12f)\big)\pr_{v_j}f\Big]_{k}(\tau,v)\mathfrak {g}_k(t-\tau,v)dvd\tau\\
    &+\nu(k,kt)^{m+\beta}\int_0^t\int_{\R^3} \mathcal{F}_x\Big[\big(\Phi^{ij}*(\pr_{v_iv_j}\mu^\fr12f)\big)f\Big]_{k}(\tau,v)\mathfrak {g}_k(t-\tau,v)dvd\tau\\
    &-\nu(k,kt)^{m+\beta}\int_0^t\int_{\R^3} \mathcal{F}_x\Big[\big(\Phi^{ij}*\pr_{v_iv_j}(\mu^\fr12f)\big)f\Big]_{k}(\tau,v)\mathfrak {g}_k(t-\tau,v)dvd\tau
    =\sum_{\ell=1}^4\mathbb{I}^{m,\beta}_{\ell,k}(t).
\end{align*}

Clearly, the first three terms $\mathbb{I}^{m,\beta}_{1,k}(t)$, $\mathbb{I}^{m,\beta}_{2,k}(t)$ and $\mathbb{I}^{m,\beta}_{3,k}(t)$ possess a similar structure, and among them the first term $\mathbb{I}^{m,\beta}_{1,k}(t)$ is  the most challenging to address due to two $v$ derivatives hitting $f$. The last term $\mathbb{I}^{m,\beta}_{4,k}(t)$ is comparatively straightforward to deal with  since the two $v$ derivatives can be absorbed by $\Phi^{ij}$. In the following,
let us focus on the treatments of $\mathbb{I}^{m,\beta}_{1,k}(t)$ and $\mathbb{I}^{m,\beta}_{4,k}(t)$. 
By \eqref{eq:Fourier-Phi} and Plancherel's theorem, we write
\begin{align*}
  \mathbb{I}_{1,k}^{m,\beta}(t)=& 8\pi\nu(k,kt)^{m+\beta}\sum_{l\in\Z^3}\Big(\int_0^{\fr{2t}{3}}+\int_{\fr{2t}{3}}^t\Big)\int_{\R^6} \frac{\xi_i\xi_j}{|\xi|^4}\widehat{(\sqrt{\mu}f)}_l(\tau,\xi) \widehat{(\pr_{v_iv_j}f)}_{k-l}(\tau,\eta-\xi)\\
  &\times\bar{\hat{\mathfrak{g}}}_{k}(t-\tau,\eta)d\xi d\eta d\tau=\mathbb{I}_{1,k}^{m,\beta}(t)_{\tau\le\fr{2t}{3}}+\mathbb{I}_{1,k}^{m,\beta}(t)_{\tau\approx t},
\end{align*}
and
\begin{align*}
  \mathbb{I}_{4,k}^{m,\beta}(t)=& 8\pi\nu(k,kt)^{m+\beta}\sum_{l\in\Z^3}\Big(\int_0^{\fr{2t}{3}}+\int_{\fr{2t}{3}}^t\Big)\int_{\R^6} \widehat{(\sqrt{\mu}f)}_l(\tau,\xi) \hat{f}_{k-l}(\tau,\eta-\xi)\\
  &\times\bar{\hat{\mathfrak{g}}}_{k}(t-\tau,\eta)d\xi d\eta d\tau=\mathbb{I}_{4,k}^{m,\beta}(t)_{\tau\le\fr{2t}{3}}+\mathbb{I}_{4,k}^{m,\beta}(t)_{\tau\approx t},
\end{align*}
{\bf Case 1: $\tau\le\fr23t$.} Noting that
\begin{align*}
    (k,kt)=(l,\xi+l\tau)+(k-l,\eta-\xi+(k-l)\tau)+(0,-\eta+k(t-\tau)),
\end{align*}
now we can use \eqref{eq:gain m2} and Lemma \ref{lem-interp1} again, and by virtue of Lemma \ref{lem-weighted-young}, similar to the treatment of ${\bf T}^{(1)}[E_1]_{\tau\le\fr{2t}{3}}$ (see \eqref{est-T1E1-short}, but here we still need the low-high and high-low decomposition). By applying Remark \ref{Rmk: coefficients a-b}, we obtain that 
\begin{align*}
    \nn&\int_0^{T}\sum_{k\in\Z^3_*} \la t\ra^{3+a-s}\sum_{|\beta|\le N}\kappa^{2|\beta|}\sum_{m\in\N^6}a_{m,s}^2(t)\Big|\mathbb{I}_{1,k}^{m,\beta}(t)_{\tau\le\fr{2t}{3}}\Big|^2dt\\
    \les\nn&\nu^{\fr23}\int_0^{T^*}\sum_{k\in\Z^3_*}\fr{1}{|k|^{9+a-s+2a\Theta}} \sum_{|\beta|\le N}\kappa^{2|\beta|}\sum_{m\in\N^6}\Bigg[\sum_{l\in\Z^3}\int_0^{\fr{2t}{3}}\fr{1}{\la t\ra^3}\int_{\R^6}a_{m,s}(\tau)\\
    \nn&\times\Bigg|(k,kt)^{m+\beta}\fr{1}{|\xi|^2}\widehat{(\sqrt{\mu}f)}_l(\tau,\xi) \nu^{\fr23}\widehat{(\pr_{v_iv_j}f)}_{k-l}(\tau,\eta-\xi)\hat{\mathfrak{g}}_{k}(t-\tau,\eta)\Bigg|d\xi d\eta d\tau\Bigg]^2dt\\
    \les\nn&\nu^{\fr23}\int_0^{T}\sup_{k\in\Z^3_*} \sum_{|\beta|\le N}\kappa^{2|\beta|}\sum_{m\in\N^6}\Bigg(\sum_{\substack{\beta'\le\beta\\\beta''\le\beta-\beta'}}\sum_{\substack{n\le m\\n'\le m-n}}\int_0^{\fr{2t}{3}}\fr{b_{m,n,n',s}}{\la t\ra^2}d\tau\Bigg)\\
    \nn&\times \sum_{\substack{n\le m\\n'\le m-n}}\int_0^{\fr{2t}{3}}\fr{1}{\la t\ra^4}\Bigg[\sum_{\substack{\beta'\le\beta,\beta''\le\beta-\beta'\\|\beta'|\le|\beta|/2}}\Big\|a_{n,s}(\tau)(l,\xi+l\tau)^{n+\beta'}\widehat{(\sqrt{\mu}f)}_l(\tau,\xi)\Big\|^2_{(L^1_\xi\cap L^\infty_{\xi}) L^2_l}\\
    \nn&\times\Big\|a_{n',s}(\tau)\mathcal{F}_{x,v}\big[\nu^{\fr23}\pr_{v_iv_j}f^{(n'+\beta'')}\big]_{k-l}(\tau)\Big\|^2_{L^2_\eta L^2_l}\\
    \nn&\times\Big\|a_{m-n-n',s}(0)\mathcal{F}_{x}\big[\mathfrak{g}^{(m-n-n'+\beta-\beta'-\beta'')}\big]_{k}(t-\tau)\Big\|_{L^2_v}^2 d\tau dt\\
    \nn&+\sum_{\substack{\beta'\le\beta,\beta''\le\beta-\beta'\\|\beta'|>|\beta|/2}}\Big\|a_{n,s}(\tau)(l,\xi+l\tau)^{n+\beta'}\widehat{(\sqrt{\mu}f)}_l(\tau,\xi)\Big\|^2_{(L^2_\xi\cap L^\infty_{\xi}) L^2_l}\\
    \nn&\times\Big\|a_{n',s}(\tau)\mathcal{F}_{x,v}\big[\nu^{\fr23}\pr_{v_iv_j}f^{(n'+\beta'')}\big]_{k-l}(\tau)\Big\|^2_{(L^1_\eta\cap L^2_\eta) L^2_l}\\
    \nn&\times\Big\|a_{m-n-n',s}(0)\mathcal{F}_{x}\big[\mathfrak{g}^{(m-n-n'+\beta-\beta'-\beta'')}\big]_{k}(t-\tau)\Big\|_{L^2_v}^2 d\tau \Bigg]dt.
\end{align*}
Thus, we have 
\begin{align}\label{est-I1k-short}
    \nn&\int_0^{T}\sum_{k\in\Z^3_*} \la t\ra^{3+a-s}\sum_{|\beta|\le N}\kappa^{2|\beta|}\sum_{m\in\N^6}a_{m,s}^2(t)\Big|\mathbb{I}_{1,k}^{m,\beta}(t)_{\tau\le\fr{2t}{3}}\Big|^2\\
    \les\nn&\nu^{\fr23}\int_0^{T}\int_0^{\fr{2t}{3}}\fr{1}{\la t\ra^4}\left\|\la Y\ra^{2}(\sqrt{\mu}f)(\tau)\right\|^2_{\mathcal{G}^{\lm(\tau),\fr{N}{2}}_{s,2}}\left\|\nu^{\fr23}\pr_{v_iv_j}f(\tau)\right\|^2_{\mathcal{G}^{\lm(\tau),N}_{s,0}}\\
    \nn&\times\left\|\mathfrak{g}(t-\tau)\right\|^2_{\bar{\mathcal{G}}^{\lm(0),N}_{s,0}}d\tau dt\\
    \nn&+\nu^{\fr23}\int_0^{T}\int_0^{\fr{2t}{3}}\fr{1}{\la t\ra^4}\left\|\sqrt{\mu}f(\tau)\right\|^2_{\mathcal{G}^{\lm(\tau),N}_{s,2}}\left\|\nu^{\fr23}\pr_{v_iv_j}\la Y\ra^{2}f(\tau)\right\|^2_{\mathcal{G}^{\lm(\tau),\fr{N}{2}}_{s,0}}\\
    \nn&\times\left\|\mathfrak{g}(t-\tau)\right\|^2_{\bar{\mathcal{G}}^{\lm(0),\fr{N}{2}}_{s,0}}d\tau dt\\
    \les&\nu^{\fr23}{  \left\|\sqrt{\mu}f\right\|^2_{L^\infty_t\mathcal{G}^{\lm,N}_{s,2}}\left\|\nu^{\fr23}\pr_{v_iv_j}f\right\|^2_{L^\infty_t\mathcal{G}^{\lm,N}_{s,0}}}\left\|\mathfrak{g}\right\|^2_{L^\infty_t\bar{\mathcal{G}}^{\lm(0),N}_{s,0}}.
\end{align}

It is easy to see that the above estimate with minor changes applies to $\mathbb{I}_{4,k}^{m,\beta}(t)_{\tau\le\fr{2t}{3}}$. We state the result here:
\begin{align}\label{est-I4k-short}
    \nn&\int_0^{T}\sum_{k\in\Z^3_*} \la t\ra^{3+a-s}\sum_{|\beta|\le N}\kappa^{2|\beta|}\sum_{m\in\N^6}a_{m,s}^2(t)\Big|\mathbb{I}_{4,k}^{m,\beta}(t)_{\tau\le\fr{2t}{3}}\Big|^2\\
    \les&\nu^{2}{  \left\|\sqrt{\mu}f\right\|^2_{L^\infty_t\mathcal{G}^{\lm,N}_{s,0}}\left\|f\right\|^2_{L^\infty_t\mathcal{G}^{\lm,N}_{s,0}}}\left\|\mathfrak{g}\right\|^2_{L^\infty_t\bar{\mathcal{G}}^{\lm(0),N}_{s,0}}.
\end{align}

\noindent{\bf Case 2: $\fr23t\le\tau\le t$.} In this case, we further split  $\mathbb{I}^{m,\beta}_{1,k}(t)_{\tau\approx t}$ into three parts:
\begin{align}\label{divid-I1k}
  \mathbb{I}_{1,k}^{m,\beta}(t)_{\tau\approx t}
=\nn&8\pi\nu \sum_{l\in\Z^3}\sum_{\beta'\leq \beta}C_{\beta}^{\beta'}\sum_{\beta''\le\beta-\beta'}C_{\beta-\beta'}^{\beta''}
    \sum_{n\leq m}C_{m}^{n}\sum_{n'\le m-n}C_{m-n}^{n'}\\
    \nn&\times \Big({\bf1}_{|\beta'|\leq |\beta|/2}{\bf1}_{l\ne0}+{\bf1}_{|\beta'|\leq |\beta|/2}{\bf1}_{l=0}
    +{\bf1}_{|\beta'|>|\beta|/2}\Big)\\
    \nn&\times\int_{\fr{2t}{3}}^t\int_{\R^6} \frac{\xi_i\xi_j}{|\xi|^4}\mathcal{F}_{x,v}\big[Z^{n+\beta'}(\mu^\fr12f)\big]_l(\tau,\xi) \mathcal{F}_{x,v}\big[Z^{n'+\beta''}\pr_{v_iv_j}f\big]_{k-l}(\tau,\eta-\xi)\\
    \nn&\quad \quad 
    \times (0,-\eta+k(t-\tau))^{m-n-n'+\beta-\beta'-\beta''}\bar{\hat{\mathfrak{g}}}_{k}(t-\tau,\eta)d\xi d\eta d\tau\\
    &=\mathbb{I}_{1,k;1)}^{m,\beta}(t)+\mathbb{I}_{1,k;2)}^{m,\beta}(t)+\mathbb{I}_{1,k;3)}^{m,\beta}(t).
\end{align}

{\it $\bullet$ Treatments of $\mathbb{I}_{1,k;1)}^{m,\beta}(t)$.} Now we have
\begin{align}\label{xi-gain-1}
    \fr{1}{|\xi|^2}\les \fr{\la\xi\ra^2}{|\xi|^2}\fr{\la \xi+l\tau\ra^2}{\la l\tau\ra^2}.
\end{align}
Combining this with Remark \ref{Rmk: coefficients a-b} and Lemma \ref{lem-weighted-young}, we have
\begin{align*}
    \nn&\int_0^{T}\la t\ra^{3+a-s}\sum_{k\in\Z^3_*}\sum_{|\beta|\le N}\kappa^{2|\beta|}\sum_{m\in\N^6}a_{m,\lm,s}^2(t)\left|\mathbb{I}_{1,k;1)}^{m,\beta}(t)\right|^2dt\\
    \nn\les&\nu^2\int_0^{T}\la t\ra^{3+a-s}\sum_{k\in\Z^3_*} \sum_{|\beta|\le N}\kappa^{2|\beta|}\sum_{m\in\N^6}\Bigg[\sum_{\substack{\beta'\leq \beta, \beta''\le\beta-\beta'\\|\beta'|\le|\beta|/2}}
    \sum_{\substack{n\leq m\\ n'\le m-n }}\\
    \nn&\sum_{l\in \Z^3_*}\int_{\fr{2t}{3}}^t\int_{\R^6} \frac{b_{m,n,n',s}}{|l, l\tau|^2\la \nu^{\fr13}(t-\tau)\ra^{\fr32}} \left|a_{n,s}(\tau)\mathcal{F}_{x,v}\big[| Z|^4(\mu^\fr12f)^{(n+\beta')}\big]_l(\tau,\xi)\fr{\la\xi\ra^2}{|\xi|^2} \right|\\
    \nn&\times\left|a_{n',s}(\tau)\mathcal{F}_{x,v}\big[\pr_{v_iv_j}f^{(n'+\beta'')}\big]_{k-l}(\tau,\eta-\xi)\right|\\
    \nn&
    \times \la\nu^\fr13(t-\tau)\ra^{\fr32}\left|a_{m-n-n',s}(0) \mathcal{F}_{x,v}\big[\mathfrak{g}^{(m-n-n'+\beta-\beta'-\beta'')}\big]_k(t-\tau,\eta)\right|d\xi d\eta d\tau\Bigg]^2dt\\
    \nn\les&\nu^2\int_0^{T}\fr{1}{\la t\ra^{1-a+s}}\sum_{k\in\Z^3_*} \sum_{|\beta|\le N}\kappa^{2|\beta|}\sum_{m\in\N^6}\Bigg(\sum_{\substack{\beta'\leq \beta, \beta''\le\beta-\beta'\\|\beta'|\le|\beta|/2}}
    \sum_{\substack{n\leq m\\ n'\le m-n }}\sum_{l\in \Z^3_*}\int_{\fr{2t}{3}}^t \frac{b_{m,n,n',s}}{|l|^4\la \nu^{\fr13}(t-\tau)\ra^{\fr32}}d\tau\Bigg)\\
    \nn&\times\sum_{\substack{\beta'\leq \beta, \beta''\le\beta-\beta'\\|\beta'|\le|\beta|/2}}
    \sum_{\substack{n\leq m\\ n'\le m-n }}\sum_{l\in \Z^3_*}\int_{\fr{2t}{3}}^t \la\nu^\fr13(t-\tau)\ra^{-\fr32}  \left\|a_{n,s}(\tau)\mathcal{F}_{x,v}\big[| Z|^4(\mu^\fr12f)^{(n+\beta')}\big]_l(\tau) \right\|^2_{L^1_\xi\cap L^\infty_\xi}\\
    \nn&\times\left\|a_{n',s}(\tau)\mathcal{F}_{x,v}\big[\pr_{v_iv_j}f^{(n'+\beta'')}\big]_{k-l}(\tau)\right\|^2_{L^2_\eta}\\
    \nn&
    \times \la\nu^\fr13(t-\tau)\ra^{3}\left\|a_{m-n-n',s}(0) \mathcal{F}_{x,v}\big[\mathfrak{g}^{(m-n-n'+\beta-\beta'-\beta'')}\big]_k(t-\tau)\right\|^2_{L^2_\eta} d\tau dt.
 \end{align*}  
 Thus, we obtain that
\begin{align}\label{est-I11}
    \nn&\int_0^{T}\la t\ra^{3+a-s}\sum_{k\in\Z^3_*}\sum_{|\beta|\le N}\kappa^{2|\beta|}\sum_{m\in\N^6}a_{m,\lm,s}^2(t)\left|\mathbb{I}_{1,k;1)}^{m,\beta}(t)\right|^2dt\\
    \nn\les&\nu^{\fr53}\int_0^{T}\fr{1}{\la t\ra^{1-a+s}}\int_{\fr{2t}{3}}^t \la\nu^\fr13(t-\tau)\ra^{-\fr32} \left\| \la Y\ra^2|Z|^4(\mu^\fr12f)(\tau)\right\|_{\mathcal{G}^{\lm,\fr{N}{2}}_{s,2}}^2\\
    \nn&\times\left\|\pr_{v_iv_j}f(\tau)\right\|^2_{\mathcal{G}^{\lm,N}_{s,0}} \Big(\la\nu^\fr13(t-\tau)\ra^{3}\left\|\mathfrak{g}(t-\tau)\right\|^2_{\bar{\mathcal{G}}^{\lm,N}_{s,0}}\Big) d\tau dt\\
    \les& { \left\| \la Z\ra^6(\mu^\fr12f)\right\|_{L^\infty_t\mathcal{G}^{\lm,\fr{N}{2}}_{s,2}}^2 \left\|\nu^{\fr23}\pr_{v_iv_j}f\right\|^2_{L^\infty_t\mathcal{G}^{\lm,N}_{s,0}}}\sup_{t}\Big(\la\nu^\fr13t\ra^{3}\left\|\mathfrak{g}(t)\right\|^2_{\bar{\mathcal{G}}^{\lm,N}_{s,0}}\Big).
\end{align}

{$\bullet$ \it Treatments of $\mathbb{I}_{1,k;3)}^{m,\beta}(t)$.} Note that
\begin{align}\label{xi-gain}
    \fr{1}{|\xi|^2}\les\fr{\la\xi\ra^2}{|\xi|^2}\fr{\la \xi-kt+l\tau \ra^2}{\la kt-l\tau\ra^2}\les \fr{\la\xi\ra^2}{|\xi|^2} \fr{1}{\la kt-l\tau\ra^2}\la \eta-\xi+(k-l)\tau \ra^2\la -\eta+k(t-\tau) \ra^2.
\end{align}
Then one  can treat $\fr{1}{\la kt-l\tau\ra^2}$ according to whether $(t,\tau, k,l)$ is in $D$ or not, where $D$ is defined in \eqref{def-D}.

If $l=k$, in view of Remark \ref{Rmk: coefficients a-b}, \eqref{xi-gain},  Lemma \ref{lem-weighted-young}, Corollaries \ref{coro-kernel} and \ref{coro-convolution},  we are led to
\begin{align}\label{est-Ik3-2}
\nn&\sum_{k\in\Z^3_*} \sum_{|\beta|\le N}\kappa^{2|\beta|}\sum_{m\in\N^6}a_{m,\lm,s}^2(t)\left|\mathbb{I}_{1,k;3)}^{m,\beta}(t)\right|^2\\
\nn\les&\nu^2 \sum_{k\in\Z^3_*} \sum_{|\beta|\le N}\kappa^{2|\beta|}\sum_{m\in\N^6}\Bigg[\sum_{\substack{\beta'\leq \beta,\beta''\le\beta-\beta'\\|\beta'|>|\beta|/2}}
    \sum_{\substack{n\leq m\\n'\le m-n}}\int_{\fr{2t}{3}}^t\int_{\R^6} \frac{b_{m,n,n',s}}{\la k(t-\tau)\ra^2}\\
    \nn&\times\left|a_{n,s}(\tau)\mathcal{F}_{x,v}\big[(\mu^\fr12f)^{(n+\beta')}\big]_k(\tau,\xi)\fr{\la\xi\ra^2}{|\xi|^2} \right|\\ 
    \nn&\times\left|a_{n',s}(\tau)\mathcal{F}_{x,v}\big[\la \nb_v\ra^2\pr_{v_iv_j}f^{(n'+\beta'')}\big]_{0}(\tau,\eta-\xi)\right|\\
    \nn&\times\left|a_{m-n-n',s}(0) \mathcal{F}_{x,v}\big[\la Y^*\ra^2\mathfrak{g}^{(m-n-n'+\beta-\beta'-\beta'')}\big]_k(t-\tau,\eta)\right|d\xi d\eta d\tau\Bigg]^2\\
    \nn\les&\nu^2 \sum_{k\in\Z^3_*} \sum_{|\beta|\le N}\kappa^{2|\beta|}\sum_{m\in\N^6} \Bigg(\sum_{\substack{\beta'\leq \beta, \beta''\le\beta-\beta'\\|\beta'|>|\beta|/2}}
    \sum_{\substack{n\leq m\\n'\le m-n}} \int_{\fr{2t}{3}}^t \frac{b_{m,n,n',s}}{\la k(t-\tau)\ra^2}d\tau\Bigg)\\
    \nn&\times\sum_{\substack{\beta'\leq \beta, \beta''\le\beta-\beta'\\|\beta'|>|\beta|/2}}
    \sum_{\substack{n\leq m\\n'\le m-n}} \int_{\fr{2t}{3}}^t \frac{b_{m,n,n',s}}{\la k(t-\tau)\ra^2} \left\|a_{n,s}(\tau)\mathcal{F}_{x}\big[(\mu^\fr12f)^{(n+\beta')}\big]_k(\tau)\la v\ra^2\right\|_{L^2_v}^2\\ \nn&\times\left\|a_{n',s}(\tau)\la \nb_v\ra^4\pr_{v_iv_j}f^{(n'+\beta'')}_0(\tau)\right\|_{L^2_v}^2\\
    \nn&\times\left\|a_{m-n-n',s}(0) \la Y^*\ra^2\mathfrak{g}_k^{(m-n-n'+\beta-\beta'-\beta'')}(t-\tau)\right\|_{L^2_v}^2 d\tau\\
    \les
    \nn&\nu^2 \int_{\fr{2t}{3}}^t \frac{1}{\la t-\tau\ra^2} \left\|(\mu^\fr12f)_{\ne}(\tau)\right\|_{\mathcal{G}^{\lm(\tau),N}_{s,2}}^2\left\|\la \nb_v\ra^4\pr_{v_iv_j}f_0(\tau)\right\|_{\mathcal{G}^{\lm(\tau),\fr{N}{2}}_{s,0}}^2\\
    &\quad 
    \times\left\| \la Y^*\ra^2\mathfrak{g}(t-\tau)\right\|_{\bar{\mathcal{G}}^{\lm(0),\fr{N}{2}}_{s,0}}^2 d\tau.
\end{align}
Then recalling  \eqref{decom-f},
we have
\begin{align}\label{est-Ik3-3}
\nn&\int_0^{T}\langle t\rangle^{3+a-s}\sum_{k\in\Z^3_*} \sum_{|\beta|\le N}\kappa^{2|\beta|}\sum_{m\in\N^6}a_{m,\lm,s}^2(t)\left|\mathbb{I}_{1,k;3)}^{m,\beta}(t)\right|^2dt\\
\les\nn&\nu\int_0^{T}\int_{\fr{2t}{3}}^t \frac{\nu\la\tau\ra^{3+a-s}}{\la t-\tau\ra^2} \left\|(\mu^\fr12f)_{\ne}(\tau)\right\|_{\mathcal{G}^{\lm(\tau),N}_{s,2}}^2\left\|\la \nb_v\ra^4\pr_{v_iv_j}f_0(\tau)\right\|_{\mathcal{G}^{\lm(\tau),\fr{N}{2}}_{s,0}}^2\\
    \nn&\quad 
    \times \left\|\la Y^*\ra^2\mathfrak{g}(t-\tau)\right\|^2_{\bar{\mathcal{G}}^{\lm(0),\fr{N}{2}}_{s,0}} d\tau dt\\
    \les\nn&\nu\int_0^{T}\int_{\tau}^{T} \frac{1}{\la t-\tau\ra^2}\left\| \la Y^*\ra^2\mathfrak{g}(t-\tau)\right\|_{\mathcal{G}^{\lm(0),\fr{N}{2}}_{s,0}}^2 dt\\
    \nn&\times\left(\nu\la\tau\ra^{3+a-s}\left\|(\mu^\fr12f)_{\ne}(\tau)\right\|_{\mathcal{G}^{\lm(\tau),N}_{s,2}}^2\right)\left\|\la \nb_v\ra^4\pr_{v_iv_j}f_0(\tau)\right\|_{\mathcal{G}^{\lm(\tau),\fr{N}{2}}_{s,0}}^2 d\tau\\
    \les\nn&\sup_t\left\| \la Y^*\ra^2\mathfrak{g}(t)\right\|_{\mathcal{G}^{\lm(0),\fr{N}{2}}_{s,0}}^2\sup_\tau\left\|\la \nb_v\ra^4\pr_{v_iv_j}f_0(\tau)\right\|_{\mathcal{G}^{\lm(\tau),\fr{N}{2}}_{s,0}}^2\\
     &\times{  \nu^{\fr23}\Bigg[\nu^{\fr13}\int_0^{T} 
    \nu\la\tau\ra^3\big\|\mu^\fr12g_{\ne}(\tau)\big\|_{\mathcal{G}^{\lm(\tau),N}_{s,2}}^2 d\tau+\nu^{\fr43}\int_0^{T} \la \tau\ra^{3+a-s}\big\|\mu^{\fr12}q(\phi)g\big\|^2_{\mathcal{G}^{\lm(\tau),N}_{s,2}} d\tau\Bigg]}.
\end{align}

If $|kt-l\tau|\ge\fr{t}{2}$ and $k\ne l$, then $\mathbb{I}_{1,k;3)}^{m,\beta}(t)$ can be treated in a similar manner as  $\mathbb{I}_{1,k;1)}^{m,\beta}(t)$: 
\begin{align}\label{est-Ik3-1}
    \nn&\int_0^{T}\la t\ra^{3+a-s}\sum_{k\in\Z^3_*} \sum_{|\beta|\le N}\kappa^{2|\beta|}\sum_{m\in\N^6}a_{m,\lm,s}^2(t)\left|\mathbb{I}_{1,k;3)}^{m,\beta}(t)\right|^2dt\\
    \nn\les&\nu^2\int_0^{T}\la t\ra^{3+a-s}\sum_{k\in\Z^3_*} \sum_{|\beta|\le N}\kappa^{2|\beta|}\sum_{m\in\N^6}\Bigg[\sum_{\substack{\beta'\leq \beta, \beta''\le\beta-\beta'\\|\beta'|>|\beta|/2}}
    \sum_{\substack{n\leq m\\ n'\le m-n }}\\
    \nn&\sum_{l\in \Z^3,l\ne k}\int_{\fr{2t}{3}}^t\int_{\R^6} \frac{b_{m,n,n',s}}{\la t\ra^2|k-l|^2\la\nu^{\fr13}(t-\tau) \ra^{\fr32}}  \left|a_{n,s}(\tau)\mathcal{F}_{x,v}\big[(\mu^\fr12f)^{(n+\beta')}\big]_l(\tau,\xi)\fr{\la\xi\ra^2}{|\xi|^2} \right|\\
    \nn&\times\left|a_{n',s}(\tau)\mathcal{F}_{x,v}\big[\pr_{v_iv_j}|\nb_x|^2\la Y\ra^2f^{(n'+\beta'')}\big]_{k-l}(\tau,\eta-\xi)\right|\la\nu^{\fr13}(t-\tau) \ra^{\fr32}\\
    \nn&
    \times \left|a_{m-n-n',s}(0) \mathcal{F}_{x,v}\big[\la Y^*\ra^2\mathfrak{g}^{(m-n-n'+\beta-\beta'-\beta'')}\big]_k(t-\tau,\eta)\right|d\xi d\eta d\tau\Bigg]^2dt\\
    \nn\les&\nu^2\int_0^{T}\fr{1}{\la t\ra^{1-a+s}}\sum_{k\in\Z^3_*} \sum_{|\beta|\le N}\kappa^{2|\beta|}\sum_{m\in\N^6}\\
    \nn&\Bigg(\sum_{\substack{\beta'\leq \beta, \beta''\le\beta-\beta'\\|\beta'|\le|\beta|/2}}
    \sum_{\substack{n\leq m\\ n'\le m-n }}\sum_{l\in \Z^3,l\ne k}\int_{\fr{2t}{3}}^t \frac{b_{m,n,n',s}}{|l-k|^4\la\nu^{\fr13}(t-\tau)\ra^{\fr32}}d\tau\Bigg)\\
    \nn&\times\sum_{\substack{\beta'\leq \beta, \beta''\le\beta-\beta'\\|\beta'|\le|\beta|/2}}
    \sum_{\substack{n\leq m\\ n'\le m-n }}\sum_{l\in \Z^3}\int_{\fr{2t}{3}}^t \la\nu^{\fr13}(t-\tau) \ra^{-\fr32}  \left\|a_{n,s}(\tau)\mathcal{F}_{x,v}\big[(\mu^\fr12f)^{(n+\beta')}\big]_l(\tau) \right\|^2_{L^2_\xi\cap L^\infty_\xi}\\
    \nn&\times\left\|a_{n',s}(\tau)\mathcal{F}_{x,v}\big[\pr_{v_iv_j}|\nb_x|^2\la Y\ra^2f^{(n'+\beta'')}\big]_{k-l}(\tau)\right\|^2_{L^2_\eta\cap L^1_\eta}\\
    \nn&
    \times \la\nu^{\fr13}(t-\tau) \ra^{3}\left\|a_{m-n-n',s}(0) \mathcal{F}_{x,v}\big[\la Y^*\ra^2\mathfrak{g}^{(m-n-n'+\beta-\beta'-\beta'')}\big]_k(t-\tau)\right\|^2_{L^2_\eta} d\tau dt\\
    \nn\les&\nu^{\fr53}\int_0^{T}\fr{1}{\la t\ra^{1-a+s}}\int_{\fr{2t}{3}}^t \la\nu^\fr13(t-\tau)\ra^{-\fr32} \left\|(\mu^\fr12f)(\tau)\right\|_{\mathcal{G}^{\lm,N}_{s,2}}^2\\
    \nn&\times\left\|\pr_{v_iv_j}|\nb_x|^2\la Y\ra^4f(\tau)\right\|^2_{\mathcal{G}^{\lm,\fr{N}{2}}_{s,0}} \Big(\la\nu^\fr13(t-\tau)\ra^{3}\left\|\la Y^*\ra^2\mathfrak{g}(t-\tau)\right\|^2_{\bar{\mathcal{G}}^{\lm,\fr{N}{2}}_{s,0}}\Big) d\tau dt\\
    \les&{ \left\| (\mu^\fr12f)\right\|_{L^\infty_t\mathcal{G}^{\lm,N}_{s,2}}^2 \left\|\nu^{\fr23}\pr_{v_iv_j}|\nb_x|^2\la Y\ra^4f\right\|^2_{L^\infty_t\mathcal{G}^{\lm,\fr{N}{2}}_{s,0}}}\sup_{t}\Big(\la\nu^\fr13t\ra^{3}\left\|\la Y^*\ra^2\mathfrak{g}(t)\right\|^2_{\bar{\mathcal{G}}^{\lm,\fr{N}{2}}_{s,0}}\Big).
\end{align}

If $(t,\tau,k,l)\in D$, namely, $|kt-l\tau|\le\fr{t}{2}$ and $l\ne k$, then \eqref{lm-gap3} holds. Furthermore, if $l=0$, $|kt-l\tau|\le\fr{t}{2}$ cannot hold due to $|k|\ne0$. In the following, we focus on the case $l\ne0$.

{\it Case 1: $(t,\tau,k,l)\in D_1=\{|kt-l\tau|\le\fr{t}{2}, l\ne k, l\ne0, |l\tau|\le2|\xi|\}$.} In this case, there holds
$\fr{1}{|\xi|^2}\le \fr{4}{|l\tau|^2}$, then one can treat $\mathbb{I}_{1,k;3)}^{m,\beta}(t)$  in the same way as that in \eqref{est-Ik3-1} again.

{\it Case 2: $(t,\tau,k,l)\in D_2=\{|kt-l\tau|\le\fr{t}{2}, l\ne k, l\ne0, |l\tau|>2|\xi|\}$.} Now we have 
\begin{align*}
\fr{2}{3}|\xi+l\tau|\le|l\tau|\le2|\xi+l\tau|.
\end{align*}
Combining this with  \eqref{lm-gap3}, for $|n|\ge1$ and $\Theta$ satisfying  \eqref{Theta1}, similar to \eqref{lm-gap5'}, we have
\begin{align*}
    \Big(\fr{\lm(t)}{\lm(\tau)}\Big)^{|n|}\fr{\tau}{|l|^2}{\bf1}_{D_2}\les {\rm w}_2(\tau,l,\xi, n),
\end{align*}
where
\begin{align*}
    {\rm w}_2(\tau,l,\xi, n)=\begin{cases}
\fr{|\xi+ l\tau|^{1+a\Theta}}{|n|^\Theta},\quad {\rm if}\quad |n|>0;\\[2mm]
\fr{|\xi+l\tau|}{|l|^3},\quad\quad  \ \ {\rm if}\quad |n|=0.
    \end{cases}
\end{align*}
Then similar to \eqref{est-resoance}, thanks to Lemma \ref{lem-interp1} and using \eqref{decom-f}, we arrive at
\begin{align}\label{est-Ik3-4}
\nn&\int_0^{T}\la t\ra^{3+a-s}\sum_{k\in\Z^3_*} \sum_{|\beta|\le N}\kappa^{2|\beta|}\sum_{m\in\N^6}a_{m,\lm,s}^2(t)\left|\mathbb{I}_{1,k;3)}^{m,\beta}(t)\right|^2dt\\
    \les\nn&\nu^2 \int_0^{T}\la t\ra^{1+a-s}\sum_{k\in\Z^3_*}  \sum_{|\beta|\le N}\kappa^{2|\beta|}\sum_{m\in\N^6}\Bigg[\sum_{\substack{l\ne k,\\ l\ne0}}\sum_{\substack{\beta'\leq \beta,\beta''\le\beta-\beta'\\|\beta'|>|\beta|/2}}
    \sum_{\substack{n\leq m\\n'\le m-n}}\\
    \nn&\int_{\fr{2t}{3}}^t \frac{b_{m,n,n',s}{ |l|}}{\la kt-l\tau\ra^2|k-l|^2} \left\|a_{n,s}(\tau)\Big(\fr{\lm(t)}{\lm(\tau)}\Big)^{|n|}\fr{\tau}{{ |l|^2}} {\bf 1}_{D_2}\mathcal{F}_{x,v}\big[(\mu^\fr12{ \nb_x}f)^{(n+\beta')}\big]_l(\tau)\right\|_{L^2_\eta\cap L^\infty_\eta}\\ 
    \nn&\times\left\|a_{n',\lm,s}(\tau)\mathcal{F}_{x,v}[\la Y\ra^2|\nb_x|^2\pr_{v_iv_j}f^{(n'+\beta'')}]_{k-l}(\tau)\right\|_{L^2_\eta\cap L^1_\eta }\\
    \nn& \times\left\|a_{m-n-n',\lm,s}(0) \la Y^*\ra^2\mathfrak{g}_k^{(m-n-n'+\beta-\beta'-\beta'')}(t-\tau)\right\|_{L^2_v} d\tau\Bigg]^2dt\\
    \les\nn&\nu^2 \int_0^{T}\sum_{k\in\Z^3_*}  \sum_{|\beta|\le N}\kappa^{2|\beta|}\sum_{m\in\N^6} \\
    \nn&\Bigg(\sum_{\substack{l\ne k\\ l\ne 0}}\sum_{\substack{\beta'\leq \beta,\beta''\le\beta-\beta'\\|\beta'|>|\beta|/2}}
    \sum_{\substack{n\leq m\\n'\le m-n}} \int_{\fr{2t}{3}}^t \frac{ b_{m,n,n',s}|l|}{\la kt-l\tau\ra^2| k-l|^4}d\tau\Bigg)\\
    \nn&\times\sum_{\substack{l\ne k\\ l\ne0}}\sum_{\substack{\beta'\leq \beta,\beta''\le\beta-\beta'\\|\beta'|>|\beta|/2}}
    \sum_{\substack{n\leq m\\n'\le m-n}} \int_{\tau}^{T} \frac{|k|}{\la kt-l\tau\ra^2} \\
    \nn&\times \left\|a_{n,s}(\tau){\rm w}_2(\tau,l,\eta,n)\mathcal{F}_{x,v}\big[(\mu^\fr12\nb_xf)^{(n+\beta')}\big]_l(\tau)\right\|_{L^2_\eta\cap L^\infty_\eta}^2\\
     \nn&\times \la \tau\ra^{1+a-s}\left\|a_{n',s}(\tau)\mathcal{F}_x\big[\la Y\ra^4| \nb_x|^3\pr_{v_iv_j}f^{(n'+\beta'')}\big]_{k-l}(\tau)\right\|_{L^2_v}^2\\
    \nn& \times \left\|a_{m-n-n',s}(0) \la Y^*\ra^2\mathfrak{g}_k^{(m-n-n'+\beta-\beta'-\beta'')}(t-\tau)\right\|_{L^2_v}^2 dtd\tau\\
    \nn\les&\nu^{\fr13}\int_0^{T} \sum_{k\in\Z^3_*}\sum_{l\in\Z^3_*}\int_{\tau}^{T^*}\fr{|k|}{\la kt-l\tau\ra^2}\left\|\mathcal{F}_x\big[\mu^{\fr12}\nb_xf\big]_{l}(\tau)\right\|^2_{\tl{\mathcal{G}}^{\lm,N}_{s,2}} \left\|\la Y^*\ra^2\mathfrak{g}_k(t-\tau)\right\|^2_{\tl{\mathcal{G}}^{\lm(0),\fr{N}{2}}_{s,0}} \\
    \nn&\times \nu^{\fr13}\la \tau\ra^{1+a-s}\left\|\mathcal{F}_x\big[\la Y\ra^4| \nb_x|^3\nu^{\fr23}\pr_{v_iv_j}f(\tau)\big]_{k-l}\right\|^2_{\mathcal{\tl{G}}^{\lm,\fr{N}{2}}_{s,0}}dtd\tau\\
    \nn\les&{ \Bigg[\nu^{\fr13}\int_0^{T}\nu^{\fr13}\la\tau\ra \left\|\nu^{\fr23}\pr_{v_iv_j}g_{\ne}(\tau)\right\|^2_{\mathcal{G}^{\lm,\fr{N}{2}}_{s,0}}d\tau} \\
    \nn&{ +\nu^{\fr23}\int_0^{T}\la\tau\ra^{1+a-s} \left\|\big(q(\phi)\nu^{\fr23}\pr_{v_iv_j}g\big)(\tau)\right\|^2_{\mathcal{G}^{\lm,N}_{s,0}}d\tau\Bigg]}\\
    &\times\left\|\la Y^*\ra^2\mathfrak{g}\right\|^2_{L^\infty_t\bar{\mathcal{G}}^{\lm(0),\fr{N}{2}}_{s,0}}\left\|\mu^{\fr12}\nb_xf\right\|^2_{L^\infty_t\mathcal{G}^{\lm,N}_{s,2}},
\end{align}
provided $N\ge14$.

{$\bullet$  \it Treatments of $\mathbb{I}_{1,k;2)}^{m,\beta}(t)$.} In  view of \eqref{decom-f} again, we write
\begin{align}\label{decom-I12}
    \mathbb{I}_{1,k;2)}^{m,\beta}(t)
    =\nn&8\pi\nu \sum_{\substack{\beta'\leq \beta\\|\beta'|\leq |\beta|/2}}C_{\beta}^{\beta'}\sum_{\beta''\le\beta-\beta'}C_{\beta-\beta'}^{\beta''}
    \sum_{n\leq m}C_{m}^{n}\sum_{n'\le m-n}C_{m-n}^{n'}\\
    \nn&\times\int_{\fr{2t}{3}}^t\int_{\R^6} \frac{(\eta_i-\xi_i)(\eta_j-\xi_j)}{|\eta-\xi|^4}\mathcal{F}_{x,v}\big[Z^{n+\beta'}(\mu^\fr12f)\big]_0(\tau,\eta-\xi) \\
    \nn&\times\Big(\mathcal{F}_{x,v}\big[Z^{n'+\beta''}\big(q(\phi)\pr_{v_iv_j}g\big)\big]_{k}(\tau,\xi)+\mathcal{F}_{x,v}\big[Z^{n'+\beta''}\pr_{v_iv_j}g\big]_{k}(\tau,\xi)\Big)\\
    \nn&\times (0,-\eta+k(t-\tau))^{m-n-n'+\beta-\beta'-\beta''}\bar{\hat{\mathfrak{g}}}_{k}(t-\tau,\eta)d\xi d\eta d\tau\\
    =&\mathbb{I}_{1,k;2),e}^{m,\beta}(t)+\mathbb{I}_{1,k;2),d}^{m,\beta}(t).
\end{align}
$\mathbb{I}_{1,k;2),e}^{m,\beta}(t)$ is relatively easy to handle since the appearance of $q(\phi)$ provides us enough decay. Indeed,
\begin{align}\label{est-I12e-1}
\nn&\sum_{k\in\Z^3_*} \sum_{|\beta|\le N}\kappa^{2|\beta|}\sum_{m\in\N^6}a_{m,\lm,s}^2(t)\left|\mathbb{I}_{1,k;2),e}^{m,\beta}(t)\right|^2\\
\les\nn&\nu^{\fr23} \sum_{k\in\Z^3_*} \sum_{|\beta|\le N}\kappa^{2|\beta|}\sum_{m\in\N^6}\Bigg[\sum_{\substack{\beta'\leq \beta, \beta''\le\beta-\beta'\\|\beta'|\le|\beta|/2}}
    \sum_{\substack{n\leq m\\n'\le m-n}}\int_{\fr{2t}{3}}^t\int_{\R^6} \frac{b_{m,n,n',s}}{\la \nu^{\fr13}(t-\tau)\ra^{\fr32}} \\
    \nn&\times\left|\fr{a_{n,s}(\tau)}{|\eta-\xi|^2}\mathcal{F}_{x,v}\big[(\mu^\fr12f)^{(n+\beta')}\big]_0(\tau,\eta-\xi)\right|\\ 
    \nn&\times\left|a_{n',s}(\tau)\mathcal{F}_{x,v}\Big[\big(q(\phi)\nu^{\fr23}\pr_{v_iv_j}g\big)^{(n'+\beta'')}\Big]_{k}(\tau,\xi)\right| \la \nu^{\fr13}(t-\tau)\ra^{\fr32}\\
    \nn&\times\left|a_{m-n-n',s}(0) \mathcal{F}_{x,v}\big[\mathfrak{g}^{(m-n-n'+\beta-\beta'-\beta'')}\big]_k(t-\tau,\eta)\right|d\xi d\eta d\tau\Bigg]^2\\
    \nn\les&\nu^{\fr23} \sum_{k\in\Z^3_*} \sum_{|\beta|\le N}\kappa^{2|\beta|}\sum_{m\in\N^6} \Bigg(\sum_{\substack{\beta'\leq \beta,\beta''\le\beta-\beta'\\|\beta'|\le|\beta|/2}}
    \sum_{\substack{n\leq m\\n'\le m-n}} \int_{\fr{2t}{3}}^t \frac{b_{m,n,n',s}}{\la \nu^{\fr13}(t-\tau)\ra^{\fr32}}d\tau\Bigg)\\
    \nn&\times\sum_{\substack{\beta'\leq \beta,\beta''\le\beta-\beta'\\|\beta'|\le|\beta|/2}}
    \sum_{\substack{n\leq m\\n'\le m-n}} \int_{\fr{2t}{3}}^t \frac{b_{m,n,n',s}}{\la \nu^{\fr13}(t-\tau)\ra^{\fr32}} \left\|a_{n,\lm,s}(\tau)\mathcal{F}_{x}\big[\la\pr_v\ra^2(\mu^\fr12f)^{(n+\beta')}\big]_0(\tau)\la v\ra^2\right\|_{L^2_v}^2\\
     \nn&\times\left\|a_{n',s}(\tau)\mathcal{F}_{x}\Big[\big(q(\phi)\nu^{\fr23}\pr_{v_iv_j}g\big)^{(n'+\beta'')}\Big]_k(\tau)\right\|_{L^2_v}^2\\
    \nn&\times \la\nu^{\fr13}(t-\tau)\ra^{3}\left\|a_{m-n-n',s}(0) \mathfrak{g}_k^{(m-n-n'+\beta-\beta'-\beta'')}(t-\tau)\right\|_{L^2_v}^2 d\tau\\
    \les
    \nn&\nu^{\fr13} \int_{\fr{2t}{3}}^t \frac{1}{\la \nu^{\fr13} (t-\tau)\ra^{\fr32}} \left\|\la\pr_v\ra^2(\mu^\fr12f_0)(\tau)\right\|_{\mathcal{G}^{\lm(\tau),\fr{N}{2}}_{s,2}}^2\left\|\big(q(\phi)\nu^{\fr23}\pr_{v_iv_j}g\big)(\tau)\right\|_{\mathcal{G}^{\lm(\tau),N}_{s,0}}^2\\
    &\times \la\nu^{\fr13}(t-\tau)\ra^{3}\left\| \mathfrak{g}(t-\tau)\right\|_{\bar{\mathcal{G}}^{\lm(0),N}_{s,0}}^2 d\tau,
\end{align}
and hence
\begin{align}\label{est-I12e-2}
\nn&\int_0^{T}\langle t\rangle^{3+a-s}\sum_{k\in\Z^3_*} \sum_{|\beta|\le N}\kappa^{2|\beta|}\sum_{m\in\N^6}a_{m,\lm,s}^2(t)\left|\mathbb{I}_{1,k;2),e}^{m,\beta}(t)\right|^2dt\\
    \les\nn&\nu^{\fr13}\int_0^{T}\int_{\tau}^{T} \frac{1}{\la\nu^{\fr13} (t-\tau)\ra^{\fr32}}\Big(\la\nu^{\fr13}(t-\tau)\ra^3\left\|\mathfrak{g}(t-\tau)\right\|_{\bar{\mathcal{G}}^{\lm(0),N}_{s,0}}^2\Big) dt\\
    \nn&\times\left\|\la \nb_v\ra^2(\mu^\fr12f_{0})(\tau)\right\|_{\mathcal{G}^{\lm(\tau),\fr{N}{2}}_{s,2}}^2 \Big(\la \tau\ra^{3+a-s}\left\|q(\phi)\nu^{\fr23}\pr_{v_iv_j}g(\tau)\right\|_{\mathcal{G}^{\lm(\tau),N}_{s,0}}^2\Big) d\tau\\
    \les\nn&\left\|\la\nb_v\ra^2(\mu^\fr12f_0)\right\|_{L^\infty_t\mathcal{G}^{\lm,\fr{N}{2}}_{s,2}}^2\sup_t\Big(\la\nu^{\fr13}t\ra^3\left\|\mathfrak{g}(t)\right\|_{\bar{\mathcal{G}}^{\lm(0),N}_{s,0}}^2\Big)\\
    &\times{  \int_0^{T} 
    \la \tau\ra^{3+a-s}\left\|q(\phi)\nu^{\fr23}\pr_{v_iv_j}g(\tau)\right\|_{\mathcal{G}^{\lm(\tau),N}_{s,0}}^2d\tau}.
\end{align}

The treatments of $\mathbb{I}_{1,k;2),d}^{m,\beta}(t)$ will be divided into three cases.

{\it  Case (i): $|\eta|\le 10|\eta-k(t-\tau)|$.} Now we have $|k(t-\tau)|\le11|\eta-k(t-\tau)|$, and thus
\begin{align*}
    1\les\fr{\la\eta-k(t-\tau) \ra^2}{\la k(t-\tau)\ra^2},\quad {\rm and}\quad |\xi|\les\la\eta-\xi\ra\la\eta-k(t-\tau)\ra.
\end{align*}
Then one can move one $v$ derivative on $g$  to $\mu^{\fr12}f$ and $\mathfrak{g}$, and similar to \eqref{est-Ik3-2}, we have
\begin{align}\label{est-I12d-1}
\nn&\sum_{k\in\Z^3_*} \sum_{|\beta|\le N}\kappa^{2|\beta|}\sum_{m\in\N^6}a_{m,\lm,s}^2(t)\left|\mathbb{I}_{1,k;2),d}^{m,\beta}(t)\right|^2\\
\les\nn&\nu^2 \sum_{k\in\Z^3_*} \sum_{|\beta|\le N}\kappa^{2|\beta|}\sum_{m\in\N^6}\Bigg[\sum_{\substack{\beta'\leq \beta, \beta''\le\beta-\beta'\\|\beta'|\le|\beta|/2}}
    \sum_{\substack{n\leq m\\n'\le m-n}}\int_{\fr{2t}{3}}^t\int_{\R^6} \frac{b_{m,n,n',s}}{\la k(t-\tau)\ra^2} \\
    \nn&\times\left|\fr{a_{n,s}(\tau)}{|\eta-\xi|^2}\mathcal{F}_{x,v}\big[\la\nb_v\ra(\mu^\fr12f)^{(n+\beta')}\big]_0(\tau,\eta-\xi)\right|\\ \nn&\times\left|a_{n',s}(\tau)\mathcal{F}_{x,v}\big[\pr_{v_j}g^{(n'+\beta'')}\big]_{k}(\tau,\xi)\right|\\
    \nn&\times\left|a_{m-n-n',s}(0) \mathcal{F}_{x,v}\big[\la Y^*\ra^3\mathfrak{g}^{(m-n-n'+\beta-\beta'-\beta'')}\big]_k(t-\tau,\eta)\right|d\xi d\eta d\tau\Bigg]^2\\
    \les\nn&\nu^2 \sum_{k\in\Z^3_*} \sum_{|\beta|\le N}\kappa^{2|\beta|}\sum_{m\in\N^6} \Bigg(\sum_{\substack{\beta'\leq \beta,\beta''\le\beta-\beta'\\|\beta'|\le|\beta|/2}}
    \sum_{\substack{n\leq m\\n'\le m-n}} \int_{\fr{2t}{3}}^t \frac{b_{m,n,n',s}}{\la k(t-\tau)\ra^2}d\tau\Bigg)\\
    \nn&\times\sum_{\substack{\beta'\leq \beta,\beta''\le\beta-\beta'\\|\beta'|\le|\beta|/2}}
    \sum_{\substack{n\leq m\\n'\le m-n}} \int_{\fr{2t}{3}}^t \frac{b_{m,n,n',s}}{\la k(t-\tau)\ra^2} \left\|a_{n,\lm,s}(\tau)\mathcal{F}_{x}\big[\la\nb_v\ra^3(\mu^\fr12f)^{(n+\beta')}\big]_0(\tau)\la v\ra^2\right\|_{L^2_v}^2\\ \nn&\times\left\|a_{n',s}(\tau)\mathcal{F}_{x}\big[\pr_{v_j}g^{(n'+\beta'')}\big]_k(\tau)\right\|_{L^2_v}^2\\
    \nn&\times\left\|a_{m-n-n',s}(0) \la Y^*\ra^3\mathfrak{g}_k^{(m-n-n'+\beta-\beta'-\beta'')}(t-\tau)\right\|_{L^2_v}^2 d\tau\\
    \les
    &\nu^2 \int_{\fr{2t}{3}}^t \frac{1}{\la t-\tau\ra^2} \left\|\la\nb_v\ra^3(\mu^\fr12f_0)(\tau)\right\|_{\mathcal{G}^{\lm(\tau),\fr{N}{2}}_{s,2}}^2\left\|\pr_{v_j}g_{\ne}(\tau)\right\|_{\mathcal{G}^{\lm(\tau),N}_{s,0}}^2\left\| \la Y^*\ra^3\mathfrak{g}(t-\tau)\right\|_{\bar{\mathcal{G}}^{\lm(0),N}_{s,0}}^2 d\tau,
\end{align}
and hence
\begin{align}\label{est-I12d-2}
\nn&\nu^2\int_0^{T}\langle t\rangle^{3+a-s}\sum_{k\in\Z^3_*} \sum_{|\beta|\le N}\kappa^{2|\beta|}\sum_{m\in\N^6}a_{m,\lm,s}^2(t)\left|\mathbb{I}_{1,k;2),d}^{m,\beta}(t)\right|^2dt\\
    \les\nn&\nu^{\fr13}\int_0^{T}\int_{\tau}^{T} \frac{1}{\la t-\tau\ra^2}\left\| \la Y^*\ra^3\mathfrak{g}(t-\tau)\right\|_{\bar{\mathcal{G}}^{\lm(0),N}_{s,0}}^2 dt\\
    \nn&\times\left\|\la \nb_v\ra^3(\mu^\fr12f_{0})(\tau)\right\|_{\mathcal{G}^{\lm(\tau),\fr{N}{2}}_{s,2}}^2 \Big(\nu\la\tau\ra^{3}\left\|\nu^{\fr13}\pr_{v_j}g_{\ne}(\tau)\right\|_{\mathcal{G}^{\lm(\tau),N}_{s,0}}^2\Big) d\tau\\
    \les\nn&\left\|\la\nb_v\ra^3(\mu^\fr12f_0)\right\|_{L^\infty_t\mathcal{G}^{\lm(\tau),\fr{N}{2}}_{s,2}}^2\left\| \la Y^*\ra^3\mathfrak{g}\right\|_{L^\infty_t\bar{\mathcal{G}}^{\lm(0),N}_{s,0}}^2\\
    &\times{  \nu^{\fr13}\int_0^{T} 
    \nu\la\tau\ra^3\left\|\nu^{\fr13}\pr_{v_j}g_{\ne}(\tau)\right\|_{\mathcal{G}^{\lm(\tau),N}_{s,0}}^2 d\tau}.
\end{align}

{\it Case (ii): $|\eta|\le1$.} Now we have $|\xi_i||\xi_j|\les\la\eta-\xi\ra^2$. Similar to \eqref{est-I12e-1} and \eqref{est-I12e-2}, we obtain
\begin{align}\label{est-I12d-eta<1-2}
\nn&\int_0^{T}\la t\ra^{3+a-s}\sum_{k\in\Z^3_*} \sum_{|\beta|\le N}\kappa^{2|\beta|}\sum_{m\in\N^6}a_{m,\lm,s}^2(t)\left|\mathbb{I}_{1,k;2),d}^{m,\beta}(t)\right|^2\\
    \les
    \nn&\nu^{\fr53} \int_{0}^{T}\int_{\tau}^{T} \frac{1}{\la \nu^{\fr13} (t-\tau)\ra^{\fr32}} \left\|\la\nb_v\ra^4(\mu^\fr12f_0)(\tau)\right\|_{\mathcal{G}^{\lm(\tau),\fr{N}{2}}_{s,2}}^2\Big(\la\tau\ra^{3+a-s}\left\|g_{\ne}(\tau)\right\|_{\mathcal{G}^{\lm(\tau),N}_{s,0}}^2\Big)\\
    \nn&\times \la\nu^{\fr13}(t-\tau)\ra^{3}\left\| \mathfrak{g}(t-\tau)\right\|_{\bar{\mathcal{G}}^{\lm(0),N}_{s,0}}^2 d\tau dt\\
    \les\nn&\nu^{\fr13}\int_0^T{  \nu\la\tau\ra^{3+a-s}}\left\|g_{\ne}(\tau)\right\|_{\mathcal{G}^{\lm(\tau),N}_{s,0}}^2d\tau \left\|\la\nb_v\ra^4(\mu^\fr12f_0)(\tau)\right\|_{L^\infty_t\mathcal{G}^{\lm,\fr{N}{2}}_{s,2}}^2\\
    &\times\sup_t\Big(\la\nu^{\fr13}t\ra^{3}\left\| \mathfrak{g}(t)\right\|_{\bar{\mathcal{G}}^{\lm(0),N}_{s,0}}^2\Big).
\end{align}

{\it Case (iii): $|\eta|> \max\{10|\eta-k(t-\tau)|,1\}$.} Now the restriction $|\eta|> \max\{10|\eta-k(t-\tau)|,1\}$ implies that $\fr{9}{10}|\eta|\le|k(t-\tau)|\le\fr{11}{10}|\eta|$. Noting that ${  \fr{2t}{3}\le \tau\le t}$, we have $|k(t-\tau)|\le\fr12|k\tau|$. Then $|\eta|\le\fr59|k\tau|$, and hence $\fr{9}{14}|\eta+k\tau|\le|k\tau|\le\fr{9}{4}|\eta+k\tau|$. Thus for any positive constant $\varepsilon$
\begin{align}\label{appro1}
    \nn1\approx& \fr{\la \eta\ra^{1+\varepsilon}}{\la k(t-\tau)\ra^{1+\varepsilon}}\fr{\la\tau\ra^{\fr{1+a}{2}}}{|k\tau|^{\fr{s}{2}}}\fr{|k\tau|^{\fr{s}{2}}}{\la\tau\ra^{\fr{1+a}{2}}}\approx \fr{\la \eta\ra^{1+\varepsilon}}{\la k(t-\tau)\ra^{1+\varepsilon}}\fr{\la\tau\ra^{\fr{1+a}{2}}}{|k\tau|^{\fr{s}{2}}}\fr{|\eta+k\tau|^{\fr{s}{2}}}{\la\tau\ra^{\fr{1+a}{2}}}\\
    \les&\la \eta-\xi\ra^{\fr{s}{2}}\fr{\la \eta\ra^{1+\varepsilon}}{\la k(t-\tau)\ra^{1+\varepsilon}}\fr{\la\tau\ra^{\fr{1+a}{2}}}{|k\tau|^{\fr{s}{2}}}\fr{\la\xi+k\tau\ra^{\fr{s}{2}}}{\la\tau\ra^{\fr{1+a}{2}}}.
\end{align}
Next, we use the gap between $\lm(t)$ and $\lm(\tau)$ to absorb the extra factor $\la\eta\ra^{1+\varepsilon}|\xi_i||\xi_j|$.
More precisely,  similar to \eqref{lm-gap1} and \eqref{lm-gap2}, keeping in mind that $\tau\approx t$, now we infer from \eqref{lm-gap} and the fact $t-\tau\approx \fr{|\eta|}{|k|}$ that there exists a positive constant $\underline{c}'$, depending on $a$, $\tl{\dl}$ and $\lm_\infty$, such that
\begin{align}\label{lm-gap4}
    \fr{\lm(t)}{\lm(\tau)}\le e^{-\underline{c}'\fr{|\eta|}{|k|\la\tau\ra^{1+a}}}.
\end{align}
Then for $|n'|>0$,
\begin{align*}
    \nn&\la \eta\ra^{1+\varepsilon}|\xi_i||\xi_j|\left(\fr{\lm(t)}{\lm(\tau)}\right)^{|n'|}{\bf 1}_{|\eta|\ge1}\les\la\eta-\xi\ra^2 \la\eta\ra^{3+\varepsilon}e^{-\underline{c}'\fr{|n'||\eta|}{|k|\la\tau\ra^{1+a}}}{\bf 1}_{|\eta|\ge1}\\
    \les&\la\eta-\xi\ra^2 \fr{\la\eta\ra^{3+\varepsilon}}{|\eta|^{3+\varepsilon}}\left(\fr{|k|\la\tau\ra^{1+a}}{|n'|}\right)^{3+\varepsilon}{\bf 1}_{|\eta|\ge1}\les \la\eta-\xi\ra^2 \left(\fr{|k|\la\tau\ra^{1+a}}{|n'|}\right)^{3+\varepsilon}.
\end{align*}
Furthermore, 
\begin{align*}
    \left(\fr{|k|\la\tau\ra^{1+a}}{|n'|}\right)^{3+\varepsilon}\nn=&\Big(|k|^{(1-s)(3+\varepsilon)}\la\tau\ra^{(1+a-s)(3+\varepsilon)}\Big)\fr{(|k|\la\tau\ra)^{s(3+\varepsilon)}}{|n'|^{3+\varepsilon}}\\
    \nn\les&\Big(|k|^{(1-s)(3+\varepsilon)}\la\tau\ra^{(1+a-s)(3+\varepsilon)}\Big)\fr{|\eta+k\tau|^{s(3+\varepsilon)}}{|n'|^{3+\varepsilon}}\\
    \les&\la\eta-\xi\ra^{s(3+\varepsilon)}\Big(|k|^{(1-s)(3+\varepsilon)}\la\tau\ra^{(1+a-s)(3+\varepsilon)}\Big)\fr{\la\xi+k\tau\ra^{s(3+\varepsilon)}}{|n'|^{3+\varepsilon}}.
\end{align*}
If $|n'|=0$, using the fact $|\eta|\les |k\tau|\approx |\eta+k\tau|\les \la\eta-\xi\ra\la\xi+k\tau \ra$, we find that
\begin{align*}
    \la \eta\ra^{1+\varepsilon}|\xi_i||\xi_j|\left(\fr{\lm(t)}{\lm(\tau)}\right)^{|n'|}{\bf 1}_{|\eta|\ge1}\les\la\eta-\xi\ra^2 \la\eta\ra^{3+\varepsilon}\les \la \eta-\xi\ra^{5+\varepsilon}\la\xi+k\tau\ra^{3+\varepsilon}.
\end{align*}
It follows from the above three inequalities that
\begin{align}\label{lm-gap6}
\nn&\la\eta\ra^{1+\varepsilon}|\xi_i||\xi_j|\left(\fr{\lm(t)}{\lm(\tau)}\right)^{|n'|}{\bf 1}_{|\eta|\ge1}\\
\les& \la \eta-\xi\ra^{5+\varepsilon}\Big(|k|^{(1-s)(3+\varepsilon)}\la\tau\ra^{(1+a-s)(3+\varepsilon)}\Big) {\rm w_3}(\tau,k,\xi,n'),
\end{align}
where
\begin{align*}
    {\rm w_3}(\tau,k,\xi,n')=\begin{cases}
\fr{\la\xi+k\tau\ra^{s(3+\varepsilon)}}{|n'|^{3+\varepsilon}},\quad\, {\rm if}\quad |n'|>0;\\
\la\xi+k\tau\ra^{3+\varepsilon},\quad{\rm if} \quad |n'|=0.
    \end{cases}
\end{align*}
Notice that
\begin{align}\label{ktau-power}
    \nn&\la\tau\ra^{\fr{3+a-s}{2}}\fr{\la\tau\ra^{\fr{1+a}{2}}}{|k\tau|^{\fr{s}{2}}}\Big(|k|^{(1-s)(3+\varepsilon)}\la\tau\ra^{(1+a-s)(3+\varepsilon)}\Big)\\
    \les&|k|^{(1-s)(3+\varepsilon)-\fr{s}{2}}\la\tau\ra^{2+a-s+(1+a-s)(3+\varepsilon)}.
\end{align}
For $s>\fr12$, we can choose $a$ and $\varepsilon$ so small that
\begin{align}\label{restriction-s>1/2}
    |k|^{(1-s)(3+\varepsilon)-\fr{s}{2}}\le { |k|^2}\quad{\rm and }\quad\la\tau\ra^{2+a-s+(1+a-s)(3+\varepsilon)}\le\la\tau\ra^{3}.
\end{align}
Now taking ${\bf a}=\fr{s(3+\varepsilon)-1}{3+\varepsilon}$, ${\bf b}=1$ and $\Theta=3+\varepsilon$ in Lemma \ref{lem-interp1},  we infer from \eqref{lm-gap6}--\eqref{restriction-s>1/2} that
\begin{align}\label{est-I12d-main}
\nn&\int_0^{T}\sum_{k\in\Z^3_*}\la t\ra^{3+a-s} \sum_{|\beta|\le N}\kappa^{2|\beta|}\sum_{m\in\N^6}a_{m,\lm,s}^2(t)\left|\mathbb{I}_{1,k;2),d}^{m,\beta}(t)\right|^2dt\\
    \les\nn& \int_0^{T}\sum_{k\in\Z^3_*} \sum_{|\beta|\le N}\kappa^{2|\beta|}\sum_{m\in\N^6}\Bigg[\sum_{\substack{\beta'\leq \beta,\beta''\le\beta-\beta'\\|\beta'|\le|\beta|/2}}
    \sum_{\substack{n\leq m\\n'\le m-n}}\int_{\fr{2t}{3}}^t\int_{\R^6}  \fr{{ |k|}b_{m,n,n',s}}{\la k(t-\tau)\ra^{1+\varepsilon}}\\
    \nn&\times\left|\fr{a_{n,s}(\tau)}{|\eta-\xi|^2}\mathcal{F}_{x,v}\big[\la\nb_v \ra^{\fr{s}{2}+5+\varepsilon}(\mu^\fr12f)^{(n+\beta')}\big]_0(\tau,\eta-\xi)\right|\\ 
    \nn&\times \nu\tau^3\left|\fr{a_{n',s}(\tau)}{\la\tau\ra^{\fr{1+a}{2}}}{\rm w_3}(\tau,k,\xi,n')\mathcal{F}_{x,v}\big[\la Y\ra^{\fr{s}{2}}{ \nb_x}g^{(n'+\beta'')}\big]_{k}(\tau,\xi)\right|\\
    \nn&\times \left|a_{m-n-n',s}(0) \mathcal{F}_{x,v}\big[\mathfrak{g}^{(m-n-n'+\beta-\beta'-\beta'')}\big]_k(t-\tau,\eta)\right| d\xi d\eta d\tau\Bigg]^2dt\\
    \les\nn& \int_0^{T}\sum_{k\in\Z^3_*} \sum_{|\beta|\le N}\kappa^{2|\beta|}\sum_{m\in\N^6}
    \Bigg(\sum_{\substack{\beta'\leq \beta, \beta''\le\beta-\beta'\\|\beta'|\le|\beta|/2}}\sum_{\substack{n\leq m\\n'\le m-n}}\int_{\fr{2t}{3}}^t\fr{|k|b_{m,n,n',s}}{\la k(t-\tau)\ra^{1+\varepsilon}}d\tau\Bigg)\\
    \nn&\times \sum_{\substack{\beta'\leq \beta, \beta''\le\beta-\beta'\\|\beta'|\le|\beta|/2}}\sum_{\substack{n\leq m\\n'\le m-n}}\int_{\tau}^{T}  \fr{|k|b_{m,n,n',s}}{\la k(t-\tau)\ra^{1+\varepsilon}} \\
    \nn&\times\left\|a_{n,s}(\tau)\mathcal{F}_{x,v}\big[\la\nb_v \ra^{\fr{s}{2}+5+\varepsilon}(\mu^\fr12f)^{(n+\beta')}\big]_0(\tau,\eta)\right\|^2_{L^1_\eta\cap L^\infty_\eta}\\ 
    \nn&\times (\nu\la\tau\ra^3)^2\left\|\fr{a_{n',s}(\tau)}{\la\tau\ra^{\fr{1+a}{2}}}{\rm w}_3(\tau,k,\eta,n')\mathcal{F}_{x,v}\big[\la Y\ra^{\fr{s}{2}}\nb_xg^{(n'+\beta'')}\big]_{k}(\tau,\eta)\right\|^2_{L^2_\eta}\\
    \nn&\times \left\|a_{m-n-n',s}(0) \mathcal{F}_{x,v}\big[\mathfrak{g}^{(m-n-n'+\beta-\beta'-\beta'')}\big]_k(t-\tau)\right\|^2_{L^2_v}  dtd\tau \\
    \les\nn& { \int_{0}^{T}   (\nu\la\tau\ra^3)^2\left\|\fr{1}{\la\tau\ra^{\fr{1+a}{2}}}\la Y\ra^{\fr{s}{2}}\nb_xg(\tau)\right\|^2_{\mathcal{G}^{\lm(\tau),N}_{s,0}}d\tau}\\
    &\times \left\|\la\nb_v \ra^{\fr{s}{2}+7+\varepsilon}(\mu^\fr12f_0)\right\|^2_{L^\infty_t\mathcal{G}^{\lm(\tau),\fr{N}{2}}_{s,2}}\|\mathfrak{g}\|^2_{L^\infty_t\bar{\mathcal{G}}^{\lm(0),N}_{s,0}}.
\end{align}

{$\bullet$  \it Treatments of $\mathbb{I}_{4,k}^{m,\beta}(t)_{\tau\approx t}$.} A simple but key observation is that $(\mu^{\fr12}f)$ and $f$ cannot be at zero mode at the same time. Thus, by virtue of the low-high and high-low decomposition, it is not difficult to verify that
\begin{align}\label{est-I4k-tau=t}
    \nn&\int_0^{T}\la t\ra^{3+a-s}\sum_{k\in\Z^3_*} \sum_{|\beta|\le N}\kappa^{2|\beta|}\sum_{m\in\N^6}a_{m,\lm,s}^2(t)\left|\mathbb{I}_{4,k}^{m,\beta}(t)_{\tau\approx t}\right|^2dt\\
    \les\nn&\nu^{\fr43}\int_0^{T}\la\tau\ra^{3+a-s}\Big(\Big\|\la Y\ra^2\la\nb_x\ra^2(\mu^\fr12f_{\ne})(\tau)\Big\|_{\mathcal{G}^{\lm,\fr{N}{2}}_{s,0}}^2\|f(\tau)\|^2_{\mathcal{G}^{\lm,N}_{s,0}}\\
    \nn&+\Big\|\la Y\ra^2\la\nb_x\ra^2(\mu^\fr12f)(\tau)\Big\|_{\mathcal{G}^{\lm,\fr{N}{2}}_{s,0}}^2\|f_{\ne}(\tau)\|^2_{\mathcal{G}^{\lm,N}_{s,0}}\Big)d\tau\sup_t\Big(\la\nu^{\fr13}t\ra^3\|\mathfrak{g}(t)\|^2_{\bar{\mathcal{G}}^{\lm(0),N}_{s,0}}\Big)\\
    \nn&+\nu^{\fr43}\int_0^{T}\la\tau\ra^{3+a-s}\Big(\Big\|(\mu^\fr12f_{\ne})(\tau)\Big\|_{\mathcal{G}^{\lm,N}_{s,0}}^2\|\la Y\ra^2\la\nb_x\ra^2f(\tau)\|^2_{\mathcal{G}^{\lm,\fr{N}{2}}_{s,0}}\\
    &+\Big\|(\mu^\fr12f)(\tau)\Big\|_{\mathcal{G}^{\lm,\fr{N}{2}}_{s,0}}^2\|\la Y\ra^2\la\nb_x\ra^2f_{\ne}(\tau)\|^2_{\mathcal{G}^{\lm,\fr{N}{2}}_{s,0}}\Big)d\tau \sup_t\Big(\la\nu^{\fr13}t\ra^3\|\mathfrak{g}(t)\|^2_{\bar{\mathcal{G}}^{\lm(0),\fr{N}{2}}_{s,0}}\Big).
\end{align}

Collecting the estimates in \eqref{est-T2E2HL-short},\eqref{est-T2E2HL-t=tau}, \eqref{est-T2E2LH-short}--\eqref{est-tlT2E2}, \eqref{est-I1k-short}, \eqref{est-I4k-short}, \eqref{est-I11}, \eqref{est-Ik3-3}--\eqref{est-Ik3-4}, \eqref{est-I12e-2}, \eqref{est-I12d-2}, \eqref{est-I12d-eta<1-2}, \eqref{est-I12d-main}, and \eqref{est-I4k-tau=t}, and using  \eqref{decom-f}, \eqref{fg1}, \eqref{fg2}, Lemmas \ref{lem-compose}, \ref{lem-Linfty-rho}, \ref{lem:est-S^*}, and Corollary \ref{coro-CK}, under the restrictions \eqref{restr-N} and \eqref{res-s-a-rho2}, we obtain \eqref{bd-rho2-iota=0}.
\end{proof}

\begin{prop}\label{prop-rho2-iota}
    Under the bootstrap hypotheses \eqref{bd-en}--\eqref{H-phi}, if \eqref{restr-N} and \eqref{res-s-a-rho2} hold, and
\begin{align}\label{res-s-a-iota}
    a\le \fr{2s}{[\ell/3]+10},
\end{align}
then for $\iota=1, 2, \cdots, [\ell/3]+1$, there holds
\begin{align}\label{est-rho2-iota}
    \nn&\big\|\la t\ra^{\fr{1+a-s}{2}}w_{\iota}(t)\rho^{(2)}\big\|^2_{L^2_t\mathcal{G}^{\lm,N}_s}\\
    \les\nn&\big\|\la t\ra^{\fr{1+a-s}{2}}w_{\iota}(t)\rho^{(2)}\big\|^2_{L^2_t\mathcal{G}^{\lm,N}_s}\|g\|^2_{L^\infty_t\mathcal{G}^{\lm,N}_{s,2}}\\
    \nn&+\Big(\|g\|^2_{L^\infty_t\mathcal{G}^{\lm,N}_{s,2}}+\|w_{\iota}g_{\ne}\|^2_{L^\infty_t\mathcal{G}^{\lm,N}_{s,2}}\Big)\sum_{|\al|\le2}\big\|\nu^{\fr{|\al|}{3}}\pr_v^\al g\big\|^2_{L^\infty_t\mathcal{G}^{\lm,N}_{s,0}}\\
    \nn&+\big\|\la t\ra^{\fr{1+a-s}{2}}w_{\iota}(t)\rho\big\|^2_{L^2_t\mathcal{G}^{\lm,N}_s}\|g\|^2_{L^\infty_t\mathcal{G}^{\lm,N}_{s,2}}\sum_{|\al|\le2}\big\|\nu^{\fr{|\al|}{3}}\pr_v^\al g\big\|^2_{L^\infty_t\mathcal{G}^{\lm,N}_{s,0}}\\
    \nn&+\nu^{\fr13}\int_0^{T}\nu\la\tau\ra^{1+a-s}w_{\iota}^2(\tau)\sum_{|\al|\le1}\big\|\nu^{\fr{|\al|}{3}}\pr_v^\al g_{\ne}(\tau)\big\|^2_{\mathcal{G}^{\lm,N}_{s,2}}d\tau\|g\|^2_{L^\infty_t\mathcal{G}^{\lm,N}_{s,2}}\\
    \nn&+\|g\|^2_{L^\infty_t\mathcal{G}^{\lm,N}_{s,0}}\sum_{|\al|\le2}\big\|w_{\iota}\nu^{\fr{|\al|}{3}}\pr_v^\al g_{\ne}\big\|^2_{L^\infty_t\mathcal{G}^{\lm,N}_{s,0}}\\
    \nn&+\nu^{\fr23}\sum_{|\al|\le2}\big\|w_{\iota}\nu^{\fr{|\al|}{3}}\pr_v^\al g_{\ne}\big\|^2_{L^\infty_t\mathcal{G}^{\lm,N}_{s,0}}\|\la\nb_x\ra g\|^2_{L^\infty_t\mathcal{G}^{\lm,N}_{s,2}}\\
    &+\|g\|^2_{L^\infty_t\mathcal{G}^{\lm,N}_{s,2}}\int_0^{T}(\nu\la\tau\ra^2)^2w_{\iota}^2(\tau)\Big(\mathfrak{CK}_0[\nb_xg(\tau)]+\fr{1}{\la\tau\ra^{1+a}}\|\nb_xg(\tau)\|^2_{\mathcal{G}^{\lm,N}_{s,0}}\Big)d\tau.
\end{align}
\end{prop}    
\begin{proof}
By \eqref{es-rho:N2-iota}, we will estimate $\Big\|
    \langle t\rangle^{\fr{1+a-s}{2}} w_{\iota}(t)\mathcal{N}^{(2)}\Big\|_{L^2_t\mathcal{G}^{\lm,N}_s}$ term by term.

\noindent{$\diamond$ \underline{\it Estimates of $\mathcal{T}_{\pr_{v_j}f}E^j_2$ and $\mathcal{T}_{{v_j}f}E^j_2$}.}
Let us denote 
\begin{align*}
    {\bf T}_{\iota}^{(2)}[E_2]=&\left\| \la t\ra^{\fr{1+a-s}{2}}w_{\iota}(t)\int_0^t\int_{\R^3}S_k(t-\tau)\big[\mathcal{F}_x[\mathcal{T}_{\pr_{v_j}f}E^j_2]_k(\tau,v)\big]\mu^{\fr12}(v)dvd\tau\right\|^2_{L^2_t\mathcal{G}^{\lm,N}_{s}}\\
=&\int_0^{T}\la t\ra^{1+a-s}w^2_{\iota}(t)\sum_{k\in\Z^3_*}\sum_{|\beta|\le N}\kappa^{2|\beta|}\sum_{m\in\N^6}\Bigg|a_{m,s}(t)(k,kt)^{m}\\
&\times\int_0^t\int_{\R^3}\Big(\sum_{\substack{\beta'\le\beta,|\beta'|\le|\beta|/2\\\beta''\le\beta-\beta'}}+\sum_{\substack{\beta'\le\beta,|\beta'|>|\beta|/2\\\beta''\le\beta-\beta'}}\Big)\sum_{l\in\Z^3_*}\sum_{J\ge8}(l,l\tau)^{\beta'}\hat{\rho}^{(2)}_l(\tau)_J\fr{l_j}{|l|^2}\\
&\times \eta_j(k-l,\eta+(k-l)\tau)^{\beta''}\hat{f}_{k-l}(\tau,\eta)_{<J/8}(0,-\eta+k(t-\tau))^{\beta-\beta'-\beta''}\bar{\hat{\mathfrak{g}}}_k(t-\tau,\eta) d\eta d\tau\Bigg|^2dt\\
=&{\bf T}_{\iota}^{(2);\rm LH}[E_2]+{\bf T}_{\iota}^{(2);\rm HL}[E_2].
\end{align*}
Moreover, if $\pr_{v_j}f$ is replaced by $v_jf$, we use $\tl{\bf T}_{\iota}^{(2)}[E_2]$ to replace ${\bf T}_{\iota}^{(2)}[E_2]$.

If $\tau\le\fr23t$, we use the gap between $\lm(t)$ and $\lm(\tau)$ to absorb the extra weight $w_\iota(t)$. Indeed, noting that $w_{\iota}^2(t)\les \la t\ra^\iota\les \la t\ra^{[\ell/3]+1}$, similar to \eqref{eq:gain m2}, for any $\Theta>0$, $|m|>0$, we have
 \begin{align}\label{eq:gain m3}
     w_{\iota}(t)\la t\ra^{\fr{1+a-s}{2}+3} \left(\fr{\lm(t)}{\lm(\tau)}\right)^{|m|}\le C\fr{|kt|^{\fr{[\ell/3]}{2}+4+a\Theta}}{|k|^{\fr{[\ell/3]}{2}+4+a\Theta}|m|^\Theta}.
 \end{align}
Taking $\Theta=\fr{\fr{[\ell/3]}{2}+4}{s-a}$, for $s\ge (\fr{[\ell/3]}{2}+5)a$, similar to \eqref{est-T2E2HL-short} and \eqref{est-T2E2LH-short}, we arrive at
\begin{align}\label{est-T2E2-iota-short}
    {\bf T}^{(2)}_\iota[E_2]_{\tau\le\fr{2t}{3}}\les \big\|\rho^{(2)}\big\|^2_{L^2_t\mathcal{G}^{\lm,N}_s}{ \big\|f \big\|^2_{L^\infty_t  \mathcal{G}^{\lm, N}_{s,0}}}
    \left\| \mathfrak{g}\right\|^2_{L^\infty_t  \bar{\mathcal{G}}^{\lm(0), N}_{s,0}},
\end{align}
and
\begin{align}\label{est-tlT2E2-iota-short}
    \tl{\bf T}^{(2)}_\iota[E_2]_{\tau\le\fr{2t}{3}}\les \big\|\rho^{(2)}\big\|^2_{L^2_t\mathcal{G}^{\lm,N}_s}{ \big\|vf \big\|^2_{L^\infty_t  \mathcal{G}^{\lm, N}_{s,0}}}
    \left\| \mathfrak{g}\right\|^2_{L^\infty_t  \bar{\mathcal{G}}^{\lm(0), N}_{s,0}}.
\end{align}

If $\fr23t\le\tau\le t$, then $w_{\iota}(t)\approx w_{\iota}(\tau)$, which will hit on $\rho(\tau)$, thus the proofs for the case $\iota=0$ apply here. Similar to \eqref{est-T2E2HL-t=tau},\eqref{est-T2E2LH-t=tau} and \eqref{est-tlT2E2}, we find that
\begin{align}\label{est-T2E2-iota-tau=t}
    {\bf T}^{(2)}_\iota[E_2]_{\tau\approx t}\les \big\|\la t\ra^{\fr{1+a-s}{2}}w_{\iota}(t)\rho^{(2)}\big\|^2_{L^2_t\mathcal{G}^{\lm,N}_s}{ \big\|f \big\|^2_{L^\infty_t  \mathcal{G}^{\lm, N}_{s,0}}}
    \left\|\la Y^*\ra^3 \mathfrak{g}\right\|^2_{L^\infty_t  \bar{\mathcal{G}}^{\lm(0), N}_{s,0}},
\end{align}
and
\begin{align}\label{est-tlT2E2-iota-tau=t}
    \tl{\bf T}^{(2)}_\iota[E_2]_{\tau\approx t}\les \big\|\la t\ra^{\fr{1+a-s}{2}}w_{\iota}(t)\rho^{(2)}\big\|^2_{L^2_t\mathcal{G}^{\lm,N}_s}{ \big\|vf \big\|^2_{L^\infty_t  \mathcal{G}^{\lm, N}_{s,0}}}
    \left\|\la Y^*\ra^2 \mathfrak{g}\right\|^2_{L^\infty_t  \bar{\mathcal{G}}^{\lm(0), N}_{s,0}}.
\end{align}
\noindent{$\diamond$ \underline{\it The collision contributions}.}
Now we consider
\begin{align*}
&\nu\left\|\langle t\rangle^{\fr{1+a-s}{2}}w_{\iota}(t) \int_0^t\int_{\R^3}S_k(t-\tau)[(\widehat{\Gamma(f,f)})_k(\tau)]\mu^{\fr12}(v)dvd\tau \right\|_{L^2_t\mathcal{G}^{\lm,N}_{s}}\\
=&\nu\Bigg[\int_0^{T}\langle t\rangle^{1+a-s}w_{\iota}^2(t)\sum_{k\in\Z^3_*} \sum_{|\beta|\le N}\kappa^{2|\beta|}\sum_{m\in\N^6}a_{m,s}^2(t)\\
    &\times\left|(k,kt)^{m+\beta}\int_0^t\int_{\R^3} S_k(t-\tau)\big[ \widehat{\Gamma(f,f)}_{k}(\tau,v)\big]\sqrt{\mu}dvd\tau\right|^2dt\Bigg]^{\fr12}.
\end{align*}
Like the case $\iota=0$, the above integral can be divided into four parts, and we adopt the notation used there.  If $\tau\le\fr23t$, we still use \eqref{eq:gain m3} to absorb $w_\iota(t)$, here we omit the estimates and only state the results. In fact, similar to \eqref{est-I1k-short} and \eqref{est-I4k-short}, one deduces that 
\begin{align}\label{est-I1k-iota-short}
    \nn&\int_0^{T}\sum_{k\in\Z^3_*} \la t\ra^{1+a-s}w_{\iota}^2(t)\sum_{|\beta|\le N}\kappa^{2|\beta|}\sum_{m\in\N^6}a_{m,s}^2(t)\Big|\mathbb{I}_{1,k}^{m,\beta}(t)_{\tau\le\fr{2t}{3}}\Big|^2dt\\
    \les&\nu^{\fr23}{  \left\|\sqrt{\mu}f\right\|^2_{L^\infty_t\mathcal{G}^{\lm,N}_{s,2}}\left\|\nu^{\fr23}\pr_{v_iv_j}f\right\|^2_{L^\infty_t\mathcal{G}^{\lm,N}_{s,0}}}\left\|\mathfrak{g}\right\|^2_{L^\infty_t\bar{\mathcal{G}}^{\lm(0),N}_{s,0}}.
\end{align}
and
\begin{align}\label{est-I4k-iota-short}
    \nn&\int_0^{T}\sum_{k\in\Z^3_*} \la t\ra^{1+a-s}w_{\iota}^2(t)\sum_{|\beta|\le N}\kappa^{2|\beta|}\sum_{m\in\N^6}a_{m,s}^2(t)\Big|\mathbb{I}_{4,k}^{m,\beta}(t)_{\tau\le\fr{2t}{3}}\Big|^2dt\\
    \les&\nu^{2}{  \left\|\sqrt{\mu}f\right\|^2_{L^\infty_t\mathcal{G}^{\lm,N}_{s,0}}\left\|f\right\|^2_{L^\infty_t\mathcal{G}^{\lm,N}_{s,0}}}\left\|\mathfrak{g}\right\|^2_{L^\infty_t\bar{\mathcal{G}}^{\lm(0),N}_{s,0}}.
\end{align}

We  are left to deal with the case when $\fr23t\le\tau\le t$ below, and we use the notation in \eqref{divid-I1k}.

{\it $\bullet$ Treatments of $\mathbb{I}_{1,k;1)}^{m,\beta}(t)$.} Using \eqref{decom-f} again, we further split   $\mathbb{I}_{1,k;1)}^{m,\beta}(t)$ into two parts:
\begin{align*}
  \mathbb{I}_{1,k;1)}^{m,\beta}(t)
=\nn&8\pi\nu \sum_{l\in\Z^3_*}\sum_{\substack{\beta'\leq \beta\\|\beta'|\leq |\beta|/2}}C_{\beta}^{\beta'}\sum_{\beta''\le\beta-\beta'}C_{\beta-\beta'}^{\beta''}
    \sum_{n\leq m}C_{m}^{n}\sum_{n'\le m-n}C_{m-n}^{n'}\\
    \nn&\times\int_{\fr{2t}{3}}^t\int_{\R^6} \frac{\xi_i\xi_j}{|\xi|^4}\Big(\mathcal{F}_{x,v}\big[Z^{n+\beta'}(\mu^\fr12q(\phi)g)\big]_l(\tau,\xi)+\mathcal{F}_{x,v}\big[Z^{n+\beta'}(\mu^{\fr12}g)\big]_l(\tau,\xi)\Big)\\
    &\times \mathcal{F}_{x,v}\big[Z^{n'+\beta''}\pr_{v_iv_j}f\big]_{k-l}(\tau,\eta-\xi)\\
    \nn&
    \times (0,-\eta+k(t-\tau))^{m-n-n'+\beta-\beta'-\beta''}\bar{\hat{\mathfrak{g}}}_{k}(t-\tau,\eta)d\xi d\eta d\tau\\
    =&\mathbb{I}_{1,k;1),e}^{m,\beta}(t)+\mathbb{I}_{1,k;1),d}^{m,\beta}(t).
\end{align*}
For $\mathbb{I}_{1,k;1),e}^{m,\beta}(t)$, we do not even need to use \eqref{xi-gain-1}. In fact, similar to \eqref{est-I12e-1} and \eqref{est-I12e-2}, we have
\begin{align}\label{est-I11e-iota}
\nn&\int_0^{T}\la t\ra^{1+a-s}w_{\iota}^2(t)
\sum_{k\in\Z^3_*} \sum_{|\beta|\le N}\kappa^{2|\beta|}\sum_{m\in\N^6}a_{m,\lm,s}^2(t)\left|\mathbb{I}_{1,k;1),e}^{m,\beta}(t)\right|^2dt\\
    \les
    \nn&\nu^{\fr13} \int_0^{T}\int_{\tau}^{T}\frac{1}{\la \nu^{\fr13} (t-\tau)\ra^{\fr32}} \Big(\la \tau\ra^{1+a-s}w_{\iota}^2(\tau)\left\|\la Y\ra^2|\nb_x|^2(\mu^\fr12q(\phi)g)(\tau)\right\|_{\mathcal{G}^{\lm(\tau),\fr{N}{2}}_{s,2}}^2\Big)\\
    \nn&\times \left\|\nu^{\fr23}\pr_{v_iv_j}f(\tau)\right\|_{\mathcal{G}^{\lm(\tau),N}_{s,0}}^2 \Big(\la\nu^{\fr13}(t-\tau)\ra^{3}\left\| \mathfrak{g}(t-\tau)\right\|_{\bar{\mathcal{G}}^{\lm(0),N}_{s,0}}^2\Big) dt d\tau\\
    \les\nn&{ \int_0^{T} \la \tau\ra^{1+a-s}w_{\iota}^2(\tau)\left\|\la Y\ra^2|\nb_x|^2(\mu^\fr12q(\phi)g)(\tau)\right\|_{\mathcal{G}^{\lm(\tau),\fr{N}{2}}_{s,2}}^2d\tau} \left\|\nu^{\fr23}\pr_{v_iv_j}f\right\|_{L^\infty_t\mathcal{G}^{\lm,N}_{s,0}}^2\\
    &\times\sup_t\Big(\la\nu^{\fr13}t\ra^{3}\left\| \mathfrak{g}(t)\right\|_{\bar{\mathcal{G}}^{\lm(0),N}_{s,0}}^2\Big).
\end{align}
As for $\mathbb{I}_{1,k;1),d}^{m,\beta}(t)$, thanks to \eqref{xi-gain-1}, similar to \eqref{est-I11} with $\mu^{\fr12}f$ replaced by $\mu^{\fr12}g$, we arrive at
\begin{align}\label{est-I11-iota-d}
    \nn&\int_0^{T}\la t\ra^{1+a-s}w_{\iota}^2(t)\sum_{k\in\Z^3_*}\sum_{|\beta|\le N}\kappa^{2|\beta|}\sum_{m\in\N^6}a_{m,\lm,s}^2(t)\left|\mathbb{I}_{1,k;1),d}^{m,\beta}(t)\right|^2dt\\
    \nn\les&\nu^{\fr53}\int_0^{T}\fr{1}{\la t\ra^{3-a+s}}\int_{\fr{2t}{3}}^t \la\nu^\fr13(t-\tau)\ra^{-\fr32} w_{\iota}^2(\tau)\left\| \la Y\ra^2|Z|^4(\mu^\fr12g)_{\ne}(\tau)\right\|_{\mathcal{G}^{\lm,\fr{N}{2}}_{s,2}}^2\\
    \nn&\times\left\|\pr_{v_iv_j}f(\tau)\right\|^2_{\mathcal{G}^{\lm,N}_{s,0}} \Big(\la\nu^\fr13(t-\tau)\ra^{3}\left\|\mathfrak{g}(t-\tau)\right\|^2_{\bar{\mathcal{G}}^{\lm,N}_{s,0}}\Big) d\tau dt\\
    \les& { \left\|w_{\iota} \la Z\ra^6(\mu^\fr12g_{\ne})\right\|_{L^\infty_t\mathcal{G}^{\lm,\fr{N}{2}}_{s,2}}^2 \left\|\nu^{\fr23}\pr_{v_iv_j}f\right\|^2_{L^\infty_t\mathcal{G}^{\lm,N}_{s,0}}}\sup_{t}\Big(\la\nu^\fr13t\ra^{3}\left\|\mathfrak{g}(t)\right\|^2_{\bar{\mathcal{G}}^{\lm,N}_{s,0}}\Big).
\end{align}

{$\bullet$ \it Treatments of $\mathbb{I}_{1,k;3)}^{m,\beta}(t)$.} If $l\ne k$, we split $\mathbb{I}_{1,k;3)}^{m,\beta}(t)$ as follows:
\begin{align*}
  \mathbb{I}_{1,k;3)}^{m,\beta}(t)
=\nn&8\pi\nu \sum_{l\in\Z^3}\sum_{\substack{\beta'\leq \beta\\|\beta'|>|\beta|/2}}C_{\beta}^{\beta'}\sum_{\beta''\le\beta-\beta'}C_{\beta-\beta'}^{\beta''}
    \sum_{n\leq m}C_{m}^{n}\sum_{n'\le m-n}C_{m-n}^{n'}\\
    \nn&\times\int_{\fr{2t}{3}}^t\int_{\R^6} \frac{\xi_i\xi_j}{|\xi|^4}\mathcal{F}_{x,v}\big[Z^{n+\beta'}(\mu^\fr12f)\big]_l(\tau,\xi) \\
    &\times\Big(\mathcal{F}_{x,v}\big[Z^{n'+\beta''}q(\phi)\pr_{v_iv_j}g\big]_{k-l}(\tau,\eta-\xi)+\mathcal{F}_{x,v}\big[Z^{n'+\beta''}\pr_{v_iv_j}g\big]_{k-l}(\tau,\eta-\xi)\Big)\\
    \nn&\quad \quad 
    \times (0,-\eta+k(t-\tau))^{m-n-n'+\beta-\beta'-\beta''}\bar{\hat{\mathfrak{g}}}_{k}(t-\tau,\eta)d\xi d\eta d\tau\\
    &=\mathbb{I}_{1,k;3),e}^{m,\beta}(t)+\mathbb{I}_{1,k;3),d}^{m,\beta}(t),
\end{align*}
while if $l=k$, we divide $\mathbb{I}_{1,k;3)}^{m,\beta}(t)$  in the following way
\begin{align*}
  \mathbb{I}_{1,k;3)}^{m,\beta}(t)
=\nn&8\pi\nu \sum_{\substack{\beta'\leq \beta\\|\beta'|>|\beta|/2}}C_{\beta}^{\beta'}\sum_{\beta''\le\beta-\beta'}C_{\beta-\beta'}^{\beta''}
    \sum_{n\leq m}C_{m}^{n}\sum_{n'\le m-n}C_{m-n}^{n'}\\
    \nn&\times\int_{\fr{2t}{3}}^t\int_{\R^6} \frac{\xi_i\xi_j}{|\xi|^4}\Big(\mathcal{F}_{x,v}\big[Z^{n+\beta'}(\mu^\fr12q(\phi)g)\big]_k(\tau,\xi)+\mathcal{F}_{x,v}\big[Z^{n+\beta'}(\mu^\fr12g)\big]_k(\tau,\xi)\Big)\\ &\times\mathcal{F}_{x,v}\big[Z^{n'+\beta''}\pr_{v_iv_j}f\big]_{0}(\tau,\eta-\xi)\\
    \nn& 
    \times (0,-\eta+k(t-\tau))^{m-n-n'+\beta-\beta'-\beta''}\bar{\hat{\mathfrak{g}}}_{k}(t-\tau,\eta)d\xi d\eta d\tau\\
    &=\mathbb{I}_{1,k;3),e}^{m,\beta}(t)+\mathbb{I}_{1,k;3),d}^{m,\beta}(t),
\end{align*}
We first treat $\mathbb{I}_{1,k;3),e}^{m,\beta}(t)$ for $l\ne k$. Indeed, similar to \eqref{est-I11e-iota}, one deduces that
\begin{align}\label{est-I13e-iota}
\nn&\int_0^{T}\la t\ra^{1+a-s}w_{\iota}^2(t)
\sum_{k\in\Z^3_*} \sum_{|\beta|\le N}\kappa^{2|\beta|}\sum_{m\in\N^6}a_{m,\lm,s}^2(t)\left|\mathbb{I}_{1,k;3),e}^{m,\beta}(t)\right|^2dt\\
    \les
    \nn&\nu^{\fr13} \int_0^{T}\int_{\tau}^{T}\frac{1}{\la \nu^{\fr13} (t-\tau)\ra^{\fr32}} \Big(\la \tau\ra^{1+a-s}w_{\iota}^2(\tau)\left\|\la Y\ra^2|\nb_x|^2(q(\phi)\nu^{\fr23}\pr_{v_iv_j}g)(\tau)\right\|_{\mathcal{G}^{\lm(\tau),\fr{N}{2}}_{s,0}}^2\Big)\\
    \nn&\times \left\|\mu^\fr12f(\tau)\right\|_{\mathcal{G}^{\lm(\tau),N}_{s,2}}^2\Big( \la\nu^{\fr13}(t-\tau)\ra^{3}\left\| \mathfrak{g}(t-\tau)\right\|_{\bar{\mathcal{G}}^{\lm(0),N}_{s,0}}^2\Big) dt d\tau\\
    \les\nn&{ \int_0^{T} \la \tau\ra^{1+a-s}w_{\iota}^2(\tau)\left\|\la Y\ra^2|\nb_x|^2(q(\phi)\nu^{\fr23}\pr_{v_iv_j}g)(\tau)\right\|_{\mathcal{G}^{\lm(\tau),\fr{N}{2}}_{s,0}}^2d\tau} \left\|\mu^\fr12f\right\|_{L^\infty_t\mathcal{G}^{\lm,N}_{s,2}}^2\\
    &\times\sup_t\Big(\la\nu^{\fr13}t\ra^{3}\left\| \mathfrak{g}(t)\right\|_{\bar{\mathcal{G}}^{\lm(0),N}_{s,0}}^2\Big).
\end{align}
If $l=k$, then $\mathbb{I}_{1,k;3),e}^{m,\beta}(t)$  can be treated in a similar way:
\begin{align}\label{est-I13e-iota-k=l}
\nn&\int_0^{T}\la t\ra^{1+a-s}w_{\iota}^2(t)
\sum_{k\in\Z^3_*} \sum_{|\beta|\le N}\kappa^{2|\beta|}\sum_{m\in\N^6}a_{m,\lm,s}^2(t)\left|\mathbb{I}_{1,k;3),e}^{m,\beta}(t)\right|^2dt\\
    \les
    \nn&\nu^{\fr13} \int_0^{T}\int_{\tau}^{T}\frac{1}{\la \nu^{\fr13} (t-\tau)\ra^{\fr32}} \Big(\la \tau\ra^{1+a-s}w_{\iota}^2(\tau)\left\|(\mu^\fr12q(\phi)g)(\tau)\right\|_{\mathcal{G}^{\lm(\tau),N}_{s,2}}^2\Big)\\
    \nn&\times \left\|\la Y\ra^2\big(\nu^{\fr23}\pr_{v_iv_j}f\big)(\tau)\right\|_{\mathcal{G}^{\lm(\tau),\fr{N}{2}}_{s,0}}^2 \Big(\la\nu^{\fr13}(t-\tau)\ra^{3}\left\| \mathfrak{g}(t-\tau)\right\|_{\bar{\mathcal{G}}^{\lm(0),N}_{s,0}}^2\Big) dt d\tau\\
    \les\nn&{ \int_0^{T} \la \tau\ra^{1+a-s}w_{\iota}^2(\tau)\left\|(\mu^\fr12q(\phi)g)(\tau)\right\|_{\mathcal{G}^{\lm(\tau),N}_{s,2}}^2d\tau} \left\|\la Y\ra^2\big(\nu^{\fr23}\pr_{v_iv_j}f\big)\right\|_{L^\infty_t\mathcal{G}^{\lm(\tau),\fr{N}{2}}_{s,0}}^2\\
    &\times\sup_t\Big(\la\nu^{\fr13}t\ra^{3}\left\| \mathfrak{g}(t)\right\|_{\bar{\mathcal{G}}^{\lm(0),N}_{s,0}}^2\Big).
\end{align}
We are left to treat $\mathbb{I}_{1,k;3),d}^{m,\beta}(t)$ whether $l=k$ or not.

If $l=k$, by \eqref{est-Ik3-2} with $\mu^{\fr12}f$
replaced by $\mu^{\fr12}g$, similar to \eqref{est-Ik3-3}, we find that
\begin{align}\label{est-I13d-k=l}
\nn&\int_0^{T}\langle t\rangle^{1+a-s}w^2_{\iota}(t)\sum_{k\in\Z^3_*} \sum_{|\beta|\le N}\kappa^{2|\beta|}\sum_{m\in\N^6}a_{m,\lm,s}^2(t)\left|\mathbb{I}_{1,k;3),d}^{m,\beta}(t)\right|^2dt\\
    \les\nn&\nu\int_0^{T}\int_{\tau}^{T} \frac{1}{\la t-\tau\ra^2}\left\| \la Y^*\ra^2\mathfrak{g}(t-\tau)\right\|_{\bar{\mathcal{G}}^{\lm(0),\fr{N}{2}}_{s,0}}^2 dt\\
    \nn&\times\left(\nu\la\tau\ra^{1+a-s}w_{\iota}^2(\tau)\left\|(\mu^\fr12g)_{\ne}(\tau)\right\|_{\mathcal{G}^{\lm(\tau),N}_{s,2}}^2\right)\left\|\la \nb_v\ra^4\pr_{v_iv_j}f_0(\tau)\right\|_{\mathcal{G}^{\lm(\tau),\fr{N}{2}}_{s,0}}^2 d\tau\\
    \les\nn&\left\| \la Y^*\ra^2\mathfrak{g}\right\|_{L^\infty_t\bar{\mathcal{G}}^{\lm(0),\fr{N}{2}}_{s,0}}^2\left\|\la \nb_v\ra^4\pr_{v_iv_j}f_0\right\|_{L^\infty_t\mathcal{G}^{\lm,\fr{N}{2}}_{s,0}}^2\\
     &\times\Bigg[\nu^{\fr13}\int_0^{T} 
    {\nu^{\fr53}\la\tau\ra^{1+a-s}}w_{\iota}^2(\tau)\big\|\mu^\fr12g_{\ne}(\tau)\big\|_{\mathcal{G}^{\lm(\tau),N}_{s,2}}^2 d\tau\Bigg].
\end{align}

If $l\ne k$, there are four potential cases, but only three are non-trivial:\par
\noindent{\it Case 1: $l\ne k, |kt-l\tau|\ge \fr{t}{2}$.}\\
{\it Case 2: $l\ne k, |kt-l\tau|< \fr{t}{2},l=0$.} (This case is impossible and ruled out.)\\
{\it Case 3: $l\ne k, |kt-l\tau|< \fr{t}{2},l\ne0, |l\tau|\le 2|\xi|$.}\\
{\it Case 4: $l\ne k, |kt-l\tau|< \fr{t}{2},l\ne0, |l\tau|> 2|\xi|$.}\par
For {\it Case 1} and {\it Case 3}, like \eqref{xi-gain}, we always have
\begin{align*}
    \fr{1}{|\xi|^2}\les \fr{\la\xi\ra^2}{|\xi|^2} \fr{1}{\la t\ra^2}\la \eta-\xi+(k-l)\tau \ra^2\la -\eta+k(t-\tau) \ra^2.
\end{align*}
Then similar to \eqref{est-Ik3-1}, we arrive at
\begin{align}\label{est-Ik3d-iota-1}
    \nn&\int_0^{T}\la t\ra^{1+a-s}w_{\iota}^2(t)\sum_{k\in\Z^3_*} \sum_{|\beta|\le N}\kappa^{2|\beta|}\sum_{m\in\N^6}a_{m,\lm,s}^2(t)\left|\mathbb{I}_{1,k;3),d}^{m,\beta}(t)\right|^2dt\\
    \les&\left\| (\mu^\fr12f)\right\|_{L^\infty_t\mathcal{G}^{\lm,N}_{s,2}}^2{   \left\|w_{\iota}|\nb_x|^2\la Y\ra^4\nu^{\fr23}\pr_{v_iv_j}g_{\ne}\right\|^2_{L^\infty_t\mathcal{G}^{\lm,\fr{N}{2}}_{s,0}}}\sup_{t}\Big(\la\nu^\fr13t\ra^{3}\left\|\la Y^*\ra^2\mathfrak{g}(t)\right\|^2_{\bar{\mathcal{G}}^{\lm,\fr{N}{2}}_{s,0}}\Big).
\end{align}

For {\it Case 4}, similar to \eqref{est-Ik3-4}, we get
\begin{align}\label{est-Ik3d-iota-2}
\nn&\int_0^{T}\la t\ra^{1+a-s}w_{\iota}^2(t)\sum_{k\in\Z^3_*} \sum_{|\beta|\le N}\kappa^{2|\beta|}\sum_{m\in\N^6}a_{m,\lm,s}^2(t)\left|\mathbb{I}_{1,k;3),d}^{m,\beta}(t)\right|^2dt\\
    \les\nn&\nu^2 \int_0^{T}\la t\ra^{-1+a-s}\sum_{k\in\Z^3_*}  \sum_{|\beta|\le N}\kappa^{2|\beta|}\sum_{m\in\N^6}\Bigg[\sum_{\substack{l\ne k,\\ l\ne0}}\sum_{\substack{\beta'\leq \beta,\beta''\le\beta-\beta'\\|\beta'|>|\beta|/2}}
    \sum_{\substack{n\leq m\\n'\le m-n}}\\
    \nn&\int_{\fr{2t}{3}}^t \frac{b_{m,n,n',s}{ |l|}}{\la kt-l\tau\ra^2|k-l|^2} \left\|a_{n,s}(\tau)\Big(\fr{\lm(t)}{\lm(\tau)}\Big)^{|n|}\fr{\tau}{{ |l|^2}} {\bf 1}_{D_2}\mathcal{F}_{x,v}\big[(\mu^\fr12{ \nb_x}f)^{(n+\beta')}\big]_l(\tau)\right\|_{L^2_\eta\cap L^\infty_\eta}\\ 
    \nn&\times w_{\iota}(\tau)\left\|a_{n',\lm,s}(\tau)\mathcal{F}_{x,v}\big[\la Y\ra^2|\nb_x|^2\pr_{v_iv_j}g^{(n'+\beta'')}\big]_{k-l}(\tau)\right\|_{L^2_\eta\cap L^1_\eta }\\
    \nn& \times\left\|a_{m-n-n',\lm,s}(0) \la Y^*\ra^2\mathfrak{g}_k^{(m-n-n'+\beta-\beta'-\beta'')}(t-\tau)\right\|_{L^2_v} d\tau\Bigg]^2dt\\
    \les\nn&\nu^{\fr13}\int_0^{T}{ \nu^{\fr13}\la \tau\ra^{-1+a-s}} {w_{\iota}^2(\tau)} \left\|\la Y\ra^4| \nb_x|^3\nu^{\fr23}\pr_{v_iv_j}g_{\ne}(\tau)\right\|^2_{\mathcal{G}^{\lm,\fr{N}{2}}_{s,0}}d\tau \\
    \nn&\times\left\|\la Y^*\ra^2\mathfrak{g}\right\|^2_{L^\infty_t\bar{\mathcal{G}}^{\lm(0),\fr{N}{2}}_{s,0}}\big\|\mu^{\fr12}\nb_xf\big\|^2_{L^\infty_t\mathcal{G}^{\lm,N}_{s,2}}\\
    \les&\nu^\fr23\big\|w_\iota \nu^\fr23\pr_{v_i}\pr_{v_j}g_{\ne}\big\|^2_{L^\infty_t\mathcal{G}^{\lm,N}_{s,0}}\left\|\la Y^*\ra^2\mathfrak{g}\right\|^2_{L^\infty_t\bar{\mathcal{G}}^{\lm(0),\fr{N}{2}}_{s,0}}\big\|\mu^{\fr12}\nb_xf\big\|^2_{L^\infty_t\mathcal{G}^{\lm,N}_{s,2}},
\end{align}
provided $N\ge14$.

{$\bullet$  \it Treatments of $\mathbb{I}_{1,k;2)}^{m,\beta}(t)$.} Here we use the decomposition of $\mathbb{I}_{1,k;2)}^{m,\beta}(t)$ in \eqref{decom-I12}. For $\mathbb{I}_{1,k;2),e}^{m,\beta}(t)$, by \eqref{est-I12e-1}, similar to \eqref{est-I12e-2}, we arrive at 
\begin{align}\label{est-I12e-iota}
\nn&\int_0^{T}\langle t\rangle^{1+a-s}w_{\iota}^2(t)\sum_{k\in\Z^3_*} \sum_{|\beta|\le N}\kappa^{2|\beta|}\sum_{m\in\N^6}a_{m,\lm,s}^2(t)\left|\mathbb{I}_{1,k;2),e}^{m,\beta}(t)\right|^2dt\\
    \les\nn&\left\|\la\nb_v\ra^2(\mu^\fr12f_0)\right\|_{L^\infty_t\mathcal{G}^{\lm,\fr{N}{2}}_{s,2}}^2\sup_t\Big(\la\nu^{\fr13}t\ra^3\left\| \mathfrak{g}(t)\right\|_{\bar{\mathcal{G}}^{\lm(0),N}_{s,0}}^2\Big)\\
    &\times{  \int_0^{T} 
    \la \tau\ra^{1+a-s}w^2_{\iota}(\tau)\left\|q(\phi)\nu^{\fr23}\pr_{v_iv_j}g(\tau)\right\|_{\mathcal{G}^{\lm(\tau),N}_{s,0}}^2d\tau}.
\end{align}
Next we investigate $\mathbb{I}_{1,k;2),d}^{m,\beta}(t)$. If $|\eta|\le 10|\eta-k(t-\tau)|$, by \eqref{est-I12d-1}, similar to \eqref{est-I12d-2}, we find that
\begin{align}\label{est-I12d-iota}
\nn&\nu^2\int_0^{T}\langle t\rangle^{1+a-s}w_{\iota}^2(t)\sum_{k\in\Z^3_*} \sum_{|\beta|\le N}\kappa^{2|\beta|}\sum_{m\in\N^6}a_{m,\lm,s}^2(t)\left|\mathbb{I}_{1,k;2),d}^{m,\beta}(t)\right|^2dt\\
    \les\nn&\left\|\la\nb_v\ra^3(\mu^\fr12f_0)\right\|_{L^\infty_t\mathcal{G}^{\lm,\fr{N}{2}}_{s,2}}^2\left\| \la Y^*\ra^3\mathfrak{g}\right\|_{L^\infty_t\bar{\mathcal{G}}^{\lm(0),N}_{s,0}}^2\\
    &\times\nu^{\fr13}\int_0^{T} 
    {  \nu\langle \tau\rangle^{1+a-s}}w_{\iota}^2(\tau)\left\|\nu^{\fr13}\pr_{v_j}g_{\ne}(\tau)\right\|_{\mathcal{G}^{\lm(\tau),N}_{s,0}}^2 d\tau.
\end{align}

If $|\eta|\le1$, similar to \eqref{est-I12d-eta<1-2}, we have
\begin{align}\label{est-I12d-eta<1-iota}
\nn&\int_0^{T}\la t\ra^{1+a-s}w_{\iota}^2(t)\sum_{k\in\Z^3_*} \sum_{|\beta|\le N}\kappa^{2|\beta|}\sum_{m\in\N^6}a_{m,\lm,s}^2(t)\left|\mathbb{I}_{1,k;2),d}^{m,\beta}(t)\right|^2\\
    \les\nn&\nu^{\fr13}\int_0^T{  \nu\la\tau\ra^{1+a-s}}w_{\iota}^2(\tau)\left\|g_{\ne}(\tau)\right\|_{\mathcal{G}^{\lm(\tau),N}_{s,0}}^2d\tau \left\|\la\nb_v\ra^4(\mu^\fr12f_0)(\tau)\right\|_{L^\infty_t\mathcal{G}^{\lm,\fr{N}{2}}_{s,2}}^2\\
    &\times\sup_t\Big(\la\nu^{\fr13}t\ra^{3}\left\| \mathfrak{g}(t)\right\|_{\bar{\mathcal{G}}^{\lm(0),N}_{s,0}}^2\Big).
\end{align}

If $|\eta|> \max\{10|\eta-k(t-\tau)|,1\}$, all the estimates in \eqref{appro1}--\eqref{lm-gap6} hold. Similar to \eqref{ktau-power} and \eqref{restriction-s>1/2}, for $s>\fr12$ and $a, \varepsilon$ chosen the same as those in \eqref{restriction-s>1/2}, we have
\begin{align}\label{ktau-power-iota}
    \la\tau\ra^{\fr{1+a-s}{2}}\fr{\la\tau\ra^{\fr{1+a}{2}}}{|k\tau|^{\fr{s}{2}}}\Big(|k|^{(1-s)(3+\varepsilon)}\la\tau\ra^{(1+a-s)(3+\varepsilon)}\Big)
    \les|k|^{2}\la\tau\ra^{2}.
\end{align}
Accordingly, the analogue of \eqref{est-I12d-main} is given by
\begin{align}\label{est-I12d-main-iota}
\nn&\int_0^{T}\sum_{k\in\Z^3_*}\la t\ra^{1+a-s} w_{\iota}^2(t)\sum_{|\beta|\le N}\kappa^{2|\beta|}\sum_{m\in\N^6}a_{m,\lm,s}^2(t)\left|\mathbb{I}_{1,k;2),d}^{m,\beta}(t)\right|^2dt\\
    \les\nn& \int_{0}^{T}   { (\nu\la\tau\ra^2)^2}{w_{\iota}^2(\tau)}\left\|\fr{1}{\la\tau\ra^{\fr{1+a}{2}}}\la Y\ra^{\fr{s}{2}}\nb_xg(\tau)\right\|^2_{\mathcal{G}^{\lm(\tau),N}_{s,0}}d\tau\\
    &\times \left\|\la\nb_v \ra^{\fr{s}{2}+7+\varepsilon}(\mu^\fr12f_0)\right\|^2_{L^\infty_t\mathcal{G}^{\lm(\tau),\fr{N}{2}}_{s,2}}\|\mathfrak{g}\|^2_{L^\infty_t\bar{\mathcal{G}}^{\lm(0),N}_{s,0}}.
\end{align}

{$\bullet$  \it Treatments of $\mathbb{I}_{4,k}^{m,\beta}(t)_{\tau\approx t}$.} Similar to \eqref{est-I4k-tau=t}, using \eqref{decom-f}, \eqref{fg1} and \eqref{Linfty-rho}, Lemmas \ref{lem-compose} and \ref{lem:est-S^*}, we are led to
\begin{align}\label{est-I4k-tau=t-iota}
    \nn&\int_0^{T}\la t\ra^{1+a-s}w^2_{\iota}(t)\sum_{k\in\Z^3_*} \sum_{|\beta|\le N}\kappa^{2|\beta|}\sum_{m\in\N^6}a_{m,\lm,s}^2(t)\left|\mathbb{I}_{4,k}^{m,\beta}(t)_{\tau\approx t}\right|^2dt\\
    \les\nn&\nu^{\fr43}\int_0^{T}\la\tau\ra^{1+a-s}w_{\iota}^2(\tau)\|f_{\ne}(\tau)\|_{\mathcal{G}^{\lm,N}_{s,0}}^2\|f(\tau)\|^2_{\mathcal{G}^{\lm,N}_{s,0}}d\tau\\
    \les&\nu^{\fr43}\int_0^{T}\la\tau\ra^{1+a-s}w_{\iota}^2(\tau)\Big(\|\rho\|_{\mathcal{G}^{\lm,N}_s}^2\|g\|^2_{\mathcal{G}^{\lm,N}_{s,0}}+\|g_{\ne}\|^2_{\mathcal{G}^{\lm,N}_{s,0}}\Big)d\tau\|g\|^2_{L^\infty_t\mathcal{G}^{\lm,N}_{s,0}}.
\end{align}

Now collecting the estimates in \eqref{est-T2E2-iota-short}--\eqref{est-I12d-eta<1-iota}, \eqref{est-I12d-main-iota} and \eqref{est-I4k-tau=t-iota},  and using \eqref{decom-f}, \eqref{fg1}, \eqref{fg2} and \eqref{Linfty-rho}, Lemmas \ref{lem-compose}, \ref{lem:est-S^*}, and Corollary \ref{coro-CK} again, we obtain \eqref{est-rho2-iota}.
\end{proof}

\section{First moment estimates in Gevrey class}
We begin this section  by deriving the equation for the first moment 
$M(t,x)$. Multiplying \eqref{f_k} by $v\sqrt{\mu}(v)$, and integrating the resulting equation with respect to $v$ over $\R^3$ yields
\be\label{eq-M_k}
\hat{M}_k(t)=-\int_0^tH_k(t-\tau)\hat{\rho}_k(\tau)d\tau+\mathsf{N}_k(t),
\ee
where
\begin{align}\label{def-H}
    H_k(t)=\fr{2ki}{|k|^2}\int_{\R^3} S_k(t)[v\mu^{\fr12}]v\mu^{\fr12}(v)dv.
\end{align}
and
\be\label{Nsf_k}
\mathsf{N}_k(t)=\int_{\R^3}S_k(t)[\hat{f}_{\rm in}(v)]v\mu^{\fr12}(v)dv+\int_0^t\int_{\R^3}S_k(t-\tau)[\hat{\mathfrak{N}}_k(\tau,v)]v\mu^{\fr12}(v)dvd\tau.
\ee
Like the decomposition of the density $\rho$, we decompose the first moment $M(t,x)$ into two parts as follows:
\begin{align}\label{e-M1}
    \hat{M}^{(1)}_k(t)=-\int_0^tH_k(t-\tau)\hat{\rho}^{(1)}_k(\tau)d\tau+\mathsf{N}^{(1)}_k(t),
\end{align}
and
\begin{align}\label{e-M2}
    \hat{M}^{(2)}_k(t)=-\int_0^tH_k(t-\tau)\hat{\rho}^{(2)}_k(\tau)d\tau+\mathsf{N}^{(2)}_k(t),
\end{align}
where
\begin{align*}
    \mathsf{N}_k^{(1)}(t)=&\int_{\R^3}S_k(t)[\hat{f}_{\rm in}(v)]v\mu^{\fr12}(v)dv+\int_0^t\int_{\R^3}S_k(t-\tau)[\hat{\mathfrak{N}}^{(1)}_k(\tau,v)]v\mu^{\fr12}(v)dvd\tau,\\
    \mathsf{N}_k^{(2)}(t)=&\int_0^t\int_{\R^3}S_k(t-\tau)[\hat{\mathfrak{N}}^{(2)}_k(\tau,v)]v\mu^{\fr12}(v)dvd\tau,
\end{align*}
with $\mathfrak{N}^{(1)}(\tau,x,v)$ and $\mathfrak{N}^{(2)}(\tau,x,v)$ defined in \eqref{eq: N_1, N_2}.

\begin{lem}\label{lem-point-H_k}
There exist   constants $0<c'<\fr12$ and $C>0$ depending only on $s_L$, and $\lambda_L$, such that  
\begin{align}\label{point-H_k}
     |H_k(t)|\le C |k|^{-1}\min\left\{e^{-\fr{c'}{2}(\nu^{\fr13}t)^{\fr13}}, e^{-\fr{c'}{2}(\nu t)^{\fr23}}\right\}e^{-\fr{c'}{2}|kt|^{s_L}}. 
\end{align}
\end{lem}
The proof of \eqref{point-H_k} is similar to the proof of Lemma \ref{lem-point-K_k}. We omit the details for simplicity.

\begin{lem}
The following estimates for $M^{(1)}$ and $M^{(2)}$ hold:
    \begin{align}\label{es-rho:M1}
    \left\|\la t\ra^b|\nb_x|^\fr32M^{(1)}\right\|_{L^2_t\mathcal{G}^{\lm,N-1}_s}\les \left\|\la t\ra^b|\nb_x|^\fr32\rho^{(1)}\right\|_{L^2_t\mathcal{G}^{\lm,N-1}_s}+ \left\|\la t\ra^b|\nb_x|^\fr32\mathsf{N}^{(1)}\right\|_{L^2_t\mathcal{G}^{\lm,N-1}_s},
\end{align}

\begin{align}\label{es-rho:M2}
    \left\|\langle t\rangle^{\fr{3+a-s}{2}}M^{(2)}\right\|_{L^2_t\mathcal{G}^{\lm,N}_s}
    \les \left\|\langle t\rangle^{\fr{3+a-s}{2}}\rho^{(2)}\right\|_{L^2_t\mathcal{G}^{\lm,N}_s}+\left\|
    \langle t\rangle^{\fr{3+a-s}{2}}\mathsf{N}^{(2)}\right\|_{L^2_t\mathcal{G}^{\lm,N}_s},
\end{align}
and for $\iota=1,2,\cdots,  [\ell/3]+1$,
\begin{align}\label{es-rho:M2-iota}
    \left\|\langle t\rangle^{\fr{1+a-s}{2}}w_{\iota}(t)M^{(2)}\right\|_{L^2_t\mathcal{G}^{\lm,N}_s}
    \les \left\|
    \langle t\rangle^{\fr{1+a-s}{2}} w_{\iota}(t)\rho^{(2)}\right\|_{L^2_t\mathcal{G}^{\lm,N}_s}+\left\|
    \langle t\rangle^{\fr{1+a-s}{2}} w_{\iota}(t)\mathsf{N}^{(2)}\right\|_{L^2_t\mathcal{G}^{\lm,N}_s}.
\end{align}
\end{lem}

We will use the duality argument as well to bound $\mathsf{N}^{(1)}_k(t)$ and $\mathsf{N}^{(2)}_k(t)$. To this end, let us define
\begin{align}\label{def-frak-h}
\mathfrak{h}_k(t,v)=\begin{cases}S_k^{*}(t)[v\sqrt{\mu}(v)],\ \ {\rm if}\ \ k\in\Z^3_*;\\
\qquad\qquad\quad\ \, 0,  \ \ {\rm if}\ \ k=0,
\end{cases}
\quad{\rm and}\quad \mathfrak{h}(t,x,v)=\sum_{k\in\Z^3}\mathfrak{h}_k(t,v)e^{ik\cdot x}.
\end{align}
Clearly, all the estimates established for $\mathfrak{g}$ also hold for $\mathfrak{h}$. Accordingly, collecting the estimates obtained in Propositions \ref{prop-rho1}--\ref{prop-rho2-iota},  recalling the definitions of $\mathcal{E}_\ell(g(t))$, $\mathcal{D}_\ell(g(t))$ and $\mathcal{CK}_\ell(g(t))$ in \eqref{energy-l}--\eqref{CK-l}, and using \eqref{coercive}, we establish the following proposition for $M^{(1)}$ and $M^{(2)}$.
\begin{prop}\label{prop-M} Let $\ell>18.$
    Under the bootstrap hypotheses \eqref{bd-en}--\eqref{H-phi}, if \eqref{res-N-b} holds, and 
\begin{align}
    s>\fr12, \quad 0<a<s-\fr12,\quad {\rm and }\quad a\le \min\Big\{\fr{s}{b+4}, \fr{2s}{[\ell/3]+10}\Big\},
\end{align}
    then for $t\le \nu^{-\fr12}$, there holds
    \begin{align*}
        \big\|\la t\ra^b|\nb_x|^\fr32M^{(1)}\big\|_{L^2_t\mathcal{G}^{\lm,N-1}_s}^2\les\nn& \|\la\nb_x,\nb_v\ra^{2b+2}f_{\rm in}\|^2_{\mathcal{G}^{\lm(0),N-1}_{s,0}}+\big\|\la t\ra^b|\nb_x|^\fr32\rho^{(1)}\big\|^2_{L^2_t\mathcal{G}^{\lm,N-1}_s}\sup_t\mathcal{E}_{\ell}(g(t))\\
        &+\big\|\la t\ra^\fr{3+a-s}{2}\rho^{(2)}\big\|^2_{L^2_t\mathcal{G}^{\lm,N}_s}\sup_t\mathcal{E}_{\ell}(g(t)),\\
        \big\|\langle t\rangle^{\fr{3+a-s}{2}}M^{(2)}\big\|_{L^2_t\mathcal{G}^{\lm,N}_s}^2\les\nn& \big\|\langle t\rangle^{\fr{3+a-s}{2}}\rho^{(2)}\big\|^2_{L^2_t\mathcal{G}^{\lm,N}_s}\sup_{ t}\mathcal{E}_\ell(g(t))+\sup_{ t}\mathcal{E}_\ell(g(t))^2\\
        \nn&+\big\|\langle t\rangle^{\fr{3+a-s}{2}}\rho\big\|^2_{L^2_t\mathcal{G}^{\lm,N}_s}\sup_{ t}\mathcal{E}_\ell(g(t))^2\\
        \nn&+\nu^{\fr13}\int_0^{T}w_3^2(t)\mathcal{D}_{\ell-6}(g_{\ne}(t))dt\sup_{ t}\mathcal{E}_\ell(g(t))\\
        \nn&+\nu^{\fr13}\int_0^{T}w_1^2(t)\mathcal{D}_{\ell-2}(g_{\ne}(t))dt\sup_{ t}\mathcal{E}_\ell(g(t))\\
        &+\int_0^{T}w_6^2(t)\mathcal{CK}_{\ell-12}(g_{\ne}(t))dt\sup_{ t}\mathcal{E}_\ell(g(t)),
    \end{align*}
    and for $\iota=1,2,\cdots,[\ell/3]+1$,
    \begin{align*}
      \big\|\langle t\rangle^{\fr{1+a-s}{2}}w_{\iota}(t)M^{(2)}\big\|^2_{L^2_t\mathcal{G}^{\lm,N}_s}
    \les\nn&  \big\|\langle t\rangle^{\fr{1+a-s}{2}}w_{\iota}(t)\rho^{(2)}\big\|^2_{L^2_t\mathcal{G}^{\lm,N}_s}\sup_{ t}\mathcal{E}_\ell(g(t))\\
    \nn&+\sup_{ t}\Big(\mathcal{E}_\ell(g(t))+w_{\iota}^2(t)\mathcal{E}_{\ell-2\iota}(g_{\ne}(t))\Big)\mathcal{E}_\ell(g(t))\\
    \nn&+\big\|\langle t\rangle^{\fr{1+a-s}{2}}w_{\iota}(t)\rho\big\|^2_{L^2_t\mathcal{G}^{\lm,N}_s}\sup_{ t}\mathcal{E}_\ell(g(t))^2\\
    \nn&+\nu^\fr13\int_0^{T}w_{\iota}^2(t)\mathcal{D}_{\ell-2\iota}(g_{\ne}(t))dt\sup_{ t}\mathcal{E}_\ell(g(t))\\
    &+\int_0^{T}w^2_{\iota}(t)\mathcal{CK}_{\ell-2\iota}(g_{\ne}(t))dt\sup_{ t}\mathcal{E}_\ell(g(t)).
    \end{align*}
\end{prop}

{\Large \part{Long-time relaxation regime $t \geq \nu^{-1/2} $}}

\section{Estimates on $(f_0^L,\tilde{f})$}
In this section, we study the distribution functions $f_0^L$ and $\tilde{f}$. Since $f_0^L$ solves the homogeneous Landau equation, the energy estimates can be treated as the $x$-averaged part of the inhomogeneous Landau equation which is simply \cite{Guo2002} removing the $v\cdot \nabla_x f$ term and using the conservation laws to keep coercivity. The treatment of $\tilde{f}$ is more complicated. Compared to the energy estimates in \cite{chaturvedi2023vlasov}, we need to treat the additional quasi-linear terms that involve $f_0^L$. Thanks to the better localization and regularity properties, we can regard these as linear terms, however, with only $O(1)$ smallness. 
\subsection{Estimates on $f_0^L$}\label{sec:f_0^L}
Recalling \eqref{eq:quasi-linearized},  $f_0^L$ satisfies the following homogeneous Landau equation:
\begin{align}\label{eq-f_0^L}
\pr_tf_0^L+\nu Lf_0^L=\nu\Gamma(f_0^L, f_0^L),\quad f_0^L(0,v)=({\rm Id}-{\rm P}_0)f_0(\nu^{-\fr12},v).
\end{align}
The aim of the section is  to establish the following energy bounds on $f_0^L$. 
\begin{prop}\label{prop-f_0^L}
Assume that $N'
\in \N$, $\ell'>0$, and $f_0^L$ is a solution to \eqref{eq-f_0^L}. Then $\exists \eps_1>0$, such that if $\|f_0^L(0)\|_{H^{N'}_{\ell'}}\le \eps_1$, then there holds
    \begin{align}\label{bd-f_0^L}
        \|f_0^L(t)\|_{H^{N'}_{\ell'}}^2+\fr12\nu\int_0^t\|f_0^L(t')\|_{H^{N'}_{\sig,\ell'}}^2dt'\le 2 \|f_0^L(0)\|_{H^{N'}_{\ell'}}^2.
    \end{align}
\end{prop}
\begin{proof}
    For $\beta\in\N^3$, applying $\pr_v^\beta$ to \eqref{eq-f_0^L} yields
    \begin{align*}
        \pr_t\pr_v^\beta f_0^L+\nu \pr_v^\beta Lf_0^L=\nu \pr_v^\beta \Gamma(f_0^L, f_0^L).
    \end{align*}
Taking the inner product of this equation with $\la v\ra^{2\ell'}\pr_v^\beta f_0^L$, we are led to
\begin{align*}
    \fr12\fr{d}{dt}\|\la v\ra^{\ell'}\pr_v^\beta f_0^L\|_{L^2_v}^2+\nu\Big\la \pr_v^\beta L f_0^L,\la v\ra^{2\ell'}\pr_v^\beta f_0^L\Big\ra_v=\nu\Big\la \pr_v^\beta \Gamma(f_0^L, f_0^L), \la v\ra^{2\ell'}\pr_v^\beta f_0^L\Big\ra_v.
\end{align*}
Moreover, the $L^2_v$ energy estimate for $f_0^L$ without velocity weight is given by
\begin{align*}
    \fr12\fr{d}{dt}\|f_0^L\|_{L^2_v}^2+\nu\big\la Lf_0^L,f_0^L\big\ra_v=\nu\big\la \Gamma(f_0^L,f_0^L),f_0^L\big\ra_v.
\end{align*}
By Lemmas \ref{Lem: Lemma-5} and \ref{Lem: Lemma-6}, we have 
\begin{align}\label{lower-Lf0-1}
    \Big\la \pr_v^\beta L f_0^L,\la v\ra^{2\ell}\pr_v^\beta f_0^L\Big\ra_v\ge |\pr_v^\beta f_0^L|^2_{\sig,\ell'}-\zeta\sum_{|\beta'|\le |\beta|}\left|\pr_v^{\beta'}f_0^L\right|^2_{\sig,\ell'}-C_{\zeta}\|\bar{\chi}_{C_\zeta}f_0^L\|_{L^2_v}^2,
\end{align}
and
\begin{align}\label{lower-Lf0-2}
    \big\la  L f_0^L, f_0^L\big\ra_v\ge \dl | ({\rm I}-{\rm P}_0)f_0^L|^2_{\sig}.
\end{align}

For ${\rm P}_0f_0^L$, we write ${\rm P}_0f_0^L(t,x)=\tl{a}(t)\sqrt{\mu}+\sum_{j=1}^3\tl{b}_j(t)v_j\sqrt{\mu}+\tl{c}(t)(|v|^2-\fr32)\sqrt{\mu}$, with
\begin{align*}
    \tl{a}(t)\int_{\R^3}\mu dv=\int_{\R^3} f_0^L(t,v)\sqrt{\mu}dv,\quad \tl{b}_j(t)\int_{\R^3}v_j^2\mu dv=\int_{\R^3}f_0^L(t,v)v_j\sqrt{\mu}dv,
\end{align*}
and
\begin{align*}
    \tl{c}(t)\int_{\R^3}\big[(|v|^2-\fr32)\sqrt{\mu}\big]^2dv=\int_{\R^3}f_0^L(t,v)(|v|^2-\fr32)\sqrt{\mu}dv.
\end{align*}
By the choice of initial data of $f_0^L$
and conservation laws,
\begin{align}\label{conserv-f_0^L}
\int_{\R^3}f_0^L(t,v)\sqrt{\mu}dv=\int_{\R^3}f_0^L(t,v)v_j\sqrt{\mu}dv=\int_{\R^3}f_0^L(t,v)|v|^2\sqrt{\mu}dv=0,\quad j=1,2,3.
\end{align}
Consequently, $\tl{a}(t)=\tl{b}_j(t)=\tl{c}(t)=0, j=1,2,3$. Thus, ${\rm P}_0f_0^L=0$. Accordingly, by choosing $\zeta$ sufficiently small in \eqref{lower-Lf0-1} and multiplying \eqref{lower-Lf0-2} by a large constant ${\rm B}_1$, we can employ the coercive estimate \eqref{coercive} to absorb the terminal term in \eqref{lower-Lf0-1}:
\begin{align*}
    \sum_{|\beta|\le N'}\nu\Big\la \pr_v^\beta L f_0^L,\la v\ra^{2\ell'}\pr_v^\beta f_0^L\Big\ra_v+{\rm B}_1\nu\big\la  L f_0^L, f_0^L\big\ra_v\ge\fr12\nu\sum_{|\beta|\le N}\big|\pr_v^\beta f_0^L\big|_{\sig,\ell}^2+\fr12\nu {\rm B}_1\dl \big|f_0^L\big|_{\sig}^2.
\end{align*}
The collsion terms can be bounded as follows
\begin{align*}
    \nn&\nu\sum_{|\beta|\le N'}\Big\la \pr_v^\beta \Gamma(f_0^L, f_0^L), \la v\ra^{2\ell'}\pr_v^\beta f_0^L\Big\ra_v\\
    \le\nn& C\sum_{|\beta|\le N'}\sum_{\substack{\beta'\le\beta\\ \beta''\le\beta'}}\Big(\big\|\la v\ra^{\ell'}\pr_v^{\beta''} f_0^L\big\|_{L^2_v}\big|\pr_v^{\beta-\beta'}f_0^L\big|_{\sig,\ell'}+\big|\pr_v^{\beta''} f_0^L\big|_{\sig,\ell'}\big\|\la v\ra^{\ell'} \pr_v^{\beta-\beta'}f_0^L\big\|_{L^2_v}\Big)\big|\pr_v^\beta f_0^L|_{\sig,\ell'}\\
    \le&C\sum_{|\beta|\le N'}\big\|\la v\ra^{\ell'}\pr_v^\beta f_0^L\big\|_{L^2_v}\Big(\nu\sum_{|\beta|\le N'}\big|\pr_v^\beta f_0^L \big|^2_{\sig,\ell'}\Big),
\end{align*}
and
\begin{align*}
    \nu\Big\la  \Gamma(f_0^L, f_0^L),  f_0^L\Big\ra_v\le C\|f_0^L\|_{L^2_v}\Big(\nu|f_0^L|^2_{\sig}\Big).
\end{align*}
Then \eqref{bd-f_0^L} follows immediately by using the standard continuity argument.
\end{proof}

\subsection{Estimates on $\tl{f}$}\label{sec: tlf}
Many steps in the $\tl{f}$ estimate are similar to those in \cite{chaturvedi2023vlasov}, so we will mainly focus only on the steps that differ significantly. 
To do so, let us classify the terms in \eqref{eq:quasi-linearized} into following three sets
\begin{itemize}
    \item Linear term $-2{E}\cdot v\mu^{\fr12}$: this term is treated in Lemma 9.5 of \cite{chaturvedi2023vlasov}.
    \item Nonlinear terms 
    \begin{itemize}
        \item Collision effect $\nu \Gamma(\tilde{f}, \tilde{f})$: this term is treated in Lemma 9.7 of \cite{chaturvedi2023vlasov}.
        \item Collisionless effect $-{E}\cdot \nabla \tilde{f}+v\cdot \tl{f}$: this term is treated in Lemma 9.6 of \cite{chaturvedi2023vlasov}.
    \end{itemize}
    \item Quasi-linear terms
    \begin{itemize}
        \item Collision effect $\nu \Gamma(\tilde{f}, f_0^L)+\nu \Gamma(f_0^L,\tilde{f})$: this term can be controlled by the linear collision term due to the smallness of $f_0^L$. 
        \item Collisionless effects ${E}\cdot \nabla f_0^L$ and ${E}\cdot v f_0^L$: the treatment of these two terms are similar to that of the linear term $-2{E}\cdot v\mu^{\fr12}$, since $f_0^L$ has higher regularity and stronger localization than $\tilde{f}$. 
    \end{itemize}
\end{itemize}

\begin{prop}\label{prop-tlf}
    Under the bootstrap hypotheses \eqref{H-f_0^L}--\eqref{H-tlrho}, for all $0\le N'\le \tl{N}$, we have for some $\underline{\eta}>0$ independent of $\nu$ that
    \begin{align}\label{bd-E-longtime}
   \nn&\fr{d}{dt}\mathbb{E}_{N',\tl\ell}(\tl{f}(t))+\underline{\eta} \nu^{\fr13}\mathbb{D}_{N',\tl\ell}(\tl{f}(t))\\
    \le\nn&C\nu^{1+4\mathfrak{b}}\eps^4\la t\ra^{-4}+C{\rm C}_{\rho}{\rm C}_{f,2}\eps^2 \nu^{2\mathfrak{b}}\la t\ra^{-2}\\
    \nn&+\nu\|f_0^L\|^2_{H^{\tl{N}+2}_{\sig,\tl\ell}}\mathbb{E}_{N',\tl\ell}(\tl{f}(t))+C\eps\la t\ra^{-2}\mathbb{E}_{N',\tl\ell}(\tl{f}(t))\\
    \nn&+\nu^{\mathfrak{b}}\mathbb{E}_{N',\tl\ell}(\tl f(t))^{\fr12}\mathbb{D}_{N',\tl\ell}(\tl f(t))^{\fr12}\mathbb{D}_{\tl N-2,\tl\ell}(\tl f(t))^{\fr12}\\
    &+\|\rho\|_{\tl{H}^{N'}_x}\min \Big\{{\mathbb{E}}_{N',\tl\ell}(\tl{f}(t))^{\fr12},{\mathbb{D}}_{N',\tl\ell}(\tl{f}(t))^{\fr12}\Big\}.
\end{align}
\end{prop}
\begin{proof}
For $m\in\N^6$, let us define a set of inhomogeneous terms  
\begin{align*}
    \mathcal{R}^{(m)}=\big\{2Z^{m}(E\cdot v\sqrt{\mu}), [E\cdot\nb_v-E\cdot v, Z^{m}]\tl{f}, -Z^m(E\cdot \nb_vf_0^L), Z^m(E\cdot v f_0^L), \nu Z^m\Gamma(\tl{f},\tl{f})\big\}.
\end{align*}
For ${Q}_{m}\in\mathcal{R}^{(m)}$, let us denote
\begin{align*}
   {\bf N}_{N',\tl\ell}[Q]=\sum_{m\in \N^6,|m|\le N'}&\Bigg[{\rm A}_0\sum_{|\al|\le1}\Big\la \pr_x^\al Q_{m}, e^{2\phi}\la v\ra^{2\tl\ell-4|\al|-4|m|}\pr_x^\al\tl{f}^{(m)}\Big\ra_{x,v}\\
    &+\sum_{1\le|\al|\le2}\kappa^{2|\al|}\nu^{\fr{2|\al|}{3}}\Big\la \pr_v^\al Q_{m}, e^{2\phi}\la v\ra^{2\tl\ell-4|\al|-4|m|}\pr_v^\al\tl{f}^{(m)}\Big\ra_{x,v}\\
    &+2\kappa\nu^{\fr13}\Big\la\nb_xQ_{m}, e^{2\phi}\la v\ra^{2\tl\ell-4-4|m|}\nb_v\tl{f}^{(m)} \Big\ra_{x,v}\\
    &+2\kappa\nu^{\fr13}\Big\la\nb_vQ_{m}, e^{2\phi}\la v\ra^{2\tl\ell-4-4|m|}\nb_x\tl{f}^{(m)} \Big\ra_{x,v}\Bigg],
\end{align*}
and we use the notation ${\bf N}_{N',\tl\ell}[f_0^L]$ if $Q_m$ is replaced by $\nu Z^m\Gamma(\tl{f},f_0^L)+\nu Z^m\Gamma(f_0^L,\tl{f})$.

Recalling the second equation of \eqref{eq:quasi-linearized}, the equation for $\tl{f}^{(m)}=Z^m\tl{f}$ can be written as
\begin{align*}
    \pr_t\tl{f}^{(m)}+v\cdot \nb_x\tl{f}^{(m)}+E\cdot\nb_v\tl{f}^{(m)}-&E\cdot v\tl{f}^{(m)}+\nu Z^m L\tl{f}\\
    &\quad=\nu Z^m\Gamma(\tl{f},f_0^L)+\nu Z^m\Gamma(f_0^L,\tl{f})+\sum_{{Q}_{m}\in\mathcal{R}^{(m)}}{ Q}_{m}.
\end{align*}
For the evolution of $\|e^\phi \tl{f}\|_{L^2_{x,v}}$,  similar to \eqref{en-0}--\eqref{est-linear0}, it is not difficult to verify that
\begin{align*}
    \nn&\fr{d}{dt}\|e^{\phi}\tl{f}\|^2_{L^2_{x,v}}+\dl\nu\|({\rm I}-{\rm P_0})e^\phi \tl{f}\|_{\sig}^2\\
    \les\nn&\|\pr_t\phi\|_{L^\infty_x}\big\|e^\phi \tl{f}\big\|_{L^2_{x,v}}^2+\|E\|_{L^2_x}\|e^\phi\tl{f}\|_{L^2_{x,v}}+\|E\|_{L^2_x}\|e^\phi\tl{f}\|_{L^2_{x,v}}\Big(\|\nb_vf_0^L\|_{L^2_v}+\|vf_0^L\|_{L^2_v}\Big)\\
    &+\nu\|e^\phi \tl{f}\|_{L^2_{x,v}}\|e^\phi\tl{f}\|^2_{\sig}+\nu\Big(\|f_0^L\|_{L^2_{x,v}}\|e^\phi \tl{f}\|_{\sig}^2+\big|f_0^L\big|_{\sig}\|e^{\phi}\tl{f}\|_{L^2_{x,v}}\|e^\phi \tl{f}\|_{\sig}\Big).
\end{align*}
Combining this with Proposition 5.11 of \cite{chaturvedi2023vlasov}, we find that there exists a small positive constant $\underline{\eta}$, independent of $\nu$, such that the evolution of $\mathbb{E}_{N',\tl\ell}(\tl{f})$ satisfies 
\begin{align}\label{evolu-E-tlf}
    \nn&\fr{d}{dt}\mathbb{E}_{N',\tl\ell}(\tl{f}(t))+2\underline{\eta} \nu^{\fr13}\mathbb{D}_{N',\tl\ell}(\tl{f}(t))\\
    \le\nn&C\nu\|E\|_{L^2_x}^4+C\big(\|\pr_t\phi\|_{L^\infty_x}+\|\phi\|_{W^{1,\infty}_x}\big)\mathbb{E}_{N',\tl\ell}(\tl{f}(t))+C\Big({\bf N}_{N',\tl\ell}[f_0^L]+\sum_{Q_m\in \mathcal{R}^{(m)}}{\bf N}_{N',\tl\ell}[Q]\Big)\\
    \nn&+\|E\|_{L^2_x}\|e^\phi\tl{f}\|_{L^2_{x,v}}+\|E\|_{L^2_x}\|e^\phi\tl{f}\|_{L^2_{x,v}}\Big(\|\nb_vf_0^L\|_{L^2_v}+\|vf_0^L\|_{L^2_v}\Big)\\
    &+\nu\|f_0^L\|_{L^2_{x,v}}\mathbb{D}_{N',\tl\ell}(\tl f(t))+\nu\big|f_0^L\big|^2_{\sig}\mathbb{E}_{N',\tl\ell}(\tl f(t)).
\end{align}
To bound ${\bf N}_{N',\tl\ell}[f_0^L]$, by virtue of Lemma 4.9 of \cite{chaturvedi2023vlasov}, one deduces that
\begin{align*}
    &\Big\la\pr_x^\al Z^{m}\Gamma(\tl{f}, f^L_0)+\pr_x^\al Z^{m}\Gamma(f^L_0,\tl{f}), e^{2\phi}\la v\ra^{2\tl\ell-4|\al|-4|m|}\pr_x^\al\tl{f}^{(m)}\Big\ra_{x,v}\\
    \les& \big\|e^{\phi}\pr_x^\al \tl{f}^{(m)}\big\|_{\sig,\tl\ell-2|\al|-2|m|}\sum_{|m'|+|m''|\le|m|}\Big(\big\|\pr_v^{
    m'}f_0^L\la v\ra^{\tl\ell}\big\|_{L^2_v}\big\|e^\phi\pr_x^\al\tl{f}^{(m'')}\big\|_{\sig,\tl\ell-2|\al|-2|m|}\\
    &+\big|\pr_v^{
    m'}f_0^L\big|_{\sig,\tl\ell}\big\|e^\phi \la v\ra^ {\tl\ell-2|\al|-2|m''|}\pr_x^\al\tl{f}^{(m'')}\big\|_{L^2_{x,v}}\Big),
\end{align*}

\begin{align*}
    &\sum_{1\le|\al|\le2}\kappa^{2|\al|}\nu^{\fr{2|\al|}{3}}\Big\la \pr_v^\al Z^{m}\Gamma(\tl{f}, f^L_0)+\pr_v^\al Z^{m}\Gamma(f^L_0,\tl{f}), e^{2\phi}\la v\ra^{2\tl\ell-4|\al|-4|m|}\pr_v^\al\tl{f}^{(m)}\Big\ra_{x,v}\\
    \les&\kappa\nu^{\fr13}\sum_{|m'|+|m''|\le|m|}\Big(\big\|\pr_v^{m'}f_0^L\la v\ra^{\tl\ell}\big\|_{L^2_v}\big\|e^{\phi}\nb_v\tl{f}^{(m'')}\big\|_{\sig,\tl\ell-2-2|m''|}\\
    &\qquad+\big|\pr_v^{m'}f_0^L\big|_{\sig,\tl\ell}\big\|e^{\phi}\la v\ra^{\tl\ell-2-2|m''|}\nb_v\tl{f}^{(m'')}\big\|_{L^2_{x,v}}\Big)\big\|\kappa\nu^{\fr13}\nb_v\tl{f}^{(m)}\big\|_{\sig,\tl\ell-2-2|m|}\\
    &+\kappa\nu^{\fr13}\sum_{|m'|+|m''|\le|m|}\sum_{|\al|\le1}\Big(\big\|\pr^{\al}_v\pr_v^{m'}f_0^L\la v\ra^{\tl\ell}\big\|_{L^2_v}\big\|e^{\phi}\tl{f}^{(m'')}\big\|_{\sig,\tl\ell-2|m''|}\\
    &\qquad+\big|\pr_v^\al\pr_v^{m'}f_0^L\big|_{\sig,\tl\ell}\big\|e^{\phi}\la v\ra^{\tl\ell-2|m''|}\tl{f}^{(m'')}\big\|_{L^2_{x,v}}\Big)\big\|\kappa\nu^{\fr13}\nb_v\tl{f}^{(m)}\big\|_{\sig,\tl\ell-2-2|m|}\\
    &+\kappa^2\nu^{\fr23}\sum_{|m'|+|m''|\le|m|}\Big(\big\|\pr_v^{m'}f_0^L\la v\ra^{\tl\ell}\big\|_{L^2_v}\big\|e^{\phi}\Dl_v\tl{f}^{(m'')}\big\|_{\sig,\tl\ell-4-2|m''|}\\
    &\qquad+\big|\pr_v^{m'}f_0^L\big|_{\sig,\tl\ell}\big\|e^{\phi}\la v\ra^{\tl\ell-4-2|m''|}\Dl_v\tl{f}^{(m'')}\big\|_{L^2_{x,v}}\Big)\big\|\kappa^2\nu^{\fr23}\Dl_v\tl{f}^{(m)}\big\|_{\sig,\tl\ell-4-2|m|}\\
    &+\kappa^2\nu^{\fr23}\sum_{|m'|+|m''|\le|m|}\sum_{|\al|\le1}\Big(\big\|\pr^{\al}_v\pr_v^{m'}f_0^L\la v\ra^{\tl\ell}\big\|_{L^2_v}\big\|e^{\phi}\nb_v\tl{f}^{(m'')}\big\|_{\sig,\tl\ell-2-2|m''|}\\
    &\qquad+\big|\pr_v^\al\pr_v^{m'}f_0^L\big|_{\sig,\tl\ell}\big\|e^{\phi}\la v\ra^{\tl\ell-2-2|m''|}\nb_v\tl{f}^{(m'')}\big\|_{L^2_{x,v}}\Big)\big\|\kappa^2\nu^{\fr23}\Dl_v\tl{f}^{(m)}\big\|_{\sig,\tl\ell-4-2|m|}\\
    &+\kappa^2\nu^{\fr23}\sum_{|m'|+|m''|\le|m|}\sum_{|\al|\le2}\Big(\big\|\pr^{\al}_v\pr_v^{m'}f_0^L\la v\ra^{\tl\ell}\big\|_{L^2_v}\big\|e^{\phi}\tl{f}^{(m'')}\big\|_{\sig,\tl\ell-2|m''|}\\
    &\qquad+\big|\pr_v^\al\pr_v^{m'}f_0^L\big|_{\sig,\tl\ell}\big\|e^{\phi}\la v\ra^{\tl\ell-2|m''|}\tl{f}^{(m'')}\big\|_{L^2_{x,v}}\Big)\big\|\kappa^2\nu^{\fr23}\Dl_v\tl{f}^{(m)}\big\|_{\sig,\tl\ell-4-2|m|},
\end{align*}
and
\begin{align*}
&2\kappa\nu^{\fr13}\Big\la\nb_xZ^m\Gamma(\tl{f},f_0^L)+\nb_x Z^m\Gamma(f_0^L,\tl{f}), e^{2\phi}\la v\ra^{2\tl\ell-4-4|m|}\nb_v\tl{f}^{(m)} \Big\ra_{x,v}\\
    &+2\kappa\nu^{\fr13}\Big\la\nb_vZ^m\Gamma(\tl{f},f_0^L)+\nb_v Z^m\Gamma(f_0^L,\tl{f}), e^{2\phi}\la v\ra^{2\tl\ell-4-4|m|}\nb_x\tl{f}^{(m)} \Big\ra_{x,v}\\
    \les&\sum_{|m'|+|m''|\le|m|}\Big(\big\|\pr_v^{m'}f_0^L\la v\ra^{\tl\ell}\big\|_{L^2_v}\big\|e^\phi \nb_x\tl{f}^{(m'')}\big\|_{\sig,\tl\ell-2-2|m''|}\\
    &\qquad+\big|\pr_v^{m'}f_0^L\big|_{\sig,\tl\ell}\big\|e^{\phi}\la v\ra^{\tl\ell-2-2|m''|}\nb_x\tl{f}^{(m'')}\big\|_{L^2_{x,v}}\Big)\big\|\kappa\nu^{\fr13}\nb_v\tl{f}^{(m)}\big\|_{\sig,\tl\ell-2-2|m|}\\
    &+\sum_{|m'|+|m''|\le|m|}\Big(\big\|\pr_v^{m'}f_0^L\la v\ra^{\tl\ell}\big\|_{L^2_v}\big\|e^\phi \kappa\nu^{\fr13}\nb_v\tl{f}^{(m'')}\big\|_{\sig,\tl\ell-2-2|m''|}\\
    &\qquad+\big|\pr_v^{m'}f_0^L\big|_{\sig,\tl\ell}\big\|e^{\phi}\la v\ra^{\tl\ell-2-2|m''|}\kappa\nu^{\fr13}\nb_v\tl{f}^{(m'')}\big\|_{L^2_{x,v}}\Big)\big\|\nb_x\tl{f}^{(m)}\big\|_{\sig,\tl\ell-2-2|m|}\\
    &+\sum_{|m'|+|m''|\le|m|}\sum_{|\al|\le1}\kappa\nu^{\fr13}\Big(\big\|\pr_v^\al\pr_v^{m'}f_0^L\la v\ra^{\tl\ell}\big\|_{L^2_v}\big\|e^\phi \tl{f}^{(m'')}\big\|_{\sig,\tl\ell-2|m''|}\\
    &\qquad+\big|\pr_v^\al\pr_v^{m'}f_0^L\big|_{\sig,\tl\ell}\big\|e^{\phi}\la v\ra^{\tl\ell-2|m''|}\tl{f}^{(m'')}\big\|_{L^2_{x,v}}\Big)\big\|\nb_x\tl{f}^{(m)}\big\|_{\sig,\tl\ell-2-2|m|}.
\end{align*}    
It follows from the above three estimates that
\begin{align}\label{est-quasi-f_0^L}
    {\bf N}_{N',\tl\ell}[f_0^L]\les\nn&\nu^{\fr13} \mathbb{D}_{N',\tl\ell}(\tl{f}(t))^{\fr12}\Big(\|f_0^L\|_{H^{N'+2}_{\tl\ell}}\mathbb{D}_{N',\tl\ell}(\tl{f}(t))^{\fr12}+\nu^{\fr13}\|f_0^L\|_{H^{N'+2}_{\sig,\tl\ell}}\mathbb{E}_{N',\tl\ell}(\tl{f}(t))^{\fr12}\Big)\\
    \le& \nu^{\fr13}\mathbb{D}_{N',\tl\ell}(\tl{f}(t))\Big(C\|f_0^L\|_{H^{N'+2}_{\tl\ell}}
    +\underline{\eta}/2\Big)+\nu\|f_0^L\|^2_{H^{\tl{N}+2}_{\sig,\tl\ell}}\mathbb{E}_{N',\tl\ell}(\tl{f}(t)).
\end{align}

We  are left to  bound $\sum_{Q_m\in \mathcal{R}^{(m)}}{\bf N}_{N',\tl\ell}[Q]$.
If $Q_m\in\mathcal{R}^{(m)}\setminus\{-Z^{m}(E\cdot\nb_vf_0^L), Z^m(E\cdot vf_0^L) \}$, ${\bf N}_{N',\tl\ell}[Q]$ has been bounded in \cite{chaturvedi2023vlasov} (see Proposition 9.2 of \cite{chaturvedi2023vlasov}). It suffices to deal with the case when 
 $Q_m=-Z^m(E\cdot\nb_vf_0^L)$ and $Z^m(E\cdot vf_0^L)$. Let us focus on the former, as the approach for the latter is similar. Basically, if $Q_m=E\cdot \nb_vf_0^L$, ${\bf N}_{N',\tl\ell}[Q]$ can be treated in a similar manner as the case when $Q_m=2Z^m(E\cdot v\sqrt{\mu})$. Moreover, strategies for addressing the four terms in in ${\bf N}_{N',\tl\ell}[Q]$ are identical. Here we only sketch the treatments of the third term in ${\bf N}_{N',\tl\ell}[Q]$ with $Q_m=-Z^m(E\cdot\nb_vf_0^L)$. Indeed, thanks to \eqref{coercive}, we find that
\begin{align*}
    &2\kappa\nu^{\fr13}\sum_{m\in \N^6,|m|\le \tl{N}}\Big\la\nb_xZ^{m}(E\cdot\nb_v f_0^L), e^{2\phi}\la v\ra^{2\tl\ell-4-4|m|}\nb_v\tl{f}^{(m)} \Big\ra_{x,v}\\
   \les&\sum_{m\in \N^6,|m|\le \tl{N}}\big\|Z^m(\nb_xE\cdot\nb_vf_0^L)\|_{L^2_{x,v}}\big\|e^{\phi}\la v\ra^{\tl\ell-2-2|m|}\kappa\nu^{\fr13}\nb_v\tl{f}^{(m)}\big\|_{L^2_{x,v}}\\
   \les&\|\rho\|_{\tl{H}^{N'}_x}\|f_0^L\|_{H^{N'+1}_{\tl\ell}}\min \Big\{{\mathbb{E}}_{N',\tl\ell}(\tl{f}(t))^{\fr12},{\mathbb{D}}_{N',\tl\ell}(\tl{f}(t))^{\fr12}\Big\}.
\end{align*}
It follows from this, \eqref{evolu-E-tlf},  \eqref{est-quasi-f_0^L}, the bootstrap hypothesis \eqref{H-f_0^L},    and Proposition 9.2 of \cite{chaturvedi2023vlasov} that \eqref{bd-E-longtime} holds.
\end{proof}

\begin{rem}\label{rem:improve-H-tlf}
    By using \eqref{bd-E-longtime}, one can improve the bounds in \eqref{H-tlf-H} and \eqref{H-tlf-L} immediately. We refer to the proof of Theorem 9.1 in \cite{chaturvedi2023vlasov} for more details.
\end{rem}

\section{Quasi-linear estimate }\label{sec:quasi-linear}
\subsection{The solution operator estimates and enhanced dissipation}
In this section, let us consider the following quasi-linear equation \eqref{eq:linearized 2}.
Taking Fourier transform of \eqref{eq:linearized 2} in $x$ yields
\begin{align}\label{QL-f-k}
    \partial_{t}\hat{{\rm f}}_k+ik\cdot v\hat{{\rm f}}_k
        +\nu L \hat{\rm f}_k=\nu\Gamma(f_0^L,\hat{\rm f}_k)+\nu\Gamma(\hat{\rm f}_k,f_0^L).
\end{align}
Let us denote by $\mathbb{S}_{k}(t,\tau), 0\le\tau\le t$ the solution operator associated with \eqref{QL-f-k}, namely, the unique solution ${\rm h}(t,v;\tau)$ to \eqref{QL-f-k} with initial data ${\rm h}(\tau, v;\tau)={\rm h}_{\rm in}(v;\tau)$ is given by
\begin{align}\label{def:solution-h}
{\rm h}(t,v;\tau)=\mathbb{S}_{k}(t,\tau)[{\rm h}_{\rm in}(v;\tau)].
\end{align}

\begin{itemize}
    \item For $k\in\Z^3_*, \zeta\in\R^3$, let us denote $Y_{\bar{\tau},\zeta}=\nb_v+i\big(k(t-\tau)+\zeta\big)$, here $\bar{\tau}$ stands for $t-\tau$. In particular, if $t=\tau$, $Y_{\bar{\tau},\zeta}$ reduces to $Y_{0,\zeta}=\nb_v+i\zeta$.
    \item In this section, for $n\in \N^3$, we still use $f^{(n)}$ to denote $Y_{\bar{\tau},\zeta}^nf$.
    \item We use the notations  $\mathfrak{E}_{\bar{\tau},\zeta;\ell}^n(f)$ and $\mathfrak{D}_{\bar{\tau},\zeta;\ell}^n(f)$ to denote the energy and dissipation for $f^{(n)}$ defined in \eqref{eq:varE_l^n} and \eqref{def-varD-nl}, respectively, with $\tl{Y}$ replaced by $Y_{\bar{\tau},\zeta}$.
    \item For $N'\in\N,\ell'>0$, let us introduce the energy functional and the corresponding dissipation term with finite regularity
\begin{align}\label{def-ENkl}
{\bf E}_{\bar{\tau},\zeta;\ell'}^{N'}({\rm h}(t;\tau))=\sum_{m\in\N^3,|m|\le N'}\kappa^{2|m|}\mathfrak{E}^m_{\bar{\tau},\zeta;\ell'}({\rm h}(t;\tau)),
\end{align}
and
\begin{align}\label{def-DNkl}
{\bf D}_{\bar{\tau},\zeta;\ell'}^{N'}({\rm h}(t;\tau))=\sum_{m\in\N^3,|m|\le N'}\kappa^{2|m|}\mathfrak{D}^m_{\bar{\tau},\zeta;\ell'}({\rm h}(t;\tau)).
\end{align}

   \item We also define the total energy and dissipation with finite regularity, incorporating the effects of enhanced dissipation:
   \begin{align}\label{def-EN-bartau}
       \tl{\mathsf{E}}_{N',\ell'}({\rm h}(t;\tau))=\sum_{\iota=0}^{[\ell/3]+1}w^2_{\iota}(t-\tau){\bf E}_{\bar{\tau},\zeta;\ell'-2\iota}^{N'}({\rm h}(t;\tau)),
   \end{align}
   and
   \begin{align}\label{def-DN-bartau}
       \tl{\mathsf{D}}_{N',\ell'}({\rm h}(t;\tau))=\sum_{\iota=0}^{[\ell/3]+1}w^2_{\iota}(t-\tau){\bf D}_{\bar{\tau},\zeta;\ell'-2\iota}^{N'}({\rm h}(t;\tau)).
   \end{align}
\end{itemize}

The  aim of this section is to bound $\tl {\mathsf{E}}_{N',\ell'}({\rm h}(t;\tau))$ in terms of the initial data ${\rm h}_{\rm in}(v;\tau)$ with suitable norm. More precisely, we have the following proposition.
\begin{prop}\label{prop:solution-operator}
Let ${\rm h}(t,v;\tau)$ be the solution to \eqref{QL-f-k} with initial data $h_{\rm in}(v;\tau)$. Then for any $N'\in\N$ with $N'\le \tl N$ and $2<\ell'\le\tl\ell$, the following estimate holds
\begin{align}\label{est-solution oper}
    \tl{\mathsf{E}}_{N',\ell'}({\rm h}(t;\tau))+\nu^{\fr13}\int_\tau^{t}\tl{\mathsf{D}}_{N',\ell'}({\rm h}(t';\tau))dt'\les \tl{\mathsf{E}}_{N',\ell'}({\rm h}_{\rm in}(\tau)).
\end{align}
\end{prop}
\begin{proof}
For any $m\in\N^3$, ${\rm h}^{(m)}=Y_{\bar{\tau},\zeta}^m{\rm h}$ satisfies 
\begin{align}\label{eq-hm}
    \pr_t{\rm h}^{(m)}+ik\cdot v{\rm h}^{(m)}+\nu Y_{\bar{\tau},\zeta}^{m}L{\rm h}=\nu Y_{\bar{\tau},\zeta}^{m}\Gamma(f_0^L,{\rm h})+\nu Y_{\bar{\tau},\zeta}^{m}\Gamma({\rm h},f_0^L).
\end{align}
Like in Section \ref{sec-semigroup est}, here we use ${\rm h}_{[\iota]}$ to denote $w_{\iota}(t-\tau){\rm h}$. Moreover, we also adapt the notations defined in \eqref{est-dissip1}, \eqref{est-dissip2} and \eqref{def-dissip4} to denote the contributions of the Landau operator:
\begin{align}
    {\rm Dis}_{k;1}^{\bar{\tau},\zeta}({\rm h}_{[\iota]}(t;\tau))=\nn&\nu {\rm A}_0\sum_{\substack{m\in\N^3\\ |m|\le N'}}\kappa^{2|m|}\sum_{|\al|\le1}\mathrm{Re}\,\left\la Y_{\bar{\tau},\zeta}^{m} Lk^\al {\rm h}_{[\iota]},k^\al {\rm h}_{[\iota]}^{(m+\beta)}\la v\ra^{2\ell'-4|\al|-4\iota}\right\ra_{v},
\end{align}
and the remaining terms ${\rm Dis}^{\bar{\tau},\zeta}_{k;2}({\rm h}_{[\iota]}(t;\tau))$ and ${\rm Dis}_{k;4}^{\bar{\tau},\zeta}({\rm h}_{[\iota]}(t;\tau))$ can be obtained by making corresponding modifications to ${\rm Dis}_{2}(g_{\iota*})$ and ${\rm Dis}_{4}(g_{\iota*})$.

Then recalling the definition of ${\bf E}_{\bar{\tau},\zeta;\ell'}^{N'}({\rm h}(t;\tau))$ in \eqref{def-ENkl}, by \eqref{eq-hm},  we have 
\begin{align}\label{en-identity-bartau}
    \nn&\fr12\fr{d}{dt}{\bf E}_{\bar{\tau},\zeta;\ell'}^{N'}({\rm h}_{[\iota]}(t;\tau))+{\rm Dis}_{k;1}^{\bar{\tau},\zeta}({\rm h}_{[\iota]}(t;\tau))+{\rm Dis}_{k;2}^{\bar{\tau},\zeta}({\rm h}_{[\iota]}(t;\tau))+{\rm Dis}_{k;4}^{\bar{\tau},\zeta}({\rm h}_{[\iota]}(t;\tau))\\
    \nn&+\kappa\nu^{\fr13}|k|^2\sum_{\substack{m\in \N^3\\ |m|\le N}}\kappa^{2|m|}\big\|\la v\ra^{\ell'-2-2\iota}{\rm h}_{[\iota]}\big\|_{L^2_v}^2\\
    \nn&+\kappa^2\nu^{\fr23}\sum_{\substack{m\in \N^3\\ |m|\le N}}\kappa^{2|m|}\mathrm{Re}\,\Big\la ik{\rm h}^{(m)}_{[\iota]},\la v\ra^{2\ell'-4-4\iota} \nb_v{\rm h}^{(m)}_{[\iota]}\Big\ra_{v}\\
    =&{\rm Er}_{k,\iota}^{\bar{\tau},\zeta}({\rm h}_{[\iota]}(t;\tau))+{\rm PD}_{k,\iota}^{\bar{\tau},\zeta}({\rm h}_{[\iota]}(t;\tau)),
\end{align}
where
\begin{align*}
    &{\rm Er}_{k,\iota}^{\bar{\tau},\zeta}({\rm h}_{[\iota]}(t;\tau))\\
    =&\fr{\iota \kappa_0\nu^{\fr13}}{2}\big(\kappa_0\nu^{\fr13}(1+t)\big)^{-1}\sum_{\substack{m\in\N^3\\ |m|\le N}}\kappa^{2|m|}\bigg({\rm A}_0\sum_{|\al|\le1}\big\|k^{\al}{\rm h}_{[\iota]}^{(m)}\la v\ra^{2\ell'-4|\al|-4\iota}\big\|_{L^2_{v}}^2\\
    \nn&+\kappa^2\nu^{\fr23}\big\|\nb_v{\rm h}_{[\iota]}^{(m)}\la v\ra^{2\ell'-4-4\iota}\big\|_{L^2_{v}}^2+2\kappa\nu^{\fr13}\mathrm{Re}\,\left\la ik{\rm h}_{[\iota]}^{(m)},\la v\ra^{2\ell'-4-4\iota}\nb_{v}{\rm h}_{[\iota]}^{(m)} \right\ra_{v}\bigg),
\end{align*}
and
\begin{align*}
    &{\rm PD}_{k,\iota}^{\bar{\tau},\zeta}({\rm h}_{[\iota]}(t;\tau))\\
    =&{\rm A}_0 \sum_{\substack{m\in \N^3\\ |m|\le N}}\kappa^{2|m|}\sum_{|\al|\le1}\nu\mathrm{Re}\,\Big\la  Y_{\bar{\tau},\zeta}^{m}\Gamma(f_0^L, k^\al{\rm h}_{[\iota]})+ Y_{\bar{\tau},\zeta}^{m}\Gamma(k^\al{\rm h},f_0^L),\la v\ra^{2\ell'-4|\al|-4\iota}k^\al{\rm h}^{(m)}_{[\iota]} \Big\ra_v\\
    \nn&+\kappa^2\nu^{\fr23}\sum_{\substack{m\in \N^3\\ |m|\le N}}\kappa^{2|m|}\nu\mathrm{Re}\,\Big\la  Y_{\bar{\tau},\zeta}^{m}\nb_v\Gamma(f_0^L,{\rm h}_{[\iota]})+ Y_{\bar{\tau},\zeta}^{m}\nb_v\Gamma({\rm h}_{[\iota]},f_0^L),\la v\ra^{2\ell'-4-4\iota}\nb_v{\rm h}^{(m)}_{[\iota]} \Big\ra_v\\
    \nn&+\kappa\nu^{\fr13}\sum_{\substack{m\in \N^3\\ |m|\le N}}\kappa^{2|m|}\Bigg[\nu\mathrm{Re}\,\Big\la  ik Y_{\bar{\tau},\zeta}^{m}\Gamma(f_0^L,{\rm h}_{[\iota]})+ ik Y_{\bar{\tau},\zeta}^{m}\Gamma({\rm h}_{[\iota]},f_0^L),\la v\ra^{2\ell'-4-4|\iota|}\nb_v{\rm h}^{(m)}_{[\iota]} \Big\ra_v\\
    &+\nu\mathrm{Re}\,\Big\la  ik{\rm h}^{(m)}_{[\iota]}\la v\ra^{2\ell'-4-4\iota}, Y_{\bar{\tau},\zeta}^{m}\nb_v\Gamma(f_0^L,{\rm h}_{[\iota]})+ Y_{\bar{\tau},\zeta}^{m}\nb_v\Gamma({\rm h}_{[\iota]},f_0^L)\Big\ra_v\Bigg],
\end{align*}
here `Er' stands for `error' and `PD' stands for `perturbed dissipation'.

Similar to \eqref{lb-dis} and \eqref{dt-(1+t)-var}, we get
\begin{align}\label{lb-dis-iota}
    \nn&{\rm Dis}_{k;1}^{\bar{\tau},\zeta}({\rm h}_{[\iota]}(t;\tau))+{\rm Dis}_{k;2}^{\bar{\tau},\zeta}({\rm h}_{[\iota]}(t;\tau))+{\rm Dis}_{k;4}^{\bar{\tau},\zeta}({\rm h}_{[\iota]}(t;\tau))\\
    \nn&+\kappa\nu^{\fr13}|k|^2\sum_{\substack{m\in \N^3\\ |m|\le N'}}\kappa^{2|m|}\big\|\la v\ra^{\ell'-2-2\iota}{\rm h}_{[\iota]}\big\|_{L^2_v}^2\\
    &+\kappa^2\nu^{\fr23}\sum_{\substack{m\in \N^3\\ |m|\le N'}}\kappa^{2|m|}\mathrm{Re}\,\Big\la ik{\rm h}^{(m)}_{[\iota]},\la v\ra^{2\ell'-4-4\iota} \nb_v{\rm h}^{(m)}_{[\iota]}\Big\ra_{v}\ge\fr{\nu^{\fr13}}{4}{\bf D}_{\bar{\tau},\zeta;\ell'-2\iota}^{N'}({\rm h}_{[\iota]}(t;\tau)),
\end{align}
and
\begin{align}\label{dt-(1+t)-var-bartau}
\sum_{\iota=0}^{[\ell/3]+1}{\rm Er}_{k,\iota}^{\bar{\tau},\zeta}({\rm h}_{[\iota]}(t;\tau))
    \le\fr{\nu^{\fr13}}{8}\sum_{\iota=0}^{[\ell/3]} {\bf D}_{\bar{\tau},\zeta;\ell'-2\iota}^{N'}({\rm h}_{[\iota]}(t)).
\end{align}
It follows from \eqref{en-identity-bartau}, \eqref{lb-dis-iota} and \eqref{dt-(1+t)-var-bartau} and the definition of $\tl{\mathsf{ E}}_{N',\ell'}({\rm h}(t;\tau))$ and $\tl{\mathsf{D}}_{N',\ell'}({\rm h}(t;\tau))$ in \eqref{def-EN-bartau} and \eqref{def-DN-bartau}  that
\begin{align}\label{en-ineq-bartau}
    \fr12\fr{d}{dt}\tl{\mathsf{E}}_{N',\ell'}({\rm h}(t;\tau))+\fr18\nu^{\fr13}\tl{\mathsf{ D}}_{N',\ell'}({\rm h}(t;\tau))\le \sum_{\iota=0}^{[\ell/3]+1} {\rm PD}_{k,\iota}^{\bar{\tau},\zeta}({\rm h}_{[\iota]}(t;\tau)).
\end{align}

We are left to bound ${\rm PD}_{k,\iota}^{\bar{\tau},\zeta}({\rm h}_{[\iota]}(t;\tau))$, which can be addressed similarly to ${\bf N}_{N',\tl\ell}[f_0^L]$ in Proposition \ref{prop-tlf}. In fact, by Lemma 4.9 of \cite{chaturvedi2023vlasov}, we infer that
    \begin{align*}
        \nn&\nu\left|\mathrm{Re}\,\Big\la  Y_{\bar{\tau},\zeta}^{m}\Gamma(f_0^L, k^\al{\rm h}_{[\iota]})+ Y_{\bar{\tau},\zeta}^{m}\Gamma(k^\al{\rm h}_{[\iota]},f_0^L),\la v\ra^{2\ell'-4\iota-4|\al|}k^\al{\rm h}_{[\iota]}^{(m)} \Big\ra_v\right|\\
        \les\nn&\nu\big|k^\al{\rm h}_{[\iota]}^{(m)}\big|_{\sig,\ell'-2\iota-2|\al|}\sum_{|m'|+|m''|\le|m|}\Big(\big\|\pr_v^{m'}f_0^L\la v\ra^{\ell'}\big\|_{L^2_v}\big|k^\al {\rm h}_{[\iota]}^{(m'')}\big|_{\sig,\ell'-2\iota-2|\al|}\\
        &+|\pr_v^{m'}f_0^L|_{\sig,\ell}\big\|k^\al {\rm h}_{[\iota]}^{(m'')}\la v\ra^{\ell'-2\iota-2|\al|}\big\|_{L^2_v}\Big),
    \end{align*}
as well as 
\begin{align*}
    \nn&\nu\left|\mathrm{Re}\,\Big\la  Y_{\bar{\tau},\zeta}^{m}\nb_v\Gamma(f_0^L,{\rm h}_{[\iota]})+ Y_{\bar{\tau},\zeta}^{m}\nb_v\Gamma({\rm h}_{[\iota]},f_0^L),\la v\ra^{2\ell'-4\iota-4}\nb_v{\rm h}_{[\iota]}^{(m)} \Big\ra_v\right|\\
    \les\nn&\nu\big|\nb_v{\rm h}_{[\iota]}^{(m)}\big|_{\sig,\ell'-2\iota-2}\sum_{|m'|+|m''|\le|m|}\Big(\big\|\pr_v^{m'}f_0^L\la v\ra^{\ell'}\big\|_{L^2_v}\big| \nb_v{\rm h}_{[\iota]}^{(m'')}\big|_{\sig,\ell'-2\iota-2}\\
        \nn&+\sum_{|\al|\le1}\big\|\pr_v^\al\pr_v^{m'}f_0^L\la v\ra^{\ell'-2\iota-2|\al|}\big\|_{L^2_v}\big| {\rm h}_{[\iota]}^{(m'')}\big|_{\sig,\ell'-2\iota}+|\pr_v^{m'}f_0^L|_{\sig,\ell'}\big\| \nb_v{\rm h}_{[\iota]}^{(m'')}\la v\ra^{\ell'-2\iota-2}\big\|_{L^2_v}\\
        &+\sum_{|\al|\le1}|\pr_v^\al\pr_v^{m'}f_0^L|_{\sig,\ell'-2\iota-2|\al|}\big\| {\rm h}_{[\iota]}^{(m'')}\la v\ra^{\ell'-2\iota}\big\|_{L^2_v}\Big),
\end{align*}
and 
\begin{align*}
\nn&\nu\left|\mathrm{Re}\,\Big\la  ik Y_{\bar{\tau},\zeta}^{m}\Gamma(f_0^L,{\rm h}_{[\iota]})+ ik Y_{\bar{\tau},\zeta}^{m}\Gamma({\rm h}_{[\iota]},f_0^L),\la v\ra^{2\ell'-4\iota-4}\nb_v{\rm h}_{[\iota]}^{(m)} \Big\ra_v\right|\\
    \nn&+\nu\left|\mathrm{Re}\,\Big\la  ik{\rm h}_{[\iota]}^{(m)}\la v\ra^{2\ell'-4\iota-4}, Y_{\bar{\tau},\zeta}^{m}\nb_v\Gamma(f_0^L,{\rm h}_{[\iota]})+ Y_{\bar{\tau},\zeta}^{m}\nb_v\Gamma({\rm h}_{[\iota]},f_0^L)\Big\ra_v\right|\\
    \les\nn&\nu\big|\nb_v{\rm h}_{[\iota]}^{(m)}\big|_{\sig,\ell'-2\iota-2}\sum_{|m'|+|m''|\le|m|}\Big(\big\|\pr_v^{m'}f_0^L\la v\ra^{\ell'}\big\|_{L^2_v}\big|k {\rm h}_{[\iota]}^{(m'')}\big|_{\sig,\ell'-2\iota-2}\\
        \nn&+|\pr_v^{m'}f_0^L|_{\sig,\ell'-2\iota}\Big\|k {\rm h}_{[\iota]}^{(m'')}\la v\ra^{\ell'-2\iota-2}\big\|_{L^2_v}\Big)\\
        \nn&+\nu\big|k{\rm h}_{[\iota]}^{(m)}\big|_{\sig,\ell'-2}\sum_{|m'|+|m''|\le|m|}\Big(\big\|\pr_v^{m'}f_0^L\la v\ra^{\ell'}\big\|_{L^2_v}\big| \nb_v{\rm h}_{[\iota]}^{(m'')}\big|_{\sig,\ell'-2\iota-2}\\
        \nn&+\sum_{|\al|\le1}\big\|\pr_v^\al\pr_v^{m'}f_0^L\la v\ra^{\ell'-2\iota-2|\al|}\big\|_{L^2_v}\big| {\rm h}_{[\iota]}^{(m'')}\big|_{\sig,\ell'-2\iota}+|\pr_v^{m'}f_0^L|_{\sig,\ell'-2\iota}\big\| \nb_v{\rm h}_{[\iota]}^{(m'')}\la v\ra^{\ell'-2\iota-2}\big\|_{L^2_v}\\
        &+\sum_{|\al|\le1}|\pr_v^\al\pr_v^{m'}f_0^L|_{\sig,\ell'-2\iota-2|\al|}\big\| {\rm h}_{[\iota]}^{(m'')}\la v\ra^{\ell'-2\iota}\big\|_{L^2_v}\Big).
\end{align*}
We infer from the above three inequalities that
\begin{align}\label{est-PD}
    \nn&{\rm PD}_{k,\iota}^{\bar{\tau},\zeta}({\rm h}_{[\iota]}(t;\tau))\\
    \le\nn&C\nu^{\fr13} {\bf D}_{\bar{\tau},\zeta;\ell'}^{N'}({\rm h}_{[\iota]}(t;\tau))^{\fr12}\Bigg(\sum_{|\al|\le1}\big\|\pr_v^\al f_0^L(t)\big\|_{H^{N'}_{\ell'-2|\al|}}{\bf D}_{\bar{\tau},\zeta;\ell'}^{N'}({\rm h}(t;\tau))^{\fr12}\\
    \nn&+\Big(\nu^{\fr13}\sum_{|\al|\le1}\big\|\pr_{v}^\al f_0^L(t)\big\|_{H^{N'}_{\sig,\ell'-2|\al|}}\Big){\bf E}_{\bar{\tau},\zeta;\ell'}^{N'}({\rm h}_{[\iota]}(t;\tau))^{\fr12}\Bigg)\\
    \le&\Big(C\big\| f_0^L(t)\big\|_{H^{{N'}+1}_{\ell'}}+\fr{1}{32}\Big)\nu^{\fr13} {\bf D}_{\bar{\tau},\zeta;\ell'}^{N'}({\rm h}_{[\iota]}(t;\tau))+C\nu\big\| f_0^L(t)\big\|^2_{H^{N'+1}_{\sig,\ell'}}{\bf E}_{\bar{\tau},\zeta;\ell'}^{N'}({\rm h}_{[\iota]}(t;\tau)).
\end{align}
Substituting this into \eqref{en-ineq-bartau}, for $\big\| f_0^L(t)\big\|_{H^{N'+1}_{\ell'}}$ sufficiently small, we obtain
\begin{align*}
    \fr{d}{dt}\tl{\mathsf{ E}}_{N',\ell'}({\rm h}(t;\tau))+\fr18\nu^{\fr13}\tl{\mathsf{D}}_{N',\ell'}({\rm h}(t;\tau))\le C\nu\big\| f_0^L(t)\big\|^2_{H^{N'+1}_{\sig,\ell'}}\mathsf{ E}_{N',\ell'}({\rm h}(t;\tau)).
\end{align*}
Then by Gr\"onwall's inequality and \eqref{bd-f_0^L}, we get \eqref{est-solution oper}.
\end{proof}

\begin{coro}\label{coro:solution operator}
    Fix $k\in\Z^3_*$ and $\zeta\in\R^3$, ${\rm h}(t,v;\tau)$ be defined in \eqref{def:solution-h}. Then for any $N'\in\N$ with $N'\le\tl N$ and $8<\ell'\le\tl\ell$, there holds
    \begin{align}\label{bd-rmh}
        \left|\int_{\R^3}{\rm h}(t,v;\tau)\sqrt{\mu}(v) dv\right|^2\les \la k(t-\tau)+\zeta\ra^{-2N'}(\nu^{\fr13}\la t-\tau\ra)^{-3}|k|^{-2}\tl{\mathsf{ E}}_{N', \ell'}({\rm h_{in}}(\tau)).
    \end{align}
\end{coro}
\begin{proof}
    By using Proposition 6.1 of \cite{chaturvedi2023vlasov}, one deduces that 
    \begin{align*}
        \left|\int_{\R^3}{\rm h}(t,v;\tau)\sqrt{\mu}(v) dv\right|\les \la k(t-\tau)+\zeta\ra^{-N'}\sum_{|m|\le N'}\|Y^{m}_{\bar{\tau},\zeta}{\rm h}(t,\cdot,\tau)\|_{L^2_v}.
    \end{align*}
    Recalling  the definitions of ${\bf E}_{\bar{\tau},\zeta;\ell'}^{N'}({\rm h}(t;\tau))$ and $\tl{\mathsf{E}}_{N',\ell'}({\rm h}(t;\tau))$ in \eqref{def-ENkl} and \eqref{def-EN-bartau}, respectively, we find that
\begin{align*}
    \sum_{|m|\le N'}\|Y^{m}_{\bar{\tau},\zeta}{\rm h}(t,\cdot,\tau)\|^2_{L^2_v}\les&|k|^{-2}(\nu^{\fr13}\la t-\tau\ra)^{-3}\Big(w_3^2(t-\tau){\bf E}^{N'}_{\bar{\tau},\zeta;\ell'-6}({\rm h}(t;\tau))\Big)\\
    \les&|k|^{-2}(\nu^{\fr13}\la t-\tau\ra)^{-3}\tl{\mathsf{E}}_{N',\ell'}({\rm h}(t;\tau)).
\end{align*}
Then in view of \eqref{est-solution oper},  then \eqref{bd-rmh} follows immediately.   
\end{proof}

\subsection{The estimates of $\mathbb{S}_k(t,\tau)[v\sqrt{\mu}]-S_k(t-\tau)[v\sqrt{\mu}]$.}
For the sake of presentation, let us denote
\[
{ f}_{(1),k}(t,v;\tau)=\mathbb{S}_k(t,\tau)[v\sqrt{\mu}],\quad { f}_{(2),k}(t,v;\tau)={S}_k(t-\tau)[v\sqrt{\mu}],
\]
and
\[
 {\bf f}_k(t,v;\tau)={ f}_{(1),k}(t,v;\tau)-{ f}_{(2),k}(t,v;\tau).
\]
In the following, we omit the  subscript $k$ in ${ f}_{(1),k}(t,v;\tau)$, ${ f}_{(2),k}(t,v;\tau)$ and ${\bf f}_k(t,v;\tau)$ for simplicity.
Then ${ f}_{(1)}(t,v;\tau)$, ${ f}_{(2)}(t,v;\tau)$ and ${\bf f}(t,v;\tau)$ satisfy 
\begin{align}
    \begin{cases}
       \pr_tf_{(1)}+ik\cdot vf_{(1)}+\nu Lf_{(1)}=\nu \Gamma(f_{(1)}, f_0^L)+\nu\Gamma(f_0^L,f_{(1)}),\\
       f_{(1)}(\tau;\tau)=v\sqrt{\mu},
    \end{cases}
\end{align}

\begin{align}
    \begin{cases}
        \pr_tf_{(2)}+ik\cdot vf_{(2)}+\nu Lf_{(2)}=0,\\
       f_{(2)}(\tau;\tau)=v\sqrt{\mu},
    \end{cases}
\end{align}
and
\begin{align}
    \begin{cases}
        \pr_t{\bf f}+ik\cdot v{\bf f}+\nu L{\bf f}=\nu \Gamma(f_{(1)}, f_0^L)+\nu\Gamma(f_0^L,f_{(1)}),\\
       {\bf f}(\tau;\tau)=0,
    \end{cases}
\end{align}
respectively. 

We have the following estimates for ${\bf f}(t,v;\tau)$.
\begin{prop}\label{prop:difference}
For any $N'\in\N$ with $N'\le \tl N$ and $2<\ell'\le\tl\ell$, the following estimate holds
    \begin{align}\label{en-difference}
        \nn&{\bf E}^{N'}_{\bar{\tau},0;\ell'}({\bf f}(t;\tau))+\nu^{\fr13}\int_\tau ^t{\bf D}^{N'}_{\bar{\tau},0;\ell'}({\bf f}(t';\tau))dt'\\
        \les&|k|^2\Big(\sup_{\tau\le t'\le t}\|f_0^L(t')\|^2_{H^{N'+1}_{\ell'}}
        +\nu \int_\tau^t\big\|f_0^L(t)\big\|^2_{H^{N'+1}_{\sig,\ell'}}dt'\Big).
    \end{align}
\end{prop}
\begin{proof}
    Noting that  the error term ${\rm Er}^{\bar{\tau},0}_{k,0}(\cdot)=0$ due to $\iota=0$, now the energy inequality \eqref{en-ineq-bartau} reduces to
   \begin{align}\label{en-ineq-bartau-0}
    \fr12\fr{d}{dt}{\bf E}^{N'}_{\bar{\tau},0;\ell'}({\bf f}(t;\tau))+\fr14\nu^{\fr13}{\bf D}^{N'}_{\bar{\tau},0;\ell'}({\bf f}(t;\tau))\le  {\rm PD}_{k,0;(1)}^{\bar{\tau},0}({\bf f}(t;\tau)),
\end{align}
where
\begin{align*}
&{\mathrm{PD}}_{k,0;(1)}^{\bar{\tau},0}({\bf f}(t;\tau))\\
    =&{\rm A}_0 \sum_{\substack{m\in \N^3\\ |m|\le N'}}\kappa^{2|m|}\sum_{|\al|\le1}\nu\mathrm{Re}\,\Big\la  Y_{\bar{\tau},0}^{m}\Gamma(f_0^L, k^\al f_{(1)})+ Y_{\bar{\tau},0}^{m}\Gamma(k^\al f_{(1)},f_0^L),\la v\ra^{2\ell'-4|\al|}k^\al{\bf f}^{(m)} \Big\ra_v\\
    \nn&+\kappa^2\nu^{\fr23}\sum_{\substack{m\in \N^3\\ |m|\le N'}}\kappa^{2|m|}\nu\mathrm{Re}\,\Big\la  Y_{\bar{\tau},0}^{m}\nb_v\Gamma(f_0^L,f_{(1)})+ Y_{\bar{\tau},0}^{m}\nb_v\Gamma(f_{(1)},f_0^L),\la v\ra^{2\ell'-4}\nb_v{\bf f}^{(m)} \Big\ra_v\\
    \nn&+\kappa\nu^{\fr13}\sum_{\substack{m\in \N^3\\ |m|\le N'}}\kappa^{2|m|}\Bigg[\nu\mathrm{Re}\,\Big\la  ik Y_{\bar{\tau},0}^{m}\Gamma(f_0^L,f_{(1)})+ ik Y_{\bar{\tau},0}^{m}\Gamma(f_{(1)},f_0^L),\la v\ra^{2\ell'-4}\nb_v{\bf f}^{(m)} \Big\ra_v\\
    &+\nu\mathrm{Re}\,\Big\la  ik{\bf f}^{(m)}\la v\ra^{2\ell'-4}, Y_{\bar{\tau},0}^{m}\nb_v\Gamma(f_0^L, f_{(1)})+ Y_{\bar{\tau},0}^{m}\nb_v\Gamma(f_{(1)},f_0^L)\Big\ra_v\Bigg].
\end{align*}
Similar to \eqref{est-PD}, we arrive at
\begin{align*}
    \nn&{\rm PD}_{k,0;(1)}^{\bar{\tau},0}({\bf f}(t;\tau))\\
    \le\nn&C\nu^{\fr13} {\bf D}_{\bar{\tau},0;\ell}^{N'}({\bf f}(t;\tau))^{\fr12}\Bigg(\sum_{|\al|\le1}\big\|\pr_v^\al f_0^L(t)\big\|_{H^{N'}_{\ell-2|\al|}}{\bf D}_{\bar{\tau},0;\ell}^{N'}(f_{(1)}(t;\tau))^{\fr12}\\
    \nn&+\Big(\nu^{\fr13}\sum_{|\al|\le1}\big\|\pr_{v}^\al f_0^L(t)\big\|_{H^{N'}_{\sig,\ell'-2|\al|}}\Big){\bf E}_{\bar{\tau},0;\ell'}^{N'}(f_{(1)}(t;\tau))^{\fr12}\Bigg).
\end{align*}
Substituting this into \eqref{en-ineq-bartau}, we are led to 
\begin{align*}
        \nn&{\bf E}^{N'}_{\bar{\tau},0;\ell'}({\bf f}(t;\tau))+\nu^{\fr13}\int_\tau ^t{\bf D}^{N'}_{\bar{\tau},0;\ell'}({\bf f}(t';\tau))dt'\\
        \les\nn&\nu^{\fr13}\int_\tau^t{\bf D}_{\bar{\tau},0;\ell'}^{N'}(f_{(1)}(t';\tau))dt'\sup_{\tau\le t'\le t}\|f_0^L(t')\|^2_{H^{N'+1}_{\ell'}}\\
        &+\nu \int_\tau^t\big\|f_0^L(t)\big\|^2_{H^{N'+1}_{\sig,\ell'}}dt'\sup_{\tau\le t'\le t}{\bf E}_{\bar{\tau},0;\ell'}^{N'}(f_{(1)}(t';\tau)).
    \end{align*}
 By Proposition \ref{prop:solution-operator}, recalling \eqref{def-ENkl} and \eqref{eq:varE_l^n}, we have
 \begin{align*}
     \sup_{\tau\le t'\le t}{\bf E}_{\bar{\tau},0;\ell'}^{N'}(f_{(1)}(t';\tau))+\nu^{\fr13}\int_\tau^t{\bf D}_{\bar{\tau},0;\ell'}^{N'}(f_{(1)}(t';\tau))dt'\les {\bf E}^{N'}_{\bar{\tau},0;\ell'}(v\sqrt{\mu})\les |k|^2.
 \end{align*}
Then \eqref{en-difference} follows immediately.
\end{proof}

\section{Density estimate in Sobolev spaces}\label{sec: density}
One can see from \eqref{conserv-f_0^L} that $f_0^L$ provides zero contribution to the density $\rho$. Thus, \eqref{exp-tl-rho} holds, and 
 the equation for $\tl{f}$ can be rewritten as follows:
\begin{align}\label{eq:tilde-f}
    \left\{
    \begin{aligned}
        &\partial_{t}\tilde{f}+v\cdot \nabla_x\tilde{f}-2E\cdot v\mu^{\fr12}
        +\nu L \tilde{f}-\nu \Gamma(\tilde{f},f_0^L)-\nu \Gamma(f_0^L,\tilde{f})=\mathbf{N}(t,x,v),\\
        &E=-\nb_x(-\Delta_x)^{-1}\rho=-\nb_x(-\Delta_x)^{-1}\int_{\mathbb{R}^3}\tilde{f}(t,x,v)\sqrt{\mu}(v)dv,\\
        &\tilde{f}(0,x,v)=f_{\ne}(\nu^{-\fr12},x,v)+{\rm P}_0f_0(\nu^{-\fr12},v),
    \end{aligned}
    \right.
\end{align}
where ${\bf N}(t,x,v)$ is defined in \eqref{eq:N}.
Taking Fourier transform in $x$, and denoting $\tl{f}_k(t,v)=\mathcal{F}[\tl{f}]_k(t,v)$, we are led to
\begin{align}\label{eq:tl-f-k}
    \pr_t\tl{f}_k+ik\cdot v\tl{f}_k+\nu L \tilde{f}_k-\nu \Gamma(\tilde{f}_k,f_0^L)-\nu \Gamma(f_0^L,\tilde{f}_k)=-\fr{2ik\cdot v \mu^{\fr12}}{|k|^2}\hat{\rho}_k+\hat{{\bf N}}_k(t,v),
\end{align}
with $\hat{\rho}_k(t)=\int_{\R^3}\tl{f}_k(t,x,v)\sqrt{\mu}dv$.

Then by   Duhamel's principle,  the solution $\tl{f}_k(t,v)$ to \eqref{eq:tl-f-k} takes the form of 
\begin{align}\label{exp-tl-f-k}
\tl f_k(t,v)
=\nn&\mathbb{S}_{k}(t,0)[\tl{f}_k(0,v)]-\int_0^t\fr{2ik}{|k|^2}\hat{\rho}_k(\tau)\cdot\mathbb{S}_{k}(t,\tau)\big[v\sqrt{\mu}\big]d\tau\\
&+\int_0^t\mathbb{S}_{k}(t,\tau)\big[\hat{{\bf N}}_k(\tau,v)\big]d\tau.
\end{align}
Consequently,
\begin{align}\label{exp-rho-k}
\hat{\rho}_k(t)=\nn&\int_{\R^3}\mathbb{S}_k(t,0)\big[\tl f_{k}(0,v)\big]\sqrt{\mu}(v)dv+\int_0^t\int_{\R^3}\mathbb{S}_k(t,\tau)\big[\hat{\mathbf{N}}_k(\tau,v)\big]\sqrt{\mu}(v)dvd\tau\\
&-\int_0^t\int_{\R^3}\fr{2ik  }{|k|^2}\hat{\rho}_k(\tau)\cdot\mathbb{S}_k(t,\tau)\big[v\sqrt{\mu}(v)\big]\sqrt{\mu}(v)dvd\tau.
\end{align}
Let us denote
\begin{align}\label{def-tlNk1}
    \tl{\mathcal{N}}^{(1)}_k(t)=\int_{\R^3}\mathbb{S}_k(t,0)\big[\tl f_k(0,v)\big]\sqrt{\mu}(v)dv+\int_0^t\int_{\R^3}\mathbb{S}_k(t,\tau)\big[\hat{\mathbf{N}}_k(\tau,v)\big]\sqrt{\mu}(v)dvd\tau
\end{align}
and
\begin{align}\label{def-tlNk2}
   \tl{\mathcal{N}}^{(2)}_k(t)= -\int_0^t\int_{\R^3}\fr{2ik  }{|k|^2}\hat{\rho}_k(\tau)\cdot\Big(\mathbb{S}_k(t,\tau)\big[v\sqrt{\mu}(v)\big]-S_k(t-\tau)\big[v\sqrt{\mu}(v)\big]\Big)\sqrt{\mu}(v)dvd\tau.
\end{align}
Then similar to \eqref{rho_k}, \eqref{exp-rho-k} can be rewritten as
\begin{align}\label{exp-rho-k-1}
    \hat{\rho}_k(t)+\int_0^tK_k(t-\tau)\hat{\rho}_k(\tau)d\tau=\tl{\mathcal{N}}^{(1)}_k(t)+\tl{\mathcal{N}}^{(2)}_k(t).
\end{align}
By Proposition \ref{prop:kernel G}, we find that
\begin{align}\label{rep-rho_k}
    \hat{\rho}_k(t)=\tl{\mathcal{N}}^{(1)}_k(t)+\tl{\mathcal{N}}^{(2)}_k(t)+\int_0^tG_k(t-\tau)\left(\tl{\mathcal{N}}^{(1)}_k(\tau)+\tl{\mathcal{N}}^{(2)}_k(\tau)\right)d\tau.
\end{align}
Thanks to \eqref{est-G_k}, similar to Lemma \ref{lem-rho:N}, one easily deduces the following bounds for $\rho$.
\begin{lem}\label{lem:rho-longtime}
    The following estimates for $\rho$ hold:
    \begin{align}\label{gbd-rho-Linfty}
        \|\rho\|^2_{L^\infty_t \tl{H}^{\tl N}_x}\les \sum_{i=1}^2\sum_{k\in\Z^3_*}\sum_{\beta\in\N^6,|\beta|\le \tl N}\kappa^{2|\beta|}\|(k,kt)^{\beta}\tl{\mathcal{N}}^{(i)}_k \|^2_{L^\infty_t},
    \end{align}
    and
    \begin{align}\label{gbd-rho-L2}
        \|\rho\|^2_{L^2_t \tl{H}^{\tl {N}}_x}\les \sum_{i=1}^2\sum_{k\in\Z^3_*}\sum_{\beta\in\N^6,|\beta|\le \tl N}\kappa^{2|\beta|}\|(k,kt)^{\beta}\tl{\mathcal{N}}^{(i)}_k \|^2_{L^2_t},
    \end{align}
    where, by \eqref{def-tlHN},
    \begin{align*}
    \|\rho(t)\|_{\tl{H}^{\tl N}_x}^2=\sum_{\beta\in\N^6,|\beta|\le \tl N}\kappa^{2|\beta|}\|Z^\beta\rho(t)\|^2_{L^2(\T^3)}=\sum_{k\in\Z^3_*}\sum_{\beta\in\N^6,|\beta|\le \tl N}\kappa^{2|\beta|}\left|(k,kt)^{\beta}\hat{\rho}_k(t)\right|^2.
\end{align*}
\end{lem}

Next we show that the contributions from $\tl{\mathcal{N}}_k^{(2)}$ in \eqref{gbd-rho-Linfty} and \eqref{gbd-rho-L2} can be absorbed by the left hand side. More precisely, we have the following proposition.
\begin{prop}\label{prop:perturb-rho}
The following estimates hold:
    \begin{align}\label{Nk2-L2-small}
     \nn&\sum_{k\in\Z^3_*}\sum_{\beta\in\N^6,|\beta|\le \tl N}\kappa^{2|\beta|}\Big(\|(k,kt)^{\beta}\tl{\mathcal{N}}^{(2)}_k \|^2_{L^\infty_t}+\nu^{\fr13}\|(k,kt)^{\beta}\tl{\mathcal{N}}^{(2)}_k \|^2_{L^2_t}\Big)\\
     \les& C\eps^2\Big(\|\rho\|^2_{L^\infty_t\tl{H}^{\tl{N}}_x}+\nu^{\fr13}\|\rho\|^2_{L^2_t\tl H^{\tl{N}}_x}\Big).
    \end{align}
\end{prop}
\begin{proof}
We only focus on the $L^2_t$ estimate, since the $L^\infty_t$ bound can be obtained in the same manner. The key point is that one can use the Landau damping of ${\bf f}_k(t,v;\tau)=\mathbb{S}_k(t,\tau)\big[v\sqrt{\mu}(v)\big]-S_k(t-\tau)\big[v\sqrt{\mu}(v)\big]$ to gain time decay because ${\bf f}_k$ has extra regularity. More precisely, note first that 
\begin{align*}
  1\les\fr{\la \eta+k(t-\tau)\ra \la\eta\ra}{\la k(t-\tau)\ra}.
\end{align*}
Then by Plancherel's theorem, we have
\begin{align}\label{est-rho-difference}
    \nn& \sum_{k\in\Z^3_*}\sum_{\beta\in\N^6,|\beta|\le \tl N}\kappa^{2|\beta|}\|(k,kt)^{\beta}\tl{\mathcal{N}}^{(2)}_k \|^2_{L^2_t}\\
    =\nn&\sum_{k\in\Z^3_*}\sum_{\substack{\beta\in\N^6\\|\beta|\le \tl N}}\kappa^{2|\beta|}\int_0^{T^*}\Bigg|\sum_{\substack{\beta'\le\beta\\ \beta''\le\beta'}}C_{\beta}^{\beta'}C_{\beta'}^{\beta''}\int_0^t\int_{\R^3}(k,k\tau)^{\beta-\beta'}\fr{2ik}{|k|^2}\hat{\rho}_k(\tau)\\
    \nn&\cdot (0,k(t-\tau)+\eta)^{\beta''}\mathcal{F}_v[{\bf f}_k](t,\eta;\tau)(0,-\eta)^{\beta'-\beta''}\overline{\mathcal{F}_v[\sqrt{\mu}]}(\eta)d\eta d\tau \Bigg|^2dt\\
    \les\nn&\sum_{k\in\Z^3_*}\sum_{\substack{\beta\in\N^6\\|\beta|\le \tl N}}\kappa^{2|\beta|}\int_0^{T^*}\Bigg[\sum_{\substack{\beta'\le\beta\\ \beta''\le\beta'}}C_{\beta}^{\beta'}C_{\beta'}^{\beta''}\int_0^t\left|(k,k\tau)^{\beta-\beta'}\fr{2ik}{|k|^2}\hat{\rho}_k(\tau)\right|\fr{1}{\la k(t-\tau)\ra^2}\\
    \nn&\times \left\|Y_{\bar{\tau},0}^{\bar{\beta''}}\la Y_{\bar{\tau},0}\ra^2{\bf f}_k(t;\tau)\right\|_{L^2_v}\Big\|\pr_v^{\bar{\beta'}-\bar{\beta''}}\la\pr_v\ra^2\sqrt{\mu}\Big\|_{L^2_v} d\tau \Bigg]^2dt\\
    \les\nn&\sum_{k\in\Z^3_*}\sum_{\substack{\beta\in\N^6\\|\beta|\le \tl N}}\sum_{\substack{\beta'\le\beta}}\int_0^{T^*}\int_0^t\kappa^{2(|\beta|-|\beta'|)}\left|(k,k\tau)^{\beta-\beta'}\fr{2ik}{|k|^2}\hat{\rho}_k(\tau)\right|^2\fr{1}{\la k(t-\tau)\ra^2} d\tau\\
    \nn&\times   \int_0^t\fr{d\tau}{\la k(t-\tau)\ra^2}dt\sup_{k\in\Z^3_*}\sup_{0\le\tau\le t\le T^*}\sum_{\substack{\al\in\N^3,|\al|\le \tl N}}\kappa^{2|\al|}\left\|Y_{\bar{\tau},0}^{\al}\la Y_{\bar{\tau},0}\ra^2{\bf f}_k(t;\tau)\right\|_{L^2_v}^2\\
    \les&\sup_{k\in\Z^3_*}\sup_{0\le\tau\le t\le T^*}\sum_{\substack{\al\in\N^3,|\al|\le \tl N}}\kappa^{2|\al|}\left\|Y_{\bar{\tau},0}^{\al}\la Y_{\bar{\tau},0}\ra^2{\bf f}_k(t;\tau)\right\|_{L^2_v}^2\|\rho\|^2_{L^2_t\tl{H}_x^{\tl N}}.
\end{align}
Clearly, for any $k\in\Z^3_*$, it holds that
\begin{align*}
    \sum_{\substack{\al\in\N^3,|\al|\le N_1}}\kappa^{2|\al|}\left\|Y_{\bar{\tau},0}^{\al}\la Y_{\bar{\tau},0}\ra^2{\bf f}_k(t;\tau)\right\|_{L^2_v}^2\les \fr{1}{|k|^2}{\bf E}^{N_1+2}_{\bar{\tau},0;2}({\bf f}_k(t;\tau)).
\end{align*}
Then by Proposition \ref{prop:difference} and the smallness of $f_0^L$, we obtain \eqref{Nk2-L2-small}.
\end{proof}

\begin{prop}\label{prop:est-tlNk1}
The following estimates hold:
    \begin{align}\label{bd-tlNk1}
     \nn&\sum_{k\in\Z^3_*}\sum_{\beta\in\N^6,|\beta|\le \tl N}\kappa^{2|\beta|}\Big(\|(k,kt)^{\beta}\tl{\mathcal{N}}^{(1)}_k \|^2_{L^\infty_t}+\nu^{\fr13}\|(k,kt)^{\beta}\tl{\mathcal{N}}^{(1)}_k \|^2_{L^2_t}\Big)\\
     \les&\eps^2\nu^{2\mathfrak{b}}+\eps^2\Big(\|\rho(t)\|^2_{L^\infty_t\tl{H}^{\tl{N}}_x}+\nu^{\fr13}\|\rho(t)\|^2_{L^2_t\tl H^{\tl{N}}_x}\Big).
    \end{align}
\end{prop}
\begin{proof}
Recalling the definitions of $\tl{\mathcal{N}}^{(1)}_k(t)$ and ${\bf N}(t,x,v)$ in \eqref{def-tlNk1} and \eqref{eq:N}, respectively, one can classify the terms in \eqref{def-tlNk1} into two sets:
\begin{itemize}
    \item terms involving $f_0^L$: $\displaystyle\int_0^t \int_{\R^3} \mathbb{S}_k(t,\tau)\big[-\widehat{(E\cdot\nb_vf_0^L)}_k(\tau,v)+\widehat{(E\cdot vf_0^L)}_k(\tau,v)\big]\sqrt{\mu}(v)dvd\tau$.
    \item terms not involving $f_0^L$: 
    \begin{align*}
        &\int_{\R^3}\mathbb{S}_k(t,0)[\tl{f}_k(0,v)]\sqrt{\mu}(v)dv\\
        +&\int_0^t\int_{\R^3}\mathbb{S}_k(t,\tau)\big[-(\widehat{E\cdot\nb_v\tl{f}})_k(\tau,v)+(\widehat{E\cdot v\tl{f}})_k(\tau,v)\big]\sqrt{\mu}(v)dvd\tau\\
        +&\nu\int_0^t\int_{\R^3}\mathbb{S}_k(t,\tau)\big[(\widehat{\Gamma(\tl{f},\tl{f})})_k(\tau,v)\big]\sqrt{\mu}(v)dvd\tau.
    \end{align*}
\end{itemize}
Now we bound the terms involving $f_0^L$.
Similar to \eqref{est-rho-difference}, we find that
\begin{align*}
    \nn&\sum_{k\in\Z^3_*}\sum_{\substack{\beta\in\N^6\\|\beta|\le \tl N}}\kappa^{2|\beta|}\int_0^{T^*}\Bigg|\int_0^t\int_{\R^3}(k,kt)^{\beta}\fr{ik}{|k|^{2}}\hat{\rho}_k(\tau)\cdot\mathbb{S}_k(t,\tau)\big[\nb_vf^L_0(\tau,v)\big]\sqrt{\mu}(v)dvd\tau\Bigg|^2dt\\
    \les\nn&\sum_{k\in\Z^3_*}\sum_{\substack{\beta\in\N^6\\|\beta|\le \tl N}}\kappa^{2|\beta|}\int_0^{T^*}\Bigg[\sum_{\substack{\beta'\le\beta\\ \beta''\le\beta'}}\int_0^t\left|(k,k\tau)^{\beta-\beta'}\fr{ik}{|k|^2}\hat{\rho}_k(\tau)\right|\fr{1}{\la k(t-\tau)\ra^2}\\
    \nn&\times \left\|Y_{\bar{\tau},0}^{\bar{\beta''}}\la Y_{\bar{\tau},0}\ra^2\mathbb{S}_k(t,\tau)\big[\nb_vf^L_0(\tau,v)\big]\right\|_{L^2_v}\Big\|\pr_v^{\bar{\beta'}-\bar{\beta''}}\la\nb_v\ra^2\sqrt{\mu}\Big\|_{L^2_v} d\tau \Bigg]^2dt\\
    \les&\sup_{k\in\Z^3_*}\sup_{0\le\tau\le t\le T^*}\sum_{\substack{\al\in\N^3,|\al|\le \tl N}}\kappa^{2|\al|}\left\|Y_{\bar{\tau},0}^{\al}\la Y_{\bar{\tau},0}\ra^2\mathbb{S}_k(t,\tau)\big[\nb_vf^L_0(\tau,v)\big]\right\|_{L^2_v}^2\|\rho\|^2_{L^2_t\tl{H}^{\tl{N}}_x},
\end{align*}
and
\begin{align*}
    \nn&\sum_{k\in\Z^3_*}\sum_{\substack{\beta\in\N^6\\|\beta|\le \tl N}}\kappa^{2|\beta|}\int_0^{T^*}\Bigg|\int_0^t\int_{\R^3}(k,kt)^{\beta}\fr{ik}{|k|^{2}}\hat{\rho}_k(\tau)\cdot\mathbb{S}_k(t,\tau)\big[vf^L_0(\tau,v)\big]\sqrt{\mu}(v)dvd\tau\Bigg|^2dt\\
    \les&\sup_{k\in\Z^3_*}\sup_{0\le\tau\le t\le T^*}\sum_{\substack{\al\in\N^3,|\al|\le \tl N}}\kappa^{2|\al|}\left\|Y_{\bar{\tau},0}^{\al}\la Y_{\bar{\tau},0}\ra^2\mathbb{S}_k(t,\tau)\big[vf^L_0(\tau,v)\big]\right\|_{L^2_v}^2\|\rho\|^2_{L^2_t\tl{H}^{\tl N}_x}.
\end{align*}
By Propositions \ref{prop:solution-operator} and \ref{prop-f_0^L}, for any $k\in\Z^3_*$, there holds
\begin{align}
    \nn&\sum_{\substack{\al\in\N^3,|\al|\le \tl N}}\kappa^{2|\al|}\left\|Y_{\bar{\tau},0}^{\al}\la Y_{\bar{\tau},0}\ra^2\mathbb{S}_k(t,\tau)\big[\nb_vf^L_0(\tau,v)\big]\right\|_{L^2_v}^2\\
    \les\nn&\fr{1}{|k|^2}{\bf E}^{\tl N+2}_{\bar{\tau},0;2}\Big(\mathbb{S}_k(t,\tau)\big[\nb_vf^L_0(\tau,v)\big]\Big)\les \fr{1}{|k|^2}{\bf E}^{\tl N+2}_{\bar{\tau},0;2}\big(\nb_vf^L_0(\tau,v)\big)\les \|f_0^L(\tau)\|_{H^{\tl N+{ 4}}_2}^2\les \eps^2,
\end{align}
and
\begin{align}
    \nn&\sum_{\substack{\al\in\N^3,|\al|\le \tl N}}\kappa^{2|\al|}\left\|Y_{\bar{\tau},0}^{\al}\la Y_{\bar{\tau},0}\ra^2\mathbb{S}_k(t,\tau)\big[vf^L_0(\tau,v)\big]\right\|_{L^2_v}^2\\
    \les\nn&\fr{1}{|k|^2}{\bf E}^{\tl N+2}_{\bar{\tau},0;2}\Big(\mathbb{S}_k(t,\tau)\big[vf^L_0(\tau,v)\big]\Big)\les \fr{1}{|k|^2}{\bf E}^{\tl N+2}_{\bar{\tau},0;2}\big(vf^L_0(\tau,v)\big)\les \|f_0^L(\tau)\|_{H^{\tl N+{ 3}}_{ 3}}^2\les \eps^2.
\end{align}

Clearly, the above estimates apply to the case when the time variable is taken $L^\infty_t$ norm. 

For the rest terms not involving $f_0^L$, thanks to Proposition \ref{prop:solution-operator} and Corollary \ref{coro:solution operator}, one can treat them in the same way as one treats $\mathbb{I}_k(t), \mathbb{II}_k(t)$ and $\mathbb{III}_k(t)$ in \cite{chaturvedi2023vlasov}, see Proposition 8.2 of \cite{chaturvedi2023vlasov} for more details. We thus complete the proof Proposition \ref{prop:est-tlNk1}.
\end{proof}

Finally, we conclude that the bound in the bootstarp hypothesis \eqref{H-tlrho} can be improved by combining the estimates in Lemma \ref{lem:rho-longtime}, and Propositions \ref{prop:perturb-rho}, \ref{prop:est-tlNk1}.

\appendix

\section{Basic properties of the Gevrey class and Landau collision operators}
\subsection{Elementary properties of $\sig^{ij}$ and $\sig^i$}\label{sec: Linear Landau operator}
In this section, we study the Gevrey regularity of the coefficients $\sig^{ij}$ and $\sig^{i}$ that appear in the linear Landau operator $L$. We refer to \cite{MR2425602, MR4147427} for the Gevrey regularity estimates of the coefficients where the Coulomb potential is replaced by the soft potentials, namely, $\Phi^{ij}$ is replaced by $\Phi^{ij}_{\gamma}=|z|^{\gamma+2}(\delta_{ij}-\frac{z_iz_j}{|z|^2})$ with $\gamma\in [-1,1]$. 

For any vector-valued function ${\bf g}(v)=(g_1(v), g_2(v), g_3(v))$, we define the projection to the vector $v$ as
\begin{align}\label{eq: P_v definition}
    P_v{\bf g}=\fr{v}{|v|}\fr{v}{|v|}\cdot{\bf g},\quad {\rm i.e.,}\quad (P_v{\bf g})_i=\fr{v_i}{|v|}\sum_{j=1}^3\fr{v_j}{|v|}{g}_j.
\end{align}

\begin{lem}\label{lem:sigma-lam}
Let $\sig^{ij}(v)$ be given in \eqref{def-sig}. Then $\sig^{ij}(v)$ can be expressed as
\be\label{exp-sig}
\sig^{ij}(v)=\lm_1(v)\fr{v_iv_j}{|v|^2}+\lm_2(v)\left(\delta_{ij}-\fr{v_iv_j}{|v|^2}\right),
\ee
where $\lm_1(v)>0$ is a simple eigenvalue associated with the vector $v$ and $\lm_2(v)$ is a double eigenvalue associated with $v^\bot$. Moreover, there are constants $c_1$ and $c_2>0$ such that as $|v|\rightarrow\infty$, we have
\be
\lm_1(v)\rightarrow c_1(1+|v|)^{-3},\quad \lm_2(v)\rightarrow c_2(1+|v|)^{-1}.
\ee
\end{lem}

\begin{coro}\label{coro-sig}
In view of \eqref{exp-sig}, there hold
\be\label{positive-sig1}
\sig^{ij}(v)g_ig_j=\lm_1(v)|P_v{\bf g}|^2+\lm_2(v)|(I-P_v){\bf g}|^2,
\ee
and
\be\label{positive-sig2}
\sig^{ij}v_iv_j=\lm_1(v)|v|^2.
\ee
There is $c>0$ such that 
\begin{align}\label{coercive}
    |h|_{\sigma, \ell}\geq c \Big(\|\langle v\rangle^{\ell-\fr12}h\|_{L^2_v}+\|\langle v\rangle^{\ell-\fr32}\nabla_vh\|_{L^2_v}\Big). 
\end{align}
\end{coro}

Let $\psi$ be a cutoff function in the Gevrey-${\fr{1}{s_0}}$ class, such that $\psi(v)\equiv1$ for $|v|\le1/2$ and $\psi(v)\equiv0$ for $|v|\ge3/4$. Set $\chi(v)=\psi(\fr{v}{2})-\psi(v)$. Then ${\rm supp\, }\chi\subset\{\fr{1}{2}\le|v|\le\fr{3}{2}\}$, and we have the
partition of unity
\[
1=\psi(v)+\sum_{j=0}^\infty\chi(\fr{v}{2^j}).
\]
Let us denote $\chi_j(v)=\chi(\fr{v}{2^j})$, for  $j=0, 1, 2, \cdots$,  $\chi_{-1}(v)=\psi(v)$ and $\chi_j(v)=0$ for  $j=-2, -3, -4, \cdots$. Then the above partition of unity can be rewritten as
\[
1=\sum_{j\in\Z}^\infty\chi(\fr{v}{2^j}).
\]
Moreover, since $\psi(v)$ and $\chi(v)$ are in the Gevrey-${\fr{1}{s_0}}$ class, there are two positive constants $M_\psi$ and $M_\chi$, such that for any multi-index $\beta\in\N^3$, there hold
\begin{align}\label{G-s}
|\pr_v^\beta\psi(v)|\les M_\psi^{|\beta|}\Gamma_{s_0}(\beta),\quad
|\pr_v^\beta\chi(v)|\les M_\chi^{|\beta|}\Gamma_{s_0}(\beta).
\end{align}

Recall that for $k\in\N$, we denote 
\[
\Gamma_s(k)=2^{-25}(k!)^{\fr{1}{s}}(k+1)^{-12}.
\]
Then the following inequality holds.
\be\label{bound-product-Gamma}
\sum_{j=0}^k\fr{k!}{j!(k-j)!}\Gamma_s(j)\Gamma_s(k-j)< \Gamma_s(k), \quad {\rm for}\quad 0<s\le1.
\ee

\begin{lem}\label{lem: sig_ij}
Let $0<c'<c$. There are $M_{\mu}$ and $M_{\sig}$, such that for any $\beta\in \mathbb{N}^3$, there hold
\begin{align}
\label{e-prmu}&\left|\pr_v^\beta e^{-c|v|^2}\right|\le M_{\mu}^{|\beta|}\Gamma_1(\beta) e^{-c'|v|^2},\\
\label{eq: est sigma_ij}&|\pr_v^\beta\sig^{ij}(v)|+|\pr_v^\beta\sig^i(v)|+|\pr_v^{\beta}(v_iv_j\sig^{ij}(v))|\le M_{\sigma}^{|\beta|}\Gamma_s(\beta) (1+|v|)^{-1-|\beta|},
\end{align}
for any $0<s<\fr23$.
\end{lem}
\begin{proof}
For any given $c>0$ and multi-index $\al=(\al_1, \al_2, \al_3)$, by using the Fa\`{a} di Bruno's formula, we have
\begin{align}
\nn&\pr^\al e^{-c|v|^2}=\pr^{\al_1}\pr^{\al_2}\pr^{\al_3} e^{-c(v_1^2+v_2^2+v_3^2)}\\
\nn=&\pr^{\al_1}\pr^{\al_2}\sum_{m_{13}+2m_{23}=\al_3}\fr{\al_3!}{m_{13}!m_{23}!2^{m_{23}}}e^{-c(v_1^2+v_2^2+v_3^2)}(-2cv_3)^{m_{13}}(-2c)^{m_{23}}\\
\nn=&\prod_{1\le k\le 3}\sum_{m_{1k}+2m_{2k}} 2^{m_{1k}}(-c)^{m_{1k}+m_{2k}}\fr{\al_k!}{m_{1k}!m_{2k}!} v_{k}^{m_{1k}}e^{-cv_k^2}.
\end{align}
For any fixed $k\in\{1, 2, 3\}$, thanks to the elementary inequality $\fr{x^n}{e^{x^2}}\le \left(\fr{n}{2e}\right)^{n/2}$ for any fixed $n>0$ and all $x\ge0$ and \eqref{bound-product-Gamma},  we  have
\begin{align*}
    &\Bigg|\sum_{m_{1k}+2m_{2k}=\al_k}\fr{\al_k!}{m_{1k}!m_{2k}!}e^{-cv_k^2}v_k^{m_{1k}}\Bigg|\\
    \le&\sum_{m_{1k}+2m_{2k}=\al_k}\fr{\al_k!}{m_{1k}!(2m_{2k})!}\fr{(2m_{2k})!}{m_{2k}!} \left(e^{-(c-c')v_k^2}|v_k|^{m_{1k}}\right) e^{-c'v_k^2}\\
\le&C\sum_{m_{1k}+2m_{2k}=\al_k}\fr{\al_k!}{m_{1k}!(2m_{2k})!}\Gamma_1(2m_{2k})\fr{(m_{2k}+1)^{12}}{m_{2k}!} (c-c')^{-\fr{m_{1k}}{2}}\left(\fr{m_{1_k}}{2e}\right)^{\fr{m_{1k}}{2}} e^{-c'v_k^2}\\
    \le&C(c-c')^{-\al_k}\sum_{m_{1k}+2m_{2k}=\al_k}\fr{\al_k!}{m_{1k}!(2m_{2k})!}\Gamma_1(2m_{2k})\Gamma_1(m_{1k})e^{-c'v_k^2}\\
    \le& C(c-c')^{-\al_k}\Gamma_1(\al_k)e^{-c'v_k^2},
\end{align*}
where we have used the inequality
\[
\left(\fr{m_{1k}}{2e}\right)^{\fr{m_{1k}}{2}}\le C \left(\left\lfloor\fr{m_{1k}}{2}\right\rfloor+1\right)!\le C \Gamma_1 (m_{1k}),
\]
by Stirling's formula.
Thus we get 
\begin{align*}
  \Big|\pr_v^\al e^{-c|v|^2}\Big|&\leq C (2c+1)^{\alpha_1+\alpha_2+\alpha_3} \prod_{1\le k\le3}\sum_{\substack{m_{1k}+2m_{2k}=\al_k,\\1\le k\le3}}\fr{\al_k!}{m_{1k}!m_{2k}!}e^{-cv_k^2}|v_k|^{m_{1k}}\\
&\leq C\big[(2c+1)(c-c')^{-1}\big]^{\alpha_1+\alpha_2+\alpha_3}\Gamma_1(\alpha_{1})\Gamma_1(\alpha_{2})\Gamma_1(\alpha_{3})e^{-c'|v|^2}\\
&\lesssim C^{|\alpha|}\Gamma_1(\alpha)e^{-c'|v|^2}.
\end{align*}

Now we turn to prove \eqref{eq: est sigma_ij}.
Note that $\pr_v^\beta\chi_l(v)=\begin{cases}\fr{1}{2^{l|\beta|}}\pr_v^\beta\chi(\fr{v}{2^l}),\quad l=0, 1, 2,\cdots,\\
\pr_v^\beta\psi(v), \quad\quad\ \  \ l=-1.\end{cases}$ 
Then in view of \eqref{G-s} and \eqref{e-prmu}, and using \eqref{bound-product-Gamma}, we find that there exists $\tl{M}_\mu\ge M_\mu$ independent of $\beta$, such that
\begin{align}\label{pr_chi_lmu}
    |\pr_v^\beta (\chi_l \mu)(v)|\le\nn&\sum_{\beta'\le\beta}C_\beta^{\beta'}|\pr_v^{\beta'}\chi_l(v)| |\pr_v^{\beta-\beta'}\mu(v)| \\
    \nn\lesssim&\sum_{\beta'\le\beta} C_\beta^{\beta'} \max\left\{M_\psi,M_\chi\right\}^{|\beta'|} \Gamma_{s_0}(\beta')M_\mu^{|\beta-\beta'|}\Gamma_1(\beta-\beta')e^{-\fr34|v|^2}\\
    \les& \max\left\{M_\psi,M_\chi,M_\mu\right\}^{|\beta|}\Gamma_{s_0}(\beta)e^{-\fr34|v|^2}
    \les \tl{M}_\mu^{|\beta|}\Gamma_{s_0}(\beta)e^{-\fr34|v|^2},\quad {\rm for \ \ all }\quad l\in\Z.
\end{align}
It is easy to check that for $\frac{1}{3}\leq |v|\leq 2$, $\Phi^{ij}$ is analytic in $v$, namely, there is a constant $C>0$, such that for any $\beta\in \mathbb{N}^3$, it holds that
\begin{align*}
    |\pr_v^{\beta}\Phi^{ij}(v)|\lesssim C_{\Phi^{ij}}^{|\beta|}\Gamma_1(\beta).
\end{align*}
Then similar to \eqref{pr_chi_lmu}, we have
\begin{align*}
    |\pr_v^{\beta}(\chi\Phi^{ij})(v)|\les&\sum_{\beta'\le\beta}C_\beta^{\beta'}|\pr_v^{\beta'}\chi(v)| |\pr_v^{\beta-\beta'}\Phi^{ij}(v)| \\
    \lesssim&\sum_{\beta'\le\beta} C_\beta^{\beta'} M_\chi^{|\beta'|} \Gamma_{s_0}(\beta')C_{\Phi^{ij}}^{|\beta-\beta'|}\Gamma_1(\beta-\beta')\les \max\left\{M_\chi, C_{\Phi^{ij}}\right\}^{|\beta|}\Gamma_{s_0}(\beta).
\end{align*}
Owing to  the fact that $(\chi_l\Phi^{ij})(v)=\chi(\fr{v}{2^l})\Phi^{ij}(v)=\chi(\fr{v}{2^l})\Phi^{ij}(\fr{v}{2^l})\fr{1}{2^l}$, we get that for $2^{l-1}\leq |v|\leq \fr322^l$,
\be\label{e-prPhi}
\begin{aligned}
    |\pr_v^{\beta}(\chi_l\Phi^{ij})(v)|\lesssim &\fr{1}{2^{l|\beta|+l}}\sup_{\fr13\leq |w|\leq 2}|\pr_v^{\beta}(\chi\Phi^{ij})(w)|\\
    \lesssim &M_{\Phi}^{|\beta|}\Gamma_{s_0}(\beta)|v|^{-1-|\beta|},\quad l=0, 1, 2, 3, \cdots.
\end{aligned}
\ee

We denote $f_k(v)=f(v)\chi_k(v)$ and $f_{<k}(v)=\sum_{k'<k}f_{k'}(v)=\psi(\fr{v}{2^{k'}})f(v)$.
 Then
\begin{align}
\sig^{ij}(v)=&[\Phi^{ij}*\mu](v)\nn=\sum_{k,k'\in\Z}\int_{\R^3}\Phi^{ij}_k(v-v')\mu_{k'}(v')dv'\\
\nn=&\sum_{k\ge3}\int_{\R^3}\Phi^{ij}_{k}(v-v')\mu_{<k-3}(v')dv'+\sum_{k\ge -4}\int_{\R^3}\Phi^{ij}_{<k+4}(v-v')\mu_k(v')dv'\\
=&{\rm T}_{\mu}\Phi^{ij}+{\rm T}'_{\Phi^{ij}}\mu.
\end{align}
On the support of the integrand of ${\rm T}_{\mu}\Phi^{ij}$, it holds that
\[
|v'|\le \fr{3}{16}|v-v'|, \quad{\rm and\  hence}\quad \fr{13}{16} |v-v'|\le |v|\le \fr{19}{16}|v-v'|,
\]
and
\begin{align*}
    \pr_v^\beta{\rm T}_{\mu}\Phi^{ij}=&\sum_{k\ge 3}\int_{\R^3}\pr_v^\beta\Phi^{ij}_{k}(v-v')\mu_{<k-3}(v')dv'\\
    =&\sum_{k\ge 3}\int_{\R^3}\sum_{l\in\Z}\chi_l(v-v')\pr_v^\beta\Phi^{ij}_{k}(v-v')\mu_{<k-3}(v')dv'\\
    =&\sum_{k\ge 3}\sum_{|l-k|\le1}\int_{\R^3}\chi_l(v-v')\pr_v^\beta\Phi^{ij}_{k}(v-v')\mu_{<k-3}(v')dv'.
\end{align*}
Thus, combining this with \eqref{e-prPhi}, we find that
\begin{align}
\big| \pr_v^{\beta}{\rm T}_{\mu}\Phi^{ij}\big|\nn\leq& \int_{\R^3}\sum_{k\ge 3}\sum_{|l-k|\le1}\chi_l(v-v')\big|\pr_v^\beta\Phi^{ij}_{k}(v-v') \big|\mu(v')dv'\\
\nn\leq&M_{\Phi}^{|\beta|}\Gamma_{s_0}(\beta)|v|^{-1-|\beta|}\int_{\R^3}\sum_{k\in\Z}\sum_{l\in\Z}\chi_l(v-v'){\bf 1}_{[-1,1]}(k-l) \mu(v')dv'\\
\les&M_{\Phi}^{|\beta|}\Gamma_{s_0}(\beta)|v|^{-1-|\beta|},
\end{align}
where we have used the Young's inequality for convolution to get
\[
\sum_{k\in\Z}\sum_{l\in\Z}\chi_l(v-v'){\bf 1}_{[-1,1]}(k-l)\le \sum_{k\in\Z}{\bf 1}_{[-1,1]}(k)\sum_{l\in\Z}\chi_l(v-v')\les 1,
\]
due to the fact $\sum_{m\in\Z}\chi_m(\cdot)=1$.

On the support of the integrand of ${\rm T}'_{\Phi^{ij}}\mu$, it holds that
\[
|v-v'|\le 24|v'|, \quad{\rm and\  hence}\quad |v|\le 25|v'|.
\]
Combining this with \eqref{pr_chi_lmu} and $|v|^{|\beta|}e^{-c|v|^2}\leq C^{|\beta|}\sqrt{|\beta|!}$, by taking $s_0$ such that $\fr{1}{s}=\fr{1}{s_0}+\fr{1}{2}$, yields
\begin{align*}
    \left|\pr_v^\beta{\rm T}'_{\Phi^{ij}}\mu\right|\le &\sum_{k\ge-4}\sum_{|k-l|\le1}\int_{\R^3}\left|\Phi^{ij}_{<k-3}(v-v')\right|\chi_l(v')\left|\pr_v^\beta\mu_k(v')\right|dv'\\
    \les & \tl{M}_\mu^{|\beta|}\Gamma_{s_0}(\beta)\int_{\R^3}\sum_{k\in\Z}\sum_{l\in Z}\chi_l(v'){\bf 1}_{[-1,1]}(k-l)\fr{e^{-\fr34|v'|^2}}{|v-v'|}dv'\\
\les&\tl{M}_\mu^{|\beta|}\Gamma_{s_0}(\beta)e^{-c_2|v|^2}\int_{\R^3}\fr{e^{-c_2|v-v'|^2}}{|v-v'|}dv'\les \tl{M}_\mu^{|\beta|}\Gamma_{s_0}(\beta)e^{-c_2|v|^2}\\
\les&M_\sig^{|\beta|}\Gamma_s(\beta)(1+|v|)^{-1-|\beta|}.
\end{align*}

Recalling  \eqref{concel1} and the definition of $\sig^{ij}$ and $\sig^i$ in \eqref{def-sig}, we write
\[
\sig^{ij}v_iv_j=\int_{\R^3}\Phi^{ij}(v-\tl{v})\big(\tl{v}_i\tl{v}_j\mu(\tl{v})\big)d\tl{v}=\Phi^{ij}*(v_iv_j\mu(v))=\Phi^{ij}*\left[\fr14\pr_{v_i}\pr_{v_j}\mu+\fr12\dl_{ij}\mu
\right].
\]
Then \eqref{eq: est sigma_ij} follows immediately.
\end{proof}

\subsection{Product estimates}
In this section, we study the product of functions in the Gevrey class. 
\begin{lem}\label{lem-Gamma-equi}
For any multi-index $\al\in\N^6$, $s\in(0,1]$, there exists a large constant $C$, independent of $\al$ and $s$, such that
\begin{equation}\label{equiv-gamma}
 \fr{1}{C(2\cdot 6^{\fr{1}{s}-1})^{|\al|}}\fr{\Gamma_s(|\al|)}{C_{|\al|}^{\al}}\le \Gamma_s(\al)\le \fr{\Gamma_s(|\al|)}{C_{|\al|}^{\al}}.
\end{equation}
\end{lem}
\begin{proof}
Note that
\begin{align*}
\fr{\Gamma_s(\al)}{\fr{\Gamma_s(|\al|)}{C_{|\al|}^{\al}}}=&\fr{2^{-150}(\al_1!)^{\fr{1}{s}-1}\cdots (\al_6!)^{\fr{1}{s}-1}(|\al_1|+1)^{-12}\cdots(|\al_6|+1)^{-12}}{2^{-25}(|\al|!)^{\fr{1}{s}-1}(|\al|+1)^{-12}}\\
=&2^{-125}\left(\fr{1}{C_{|\al|}^{\al}}\right)^{\fr{1}{s}-1}\fr{(|\al|+1)^{12}}{(|\al_1|+1)^{12}\cdots(|\al_6|+1)^{12}}.
\end{align*}
This, together with the facts $1\le C^{\al}_{|\al|}\le 6^{|\al|}$, and  
\[
1\le \fr{(|\al_1|+1)^{12}\cdots(|\al_6|+1)^{12}}{(|\al|+1)^{12}}\le (|\al|+1)^{60}\le C_12^{|\al|}
\]
implies that \eqref{equiv-gamma} holds with $C=2^{125}C_1$.
\end{proof}

\begin{lem}\label{lem-summable}
For all $m\in\N^6$, there holds
    \begin{equation}\label{summable1}
\sum_{n\le m}\fr{C_{m}^{n}C_{|m|}^{m}}{C_{|n|}^{n}C_{|m-n|}^{m-n}}\fr{\Gamma_s(|n|)\Gamma_s(|m-n|)}{\Gamma_s(|m|)} \le1.
    \end{equation}
\end{lem}
\begin{proof}
Note that for any $i\in \{1, 2, \cdots,6\}$, there holds
\[
\sum_{n_i\in\N}\fr{2^{-3}}{(n_i+1)^2}=2^{-3}\fr{\pi^2}{6}<\fr{1}{2}.
\]
Then
    \begin{align*}
&\sum_{n\le m}\fr{C_{m}^{n}C_{|m|}^{m}}{C_{|n|}^{n}C_{|m-n|}^{m-n}}\fr{\Gamma_s(|n|)\Gamma_s(|m-n|)}{\Gamma_s(|m|)}\\
=&\sum_{n\le m}\fr{\fr{m_1!\cdots m_6!}{n_1!\cdots n_6!(m_1-n_1)!\cdots(m_6-n_6)!}}{\fr{|n|!}{n_1!\cdots n_6!}\fr{|m-n|!}{(m_1-n_1)!\cdots(m_6-n_6)!}}\cdot \fr{|m|!}{m_1!\cdots m_6!}\\
&\times \fr{2^{-25}(|n|!)^\fr{1}{s}(|n|+1)^{-12}2^{-25}(|m-n|!)^\fr{1}{s}(|m-n|+1)^{-12}}{2^{-25}(|m|!)^\fr{1}{s}(|m|+1)^{-12}}\\
=&\sum_{n\le m}\left(\fr{|n|!|m-n|!}{|m|!}\right)^{\fr{1}{s}-1}\fr{2^{-25}(|m|+1)^{12}}{(|n|+1)^{12}(|m-n|+1)^{12}}\\
\le&\left(\sum_{|n|\le \fr{|m|}{2}}+\sum_{\fr{|m|}{2}<|n|\le|m| }\right)\fr{2^{-25}(|m|+1)^{12}}{(|n|+1)^{12}(|m-n|+1)^{12}}\\
\le&2^5\sum_{|n|\le \fr{|m|}{2}}\fr{2^{-18}}{(|n|+1)^{12}}+2^5\sum_{\fr{|m|}{2}<|n|\le|m|}\fr{2^{-18}}{(|m-n|+1)^{12}}\\
\le&2^6\sum_{n\in\N^6}\fr{2^{-18}}{(|n|+1)^{12}}\le 2^6 \prod_{1\le i\le 6}\sum_{n_i\in\N}\fr{2^{-3}}{(n_i+1)^2}\le 1.
   \end{align*}
\end{proof}

\begin{coro}\label{coro-kernel}
    For all $m\in\N^6$, there holds
    \begin{align}\label{summable2}
        \sum_{\substack{n\le m\\ n'\le n}}\fr{C_m^nC_{|m|}^m C_n^{n'}}{C_{|m-n|}^{m-n}C_{|n-n'|}^{n-n'}C_{|n'|}^{n'}}\fr{\Gamma_s(|n'|)\Gamma_s(|n-n'|)\Gamma_s(|m-n|)}{\Gamma_s(|m|)}\le1.
    \end{align}
\end{coro}
\begin{proof}
By virtue of the proof of \eqref{summable1}, we find that
    \begin{align*}
        &\sum_{\substack{n\le m\\ n'\le n}}\fr{C_m^nC_{|m|}^m C_n^{n'}}{C_{|m-n|}^{m-n}C_{|n-n'|}^{n-n'}C_{|n'|}^{n'}}\fr{\Gamma_s(|n'|)\Gamma_s(|n-n'|)\Gamma_s(|m-n|)}{\Gamma_s(|m|)}\\
        =&2^{-50}\sum_{\substack{n\le m\\ n'\le n}}\left(\fr{|n'|!|n-n'|!|m-n|!}{|m|!}\right)^{\fr1s-1}\fr{(|m|+1)^{12}}{(|n'|+1)^{12}(|n-n'|+1)^{12}(|m-n|+1)^{12}}\\
        \le&\sup_{m\in\N^6}\sum_{n\le m}\fr{2^{-25}(|m|+1)^{12}}{(|m-n|+1)^{12}(|n|+1)^{12}}\sup_{n\in\N^6}\sum_{n'\le n}\fr{2^{-25}(|n|+1)^{12}}{(|n'|+1)^{12}(|n-n'|+1)^{12}}\le1.
    \end{align*}
\end{proof}

\begin{lem}\label{lem-summable2}
For all $m\in\N^6$ with $|m|\ge1$, there holds
\begin{align}\label{CK-summa}
    \sum_{\substack{ {  0<}n\le m}}\fr{C_m^nC_{|m|}^m }{C_{|n|}^nC_{|m-n+\bar{e}_j|}^{m-n+\bar{e}_j}}\fr{\Gamma_s(|n|)\Gamma_s(|m-n+\bar{e}_j|)}{\Gamma_s(|m|)\sqrt{|m-n+\bar{e}_j||m|}}\le1.
\end{align}
\end{lem}
\begin{proof}
For any $n\le m$ with $|n|>0$, we have $|m-n|+1\le |m|$. Then  $\fr{\bar{m}_j-\bar{n}_j+1}{\sqrt{(|m-n|+1)|m|}}\le  1$. Using again the facts $n\le m$ and $|n|\ge1$, we find that
\begin{align*}
    \fr{|n|!(|m-n|+1)!}{|m|!}=\fr{ 2\cdot3\cdots|n|}{(|m-n|+2)(|m-n|+3)\dots(|m-n|+|n|)}\le1.
\end{align*}
Therefore, similar to the proof of \eqref{summable1}, we are led to
\begin{align*}
    &\sum_{\substack{ 0<n\le m}}\fr{C_m^nC_{|m|}^m }{C_{|n|}^nC_{|m-n+\bar{e}_j|}^{m-n+\bar{e}_j}}\fr{\Gamma_s(|n|)\Gamma_s(|m-n+\bar{e}_j|)}{\Gamma_s(|m|)\sqrt{|m-n+\bar{e}_j||m|}}\\
    =&\sum_{\substack{ 0<n\le m}}\left(\fr{|n|!(|m-n|+1)!}{|m|!}\right)^{1-\fr{1}{s}}\fr{2^{-25}(|m|+1)^{12}}{(|n|+1)^{12}(|m-n+\bar{e}_j|+1)^{12}} \fr{(\bar{m}_j-\bar{n}_j+1)}{\sqrt{(|m-n|+1)|m|}}\\
    \le&\sum_{\substack{ 0<n\le m}}\fr{2^{-25}(|m|+1)^{12}}{(|n|+1)^{12}(|m-n|+2)^{12}} 
    \le1.
\end{align*}
This completes the proof.
\end{proof}
\begin{rem}\label{rem-summable}
    It is clear to see that \eqref{CK-summa} still holds with $\bar{e}_j$ replaced by $\tl{e}_j$.
\end{rem}

\begin{lem}\label{lem-convolution}
    Assume that multi-indices  $m, n\in \N^6$, and
    $K(m,n)>0$ satisfies 
    \be\label{summable-a.18}
   \sup_{m\in\N^6} \sum_{n\le m}K(m,n)\le M.
    \ee
    Then
    \begin{align}\label{convo1}
    \sum_{m\in\N^6}\sum_{n\le m}K(m,n)f(n)g(m-n)h(m)\le M\|f\|_{L^2_m}\|g\|_{L^2_m}\|h\|_{L^2_m}.
    \end{align}
\end{lem}
\begin{proof}
Thanks to \eqref{summable-a.18}, by Schur’s test, we obtain
\begin{align*}
    &\sum_{m\in\N^6}\sum_{n\le m}K(m,n)f(n)g(m-n)h(m)\\
    \le&\Big(\sum_{m\in\N^6}\Big|\sum_{n\le m}K(m,n)f(n)g(m-n)\Big|^2\Big)^\fr12\Big(\sum_{m\in\N^6}h^2(m)\Big)^\fr12\\
    \le&\left(\sum_{m\in\N^6}\sum_{n\le m}K(m,n)\sum_{n\le m} K(m,n)|f(n)|^2|g(m-n)|^2\right)^\fr12\Big(\sum_{m\in\N^6}h^2(m)\Big)^\fr12\\
    \le&M^\fr12\left(\sum_{n\in\N^6}|f(n)|^2\sum_{m\ge n} K(m,n)|g(m-n)|^2\right)^\fr12\Big(\sum_{m\in\N^6}h^2(m)\Big)^\fr12\\
    \le&M\Big(\sum_{n\in\N^6}f^2(n)\Big)^\fr12\Big(\sum_{m\in\N^6}g^2(m)\Big)^\fr12\Big(\sum_{m\in\N^6}h^2(m)\Big)^\fr12.
\end{align*}
The proof of this lemma is complete.
\end{proof}
Following the proof of \eqref{convo1} line by line, one easily obtain the following corollary.
\begin{coro}\label{coro-convolution}
Assume that multi-indices  $m, n, n'\in \N^6$, and
    $K(m,n, n')>0$ satisfies 
    \be\label{summable}
   \sup_{m\in\N^6} \sum_{\substack{n\le m\\ n'\le n}}K(m,n,n')\le M.
    \ee
    Then
    \begin{align}\label{convo2}
    \sum_{m\in\N^6}\sum_{\substack{n\le m\\n'\le n}}K(m,n,n')f(n')g(n-n')h(m-n)r(m)\le M\|f\|_{l^2_m}\|g\|_{l^2_m}\|h\|_{l^2_m}\|r\|_{l^2_m}.
    \end{align}    
\end{coro}

\begin{lem}\label{lem-product}
    Let $\iota>0$, $g_1=g_1(x), x\in\T^3, g_2=g_2(x,v), (x,v)\in\T^3\times\R^3$. Then
    \begin{align}\label{product1}
\|g_1g_2\|_{\mathcal{G}^{\lm,N}_{s,\iota}}\le C\left(\|\la \nb_x\ra^2g_1\|_{\mathcal{G}^{\lm,\fr{N}{2}}_{s}}\|g_2\|_{\mathcal{G}^{\lm,N}_{s,\iota}}+\|g_1\|_{\mathcal{G}^{\lm,N}_{s}}\|\la \nb_x\ra^2g_2\|_{\mathcal{G}^{\lm,\fr{N}{2}}_{s,\iota}}\right),
    \end{align}
    and similarly,
    \begin{align}\label{product2}
\|g_1g_2\|_{\mathcal{G}^{\lm,N}_{s,\sig,\iota}}\le C\left(\|\la \nb_x\ra^2g_1\|_{\mathcal{G}^{\lm,\fr{N}{2}}_{s}}\|g_2\|_{\mathcal{G}^{\lm,\sig, N}_{s,\sig,\iota}}+\|g_1\|_{\mathcal{G}^{\lm,N}_{s}}\|\la \nb_x\ra^2g_2\|_{\mathcal{G}^{\lm,\fr{N}{2}}_{s,\sig,\iota}}\right).
    \end{align}
\end{lem}
\begin{proof}
To begin with, we split $\|\la v\ra^{\iota}Z^{m+\beta}(g_1g_2)\|_{L^2_{x,v}}$ as follows:
\begin{align*}
    &\|\la v\ra^{\iota}Z^{m+\beta}(g_1g_2)\|_{L^2_{x,v}}\\
    \le&\sum_{\beta'\le\beta,|\beta'|\le|\beta|/2}C_\beta^{\beta'}\sum_{n\le m}C_m^n\left\|g_1^{(n+\beta')}\right\|_{L^\infty_x}\left\|\la v\ra^{\iota}g_2^{(m-n+\beta-\beta')}\right\|_{L^2_{x,v}}\\
    &+\sum_{\beta'\le\beta,|\beta'|>|\beta|/2}C_\beta^{\beta'}\sum_{n\le m}C_m^n\left\|g_1^{(n+\beta')}\right\|_{L^2_x}\left\|\la v\ra^{\iota}g_2^{(m-n+\beta-\beta')}\right\|_{L^\infty_xL^2_{v}}
\end{align*}

We will complete the proof by using the duality argument. For any $g_3(x,v)$ with $\|g_3\|_{\mathcal{G}^{\lm,100}_{s,\iota}}\le\infty$, by Lemmas \ref{lem-summable} and \ref{lem-convolution}, we find that
\begin{align*}
    &\sum_{m\in\N^6,|\beta|\le100}\kappa^{2|\beta|}\left(\fr{\lm^{|m|}}{\Gamma_s(|m|)}C_{|m|}^m\right)^2\left\la Z^{m+\beta}(g_1g_2),\la v\ra^{2\iota}g_3^{(m+\beta)}\right\ra_{x,v}\\
    \le&\sum_{m\in\N^6,|\beta|\le100}\sum_{\beta'\le\beta,|\beta'|\le|\beta|/2}C_\beta^{\beta'}\sum_{n\le m}\fr{C_{m}^{n}C_{|m|}^{m}}{C_{|n|}^{n}C_{|m-n|}^{m-n}}\fr{\Gamma_s(|n|)\Gamma_s(|m-n|)}{\Gamma_s(|m|)}\\
    &\times\kappa^{|\beta'|}\left(\fr{\lm^{|n|}}{\Gamma_s(|n|)}C_{|n|}^{n}\right)\left\|g_1^{(n+\beta')}\right\|_{L^\infty_x}\\
    &\times\kappa^{|\beta-\beta'|}\left(\fr{\lm^{|m-n|}}{\Gamma_s(|m-n|)}C_{|m-n|}^{m-n}\right)\left\|\la v\ra^{\iota}g_2^{(m-n+\beta-\beta')}\right\|_{L^2_{x,v}}\\
    &\times\kappa^{|\beta|}\left(\fr{\lm^{|m|}}{\Gamma_s(|m|)}C_{|m|}^m\right)\left\|\la v\ra^{\iota}g_3^{(m+\beta)}\right\|_{L^2_{x,v}}\\
    &+\sum_{m\in\N^6,|\beta|\le100}\sum_{\beta'\le\beta,|\beta'|>|\beta|/2}C_\beta^{\beta'}\sum_{n\le m}\fr{C_{m}^{n}C_{|m|}^{m}}{C_{|n|}^{n}C_{|m-n|}^{m-n}}\fr{\Gamma_s(|n|)\Gamma_s(|m-n|)}{\Gamma_s(|m|)}\\
    &\times\kappa^{|\beta'|}\left(\fr{\lm^{|n|}}{\Gamma_s(|n|)}C_{|n|}^{n}\right)\left\|g_1^{(n+\beta')}\right\|_{L^2_x}\\
    &\times\kappa^{|\beta-\beta'|}\left(\fr{\lm^{|m-n|}}{\Gamma_s(|m-n|)}C_{|m-n|}^{m-n}\right)\left\|\la v\ra^{\iota}g_2^{(m-n+\beta-\beta')}\right\|_{L^\infty_xL^2_{v}}\\
    &\times\kappa^{|\beta|}\left(\fr{\lm^{|m|}}{\Gamma_s(|m|)}C_{|m|}^m\right)\left\|\la v\ra^{\iota}g_3^{(m+\beta)}\right\|_{L^2_{x,v}}\\
    \le&C\left(\|\la \nb_x\ra^2g_1\|_{\mathcal{G}^{\lm,50}_{s}}\|g_2\|_{\mathcal{G}^{\lm,100}_{s,\iota}}+\|g_1\|_{\mathcal{G}^{\lm,100}_{s}}\|\la \nb_x\ra^2g_2\|_{\mathcal{G}^{\lm,50}_{s,\iota}}\right)\|g_3\|_{\mathcal{G}^{\lm,100}_{s,\iota}}.
\end{align*}
The second inequality \eqref{product2} can be proved  in the same way. Indeed, one can replace 
$\left\la Z^{m+\beta}(g_1g_2),\la v\ra^{2\iota}g_3^{(m+\beta)}\right\ra_{x,v}$ in the above inequality by $\left\la Z^{m+\beta}(g_1g_2),g_3^{(m+\beta)}\right\ra_{\sig,\iota}$, and noting the $g_1$ is independent of $v$, we have
\begin{align*}
    &\left\la Z^{m+\beta}(g_1g_2),g_3^{(m+\beta)}\right\ra_{\sig,\iota}
    \le\left\|Z^{m+\beta}(g_1g_2)\right\|_{\sig,\iota}\left\|g^{(m+\beta)}\right\|_{\sig,\iota}\\
    \le&\sum_{\beta'\le\beta,|\beta'|\le|\beta|/2}C_\beta^{\beta'}\sum_{n\le m}C_m^n\left\|g_1^{(n+\beta')}\right\|_{L^\infty_x}\left\|g_2^{(m-n+\beta-\beta')}\right\|_{\sig,\iota}\left\|g^{(m+\beta)}\right\|_{\sig,\iota}\\
    &+\sum_{\beta'\le\beta,|\beta'|>|\beta|/2}C_\beta^{\beta'}\sum_{n\le m}C_m^n\left\|g_1^{(n+\beta')}\right\|_{L^2_x}\left\|\left|g_2^{(m-n+\beta-\beta')}\right|_{\sig,\iota}\right\|_{L^\infty_x}\left\|g^{(m+\beta)}\right\|_{\sig,\iota}.
\end{align*}
Then \eqref{product2} follows immediately.
\end{proof}

\begin{lem}\label{lem-weighted-young}
There is a universal constant $C>1$, such that
\begin{align}\label{weighted-young}
    \nn&\int_{\R^6}\fr{\la\xi\ra^2}{|\xi|^2}\hat{f}(\xi)\hat{g}(\eta-\xi) \hat{h}(\eta)d\xi d\eta\\
    \le &C\min\left\{\big\|\la v\ra^2\la Y\ra^2f\big\|_{L^2}\|g\|_{L^2}, \big\|\la v\ra^2f\big\|_{L^2}\big\|\la Y\ra^2g\big\|_{L^2} \right\}\|h\|_{L^2}.
\end{align}
\end{lem}

\begin{proof}
Nothing that $\fr{1}{|\xi|^2}\in L^1_{\rm loc}(\R^3)$,  by Young's inequality, it is easy to see that
\begin{align*}
    &\int_{\R^6}\fr{\la\xi\ra^2}{|\xi|^2}\hat{f}(\xi)\hat{g}(\eta-\xi) \bar{\hat{h}}(\eta)d\xi d\eta\le\int_{\R^3}\Big(\int_{|\xi|\le1}+\int_{|\xi|\ge1}\Big)\fr{\la\xi\ra^2}{|\xi|^2}\left|\hat{f}(\xi)\hat{g}(\eta-\xi) \bar{\hat{h}}(\eta)\right|d\xi d\eta\\
    \les &\left\|\fr{{\bf 1}_{|\xi|\le1}}{|\xi|^2}\hat{f}(\xi)\right\|_{L^1_\xi}\|\hat{g}\|_{L^2_\eta}\|\hat{h}\|_{L^2_\eta}+\min\left\{\|\hat{f}\|_{L^1_\eta}\|\hat{g}\|_{L^2_\eta}, \|\hat{f}\|_{L^2_\eta}\|\hat{g}\|_{L^1_\eta}\right\}\|\hat{h}\|_{L^2_\eta}\\
    \les& \|\hat{f}\|_{L^\infty_\eta} \|\hat{g}\|_{L^2_\eta}\|\hat{h}\|_{L^2_\eta}+\min\left\{\|\hat{f}\|_{L^1_\eta}\|\hat{g}\|_{L^2_\eta}, \|\hat{f}\|_{L^2_\eta}\|\hat{g}\|_{L^1_\eta}\right\}\|\hat{h}\|_{L^2_\eta}\\
    \les&\min\left\{\|\hat{f}\|_{L^1_\eta\cap L^\infty_\eta}\|\hat{g}\|_{L^2_\eta}, \|\hat{f}\|_{L^2_\eta\cap L^\infty_\eta}\|\hat{g}\|_{L^1_\eta\cap L^2_\eta} \right\}\|\hat{h}\|_{L^2}.
\end{align*}
Moreover, 
\begin{align*}
    \|\hat{f}\|_{L^1_\eta\cap L^\infty_\eta}\les\big\|\widehat{\la Y\ra^2f}\big\|_{L^2_\eta}+\big\|\la \nb_\eta\ra^2\hat{f}\big\|_{L^2_\eta}= \big\|\la Y\ra^2f\big\|_{L^2}+\big\|\la v\ra^2 f\big\|_{L^2},
\end{align*}

\begin{align*}
     \|\hat{f}\|_{L^2_\eta\cap L^\infty_\eta}\les\big\|\la \nb_\eta\ra^2\hat{f}\big\|_{L^2_\eta}\les \big\|\la v\ra^2f\big\|_{L^2},
\end{align*}
and
\begin{align*}
\|\hat{g}\|_{L^1_\eta\cap L^2_\eta}\les \big\|\la Y\ra^2g\big\|_{L^2}.
\end{align*}
then \eqref{weighted-young} follows immediately.
\end{proof}

\subsection{Properties of the Gevrey class on the physical side}
In this section, we study the Gevrey class on the physical side. We study the relationship between the Gevrey regularity modulator $a_{m,\lambda, s}$ and the derivatives. Roughly speaking, in terms of regularity, $a_{m,\lambda,s}|m|Z^mf\sim a_{m,\lambda,s}\langle Z\rangle^{s}Z^mf$. 
\begin{lem}\label{lem-interp1}
Assume that $0<{\bf a}\le s$, ${\bf b}\in\N$, and $\Theta>0$ satisfying 
\begin{align}\label{lb-Theta}
    {\bf b}\le(s-{\bf a})\Theta, \quad{\rm and}\quad\Theta\le\fr{1}{\bf a}.
\end{align}
Then there exists a positive constant $C$  depending on ${\bf b}$ and the lower bound of $\lm(t)$, such that
\begin{align}\label{interp1}
    \sum_{0<m\in\N^6}\fr{a^2_{m,\lm,s}}{|m|^{2\Theta}}\left\|\la Z\ra^{{\bf b}+{\bf a}\Theta }f^{(m)}\right\|_{L^2_{x,v}}^2\le C\sum_{0<m\in\N^6}a^2_{m,\lm,s}\left\|f^{(m)}\right\|_{L^2_{x,v}}^2.
\end{align}
\end{lem}
\begin{proof}
By the multinominal theorem,
\begin{align*}
    &\left|\mathcal{F}_{x,v}\big[|Z|^{{\bf b}+{\bf a}\Theta}f^{(m)}\big]_k(\eta)\right|^2=|k,\eta+kt|^{2{\bf b}+2{\bf a}\Theta}(k,\eta+kt)^{2m}|\hat{f}_k(\eta)|^2\\
    =&\Big(\sum_{i=1}^6(k,\eta+kt)^{2e_i}\Big)^{{\bf a}\Theta}\sum_{|\al|={\bf b}}C_{|\al|}^{\al}(k,\eta+kt)^{2m+2\al}|\hat{f}_k(\eta)|^2\\
    =&\Big(\sum_{\substack{|\al|={\bf b}\\1\le i\le 6}}C_{|\al|}^{\al}(k,\eta+kt)^{2(m+\al+e_i)}|\hat{f}_k(\eta)|^2\Big)^{{\bf a}\Theta}\Big(\sum_{|\al|={\bf b}}C_{|\al|}^{\al}(k,\eta+kt)^{2m+2\al}|\hat{f}_k(\eta)|^2\Big)^{1-{\bf a}\Theta}.
\end{align*}
Then
\begin{align}\label{est-interp1}
    a^2_{m,\lm,s}\left|\mathcal{F}_{x,v}\big[|Z|^{{\bf b}+{\bf a}\Theta}f^{(m)}\big]_k(\eta)\right|^2
    \nn\le&\Big(\sum_{\substack{|\al|={\bf b}\\1\le i\le 6}}a^2_{m+\al+e_i,\lm,s}\big|\mathcal{F}_{x,v}\big[f^{(m+\al+e_i)}\big]_k(\eta)\big|^2\Big)^{{\bf a}\Theta}\\
    \nn&\times\Big(\sum_{|\al|={\bf b}}a^2_{m+\al,\lm,s}\big|\mathcal{F}_{x,v}\big[f^{(m+\al)}\big]_k(\eta)\big|^2\Big)^{1-{\bf a}\Theta}\\
    &\times\max_{\substack{|\al|={\bf b}\\1\le i\le 6}}\Big(C_{|\al|}^{\al}\fr{a_{m,\lm,s}}{a_{m+\al+e_i,\lm,s}}\Big)^{2{\bf a}\Theta}\max_{|\al|={\bf b}}\Big(C_{|\al|}^{\al}\fr{a_{m,\lm,s}}{a_{m+\al,\lm,s}}\Big)^{2(1-{\bf a}\Theta)}.
\end{align}
Moreover, the lower bound of $\Theta$ in \eqref{lb-Theta} ensures that
\begin{align}\label{est-interp2}
    \nn&\max_{\substack{|\al|={\bf b}\\1\le i\le 6}}\Big(C_{|\al|}^{\al}\fr{a_{m,\lm,s}}{a_{m+\al+e_i,\lm,s}}\Big)^{2{\bf a}\Theta}\max_{|\al|={\bf b}}\Big(C_{|\al|}^{\al}\fr{a_{m,\lm,s}}{a_{m+\al,\lm,s}}\Big)^{2(1-{\bf a}\Theta)}\\
    \nn\le&C\max_{\substack{|\al|={\bf b}\\1\le i\le 6}}\Bigg(\big[(|m|+1)(|m|+2)\cdots (|m|+{\bf b}+1)\big]^{\fr1s}\\
    \nn&\times\fr{(m_i+1)\cdots(m_i+\al_i+1)\prod_{1\le j\le6,j\ne i}(m_j+1)\cdots(m_j+\al_j)}{(|m|+1)(|m|+2)\cdots (|m|+{\bf b}+1)}\Bigg)^{2{\bf a}\Theta}\\
    \nn&\times\max_{|\al|={\bf b}}\Bigg(\big[(|m|+1)(|m|+2)\cdots (|m|+{\bf b})\big]^{\fr1s}\\
    \nn&\times\fr{\prod_{1\le i\le6}(m_i+1)\cdots(m_i+\al_i)}{(|m|+1)(|m|+2)\cdots (|m|+{\bf b})}\Bigg)^{2(1-{\bf a}\Theta)}\\
    \le&C\Big(|m|^{({\bf b}+1)\fr{{\bf a}\Theta}{s}}|m|^{{\bf b}\fr{1-{\bf a}\Theta}{s}}\Big)^2=C|m|^{2\fr{{\bf b}+{\bf a}\Theta}{s}}\le C|m|^{2\Theta},
\end{align}
where $C$ is positive constant depending on ${\bf b}$ and the lower bound of $\lm(t)$. It follows that
\begin{align*}
    \sum_{0<m\in\N^6}\fr{a^2_{m,\lm,s}}{|m|^{2\Theta}}\left\||Z|^{b+a\Theta}f^{(m)}\right\|_{L^2_{x,v}}^2\le& C\Bigg(\sum_{0<m\in\N^6}\sum_{\substack{|\al|={\bf b}\\1\le i\le 6}}a^2_{m+\al+e_i,\lm,s}\big\|f^{(m+\al+e_i)}\big\|_{L^2_{x,v}}^2\Bigg)^{{\bf a}\Theta}\\
    &\times\Bigg(\sum_{0<m\in\N^6}\sum_{|\al|={\bf b}}a^2_{m+\al,\lm,s}\big\|f^{(m+\al)}\big\|_{L^2_{x,v}}^2\Bigg)^{1-{\bf a}\Theta}.
\end{align*}
Then \eqref{interp1} follows immediately.
\end{proof}

\begin{coro}\label{coro-CK}
Let $\iota\ge0$.
There exists a positive constant $C$  depending on $s$,  the lower bound of $\lm(t)$ and $\iota$, such that
\begin{align}\label{interp2}
    \nn&\sum_{m\in\N^6}a^2_{m,\lm,s}(t)\left\|\la Z\ra^{\fr{s}{2}}  \left(\la v\ra^{\iota}f^{(m)}(t)\right)\right\|^2_{L^2_{x,v}}\\
    \le& C \sum_{0<|m|\in\N^6}|m|a^2_{m,\lm,s}(t)\left\|\la v\ra^{\iota}f^{(m)}(t)\right\|^2_{L^2_{x,v}}+\|f\|_{\mathcal{G}^{\lm}_{s,\iota}}^2.
\end{align}

\end{coro}

\begin{proof}
Taking ${\bf b}=0$ and ${\bf a}\Theta=\fr{s}{2}$   in \eqref{est-interp1} and \eqref{est-interp2} yields, for $|m|>1$,
\begin{align*}
    \nn&a^2_{m,\lm,s}\left|\mathcal{F}_{x,v}\big[|Z|^{\fr{s}{2}}(\la v\ra^{\iota}f^{(m)})\big]_k(\eta)\right|^2
    \\
    =&\Big(\sum_{1\le i\le 6}a^2_{m,\lm,s}\big|\mathcal{F}_{x,v}\big[Z^{e_i}(\la v\ra^{\iota}f^{(m)})\big]_k(\eta)\big|^2\Big)^{\fr{s}{2}}\Big(a^2_{m,\lm,s}\big|\mathcal{F}_{x,v}\big[\la v\ra^{\iota}f^{(m)}\big]_k(\eta)\big|^2\Big)^{1-\fr{s}{2}}\\
    \les&\Big(\sum_{1\le i\le 6}|m|a^2_{m+e_i,\lm,s}\big|\mathcal{F}_{x,v}\big[\la v\ra^{\iota}f^{(m+e_i)}\big]_k(\eta)\big|^2\Big)^{\fr{s}{2}}\Big(|m|a^2_{m,\lm,s}\big|\mathcal{F}_{x,v}\big[\la v\ra^{\iota}f^{(m)}\big]_k(\eta)\big|^2\Big)^{1-\fr{s}{2}}\\
    &+\Big(\sum_{1\le j\le 3}a^2_{m,\lm,s}\big|\mathcal{F}_{x,v}\big[\pr_{v_j}\la v\ra^{\iota}f^{(m)})\big]_k(\eta)\big|^2\Big)^{\fr{s}{2}}\Big(a^2_{m,\lm,s}\big|\mathcal{F}_{x,v}\big[\la v\ra^{\iota}f^{(m)}\big]_k(\eta)\big|^2\Big)^{1-\fr{s}{2}}.
\end{align*}
Clearly, if $|m|=0$, then $\fr{a_{m,\lm,s}}{a_{m+e_i,\lm,s}}$ is bounded from above, and hence
\begin{align*}
    \nn&a^2_{m,\lm,s}\left|\mathcal{F}_{x,v}\big[|Z|^{\fr{s}{2}}(\la v\ra^{\iota}f^{(m)})\big]_k(\eta)\right|^2
    \\
    \les&\Big(\sum_{1\le i\le 6}a^2_{m+e_i,\lm,s}\big|\mathcal{F}_{x,v}\big[\la v\ra^{\iota}f^{(m+e_i)}\big]_k(\eta)\big|^2\Big)^{\fr{s}{2}}\Big(a^2_{m,\lm,s}\big|\mathcal{F}_{x,v}\big[\la v\ra^{\iota}f^{(m)}\big]_k(\eta)\big|^2\Big)^{1-\fr{s}{2}}\\
    &+\Big(\sum_{1\le j\le 3}a^2_{m,\lm,s}\big|\mathcal{F}_{x,v}\big[\pr_{v_j}\la v\ra^{\iota}f^{(m)})\big]_k(\eta)\big|^2\Big)^{\fr{s}{2}}\Big(a^2_{m,\lm,s}\big|\mathcal{F}_{x,v}\big[\la v\ra^{\iota}f^{(m)}\big]_k(\eta)\big|^2\Big)^{1-\fr{s}{2}}.
\end{align*}
It follows from the above two inequalities that
\begin{align*}
    &\sum_{m\in\N^6}a^2_{m,s}(t)\left\||Z|^{\fr{s}{2}}  \left(\la v\ra^{\iota}f^{(m)}(t)\right)\right\|^2_{L^2_{x,v}}\\
    \les&\Big(\sum_{1\le i\le 6}\sum_{0<|m|\in\N^6}|m+e_i|a^2_{m+e_i,\lm,s}\big\|\la v\ra^{\iota}f^{(m+e_i)}\big\|_{L^2_{x,v}}^2\Big)^{\fr{s}{2}}\Big(\sum_{0<|m|\in\N^6}|m|a^2_{m,\lm,s}\big\|\la v\ra^{\iota}f^{(m)}\big\|^2_{L^2_{x,v}}\Big)^{1-\fr{s}{2}}\\
    &+\Big(\sum_{1\le j\le 3}\sum_{m\in\N^6}a^2_{m,\lm,s}\big\|\pr_{v_j}\la v\ra^{\iota}f^{(m)}\big\|^2_{L^2_{x,v}}\Big)^{\fr{s}{2}}\Big(\sum_{m\in\N^6}a^2_{m,\lm,s}\big\|\la v\ra^{\iota}f^{(m)}\big\|^2_{L^2_{x,v}}\Big)^{1-\fr{s}{2}}\\
    &+\sum_{m\in\N^6,0\le |m|\le1}a^2_{m,\lm,s}\big\|\la v\ra^{\iota}f^{(m)}\big\|^2_{L^2_{x,v}},
\end{align*}
which implies \eqref{interp2}.
\end{proof}

\subsection{Bony decomposition with respect to the vector field}\label{sec-Bony}
Let $\varphi\in C_0^\infty(\mathbb{R})$ be such that
$\varphi(\xi)=1$ for $|\xi|\le\fr12$ and $\varphi(\xi)=0$ for $|\xi|\ge\fr34$ and define $\theta(\xi)=\varphi\left(\fr{\xi}{2}\right)-\varphi(\xi)$, supported in the range $\xi\in\left(\fr12,\fr32\right)$. Then we have the partition of unity
\[
1=\varphi(\xi)+\sum_{J\in 2^{\mathbb{N}}}\theta_J(\xi),
\]
where $\theta_J(\xi)=\theta\left(\fr{\xi}{J}\right)$. For $f\in L^2(\mathbb{T}^3\times\mathbb{R}^3)$, recalling that $Z=(\nb_x,\nb_v+t\nb_x)$, let us define
\begin{align*}
f_J&=\theta_J(|Z|)f,\quad
f_{\fr12}=\varphi(|Z|)f,\quad
f_{<J}=f_{\fr12}+\sum_{I\in2^{\mathbb{N}}: I<J}f_{I}=\varphi\left(\fr{|Z|}{J}\right)f.
\end{align*}
Then the Littlewood-Paley decomposition is defined by
\[
f=f_{\fr12}+\sum_{J\in2^{\mathbb{N}}}f_J.
\]
The Littlewood-Paley decomposition in the $x\in\T^3$ variables can be defined analogously (with the vector field $Z$ replaced by $\nb_x$).

Now we are in a position to define the paraproduct decomposition, introduced by Bony \cite{bony1981}. Given suitable functions $f,g$, the product of $fg$ can be decomposed as follows
\begin{align*}
fg=&\mathcal{T}_fg+\mathcal{T}_gf+\mathcal{R}(f,g)\\
=&\sum_{J\ge8}f_{<\fr{J}{8}}g_J+\sum_{J\ge8}g_{<\fr{J}{8}}f_J+\sum_{J\in\mathbb{B}}\sum_{\fr{J}{8}\le J'\le8J}g_{J'}f_J
=\sum_{J\ge8}f_{<\fr{J}{8}}g_J+\sum_{J\in\mathbb{B}}g_{<16J}f_J,
\end{align*}
where all the sums are understood to run over $\mathbb{B}=\{\fr12\}\cup 2^{\N}$. In particular, we write
\begin{align*}
    \mathcal{T}'_gf=\mathcal{T}_gf+\mathcal{R}(f,g)=\sum_{J\in\mathbb{B}}g_{<16J}f_J.
\end{align*}

\subsection{Estimates on the Landau operator}
In this section, we introduce some basic properties of the linear Landau operator $L$ and study the nonlinear Landau operator $\Gamma(g_1,g_2)$ in the Gevrey class. 

\begin{lem}[Lemma 5 in \cite{Guo2002}]\label{Lem: Lemma-5}
    Let $\ell\in \mathbb{R}$. For any $0<\zeta<1$, there is $0<C_\zeta<\infty$, such that 
    \begin{align}\label{low-L-0}
        \left|\left\langle \la v\ra^{2\ell}\partial_{v_i}\sigma^i g_1, g_2\right\rangle_v\right|
        +\left|\left\langle \la v\ra^{2\ell} \mathcal{K}g_1, g_2\right\rangle_v\right|\leq \zeta|g_1|_{\sigma, \ell}|g_2|_{\sigma, \ell}
        +C_\zeta\|\bar{\chi}_{C_\zeta}g_1\|_{L^2_v}\|\bar{\chi}_{C_\zeta}g_2\|_{L^2_v},
    \end{align}
    where
\be\label{def-chi}
\bar{\chi}_R(v)=\bar{\chi}(\fr{v}{R}),\quad \chi: \R^d\rightarrow[0,\infty) \ {\rm is \ smooth}, \bar{\chi}(v)=\bar{\chi}(|v|)=\begin{cases}1,\  {\rm if}\  |v|\le1,\\ 0,\ {\rm if}\  |v|\ge2.\end{cases}
\ee
Moreover, there is $\delta>0$, such that
    \begin{align}\label{coercive-0}
        \langle Lg, g\rangle_v\geq \delta |(\mathrm{I}-\mathrm{P}_0)g|_{\sigma}^2.
    \end{align}
\end{lem}
By \eqref{low-L-0} and  \eqref{coercive}, we immediately have the following corollary.
\begin{coro}\label{coro-upL-0}
Let $\ell\in \mathbb{R}$.
\[
\left\la \la v\ra^{2\ell} Lg_1, g_2 \right\ra_v\les \left|g_1 \right|_{\sig, \ell}\left|g_2\right|_{\sig, \ell}. 
\]
    
\end{coro}

\begin{lem}[Lemma 6 in \cite{Guo2002}]\label{Lem: Lemma-6}
Let $\ell\in \mathbb{R}$, $|\beta|>0$. For any $0<\zeta<1$, there is $0<C_\zeta<\infty$, such that
\begin{align*}
-\left\langle \la v\ra^{2\ell}\pr^\beta_v\mathcal{A} g, \pr_v^\beta g\right\rangle_v
        \ge|\pr_v^\beta g|^2_{\sig, \ell}-\zeta\sum_{|\beta'|\le |\beta|}\left|\pr_v^{\beta'}g\right|^2_{\sig,\ell}-C_{\zeta}\|\bar{\chi}_{C_\zeta}g\|_{L^2_v}^2,
\end{align*}
and
\begin{align*}
\left|\left\langle \la v\ra^{2\ell} \pr_v^\beta\mathcal{K}g_1, \pr_v^\beta g_2\right\rangle_v\right|\leq\left[\zeta\sum_{|\beta'|\le |\beta|}\left|\pr_v^{\beta'}g_1\right|_{\sig,\ell}+C_{\zeta}\|\bar{\chi}_{C_\zeta}g_1\|_{L^2_v}\right]|\pr_v^\beta g_2|_{\sig,\ell}.
\end{align*}
\end{lem}

\begin{lem}\label{coro-upL-1}
Let $\ell\in \mathbb{R}$, $|\beta|\ge0$.
    \be\label{upL}
    \big\la \la v\ra^{2\ell} \pr_v^\beta Lg_1, g_2 \big\ra_v\les\sum_{|\beta'|\le|\beta|}\big|\pr_v^{\beta'} g_1\big|_{\sig,\ell}|g_2|_{\sig,\ell}.
    \ee
\end{lem}

\begin{proof}
We only focus on the case when $|\beta|>0$, because the case $|\beta|=0$ is given in Corollary \ref{coro-upL-0}.\par
\noindent{\bf (I): Treatment of $\left\la\la v\ra^{2\ell}\pr_v^\beta\mathcal{A}[g_1], g_2\right\ra_v$.} Recalling \eqref{def-A}, we write
\begin{align}\label{pr_vAg1g2}
\left\la\la v\ra^{2\ell}\pr_v^\beta\mathcal{A}[g_1], g_2\right\ra_v=\nn&-\int_{\R^3}\la v\ra^{2\ell}\Big[\sig^{ij}\pr_{v_j}\pr_v^\beta g_1\pr_{v_i}g_2+\sig^{ij}v_iv_j\pr_v^\beta g_1 g_2\Big] dv\\
\nn&-\sum_{0<\beta'\le\beta}C_\beta^{\beta'}\left\la \la v\ra^{2\ell} \pr_v^{\beta'}\sig^{ij}\pr_v^{\beta-\beta'}\pr_{v_j}g_1, \pr_{v_i}g_2\right\ra_v\\
\nn&-\sum_{0\le\beta''\le\beta}C_\beta^{\beta''}\left\la \pr_{v_i}\la v\ra^{2\ell} \pr_v^{\beta''}\sig^{ij}\pr_v^{\beta-\beta''}\pr_{v_j}g_1, g_2\right\ra_v\\
\nn&-\sum_{0<\beta'\le\beta}C_\beta^{\beta'}\left\la \la v\ra^{2\ell} \pr_v^{\beta'}[\sig^{ij}v_iv_j]\pr_v^{\beta-\beta'}g_1, g_2\right\ra_v\\
&+\sum_{0\le\beta''\le\beta}C_\beta^{\beta''}\left\la \la v\ra^{2\ell} \pr_v^{\beta''}\pr_{v_i}\sig^{i}\pr_v^{\beta-\beta''}g_1, g_2\right\ra_v.
\end{align}
Recalling \eqref{L-norm}, one can define 
\be
\la f,g\ra_{v;\sig,\ell}:=\int_{\R^3}\la v\ra^{2\ell}\big[\sig^{ij}\pr_{v_j}f\pr_{v_i}g+\sig^{ij}v_iv_jfg\big]dv.
\ee
Then $\la \cdot, \cdot\ra_{\sig,\ell}$ is a inner product, since $\sig^{ij}=\sig^{ji}$. By Cauchy-Schwarz inequality,  we have
\[
|\la f,g\ra_{v;\sig,\ell}|\le |f|_{\sig,\ell}|g|_{\sig,\ell}.
\]
Then the first term on the right hand side of \eqref{pr_vAg1g2} can be bounded by $\big|\pr_v^\beta g_1\big|_{\sig,\ell}| g_2|_{\sig,\ell}$.

Next, it is not difficult to verify that
\[
\left|\la v\ra^{2\ell}\pr_v^{\beta'}\sig^{ij}\right|+\left|\pr_{v_i}\la v\ra^{2\ell}\pr_v^{\beta''}\sig^{ij}\right|+\left|\la v\ra^{2\ell} \pr_v^{\beta'}[\sig^{ij}v_iv_j]\right|+\left|\la v\ra^{2\ell} \pr_v^{\beta''}\pr_{v_i}\sig^{i}\right|\les \la v\ra^{2\ell-2},
\]
for any $|\beta'|\ge1$ and $|\beta''|\ge0$. Then the last three terms on the right hand side of \eqref{pr_vAg1g2} can be bounded by
\begin{align*}
&C\int_{\R^3}\la v\ra^{2\ell-2}\big(\big|\pr_v^{\beta-\beta''}\pr_{v_j}g_1\big|+\big|\pr_v^{\beta-\beta'}g_1\big|+\big|\pr_v^{\beta-\beta''}g_1\big|\big)|g_2|dv\\
\le&C\sum_{\al\le \beta}\left(\big\|\la v\ra^{\ell-\fr32}\pr_{v_j}\pr_v^{\al}g_1\big\|_{L^2_v}+\big\|\la v\ra^{\ell-\fr12}\pr_v^{\al}g_1\big\|_{L^2_v}\right)\big\|\la v\ra^{\ell-\fr12}g_2\big\|_{L^2_v}\le C\sum_{\al\le\beta}
\big|\pr_v^{\al}g_1\big|_{\sig, \ell}|g_2|_{\sig,\ell},
\end{align*}
and the second term on the right hand of \eqref{pr_vAg1g2} can be bounded by
\begin{align*}
C\int_{\R^3}\la v\ra^{2\ell-2}\big|\pr_v^{\beta-\beta'}\pr_{v_j}g_1\big||\pr_{v_i}g_2|dv\le&C\big\|\la v\ra^{\ell-\fr12} \pr_v^{\beta-\beta'}\pr_{v_j}g_1\big\|_{L^2_v} \big\|\la v\ra^{\ell-\fr32} \pr_{v_i}g_2\big\|_{L^2_v}\\
\le&C\sum_{|\al|\le|\beta|}\big|\pr_v^{\al}g_1\big|_{\sig,\ell}|g_2|_{\sig,\ell}.
\end{align*}

\noindent{\bf (II): Treatment of $\left\la\la v\ra^{2\ell}\pr_v^\beta\mathcal{K}[g_1], g_2\right\ra_v$.}
Recalling the definition of $\mathcal{K}$, we write
\begin{align*}
    &\left\la\la v\ra^{2\ell}\pr_v^\beta\mathcal{K}[g_1], g_2\right\ra_v\\
    =&\sum_{\substack{\beta'\le \beta\\ \beta''\le\beta'}}C_\beta^{\beta'}C_{\beta'}^{\beta''}\int_{\R^6}\la v\ra^{2\ell}\pr_v^{\beta-\beta'}\mu^{\fr12}(v)\Phi^{ij}(v-\tl{v})\pr_v^{\beta'-\beta''}\mu^{\fr12}(\tl{v})\pr_v^{\beta''}\pr_{v_j}g_1(\tl{v})d\tl{v}\pr_{v_i}g_2(v)dv\\
    &+\sum_{\substack{\beta'\le \beta\\ \beta''\le\beta'}}C_\beta^{\beta'}C_{\beta'}^{\beta''}\int_{\R^6}\la v\ra^{2\ell}\pr_v^{\beta-\beta'}\mu^{\fr12}(v)\Phi^{ij}(v-\tl{v})\pr_v^{\beta'-\beta''}\big(\tl{v}_j\mu^{\fr12}(\tl{v})\big)\pr_v^{\beta''}g_1(\tl{v})d\tl{v}\pr_{v_i}g_2(v)dv\\
    &+\sum_{\substack{\beta'\le \beta\\ \beta''\le\beta'}}C_\beta^{\beta'}C_{\beta'}^{\beta''}\int_{\R^6}\pr_{v_i}\la v\ra^{2\ell}\pr_v^{\beta-\beta'}\mu^{\fr12}(v)\Phi^{ij}(v-\tl{v})\pr_v^{\beta'-\beta''}\mu^{\fr12}(\tl{v})\pr_v^{\beta''}\pr_{v_j}g_1(\tl{v})d\tl{v}g_2(v)dv\\
    &+\sum_{\substack{\beta'\le \beta\\ \beta''\le\beta'}}C_\beta^{\beta'}C_{\beta'}^{\beta''}\int_{\R^6}\pr_{v_i}\la v\ra^{2\ell}\pr_v^{\beta-\beta'}\mu^{\fr12}(v)\Phi^{ij}(v-\tl{v})\pr_v^{\beta'-\beta''}\big(\tl{v}_j\mu^{\fr12}(\tl{v})\big)\pr_v^{\beta''}g_1(\tl{v})d\tl{v}g_2(v)dv\\
    &+\sum_{\substack{\beta'\le \beta\\ \beta''\le\beta'}}C_\beta^{\beta'}C_{\beta'}^{\beta''}\int_{\R^6}\la v\ra^{2\ell}\pr_v^{\beta-\beta'}\big(v_i\mu^{\fr12}(v)\big)\Phi^{ij}(v-\tl{v})\pr_v^{\beta'-\beta''}\mu^{\fr12}(\tl{v})\pr_v^{\beta''}\pr_{v_j}g_1(\tl{v})d\tl{v}g_2(v)dv\\
    &+\sum_{\substack{\beta'\le \beta\\ \beta''\le\beta'}}C_\beta^{\beta'}C_{\beta'}^{\beta''}\int_{\R^6}\la v\ra^{2\ell}\pr_v^{\beta-\beta'}\big(v_i\mu^{\fr12}(v)\big)\Phi^{ij}(v-\tl{v})\pr_v^{\beta'-\beta''}\big(\tl{v}_j\mu^{\fr12}(\tl{v})\big)\pr_v^{\beta''}g_1(\tl{v})d\tl{v}g_2(v)dv\\
    =&\sum_{1\le j\le 6}{\bf k}_j.
\end{align*}
Since $\Phi^{ij}(v-\tl{v})=O(|v|^{-1})\in L^2_{loc}(\R^3)$, Fubini's Theorem implies 
\[
\Phi^{ij}(v-\tl{v})\mu^{\fr14}(v)\mu^{\fr14}(\tl{v})\in L^2(\R^3\times\R^3).
\]
Then ${\bf k}_j, 1\le j\le 6$ can be estimated in the same way. Let us take the first term as an example to estimate as follows: 
\begin{align*}
    {\bf k}_1=&\sum_{\substack{\beta'\le \beta'\\ \beta''\le\beta}}C_\beta^{\beta'}C_{\beta'}^{\beta''}\int_{\R^3}\int_{\R^3}\left(\Phi^{ij}(v-\tl{v})\mu^{\fr14}(v)\mu^{\fr14}(\tl{v})\right)\\
    &\times \underbrace{\left(\mu^{-\fr14}(v)\mu^{-\fr14}(\tl{v})\pr_v^{\beta-\beta'}\mu^{\fr12}(v)\pr_v^{\beta'-\beta''}\mu^{\fr12}(\tl{v})\right)}_{\mu_1^{\beta'\beta''}(v,\tl{v})}\la v\ra^{\ell }\pr_v^{\beta''}\pr_{v_j}g_1(\tl{v})\la v\ra^{\ell }\pr_{v_i}g_2(v)d\tl{v}dv\\
    \le&\sum_{\substack{\beta'\le \beta\\ \beta''\le\beta'}}C_\beta^{\beta'}C_{\beta'}^{\beta''}\left\|\Phi^{ij}(v-\tl{v})\mu^{\fr14}(v)\mu^{\fr14}(\tl{v})\right\|_{L^2(\R^3\times\R^3)}\\
    &\times \left\|\mu_1^{\beta'\beta''}(v,\tl{v})\la v\ra^{\ell }\pr_v^{\beta''}\pr_{v_j}g_1(\tl{v})\la v\ra^{\ell }\pr_{v_i}g_2(v) \right\|_{L^2(\R^3\times\R^3)}\\
    \les&\sum_{\al\le\beta}\left|\pr_v^{\al}g_1\right|_{\sig,\ell}|g_2|_{\sig,\ell}.
\end{align*}
The proof is completed.
\end{proof}
\begin{rem}
In the treatment of the $\mathcal{A}g_1$ involved terms, there is no need to split the integral into two parts: $|v|\le R$ and $|v|\ge R$, and in the treatment of the $\mathcal{K}g_1$ involved terms, there is no need to split $\Phi^{ij}(v-\tl{v})\mu^{\fr14}(v)\mu^{\fr14}(\tl{v})$ into two parts: $\Phi^{ij}(v-\tl{v})\mu^{\fr14}(v)\mu^{\fr14}(\tl{v})-\phi^{ij}(v,\tl{v})$ and $\phi^{ij}(v,\tl{v})$, so that
\begin{gather*}
\left\|\Phi^{ij}(v-\tl{v}) \mu^{\fr14}(v)\mu^{\fr14}(\tl{v})-\phi^{ij}(v,\tl{v})\right\|_{L^2(\R^3\times\R^3)}\le \fr{1}{R},\\
{\rm supp}\, \phi^{ij}(v,\tl{v})\subset \{|v|+|\tl{v}|\le C_R\},
\end{gather*}
because we do not  intend to  extract any small coefficient now.

In addition, one cannot replace $|\beta'|\le|\beta|$ by $\beta'\le \beta$ in \eqref{upL} due to the presence of $\pr_v^{\beta-\beta'}\pr_{v_j}$ with $|\beta'|\ge1$ in the second term on the right hand side of  \eqref{pr_vAg1g2}.

\end{rem}

\begin{lem}\label{lem-NL-collision}
For any $\iota>0$, $\al\in\N^3$, there holds
    \begin{align}\label{est-NG}
        \nn&\sum_{m\in\N^6,|\beta|\le N}\kappa^{2|\beta|}\left(\fr{\lm^{|m|}}{\Gamma_s(|m|)}C_{|m|}^m\right)^2\left\la \pr_v^\al Z^{m+\beta}\Gamma(g_1,g_2),\la v\ra^{2\iota-4|\al|} g_3^{(m+\beta)} \right\ra_{x,v}\\
        \le \nn&C\sum_{\substack{\al'\le\al\\\al''\le\al'}}\Bigg(\left\| \pr_{v}^{\al'-\al''}g_1\right\|_{\mathcal{G}^{\lm,N}_{s,\iota-2|\al'-\al''|}}\|\pr_v^{\al-\al'} g_2\|_{\mathcal{G}^{\lm,N}_{s,{ \sig},\iota-2|\al-\al'|}}\\
        &+\left\|\pr_v^{\al'-\al''}g_1\right\|_{\mathcal{G}^{\lm,N}_{s,{ \sig},\iota-2|\al'-\al''|}}\big\|\pr_v^{\al-\al'} g_2\big\|_{\mathcal{G}^{\lm,N}_{s,\iota-2|\al-\al'|}}\Bigg)\| g_3\|_{\mathcal{G}^{\lm,N}_{s,{ \sig},\iota-2|\al|}}.
    \end{align}
\end{lem}
\begin{proof}
Let us denote $\displaystyle\sum_{\al',\beta',n}^{\al,\beta,m}:=\sum_{\al'\le\al}C_{\al}^{\al'}\sum_{\beta'\le\beta}C_{\beta}^{\beta'}\sum_{n\le m}C_m^n$.    Recalling \eqref{def-Gamma}, we write
\begin{align*}
    &\left\la \pr_{v}^{\al}Z^{m+\beta}\Gamma(g_1,g_2),\la v\ra^{2\iota-4|\al|} g_3^{(m+\beta)} \right\ra_{x,v}\\
    =&-\sum_{\al',\beta',n}^{\al,\beta,m}\left\la  \pr_{v}^{\al'}Z^{n+\beta'}[\Phi^{ij}*(\mu^\fr12g_1)]\pr_{v_j} \pr_v^{\al-\al'}g_2^{(m-n+\beta-\beta')},\la v\ra^{2\iota-4|\al|}\pr_{v_i}g_3^{(m+\beta)}\right\ra_{x,v}\\
    &-\sum_{\al',\beta',n}^{\al,\beta,m}\left\la  \pr_{v}^{\al'}Z^{n+\beta'}[\Phi^{ij}*(v_i\mu^\fr12g_1)]\pr_{v_j} \pr_{v}^{\al-\al'}g_2^{(m-n+\beta-\beta')},\la v\ra^{2\iota-4|\al|}g_3^{(m+\beta)}\right\ra_{x,v}\\
    &+\sum_{\al',\beta',n}^{\al,\beta,m}\left\la  \pr_{v}^{\al'}Z^{n+\beta'}[\Phi^{ij}*(\mu^\fr12\pr_{v_j}g_1)] \pr_{v}^{\al-\al'}g_2^{(m-n+\beta-\beta')} ,\la v\ra^{2\iota-4|\al|}\pr_{v_i}g_3^{(m+\beta)}\right\ra_{x,v}\\
    &+\sum_{\al',\beta',n}^{\al,\beta,m}\left\la  \pr_{v}^{\al'}Z^{n+\beta'}[\Phi^{ij}*(v_i\mu^\fr12\pr_{v_j}g_1)] \pr_{v}^{\al-\al'}g_2^{(m-n+\beta-\beta')},\la v\ra^{2\iota-4|\al|}g_3^{(m+\beta)}\right\ra_{x,v}\\
    &-\sum_{\al',\beta',n}^{\al,\beta,m}\left\la   \pr_{v}^{\al'}Z^{n+\beta'}[\Phi^{ij}*(\mu^\fr12 g_1)]\pr_{v_j}\pr_{v}^{\al-\al'} g_2^{(m-n+\beta-\beta')},\pr_{v_i}\la v\ra^{2\iota-4|\al|}g_3^{(m+\beta)}\right\ra_{x,v}\\
    &+\sum_{\al',\beta',n}^{\al,\beta,m}\left\la  \pr_{v}^{\al'}Z^{n+\beta'}[\Phi^{ij}*(\mu^\fr12\pr_{v_j}g_1)] \pr_{v}^{\al-\al'}g_2^{(m-n+\beta-\beta')},\pr_{v_i}\la v\ra^{2\iota-4|\al|}g_3^{(m+\beta)}\right\ra_{x,v}\\
    =&\sum_{1\le i\le 6}{\rm N}[\Gamma]^{m,\beta}_{(i)}.
\end{align*}
By Stirling's approximation, for $\al''\le \al$, $\beta''\le\beta'$ and $n'\le n$, we have
\begin{align}\label{est-Gamma1s}
    \nn &M_\mu^{|\bar{n'}|}\fr{\Gamma_1(\al''+\bar{\beta''}+\bar{n'})}{\Gamma_s(\bar{n'})}\\
    \nn=& M_\mu^{|\bar{n'}|}\prod_{i=1}^3\fr{(n'_{i+3}+\al^{''}_{i}+\beta''_{i+3})!}{(n'_{i+3}!)^{\fr1s}}\fr{(n'_{i+3}+1)^{12}}{(n'_{i+3}+\al''_{i}+\beta''_{i+3}+1)^{12}}\\
    \nn\le&\prod_{i=1}^3\fr{M_\mu^{n'_{i+3}}(n'_{i+3}+1)\cdots(n'_{i+3}+\al''_{i}+\beta''_{i+3})}{(n'_{i+3}!)^{\fr1s-1}}\\
    \le& C\prod_{i=1}^3\fr{M_\mu^{n'_{i+3}}(n'_{i+3}+N+|\al|)^{N+|\al|}}{\Big(\sqrt{2\pi n'_{i+3}}(\fr{n'_{i+3}}{e})^{n'_{i+3}}\Big)^{\fr1s-1}}\le C.
\end{align}
Combining this with \eqref{equiv-gamma} yields
\begin{align}\label{est-Gamma1}
    M_{\mu}^{|\bar{n'}|}\Gamma_1(\al''+\bar{\beta''}+\bar{n'})\le C \Gamma_s(\bar{n'})\le C\Gamma_s(n')\le C\fr{\Gamma_s(|n'|)}{C_{|n'|}^{n'}}. 
\end{align}
It follows from this and \eqref{e-prmu}, by Cauchy-Schwarz inequality, we have
\begin{align}\label{est-convolution}
    \nn&\int_{\R^3}\Phi^{ij}(v-\tl{v})\pr_v^{\al''}Z^{n'+\beta''}\mu^{\fr12}(\tl{v})\pr_{v}^{\al'-\al''}g_1^{(n-n'+\beta'-\beta'')}(t,x,\tl{v})d\tl{v}\\
    \le\nn& M_{\mu}^{|\al''|+|\bar{\beta''}|+|\bar{n'}|}\Gamma_1(\al''+\bar{\beta''}+\bar{n'})\left(\int_{\R^3}|\Phi^{ij}(v-\tl{v})|^2 \mu^{\fr14}(\tl{v})d\tl{v}\right)^{\fr12}\\
    \nn&\times\left\|\mu^\fr{1}{8} \pr_{v}^{\al'-\al''}g_1^{(n-n'+\beta'-\beta'')}\right\|_{L^2_v}\\
    \le& CM_{\mu}^{|\al''|+|\bar{\beta''}|}\fr{\Gamma_s(|n'|)}{C_{|n'|}^{n'}}\la v\ra^{-1}\left\|\mu^\fr{1}{8}\pr_{v}^{\al'-\al''}g_1^{(n-n'+\beta'-\beta'')}\right\|_{L^2_v}.
\end{align}
For ${\rm N}[\Gamma]^{m,\beta}_{(5)}$,  we further write
\begin{align*}
    &{\rm N}[\Gamma]^{m,\beta}_{(5)}\\
    =&-\sum_{\substack{\al'\le\al\\\al''\le\al'}}C_{\al}^{\al'}C_{\al'}^{\al''}\sum_{\substack{\beta'\le\beta\\ \beta''\le\beta'}}C_{\beta}^{\beta'}C_{\beta'}^{\beta''}\sum_{\substack{n\le m\\ n'\le n}}C_m^n C_n^{n'}\int_{\T^3\times\R^3}\int_{\R^3}\Phi^{ij}(v-\tl{v})\pr_v^{\al''}Z^{n'+\beta''}\mu^{\fr12}(\tl{v})\\
    &\times \pr_v^{\al'-\al''}g_1^{(n-n'+\beta'-\beta'')}(t,x,\tl{v})d\tl{v}\\
    &\times \pr_{v_j} \pr_v^{\al-\al'}g_2^{(m-n+\beta-\beta')}(t,x,v)g_3^{(m+\beta)}(t,x,v)\pr_{v_i}\la v\ra^{2\iota-4|\al|}dxdv\\
    \le&C\sum_{\substack{\al'\le\al\\\al''\le\al'}}\sum_{\substack{\beta'\le\beta, |\beta'|\le |\beta|/2\\ \beta''\le\beta'}}\sum_{\substack{n\le m\\ n'\le n}}C_m^n C_n^{n'}\fr{\Gamma_s(|n'|)}{C_{|n'|}^{n'}}\left\|\mu^\fr{1}{8}\pr_v^{\al'-\al''}g_1^{(n-n'+\beta'-\beta'')}\right\|_{L^\infty_xL^2_v}\\
    &\times \left\|\la v\ra^{\iota-2|\al|-\fr32}\pr_{v_j} \pr_v^{\al-\al'}g_2^{(m-n+\beta-\beta')}\right\|_{L^2_{x,v}}\left\|\la v\ra^{\iota-2|\al|-\fr12} g_3^{(m+\beta)}\right\|_{L^2_{x,v}}\\
    &+C\sum_{\substack{\al'\le\al\\\al''\le\al'}}\sum_{\substack{\beta'\le\beta, |\beta'|>|\beta|/2\\ \beta''\le\beta'}}\sum_{\substack{n\le m\\ n'\le n}}C_m^n C_n^{n'}\fr{\Gamma_s(|n'|)}{C_{|n'|}^{n'}}\left\|\mu^\fr{1}{8}\pr_v^{\al'-\al''}g_1^{(n-n'+\beta'-\beta'')}\right\|_{L^2_{x,v}}\\
    &\times \left\|\la v\ra^{\iota-2|\al|-\fr32}\pr_{v_j} \pr_{v}^{\al-\al'}g_2^{(m-n+\beta-\beta')}\right\|_{L^\infty_xL^2_v}\left\|\la v\ra^{\iota-2|\al|-\fr12} g_3^{(m+\beta)}\right\|_{L^2_{x,v}}\\
    =&{\rm N}^{\rm LH}[\Gamma]^{m,\beta}_{(5)}+{\rm N}^{\rm HL}[\Gamma]^{m,\beta}_{(5)}.
\end{align*}
Then thanks to \eqref{summable2} and \eqref{convo2}, we find that
\begin{align}\label{est-NG5LH}
\nn&\sum_{m\in\N^6,|\beta|\le100}\kappa^{2|\beta|}\left(\fr{\lm^{|m|}}{\Gamma_s(|m|)}C_{|m|}^m\right)^2\left|{\rm N}^{\rm LH}[\Gamma]^{m,\beta}_{(5)}\right|\\
\le \nn&C\sum_{\substack{\al'\le\al\\\al''\le\al'}}\sum_{|\beta|\le N}\kappa^{2|\beta|}\sum_{\substack{\beta'\le\beta, |\beta'|\le |\beta|/2\\ \beta''\le\beta'}}\sum_{m\in\N^6}\sum_{\substack{n\le m\\ n'\le n}}b_{m,m-n,n',s}\\
\nn&\times\lm^{|n'|}\left(a_{n-n',\lm,s}(t)\left\|\mu^\fr{1}{4}\pr_v^{\al'-\al''}g_1^{(n-n'+\beta'-\beta'')}\right\|_{L^\infty_xL^2_v}\right)\\
\nn&\times \left(a_{m-n,\lm,s}(t)\left\|\pr_v^{\al-\al'} g_2^{(m-n+\beta-\beta')}\right\|_{\sig,\iota-|\al-\al'|}\right) \left(a_{m,lm,s}(t)\left\| g_3^{(m+\beta)}\right\|_{\sig,\iota-2|\al|}\right)\\
\le& C\Big(\sum_{m\in\N^6}\lm^{2|m|}\Big)^{\fr12}\sum_{\substack{\al'\le\al\\\al''\le\al'}}\left\|\la\nb_x\ra^2 \pr_v^{\al'-\al''}g_1\right\|_{\mathcal{G}^{\lm,\fr{N}{2}}_{s,\iota-|\al'-\al''|}}\| \pr_v^{\al-\al'}g_2\|_{\mathcal{G}^{\lm,N}_{s,{ \sig},\iota-2|\al-\al'|}}\| g_3\|_{\mathcal{G}^{\lm,N}_{s,{ \sig},\iota-2|\al|}}.
\end{align}
Similarly, we have
\begin{align*}
   &\sum_{m\in\N^6,|\beta|\le N}\kappa^{2|\beta|}\left(\fr{\lm^{|m|}}{\Gamma_s(|m|)}C_{|m|}^m\right)^2\left|{\rm N}^{\rm HL}[\Gamma]^{m,\beta}_{(5)}\right|\\
\le & C\Big(\sum_{m\in\N^6}\lm^{2|m|}\Big)^{\fr12}\sum_{\substack{\al'\le\al\\\al''\le\al'}}\left\| \pr_v^{\al'-\al''}g_1\right\|_{\mathcal{G}^{\lm,N}_{s,\iota-|\al'-\al''|}}\| \la\nb_x\ra^2\pr_v^{\al-\al'}g_2\|_{\mathcal{G}^{\lm,\fr{N}{2}}_{s,{ \sig},\iota-2|\al-\al'|}}\| g_3\|_{\mathcal{G}^{\lm,N}_{s,{ \sig},\iota-2|\al|}}.
\end{align*}
Clearly, ${\rm N}[\Gamma]^{m,\beta}_{(6)}$ can be treated in a similar manner as ${\rm N}[\Gamma]^{m,\beta}_{(5)}$. Here we only state the result:
    \begin{align}\label{est-NG6}
        \nn&\sum_{m\in\N^6,|\beta|\le N}\kappa^{2|\beta|}\left(\fr{\lm^{|m|}}{\Gamma_s(|m|)}C_{|m|}^m\right)^2\left|{\rm N}[\Gamma]^{m,\beta}_{(6)}\right|\\
\nn\le & C\sum_{\substack{\al'\le\al\\\al''\le\al'}}\Big(\left\|\la\nb_x\ra^2 \pr_v^{\al'-\al''}g_1\right\|_{\mathcal{G}^{\lm,\fr{N}{2}}_{s,{ \sig},\iota-|\al'-\al''|}}\| \pr_v^{\al-\al'}g_2\|_{\mathcal{G}^{\lm,N}_{s,\iota-2|\al-\al'|}}\\
&+\left\| \pr_v^{\al'-\al''}g_1\right\|_{\mathcal{G}^{\lm,N}_{s,{ \sig},\iota-|\al'-\al''|}}\| \la\nb_x\ra^2\pr_v^{\al-\al'}g_2\|_{\mathcal{G}^{\lm,\fr{N}{2}}_{s,\iota-2|\al-\al'|}}\Big)\| g_3\|_{\mathcal{G}^{\lm,N}_{s,{ \sig},\iota-2|\al|}}.
    \end{align}

For the treatments of ${\rm N}[\Gamma]^{m,\beta}_{(i)}, i=1,2,3,4$, it suffices to focus on the case where $(v,\tl{v})\in\{2|\tl{v}|\le|v|,|v|\ge1\}$ as in \cite{Guo2002}. In the following, we give the details of the treatment of  the most tricky one is ${\rm N}[\Gamma]^{m,\beta}_{(1)}$. To this end, as in \cite{Guo2002}, we expand $\Phi^{ij}(v-\tl{v})$ to get
\begin{align*}
    \Phi^{ij}(v-\tl{v})=\Phi^{ij}(v)-\pr_{v_k}\Phi^{ij}(v)\tl{v}_k+\fr12\pr_{v_kv_l}^2\Phi^{ij}(\bar{v})\tl{v}_k\tl{v}_l,
\end{align*}
where $\bar{v}$ is between $v$ and $v-\tl{v}$ satisfying $\fr12|v|\le|\bar{v}|\le\fr32|v|$ since the case under consideration is $(v,\tl{v})\in \{2|\tl{v}|\le|v|,|v|\ge1\}$. Moreover, now we have the following upper bounds
\begin{align*}
    |\Phi^{ij}(v)|\le C\la v\ra^{-1},\quad |\nb_v\Phi^{ij}(v)|\le C\la v\ra^{-2},\quad |\nb_v^2\Phi^{ij}(v)|\le C\la v\ra^{-3}.
\end{align*}
In view of the cancellation \eqref{concel1}, one deduces that $\sum_{i,j}\pr_{v_k}\Phi^{ij}(v)v_iv_j=0$, and hence
\begin{align*}
    &\Phi^{ij}(v)\pr_{v_j}\pr_{v}^{\al-\al'}g_2^{(m-n+\beta-\beta')}\pr_{v_i}g_3^{(m+\beta)}\\
    =&\Phi^{ij}(v)\Big((I-P_v)\nb_v\pr_{v}^{\al-\al'}g_2^{(m-n+\beta-\beta')}\Big)_j\Big((I-P_v)\nb_vg_3^{(m+\beta)}\Big)_i,
\end{align*}
and
\begin{align*}
    &\pr_{v_k}\Phi^{ij}(v)\pr_{v_j}\pr_{v}^{\al-\al'}g_2^{(m-n+\beta-\beta')}\pr_{v_i}g_3^{(m+\beta)}\\
    =&\pr_{v_k}\Phi^{ij}(v)\Big(P_v\nb_v\pr_{v}^{\al-\al'}g_2^{(m-n+\beta-\beta')}\Big)_j\Big((I-P_v)\nb_vg_3^{(m+\beta)}\Big)_i\\
    &+\pr_{v_k}\Phi^{ij}(v)\Big((I-P_v)\nb_v\pr_{v}^{\al-\al'}g_2^{(m-n+\beta-\beta')}\Big)_j\Big(P_v\nb_vg_3^{(m+\beta)}\Big)_i\\
    &+\pr_{v_k}\Phi^{ij}(v)\Big((I-P_v)\nb_v\pr_{v}^{\al-\al'}g_2^{(m-n+\beta-\beta')}\Big)_j\Big((I-P_v)\nb_vg_3^{(m+\beta)}\Big)_i,
\end{align*}
where $P_v$ is given in \eqref{eq: P_v definition}. 
It follows that
\begin{align*}
    &\left|\Phi^{ij}(v-\tl{v})\pr_{v_j}g_2^{(m-n+\beta-\beta')}\pr_{v_i}g_3^{(m+\beta)} \la v\ra^{2\iota-4|\al|}\right|\\
    \le&C\la \tl{v}\ra^2\Big(\la v\ra^{\iota-2|\al|-\fr12}\big|(I-P_v)\nb_v\pr_{v}^{\al-\al'}g_2^{(m-n+\beta-\beta')}\big|\Big) \Big( \la v\ra^{\iota-2|\al|-\fr12}\big|(I-P_v)\nb_vg_3^{(m+\beta)}\big|\Big)\\
    &+C\la \tl{v}\ra^2\Big(\la v\ra^{\iota-2|\al|-\fr32}\big|P_v\nb_v\pr_{v}^{\al-\al'}g_2^{(m-n+\beta-\beta')}\big|\Big) \Big( \la v\ra^{\iota-2|\al|-\fr12}\big|(I-P_v)\nb_vg_3^{(m+\beta)}\big|\Big)\\
    &+C\la \tl{v}\ra^2\Big(\la v\ra^{\iota-2|\al|-\fr12}\big|(I-P_v)\nb_v\pr_{v}^{\al-\al'}g_2^{(m-n+\beta-\beta')}\big|\Big) \Big( \la v\ra^{\iota-2|\al|-\fr32}\big|P_v\nb_vg_3^{(m+\beta)}\big|\Big)\\
    &+C\la \tl{v}\ra^2\Big(\la v\ra^{\iota-2|\al|-\fr32}\big|\nb_v\pr_{v}^{\al-\al'}g_2^{(m-n+\beta-\beta')}\big|\Big) \Big( \la v\ra^{\iota-2|\al|-\fr32}\big|\nb_vg_3^{(m+\beta)}\big|\Big).
\end{align*}
On the other hand, using \eqref{e-prmu} and \eqref{est-Gamma1} again, we have
\begin{align*}
    &\int_{\R^3}\la\tl{v}\ra^2\left|
    \pr_v^{\al''}Z^{n'+\beta''}\mu^{\fr12}(\tl{v})\pr_v^{\al'-\al''}g_1^{(n-n'+\beta'-\beta'')}(t,x,\tl{v})\right|d\tl{v}\\
    \le& CM_{\mu}^{|\al''|+|\bar{\beta''}|}\fr{\Gamma_s(|n'|)}{C_{|n'|}^{n'}}\int_{\R^3}\la\tl{v}\ra^2\mu^{\fr14}(\tl{v})\left|\pr_v^{\al'-\al''}g_1^{(n-n'+\beta'-\beta'')}(t,x,\tl{v})\right|d\tl{v}\\
    \le&CM_{\mu}^{|\al''|+|\bar{\beta''}|}\fr{\Gamma_s(|n'|)}{C_{|n'|}^{n'}}\left\|\mu^{\fr14}\pr_v^{\al'-\al''}g_1^{(n-n'+\beta'-\beta'')}\right\|_{L^2_v}.
\end{align*}
Substituting the above two inequalites into the expression of ${\rm N}[\Gamma]^{m,\beta}_{(1)}$, using Corollary \ref{coro-sig}, we are led to
\begin{align*}
{\rm N}[\Gamma]^{m,\beta}_{(1)}
\le&C\sum_{\substack{\al'\le\al\\\al''\le\al'}}\sum_{\substack{\beta'\le\beta, |\beta'|\le|\beta|/2|\\ \beta''\le\beta'}}\sum_{\substack{n\le m\\ n'\le n}}C_m^n C_n^{n'}\fr{\Gamma_s(|n'|)}{C_{|n'|}^{n'}}\\
&\times\left\|\mu^{\fr14}\pr_v^{\al'-\al''}g_1^{(n-n'+\beta'-\beta'')}\right\|_{L^\infty_xL^2_v} \left\|\pr_v^{\al-\al'}g_2^{(m-n+\beta-\beta')}\right\|_{\sig,\iota-2|\al-\al'|}\left\|g_3^{(m+\beta)}\right\|_{\sig,\iota-2|\al|}\\
&+C\sum_{\substack{\al'\le\al\\\al''\le\al'}}\sum_{\substack{\beta'\le\beta, |\beta'|>|\beta|/2|\\ \beta''\le\beta'}}\sum_{\substack{n\le m\\ n'\le n}}C_m^n C_n^{n'}\fr{\Gamma_s(|n'|)}{C_{|n'|}^{n'}}\\
&\times\left\|\mu^{\fr14}\pr_v^{\al'-\al''}g_1^{(n-n'+\beta'-\beta'')}\right\|_{L^2_{x,v}} \left\|\left|\pr_v^{\al-\al'}g_2^{(m-n+\beta-\beta')}\right|_{\sig,\iota}\right\|_{L^\infty_x}\left\|g_3^{(m+\beta)}\right\|_{\sig,\iota}.
\end{align*}
By virtue of Corollaries \ref{coro-kernel} and \ref{coro-convolution}, similar to \eqref{est-NG6}, we arrive at
\begin{align}\label{est-NG1}
    \nn&\sum_{m\in\N^6,|\beta|\le N}\kappa^{2|\beta|}\left(\fr{\lm^{|m|}}{\Gamma_s(|m|)}C_{|m|}^{m}\right)^2{\rm N}[\Gamma]^{m,\beta}_{(1)}\\
    \nn\le&C\sum_{\substack{\al'\le\al\\\al''\le\al'}}\Big(\left\|\la\nb_x\ra^2 \pr_v^{\al'-\al''}g_1\right\|_{\mathcal{G}^{\lm,\fr{N}{2}}_{s,\iota-|\al'-\al''|}}\| \pr_v^{\al-\al'}g_2\|_{\mathcal{G}^{\lm,N}_{s,{ \sig},\iota-2|\al-\al'|}}\\
&+\left\| \pr_v^{\al'-\al''}g_1\right\|_{\mathcal{G}^{\lm,N}_{s,\iota-|\al'-\al''|}}\| \la\nb_x\ra^2\pr_v^{\al-\al'}g_2\|_{\mathcal{G}^{\lm,\fr{N}{2}}_{s,{ \sig},\iota-2|\al-\al'|}}\Big)\| g_3\|_{\mathcal{G}^{\lm,N}_{s,{ \sig},\iota-2|\al|}}.
\end{align}

 Finally, we would like to point out that \eqref{est-NG6} or \eqref{est-NG1} holds  with ${\rm N}[\Gamma]^{m,\beta}_{(6)}$ or ${\rm N}[\Gamma]^{m,\beta}_{(1)}$ replaced by ${\rm N}[\Gamma]^{m,\beta}_{(i)}, i=2,3,4$, and hence we get \eqref{est-NG}. The proof of this lemma is completed.
\end{proof}
\begin{rem}\label{rem-NL-collision}
    It is easy to verify that \eqref{est-NG} still holds with $\pr_v$ replaced by $Y$. More precisely, we have
\begin{align*}
        \nn&\sum_{m\in\N^6,|\beta|\le N}\kappa^{2|\beta|}\left(\fr{\lm^{|m|}}{\Gamma_s(|m|)}C_{|m|}^m\right)^2\left\la Y Z^{m+\beta}\Gamma(g_1,g_2),\la v\ra^{2\iota-4} g_3^{(m+\beta)} \right\ra_{x,v}\\
        \le \nn&C\Bigg(\left\| g_1\right\|_{\mathcal{G}^{\lm,N}_{s,\iota}}\| g_2\|_{\mathcal{G}^{\lm,N}_{s,{ \sig},\iota}}+\left\|g_1\right\|_{\mathcal{G}^{\lm,N}_{s,{ \sig},\iota}}\big\| g_2\big\|_{\mathcal{G}^{\lm,N}_{s,\iota}}\Bigg)\| g_3\|_{\mathcal{G}^{\lm,N}_{s,{ \sig},\iota-2}}\\
        &+C\Bigg(\left\|Y g_1\right\|_{\mathcal{G}^{\lm,N}_{s,\iota-2}}\| g_2\|_{\mathcal{G}^{\lm,N}_{s,{ \sig},\iota}}+\left\|Y g_1\right\|_{\mathcal{G}^{\lm,N}_{s,\sig,\iota-2}}\| g_2\|_{\mathcal{G}^{\lm,N}_{s,\iota}}\\
        &+\left\|g_1\right\|_{\mathcal{G}^{\lm,N}_{s,\iota}}\big\|Y g_2\big\|_{\mathcal{G}^{\lm,N}_{s,\sig,\iota-2}}+\left\|g_1\right\|_{\mathcal{G}^{\lm,N}_{s,{ \sig},\iota}}\big\|Y g_2\big\|_{\mathcal{G}^{\lm,N}_{s,\iota-2}}\Bigg)\| g_3\|_{\mathcal{G}^{\lm,N}_{s,{ \sig},\iota-2}}.
    \end{align*}
\end{rem}

\noindent{\bf Acknowledgments}
JB was supported by DMS-2108633.
  RZ is partially  supported by NSF of China under  Grants 12222105.

\bibliographystyle{abbrv.bst} 
\bibliography{references.bib}

\end{document}